\documentclass[english,leqno]{article}

\usepackage{latexsym}
\usepackage{linearb}
\usepackage[LGR, OT1]{fontenc}
\usepackage{amsmath,amsthm}
\usepackage{amssymb}
\usepackage{MnSymbol}
\usepackage{graphicx}
\usepackage{marvosym}
\usepackage{eufrak}
\usepackage[margin=1.25in,dvips]{geometry}
\usepackage{color}
\usepackage{comment}
\usepackage{mathrsfs}\usepackage[LGR,T1]{fontenc}
\usepackage[latin9]{inputenc}
\usepackage{geometry}
\usepackage{babel}
\usepackage{enumitem}
\usepackage{amsthm}
\usepackage{amstext}
\usepackage{amssymb}
\usepackage{esint}
\usepackage{float}
\usepackage[unicode=true,pdfusetitle,
 bookmarks=true,bookmarksnumbered=false,bookmarksopen=false,
 breaklinks=false,pdfborder={0 0 1},backref=false,colorlinks=false]
 {hyperref}
\makeatletter

\DeclareRobustCommand{\greektext}{%
  \fontencoding{LGR}\selectfont\def\encodingdefault{LGR}}
\DeclareRobustCommand{\textgreek}[1]{\leavevmode{\greektext #1}}
\ProvideTextCommand{\~}{LGR}[1]{\char126#1}

\numberwithin{equation}{section}
\numberwithin{figure}{section}
\newcommand{\lyxaddress}[1]{
\par {\raggedright #1
\vspace{1.4em}
\noindent\par}
}
  \theoremstyle{remark}
  \newtheorem*{rem*}{\protect\remarkname}
  \theoremstyle{definition}
  \newtheorem{defn}{\protect\definitionname}[section]
 \theoremstyle{definition}
 \newtheorem*{defn*}{\protect\definitionname}
  \theoremstyle{plain}
  \newtheorem{lem}{\protect\lemmaname}[section]
  \theoremstyle{plain}
  \newtheorem{prop}{\protect\propositionname}[section]
  \theoremstyle{plain}
  \newtheorem{cor}{\protect\corollaryname}[section]

\newtheorem{innercustomthm}{Theorem}
\newenvironment{customthm}[1]
  {\renewcommand\theinnercustomthm{#1}\innercustomthm}
  {\endinnercustomthm}

\makeatother

\usepackage{babel}
  \providecommand{\definitionname}{Definition}
  \providecommand{\lemmaname}{Lemma}
  \providecommand{\propositionname}{Proposition}
  \providecommand{\remarkname}{Remark}
\providecommand{\corollaryname}{Corollary}

\begin{document}
\title{A proof of the instability of AdS\\ for the Einstein--massless Vlasov system}

\author{Georgios Moschidis}
\maketitle

\lyxaddress{University of California Berkeley, Department of Mathematics, Evans
Hall, Berkeley, CA 94720-3840, United States, \tt gmoschidis@berkeley.edu}
\begin{abstract}
In recent years, the conjecture on the insability of Anti-de~Sitter
spacetime, put forward by Dafermos\textendash Holzegel \cite{DafHol,DafermosTalk}
in 2006, has attracted a substantial amount of numerical and heuristic
studies. Following the pioneering work \cite{bizon2011weakly} of
Bizon\textendash Rostworowski, research efforts have been mainly focused
on the study of the spherically symmetric Einstein\textendash scalar
field system. The first rigorous proof of the instability of AdS in
the simplest spherically symmetric setting, namely for the \emph{Einstein\textendash null
dust system}, was obtained in \cite{MoschidisNullDust}. In order
to circumvent problems associated with the trivial break down of the
Einstein\textendash null dust system occuring at the center $r=0$,
\cite{MoschidisNullDust} studied the evolution of the system in the
exterior of an \emph{inner mirror} placed at $r=r_{0}$, $r_{0}>0$.
However, in view of additional considerations on the nature of the
instability, it was necessary for \cite{MoschidisNullDust} to allow
the mirror radius $r_{0}$ to shrink to $0$ with the size of the
initial perturbation; well-posedness in the resulting complicated
setup (involving low-regularity estimates of uniform modulus with
respect to $r_{0}$) was obtained in \cite{MoschidisMaximalDevelopment}.

In this paper, we establish the instability of AdS for the \emph{Einstein\textendash massless
Vlasov} system in spherical symmetry; this will be the first proof
of the AdS instability conjecture for an Einstein\textendash matter
system which is well-posed for regular initial data in the standard
sense, \emph{without the addition of an inner mirror}. The necessary
well-posedness results for this system are obtained in our companion
paper \cite{MoschidisVlasovWellPosedness}.

Our proof utilises an instability mechanism based on beam interactions
which is superficially similar to the one appearing in \cite{MoschidisNullDust}.
However, new difficulties associated with the Einstein\textendash massless
Vlasov system (such as the need for control on the paths of non-radial
geodesics in a large curvature regime) will force us to develop a
different strategy of proof involving a novel configuration of beam
interactions. One of the main novelties of our construction is the
introduction of a multi-scale hierarchy of domains in phase space,
on which the initial support of the Vlasov field $f$ is localised.
The propagation of this hierarchical structure of the support of $f$
along the evolution will be crucial both for controlling the geodesic
flow under minimal regularity assumptions and for guaranteeing the
existence of the solution until the time of trapped surface formation.
\end{abstract}
\tableofcontents{}

\section{Introduction}

In the presence of a \emph{negative} cosmological constant $\Lambda$,
the maximally symmetric solution of the \emph{vacuum Einstein equations
}
\begin{equation}
Ric_{\text{\textgreek{m}\textgreek{n}}}-\frac{1}{2}Rg_{\text{\textgreek{m}\textgreek{n}}}+\Lambda g_{\text{\textgreek{m}\textgreek{n}}}=0\label{eq:VacuumEinsteinEquations}
\end{equation}
in $n+1$ dimensions, $n\ge3$, is Anti-de~Sitter spacetime $(\mathcal{M}_{AdS},g_{AdS})$.
Expressed in the standard polar coordinate chart on $\mathcal{M}_{AdS}\simeq\mathbb{R}^{n+1}$,
the AdS metric $g_{AdS}$ takes the form 
\begin{equation}
g_{AdS}=-\big(1-\frac{2}{n(n-1)}\Lambda r^{2}\big)dt^{2}+\big(1-\frac{2}{n(n-1)}\Lambda r^{2}\big)^{-1}dr^{2}+r^{2}g_{\mathbb{S}^{n-1}},\label{eq:AdSMetricPolarCoordinates}
\end{equation}
where $g_{\mathbb{S}^{n-1}}$ is the standard metric on the round
sphere of radius $1$. A \emph{conformal boundary} $\mathcal{I}$
can be naturally attached to $(\mathcal{M}_{AdS},g_{AdS})$ at $r=\infty$,
with $\mathcal{I}$ having the conformal structure of a \emph{timelike}
hypersurface diffeomorphic to $\mathbb{R}\times\mathbb{S}^{n-1}$
(see Figure \ref{fig:AdS_Conformal}). 

\begin{figure}[h] 
\centering 
\scriptsize
\begingroup%
  \makeatletter%
  \providecommand\color[2][]{%
    \errmessage{(Inkscape) Color is used for the text in Inkscape, but the package 'color.sty' is not loaded}%
    \renewcommand\color[2][]{}%
  }%
  \providecommand\transparent[1]{%
    \errmessage{(Inkscape) Transparency is used (non-zero) for the text in Inkscape, but the package 'transparent.sty' is not loaded}%
    \renewcommand\transparent[1]{}%
  }%
  \providecommand\rotatebox[2]{#2}%
  \newcommand*\fsize{\dimexpr\f@size pt\relax}%
  \newcommand*\lineheight[1]{\fontsize{\fsize}{#1\fsize}\selectfont}%
  \ifx\svgwidth\undefined%
    \setlength{\unitlength}{144bp}%
    \ifx\svgscale\undefined%
      \relax%
    \else%
      \setlength{\unitlength}{\unitlength * \real{\svgscale}}%
    \fi%
  \else%
    \setlength{\unitlength}{\svgwidth}%
  \fi%
  \global\let\svgwidth\undefined%
  \global\let\svgscale\undefined%
  \makeatother%
  \begin{picture}(1,1.04166667)%
    \lineheight{1}%
    \setlength\tabcolsep{0pt}%
    \put(0,0){\includegraphics[width=\unitlength,page=1]{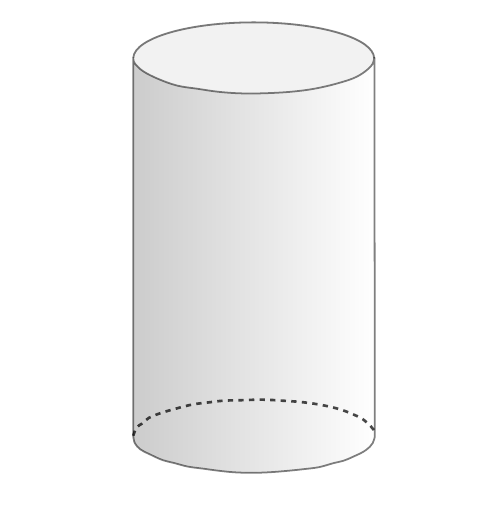}}%
    \put(0.47656002,0.1566871){\color[rgb]{0,0,0}\makebox(0,0)[lt]{\lineheight{0}\smash{\begin{tabular}[t]{l}$\mathbb{S}^n_+$\end{tabular}}}}%
    \put(0,0){\includegraphics[width=\unitlength,page=2]{AdS.pdf}}%
    \put(0.13184435,0.50467){\color[rgb]{0,0,0}\makebox(0,0)[lt]{\lineheight{0}\smash{\begin{tabular}[t]{l}$t$\end{tabular}}}}%
    \put(0.80038778,0.52348365){\color[rgb]{0,0,0}\makebox(0,0)[lt]{\lineheight{0}\smash{\begin{tabular}[t]{l}$\mathcal{I}\simeq \mathbb{R}\times \mathbb{S}^{n-1}$\end{tabular}}}}%
    \put(0.44462943,0.51193996){\color[rgb]{0,0,0}\makebox(0,0)[lt]{\lineheight{0}\smash{\begin{tabular}[t]{l}$\mathcal{M}_{AdS}$\end{tabular}}}}%
  \end{picture}%
\endgroup%
 
\caption{The AdS spacetime $(\mathcal{M}_{AdS}^{n+1},g_{AdS})$ can be conformally identified with the interior of $(\mathbb{R}\times\mathbb{S}_{+}^{n},g_{E}^{(\Lambda)})$, where $\mathbb{S}_{+}^{n}$ is the northern hemisphere of $\mathbb{S}^{n}$ and $g_{E}^{(\Lambda)}\doteq-dt^{2}+\Big(\frac{n(n-1)}{-2\Lambda}\Big)g_{\mathbb{S}_{+}^{n}}$. The timelike boundary $\mathcal{I}$ of $(\mathbb{R}\times\mathbb{S}_{+}^{n},g_{E}^{(\Lambda)})$ corresponds to the conformal boundary of $(\mathcal{M}_{AdS}^{n+1},g_{AdS})$ at infinity. \label{fig:AdS_Conformal}}
\end{figure}

More generally, a conformal boundary $\mathcal{I}$ with similar properties
can be attached to any spacetime $(\mathcal{M},g)$ which is merely
\emph{asymptotically AdS}, i.\,e.~possesses an asymptotic region
with geometry resembling that of (\ref{eq:AdSMetricPolarCoordinates})
in the region $\{r\ge R_{0}\}$, $R_{0}\gg\frac{1}{\sqrt{-\Lambda}}$.
For a more detailed exposition of the geometric properties associated
to AdS asymptotics, see \cite{HawkingEllis1973}.

The hyperbolic nature of the system (\ref{eq:VacuumEinsteinEquations})
and the timelike character of $\mathcal{I}$ imply that the right
framework to study asymptotically AdS solutions of (\ref{eq:VacuumEinsteinEquations})
is that of an \emph{initial-boundary} value problem, with boundary
conditions imposed asymptotically on $\mathcal{I}$. The well-posedness
of the asymptotically AdS initial-boundary value problem for (\ref{eq:VacuumEinsteinEquations})
was first addressed by Friedrich \cite{Friedrich1995}, who established
the existence of solutions for a broad class of boundary conditions
on $\mathcal{I}$, including examples both of \emph{reflecting} and
of \emph{dissipative} conditions (see also the discussion in \cite{Friedrich2014,HolzegelLukSmuleviciWarnick}).
\textgreek{T}he formulation of appropriate boundary conditions for
(\ref{eq:VacuumEinsteinEquations}) on $\mathcal{I}$ and their effects
on the spacetime geometry have also been investigated in the high
energy physics literature; the recent surge of interest on these topics
was sparked by the the putative \emph{AdS/CFT correspondence},\emph{
}put forward by Maldacena \cite{Maldacena}, Gubser\textendash Klebanov\textendash Polyakov
\cite{gubser1998gauge} and Witten \cite{witten1998anti} in 1998
(see \cite{AGMOO2000,Hartnoll2009,AmmonErdmenger}).

The well-posedness of the initial-boundary value problem for (\ref{eq:VacuumEinsteinEquations})
allows discussing the \emph{dynamics} associated to families of asymptotically
AdS initial data sets for (\ref{eq:VacuumEinsteinEquations}). Thus,
the question of stability of the trivial solution $(\mathcal{M}_{AdS},g_{AdS})$
under perturbation of its initial data arises naturally in this context.
When reflecting boundary conditions are imposed on $\mathcal{I}$,
the possibility of non-linear instability for $(\mathcal{M}_{AdS},g_{AdS})$
is already insinuated by the lack of asymptotic stability for solutions
to linear toy-models for (\ref{eq:VacuumEinsteinEquations}); this
is already illustrated by the simple example of the conformally coupled
Klein\textendash Gordon equation 
\begin{equation}
\square_{g_{AdS}}\varphi-\frac{2}{3}\Lambda\varphi=0
\end{equation}
 on $(\mathcal{M}_{AdS}^{3+1},g_{AdS})$, where imposing Dirichlet
conditions for $r\varphi$ on $\mathcal{I}$ results in the energy
flux of $\varphi$ through the foliation $\{t=\tau\}$ to be constant
in $\tau$, thus preventing any non-trivial solution $\varphi$ from
decaying to $0$ as $\tau\rightarrow+\infty$.\footnote{The failure of asymptotic stability for $(\mathcal{M}_{AdS},g_{AdS})$
as a solution of the \emph{non-linear} system (\ref{eq:VacuumEinsteinEquations})
with reflecting conditions on $\mathcal{I}$ follows from the results
of M. Anderson \cite{AndersonAdS}.} Motivated by additional considerations in the setting of the biaxial
Bianchi IX symmetry class for (\ref{eq:VacuumEinsteinEquations})
in 4+1 dimensions, Dafermos and Holzegel \cite{DafHol,DafermosTalk}
in fact conjectured a stronger instability statement in 2006:

\medskip{}

\noindent \textbf{AdS instability conjecture.} \emph{There exist
arbitrarily small perturbations to the initial data of $(\mathcal{M}_{AdS},g_{AdS})$
which, under evolution by the vacuum Einstein equations (\ref{eq:VacuumEinsteinEquations})
with a reflecting boundary condition on $\mathcal{I}$, lead to the
development of black hole regions. In particular, $(\mathcal{M}_{AdS},g_{AdS})$
is non-linearly unstable.}

\medskip{}

\noindent The scenario proposed by the conjecture can be also viewed
as a manifestation of \emph{gravitational turbulence}: The formation
of black hole regions signifies the emergence of non-trivial geometric
structures at small scales, arising from the non-linear evolution
of initial data which were almost trivial at the same spatial scales.
\begin{rem*}
We should point out that the above formulation of the AdS instability
conjecture is ambiguous with respect to the initial data norm $||\cdot||_{data}$
measuring the ``smallness'' of the perturbations. A minimal requirement
for $||\cdot||_{data}$ is that perturbations of $(\mathcal{M}_{AdS},g_{AdS})$
which are small with respect to $||\cdot||_{data}$ should give rise
to solutions $g$ of (\ref{eq:VacuumEinsteinEquations}) which exist
(and remain close to $g_{AdS}$) for long time intervals $\{0\le t\le T_{*}\}$,
i.\,e.~that$(\mathcal{M}_{AdS},g_{AdS})$ is Cauchy stable as a
solution of the initial-boundary value problem for (\ref{eq:VacuumEinsteinEquations});\footnote{Here, Cauchy stability should be understood as stability over compact
subsets of $\mathcal{M}_{AdS}\cup\mathcal{I}$ in the conformal picture.} this condition implies, in particular, that the timescale of black
hole formation tends to $+\infty$ as the size of the initial perturbation
shrinks to $0$. The requirement for long-time existence in fact imposes
a condition on $||\cdot||_{data}$ as a measure of regularity of initial
data sets: The trapped surface formation results of \cite{Christodoulou2009,KlainermanLukRodnianski2014,AnLuk2017}
imply that there is no uniform time of existence for solutions to
(\ref{eq:VacuumEinsteinEquations}) in terms of initial data norms
which are not sufficiently strong so as to measure the concentration
of energy at small scales (such as norms of regularity below $||\cdot||_{H^{\frac{3}{2}}}$
when $n=3$, as a corollary of \cite{AnLuk2017}). 

The choice of reflecting boundary conditions on $\mathcal{I}$ is
also crucial for the validity of the conjecture: Assuming, instead,
``optimally dissipative'' conditions on $\mathcal{I}$, Holzegel\textendash Luk\textendash Smulevici\textendash Warnick
\cite{HolzegelLukSmuleviciWarnick} showed that solutions to the linearized\emph{
}vacuum Einstein equations on $(\mathcal{M}_{AdS},g_{AdS})$ decay
at a superpolynomial rate in $t$, providing a strong indication of
non-linear asymptotic stability for perturbations of $(\mathcal{M}_{AdS},g_{AdS})$
in this setting.
\end{rem*}
The study of the AdS instability conjecture in $3+1$ dimensions has
been mainly focused, so far, on Einstein\textendash matter systems
which admit non-trivial spherically symmetric dynamics (thus reducing
the problem to the more tractable setting of $1+1$ dimensional hyperbolic
systems), while still retaining many of the qualitative properties
of the vacuum equations (\ref{eq:VacuumEinsteinEquations});\footnote{As a consequence of the extension of Birkhoff's theorem to the case
$\Lambda<0$ (see \cite{Eiesland1925}), the vacuum equations (\ref{eq:VacuumEinsteinEquations})
become trivial in spherical symmetry, which is the only surface symmetry
class compatible with AdS asymptotics in $3+1$ dimensions. However,
this problem can be circumvented in $4+1$ dimensions in the biaxial
Bianchi IX symmetry class (see \cite{BizonChmajSchmidt2005}).} a prominent example of such a model is provided by the \emph{Einstein\textendash Klein-Gordon
system}
\begin{equation}
\begin{cases}
Ric_{\mu\nu}-\frac{1}{2}Rg_{\mu\nu}+\Lambda g_{\mu\nu}=8\pi T_{\mu\nu}[\varphi],\\
\square_{g}\varphi-\mu\varphi=0,\\
T_{\mu\nu}[\varphi]\doteq\partial_{\mu}\varphi\partial_{\nu}\varphi-\frac{1}{2}g_{\mu\nu}\partial^{\alpha}\varphi\partial_{\alpha}\varphi.
\end{cases}\label{eq:EinsteinScalarField}
\end{equation}
In the case when the Klein\textendash Gordon mass $\mu$ satisfies
the so-called Breitenlohner\textendash Freedman bound, well-posedness
for the initial-boundary vlaue problem for (\ref{eq:EinsteinScalarField})
in the spherically symmetric case was established, for a wide class
of boundary conditions on $\mathcal{I}$, by Holzegel\textendash Smulevici
\cite{HolzSmul2012} and Holzegel\textendash Warnick \cite{HolzWarn2015}.\footnote{See also \cite{VasyAdS} and \cite{Warnick2013} for well-posedness
results for the \emph{linear} Klein\textendash Gordon equation on
general asymptotically AdS backgrounds, without symmetry assumptions.} 

The first numerical and heuristic study in the direction of establishing
the AdS instability conjecture for (\ref{eq:EinsteinScalarField})
was carried out by Bizon\textendash Rostworowski \cite{bizon2011weakly}
in 2011. The numerical simulations of \cite{bizon2011weakly} verified
the existence of spherically symmetric initial data sets for (\ref{eq:EinsteinScalarField}),
with small initial size, the evolution of which (under Dirichlet conditions
at $\mathcal{I}$) reaches the threshold of trapped surface formation
after sufficiently long time. In addition, \cite{bizon2011weakly}
was the first to propose a mechanism driving the initial stage of
the instability: Analyzing perturbatively the interactions of different
frequency modes of the scalar function $\varphi$, \cite{bizon2011weakly}
suggested that the transfer of energy from low to high frequencies
was propelled by a hierarchy of resonant interactions. 

Following \cite{bizon2011weakly}, a vast amount of numerical and
heuristic works were dedicated to the study of the AdS intability
conjecture for the spherically symmetric system (\ref{eq:EinsteinScalarField}),
addressing, in addition, questions related to the long time dynamics
of \emph{generic} perturbations to $(\mathcal{M}_{AdS},g_{AdS})$
and the possibility of existence of ``islands of stability'' in
the moduli space of initial data for (\ref{eq:EinsteinScalarField})
close to $(\mathcal{M}_{AdS},g_{AdS})$; see, e.\,g.~\cite{DiasHorSantos,BuchelEtAl2012,DiasEtAl,MaliborskiEtAl,BalasubramanianEtAl,CrapsEtAl2014,CrapsEtAl2015,BizonMalib,DimitrakopoulosEtAl,GreenMailardLehnerLieb,HorowitzSantos,DimitrakopoulosEtAl2015,DimitrakopoulosEtAl2016,Bizon2018}.
For works moving outside the realm of $1+1$ systems, see also \cite{DiasSantos2016,Bantilan2017,Rostworowski2017}.
Most of the aforementioned works utilised a frequency space analysis
similar to the one introduced in \cite{bizon2011weakly}, with the
notable exception of \cite{DimitrakopoulosEtAl}. A more detailed
discussion on the numerics literature surrounding the AdS instability
conjecture can be found in \cite{MoschidisNullDust}.

The first rigorous proof of the instability of AdS in a spherically
symmetric setting was obtained in \cite{MoschidisNullDust}, for the
case of the \emph{Einstein\textendash null dust} system with both
ingoing and outgoing dust; this system can be formally viewed as a
high frequency limit of (\ref{eq:EinsteinScalarField}) in spherical
symmetry (see also the discussion in \cite{PoissonIsrael1990}). The
proof of \cite{MoschidisNullDust} uncovered and utilised an alternative
instability mechanism at the level of position space: Arranging the
null dust into a specific configuration of localised spherically symmetric
beams, \cite{MoschidisNullDust} showed that successive reflections
off $\mathcal{I}$ lead to the concentration of energy in the beam
lying initially to the exterior of the rest. However, in order to
circumvent a trivial break down of the Einstein\textendash null dust
system occuring once the dust reaches the center of symmetry, \cite{MoschidisNullDust}
placed an \emph{inner mirror} at a finite radius $r=r_{0}>0$ and
studied the evolution restricted to the region $r>r_{0}$. Moreover,
further considerations on the nature of the dynamics around $(\mathcal{M}_{AdS},g_{AdS})$
necessitated that the mirror radius $r_{0}$ in \cite{MoschidisNullDust}
was allowed to shrink to $0$ at a rate proportional to the size of
the initial perturbation. Well-posedness for the initial-boundary
value problem in this rather complicated setup, in an initial data
topology allowing for estimates with uniform modulus with respect
to $r_{0}$, was obtained in \cite{MoschidisMaximalDevelopment}. 

In this paper, we will prove the instability of $(\mathcal{M}_{AdS}^{3+1},g_{AdS})$
for the \emph{Einstein\textendash massless Vlasov} system in spherical
symmetry; this system is well-posed for regular initial data in the
standard sense, \emph{without the addition of an inner mirror} around
the center of symmetry. Novel difficulties associated with the Einstein\textendash massless
Vlasov system (both at a technical and at a more conceptual level)
will force us to depart from the main strategy of proof followed in
\cite{MoschidisNullDust}, and develop a new physical space configuration
of beam interactions, where the sizes of the phase-space domains corresponding
to each beam are part of a complicated hierarchy of scales. We will
now proceed to review the main result of this paper in more detail;
a discussion on the complications arising in the proof as well as
a brief comparison with the methods of \cite{MoschidisNullDust} will
then follow in Section \ref{subsec:Discussion-on-the}.

\subsection{\label{subsec:The-main-result:}The main result: AdS instability
for the spherically symmetric Einstein\textendash massless Vlasov
system}

Let $(\mathcal{M},g)$ be a $3+1$ dimensional, smooth Lorentzian
manifold and let $f$ be a non-negative measure on $T\mathcal{M}$
supported on the set of future directed null vectors. The Einstein\textendash massless
Vlasov system for $(\mathcal{M},g;f)$ takes the form 
\begin{equation}
\begin{cases}
Ric_{\mu\nu}-\frac{1}{2}Rg_{\mu\nu}+\Lambda g_{\mu\nu}=8\pi T_{\mu\nu}[f],\\
\mathcal{L}^{(g)}f=0,
\end{cases}\label{eq:EinsteinMasslessVlasovIntro}
\end{equation}
where $\mathcal{L}^{(g)}$ is the geodesic spray on $T\mathcal{M}$
and $T_{\mu\nu}[f]$ is expressed in terms of $f$ and $g$ by (\ref{eq:StressEnergy}).
In the spherically symmetric setting, there is a unique reflecting
boundary condition for (\ref{eq:EinsteinMasslessVlasovIntro}) at
conformal infinity $\mathcal{I}$; it is formulated simply as the
requirement that the Vlasov field $f$ is conserved along the \emph{reflection}
of null geodesics $\gamma$ off $\mathcal{I}$ (see Section \ref{subsec:The-Einstein-equations}).

The main result of this paper is the proof of the AdS instability
conjecture for the system (\ref{eq:EinsteinMasslessVlasovIntro})
in spherical symmetry:

\begin{customthm}{1}[rough version]\label{thm:TheTheoremIntro}

There exists a one-parameter family of \emph{smooth, spherically symmetric,
asymptotically AdS} initial data $\mathcal{D}^{(\varepsilon)}$ for
(\ref{eq:EinsteinMasslessVlasovIntro}), $\varepsilon\in(0,1]$, satisfying
the following properties:

\begin{itemize}

\item As $\varepsilon\rightarrow0$, $\mathcal{D}^{(\varepsilon)}$
converge to the trivial data $\mathcal{D}^{(0)}$ of $(\mathcal{M}_{AdS},g_{AdS};0)$
with respect to a suitable initial data norm $||\cdot||_{data}$.

\item For any $\varepsilon\in(0,1)$, the (unique) maximally extended
solution $(\mathcal{M},g;f)^{(\varepsilon)}$ of (\ref{eq:EinsteinMasslessVlasovIntro})
arising from $\mathcal{D}^{(\varepsilon)}$ with reflecting boundary
conditions on $\mathcal{I}$ contains a trapped sphere, and, hence,
a black hole region.

\end{itemize}

In particular, $(\mathcal{M}_{AdS},g_{AdS})$ is unstable as a solution
of (\ref{eq:EinsteinMasslessVlasovIntro}) under spherically symmetric
perturbations which are small with respect to $||\cdot||_{data}$.

\end{customthm}

For a more detailed statement of Theorem \ref{thm:TheTheoremIntro},
see Section \ref{sec:The-main-result}. For the definition the maximal
future development $(\mathcal{M},g;f)^{(\varepsilon)}$ of an initial
data set $\mathcal{D}^{(\varepsilon)}$ and the notion of a trapped
sphere, see Section \ref{sec:Well-posedness-and-extension}.
\begin{rem*}
The initial data norm $||\cdot||_{data}$ appearing in the statement
of Theorem \ref{thm:TheTheoremIntro} is a scale invariant norm that
has just enough regularity to provide control of the integrals of
the right hand sides of the constraint equations (\ref{eq:ConstraintUFinal})\textendash (\ref{eq:ConstraintVFinal})
in the evolution; a precise definition of the norm is given in Section
\ref{subsec:Well-posedness} (see Definition \ref{def:InitialDataNorm}).
The necessary well-posedness results for the initial-boundary value
problem for (\ref{eq:EinsteinMasslessVlasovIntro}), as well as the
crucial Cauchy stability statement for $(\mathcal{M}_{AdS},g_{AdS})$
in the topology defined by $||\cdot||_{data}$ (see the remark below
the statement of the AdS instability conjecture), are obtained in
our companion paper \cite{MoschidisVlasovWellPosedness} and are also
reviewed in Section \ref{sec:Well-posedness-and-extension}.

We should point out that, switching to the case $\Lambda=0$, Minkowski
spacetime $(\mathbb{R}^{3+1},\eta)$ is \emph{non-linearly stable}
as a solution of (\ref{eq:EinsteinMasslessVlasovIntro}) under spherically
symmetric perturbations which are initially small with respect to
the norm $||\cdot||_{data}$ (suitably modified in the region $r\gg1$
to accommodate for the change in the value of $\Lambda$). This result,
which can be viewed as a straightforward corollary of our method of
proof of Cauchy stability for $(\mathcal{M}_{AdS},g_{AdS})$ and is
discussed in more detail in Section 6 of \cite{MoschidisVlasovWellPosedness},
provides further justification for the use of the initial data norm
$||\cdot||_{data}$ in the study of the AdS instability conjecture.\footnote{The non-linear stability of $(\mathbb{R}^{3+1},\eta)$ as a solution
of the Einstein\textendash massless Vlasov system (\ref{eq:EinsteinMasslessVlasovIntro})
\emph{without any symmetry assumptions} was shown by Taylor \cite{Taylor2017},
with respect to initial perturbations which are small in a higher
order, weighted Sobolev space. Our argument for obtaining a global
stability statement for $(\mathbb{R}^{3+1},\eta)$ as a corollary
of a Cauchy stability statement is in fact analogous (albeit much
simpler) to the strategy implemented in \cite{Taylor2017}, where
global stability is also inferred as a corollary of a quantitative,
semi-global Cauchy stability statement (see \cite{Taylor2017} for
more details).}
\end{rem*}

\subsection{\label{subsec:Discussion-on-the}Sketch of the proof and further
discussion}

In this section, we will briefly sketch the proof of Theorem \ref{thm:TheTheoremIntro},
highlighting the main technical complications and obstacles shaping
our strategy. We will then comment on the relation between the proof
of Theorem \ref{thm:TheTheoremIntro} and the ideas appearing in \cite{MoschidisNullDust}.

The proof of Theorem \ref{thm:TheTheoremIntro} is carried out in
double null coordinates $(u,v,\theta,\varphi)$, in which a general
spherically symmetric metric $g$ takes the form 
\begin{equation}
g=-\Omega^{2}(u,v)dudv+r^{2}(u,v)g_{\mathbb{S}^{2}}\label{eq:SphericallySymmetricMetric-1}
\end{equation}
(see Section \ref{subsec:Spherically-symmetric-spacetimes}). The
initial data family $\mathcal{D}^{(\varepsilon)}$ in the statement
of Theorem \ref{thm:TheTheoremIntro} is then constructed as a family
of \emph{characteristic} smooth initial data prescribed at $u=0$;
the necessary well-posedness results for the characteristic initial-boundary
value problem in this setting are established in our companion paper
\cite{MoschidisVlasovWellPosedness} and are also reviewed in Section
\ref{sec:Well-posedness-and-extension}.

\begin{figure}[h] 
\centering 
\scriptsize
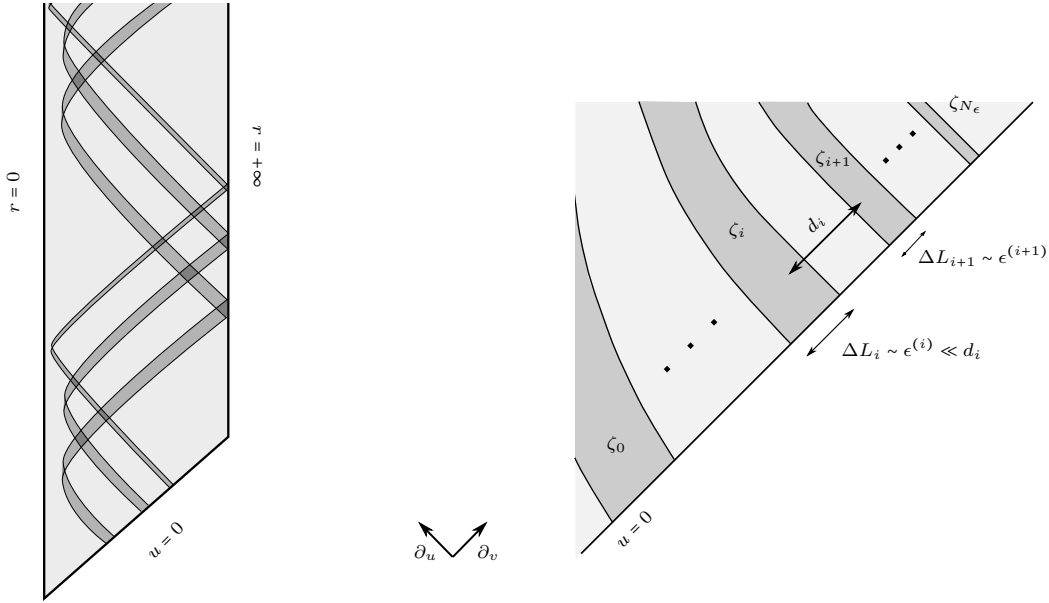 
\caption{The initial data family $\mathcal{D}^{(\epsilon)}$ gives rise to a large number $\mathcal{N}_{\epsilon}$ of spherically symmetric Vlasov beams, which are initially ingoing. The left part of the figure provides a schematic depiction of just three of these beams, projected onto the $(u,v)$-plane. Each successive beam is increasingly narrower compared to the previous ones in the configuration (as shown schematically on the right), and contains geodesics of increasingly smaller angular momenta. \label{fig:BeamsIntro}}
\end{figure}

The family $\mathcal{D}^{(\varepsilon)}$ is constructed so that the
physical space support of the corresponding Vlasov field $f_{\varepsilon}$
is initially separated into a large number $N_{\varepsilon}\gg1$
of narrow ingoing beams (see Figure \ref{fig:BeamsIntro}), organised
in terms of a particular \emph{multi-scale hierarchy}, which we will
now describe: Denoting with $\zeta_{i}$ the $i$-th Vlasov beam (with
$i$ increasing with the initial distance of the beam from $r=0$),
the configuration of beams is set up so that, at $u=0$, $\zeta_{i}$
has physical space width $\text{\textgreek{D}}L_{i}$ satisfying
\begin{equation}
\text{\textgreek{D}}L_{i}\sim\varepsilon^{(i)}(-\Lambda)^{-\frac{1}{2}},\label{eq:Length}
\end{equation}
 where the hierarchy of small parameters $\{\varepsilon^{(i)}\}_{i=0}^{N_{\varepsilon}}$
(each given by an explicit formula in terms of $\varepsilon$ and
$i$) satisfies
\begin{equation}
\varepsilon^{(i+1)}\ll\varepsilon^{(i)}\text{ and }\varepsilon^{(i)}\xrightarrow{\varepsilon\rightarrow0}0\text{ for all }i=0,\ldots,N_{\varepsilon}-1.
\end{equation}
The initial separation $d_{i}$ (with respect to the $v$ coordinate)
between the beams $\zeta_{i}$ and $\zeta_{i+1}$ is chosen to satisfy
\begin{equation}
\text{\textgreek{D}}L_{i}\ll d_{i}\ll\text{\textgreek{D}}L_{i-1}.\label{eq:InitialSeparation}
\end{equation}

The \emph{energy content} $\mathcal{E}_{i}$ of the beam $\zeta_{i}$
at $u=0$ is defined as the difference of the renormalised Hawking
mass 
\[
\tilde{m}\doteq\frac{r}{2}\Big(1-4\Omega^{-2}\partial_{u}r\partial_{v}r\Big)-\frac{1}{6}\Lambda r^{3}
\]
between the inner and outer boundary of $\zeta_{i}$ at $u=0$, i.\,e.:
\[
\mathcal{E}_{i}\doteq\tilde{m}(0,v_{i}^{+})-\tilde{m}(0,v_{i}^{-})
\]
(where $\zeta_{i}\cap\{u=0\}=\{0\}\times[v_{i}^{-},v_{i}^{+}]$ in
the $(u,v)$-plane). The Vlasov field $f_{\varepsilon}$ is chosen
so that $\mathcal{E}_{i}$ satisfies
\begin{equation}
\mathcal{E}_{i}\sim a_{i}\text{\textgreek{D}}L_{i},\label{eq:Energy}
\end{equation}
where the parameters $0\le a_{i}\ll1$ are only fixed at a later stage
in the proof (we will come back to this point later in the discussion). 
\begin{rem*}
\noindent In order to ensure the condition 
\begin{equation}
||\mathcal{D}^{(\varepsilon)}||_{data}\xrightarrow{\varepsilon\rightarrow0}0\label{eq:InitialDataGoingToZeroIntro}
\end{equation}
in the statement of Theorem \ref{thm:TheTheoremIntro}, an additional
smallness condition needs to be imposed on $\sum_{i=0}^{N_{\varepsilon}}a_{i}$;
we refer the reader to the detailed construction of the initial data
family in Section \ref{subsec:The-initial-data-family}.
\end{rem*}
Regarding the momentum space conditions imposed on the beam configuration,
the beam $\zeta_{i}$ is chosen to consist only of null geodesics
$\gamma$ with normalised angular momentum $\frac{l}{E_{0}}$ satisfying:\footnote{Here, we define the normalised angular momentum of a geodesic $\gamma$
in a spherically symmetric spacetime simply as the ratio $\frac{l}{E_{0}}$
between the usual angular momentum $l$ of $\gamma$ and its initial
energy $E_{0}=g(\partial_{u}+\partial_{v},\dot{\gamma})|_{u=0}$ (with
respect to the timelike coordinate vector field $\partial_{u}+\partial_{v}$)
at $u=0$; this ratio is independent of the choice of affine parametrisation
of $\gamma$.}
\begin{equation}
\frac{l}{E_{0}}\sim\varepsilon^{(i)}(-\Lambda)^{-\frac{1}{2}}.\label{eq:AngularMomentaAssumption}
\end{equation}
As a result, the geodesics in the support of the Vlasov beam $f_{\varepsilon}$
are nearly radial when $\varepsilon\ll1$.
\begin{rem*}
\noindent For a null geodesic $\gamma$ in AdS spacetime $(\mathcal{M}_{AdS},g_{AdS})$,
the normalised angular momentum $\frac{l}{E_{0}}$ determines the
minimum value of $r$ along $\gamma$, with geodesics having smaller
normalised angular momentum approaching closer to the center $r=0$;
in the case when $\frac{l}{E_{0}}\ll(-\Lambda)^{-\frac{1}{2}}$, the
following approximate relation holds on $(\mathcal{M}_{AdS},g_{AdS})$:
\begin{equation}
\min_{\gamma}r\sim\frac{l}{E_{0}}.\label{eq:PoximityGeodesicCenter}
\end{equation}
\end{rem*}
In the maximal future development $(\mathcal{M}_{\varepsilon};f_{\varepsilon})$
of $\mathcal{D}^{(\varepsilon)}$, the Vlasov beams $\zeta_{i}$ are
reflected off $\mathcal{I}$ multiple times (see Figure \ref{fig:BeamsIntro});
between any two successive reflections, the Vlasov beams $\zeta_{i}$
exchange energy through their non-linear interactions, possibly at
a loss of coherence. The proof of Theorem \ref{thm:TheTheoremIntro}
will consist of showing that, after a large number of reflections
off $\mathcal{I}$, the non-linear interactions lead to the concentration
of sufficient energy at the top beam $\zeta_{N_{\varepsilon}}$ so
that a trapped surface can form as $\zeta_{N_{\varepsilon}}$ approaches
the center $r=0$ for the last time. 

Controlling the coherence of the Vlasov beams for sufficiently long
time (ensuring, in particular, that the qualitative picture of the
configuration is similar to the one depicted in Figure \ref{fig:BeamsIntro})
will consititue a major technical challenge of the proof, but the
relevant details will be only briefly sketched in this discussion.
The fact that, for $\Lambda=0$, the system (\ref{eq:EinsteinMasslessVlasovIntro})
admits static solutions with $(\mathcal{M}_{st},g_{st};f_{st})$
\[
C_{0}\le\sup_{\mathcal{M}_{st}}\frac{2\tilde{m}}{r}\le C_{1}<1
\]
 (see \cite{Andreasson2017}) shows that, in general, when $\frac{2\tilde{m}}{r}$
exceeds a certain threshold in the evolution, the configuration of
Vlasov beams can not be expected to behave in a qualitatively similar
fashion as on $(\mathcal{M}_{AdS},g_{AdS})$ (i.\,e.~approach $r=0$
only for a brief period of time separating an ingoing and an outgoing
phase, like the beams depicted in Figure \ref{fig:BeamsIntro}). The
additional flexibility provided by the freedom in the choice of the
parameters $\varepsilon^{(i)}$ in the multi-scale hierarchy of parameters
(\ref{eq:Length})\textendash (\ref{eq:Energy}) will be crucial for
circumventing this obstacle. 

The evolution of $\mathcal{D}^{(\varepsilon)}$ will be studied in
two steps:

\begin{enumerate}

\item In the first step, we will show that a scale invariant norm
measuring the concentration of energy of $(\mathcal{M},g;f)^{(\varepsilon)}$
grows in time at a specific rate, driven by an instability mechanism
based on the interactions of the beams similar to the one implemented
in \cite{MoschidisNullDust}. Provided the initial parameters $a_{i}$
in (\ref{eq:Energy}) are chosen appropriately, we will show that
the beam interactions lead to the formation of a specific, predetermined
profile $\mathcal{S}_{*}=(\Omega^{2},r;f)^{(\varepsilon)}|_{u=u_{*}}$
at a late enough retarded time $u=u_{*}(\varepsilon)\gg1$. The freedom
in the choice of the parameters $\varepsilon_{i}$ and a collection
of robust estimates on the exchange of energy between the beams will
enable us to control a priori for $0\le u\le u_{*}$
\begin{equation}
\frac{2\tilde{m}}{r}\le\delta_{*}\label{eq:SmallnessTrappingIntro}
\end{equation}
 where $\delta_{*}\ll1$ is fixed; the bound (\ref{eq:SmallnessTrappingIntro})
will be crucial for controlling the paths of geodesics in the support
of $f^{(\varepsilon)}$ for $u\in[0,u_{*}]$. 

\item In a second step, we will show that the specific features of
$\mathcal{S}_{*}$ (inherited by the properties of the multi-scale
hierarchy (\ref{eq:Length})\textendash (\ref{eq:Energy})) imply
that a trapped surface necessarily forms along $\zeta_{N_{\varepsilon}}\cap\{u=u_{\dagger}\}$
for some $u_{*}<u_{\dagger}\le u_{*}+O(1)$, i.\,e.~that 
\begin{equation}
\sup_{\zeta_{N_{\varepsilon}}\cap\{u=u_{\dagger}\}}\frac{2\tilde{m}}{r}>1.
\end{equation}
It will thus follow that $(\mathcal{M},g;f^{(\varepsilon)})$ contains
a black hole region. 

\end{enumerate}

We will now proceed to discuss the above steps in more detail.

\subsubsection{\label{subsec:First-stage-of}First stage of the instability: Growth
of the scale invariant norm and formation of the intermediate profile}

The first step in the proof of Theorem \ref{thm:TheTheoremIntro}
will consist of showing that the interactions of the beams $\zeta_{i}$
lead to a gradual increase in the energy content of all beams $\zeta_{j}$
with $j\ge1$ (implying the concentration of energy at finer scales,
in view of the length-scale hierarchy (\ref{eq:Length})). In particular,
our aim at this stage would be to show that, for some $u_{*}\gg1$
determined in terms of $\varepsilon$, the energy content $\mathcal{E}_{j}^{*}$
of $\zeta_{j}\cap\{u=u_{*}\}$ satisfies the relation
\begin{equation}
\frac{2\mathcal{E}_{j}^{*}}{d_{j}^{*}}\sim\frac{C}{N_{\varepsilon}}\exp\big(-2\frac{C}{N_{\varepsilon}}j\big)\text{ for all }1\le j\le N_{\varepsilon}-1\label{eq:AllButFinalBeamProfileIntro}
\end{equation}
and:
\begin{equation}
\mathcal{E}_{N_{\varepsilon}}^{*}\sim\frac{\varepsilon^{(N_{\varepsilon})}}{\sqrt{-\Lambda}},\label{eq:FinalBeamProfileIntro}
\end{equation}
where $d_{j}^{*}$ is the distance between $\zeta_{j}\cap\{u=u_{*}\}$
and $\zeta_{j+1}\cap\{u=u_{*}\}$. For simplicity, for the rest of
this discussion, we will refer to the initial data induced on $\{u=u_{*}\}$
by the solution $(\mathcal{M},g;f)^{(\varepsilon)}$ simply as the
\emph{intermediate profile} $\mathcal{S}_{*}$. This step will in
fact occupy the bulk of the proof of Theorem \ref{thm:TheTheoremIntro}.

\medskip{}

\noindent \emph{Estimates for the geodesic flow.} Obtaining estimates
for the null geodesic flow on $(\mathcal{M},g;f)^{(\varepsilon)}$
for sufficiently long times is a prerequisite for studying the interactions
of the beams $\zeta_{i}$. More precisely, we would like to show that,
at least until the formation of the intermediate profile $\mathcal{S}_{*}$,
null geodesics in $(\mathcal{M},g;f)^{(\varepsilon)}$ follow trajectories
which are similar (in a certain sense) to the trajectories of null
geodesics on $(\mathcal{M}_{AdS},g_{AdS})$. 

The only quantitative bound assumed on the initial data family is
a smallness condition in terms of the low-regularity norm $||\cdot||_{data}$.
The well-posedness estimates established in our companion paper \cite{MoschidisVlasovWellPosedness}
imply that, for any fixed $\bar{U}>0$, the following scale-invariant
bound holds for $(\mathcal{M},g;f)^{(\varepsilon)}$ as $\varepsilon\rightarrow0$:
\begin{align}
||(\mathcal{M},g;f)^{(\varepsilon)}||_{u\le\bar{U}}\doteq\sup_{\bar{u}\le\bar{U}}\int_{\{u=\bar{u}\}}r\Big(\frac{T_{vv}[f]}{\partial_{v}r}+\frac{T_{uv}[f]}{-\partial_{u}r}\Big) & (\bar{u},v)\,dv+\label{eq:SmallnessRightHandSideConstraintsIntro}\\
+\sup_{\bar{v}}\int_{\{v=\bar{v}\}\cap\{u\le\bar{U}\}}r & \Big(\frac{T_{uv}[f]}{\partial_{v}r}+\frac{T_{uu}[f]}{-\partial_{u}r}\Big)(u,\bar{v})\,du\lesssim||\mathcal{D}^{(\varepsilon)}||_{data},\nonumber 
\end{align}
where $T_{\mu\nu}[f]$ are the components of the energy momentum tensor
of $f^{(\varepsilon)}$. However, for values of $\bar{U}$ closer
to $u_{*}$, we will only be able to estimate
\begin{equation}
||(\mathcal{M},g;f)^{(\varepsilon)}||_{u\le\bar{U}}\le C\label{eq:SmallnessRightHandSideConstraintsIntro-1}
\end{equation}
for some $C\gg1$. Therefore, it will be necessary for us to obtain
sufficient control on the phase space trajectories of null geodesics
$\gamma:[0,a)\rightarrow(\mathcal{M},g)^{(\varepsilon)}$ merely under
the (rather weak) assumption that (\ref{eq:SmallnessRightHandSideConstraintsIntro-1})
and the a priori estimate (\ref{eq:SmallnessTrappingIntro}) hold.
To this end, we will rely crucially on a reformulation of the equations
of motion for null geodesics, making use of the fact that $(\mathcal{M},g;f)^{(\varepsilon)}$
satisfies (\ref{eq:EinsteinMasslessVlasovIntro}), yielding identities
such as the following:
\begin{align}
\log\big(\Omega^{2}\dot{\gamma}^{u}\big)(s)-\log\big(\Omega^{2}\dot{\gamma}^{u}\big)(0)= & \int_{v(\gamma(0))}^{v(\gamma(s))}\int_{u_{1}(v)}^{u(\gamma(s_{v}))}\Big(\frac{1}{2}\frac{\frac{6\tilde{m}}{r}-1}{r^{2}}\Omega^{2}-24\pi T_{uv}[f]\Big)\,dudv+\label{eq:UsefulRelationForGeodesicWithMu-U-1}\\
 & +\int_{v(\gamma(0))}^{v(\gamma(s))}\big(\partial_{v}\log(\Omega^{2})-2\frac{\partial_{v}r}{r}\big)(0,v)\,dv\nonumber 
\end{align}
(assuming that the initial point $\gamma(0)$ of $\gamma$ belongs
to $\{u=0\}$; see Figure \ref{fig:Geodesic_Integration_Intro}).\footnote{Note that it is already apparent in (\ref{eq:UsefulRelationForGeodesicWithMu-U-1})
that the a priori estimate (\ref{eq:SmallnessTrappingIntro}) for
$\frac{2\tilde{m}}{r}$ is important for controlling $\gamma$; when
$\frac{2\tilde{m}}{r}\ge\frac{1}{3}$, the bulk term in (\ref{eq:UsefulRelationForGeodesicWithMu-U-1})
no longer has a definite sign.} We refer the reader to Section \ref{sec:GeodesicPathsAndDifferenceEstimates}
for more details; for the rest of this discussion, we will suppress
any technical issues related to the precise estimates on the geodesic
flow on $(\mathcal{M},g)^{(\varepsilon)}$.

\begin{figure}[t] 
\centering 
\scriptsize
\begingroup%
  \makeatletter%
  \providecommand\color[2][]{%
    \errmessage{(Inkscape) Color is used for the text in Inkscape, but the package 'color.sty' is not loaded}%
    \renewcommand\color[2][]{}%
  }%
  \providecommand\transparent[1]{%
    \errmessage{(Inkscape) Transparency is used (non-zero) for the text in Inkscape, but the package 'transparent.sty' is not loaded}%
    \renewcommand\transparent[1]{}%
  }%
  \providecommand\rotatebox[2]{#2}%
  \newcommand*\fsize{\dimexpr\f@size pt\relax}%
  \newcommand*\lineheight[1]{\fontsize{\fsize}{#1\fsize}\selectfont}%
  \ifx\svgwidth\undefined%
    \setlength{\unitlength}{150bp}%
    \ifx\svgscale\undefined%
      \relax%
    \else%
      \setlength{\unitlength}{\unitlength * \real{\svgscale}}%
    \fi%
  \else%
    \setlength{\unitlength}{\svgwidth}%
  \fi%
  \global\let\svgwidth\undefined%
  \global\let\svgscale\undefined%
  \makeatother%
  \begin{picture}(1,1.2)%
    \lineheight{1}%
    \setlength\tabcolsep{0pt}%
    \put(0,0){\includegraphics[width=\unitlength,page=1]{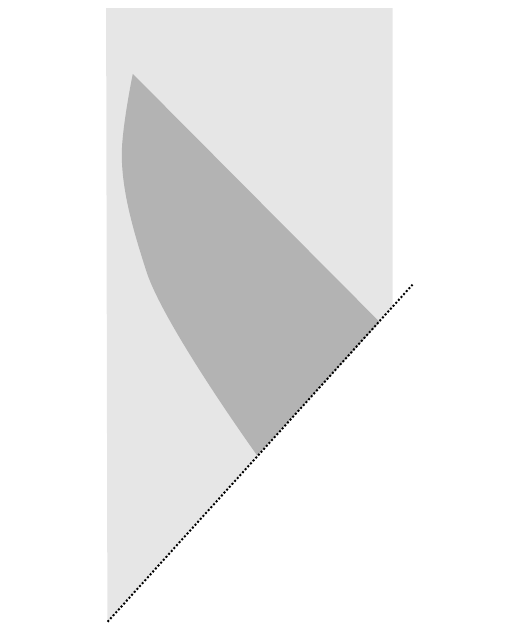}}%
    \put(0.30357129,0.02530902){\color[rgb]{0,0,0}\rotatebox{45}{\makebox(0,0)[lt]{\lineheight{1.25}\smash{\begin{tabular}[t]{l}$u=0$\end{tabular}}}}}%
    \put(0,0){\includegraphics[width=\unitlength,page=2]{Geodesic_Integration_Intro.pdf}}%
    \put(0.24133665,0.57796987){\color[rgb]{0,0,0}\makebox(0,0)[lt]{\lineheight{1.25}\smash{\begin{tabular}[t]{l}$\gamma$\end{tabular}}}}%
    \put(0.29207921,1.05940576){\color[rgb]{0,0,0}\makebox(0,0)[lt]{\lineheight{1.25}\smash{\begin{tabular}[t]{l}$\gamma (s)$\end{tabular}}}}%
    \put(0.33787126,0.30816814){\color[rgb]{0,0,0}\makebox(0,0)[lt]{\lineheight{1.25}\smash{\begin{tabular}[t]{l}$\gamma (0)$\end{tabular}}}}%
    \put(0.80791781,0.54451535){\color[rgb]{0,0,0}\rotatebox{-45}{\makebox(0,0)[lt]{\lineheight{1.25}\smash{\begin{tabular}[t]{l}$v=v(\gamma (s))$\end{tabular}}}}}%
    \put(0.55544255,0.29822839){\color[rgb]{0,0,0}\rotatebox{-45}{\makebox(0,0)[lt]{\lineheight{1.25}\smash{\begin{tabular}[t]{l}$v=v(\gamma (0))$\end{tabular}}}}}%
  \end{picture}%
\endgroup%
 
\caption{Schematic depiction of the domain of integration appearing in the right hand side of (\ref{eq:UsefulRelationForGeodesicWithMu-U-1}) for a null geodesic $\gamma$. \label{fig:Geodesic_Integration_Intro}}
\end{figure}

\medskip{}

\noindent \emph{Beam interactions and energy concentration.} Let us
now proceed to consider a pair of beams $\zeta_{i}$ $\zeta_{j}$,
with 
\[
0\le i<j\le N_{\varepsilon}
\]
(the fact that $i$ is smaller than $j$, i.\,e.~that $\zeta_{i}$
initially lies in the interior of $\zeta_{j}$, will be crucial for
this part of this discussion). The beams $\zeta_{i}$ and $\zeta_{j}$
will be successively reflected off $\mathcal{I}$ multiple times in
the time interval $u\in[0,u_{*}]$, intersecting each other twice
between each successive pair of reflections, in a pattern as depicted
in Figure \ref{fig:BeamsIntroPair}. In particular, assuming that
the geodesic flow on $(\mathcal{M},g)^{(\varepsilon)}$ behaves in
a similar fashion as on AdS spacetime $(\mathcal{M}_{AdS},g_{AdS})$,
the condition (\ref{eq:InitialSeparation}) on the initial separation
of the beams implies that (with notations as in Figure \ref{fig:BeamsIntroPair}):
\begin{equation}
\varepsilon^{(i)}(-\Lambda)^{-\frac{1}{2}}\ll\sup_{\mathcal{R}_{0}}r\sim d_{i}\label{eq:RelationForRInIntersectionRegionsNear}
\end{equation}
and 
\begin{equation}
\inf_{\mathcal{R}_{\infty}}r\sim\frac{1}{\varepsilon^{(i)}}(-\Lambda)^{-\frac{1}{2}}\gg\sup_{\mathcal{R}_{0}}r.\label{eq:ComparisonIntersectionRegion}
\end{equation}

\begin{figure}[h] 
\centering 
\scriptsize
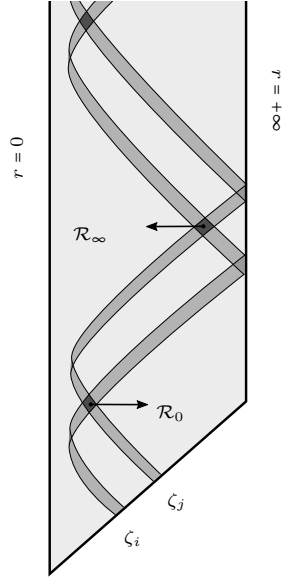 
\caption{Any two beams $\zeta_i$ and $\zeta_j$, $0\le i< j\le N_{\epsilon}$, will intersect twice between each successive pair of reflections off conformal infinity; here, $\mathcal{R}_{0}$ denotes the intersection region closer to the axis, while $\mathcal{R}_{\infty}$ denotes the intersection region closer to conformal infinity. Note that, since $i<j$, the beam $\zeta_i$ lies initially in the interior of $\zeta_j$. As a result, $\zeta_i$ is outgoing at the first intersection $\mathcal{R}_0$, while $\zeta_j$ is ingoing. \label{fig:BeamsIntroPair}}
\end{figure}

In view of the condition (\ref{eq:AngularMomentaAssumption}) on the
angular momenta of the geodesics in the beams $\zeta_{i}$ and $\zeta_{j}$,
the relations (\ref{eq:RelationForRInIntersectionRegionsNear}) and
(\ref{eq:ComparisonIntersectionRegion}) imply that, on the intersection
regions $\mathcal{R}_{0}$ and $\mathcal{R}_{\infty}$, the geodesics
of $\zeta_{i}$ and $\zeta_{j}$ can be essentially viewed as purely
radial (since their angular momentum is negligible compared to sphere
radius $r$ in these regions). Therefore, it is reasonable to expect
that the exchange of energy occuring between $\zeta_{i}$ and $\zeta_{j}$
is governed by the same mechanism as for beams of \emph{null dust,}
evolving according to the Einstein\textendash null dust system; this
is the mechanism employed in \cite{MoschidisNullDust}. 

According to \cite{MoschidisNullDust}, when a localised, spherically
symmetric and \emph{ingoing} null-dust beam $\bar{\zeta}$ intersects
a similar \emph{outgoing} beam $\zeta$ over a region $\mathcal{R}$
(see Figure \ref{fig:Intersection_Intro}), the energy contents $\mathcal{E}[\bar{\zeta}]$,
$\mathcal{E}[\zeta]$ of $\bar{\zeta},\zeta$, respectively, right
before and right after the interaction are related by the following
approximate formulas (assuming that $\mathcal{E}[\zeta],\mathcal{E}[\bar{\zeta}]\ll r|_{\mathcal{\mathcal{R}}}$
and that (\ref{eq:SmallnessTrappingIntro}) holds): 
\begin{equation}
\mathcal{E}_{+}[\bar{\zeta}]=\mathcal{E}_{-}[\bar{\zeta}]\cdot\exp\Big(\frac{2\mathcal{E}_{-}[\zeta]}{r|_{\mathcal{R}}}+\mathfrak{Err}\Big)\label{eq:IncreaseEnergyIngoingIntro}
\end{equation}
and 
\begin{equation}
\mathcal{E}_{+}[\zeta]=\mathcal{E}_{-}[\zeta]\cdot\exp\Big(-\frac{2\mathcal{E}_{-}[\bar{\zeta}]}{r|_{\mathcal{R}}}+\mathfrak{Err}\Big),\label{eq:DecreaseEnergyOutgoingIntro}
\end{equation}
where $\mathcal{E}_{-}$ and $\mathcal{E}_{+}$ denote the energy
contents of the beams before and after the interaction, respectively,
defined by the difference in the values of the renormalised Hawking
mass $\tilde{m}$ at the two vacuum regions bounding each beam before
and right after the intersection (see Figure \ref{fig:Intersection_Intro});
for the purpose of this discussion, we will assume that the error
terms $\mathfrak{Err}$ in (\ref{eq:IncreaseEnergyIngoingIntro})\textendash (\ref{eq:DecreaseEnergyOutgoingIntro})
are negligible and can be ignored. The formula (\ref{eq:IncreaseEnergyIngoingIntro})
can be deduced by tracking the change in the mass difference around
$\bar{\zeta}$ through the relation 
\begin{equation}
\partial_{u}\partial_{v}\tilde{m}=2\pi\partial_{u}\Big(\frac{1-\frac{2m}{r}}{\partial_{v}r}r^{2}T_{vv}[f|_{\bar{\zeta}}]\Big)\label{eq:MassIntro}
\end{equation}
(see the relation (6.57) in \cite{MoschidisNullDust}) using the following
facts:
\begin{enumerate}
\item The change in $\frac{1-\frac{2m}{r}}{\partial_{v}r}$ is determined
in terms of $\zeta$ by the constraint equation (see (\ref{eq:EksiswshGiaK}))
\begin{equation}
\partial_{u}\log\Big(\frac{\partial_{v}r}{1-\frac{2m}{r}}\Big)=-\frac{4\pi}{r}\frac{r^{2}T_{uu}[f]}{-\partial_{u}r}.\label{eq:ConstraintINtro}
\end{equation}
\item In the null dust limit, the quantity $r^{2}T_{vv}[f|_{\bar{\zeta}}]$
is constant in $u$ as a consequence of the conservation of energy
relation, i.\,e.:
\begin{equation}
\partial_{u}\big(r^{2}T_{vv}[f|_{\bar{\zeta}}]\big)=0\label{eq:ConservationIntro}
\end{equation}
(see the relation (2.38) in \cite{MoschidisNullDust}).
\end{enumerate}
The formula (\ref{eq:DecreaseEnergyOutgoingIntro}) is obtained by
following the same procedure for $\zeta$ with the roles of $u$ and
$v$ inverted (resulting in a change of sign in (\ref{eq:ConstraintINtro})).
Note that the (\ref{eq:IncreaseEnergyIngoingIntro})\textendash (\ref{eq:DecreaseEnergyOutgoingIntro})
imply that \emph{the energy of the ingoing beam increases, while that
of the outgoing beam decreases}.

\begin{figure}[h] 
\centering 
\scriptsize
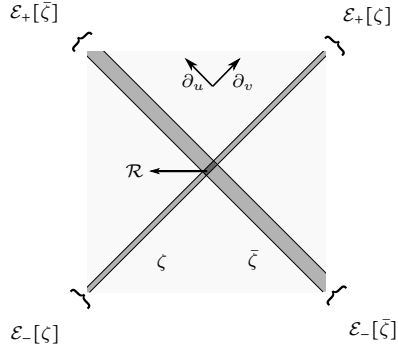 
\caption{Schematic depiction of a pair $\zeta$, $\bar{\zeta}$ of intersecting Vlasov beams supported on nearly radial null geodesics. Due to the non-linear interaction of the beams, the energy $\mathcal{E}[\bar{\zeta}]$ of the ingoing beam $\bar{\zeta}$ increases, while the energy $\mathcal{E}[\zeta]$ of $\zeta$ (the total energy being conserved during the interaction).  \label{fig:Intersection_Intro}}
\end{figure}

In this paper, we show that the formulas (\ref{eq:IncreaseEnergyIngoingIntro})\textendash (\ref{eq:DecreaseEnergyOutgoingIntro})
also hold in the case when we are dealing with solutions of the system
(\ref{eq:EinsteinMasslessVlasovIntro}) instead of the Einstein\textendash null
dust system, under the condition the beams $\zeta$, $\bar{\zeta}$
consist of null geodesics which are nearly radial at their intersection
region $\mathcal{R}$. In this case, the additional error terms appearing
in the analogues of the relations (\ref{eq:MassIntro}) and (\ref{eq:ConservationIntro})
(see (\ref{eq:ODEforEnergyFlux}) and (\ref{eq:ConservationOfEnergyDu}),
respectively) can be eventually controlled; we should point out, however,
that the error terms appearing in this case in (\ref{eq:ConservationIntro})
are of \emph{higher order} in terms of derivatives of the metric,
and are not controlled by the norm $||\cdot||$; estimating their
size requires a novel set of higher order bounds and precise control
on the size of the interaction region in terms of the hierarchy $\varepsilon^{(i)}$.
We will suppress this technical issue here (for more details, see
Section \ref{sec:The-technical-core}). 

Let us now return to the interaction of the beams $\zeta_{i}$ and
$\zeta_{j}$ in Figure \ref{fig:BeamsIntroPair}. Note that, for the
interaction taking place in the region $\mathcal{R}_{0}$, $\zeta_{i}$
has the role of the outgoing beam $\zeta$ in (\ref{eq:IncreaseEnergyIngoingIntro})\textendash (\ref{eq:DecreaseEnergyOutgoingIntro}),
while $\zeta_{j}$ is the ingoing beam $\bar{\zeta}$; in the case
of $\mathcal{R}_{\infty}$, these roles are inverted. 
\begin{rem*}
Note here the asymmetry between $\zeta_{i}$ and $\zeta_{j}$: For
this part of the discussion, it is important that $i<j$, which is
the order convention fixing $\zeta_{i}$ to be the outgoing beam in
the region $\mathcal{R}_{0}$ closest to the axis.
\end{rem*}
In view of the relation (\ref{eq:ComparisonIntersectionRegion}) between
$r|_{\mathcal{R}_{0}}$ and $r|_{\mathcal{R}_{\infty}}$, the formulas
(\ref{eq:IncreaseEnergyIngoingIntro})\textendash (\ref{eq:DecreaseEnergyOutgoingIntro})
imply that the loss of energy occuring for the beam $\zeta_{j}$ at
$\mathcal{R}_{\infty}$ is negligible compared to the gain of energy
for the same beam occuring earlier in the region $\mathcal{R}_{0}$;
the opposite is true for $\zeta_{i}$. As a result, the net contribution
for the energy $\mathcal{E}[\zeta_{j}]$ of $\zeta_{j}$ after the
pair of interactions with the beam $\zeta_{i}$ between two successive
reflections off $\mathcal{I}$ is strictly positive, i.\,e.~$\mathcal{E}[\zeta_{j}]$
strictly increases, with the total energy gain estimated as follows:
\begin{equation}
\mathcal{E}_{after}[\zeta_{j}]=\mathcal{E}_{before}[\zeta_{j}]\cdot\exp\Big(-\frac{2\mathcal{E}_{before}[\zeta_{i}]}{\sup_{\mathcal{R}_{0}}r}+\mathfrak{Err}\Big).\label{eq:INcreaseIntro}
\end{equation}

On the other hand, the energy of $\zeta_{i}$ strictly decreases as
a result of this interaction:
\begin{equation}
\mathcal{E}_{after}[\zeta_{i}]=\mathcal{E}_{before}[\zeta_{i}]\cdot\exp\Big(-\frac{2\mathcal{E}_{before}[\zeta_{j}]}{r|_{\mathcal{R}_{0}}}+\frac{2\mathcal{E}_{before}[\zeta_{j}]}{r|_{\mathcal{R}_{\infty}}}+\mathfrak{Err}\Big)<\mathcal{E}_{before}[\zeta_{i}].\label{eq:DecreaseIntro}
\end{equation}
However, using the fact that $\varepsilon^{(i)}\gg\varepsilon^{(j)}$
and assuming that the hierarchy (\ref{eq:Energy}) on the initial
energy content of the beams is propagated in the evolution, from (\ref{eq:RelationForRInIntersectionRegionsNear})\textendash (\ref{eq:ComparisonIntersectionRegion})
we deduce that 
\begin{equation}
\frac{2\mathcal{E}_{before}[\zeta_{j}]}{r|_{\mathcal{R}_{0}}},\frac{2\mathcal{E}_{before}[\zeta_{j}]}{r|_{\mathcal{R}_{\infty}}}\lesssim\frac{\varepsilon^{(j)}}{\varepsilon^{(i)}}\ll1.\label{eq:BoundSmallnessIntro}
\end{equation}
Substituting (\ref{eq:BoundSmallnessIntro}) in the formula (\ref{eq:DecreaseIntro}),
we infer that the energy of $\zeta_{i}$ remains essentially unaffected
by its interaction with $\zeta_{j}$, i.\,e.
\begin{equation}
\mathcal{E}_{after}[\zeta_{i}]-\mathcal{E}_{before}[\zeta_{i}]=\mathfrak{Err}\ll\mathcal{E}_{before}[\zeta_{i}].
\end{equation}

To sum up, for any beam $\zeta_{i_{0}}$, $0\le i_{0}\le N_{\varepsilon}$,
the energy content of $\zeta_{i_{0}}$ between two successive reflections
off $\mathcal{I}$ changes according to the following rules:
\begin{enumerate}
\item The interaction of $\zeta_{i_{0}}$ with any beam $\zeta_{i_{1}}$
with $i_{1}<i_{0}$ results in an increase in the energy of $\zeta_{0}$,
quantified by (\ref{eq:INcreaseIntro}).
\item The interaction of $\zeta_{i_{0}}$ with any beam $\zeta_{i_{1}}$
with $i_{1}>i_{0}$ has virtually no effect on the energy content
of $\zeta_{0}$.
\item The energy content of $\zeta_{i_{0}}$ before and after each reflection
off $\mathcal{I}$ remains the same (as a consequence of the reflecting
boundary conditions imposed on $\mathcal{I}$).
\end{enumerate}
In particular, the energy content of each beam, \emph{except for}
$\zeta_{0}$, strictly increases with the number of reflections off
$\mathcal{I}$.\footnote{The sum of the energies of all beams is conserved and is proportional
to the value of $\tilde{m}$ at $\mathcal{I}$. In particular, the
energy gain for the beams $\zeta_{i}$, $1\le i\le N_{\varepsilon}$
eventually comes at the expense of a decrease in the energy of $\zeta_{0}$.
However, this decrease is negligible compared to the initial value
of $\mathcal{E}[\zeta_{0}]$ (which also satisfies $\mathcal{E}[\zeta_{0}]\gg\mathcal{E}[\zeta_{i}]$
for all $i\ge1$), hence we can safely assume that $\mathcal{E}[\zeta_{0}]$
remains virtually unchanged in the evolution.}

\medskip{}

\noindent \emph{Formation of the intermediate profile $\mathcal{S}_{*}$}.
For any beam index $0\le i\le N_{\varepsilon}-1$, and any $n\in\mathbb{N}$
corresponding to a specific number of reflections of the beam $\zeta_{i}$
off $\mathcal{I}$, let us introduce the following dimensionless quantity:
\begin{equation}
\mu_{i}[n]\doteq\frac{2\mathcal{E}^{(n)}[\zeta_{i}]}{d_{i}^{(n)}},\label{eq:MuIntro}
\end{equation}
where $\mathcal{E}^{(n)}[\zeta_{i}]$ is the energy of $\zeta_{\iota}$
after the $n$-th reflection off $\mathcal{I}$, while $d_{i}^{(n)}$
denotes the distance (defined in an appropriate sense) between the
beams $\zeta_{i}$ and $\zeta_{i+1}$ after the same reflection. Using
the relations (\ref{eq:INcreaseIntro}) and (\ref{eq:DecreaseIntro}),
combined with an analogous set of estimates for the change in the
separation of the beams over time, we infer that the quantities $\mu_{i}[n]$
satisfy the following recursive system of relations: 
\begin{equation}
\mu_{i}[n]=\mu_{i}[n-1]\exp\Big(2\sum_{\bar{i}=0}^{i-1}\mu_{\bar{i}}[n]+\mathfrak{Err}\Big).\label{eq:RecursiveIntro}
\end{equation}
Ignoring the error terms $\mathfrak{Err}$, the relation (\ref{eq:RecursiveIntro})
can be readily solved inductively in $i$: For $i=0$, (\ref{eq:RecursiveIntro})
implies that $\mu_{0}[n]\simeq\mu_{0}[0]$, for $i=1$ we infer that
$\mu_{1}[n]\simeq\mu_{1}[0]e^{2n\mu_{0}[0]}$, and so on. In particular,
all the quantities $\mu_{i}[n]$ \emph{except for} $\mu_{0}[n]$ are
strictly increasing in $n$.
\begin{rem*}
When $n=0$, the definition of the initial data norm $||\cdot||_{data}$
implies that 
\[
\sum_{i\ge0}\mu_{i}[0]\lesssim||\mathcal{D}^{(\varepsilon)}||_{data}.
\]
Therefore, while (\ref{eq:RecursiveIntro}) implies a fast rate of
growth for $\mu_{i}[n]$, $i\ge1$, the smallness condition (\ref{eq:InitialDataGoingToZeroIntro})
on the initial data necessitates that, for any \emph{fixed} value
of $n$, $\max_{i}\mu_{i}[n]\rightarrow0$ as $\varepsilon\rightarrow0$;
this is, of course, also implied by the Cauchy stability of the trivial
solution $(\mathcal{M}_{AdS},g_{AdS})$ with respect to $||\cdot||_{data}$.
\end{rem*}
Given any natural number $n_{*}$, it is not difficult to choose the
initial data parameters $a_{i}$ in (\ref{eq:Energy}) so that the
quantities $\mu_{j}[n_{*}]$, $0\le j\le N_{\varepsilon}-1$, obtained
by solving (\ref{eq:RecursiveIntro}) (with initial values $\mu_{j}[0]$
computed explicitly in terms of $a_{j}$), are equal to the right
hand side of (\ref{eq:AllButFinalBeamProfileIntro}), i.\,e. 
\begin{equation}
\mu_{j}[n_{*}]=\frac{C}{N_{\varepsilon}}\exp\big(-2\frac{C}{N_{\varepsilon}}j\big).\label{eq:MuJProfile}
\end{equation}
 Similarly, using (\ref{eq:INcreaseIntro}), $a_{N_{\varepsilon}}$
can be chosen so that the energy $\mathcal{E}^{(n_{*})}[\zeta_{N_{\varepsilon}}]$
of the last beam is equal to the right hand side of (\ref{eq:FinalBeamProfileIntro}),
i.\,e.: 
\begin{equation}
\mathcal{E}^{(n_{*})}[\zeta_{N_{\varepsilon}}]=\frac{\varepsilon^{(N_{\varepsilon})}}{\sqrt{-\Lambda}}.\label{eq:ENProfile}
\end{equation}
 Provided $n_{*}$ is large enough in terms of $\varepsilon$, we
can estimate a priori (using the explicit solution of (\ref{eq:RecursiveIntro})
and the fast growth of $\mu_{i}[n]$ in $n$) that the aforementioned
values of the parameters $a_{i}$ (which were defined in terms of
$n_{*}$) are consistent with the initial smallness assumption (\ref{eq:InitialDataGoingToZeroIntro}).
It can be then readily shown that, between the $n_{*}$-th and the
$(n_{*}+1)$-th reflection of the beams off $\mathcal{I}$, there
exists a time $u=u_{*}\sim n_{*}(-\Lambda)^{-\frac{1}{2}}$ such that
the beam slices $\zeta_{i}\cap\{u=u_{*}\}$ satisfy (\ref{eq:AllButFinalBeamProfileIntro})
and (\ref{eq:FinalBeamProfileIntro}); see Section (\ref{sec:ThefirstStage}).
\begin{rem*}
At the time $u=u_{*}$ when the intermediate profile $\mathcal{S}_{*}$
is formed, the $||\cdot||$-norm of the solution (defined by the left
hand side of (\ref{eq:SmallnessRightHandSideConstraintsIntro})) satisfies
\begin{equation}
||(\mathcal{M},g;f)^{(\varepsilon)}||_{u\le u_{*}}\gtrsim\sum_{i=0}^{N_{\varepsilon}-1}\mu_{i}[n_{*}]\gtrsim\sum_{i=0}^{N_{\varepsilon}-1}\frac{C}{N_{\varepsilon}}e^{-2\frac{C}{N_{\varepsilon}}i}\sim C\gg1.\label{eq:LargeNorm}
\end{equation}
As a result, the formation of $\mathcal{S}_{*}$ already provides
an instability statement for $(\mathcal{M}_{AdS},g_{AdS})$ with respect
to the initial data norm $||\cdot||_{data}$; for the proof of the
AdS instability conjecture, however, it is necessary to move beyond
$u=u_{*}$ and establish that, moreover, a trapped surface forms in
$(\mathcal{M},g;f)^{(\varepsilon)}$.

We should also point out that, as a consequence of the explicit formulas
(\ref{eq:AllButFinalBeamProfileIntro})\textendash (\ref{eq:FinalBeamProfileIntro})
for the energy content of $\zeta_{i}$ at $\{u=u_{*}\}$ and the fact
that $\mu_{i}[n]$ is non-decreasing in $n$, we can estimate a priori
that, for all $0\le n\le n_{*}$: 
\begin{equation}
\max_{0\le i\le N_{\varepsilon}-1}\mu_{i}[n]\le\frac{C}{N_{\varepsilon}}.\label{eq:MaxMu}
\end{equation}
Provided that the number $N_{\varepsilon}$ of beams is sufficiently
large in terms of $C$ and satisfies 
\[
\frac{N_{\varepsilon}}{C}\gg\max_{i}\frac{d_{i}}{\varepsilon^{(i)}}
\]
(where $d_{i}$ is the initial separation between $\zeta_{i}$ and
$\zeta_{i+1}$), the estimate (\ref{eq:MaxMu}) (combined with a number
of technical lemmas related to the geodesic flow on $(\mathcal{M},g)^{(\varepsilon)}$)
allows us to obtain the crucial a priori bound (\ref{eq:SmallnessTrappingIntro})
for $\frac{2\tilde{m}}{r}$. As mentioned earlier, this bound is fundamental
for rigorously implementing the heuristic ideas discussed in this
section.\footnote{Observe that, while the scale invariant norm $||\cdot||$ of the solution
$(\mathcal{M},g;f)^{(\varepsilon)}$ becomes large at $u=u_{*}$ (see
(\ref{eq:LargeNorm})), the slightly weaker scale invariant quantity
$\frac{2\tilde{m}}{r}$ remains bounded by a small constant; introducing
the multiscale hierarchy (\ref{eq:Length})\textendash (\ref{eq:Energy})
was fundamental for achieving the construction of such a configuration
of beams.}
\end{rem*}

\subsubsection{\label{subsec:Second-stage-of}Second stage of the instability: Trapped
surface formation}

The second step of the proof of Theorem \ref{thm:TheTheoremIntro}
will consist of showing that a trapped surface (and, hence, a black
hole region) is formed at a time $u=u_{\dagger}>u_{*}$ with $u_{\dagger}-u_{*}\lesssim1$.
More precisely, we will show that, in the development of the intermediate
profile $\mathcal{S}_{*}$, the configuration of the beams $\zeta_{i}$,
$0\le i\le N_{\varepsilon}$ behaves as follows (see Figure \ref{fig:Final_Step_Intro}):
\begin{enumerate}
\item For $0\le i\le N_{\varepsilon}-1$, the geodesics in the beams $\zeta_{i}$
obey dynamics which are qualitatively similar to those on AdS spacetime
(albeit satisfying weaker bounds than in the region $u<u_{*}$); in
particular, the beams $\zeta_{i}$ briefly approach the center $r=0$
before being deflected away, intersecting with each other, in the
meantime, as depicted in Figure \ref{fig:Final_Step_Intro}. Up until
the time $u=u'$ when the last intersection between these beams and
the outermost beam $\zeta_{N_{\varepsilon}}$ occurs (see Figure \ref{fig:Final_Step_Intro}),
the mass ratio $\frac{2\tilde{m}}{r}$ satisfies the smallness condition
(\ref{eq:SmallnessTrappingIntro}).
\item The final beam $\zeta_{N_{\varepsilon}}$, moving in the ingoing direction,
interacts with all the beams $\zeta_{i}$, $i=0,\ldots,N_{\varepsilon}-1$,
increasing its energy content $\mathcal{E}[\zeta_{N_{\varepsilon}}]$.
The increase in $\mathcal{E}[\zeta_{N_{\varepsilon}}]$ is sufficient
for a trapped surface to form before $\zeta_{N_{\varepsilon}}$ has
the chance to be deflected off to infinity again: There exists a point
$p_{\dagger}\in\zeta_{N_{\varepsilon}}\cap\{u\ge u'\}$ such that
\begin{equation}
\frac{2\tilde{m}}{r}(p_{\dagger})>1.\label{eq:TrappedSphereIntro}
\end{equation}
(see Figure \ref{fig:Final_Step_Intro}).
\end{enumerate}
\begin{figure}[h] 
\centering 
\scriptsize
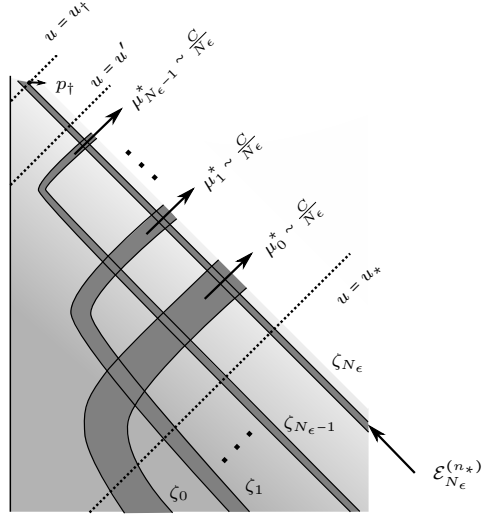 
\caption{Schematic depiction of the evolution of the intermediate profile $\mathcal{S}_*$. After interacting with each other for one last time, the beams $\zeta_i$, $1\le i \le N_{\epsilon}-1$ gain a sufficient amount of energy so that, at the region of intersection of any of those beams with the last beam $\zeta_{N_{\epsilon}}$, the mass ratio $\frac{2\tilde{m}}{r}$ is proportional to the (small) constant $\frac{C}{N_{\epsilon}}$. In turn, $\zeta_{N_{\epsilon}}$ gains enough energy from those interactions, so that a trapped sphere $p_{\dagger}$ is created before $\zeta_{N_{\epsilon}}$ is deflected again to infinity. \label{fig:Final_Step_Intro}} 
\end{figure}

In order to establish that, in the region $\{u_{*}\le u\le u'\}$,
the beams $\zeta_{i}$, $0\le i\le N_{\varepsilon}-1$ approximately
obey the AdS dynamics, we appeal to arguments similar to the ones
implemented in the previous step. In this case, however, additional
technical issues force us to choose the various constants implicit
in the hierarchy (\ref{eq:Length})\textendash (\ref{eq:Energy})
more carefully. To this end, we actually introduce an additional $\varepsilon$-dependent
hierarchy with respect to some of these constants; we refer the reader
to Sections \ref{subsec:The-hierarchy-of-parameters} and \ref{sec:The-final-stage}. 

Arguing similarly as for the proof of (\ref{eq:RecursiveIntro}),
we can show thatan an alogous approximate formula holds for the energy
content $\mathcal{E}^{*}[\zeta_{i}]$ of the beams $\zeta_{i}$, $0\le i\le N_{\varepsilon}-1$,
right before their intersection with $\zeta_{N_{\varepsilon}}$ (see
Figure \ref{fig:Final_Step_Intro}): Setting, for $0\le i\le N_{\varepsilon}-1$,
\[
\mu_{i}^{*}\doteq\frac{2\mathcal{E}^{*}[\zeta_{i}]}{\sup_{\zeta_{i}\cap\zeta_{N_{\varepsilon}}}r},
\]
the analogue of formula (\ref{eq:RecursiveIntro}) reads 
\begin{equation}
\mu_{i}^{*}=\mu_{i}[n_{*}]\exp\Big(2\sum_{\bar{i}=0}^{i-1}\mu_{\bar{i}}^{*}+\mathfrak{Err}\Big),\label{eq:FinalMuI}
\end{equation}
where the quantities $\mu_{i}[n_{*}]$ are given by (\ref{eq:MuJProfile}).
We therefore readily deduce from (\ref{eq:MuJProfile}) and (\ref{eq:FinalMuI})
(ignoring the error terms $\mathfrak{Err}$ in (\ref{eq:FinalMuI}))
that, for all $0\le i\le N_{\varepsilon}-1$: 
\begin{equation}
\mu_{i}^{*}\sim\frac{C}{N_{\varepsilon}}.\label{eq:MuStar}
\end{equation}

Let us now consider the beam $\zeta_{N_{\varepsilon}}$. Using the
formula (\ref{eq:IncreaseEnergyIngoingIntro}) for $\zeta_{N_{\varepsilon}}$
at every intersection between $\zeta_{N_{\varepsilon}}$ and the beams
$\zeta_{i}$, $0\le i\le N_{\varepsilon}-1$, we infer that the energy
content $\mathcal{E}_{final}[\zeta_{N_{\varepsilon}}]$ of $\zeta_{N_{\varepsilon}}$
at $u=u'$ satisfies the following lower bound in terms of the associated
energy $\mathcal{E}^{(n_{*})}[\zeta_{N_{\varepsilon}}]$ at $u=u_{*}$:
\[
\mathcal{E}_{final}[\zeta_{N_{\varepsilon}}]\gtrsim\mathcal{E}^{(n_{*})}[\zeta_{N_{\varepsilon}}]\cdot\exp\Big(\sum_{i=0}^{N_{\varepsilon}-1}\mu_{i}^{*}\Big).
\]
In view of the relations (\ref{eq:ENProfile}) and (\ref{eq:MuStar})
for $\mathcal{E}^{(n_{*})}[\zeta_{N_{\varepsilon}}]$ and $\mu_{i}^{*}$,
respectively, we thus deduce that 
\begin{equation}
\mathcal{E}_{final}[\zeta_{N_{\varepsilon}}]\sim e^{C}\frac{\varepsilon^{(N_{\varepsilon})}}{\sqrt{-\Lambda}}\gg\frac{\varepsilon^{(N_{\varepsilon})}}{\sqrt{-\Lambda}}.\label{eq:HighEnoughEnergy}
\end{equation}

We will now argue that the lower bound (\ref{eq:HighEnoughEnergy})
and the fact that the beam $\zeta_{N_{\varepsilon}}$ consists of
geodesics satisfying initially the angular momentum condition
\begin{equation}
\frac{l}{E_{0}}\sim\frac{\varepsilon^{(N_{\varepsilon})}}{\sqrt{-\Lambda}}\label{eq:SmallAngularMomentumInitially}
\end{equation}
(see (\ref{eq:AngularMomentaAssumption})) imply that there exists
a point $p_{\dagger}\in\zeta_{N_{\varepsilon}}\cap\{u\ge u'\}$ such
that (\ref{eq:TrappedSphereIntro}) holds. Notice that, for any $u_{0}>u'$,
we can trivially estimate from below 
\[
\sup_{\zeta_{N_{\varepsilon}}\cap\{u=u_{0}\}}\frac{2\tilde{m}}{r}\gtrsim\frac{2\mathcal{E}_{final}[\zeta_{N_{\varepsilon}}]}{\sup_{\zeta_{N_{\varepsilon}}\cap\{u=u_{0}\}}r}.
\]
Thus, in order to establish (\ref{eq:TrappedSphereIntro}) and complete
the proof of Theorem \ref{thm:TheTheoremIntro}, it suffices to show
that, for some $u_{\dagger}>u'$, the beam slice $\zeta_{N_{\varepsilon}}\cap\{u=u_{\dagger}\}$
is contained in the region $\{r\le r_{0}\}$, where 
\[
r_{0}\sim\frac{\varepsilon^{(N_{\varepsilon})}}{\sqrt{-\Lambda}}.
\]

Heuristically, on a spacetime where the geodesic flow behaves similarly
as on $(\mathcal{M}_{AdS},g_{AdS})$, the inclusion 
\begin{equation}
\zeta_{N_{\varepsilon}}\cap\{u=u_{\dagger}\}\subset\{r\le r_{0}\}\label{eq:Inclusion}
\end{equation}
 would follow from the fact that, for every null geodesic $\gamma$
of $(\mathcal{M}_{AdS},g_{AdS})$, the minimum value of $r$ along
$\gamma$ satisfies 
\[
\min_{\gamma}r\sim\frac{l}{E_{0}}.
\]
However, in our case, for $u\ge u'$, the spacetime metric $g$ is
no longer close to $g_{AdS}$; the fact that there exists a $u_{\dagger}$
such that the inclusion (\ref{eq:Inclusion}) holds follows from a
careful manipulation of the equations of the geodesic flow (in the
regime where the condition (\ref{eq:SmallnessTrappingIntro}) is violated),
using in addition some of the monotonicity properties of the system
(\ref{eq:EinsteinMasslessVlasovIntro}), as well as the estimates
on the geodesic flow for $u\le u_{*}$ obtained in the previous step.
For the relevant details, see Section \ref{subsec:Energy-growth-for-final-beam}. 
\begin{rem*}
A technical issue which was not highlighted so far in this discussion
is the fact that the existence and smoothness of the development $(\mathcal{M},g;f)^{(\varepsilon)}$
up to the time $u=u_{\dagger}$ of trapped surface formation (including,
in particular, the statement that a naked singularity does not appear
earlier in the evolution) is non-trivial. The Cauchy stability statement
for $(\mathcal{M}_{AdS},g_{AdS})$ guarantees the existence of $(\mathcal{M},g;f)^{(\varepsilon)}$
only up to times $U$ when 
\[
||(\mathcal{M},g;f)^{(\varepsilon)}||_{u\le U}\ll1.
\]
Beyond that point, and up to time $u=u'$, we infer the existence
and smoothness of $(\mathcal{M},g;f)^{(\varepsilon)}$ using an extension
principle established in our companion paper \cite{MoschidisVlasovWellPosedness},
guaranteeing the smooth extendibility of a development under the smallness
condition (\ref{eq:SmallnessTrappingIntro}) for $\frac{2\tilde{m}}{r}$.
From $u=u'$ up to $u=u_{\dagger}$, the existence and smoothness
of the solution follows from our explicit a priori estimates for the
geodesics in the beam $\zeta_{N_{\varepsilon}}$ and the fact that,
in the part of $\{u'\le u\le u_{\dagger}\}$ consisting of the past
of the point $p_{\dagger}$, the spacetime is vacuum (and hence trivially
extendible) outside $\zeta_{N_{\varepsilon}}$; for a review of the
relevant extension principles, see Section \ref{sec:Well-posedness-and-extension},
as well as Section \ref{subsec:Notational-conventions-and}.
\end{rem*}

\subsubsection{Discussion: Comparison with the case of the Einstein\textendash null
dust system with an inner mirror}

\begin{figure}[h] 
\centering 
\scriptsize
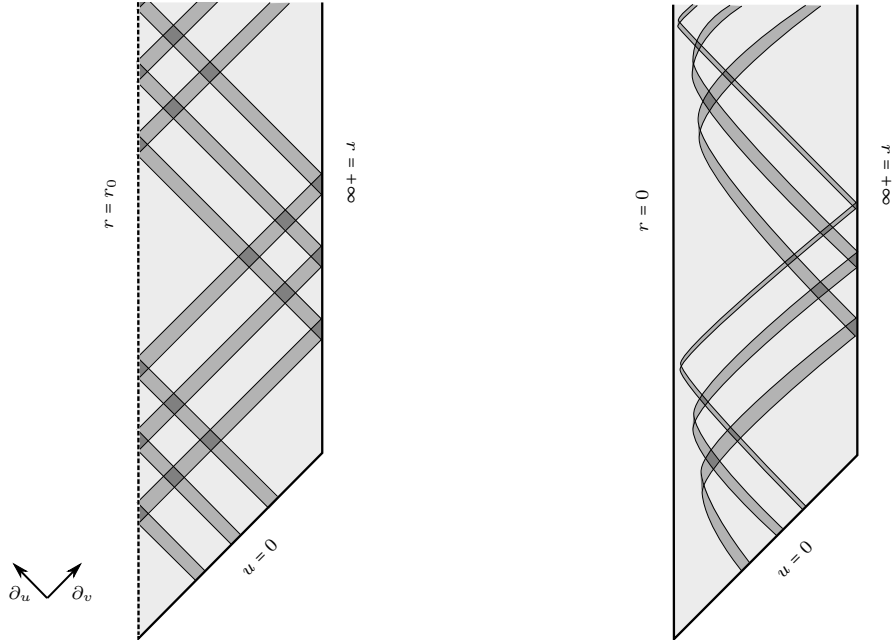 
\caption{Schematic depiction of the configuration of beam interactions in the case of the Einstein--null dust system with an inner mirror (left), which was treated in \cite{MoschidisNullDust}, and the multi-scale configuration employed in this paper for the case of the Einstein--massless Vlasov system (right). For simplicity, only three beams are depicted in each case. \label{fig:Comparison_null_dust_intro}} 
\end{figure}

In this section, we will highlight the differences between the strategy
of proof of Theorem \ref{thm:TheTheoremIntro}, sketched in the previous
sections, and the one implemented in \cite{MoschidisNullDust} for
the case of the spherically symmetric Einstein\textendash null dust
system.

In \cite{MoschidisNullDust}, the instability of $(\mathcal{M}_{AdS},g_{AdS})$
as a solution of the Einstein\textendash null dust system with an
inner mirror was established by setting up a family of initial data
$(r,\Omega^{2};\bar{\tau})^{(\varepsilon)}|_{u=0}$ which gave rise
to a configuration of null dust beams $\zeta_{i}^{\prime}$, $0\le i\le N_{\varepsilon}$,
of comparable size (see Figure \ref{fig:Comparison_null_dust_intro}).
These beams were successively reflected off an inner mirror at $r=r_{0}^{(\varepsilon)}$
(with $r_{0}^{(\varepsilon)}$ proportional to the total energy $\tilde{m}^{(\varepsilon)}|_{\mathcal{I}}$
of $(r,\Omega^{2};\bar{\tau})^{(\varepsilon)}|_{u=0}$) and conformal
infinity $\mathcal{I}$, exchanging energy through their non-linear
interactions. Using the relations (\ref{eq:MassIntro})\textendash (\ref{eq:ConservationIntro})
and the fact that the beams $\zeta_{i}^{\prime}$ were initially comparable
in size, it was shown in \cite{MoschidisNullDust} that, for this
family of configurations, the quantities 
\[
\mu_{n}\doteq\max_{\{U_{n}\le u\le U_{n+1}\}}\frac{2\tilde{m}}{r}
\]
and 
\[
\rho_{n}=\text{geometric separation between }\zeta_{N_{\varepsilon}}^{\prime}\text{ and }\zeta_{0}^{\prime}\text{ in the region }\{U_{n}\le u\le U_{n+1}\}
\]
 (where $U_{n}$ is the value of $u$ at the point where the beam
$\zeta_{0}^{\prime}$ is reflected off $\mathcal{I}$ for the $n$-th
time\footnote{Assuming that a black hole has not formed for $\{u\le U_{n}\}$.}),
satisfy the system of relations
\begin{align}
\text{\textgreek{r}}_{n+1} & \le\text{\textgreek{r}}_{n}+C_{1}r_{0}^{(\varepsilon)}\log\big((1-\text{\textgreek{m}}_{n})^{-1}+1\big),\label{eq:InductiveRelationSequence-1}\\
\text{\textgreek{m}}_{n+1} & \ge\text{\textgreek{m}}_{n}\exp\big(\frac{c_{1}r_{0}^{(\varepsilon)}}{\text{\textgreek{r}}_{n+1}}\big),\nonumber 
\end{align}
for some $0<c_{1}<1<C_{1}$ (see the relation (6.165) in \cite{MoschidisNullDust}).
It was then shown that the system (\ref{eq:InductiveRelationSequence-1})
guarantees the existence of some $n_{0}=n_{0}(\varepsilon)\in\mathbb{N}$
such that 
\begin{equation}
\mu_{n_{0}}\ge1-\delta_{\varepsilon}\label{eq:AlmostTrapping}
\end{equation}
 for some $\delta_{\varepsilon}\ll1$. From (\ref{eq:AlmostTrapping}),
it was concluded using a suitable Cauchy stability statement that,
by possibly perturbing the initial data set $(r,\Omega^{2};\bar{\tau})^{(\varepsilon)}|_{u=0}$
ever so slightly (with the size of the pertrbation determined by $\delta_{\varepsilon}$),
one could in fact achieve 
\begin{equation}
\mu_{n_{0}}>1.\label{eq:TrappingNullDust}
\end{equation}
The lower bound (\ref{eq:TrappingNullDust}) then implied the existence
of a trapped sphere in the development of $(r,\Omega^{2};\bar{\tau})^{(\varepsilon)}|_{u=0}$
at time $u\sim U_{n_{0}}+O(1)$; see \cite{MoschidisNullDust}.

The analysis of \cite{MoschidisNullDust} leading to the recursive
system of inequalities (\ref{eq:InductiveRelationSequence-1}) relied
crucially on the fact that the null-dust beams $\zeta_{i}^{\prime}$
consisted entirely of radial null geodesics, which, in a double null
coordinate chart $(u,v,\theta,\varphi)$, necessarily move along lines
of the form $\{u=const\}$ or $\{v=const\}$ (see Figure \ref{fig:Comparison_null_dust_intro}).
This trivial a priori control on the paths of radial null geodesics
in the $(u,v)$-plane implies, in particular, that the qualititative
picture of beam interactions depicted in Figure \ref{fig:Comparison_null_dust_intro}
remains valid even in the regime where $\frac{2\tilde{m}}{r}\sim1$,
i.\,e.~in the last few reflections of the beams off $r=r_{0}^{(\varepsilon)}$
and $\mathcal{I}$ before a trapped surface is formed. Furthermore,
the presence of an inner mirror at $r=r_{0}^{(\varepsilon)}>0$ in
the setup of \cite{MoschidisNullDust} guaranteed the absence of naked
singularities in the evolution of the initial data family $(r,\Omega^{2};\bar{\tau})^{(\varepsilon)}|_{u=0}$
(as a consequence of the results of \cite{MoschidisMaximalDevelopment}).

In contrast, in the case of the Einstein\textendash massless Vlasov
system (\ref{eq:EinsteinMasslessVlasovIntro}), there is no useful
general a priori estimate for the shape of beams consisting of non-radial
geodesics in the regime where $\frac{2\tilde{m}}{r}\sim1$ (as suggested
already by the relation (\ref{eq:UsefulRelationForGeodesicWithMu-U-1})).
Therefore, in order to establish the formation of a trapped sphere
in this setting, we were forced to design a configuration of interacting
Vlasov beams with the property that all the beam interactions preceding
the first point $p_{\dagger}$ where $\frac{2\tilde{m}}{r}=1$ lie
in the regime $\frac{2\tilde{m}}{r}\ll1$; in this regime, the qualititative
picture of the (right half of) Figure \ref{fig:Comparison_null_dust_intro}
can be shown to remain relevant (see Section \ref{sec:GeodesicPathsAndDifferenceEstimates}).
In particular, this was achived by first identifying the profile $\mathcal{S}_{*}$
(described at the end of Section \ref{subsec:First-stage-of}) as
a useful intermediate step for trapped surface formation. In turn,
the structure of $\mathcal{S}_{*}$ necessitated imposing the multi-scale
hierarchy (\ref{eq:Length})\textendash (\ref{eq:Energy}) on the
construction of the initial data family $\mathcal{D}^{(\varepsilon)}$.
It is a remarkable feature of the system (\ref{eq:EinsteinMasslessVlasovIntro})
that the same hierarchy of scales greatly simplifies the formulas
of energy exchange occuring between the Vlasov beams, resulting in
the approximate \emph{monotonicity }relations (\ref{eq:RecursiveIntro});
the monotonicity properties of (\ref{eq:RecursiveIntro}) are crucial
for obtaining a priori control of $\frac{2\tilde{m}}{r}$ in the evolution
until the formation of $\mathcal{S}_{*}$, thus ensuring ensure the
absence of naked singularities in the solution in view of results
obtained in our companion paper \cite{MoschidisVlasovWellPosedness}.

\subsection{Outline of the paper}

The structure of the paper is as follows:

In Section \ref{sec:The-Einstein--Vlasov-system}, we will introduce
the Einstein\textendash massless Vlasov system (\ref{eq:EinsteinMasslessVlasovIntro})
in spherical symmetry. In addition, we will state a number of notational
conventions related to asymptotically AdS spacetimes and we will introduce
the notion of a reflecing boundary condition for (\ref{eq:EinsteinMasslessVlasovIntro})
on $\mathcal{I}$ .

In Section \ref{sec:Well-posedness-and-extension}, we will introduce
the asymptotically AdS characteristic initial-boundary value problem
for (\ref{eq:EinsteinMasslessVlasovIntro}) and present a number of
well-posedness results in this context. These results will include
a fundamental Cauchy stability statement for $(\mathcal{M}_{AdS},g_{AdS})$
in a low regularity topology. The proofs of the results of Section
\ref{sec:Well-posedness-and-extension} are obtained in our companion
paper \cite{MoschidisVlasovWellPosedness}.

The main result of this paper, namely Theorem \ref{thm:TheTheoremIntro},
will be presented in detail in Section \ref{sec:The-main-result}.

The proof of Theorem \ref{thm:TheTheoremIntro} will occupy Sections
\ref{sec:GeodesicPathsAndDifferenceEstimates}\textendash \ref{sec:The-final-stage}.
In particular, the arguments sketched in Section \ref{subsec:First-stage-of}
regarding the first stage of the instability will be presented in
detail in Sections \ref{sec:GeodesicPathsAndDifferenceEstimates}\textendash \ref{sec:ThefirstStage}
(with Sections \ref{sec:GeodesicPathsAndDifferenceEstimates} and
\ref{sec:The-technical-core} devoted to the development of the necessary
technical machinery); the proof of trapped surface formation (roughly
discussed in Section \ref{subsec:Second-stage-of}) will then be presented
in Section \ref{sec:The-final-stage}.

\subsection{Acknowledgements}

I would like to thank Mihalis Dafermos for originally suggesting this
problem to me, as well as for numerous insightful discussions and
for valuable comments on earlier versions of the manuscript. I am
also grateful to Igor Rodnianski for many helpful comments on earlier
ideas on the problem. I would like to acknowledge support from the
Miller Institute for Basic Research in Science, University of California
Berkeley.

\section{\label{sec:The-Einstein--Vlasov-system}The Einstein\textendash massless
Vlasov system in spherical symmetry}

In this section, we will introduce the spherically symmetric Einstein\textendash massless
Vlasov system in $3+1$ dimensions, expressed in a double null coordinate
chart. We will also formulate the reflecting boundary condition for
a massless Vlasov field at conformal infinity $\mathcal{I}$ in the
asymptotically AdS setting. A more detailed statement of the notions
and the results appearing in this section can be found in our companion
paper \cite{MoschidisVlasovWellPosedness}.

\subsection{\label{subsec:Spherically-symmetric-spacetimes}Spherically symmetric
spacetimes and double null coordinate pairs}

In this paper, we will follow the same conventions regarding spherically
symmetric double null coordinate charts as in our companion paper
\cite{MoschidisVlasovWellPosedness} (similar also to those of \cite{MoschidisMaximalDevelopment,MoschidisNullDust}).
Our assumptions on the topology and regularity of the underlying spacetimes
will be satisfied by the solutions of the Einstein\textendash massless
Vlasov system (\ref{eq:EinsteinMasslessVlasovIntro}) constructed
in the proof of Theorem \ref{thm:TheTheoremIntro}. 

In particular, we will only consider \emph{smooth, connected} and
\emph{time oriented} spacetimes $(\mathcal{M}^{3+1},g)$ which are
\emph{spherically symmetric }with a non-empty \emph{axis} $\mathcal{Z}$
(see \cite{MoschidisVlasovWellPosedness}). We will further assume
that $\mathcal{Z}$ is connected and that $\mathcal{M}\backslash\mathcal{Z}$
splits diffeomorphically under the action of $SO(3)$ as
\begin{equation}
\mathcal{M}\backslash\mathcal{Z}\simeq\mathcal{U}\times\mathbb{S}^{2}.\label{eq:NewSphericallySymmetricmanifold}
\end{equation}
We will also restrict ourselves to spacetimes $(\mathcal{M},g)$ such
that the region $\mathcal{M}\backslash\mathcal{Z}$ is regularly foliated
by the two families of spherically symmetric null hypersurfaces $\mathcal{H}=\big\{\mathcal{C}^{+}(p):\,p\in\mathcal{Z}\big\}$
and $\overline{\mathcal{H}}=\big\{\mathcal{C}^{-}(p):\,p\in\mathcal{Z}\big\}$,
where $\mathcal{C}^{+}(p),\mathcal{C}^{-}(p)$ are the future and
past light cones emanating from $p$, respectively. See \cite{MoschidisVlasovWellPosedness}
for a more detailed discussion on the properties of spacetimes $(\mathcal{M},g)$
satisfying the aforementioned conditions.

A \emph{double null coordinate pair} $(u,v)$ on $(\mathcal{M},g)$
will consist of a pair of continuous functions $u,v:\mathcal{M}\rightarrow\mathbb{R}$
which are a smooth parametrization of the foliations $\mathcal{H},\overline{\mathcal{H}}$,
respectively, on $\mathcal{M}\backslash\mathcal{Z}$. Note that any
choice of double null coordinate pair $(u,v)$ on $\mathcal{M}$ fixes
a smooth embedding $(u,v):\mathcal{U}\rightarrow\mathbb{R}^{2}$;
from now on, we will identify $\mathcal{U}$ with its image in $\mathbb{R}^{2}$
associated to a given null coordinate pair.
\begin{rem*}
We will only consider double null coordinate pairs $(u,v)$ for which
$\partial_{u}+\partial_{v}$ is a timelike and future directed vector
field on $\mathcal{M}\backslash\mathcal{Z}$.
\end{rem*}
Given a double null coordinate pair $(u,v)$, the metric $g$, restricted
on $\mathcal{M}\backslash\mathcal{Z}$, is expressed as follows:
\begin{equation}
g=-\Omega^{2}(u,v)dudv+r^{2}(u,v)g_{\mathbb{S}^{2}},\label{eq:NewSphericallySymmetricMetric}
\end{equation}
where $g_{\mathbb{S}^{2}}$ is the standard round metric on $\mathbb{S}^{2}$
and $\text{\textgreek{W}},r:\mathcal{U}\rightarrow(0,+\infty)$ are
smooth functions, with $r$ extending continuously to $0$ on the
axis $\mathcal{Z}$. 

For any pair of smooth functions $h_{1},h_{2}:\mathbb{R}\rightarrow\mathbb{R}$
with $h_{1}^{\prime},h_{2}^{\prime}\neq0$, we can define a new double
null coordinate pair on $\mathcal{M}$ by the relation
\begin{equation}
(\bar{u},\bar{v})=(h_{1}(u),h_{2}(v)).\label{eq:NewGeneralCoordinateTransformation}
\end{equation}
In the new coordinates, the metric $g$ takes the form 
\begin{equation}
g=-\bar{\text{\textgreek{W}}}^{2}(\bar{u},\bar{v})d\bar{u}d\bar{v}+r^{2}(\bar{u},\bar{v})g_{\mathbb{S}^{2}},\label{eq:NewSphericallySymmetricMetricNewGauge}
\end{equation}
where 
\begin{gather}
\bar{\text{\textgreek{W}}}^{2}(\bar{u},\bar{v})=\frac{1}{h_{1}^{\prime}h_{2}^{\prime}}\text{\textgreek{W}}^{2}(h_{1}^{-1}(\bar{u}),h_{2}^{-1}(\bar{v})),\label{eq:NewOmega}\\
r(\bar{u},\bar{v})=r(h_{1}^{-1}(\bar{u}),h_{2}^{-1}(\bar{v})).\label{eq:NewR}
\end{gather}
\begin{rem*}
We will frequently make use of such coordinate transformations, without
renaming the coordinates each time. 
\end{rem*}
Let $(y^{1},y^{2})$ be a local coordinate chart on $\mathbb{S}^{2}$.
Then, the non-zero Christoffel symbols $\Gamma_{\beta\gamma}^{\alpha}$
of (\ref{eq:NewSphericallySymmetricMetric}) in the $(u,v,y^{1},y^{2})$
local coordinate chart on $\mathcal{M}\backslash\mathcal{Z}$ take
the following form:
\begin{gather}
\Gamma_{uu}^{u}=\partial_{u}\log(\text{\textgreek{W}}^{2}),\hphantom{A}\Gamma_{vv}^{v}=\partial_{v}\log(\text{\textgreek{W}}^{2}),\label{eq:NewChristoffelSymbols}\\
\Gamma_{uB}^{A}=r^{-1}\partial_{u}r\delta_{B}^{A},\hphantom{\,}\Gamma_{vB}^{A}=r^{-1}\partial_{v}r\delta_{B}^{A},\nonumber \\
\Gamma_{AB}^{u}=\text{\textgreek{W}}^{-2}\partial_{v}(r^{2})(g_{\mathbb{S}^{2}})_{AB},\hphantom{\,}\Gamma_{AB}^{v}=\text{\textgreek{W}}^{-2}\partial_{u}(r^{2})(g_{\mathbb{S}^{2}})_{AB},\nonumber \\
\Gamma_{BC}^{A}=(\Gamma_{\mathbb{S}^{2}})_{BC}^{A}.\nonumber 
\end{gather}
In the above, the latin indices $A,B,C$ are associated to the spherical
coordinates $y^{1},y^{2}$, $\delta_{B}^{A}$ is Kronecker delta and
$\Gamma_{\mathbb{S}^{2}}$ are the Christoffel symbols of the round
sphere in the $(y^{1},y^{2})$ coordinate chart.

We will define the \emph{Hawking mass} $m:\mathcal{M}\rightarrow\mathbb{R}$
by 
\begin{equation}
m=\frac{r}{2}\big(1-g(\nabla r,\nabla r)\big).
\end{equation}
 Notice that, when viewed as a function on $\mathcal{U}$, the Hawking
mass $m$ is related to the metric coefficients $\Omega$ and $r$
by the formula: 
\begin{equation}
m=\frac{r}{2}\big(1+4\Omega^{-2}\partial_{u}r\partial_{v}r\big)\Leftrightarrow\Omega^{2}=\frac{4\partial_{v}r(-\partial_{u}r)}{1-\frac{2m}{r}}.\label{eq:DefinitionHereHawkingMass}
\end{equation}

Finally, on pure AdS spacetime $(\mathcal{M}_{AdS}^{3+1},g_{AdS})$,
where $g_{AdS}$ is defined by (\ref{eq:AdSMetricPolarCoordinates}),
we will fix a distinguished double null coordinate pair $(u,v)$ by
the relations
\begin{gather}
u=t-\sqrt{-\frac{3}{\Lambda}}\text{Arctan}\Big(\sqrt{-\frac{\Lambda}{3}}r\Big),\label{eq:DoubleNullPairAdS}\\
v=t+\sqrt{-\frac{3}{\Lambda}}\text{Arctan}\Big(\sqrt{-\frac{\Lambda}{3}}r\Big).\nonumber 
\end{gather}
In the resulting double null coordinate chart, $g_{AdS}$ is expressed
as 
\begin{equation}
g_{AdS}=-\Omega_{AdS}^{2}dudv+r^{2}g_{\mathbb{S}^{2}},\label{eq:DoubleNullAdSMetric}
\end{equation}
where 
\begin{align}
r(u,v) & =\sqrt{-\frac{3}{\Lambda}}\tan\Big(\frac{1}{2}\sqrt{-\frac{\Lambda}{3}}(v-u)\Big),\label{eq:AdSMetricValues-1}\\
\Omega_{AdS}^{2}(u,v) & =1-\frac{1}{3}\Lambda r^{2}(u,v).\nonumber 
\end{align}

\subsection{\label{subsec:Asymptotically-AdS-spacetimes}Asymptotically Anti-de~Sitter
spacetimes }

In this section, we will introduce the class of asymptotically AdS
spacetimes in spherical symmetry; the geometry of these spacetimes
will resemble that of (\ref{eq:DoubleNullAdSMetric}) in a neighborhood
of $r=\infty$. In particular, in accordance with \cite{MoschidisVlasovWellPosedness},
we will adopt the following definition: 

\begin{figure}[h] 
\centering 
\begingroup%
  \makeatletter%
  \providecommand\color[2][]{%
    \errmessage{(Inkscape) Color is used for the text in Inkscape, but the package 'color.sty' is not loaded}%
    \renewcommand\color[2][]{}%
  }%
  \providecommand\transparent[1]{%
    \errmessage{(Inkscape) Transparency is used (non-zero) for the text in Inkscape, but the package 'transparent.sty' is not loaded}%
    \renewcommand\transparent[1]{}%
  }%
  \providecommand\rotatebox[2]{#2}%
  \newcommand*\fsize{\dimexpr\f@size pt\relax}%
  \newcommand*\lineheight[1]{\fontsize{\fsize}{#1\fsize}\selectfont}%
  \ifx\svgwidth\undefined%
    \setlength{\unitlength}{150bp}%
    \ifx\svgscale\undefined%
      \relax%
    \else%
      \setlength{\unitlength}{\unitlength * \real{\svgscale}}%
    \fi%
  \else%
    \setlength{\unitlength}{\svgwidth}%
  \fi%
  \global\let\svgwidth\undefined%
  \global\let\svgscale\undefined%
  \makeatother%
  \begin{picture}(1,1.5)%
    \lineheight{1}%
    \setlength\tabcolsep{0pt}%
    \put(0,0){\includegraphics[width=\unitlength,page=1]{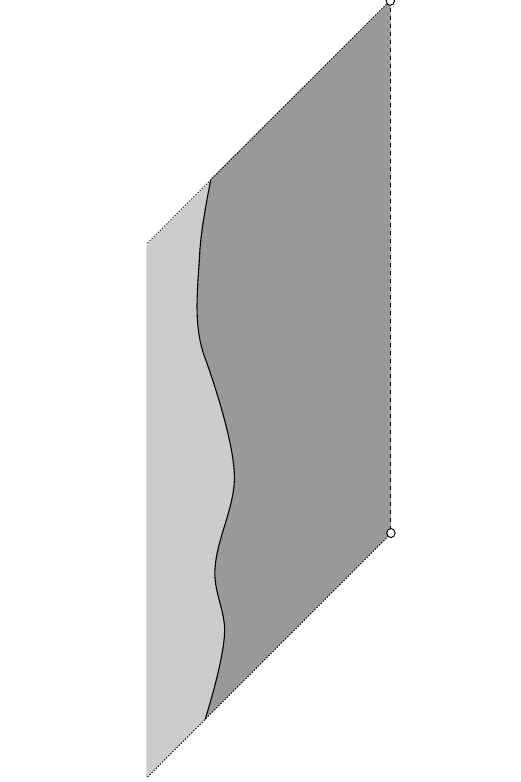}}%
    \put(0.46447705,0.10161798){\color[rgb]{0,0,0}\rotatebox{45}{\makebox(0,0)[lt]{\lineheight{1.25}\smash{\begin{tabular}[t]{l}$u=u_1$\end{tabular}}}}}%
    \put(0.40012062,1.1878184){\color[rgb]{0,0,0}\rotatebox{45}{\makebox(0,0)[lt]{\lineheight{1.25}\smash{\begin{tabular}[t]{l}$u=u_2$\end{tabular}}}}}%
    \put(0.52599007,0.83539608){\color[rgb]{0,0,0}\makebox(0,0)[lt]{\lineheight{1.25}\smash{\begin{tabular}[t]{l}$\mathcal{V}_{as}$\end{tabular}}}}%
    \put(0.79085451,1.08158867){\color[rgb]{0,0,0}\rotatebox{-90}{\makebox(0,0)[lt]{\lineheight{1.25}\smash{\begin{tabular}[t]{l}$v=u+v_{\mathcal{I}}$\end{tabular}}}}}%
    \put(0.38483024,0.29399715){\color[rgb]{0,0,0}\rotatebox{90}{\makebox(0,0)[lt]{\lineheight{1.25}\smash{\begin{tabular}[t]{l}$v=u+v_{R_0}(u)$\end{tabular}}}}}%
  \end{picture}%
\endgroup%
 
\caption{Schematic depiction of the asymptotic region $\mathcal{V}_{as}=\{r \ge R_0 \gg 1 \}$ of an asymptotically AdS spacetime.}
\end{figure}
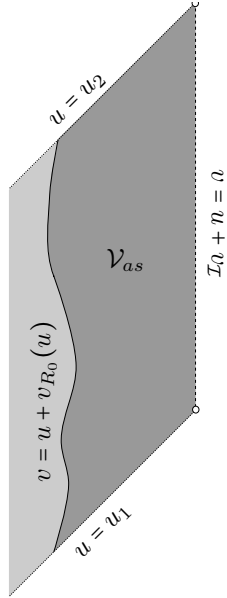
\begin{defn}
\label{def:AsADS}Let $(\mathcal{M},g)$ be a spherically symmetric
spacetime as in Section \ref{subsec:Spherically-symmetric-spacetimes},
with $\sup_{\mathcal{M}}r=+\infty$. We will say that $(\mathcal{M},g)$
is \emph{asymptotically AdS} if, for some $R_{0}\gg1$, there exists
a spherically symmetric double null coordinate pair $(u,v)$ on $\mathcal{M}$
as in Section \ref{subsec:Spherically-symmetric-spacetimes}, such
that the following conditions hold:

\begin{enumerate}

\item The region $\mathcal{V}_{as}$ has the form
\[
\mathcal{V}_{as}=\big\{ u_{1}<u<u_{2}\big\}\cap\big\{ u+v_{R_{0}}(u)\le v<u+v_{\mathcal{I}}\big\}
\]
for some $u_{1}<u_{2}\in\mathbb{R}\cup\{\pm\infty\}$, $v_{\mathcal{I}}\in\mathbb{R}$
and $v_{R_{0}}:(u_{1},u_{2})\rightarrow\mathbb{R}$ with $v(u)<v_{\mathcal{I}}$.

\item The function $\frac{1}{r}$ on $\mathcal{U}$ extends smoothly
on 
\begin{equation}
\mathcal{I}\doteq\big\{ u_{1}<u<u_{2}\big\}\cap\big\{ v=u+v_{\mathcal{I}}\big\}\subset clos(\mathcal{U})\label{eq:TimelikeInfinity}
\end{equation}
and satisfies
\begin{equation}
\frac{1}{r}\Big|_{\mathcal{I}}=0.\label{eq:1/r0}
\end{equation}

\item The function $\frac{\Omega^{2}}{r^{2}}$ extends smoothly on
$\mathcal{I}$, with 
\begin{equation}
\frac{\Omega^{2}}{r^{2}}\Big|_{\mathcal{I}}\neq0.\label{eq:ConformalExtensionI}
\end{equation}

\end{enumerate}
\end{defn}
See \cite{MoschidisVlasovWellPosedness} for further discussion on
the above definition and its relation with the standard definition
of asymptotically AdS spacetimes (appearing, e.\,g., in \cite{Friedrich1995}).
For a spherically symmetric, asymptotically AdS spacetime $(\mathcal{M},g)$
as above, we will use the term \emph{conformal infinity} both for
the planar boundary curve $\mathcal{I}$ and for the spacetime conformal
boundary $\mathcal{I}^{(2+1)}$ of $(\mathcal{M},g)$.

\subsection{\label{subsec:VlasovEquations}Properties of the null geodesic flow
and the massless Vlasov equation}

Let $(\mathcal{M},g)$ be a time oriented, spherically symmetric spacetime
as in Section \ref{subsec:Spherically-symmetric-spacetimes}. In this
section, we will briefly review the properties of the geodesic flow
on $(\mathcal{M},g)$ and we will introduce the Vlasov field equations
on $T\mathcal{M}$. We will use the same notations as those adopted
in \cite{MoschidisVlasovWellPosedness}.

\subsubsection*{The geodesic flow on $(\mathcal{M},g)$}

The equations of motion for a geodesic of $(\mathcal{M},g)$, expressed
in a local coordinate chart $(x^{0},x^{1},x^{2},x^{3})$ on $\mathcal{M}$
with dual momentum coordinates $(p^{0},p^{1},p^{2},p^{3})$ on the
fibers of $T\mathcal{M}$, takes the following form
\begin{equation}
\begin{cases}
\dot{x}^{\alpha}=p^{\alpha},\\
\dot{p}^{a}+\Gamma_{\beta\gamma}^{\alpha}p^{\beta}p^{\gamma}=0,
\end{cases}\label{eq:NewGeodesicFlow}
\end{equation}
 where $\Gamma_{\text{\textgreek{b}\textgreek{g}}}^{\text{\textgreek{a}}}$
are the Christoffel symbols of $g$ with respect to the chart $(x^{0},x^{1},x^{2},x^{3})$.
Fixing a non-vanishing future directed vector field $Q$ on $\mathcal{M}$
(e.\,g.~the vector field $\partial_{u}+\partial_{v}$ in the notation
of Section \ref{subsec:Spherically-symmetric-spacetimes}), the set
\begin{equation}
\mathcal{P}^{+}\doteq\Big\{(x;p)\in T\mathcal{M}:\text{ }g_{\text{\textgreek{a}\textgreek{b}}}(x)p^{\text{\textgreek{a}}}p^{\text{\textgreek{b}}}=0,\text{ }g_{\text{\textgreek{a}\textgreek{b}}}(x)p^{\text{\textgreek{a}}}Q^{\text{\textgreek{b}}}(x)\le0\Big\},\label{eq:NullNewShell}
\end{equation}
i.\,e.~the set of future directed null vectors in $T\mathcal{M}$,
is invariant under (\ref{eq:NewGeodesicFlow}).

The \emph{angular momentum} function $l:T\mathcal{M}\rightarrow[0,+\infty)$
is defined in a local coordinate chart $(u,v,y^{1},y^{2})$ as in
Section \ref{subsec:Spherically-symmetric-spacetimes} by
\begin{equation}
l^{2}\doteq r^{2}g_{AB}p^{A}p^{B}=r^{4}(g_{\mathbb{S}^{2}})_{AB}p^{A}p^{B}\label{eq:AngularMomentum}
\end{equation}
(note that $l$ is in fact coordinate independent). The spherical
symmetry of $(\mathcal{M},g)$ implies that $l$ is a constant of
motion for the geodesic flow (\ref{eq:NewGeodesicFlow}). As a result,
(\ref{eq:NewGeodesicFlow}) can be reduced to a system in terms only
of $u$, $v$, $p^{u}$, $p^{v}$ and $l$. Reexpressed in terms of
these variables, the null-shell relation defining $\mathcal{P}^{+}$
in (\ref{eq:NullNewShell}) takes the form 
\begin{equation}
\Omega^{2}p^{u}p^{v}=\frac{l^{2}}{r^{2}},\text{ }p^{u}\ge0\label{eq:NullShellNewAngularMomentum}
\end{equation}
 while (\ref{eq:NewGeodesicFlow}) restricted on $\mathcal{P}^{+}$
is reduced to
\begin{equation}
\begin{cases}
\frac{du}{ds}=p^{u},\\
\frac{dv}{ds}=p^{v},\\
\frac{d}{ds}\big(\Omega^{2}p^{u}\big)=\Big(\partial_{v}\log(\Omega^{2})-2\frac{\partial_{v}r}{r}\Big)\frac{l^{2}}{r^{2}},\\
\frac{d}{ds}\big(\Omega^{2}p^{v}\big)=\Big(\partial_{u}\log(\Omega^{2})-2\frac{\partial_{u}r}{r}\Big)\frac{l^{2}}{r^{2}},\\
\frac{dl}{ds}=0.
\end{cases}\label{eq:NewNullGeodesicsSphericalSymmetry}
\end{equation}
\begin{rem*}
Identifying a geodesic in $(\mathcal{M},g)$ with its image in the
planar domain $\mathcal{U}$, we will frequently refer to (\ref{eq:NewNullGeodesicsSphericalSymmetry})
simply as the equations of motion for a ``geodesic in $\mathcal{U}$''.
Let us also note that, on a smooth spacetime $(\mathcal{M},g)$ as
above, the relations (\ref{eq:NullShellNewAngularMomentum}) and (\ref{eq:NewNullGeodesicsSphericalSymmetry})
imply that a geodesic $\gamma$ with $l>0$ cannot cross the axis
$\mathcal{Z}\equiv\{r=0\}$.
\end{rem*}

\subsubsection*{The Vlasov equation}

We will adopt the following definition for a Vlasov field $f$ on
$T\mathcal{M}$:
\begin{defn*}
A \emph{Vlasov field} $f$ is a non-negative measure on $T\mathcal{M}$
which is constant along the flow lines of (\ref{eq:NewGeodesicFlow}).
A Vlasov field $f$ supported on (\ref{eq:NullNewShell}) will be
called a \emph{massless Vlasov field}.
\end{defn*}
As a consequence of the above definition, in any local coordinate
chart $(x^{\alpha};p^{\alpha})$ on $T\mathcal{M}$ (with $p^{\alpha}$
dual to $x^{a}$), $f$ satisfies the following equation (refered
to, from now on, as the \emph{Vlasov field equation})
\begin{equation}
p^{\alpha}\partial_{x^{\alpha}}f-\text{\textgreek{G}}_{\beta\gamma}^{\alpha}p^{\beta}p^{\gamma}\partial_{p^{\alpha}}f=0.\label{eq:Vlasov}
\end{equation}

The \emph{energy momentum }tensor of a Vlasov field $f$ is a symmetric
$(0,2)$-form $T_{\alpha\beta}$ on $\mathcal{M}$ (possibly defined
only in the sense of distributions), given by the expression 
\begin{equation}
T_{\alpha\beta}(x)\doteq\int_{T_{x}\mathcal{M}}p_{\alpha}p_{\beta}f\,\sqrt{-det(g(x))}dp^{0}\cdots dp^{3},\label{eq:StressEnergy}
\end{equation}
where $T_{x}\mathcal{M}$ denotes the fiber of $T\mathcal{M}$ over
$x\in\mathcal{M}$ and 
\begin{equation}
p_{\gamma}=g_{\gamma\delta}(x)p^{\delta}.
\end{equation}
Equation (\ref{eq:Vlasov}) implies that
\begin{equation}
\nabla^{\alpha}T_{\alpha\beta}=0,\label{eq:ConservedEnergy}
\end{equation}
i.\,e.~that $T_{\alpha\beta}$ is conserved. 

Another conserved quantity associated to a Vlasov field $f$ is a
1-form called the \emph{particle current,} defined by the formula

\begin{equation}
N_{a}(x)\doteq\int_{T_{x}\mathcal{M}}p_{a}f\,\sqrt{-det(g(x))}dp^{0}\cdots dp^{3}.\label{eq:NewParticleCurrent}
\end{equation}
The Vlasov equation (\ref{eq:Vlasov}) readily implies that 
\begin{equation}
\nabla^{\alpha}N_{a}=0.\label{eq:ConservedParticle}
\end{equation}
\begin{rem*}
In this paper, we will only consider smooth Vlasov fields $f$ which
are compactly supported in the momentum coordinates $p^{a}$ for any
fixed $x$. Under this condition, it can be readily shown that $N_{\alpha}(x)$,
$T_{\alpha\beta}(x)$ are smooth tensor fields on $\mathcal{M}$. 
\end{rem*}
A \emph{spherically symmetric} Vlasov field $f$, i.\,e.~a Vlasov
field which is invariant under the induced action of $SO(3)$ on $T\mathcal{M}$,
only depends on the $u$, $v$, $p^{u}$, $p^{v}$ and $l$ varables.
Assuming, in addition, that $f$ is massless, it follows that $f$
is conserved along the flow lines of the reduced system (\ref{eq:NewNullGeodesicsSphericalSymmetry}).
The Vlasov field equation formally reduces, in this case, to (\ref{eq:Vlasov}):
\begin{equation}
p^{u}\partial_{u}f+p^{v}\partial_{v}f=\Big(\partial_{u}\log(\Omega^{2})(p^{u})^{2}+\frac{2}{r}\Omega^{-2}\partial_{v}r\frac{l^{2}}{r^{2}}\Big)\partial_{p^{u}}f+\Big(\partial_{v}\log(\Omega^{2})(p^{v})^{2}+\frac{2}{r}\Omega^{-2}\partial_{u}r\frac{l^{2}}{r^{2}}\Big)\partial_{p^{v}}f\label{eq:VlasovEquationNewSphericalSymmetry}
\end{equation}
(note that (\ref{eq:VlasovEquationNewSphericalSymmetry}) does not
contain derivatives in $l$).
\begin{rem*}
In this paper, we will only consider \emph{smooth }spherically symmetric
massless Vlasov fields $f$, i.\,e.~$f$ will be of the form 
\begin{equation}
f(u,v;p^{u},p^{v},l)=\bar{f}(u,v;p^{u},p^{v},l)\cdot\delta\big(\Omega^{2}p^{u}p^{v}-\frac{l^{2}}{r^{2}}\big),\label{eq:ExpressionRegularF}
\end{equation}
where $\bar{f}$ is smooth in its variables and $\delta$ is Dirac's
delta function. For a smooth and spherically symmetric massless Vlasov
field $f$, we will frequently denote with $\bar{f}$ any smooth function
for which (\ref{eq:ExpressionRegularF}) holds; note that $\bar{f}$
is uniquely determined only along the null set (\ref{eq:NullNewShell}).
\end{rem*}
The energy-momentum tensor (\ref{eq:StressEnergy}) associated to
a smooth, spherically symmetric Vlasov field $f$ takes the form 
\begin{equation}
T=T_{uu}(u,v)du^{2}+2T_{uv}(u,v)dudv+T_{vv}(u,v)dv^{2}+T_{AB}(u,v)dy^{A}dy^{B}.\label{eq:SphSymEnergy}
\end{equation}
In the case when $f$ is in addition massless, the components of (\ref{eq:SphSymEnergy})
can be expressed as
\begin{align}
T_{uu} & =\frac{\pi}{2}r^{-2}\int_{0}^{+\infty}\int_{0}^{+\infty}\big(\Omega^{2}p^{v}\big)^{2}\bar{f}(u,v;p^{u},p^{v},l)\Big|_{\Omega^{2}p^{u}p^{v}=\frac{l^{2}}{r^{2}}}\,\frac{dp^{u}}{p^{u}}ldl,\label{eq:ComponentsStressEnergy}\\
T_{vv} & =\frac{\pi}{2}r^{-2}\int_{0}^{+\infty}\int_{0}^{+\infty}\big(\Omega^{2}p^{u}\big)^{2}\bar{f}(u,v;p^{u},p^{v},l)\Big|_{\Omega^{2}p^{u}p^{v}=\frac{l^{2}}{r^{2}}}\,\frac{dp^{u}}{p^{u}}ldl,\nonumber \\
T_{uv} & =\frac{\pi}{2}r^{-2}\int_{0}^{+\infty}\int_{0}^{+\infty}\big(\Omega^{2}p^{u}\big)\cdot\big(\Omega^{2}p^{v}\big)\bar{f}(u,v;p^{u},p^{v},l)\Big|_{\Omega^{2}p^{u}p^{v}=\frac{l^{2}}{r^{2}}}\,\frac{dp^{u}}{p^{u}}ldl,\nonumber \\
g^{AB}T_{AB} & =4\Omega^{-2}T_{uv}.\nonumber 
\end{align}
Similarly, the particle current (\ref{eq:NewParticleCurrent}) associated
to $f$ is of the form 
\begin{equation}
N=N_{u}du+N_{v}dv,
\end{equation}
where, in the case when $f$ is in addition massless:
\begin{align}
N_{u} & =\pi r^{-2}\int_{T_{x}\mathcal{M}\cap\{\Omega^{2}p^{u}p^{v}=\frac{l^{2}}{r^{2}}\}}\Omega^{2}p^{v}\bar{f}(u,v;p^{u},p^{v},l)\Big|_{\Omega^{2}p^{u}p^{v}=\frac{l^{2}}{r^{2}}}\,\frac{dp^{u}}{p^{u}}ldl,\label{eq:ComponentsParticle}\\
N_{v} & =\pi r^{-2}\int_{T_{x}\mathcal{M}\cap\{\Omega^{2}p^{u}p^{v}=\frac{l^{2}}{r^{2}}\}}\Omega^{2}p^{u}\bar{f}(u,v;p^{u},p^{v},l)\Big|_{\Omega^{2}p^{u}p^{v}=\frac{l^{2}}{r^{2}}}\,\frac{dp^{u}}{p^{u}}ldl.\nonumber 
\end{align}

The following estimate of $T_{\mu\nu}$ in terms of $N_{\mu}$ will
be useful later in the paper (see also \cite{MoschidisVlasovWellPosedness}):
In view of the expressions (\ref{eq:DefinitionHereHawkingMass}),
(\ref{eq:ComponentsStressEnergy}) and (\ref{eq:ComponentsParticle}),
we can bound
\begin{equation}
\frac{1-\frac{2m}{r}}{\partial_{v}r}T_{vv}(u,v)+\frac{1-\frac{2m}{r}}{-\partial_{u}r}T_{uv}(u,v)\le2\sup_{supp\big(f(u,v;\cdot,\cdot,l)\big)}\Big(\partial_{v}r(u,v)p^{v}-\partial_{u}r(u,v)p^{u}\Big)\cdot N_{v}(u,v)\label{eq:VDerivativeMassfromTotalParticleCurrent}
\end{equation}
and 
\begin{equation}
\frac{1-\frac{2m}{r}}{\partial_{v}r}T_{uv}(u,v)+\frac{1-\frac{2m}{r}}{-\partial_{u}r}T_{uu}(u,v)\le2\sup_{supp\big(f(u,v;\cdot,\cdot,l)\big)}\Big(\partial_{v}r(u,v)p^{v}-\partial_{u}r(u,v)p^{u}\Big)\cdot N_{u}(u,v).\label{eq:UDerivativeMassFromReducedTotalCurrent}
\end{equation}

\subsection{\label{subsec:The-Einstein-equations}The Einstein\textendash massless
Vlasov system}

The \emph{Einstein\textendash massless Vlasov} system with cosmological
constant $\Lambda$ takes the form 
\begin{equation}
\begin{cases}
Ric_{\mu\nu}(g)-\frac{1}{2}R(g)g_{\mu\nu}+\Lambda g_{\mu\nu}=8\pi T_{\mu\nu}[f],\\
p^{\alpha}\partial_{x^{\alpha}}f-\text{\textgreek{G}}_{\beta\gamma}^{\alpha}p^{\beta}p^{\gamma}\partial_{p^{\alpha}}f=0,\\
supp(f)\subset\mathcal{P}^{+}
\end{cases}\label{eq:The mainSystemEinsteinVlasov}
\end{equation}
where $(\mathcal{M},g)$ is a Lorentzian manifold, $f$ is a non-negative
measure on $T\mathcal{M}$, $T_{\mu\nu}[f]$ is expressed in terms
of $f$ by (\ref{eq:StressEnergy}) and $\mathcal{P}^{+}\subset T\mathcal{M}$
is defined by (\ref{eq:NullNewShell}) (see also \cite{DafermosRendall,MoschidisMaximalDevelopment,MoschidisNullDust}).
In this paper, we will only consider the case when the cosmological
constant $\Lambda$ is \emph{negative}.

Reduced to the case where $(\mathcal{M},g)$ is a spherically symmetric
spacetime (see Section \ref{subsec:Spherically-symmetric-spacetimes})
and $f$ is a spherically symmetric massless Vlasov field (see Section
\ref{subsec:VlasovEquations}), the system (\ref{eq:The mainSystemEinsteinVlasov})
is equivalent to the following set of relations for $(r,\Omega^{2},f)$:
\begin{align}
\partial_{u}\partial_{v}(r^{2})= & -\frac{1}{2}(1-\Lambda r^{2})\Omega^{2}+8\pi r^{2}T_{uv},\label{eq:RequationFinal}\\
\partial_{u}\partial_{v}\log(\Omega^{2})= & \frac{\Omega^{2}}{2r^{2}}\big(1+4\Omega^{-2}\partial_{u}r\partial_{v}r\big)-8\pi T_{uv}-2\pi\Omega^{2}g^{AB}T_{AB},\label{eq:OmegaEquationFinal}\\
\partial_{u}(\Omega^{-2}\partial_{u}r)= & -4\pi rT_{uu}\Omega^{-2},\label{eq:ConstraintUFinal}\\
\partial_{v}(\Omega^{-2}\partial_{v}r)= & -4\pi rT_{vv}\Omega^{-2},\label{eq:ConstraintVFinal}\\
p^{\alpha}\partial_{x^{\alpha}}f= & \text{\textgreek{G}}_{\beta\gamma}^{\alpha}p^{\beta}p^{\gamma}\partial_{p^{\alpha}}f,\label{eq:VlasovFinal}\\
supp(f)\subseteq & \big\{\Omega^{2}(u,v)p^{u}p^{v}-\frac{l^{2}}{r^{2}(u,v)}=0,\text{ }p^{u}\ge0\big\}.\label{NullShellFinal}
\end{align}
\begin{rem*}
In view of the relation $4\Omega^{-2}T_{uv}=g^{AB}T_{AB}$ (following
from the fact that $f$ is supported on the null set $\mathcal{P}$)
and the definition (\ref{eq:DefinitionHereHawkingMass}) of $m$,
equation (\ref{eq:OmegaEquationFinal}) is equivalent to 
\begin{equation}
\partial_{u}\partial_{v}\log(\Omega^{2})=4\frac{m}{r^{3}}\frac{(-\partial_{u}r)\partial_{v}r}{1-\frac{2m}{r}}-16\pi T_{uv}.\label{eq:OneMoreUsefulEquationOmega}
\end{equation}
\end{rem*}
It is useful, in general, to consider transformations of the double
null coordinate pair $(u,v)$ of the form $(u,v)\rightarrow(u',v')=(U(u),V(v))$
(see Section \ref{subsec:Spherically-symmetric-spacetimes}). Under
such a gauge transformation, a solution $(r,\Omega^{2},f)$ is transformed
into a solution $(r',(\Omega^{\prime})^{2},f')$ in the new coordinate
system through the relations:
\begin{align}
r^{\prime}(u',v') & \doteq r(u,v),\label{eq:GeneralGaugeTransformationSpacetime}\\
(\Omega^{\prime})^{2}(u',v') & \doteq\frac{1}{\frac{dU}{dv}(u)\cdot\frac{dV}{dv}(v)}\Omega^{2}(u,v),\nonumber \\
f^{\prime}(u',v';\frac{dU}{du}(u)p^{u^{\prime}},\frac{dV}{dv}(v)p^{v^{\prime}}l) & \doteq f(u,v;p^{u},p^{v},l).\nonumber 
\end{align}

Let us introduce the \emph{renormalised Hawking mass} $\tilde{m}$
by the relation
\begin{equation}
\tilde{m}\doteq m-\frac{1}{6}\Lambda r^{3},\label{eq:RenormalisedNewHawkingMass}
\end{equation}
where $m$ is defined by (\ref{eq:DefinitionHereHawkingMass}). Equations
(\ref{eq:RequationFinal})\textendash (\ref{eq:ConstraintVFinal})
yield (formally, at least) the following system for $(r,\tilde{m},f)$
on the subset of $\mathcal{M}$ where $1-\frac{2m}{r}>0$ and $\partial_{u}r<0<\partial_{v}r$:

\begin{align}
\partial_{u}\partial_{v}r= & -\frac{2\tilde{m}-\frac{2}{3}\Lambda r^{3}}{r^{2}}\frac{(-\partial_{u}r)\partial_{v}r}{1-\frac{2m}{r}}+4\pi rT_{uv},\label{eq:EquationROutside}\\
\partial_{v}\log\big(\frac{-\partial_{u}r}{1-\frac{2m}{r}}\big)= & 4\pi r^{-1}\frac{r^{2}T_{vv}}{\partial_{v}r},\label{eq:EksiswshGiaKappa-}\\
\partial_{u}\log\big(\frac{\partial_{v}r}{1-\frac{2m}{r}}\big)= & -4\pi r^{-1}\frac{r^{2}T_{uu}}{-\partial_{u}r},\label{eq:EksiswshGiaK}\\
\partial_{v}\tilde{m}= & 2\pi\big(1-\frac{2m}{r}\big)\Big(\frac{r^{2}T_{vv}}{\partial_{v}r}+\frac{r^{2}T_{uv}}{-\partial_{u}r}\Big),\label{eq:TildeVMaza}\\
\partial_{u}\tilde{m}= & -2\pi\big(1-\frac{2m}{r}\big)\Big(\frac{r^{2}T_{uu}}{-\partial_{u}r}+\frac{r^{2}T_{uv}}{\partial_{v}r}\Big).\label{eq:TildeUMaza}
\end{align}

\paragraph*{Useful relations for null-geodesics on solutions of the system (\ref{eq:RequationFinal})\textendash (\ref{NullShellFinal})}

We will now present a number of relations for null geodesics on solutions
$(\mathcal{M},g;f)$ of the system (\ref{eq:RequationFinal})\textendash (\ref{NullShellFinal}).
These relations, appearing also in our companion paper \cite{MoschidisVlasovWellPosedness},
will be useful for the construction of localised Vlasov beams appearing
in the proof of Theorem \ref{thm:TheTheoremIntro}.

Let $\gamma$ be the projection on $\mathcal{U}$ of a null geodesic
in $(\mathcal{M},g)$ with $l>0$. The relation (\ref{eq:NullShellNewAngularMomentum})
implies that $\gamma$ is a timelike curve in $\mathcal{U}$ with
respect to the reference metric 
\begin{equation}
g_{ref}=-dudv.\label{eq:ReferenceMetric}
\end{equation}
Let $u_{1}(v)$ be a function of $v$ and let $\gamma:[0,a)\rightarrow\{u\ge u_{1}(v)\}\subset\mathcal{M}$
with $\gamma(0)\in\{u=u_{1}(v)\}$ be a null geodesic with angular
momentum $l>0$. Then, using the equations of motion (\ref{eq:NewNullGeodesicsSphericalSymmetry})
combined with the null shell relation (\ref{eq:NullShellNewAngularMomentum}),
we infer that, for all $s\in[0,a)$:

\begin{align}
\log\big(\Omega^{2}\dot{\gamma}^{u}\big)(s)-\log\big(\Omega^{2}\dot{\gamma}^{u}\big)(0)= & \int_{\text{\textgreek{g}}([0,s]))}\Big(\partial_{v}\log(\Omega^{2})-2\frac{\partial_{v}r}{r}\Big)\,dv=\label{eq:AllhMiaEksiswshGiaGevdaisiakes}\\
= & \int_{v(\gamma(0))}^{v(\gamma(s))}\int_{u_{1}(v)}^{u(\gamma(s_{v}))}\Big(\partial_{u}\partial_{v}\log(\Omega^{2})-2\partial_{u}\frac{\partial_{v}r}{r}\Big)\,dudv+\nonumber \\
 & +\int_{v(\gamma(0))}^{v(\gamma(s))}\big(\partial_{v}\log(\Omega^{2})-2\frac{\partial_{v}r}{r}\big)(u_{1}(v),v)\,dv,\nonumber 
\end{align}
where $s_{\bar{v}}$ is defined as the value of the parameter $s$
determined by the condition
\begin{equation}
v(\gamma(s_{\bar{v}}))=\bar{v}.\label{eq:s_v}
\end{equation}
Substituting the relations (\ref{eq:EquationROutside}) and (\ref{eq:OneMoreUsefulEquationOmega})
for $\partial_{u}\partial_{v}r$ and $\partial_{u}\partial_{v}\log\Omega^{2}$
in the right hand side of (\ref{eq:AllhMiaEksiswshGiaGevdaisiakes})
and recalling the definition (\ref{eq:RenormalisedNewHawkingMass})
of $\tilde{m}$, we readily infer from (\ref{eq:AllhMiaEksiswshGiaGevdaisiakes})
the equivalent relation: 
\begin{align}
\log\big(\Omega^{2}\dot{\gamma}^{u}\big)(s)-\log\big(\Omega^{2}\dot{\gamma}^{u}\big)(0)= & \int_{v(\gamma(0))}^{v(\gamma(s))}\int_{u_{1}(v)}^{u(\gamma(s_{v}))}\Big(\frac{1}{2}\frac{\frac{6\tilde{m}}{r}-1}{r^{2}}\Omega^{2}-24\pi T_{uv}\Big)\,dudv+\label{eq:XrhsimhGewdaisiakh-U}\\
 & +\int_{v(\gamma(0))}^{v(\gamma(s))}\big(\partial_{v}\log(\Omega^{2})-2\frac{\partial_{v}r}{r}\big)(u_{1}(v),v)\,dv.\nonumber 
\end{align}

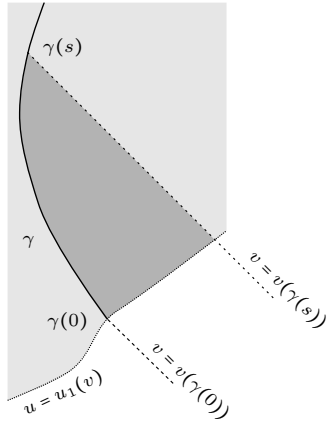
\begin{figure}[h] 
\centering 
\scriptsize
\begingroup%
  \makeatletter%
  \providecommand\color[2][]{%
    \errmessage{(Inkscape) Color is used for the text in Inkscape, but the package 'color.sty' is not loaded}%
    \renewcommand\color[2][]{}%
  }%
  \providecommand\transparent[1]{%
    \errmessage{(Inkscape) Transparency is used (non-zero) for the text in Inkscape, but the package 'transparent.sty' is not loaded}%
    \renewcommand\transparent[1]{}%
  }%
  \providecommand\rotatebox[2]{#2}%
  \newcommand*\fsize{\dimexpr\f@size pt\relax}%
  \newcommand*\lineheight[1]{\fontsize{\fsize}{#1\fsize}\selectfont}%
  \ifx\svgwidth\undefined%
    \setlength{\unitlength}{150bp}%
    \ifx\svgscale\undefined%
      \relax%
    \else%
      \setlength{\unitlength}{\unitlength * \real{\svgscale}}%
    \fi%
  \else%
    \setlength{\unitlength}{\svgwidth}%
  \fi%
  \global\let\svgwidth\undefined%
  \global\let\svgscale\undefined%
  \makeatother%
  \begin{picture}(1,1.1)%
    \lineheight{1}%
    \setlength\tabcolsep{0pt}%
    \put(0,0){\includegraphics[width=\unitlength,page=1]{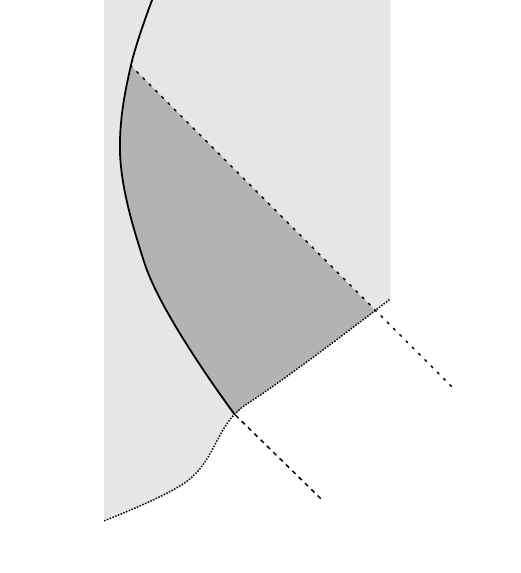}}%
    \put(0.23762377,0.49381178){\color[rgb]{0,0,0}\makebox(0,0)[lt]{\lineheight{1.25}\smash{\begin{tabular}[t]{l}$\gamma$\end{tabular}}}}%
    \put(0.28836633,0.97524738){\color[rgb]{0,0,0}\makebox(0,0)[lt]{\lineheight{1.25}\smash{\begin{tabular}[t]{l}$\gamma (s)$\end{tabular}}}}%
    \put(0.29950493,0.28589105){\color[rgb]{0,0,0}\makebox(0,0)[lt]{\lineheight{1.25}\smash{\begin{tabular}[t]{l}$\gamma (0)$\end{tabular}}}}%
    \put(0.24679985,0.06206183){\color[rgb]{0,0,0}\rotatebox{25}{\makebox(0,0)[lt]{\lineheight{1.25}\smash{\begin{tabular}[t]{l}$u=u_1(v)$\end{tabular}}}}}%
    \put(0.80420483,0.46035699){\color[rgb]{0,0,0}\rotatebox{-45}{\makebox(0,0)[lt]{\lineheight{1.25}\smash{\begin{tabular}[t]{l}$v=v(\gamma (s))$\end{tabular}}}}}%
    \put(0.55915533,0.23263432){\color[rgb]{0,0,0}\rotatebox{-45}{\makebox(0,0)[lt]{\lineheight{1.25}\smash{\begin{tabular}[t]{l}$v=v(\gamma (0))$\end{tabular}}}}}%
  \end{picture}%
\endgroup%
 
\caption{Formula (\ref{eq:XrhsimhGewdaisiakh-U}) expresses the change in the magnitude of $\Omega^{2}\dot{\gamma}^{u}$ for a future directed null geodesic $\gamma$ in terms of a spacetime integral over a region as depiced above. \label{fig:Geodesic_Integration}}
\end{figure}

For null geodesics $\gamma:[0,a)\rightarrow\{v\ge v_{1}(u)\}$ with
$\gamma(0)\in\{v=v_{1}(u)\}$, the analogue of formula (\ref{eq:XrhsimhGewdaisiakh-U})
is
\begin{align}
\log\big(\Omega^{2}\dot{\gamma}^{v}\big)(s)-\log\big(\Omega^{2}\dot{\gamma}^{v}\big)(0)= & \int_{u(\gamma(0))}^{u(\gamma(s))}\int_{v_{1}(u)}^{v(\gamma(s_{u}))}\Big(\frac{1}{2}\frac{\frac{6\tilde{m}}{r}-1}{r^{2}}\Omega^{2}-24\pi T_{uv}\Big)\,dvdu+\label{eq:XrhsimhGewdaisiakh-V}\\
 & +\int_{u(\gamma(0))}^{u(\gamma(s))}\big(\partial_{u}\log(\Omega^{2})-2\frac{\partial_{u}r}{r}\big)(u,v_{1}(u))\,du,\nonumber 
\end{align}
where $s_{u}$ is defined by the relation
\begin{equation}
u(\gamma(s_{\bar{u}}))=\bar{u}.\label{eq:s_u}
\end{equation}
See also \cite{MoschidisVlasovWellPosedness}.

\paragraph*{Asymptotically AdS solutions and the reflecting boundary condition
at infinity}

Let $(\mathcal{M},g;f)$ be a spherically symmetric solution of (\ref{eq:The mainSystemEinsteinVlasov}),
such that, in addition, $(\mathcal{M},g)$ is asymptotically AdS,
in accordance with the Definition \ref{def:AsADS}. In this case,
the following quantities will be useful as renormalised substitutes
of $r$, $\Omega^{2}$ and $T_{\mu\nu}$ near conformal infinity (see
Section \ref{subsec:Asymptotically-AdS-spacetimes}):
\begin{align}
\rho & \doteq\tan^{-1}\big(\sqrt{-\frac{\Lambda}{3}}r\big),\label{eq:RenormalisedQuantities}\\
\widetilde{\Omega}^{2} & \doteq\frac{\Omega^{2}}{1-\frac{1}{3}\Lambda r^{2}},\nonumber \\
\tau_{\mu\nu} & \doteq r^{2}T_{\mu\nu}.\nonumber 
\end{align}
From (\ref{eq:EquationROutside}) and (\ref{eq:OneMoreUsefulEquationOmega}),
it readily follows that $(\rho,\widetilde{\Omega}^{2},\tau_{\mu\nu})$
satisfy the relations
\begin{align}
\partial_{u}\partial_{v}\rho & =-\frac{1}{2}\sqrt{-\frac{\Lambda}{3}}\frac{\tilde{m}}{r^{2}}\frac{1-\frac{2}{3}\Lambda r^{2}}{1-\frac{1}{3}\Lambda r^{2}}\widetilde{\Omega}^{2}+4\pi\sqrt{-\frac{\Lambda}{3}}\frac{1}{r-\frac{1}{3}\Lambda r^{3}}\tau_{uv},\label{eq:RenormalisedEquations}\\
\partial_{u}\partial_{v}\log(\widetilde{\Omega}^{2}) & =\frac{\tilde{m}}{r}\Big(\frac{1}{r^{2}}+\frac{1}{3}\Lambda\frac{\Lambda r^{2}-1}{1-\frac{1}{3}\Lambda r^{2}}\Big)\widetilde{\Omega}^{2}-16\pi\frac{1-\frac{1}{2}\Lambda r^{2}}{1-\frac{1}{3}\Lambda r^{2}}r^{-2}\tau_{uv}.\nonumber 
\end{align}

In the asymptotically AdS setting, it is natural to study the system
(\ref{eq:RequationFinal})\textendash (\ref{NullShellFinal}) with
boundary conditions imposed for $f$ on $\mathcal{I}$. In this paper,
we will consider the \emph{reflecting} boundary condition. Defined
in terms of the reflection of null geodesics off $\mathcal{I}^{(2+1)}$,
the reflecting boundary condition can be formulated as follows (see
\cite{MoschidisVlasovWellPosedness} for more details): 
\begin{defn}
\label{def:ReflectingBoundaryCondition} Let $(\mathcal{M},g)$ be
as in Definition \ref{def:AsADS}, and let $f$ be a smooth massless
Vlasov field on $T\mathcal{M}$, as defined in Section \ref{subsec:VlasovEquations}.
We will say that $f$ satisfies the \emph{reflecting} boundary condition
on conformal infinity if, for any pair of future directd null geodesics
$\gamma:(a,+\infty)\rightarrow\mathcal{M}$ and $\gamma_{\models}:(-\infty,b)\rightarrow\mathcal{M}$
such that $\gamma_{\models}$ is the reflection of $\gamma$ off conformal
infinity $\mathcal{I}^{(2+1)}$, according to Definition 2.2 in \cite{MoschidisVlasovWellPosedness},
$f$ satisfies 
\begin{equation}
f|_{(\gamma,\dot{\gamma})}=f|_{(\gamma_{\models},\dot{\gamma}_{\models})},\label{eq:ReflectingCondition}
\end{equation}
where $f|_{(\gamma,\dot{\gamma})}$ is the (constant) value of $f$
along the curve $(\gamma,\dot{\gamma})$ in $T\mathcal{M}$. 
\end{defn}
\begin{rem*}
Equivalently, $f$ satisfies the reflecting condition on $\mathcal{I}^{(3+1)}$
if $f$ is constant along the trajectory of $(\gamma,\dot{\gamma})$
for any future directed, affinely paramterised null geodesic $\gamma$
which is \emph{maximally extended through reflections, }in accordance
with Definition 2.3 in \cite{MoschidisVlasovWellPosedness} (see also
Figure \ref{fig:Maximal_Geodesic}).
\end{rem*}
\begin{figure}[h] 
\centering 
\begingroup%
  \makeatletter%
  \providecommand\color[2][]{%
    \errmessage{(Inkscape) Color is used for the text in Inkscape, but the package 'color.sty' is not loaded}%
    \renewcommand\color[2][]{}%
  }%
  \providecommand\transparent[1]{%
    \errmessage{(Inkscape) Transparency is used (non-zero) for the text in Inkscape, but the package 'transparent.sty' is not loaded}%
    \renewcommand\transparent[1]{}%
  }%
  \providecommand\rotatebox[2]{#2}%
  \newcommand*\fsize{\dimexpr\f@size pt\relax}%
  \newcommand*\lineheight[1]{\fontsize{\fsize}{#1\fsize}\selectfont}%
  \ifx\svgwidth\undefined%
    \setlength{\unitlength}{150bp}%
    \ifx\svgscale\undefined%
      \relax%
    \else%
      \setlength{\unitlength}{\unitlength * \real{\svgscale}}%
    \fi%
  \else%
    \setlength{\unitlength}{\svgwidth}%
  \fi%
  \global\let\svgwidth\undefined%
  \global\let\svgscale\undefined%
  \makeatother%
  \begin{picture}(1,1.75)%
    \lineheight{1}%
    \setlength\tabcolsep{0pt}%
    \put(0,0){\includegraphics[width=\unitlength,page=1]{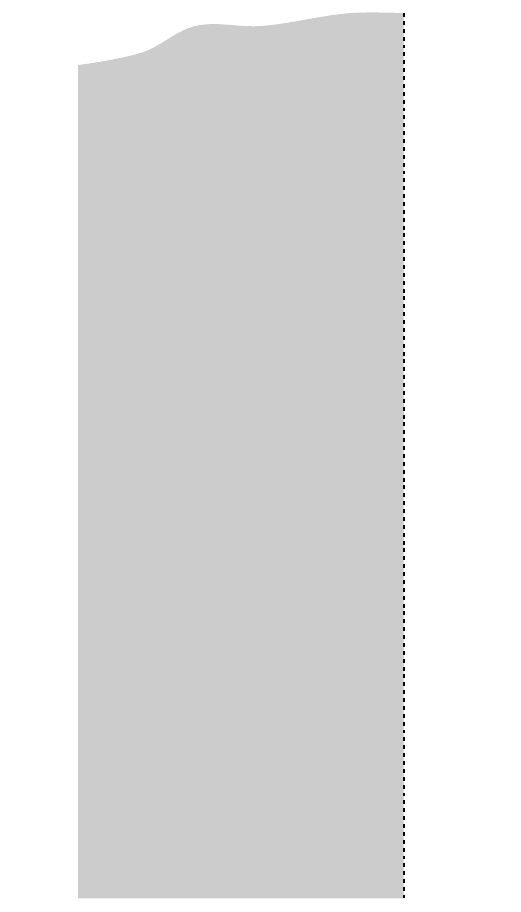}}%
    \put(0.83582917,0.75705444){\color[rgb]{0,0,0}\makebox(0,0)[lt]{\lineheight{1.25}\smash{\begin{tabular}[t]{l}$\mathcal{I}$\end{tabular}}}}%
    \put(0,0){\includegraphics[width=\unitlength,page=2]{Maximal_geodesic.pdf}}%
    \put(0.41367573,1.41862625){\color[rgb]{0,0,0}\makebox(0,0)[lt]{\lineheight{1.25}\smash{\begin{tabular}[t]{l}$\gamma_N$\end{tabular}}}}%
    \put(0.37252476,0.75804458){\color[rgb]{0,0,0}\makebox(0,0)[lt]{\lineheight{1.25}\smash{\begin{tabular}[t]{l}$\gamma_{N-1}$\end{tabular}}}}%
    \put(0.38551977,0.05414621){\color[rgb]{0,0,0}\makebox(0,0)[lt]{\lineheight{1.25}\smash{\begin{tabular}[t]{l}$\gamma_{N-2}$\end{tabular}}}}%
  \end{picture}%
\endgroup%
 
\caption{Schematic depiction of the components $\gamma_n$ of a maximally  extended geodesic $\gamma=\bigcup_{n=0}^{N}\gamma_n$ through reflections off conformal infinity, as defined in \cite{MoschidisVlasovWellPosedness}. Each component $\gamma_n$ is the reflection off $\mathcal{I}$ of $\gamma_{n-1}$. A massless Vlasov field $f$ satisfying the reflecting boundary condition on $\mathcal{I}$ is constant along any such maximally extended null geodesic. \label{fig:Maximal_Geodesic}}
\end{figure}

The following Lemma is a trivial corollary of the relations (\ref{eq:TildeUMaza})\textendash (\ref{eq:TildeVMaza})
for $\tilde{m}$, the condition (\ref{eq:1/r0}) on conformal infinity
$\mathcal{I}$ and the reflecting boundary condition (\ref{eq:ReflectingCondition})
for $f$:
\begin{lem}
Let $(r,\Omega^{2},f)$ be an asymptotically AdS solution of (\ref{eq:RequationFinal})\textendash (\ref{NullShellFinal})
as above, satisfying on $\mathcal{I}$ the reflecting boundary condition,
in accordance with Definition \ref{def:ReflectingBoundaryCondition}.
Then, the renormalised Hawking mass $\tilde{m}$ is constant along
$\mathcal{I}$:
\begin{equation}
(\partial_{u}+\partial_{v})\tilde{m}|_{\mathcal{I}}=0.\label{eq:ConservationHawkingMassAtInfinity}
\end{equation}
\end{lem}
See also Lemma 2.1 in \cite{MoschidisVlasovWellPosedness}.

\section{\label{sec:Well-posedness-and-extension}The asymptotically AdS characteristic
initial-boudary value problem}

In this section, we will review the well-posedness results regarding
the characteristic-boundary initial value problem for (\ref{eq:RequationFinal})\textendash (\ref{NullShellFinal})
established in \cite{MoschidisVlasovWellPosedness}. In particular,
we will introduce the notion of a smoothly compatible, characteristic
asymptotically AdS initial data set for the system (\ref{eq:RequationFinal})\textendash (\ref{NullShellFinal})
and we will present a result on the existense and uniqueness of a
maximal future development for (\ref{eq:RequationFinal})\textendash (\ref{NullShellFinal})
when refelcting boundary conditions are imposed on $\mathcal{I}$.
We will also state a few continuation criteria for smooth solutions
of (\ref{eq:RequationFinal})\textendash (\ref{NullShellFinal}),
which will be crucial for the constructions involved in the proof
of Theorem \ref{thm:TheTheoremIntro}. We will end this section by
presenting a Cauchy stability statement for the trivial solution of
(\ref{eq:RequationFinal})\textendash (\ref{NullShellFinal}) in a
scale invariant initial data topology, which will later allow us to
address the AdS instability conjecture in a low regularity setting
in Section \ref{sec:The-main-result}. The proofs of the results appearing
in this section are presented in detail in our companion paper \cite{MoschidisVlasovWellPosedness}.

\subsection{\label{subsec:Asymptotically-AdS-Initial-Data}Smoothly compatible
characteristic initial data sets for (\ref{eq:RequationFinal})\textendash (\ref{NullShellFinal})}

In this paper, the study of the dynamics of the system (\ref{eq:RequationFinal})\textendash (\ref{NullShellFinal})
in the asymptotically AdS setting will take place in the the framework
of the characteristic initial-boundary value problem, with initial
data prescribed at $u=0$, satisfying the constraint equation (\ref{eq:ConstraintVFinal}).
We will consider the following class of initial data which is compatible
with smoothness of the associated development for (\ref{eq:RequationFinal})\textendash (\ref{NullShellFinal})
at the axis $\mathcal{Z}$ and at conformal infinity $\mathcal{I}$
(see also \cite{MoschidisVlasovWellPosedness}):
\begin{defn}
\label{def:SmoothlyCompatibleInitialData} For a given $v_{\mathcal{I}}>0$,
let $r_{/},\Omega_{/}:[0,v_{\mathcal{I}})\rightarrow[0,+\infty)$
and $\bar{f}_{/}:(0,v_{\mathcal{I}})\times[0,+\infty)^{2}\rightarrow[0,+\infty)$
be smooth functions. The quadruplet $(r_{/},\Omega_{/}^{2},\bar{f}_{/};v_{\mathcal{I}})$
(simplified to $(r_{/},\Omega_{/}^{2},\bar{f}_{/})$ when the value
of $v_{\mathcal{I}}$ is clear from the context) will be called a
\emph{smoothly compatible asymptotically AdS initial data set }for
the system (\ref{eq:RequationFinal})\textendash (\ref{NullShellFinal})
if it satisfies the following conditions: 

\begin{enumerate}

\item The functions $(r_{/},\Omega_{/}^{2},\bar{f}_{/})$ satisfy
on $(0,v_{\mathcal{I}})$ the constraint equation (\ref{eq:ConstraintVFinal}),
with $T_{vv}$ is defined in terms of $(r_{/},\Omega_{/}^{2},\bar{f}_{/})$
by the second relation in (\ref{eq:ComponentsStressEnergy}) with
$\bar{f}(u,v;p^{u},p^{v},l)\Big|_{\Omega^{2}p^{u}p^{v}=\frac{l^{2}}{r^{2}}}$
replaced by $\bar{f}_{/}(v;p^{u},l)$.

\item At $v=0$, the functions $r_{/}$, $\Omega_{/}^{2}$ extend
smoothly and satisfy
\begin{equation}
\Omega_{/}^{2}(0)>0
\end{equation}
and
\begin{equation}
r_{/}(0)=0.
\end{equation}

\item The functions $1/r_{/}$ and $r_{/}^{-2}\Omega_{/}^{2}$ extend
smoothly on $v=v_{\mathcal{I}}$ and satisfy
\begin{equation}
1/r_{/}(v_{\mathcal{I}})=0,
\end{equation}
\begin{equation}
r_{/}^{-2}\Omega_{/}^{2}(v_{\mathcal{I}})>0,
\end{equation}
\begin{equation}
\partial_{v}(1/r_{/})(v_{\mathcal{I}})<0.\label{eq:PositiveDvRInitially}
\end{equation}
Furthermore, for any $\bar{p}\ge0$ and $l\ge0$, the function $\bar{f}_{/}(v,\Omega_{/}^{-2}(v)\bar{p},l)$
extends smoothly on $v=v_{\mathcal{I}}$.

\item The functions $r_{/},\text{ }\Omega_{/}^{2},\text{ }\bar{f}_{/}$
satisfy Conditions 1\textendash 3 of Definition 3.5 of \cite{MoschidisVlasovWellPosedness}
on smooth compatibility at $v=0$ and $v=v_{\mathcal{I}}$.

\end{enumerate}

We will also denote by $\mathfrak{B}_{0}$ the set of all smoothly
compatible, asymptotically AdS initial data sets $(r_{/},\Omega_{/}^{2},\bar{f}_{/};v_{\mathcal{I}})$
for (\ref{eq:RequationFinal})\textendash (\ref{NullShellFinal})
which have \emph{bounded support in phase space}, i.\,e.~satisfy
for every $v\in(0,v_{\mathcal{I}})$ and every $l\ge0$:
\begin{equation}
\sup_{p^{u}\in supp(f_{/}(v;\cdot,l))}\Big(\Omega_{/}^{2}\big(p^{u}+\frac{l^{2}}{\Omega_{/}^{2}r_{/}^{2}p^{u}}\big)\Big)\le C\label{eq:BoundedSupportDefinition}
\end{equation}
for some constant $C<+\infty$ \emph{independent} of $v$, $l$.
\end{defn}
For a more detailed discussion on Definition \ref{def:SmoothlyCompatibleInitialData}
and the properties of initial data sets in $\mathfrak{B}_{0}$, see
Definitions 3.4 and 3.5 in \cite{MoschidisVlasovWellPosedness}.

The following remarks regarding Definition \ref{def:SmoothlyCompatibleInitialData}
were also discussed in \cite{MoschidisVlasovWellPosedness} (see Section
3.2 of \cite{MoschidisVlasovWellPosedness}):
\begin{itemize}
\item Under a gauge transformation of the $(u,v)$-plane of the form $(u,v)\rightarrow(u',v')=(U(u),V(v))$,
$\frac{dU}{du},\frac{dV}{dv}\neq0$, solutions $(r,\Omega^{2},f)$
to (\ref{eq:RequationFinal})\textendash (\ref{NullShellFinal}) transform
according to (\ref{eq:GeneralGaugeTransformationSpacetime}). Considering
the restriction of such a transformation with $U(0)=V(0)=0$ at the
initial data $(r_{/},\Omega_{/}^{2},\bar{f}_{/})$ induced on $\{u=0\}$,
we infer that $(r_{/},\Omega_{/}^{2},\bar{f}_{/})$ transform as follows:
\begin{align}
r_{/}^{\prime}(v') & \doteq r_{/}(v),\label{eq:InitialDataGaugeTransformation}\\
(\Omega_{/}^{\prime})^{2}(v') & \doteq\frac{1}{\frac{dU}{du}(0)\cdot\frac{dV}{dv}(v)}\Omega_{/}^{2}(v),\nonumber \\
\bar{f}_{/}^{\prime}(v';\frac{dU}{du}(0)\cdot p,l) & \doteq\bar{f}_{/}(v;p,l).\nonumber 
\end{align}
\item In this paper, following the conventions of \cite{MoschidisVlasovWellPosedness},
we will study asymptotically AdS solutions $(r,\Omega^{2},f)$ of
(\ref{eq:RequationFinal})\textendash (\ref{NullShellFinal}) under
the gauge condition that $r=0$ on $\{u=v\}$ and $r=\infty$ on $\{u=v-v_{\mathcal{I}}\}$
(see Definition \ref{def:DevelopmentSets} in the next section). For
a gauge transformation $(u,v)\rightarrow(U(u),V(v))$ to preserve
this condition, it is necessary that 
\begin{equation}
U(v)=V(v)\text{ and }U(v-v_{\mathcal{I}})=V(v)-v_{\mathcal{I}}.\label{eq:GaugeConditionAxisInfinity}
\end{equation}
At the level of the initial data transformation at $u=0$ associated
to the coordinate transformation $v\rightarrow V(v)$ and the parameter
$\frac{dU}{du}(0)$, (\ref{eq:GaugeConditionAxisInfinity}) implies
that 
\begin{equation}
\frac{dU}{du}(0)=\frac{dV}{dv}(0)\text{ and }V(0)=0\text{, }V(v_{\mathcal{I}})=v_{\mathcal{I}}.\label{eq:ConditionForAxisANdInfinityInitialData}
\end{equation}
Note that, in general, the property of an initial data set $(r_{/},\Omega_{/}^{2},\bar{f}_{/};v_{\mathcal{I}})$
being smoothly compatible is gauge dependent. In particular, when
the transformed initial data set $(r_{/}^{\prime},(\Omega_{/}^{\prime}),\bar{f}_{/}^{\prime};V(v_{\mathcal{I}}))$
is also smoothly compatible, Condition 4 of Definition \ref{def:SmoothlyCompatibleInitialData}
implies that a certain relation holds between $\frac{d^{k}V}{(dv)^{k}}(0)$
and $\frac{d^{k}V}{(dv)^{k}}(v_{\mathcal{I}})$ for all $k\in\mathbb{N}$;
this relation does not hold, in general, for gauge transformations
as above, even when $V$ satisfies the (necessary) condition $V\in C^{\infty}([0,v_{\mathcal{I}}])$.
See also the discussion in Section 3.2 of \cite{MoschidisVlasovWellPosedness}. 
\item Let $(r_{/},\Omega_{/}^{2},\bar{f}_{/};v_{\mathcal{I}})\in\mathfrak{B}_{0}$.
We will define the function $(\partial_{u}r)_{/}$ on $(0,v_{\mathcal{I}})$
(coinciding formally with $\partial_{u}r|_{u=0}$ in a development
of $(r_{/},\Omega_{/}^{2},\bar{f}_{/};v_{\mathcal{I}})$ solving (\ref{eq:RequationFinal})\textendash (\ref{NullShellFinal})
and satisfying $r=0$ on $\{u=v\}$) by integrating equation (\ref{eq:RequationFinal})
in $v$. In particular: 
\begin{equation}
(\partial_{u}r)_{/}(v)\doteq\frac{1}{2r_{/}(v)}\int_{0}^{v}\Big(-\frac{1}{2}(1-\Lambda r_{/}^{2})\Omega_{/}^{2}+8\pi r_{/}^{2}(T_{/})_{uv}\Big)\,d\bar{v},\label{eq:DuRInitially}
\end{equation}
where $(T_{uv})_{/}$ is defined in terms of $(r_{/},\Omega_{/}^{2},\bar{f}_{/})$
by (\ref{eq:ComponentsStressEnergy}) with $\frac{l^{2}}{\Omega_{/}^{2}r_{/}^{2}p^{u}}$
in place of $p^{v}$ and $\bar{f}_{/}\big(v;p^{u},l\big)$ in place
of $\bar{f}(0,v;p^{u},p^{v},l)\Big|_{\Omega^{2}p^{u}p^{v}=\frac{l^{2}}{r^{2}}}$.
We will also define the functions $m_{/},\tilde{m}_{/}$ on $(0,v_{\mathcal{I}})$
through the relations (\ref{eq:DefinitionHereHawkingMass}), (\ref{eq:RenormalisedNewHawkingMass}),
as well as the energy-momentum components $(T_{/})_{\mu\nu}$ through
the relation (\ref{eq:ComponentsStressEnergy}) (again, with $\frac{l^{2}}{\Omega_{/}^{2}r_{/}^{2}p^{u}}$
in place of $p^{v}$ and $\bar{f}_{/}\big(v;p^{u},l\big)$ in place
of $\bar{f}(0,v;p^{u},p^{v},l)\Big|_{\Omega^{2}p^{u}p^{v}=\frac{l^{2}}{r^{2}}}$).
\item For $(r_{/},\Omega_{/}^{2},\bar{f}_{/};v_{\mathcal{I}})\in\mathfrak{B}_{0}$
as above, the functions $(T_{/})_{\mu\nu}$, $m_{/}$ and $\tilde{m}_{/}$
extend smoothly on $v=0$, with 
\[
m_{/},\tilde{m}_{/}=O(r_{/}^{3}).
\]
Furthermore, the condition (\ref{eq:BoundedSupportDefinition}) implies
that 
\begin{equation}
\lim_{v\rightarrow v_{\mathcal{I}}^{-}}\tilde{m}(v)<+\infty,\label{eq:FiniteTotalMassInitialData}
\end{equation}
while (\ref{eq:ConstraintVFinal}) and (\ref{eq:PositiveDvRInitially})
imply that 
\begin{equation}
\inf_{v\in(0,v_{\mathcal{I}})}\partial_{v}r_{/}(v)>0\label{eq:NonTrappingForAsymptoticallyAdSInitialData}
\end{equation}
(see also Remark 2 in Section 3.2 of \cite{MoschidisVlasovWellPosedness}). 
\end{itemize}
The following normalised gauge condition for initial data sets $(r_{/},\Omega_{/}^{2},\bar{f}_{/};v_{\mathcal{I}})$
was introduced in \cite{MoschidisVlasovWellPosedness} for the purpose
of fixing a simple representation of the trivial initial data set
$(r_{AdS/},\Omega_{AdS/}^{2},0;v_{\mathcal{I}})$ (see Definition
3.6 in \cite{MoschidisVlasovWellPosedness}):
\begin{defn}
\label{def:GaugeNormalisation} Let $(r_{/},\Omega_{/}^{2},\bar{f}_{/};v_{\mathcal{I}})$
be a smoothly compatible, asymptotically AdS initial data set for
(\ref{eq:RequationFinal})\textendash (\ref{NullShellFinal}), as
in Definition \ref{def:SmoothlyCompatibleInitialData}. Let also $v\rightarrow v'=V(v)$
(with $V\in C^{\infty}([0,v_{\mathcal{I}}])$), $(r_{/},\Omega_{/}^{2},\bar{f}_{/};v_{\mathcal{I}})\rightarrow(r_{/}^{\prime},(\Omega_{/}^{\prime}),\bar{f}_{/}^{\prime};v_{\mathcal{I}})$,
be a gauge transformation, defined by (\ref{eq:InitialDataGaugeTransformation}),
satisfying the condition (\ref{eq:ConditionForAxisANdInfinityInitialData}).
We will say that $(r_{/}^{\prime},(\Omega_{/}^{\prime}),\bar{f}_{/}^{\prime};v_{\mathcal{I}})$
satisfies the \emph{normalised gauge condition} if 
\begin{equation}
\frac{\partial_{v}r_{/}^{\prime}}{1-\frac{1}{3}\Lambda(r_{/}^{\prime})^{2}}(v)=\frac{(\Omega_{/}^{\prime})^{2}}{4\partial_{v}r_{/}^{\prime}}(v)\text{ for }v\in(0,v_{\mathcal{I}}).\label{eq:NormalisedGaugeCondition}
\end{equation}
In this case, we will say that $(r_{/},\Omega_{/}^{2},\bar{f}_{/};v_{\mathcal{I}})\rightarrow(r_{/}^{\prime},(\Omega_{/}^{\prime}),\bar{f}_{/}^{\prime};v_{\mathcal{I}})$
is a gauge normalising transformation.
\end{defn}
We should make the following remarks regarding Definition \ref{def:GaugeNormalisation}:
\begin{itemize}
\item It can be readily shown (see Lemma 3.2 in \cite{MoschidisVlasovWellPosedness})
that, for any $(r_{/},\Omega_{/}^{2},\bar{f}_{/};v_{\mathcal{I}})\in\mathfrak{B}_{0}$,
there exists a unique gauge normalising transformation as in Definition
\ref{def:GaugeNormalisation}. The trivial (normalised) initial data
set $(r_{AdS/},\Omega_{AdS/}^{2},0;\sqrt{-\frac{3}{\Lambda}}\pi)$
is expressed as: 
\begin{align}
r_{AdS/}(v) & =\sqrt{-\frac{3}{\Lambda}}\tan\Big(\frac{1}{2}\sqrt{-\frac{\Lambda}{3}}v\Big),\label{eq:AdSMetricValues}\\
\Omega_{AdS/}^{2}(v) & =1-\frac{1}{3}\Lambda r_{AdS}^{2}(v),\nonumber 
\end{align}
which are the data induced at $u=0$ by the AdS metric expressed in
the standard double null coordinate chart (\ref{eq:DoubleNullPairAdS}).
For different values of the endpoint parameter $v_{\mathcal{I}}>0$,
we obtain by rescaling: 
\begin{align}
r_{AdS/}^{(v_{\mathcal{I}})}(v) & =r_{AdS/}\Big(\sqrt{-\frac{3}{\Lambda}}\pi\frac{v}{v_{\mathcal{I}}}\Big),\label{eq:AdSMetricValuesRescaled1}\\
\big(\Omega_{AdS/}^{(v_{\mathcal{I}})}\big)^{2}(v) & =-\frac{3}{\Lambda}\frac{\pi^{2}}{v_{\mathcal{I}}^{2}}\Omega_{AdS/}^{2}\Big(\sqrt{-\frac{3}{\Lambda}}\pi\frac{v}{v_{\mathcal{I}}}\Big).\nonumber 
\end{align}
\item By integrating the constraint equation (\ref{eq:ConstraintVFinal}),
we infer that the gauge condition (\ref{eq:NormalisedGaugeCondition})
is equivalent to 
\begin{equation}
\frac{\partial_{v}r_{/}}{1-\frac{1}{3}\Lambda r_{/}^{2}}(v)=\frac{1}{2a}\exp\Big(4\pi\int_{0}^{v}\frac{r_{/}(T_{/})_{vv}}{(\partial_{v}r_{/})^{2}}(\bar{v})\,(\partial_{v}r_{/})d\bar{v}\Big),\label{eq:FormulaForDvR/2}
\end{equation}
where 
\begin{equation}
a\doteq\sqrt{-\frac{\Lambda}{3}}\frac{1}{\pi}\int_{0}^{v_{\mathcal{I}}}\exp\Big(4\pi\int_{0}^{v}\frac{r_{/}(T_{/})_{vv}}{(\partial_{v}r_{/})^{2}}(\bar{v})\,(\partial_{v}r_{/})d\bar{v}\Big)\,dv
\end{equation}
and $(T_{/})_{vv}$ is defined in terms of $(r_{/},\Omega_{/}^{2},\bar{f}_{/})$
by in (\ref{eq:ComponentsStressEnergy}) with $\bar{f}_{/}(v;p^{u},l)$
in place of $\bar{f}(u,v;p^{u},p^{v},l)\Big|_{\Omega^{2}p^{u}p^{v}=\frac{l^{2}}{r^{2}}}$.
Alternatively, the gauge condition (\ref{eq:NormalisedGaugeCondition})
can be expressed as
\begin{equation}
\frac{\partial_{v}r_{/}}{1-\frac{1}{3}\Lambda r_{/}^{2}}=-\frac{(\partial_{u}r)_{/}}{1-\frac{2m_{/}}{r_{/}}}.\label{eq:EqualDuRDvRInitially}
\end{equation}
\end{itemize}
A comparative advantage of considering the gauge condition (\ref{eq:NormalisedGaugeCondition})
when constructing asymptotically AdS initial data sets $(r_{/},\Omega_{/}^{2},\bar{f}_{/};v_{\mathcal{I}})$
for (\ref{eq:RequationFinal})\textendash (\ref{NullShellFinal})
is that (\ref{eq:NormalisedGaugeCondition}) allows one to completely
determine $(r_{/},\Omega_{/}^{2},\bar{f}_{/};v_{\mathcal{I}})$ in
terms of $v_{\mathcal{I}}$ and $\bar{f}_{/}$, which can be freely
prescribed. In particular, the following result was established in
\cite{MoschidisVlasovWellPosedness}:
\begin{lem}
\label{lem:FreeFVlasov} \emph{(Lemma 3.1 in \cite{MoschidisVlasovWellPosedness}).}
Let $v_{\mathcal{I}}>0$ and let $F:[0,v_{\mathcal{I}})\times[0,+\infty)^{2}\rightarrow[0,+\infty)$
be a smooth function such that $\text{supp}(F)$ is a compact subset
of $(0,v_{\mathcal{I}})\times(0,+\infty)^{2}$. There exists a unique
asymptotically AdS initial data set $(r_{/},\Omega_{/}^{2},\bar{f}_{/};v_{\mathcal{I}})$
for (\ref{eq:RequationFinal})\textendash (\ref{NullShellFinal})
satisfying Conditions 1\textendash 3 of Definition \ref{def:SmoothlyCompatibleInitialData}
and the gauge condition (\ref{eq:NormalisedGaugeCondition}), such
that 
\begin{equation}
\bar{f}_{/}(v;p^{u},l)=F\big(v;\text{ }\partial_{v}r_{/}(v)p^{u},\text{ }l\big).\label{eq:AlmostVlasovFieldInvariant Form}
\end{equation}

Assume, in addition, that $F$ satisfies the smallness condition 
\begin{equation}
\mathcal{M}[F]\doteq\int_{0}^{v_{\mathcal{I}}}\frac{r_{AdS/}^{(v_{\mathcal{I}})}(T_{AdS}^{(v_{\mathcal{I}})}[F])_{vv}}{\partial_{v}r_{AdS/}^{(v_{\mathcal{I}})}}(\bar{v})\,d\bar{v}<c_{0}\ll1,\label{eq:SmallnessConditionForConstruction}
\end{equation}
where $c_{0}>0$ is an absolute constant, $(T_{AdS}^{(v_{\mathcal{I}})}[F])_{vv}$
is defined by
\begin{equation}
(T_{AdS}^{(v_{\mathcal{I}})}[F])_{vv}(v)\doteq\frac{\pi}{2}\frac{1}{(r_{AdS/}^{(v_{\mathcal{I}})})^{2}(v)}\int_{0}^{+\infty}\int_{0}^{+\infty}p^{2}F(v;p,l)\,\frac{dp}{p}ldl\label{eq:DefinitionAdST}
\end{equation}
 and $r_{AdS/}^{(v_{\mathcal{I}})},\text{ }(\Omega_{AdS/}^{(v_{\mathcal{I}})})^{2}$
are the rescaled AdS metric coefficients given by (\ref{eq:AdSMetricValuesRescaled1}).
Then, the following bounds hold:
\begin{equation}
\Big|\frac{\partial_{v}r_{/}}{1-\frac{1}{3}\Lambda r_{/}^{2}}(v)-\frac{\partial_{v}r_{AdS/}^{(v_{\mathcal{I}})}}{1-\frac{1}{3}\Lambda(r_{AdS/}^{(v_{\mathcal{I}})})^{2}}(v)\Big|\le C\mathcal{M}[F]\text{ for all }v\in(0,v_{\mathcal{I}})\label{eq:InitialEstimateDifferenceFromAdS}
\end{equation}
and 
\begin{equation}
\int_{0}^{v_{\mathcal{I}}}\frac{r_{/}(T_{/})_{vv}}{\partial_{v}r_{/}}(\bar{v})\,d\bar{v}\le C\mathcal{M}[F],\label{eq:InitialSmallNorm}
\end{equation}
where $C>0$ is an absolute constant and is $(T_{/})_{vv}$ is defined
in terms of $(r_{/},\Omega_{/}^{2},\bar{f}_{/};v_{\mathcal{I}})$
by the second relation in (\ref{eq:ComponentsStressEnergy}) with
$\bar{f}(u,v;p^{u},p^{v},l)\Big|_{\Omega^{2}p^{u}p^{v}=\frac{l^{2}}{r^{2}}}$
replaced by $\bar{f}_{/}(v;p^{u},l)$.
\end{lem}
For the proof of Lemma \ref{lem:FreeFVlasov}, see \emph{\cite{MoschidisVlasovWellPosedness}.}

In general, a gauge normalising transformation (as in Definition \ref{def:GaugeNormalisation})
is not smoothly compatible; that is to say, an initial data set $(r_{/},\Omega_{/}^{2},\bar{f}_{/};v_{\mathcal{I}})$
expressed in a gauge where (\ref{eq:NormalisedGaugeCondition}) holds
will not, in general, satisfy Condition 4 of Definition \ref{def:SmoothlyCompatibleInitialData}
(see also the more detailed discussion in Sections 3.2 and 3.3 in
\cite{MoschidisVlasovWellPosedness}). 
\begin{rem*}
For the trivial initial data set $(r_{AdS/},\Omega_{AdS/}^{2},0;\sqrt{-\frac{3}{\Lambda}}\pi)$,
the gauge normalised form (\ref{eq:AdSMetricValues}) is also smoothly
compatible (see \cite{MoschidisVlasovWellPosedness}).
\end{rem*}
The following lemma, which is established in \cite{MoschidisVlasovWellPosedness},
shows that any initial data set $(r_{/},\Omega_{/}^{2},\bar{f}_{/};v_{\mathcal{I}})$
which satisfies Conditions 1\textendash 3 of Definition \ref{def:SmoothlyCompatibleInitialData}
when expressed in a gauge where (\ref{eq:NormalisedGaugeCondition})
holds can be gauge transformed into a smoothly compatible initial
data set, provided $\bar{f}_{/}$ is supported away from $v=0,v_{\mathcal{I}}$
and $l=0$. 
\begin{lem}
\label{lem:NormalisationToSmoothCompatibility} \emph{(Lemma 3.3 in
\cite{MoschidisVlasovWellPosedness}).} Let $(r_{/},\Omega_{/}^{2},\bar{f}_{/};v_{\mathcal{I}})$
satisfy Conditions 1\textendash 3 of Definition \ref{def:SmoothlyCompatibleInitialData},
as well as the normalised gauge condition (\ref{eq:NormalisedGaugeCondition}).
Assume that $\bar{f}_{/}$ is supported away from $v=0,v_{\mathcal{I}}$
and $l=0$, i.\,e.~there exists some $\bar{\delta}>0$, such that
$\bar{f}_{/}$ satisfies 
\begin{equation}
\bar{f}_{/}(v;p,l)=0\text{ for }v\in(0,\bar{\delta}]\cup[v_{\mathcal{I}}-\bar{\delta},v_{\mathcal{I}})\label{eq:ZeroFNearBoundary}
\end{equation}
and 
\begin{equation}
\bar{f}_{/}(v;p,l)=0\text{ for }l\in[0,\bar{\delta}].\label{eq:ZeroFSmallL}
\end{equation}
Then, there exists a gauge transformation $v\rightarrow v'(V)$, $(r_{/},\Omega_{/}^{2},\bar{f}_{/};v_{\mathcal{I}})\rightarrow(r_{/}^{\prime},(\Omega_{/}^{\prime})^{2},\bar{f}_{/}^{\prime};v_{\mathcal{I}})$
(of the form (\ref{eq:InitialDataGaugeTransformation})), satisfying
$V\in C^{\infty}[0,+\infty)$, (\ref{eq:ConditionForAxisANdInfinityInitialData})
and
\begin{equation}
V(v)=v\text{ for }v\le v_{\mathcal{I}}-\frac{1}{2}\bar{\delta},\label{eq:ConditionGaugeOnlAway}
\end{equation}
such that the transformed initial data set $(r_{/}^{\prime},(\Omega_{/}^{\prime})^{2},\bar{f}_{/}^{\prime};v_{\mathcal{I}})$
satisfies all of the Conditions 1\textendash 4 of Definition \ref{def:SmoothlyCompatibleInitialData}. 

Furthermore, for any $\varepsilon_{0}\in(0,1)$, the gauge transformation
can be chosen so that 
\begin{equation}
1-\varepsilon_{0}\le\frac{dV}{dv}(v)\le1+\varepsilon_{0}\text{ for }v\in[0,v_{\mathcal{I}}]\label{eq:FirstDerivativeTransformation}
\end{equation}
and 
\begin{equation}
\max_{v\in[0,v_{\mathcal{I}}]}\Big|\frac{d^{2}V}{(dv)^{2}}(v)\Big|\le\int_{0}^{v_{\mathcal{I}}}\Big(\frac{1-\Lambda r_{/}^{2}}{1-\frac{1}{3}\Lambda r_{/}^{2}}(T_{/})_{vv}+3(T_{/})_{uv}\Big)(v)\,dv+\frac{\varepsilon_{0}}{v_{\mathcal{I}}},\label{eq:SecondDerivativeTransformation}
\end{equation}
where $(T_{/})_{vv}$, $(T_{/})_{uv}$ are defined in terms of $(r_{/},\Omega_{/}^{2},\bar{f}_{/})$
by (\ref{eq:ComponentsStressEnergy}) with $\bar{f}_{/}(v;p^{u},l)$
in place of $\bar{f}(u,v;p^{u},p^{v},l)\Big|_{\Omega^{2}p^{u}p^{v}=\frac{l^{2}}{r^{2}}}$
and $\frac{l^{2}}{\Omega_{/}^{2}r_{/}^{2}p^{u}}$ in place of $p^{v}$. 
\end{lem}
For the proof of Lemma \ref{lem:NormalisationToSmoothCompatibility},
see \emph{\cite{MoschidisVlasovWellPosedness}.}
\begin{rem*}
The above lemma applies, in particular, to the normalised initial
data sets $(r_{/},\Omega_{/}^{2},\bar{f}_{/};v_{\mathcal{I}})$ provided
by Lemma \ref{lem:FreeFVlasov} for any function $F$ which is compactly
supported in $(0,v_{\mathcal{I}})\times(0,+\infty)^{2}$. 
\end{rem*}

\subsection{\label{subsec:Well-posedness-of-the}Well-posedness of the characteristic
initial-boundary value problem and the maximal future development}

In this section, we will formulate the notion of a \emph{development}
of a smoothly compatible, asymptotically AdS initial data set (see
Definition \ref{def:SmoothlyCompatibleInitialData}) with respect
to the system (\ref{eq:RequationFinal})\textendash (\ref{NullShellFinal}),
assuming the reflecting boundary condition on $\mathcal{I}$ (see
Definition \ref{def:ReflectingBoundaryCondition}). We will then present
a fundamental well-posedness result for the associated characteristic
initial\textendash boundary value problem for (\ref{eq:RequationFinal})\textendash (\ref{NullShellFinal}),
culminating in the statement of the existence and uniqueness of a
maximal future development for any given smoothly compatible, asymptotically
AdS initial data set with bounded support in phase space. The proofs
of the results presented in this section (together with a wider collection
of well-posedness results) can be found in our companion paper \cite{MoschidisVlasovWellPosedness}.

The following class of domains in the $(u,v)$-plane will appear naturally
as the class of domains of definition for solutions $(r,\Omega^{2},f)$
to the characteristic initial-boundary value problem for (\ref{eq:RequationFinal})\textendash (\ref{NullShellFinal});
see also \cite{MoschidisVlasovWellPosedness}.
\begin{defn}
\label{def:DevelopmentSets}For any given $v_{\mathcal{I}}>0$, we
will define $\mathscr{U}_{v_{\mathcal{I}}}$ to be the set of all
connected open domains $\mathcal{U}$ of the $(u,v)$-plane with piecewise
Lipschitz boundary $\partial\mathcal{U}$, with the property that
\begin{equation}
\partial\mathcal{U}=\mathcal{S}_{v_{\mathcal{I}}}\cup\gamma_{\mathcal{Z}}\cup\mathcal{I}\cup clos(\zeta),\label{eq:BoundaryOfU}
\end{equation}
where, for some $u_{\gamma_{\mathcal{Z}}},u_{\mathcal{I}}\in(0,+\infty]$,
\begin{equation}
\mathcal{S}_{v_{\mathcal{I}}}=\{0\}\times[0,v_{\mathcal{I}}],
\end{equation}
\begin{equation}
\gamma_{\mathcal{Z}}=\{u=v\}\cap\{0\le u<u_{\gamma_{\mathcal{Z}}}\},\label{eq:AxisForm}
\end{equation}
\begin{equation}
\mathcal{I}=\{u=v-v_{\mathcal{I}}\}\cap\{0\le u<u_{\mathcal{I}}\}\label{eq:InfinityForm}
\end{equation}
 and the Lipschitz curve $\zeta$ is achronal with respect to the
reference Lorentzian metric 
\begin{equation}
g_{ref}\doteq-dudv\label{eq:ComparisonUVMetric}
\end{equation}
on $\mathbb{R}^{2}$. In particular, the case $\zeta=\emptyset$ is
allowed. 
\end{defn}
\begin{figure}[h] 
\centering 
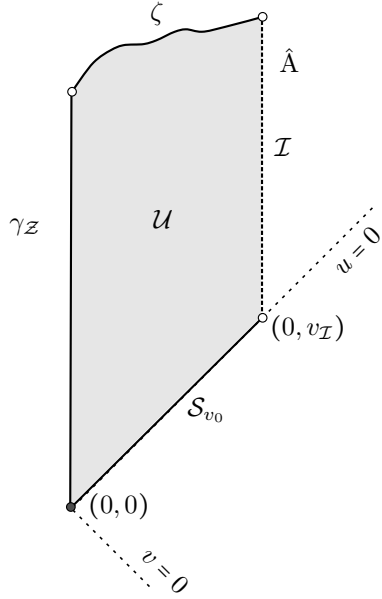 
\caption{Depicted above is a typical domain $\mathcal{U}\in\mathscr{U}_{v_{\mathcal{I}}}$. In the case when the boundary set $\zeta$ is empty, it is necessary that both $\gamma_{\mathcal{Z}}$ and $\mathcal{I}$ are unbounded (i.\,e.~extend all the way to $u+v=\infty$).}
\end{figure}
\begin{rem*}
In the case when $\zeta=\emptyset$ in (\ref{eq:BoundaryOfU}), it
is necessary that both $\text{\textgreek{g}}_{\mathcal{Z}}$ and $\mathcal{I}$
extend all the way to $u+v=+\infty$. 
\end{rem*}
We will define a \emph{future development} of an asymptotically AdS
characteristic initial data set for (\ref{eq:RequationFinal})\textendash (\ref{NullShellFinal})
with reflecting boundary conditions on $\mathcal{I}$ as follows:
\begin{defn}
\label{def:Development} For a given $v_{\mathcal{I}}>0$, let $(r_{/},\Omega_{/}^{2},\bar{f}_{/};v_{\mathcal{I}})$
be a smoothly compatible, asymptotically AdS initial data set for
the system (\ref{eq:RequationFinal})\textendash (\ref{NullShellFinal}),
as introduced by Definition \ref{def:SmoothlyCompatibleInitialData}.
A future development of $(r_{/},\Omega_{/}^{2},\bar{f}_{/};v_{\mathcal{I}})$
for (\ref{eq:RequationFinal})\textendash (\ref{NullShellFinal})
with reflecting boundary conditions on $\mathcal{I}$ will consist
of a domain $\mathcal{U}\subset\mathbb{R}^{2}$ belonging to the class
$\mathscr{U}_{v_{\mathcal{I}}}$ introduced in Definition \ref{def:DevelopmentSets},
together with a solution $(r,\Omega^{2},f)$ of the system (\ref{eq:RequationFinal})\textendash (\ref{NullShellFinal})
on $\mathcal{U}$, such that the following conditions hold:

\begin{enumerate}

\item $(\mathcal{U};r,\Omega^{2},f)$ is a smooth solution of (\ref{eq:RequationFinal})\textendash (\ref{NullShellFinal})
with smooth axis $\gamma_{\mathcal{Z}}$ and smooth conformal infinity
$\mathcal{I}$, in accordance with Definitions 3.1 and 3.2 of \cite{MoschidisVlasovWellPosedness}.

\item The solution $(r,\Omega^{2},f)$ induces the initial data $(r_{/},\Omega_{/}^{2},\bar{f}_{/};v_{\mathcal{I}})$
at $u=0$:
\begin{equation}
(r,\Omega^{2})(0,v)=(r_{/},\Omega_{/}^{2})(v)\label{eq:InitialROmegaForExistence}
\end{equation}
 and 
\begin{equation}
f(0,v;p^{u},p^{v},l)=\bar{f}_{/}(v;p^{u},l)\cdot\delta\Big(\Omega_{/}^{2}(v)p^{u}p^{v}-\frac{l^{2}}{r_{/}(v)}\Big).\label{eq:InitialFForExistence}
\end{equation}

\item The reflecting boundary condition (\ref{eq:ReflectingCondition})
is satisfied by $f$ along conformal infinity $\mathcal{I}$.

\end{enumerate}
\end{defn}
\begin{rem*}
For any smooth development $(\mathcal{U};r,\Omega^{2},f)$ as in Definition
\ref{def:Development}, the fact that the vector field $\partial_{v}+\partial_{u}$
is tangential to $\gamma_{\mathcal{Z}}$, $\mathcal{I}$ implies that
\begin{align}
\partial_{v}r|_{\gamma_{\mathcal{Z}}} & =-\partial_{u}r|_{\gamma_{\mathcal{Z}}},\label{eq:BoundaryConditionDrAxis}\\
\partial_{v}(\frac{1}{r})|_{\mathcal{I}} & =-\partial_{u}(\frac{1}{r})|_{\mathcal{I}}.\label{eq:BoundaryConditionDrInfinity}
\end{align}
Moreover, using the formula (\ref{eq:DefinitionHereHawkingMass})
for $\Omega^{2}$, equations (\ref{eq:RequationFinal})\textendash (\ref{eq:OmegaEquationFinal})
and (\ref{eq:RenormalisedEquations}) for $(r,\Omega^{2})$, the smoothness
of $(r,\Omega^{2},f)$ at $\gamma_{\mathcal{Z}}$, $\mathcal{I}$
(implying, among other things, that $m=O(r^{3})$ at $\gamma_{\mathcal{Z}}$)
and the boundary conditions (\ref{eq:BoundaryConditionDrAxis})\textendash (\ref{eq:BoundaryConditionDrInfinity})
for $r$, we infer that $\Omega^{2}$ satisfies the following Neuman-type
boundary conditions at $\gamma_{\mathcal{Z}}$, $\mathcal{I}$: 
\begin{align}
(\partial_{v}-\partial_{u})\Omega^{2}\Big|_{\gamma_{\mathcal{Z}_{\varepsilon}}} & =0,\label{eq:BoundaryConditionOmegaAxis}\\
(\partial_{v}-\partial_{u})\Big(\frac{\Omega^{2}}{1-\frac{1}{3}\Lambda r^{2}}\Big)\Bigg|_{\mathcal{I}} & =0.\label{eq:BoundaryConditionOmegaInfinity}
\end{align}
Let us also point out that, for any smooth development $(\mathcal{U};r,\Omega^{2},f)$
as in Definition \ref{def:Development} and any $u_{*}\in[0,u_{\gamma_{\mathcal{Z}}})$,
the characteristic initial data set $(r_{/u_{*}},\Omega_{/u_{*}}^{2},\bar{f}_{/u_{*}};v_{\mathcal{I}})$
induced on the slice $\{u=u_{*}\}\cap\mathcal{U}$ by $(r,\Omega^{2},f)$
is smoothly compatible, in accordance with Definition \ref{def:SmoothlyCompatibleInitialData};
see also the Discussion in Section 3.2 of \cite{MoschidisVlasovWellPosedness}.
\end{rem*}
The following proposition establishes the well-posedness of the innitial-boundary
value problem for (\ref{eq:RequationFinal})\textendash (\ref{NullShellFinal})
with reflecting boundary conditions on $\mathcal{I}$ in the class
$\mathfrak{B}_{0}$ of smoothly compatible, asymptotically AdS initial
data with bounded support in phase space, as introduced in Definition
\ref{def:SmoothlyCompatibleInitialData}: 
\begin{prop}
\emph{(Theorem 4.1 of \cite{MoschidisVlasovWellPosedness}).} \label{thm:LocalExistenceUniqueness}
Let $(r_{/},\Omega_{/}^{2},\bar{f}_{/};v_{\mathcal{I}})\in\mathfrak{B}_{0}$
(see Definition \ref{def:SmoothlyCompatibleInitialData}). Then, there
exists a $u_{*}>0$ (depending on $(r_{/},\Omega_{/}^{2},\bar{f}_{/};v_{\mathcal{I}})$)
and a unique solution $(r,\Omega^{2}f)$ of (\ref{eq:RequationFinal})\textendash (\ref{NullShellFinal})
on the domain 
\begin{equation}
\mathcal{U}_{u_{*};v_{\mathcal{I}}}\doteq\big\{0\le u<u_{*}\big\}\cap\big\{ u<v<u+v_{\mathcal{I}}\big\},\label{eq:GeneralDomain}
\end{equation}
 such that $(\mathcal{U}_{u_{*},v_{\mathcal{I}}};r,\Omega^{2}f)$
is a future development of $(r_{/},\Omega_{/}^{2},\bar{f}_{/};v_{\mathcal{I}})$
with reflecting boundary conditions on $\mathcal{I}$, in accordance
with Definition \ref{def:Development}. 
\end{prop}
For the proof of Proposition \ref{thm:LocalExistenceUniqueness},
see Section 4.3 of \cite{MoschidisVlasovWellPosedness}.

The existence of a unique \emph{maximal }future development for smoothly
compatible, asymptotically AdS characteristic initial data with bounded
support in phase space was also established in \cite{MoschidisVlasovWellPosedness}: 
\begin{prop}
(Corollary 4.2 of \emph{\cite{MoschidisVlasovWellPosedness}}). \label{cor:MaximalDevelopment}
Let $(r_{/},\Omega_{/}^{2},\bar{f}_{/};v_{\mathcal{I}})$ be initial
data set in $\mathfrak{B}_{0}$. Then there exists a unique future
development $(\mathcal{U}_{max};r,\Omega^{2},f)$ of $(r_{/},\Omega_{/}^{2},\bar{f}_{/};v_{\mathcal{I}})$
with reflecting boundary conditions on $\mathcal{I}$ having the following
property: If $(\mathcal{U}_{*};r_{*},\Omega_{*}^{2},f_{*})$ is any
other future development of $(r_{/},\Omega_{/}^{2},\bar{f}_{/};v_{\mathcal{I}})$
with reflecting boundary conditions on $\mathcal{I}$, then 
\begin{equation}
\mathcal{U}_{*}\subseteq\mathcal{U}_{max}
\end{equation}
 and 
\begin{equation}
(r,\Omega^{2},f)|_{\mathcal{U}_{*}}=(r_{*},\Omega_{*}^{2},f_{*}).
\end{equation}
The solution $(\mathcal{U}_{max};r,\Omega^{2},f)$ will be called
the maximal future development of $(r_{/},\Omega_{/}^{2},\bar{f}_{/};v_{\mathcal{I}})$
under the reflecting boundary condition on $\mathcal{I}$.
\end{prop}
For a a more detailed presentation and a discussion on the proof of
Proposition \ref{cor:MaximalDevelopment}, see Section 4.2 of \cite{MoschidisVlasovWellPosedness}.

The following notions regarding conformal infinity for future developments
of smoothly compatible, asymptotically AdS initial data set will be
frequently used in this paper:
\begin{defn}
\label{def:Black_hole_and_completeness} Let $(\mathcal{U};r,\Omega^{2},f)$
be a future development of a smoothly compatible, asymptotically AdS
initial data set $(r_{/},\Omega_{/}^{2},\bar{f}_{/};v_{\mathcal{I}})$
for the system (\ref{eq:RequationFinal})\textendash (\ref{NullShellFinal}),
and let $u_{\mathcal{I}}$ and $\mathcal{I}$ be defined according
to Definition \ref{def:DevelopmentSets}. 

\begin{itemize}

\item The \emph{black hole} region of $(\mathcal{U};r,\Omega^{2},f)$
will be defined as the set 
\begin{equation}
\mathcal{B}\doteq\{u\ge u_{\mathcal{I}}\}\cap\mathcal{U}.
\end{equation}
We will say that $(\mathcal{U};r,\Omega^{2},f)$ contains a black
hole if $\mathcal{B}\neq\emptyset$. 

\item We will say that a point $p\in\mathcal{U}$ corresponds to
a \emph{trapped sphere} of $(\mathcal{U};r,\Omega^{2},f)$ if 
\begin{equation}
\frac{2m}{r}(p)>1.
\end{equation}

\item We will say that $(\mathcal{U};r,\Omega^{2},f)$ has \emph{future
complete} conformal infinity $\mathcal{I}$ if 
\begin{equation}
\int_{0}^{u_{\mathcal{I}}}\frac{\Omega}{\big(1-\frac{1}{3}\Lambda r^{2}\big)^{\frac{1}{2}}}(u,u+v_{\mathcal{I}})\,du=+\infty.\label{eq:ConditionFutureCompleteInfinity}
\end{equation}

\end{itemize}
\end{defn}
\begin{rem*}
As a consequence of the relation (\ref{eq:DefinitionHereHawkingMass})
and the fact that $\partial_{u}r<0$ everywhere on $\mathcal{U}$
(following from (\ref{eq:ConstraintUFinal}) and the fact that $\partial_{u}r<0$
on $\{u=0\}\cup\mathcal{I}$), if $(\bar{u},\bar{v})\in\mathcal{U}$
satisfies 
\[
\frac{2m}{r}(\bar{u},\bar{v})\ge1,
\]
 then 
\[
\partial_{v}r(\bar{u},\bar{v})<0.
\]
Hence, as a consequence of (\ref{eq:ConstraintVFinal}):
\begin{equation}
\sup_{v\ge\bar{v}}\partial_{v}r(\bar{u},v)<0.
\end{equation}
Therefore, in this case the function $r$ is bounded from above along
$\{u=\bar{u}\}$ and hence the line $\{u=\bar{u}\}$ does not intersect
$\{r=\infty\}=\mathcal{I}$, i.\,e.~$(\bar{u},\bar{v})$ is contained
in the black hole region $\mathcal{B}$. Equivalently, 
\begin{equation}
\frac{2m}{r}(u,v)<1\text{ for all }(u,v)\in\{u<u_{\mathcal{I}}\}\cap\mathcal{U}.\label{eq:NoTrappingPastInfinity}
\end{equation}
\end{rem*}
In general, we will not be able to show that the maximal future development
of a smooth initial data set $(r_{/},\Omega_{/}^{2},\bar{f}_{/};v_{\mathcal{I}})$
has future complete conformal infinity $\mathcal{I}$.\footnote{The statement that for generic initial data, $\mathcal{I}$ is future
complete, is of course equivalent to the statement of the weak cosmic
censorship conjecture in the asymptotically AdS settings for (\ref{eq:The mainSystemEinsteinVlasov})
in spherical symmetry. } However, in the presence of a trapped sphere, the following statement
holds:
\begin{lem}
\emph{(Lemma B.1 in \cite{MoschidisVlasovWellPosedness})}\label{lem:CompletenessOfI}
Let $(r_{/},\Omega_{/}^{2},\bar{f}_{/};v_{\mathcal{I}})\in\mathfrak{B}_{0}$
and let $(\mathcal{U}_{max};r,\Omega^{2},f)$ be the maximal future
development of $(r_{/},\Omega_{/}^{2},\bar{f}_{/};v_{\mathcal{I}})$
with reflecting boundary conditions on $\mathcal{I}$. Assume that
there exists a point $(\bar{u},\bar{v})\in\mathcal{U}_{max}$ satisfying
\begin{equation}
\frac{2m}{r}(\bar{u},\bar{v})>1.\label{eq:TrappedPoint}
\end{equation}
Then, $(\mathcal{U}_{max};r,\Omega^{2},f)$ has future complete conformal
infinity $\mathcal{I}$, i.\,e.~(\ref{eq:ConditionFutureCompleteInfinity})
holds.
\end{lem}
For a proof of Lemma \ref{lem:CompletenessOfI}, see Section B of
the Appendix of \cite{MoschidisVlasovWellPosedness}.

\subsection{Continuation criteria for smooth solutions of (\ref{eq:RequationFinal})\textendash (\ref{NullShellFinal}) }

In this section, we will state two criteria that will allow us to
extend smooth solutions $(r,\Omega^{2},f)$ of (\ref{eq:RequationFinal})\textendash (\ref{NullShellFinal})
beyond their original domain of definition. These criteria will be
applied in our proof of Theorem \ref{thm:TheTheoremIntro} in Sections
\ref{sec:The-technical-core}\textendash \ref{sec:The-final-stage}.
For a wider class of continuation criteria, as well as for a proof
of the results of this section, see Section 5 of our companion paper
\cite{MoschidisVlasovWellPosedness}. 

The main extension principle of this section is the following:
\begin{prop}
\emph{(Corollary 5.1 in \cite{MoschidisVlasovWellPosedness}).} \label{cor:GeneralContinuationCriterion}For
any $v_{\mathcal{I}}>0$ and $u_{1}>0$, let $(r,\Omega^{2},f)$ be
a smooth solution of the system (\ref{eq:RequationFinal})\textendash (\ref{NullShellFinal})
on the domain $\mathcal{U}_{u_{1};v_{\mathcal{I}}}$ (defined as in
(\ref{eq:GeneralDomain})), with smooth axis $\{u=v\}$ and smooth
conformal infinity $\{u=v-v_{\mathcal{I}}\}$ (see Definitions 3.1\textendash 3.3
of \cite{MoschidisVlasovWellPosedness}). Assume that $(r,\Omega^{2},f)$
satisfies 
\begin{equation}
\sup_{\mathcal{U}_{u_{1};v_{\mathcal{I}}}}\frac{2m}{r}<1,\label{eq:NoTrappingContinuity}
\end{equation}
\begin{equation}
\limsup_{(u,v)\rightarrow(u_{1},u_{1})}\frac{2\tilde{m}}{r}\le\delta_{0},\label{eq:NoConcentrationContinuity}
\end{equation}
where $\delta_{0}$ is the constant appearing in the statement of
Proposition 4.1 in \cite{MoschidisVlasovWellPosedness}, and, moreover,
at $u=0$, we have
\begin{equation}
supp\Big(f(0,\cdot;\cdot)\Big)\subset\Big\{\Omega^{2}(p^{u}+p^{v})\le C_{0}\Big\}\text{ for some }C_{0}<+\infty.\label{eq:CompactSupportFContinuity}
\end{equation}
Then, there exists some $\bar{u}_{1}>u_{1}$, such that $(r,\Omega^{2},f)$
extends on the whole of the domain $\mathcal{U}_{\bar{u}_{1};v_{\mathcal{I}}}\supset\mathcal{U}_{u_{1};v_{\mathcal{I}}}$
as a smooth solution of (\ref{eq:RequationFinal})\textendash (\ref{NullShellFinal})
with smooth axis $\gamma_{\bar{u}_{1}}$ and smooth conformal infinity
$\mathcal{I}_{\bar{u}_{1}}$.
\end{prop}
For a proof of Proposition \ref{cor:GeneralContinuationCriterion},
see Section 5.3 in \cite{MoschidisVlasovWellPosedness}.

The next extension principle, which is also presented in \cite{MoschidisVlasovWellPosedness},
applies to the case of smooth solutions of (\ref{eq:RequationFinal})\textendash (\ref{NullShellFinal}),
restricted to domains on which $r$ is bounded away from $0$ and
$+\infty$:
\begin{prop}
\emph{(Proposition 5.1 in \cite{MoschidisVlasovWellPosedness}).}
\label{prop:ExtensionPrincipleMihalis}For any $u_{1}<u_{2}$, any
$v_{1}<v_{2}$ and any $\Lambda\in\mathbb{R}$, let $(r,\Omega^{2},f)$
be a smooth solution of the system (\ref{eq:RequationFinal})\textendash (\ref{NullShellFinal})
on an open neighborhood $\mathcal{V}$ of the rectangular region
\[
\mathcal{R}\doteq[u_{1},u_{2}]\times[v_{1},v_{2}]\backslash\{(u_{2},v_{2})\},
\]
satisfying 
\begin{equation}
\inf_{\mathcal{V}}r>0,\label{eq:LowerBoundRGiven}
\end{equation}
\begin{equation}
\sup_{\mathcal{V}}r<+\infty,\label{eq:UpperBoundRGiven}
\end{equation}
\begin{equation}
\sup_{\mathcal{V}}\tilde{m}<+\infty,\label{eq:UpperBoundMass}
\end{equation}
\begin{equation}
\sup_{(\{u_{1}\}\times[v_{1},v_{2}])\cup([u_{1},u_{2}]\times\{v_{1}\})}\partial_{u}r<0,\label{eq:InitialNegativeDerivativeInU}
\end{equation}
and, for some $C<+\infty$:
\begin{equation}
supp\Big(f(u_{1},\cdot;\cdot)\Big),supp\Big(f(\cdot,v_{1};\cdot)\Big)\subseteq\big\{\Omega^{2}(p^{v}+p^{u})\le C\big\}.\label{eq:CompactSupportF}
\end{equation}
Then, $(r,\Omega^{2},f)$ extends smoothly in a neighborhood of $\{(u_{2},v_{2})\}$.
\end{prop}
For a proof of Proposition \ref{prop:ExtensionPrincipleMihalis},
as well as a discussion on the connection between Proposition \ref{prop:ExtensionPrincipleMihalis}
and an analogous result established in \cite{DafermosRendall}, see
\cite{MoschidisVlasovWellPosedness}.

\subsection{\label{subsec:Well-posedness}Cauchy stability of $(\mathcal{M}_{AdS},g_{AdS})$
for (\ref{eq:RequationFinal})\textendash (\ref{NullShellFinal})
in a low regularity topology}

In this section, we will introduce a low regularity, scale invariant
topology on the space $\mathfrak{B}_{0}$ of smoothly compatible,
asymptotically AdS initial data of bounded support in phase space
(see Definition \ref{def:SmoothlyCompatibleInitialData}). We will
then formulate a Cauchy stability statement for the trivial solution
$(\mathcal{M}_{AdS},g_{AdS})$ in this topology. This statement will
be crucial for addressing the AdS instability conjecture in the associated
low regularity topology. A more detailed discussion on the results
of this section can be found in our companion paper \cite{MoschidisVlasovWellPosedness}.

In accordance with \cite{MoschidisVlasovWellPosedness}, we will introduce
the following map from $\mathfrak{B}_{0}$ to the space of smooth
solution of the (free) massless Vlasov equation (\ref{eq:Vlasov})
on AdS spacetime:
\begin{defn}
\label{def:ComparisonVlasovField} For any given $v_{\mathcal{I}}>0$,
let $(r_{/},\Omega_{/}^{2},\bar{f}_{/};v_{\mathcal{I}})$ be an asymptotically
AdS initial data set in the class $\mathfrak{B}_{0}$ (see Definition
\ref{def:SmoothlyCompatibleInitialData}). Let also $(r_{/},\Omega_{/}^{2},\bar{f}_{/};v_{\mathcal{I}})\rightarrow(r_{/}^{\prime},(\Omega_{/}^{\prime})^{2},\bar{f}_{/}^{\prime};v_{\mathcal{I}})$
be the (unique) gauge transformation such that $(r_{/}^{\prime},(\Omega_{/}^{\prime})^{2},\bar{f}_{/}^{\prime};v_{\mathcal{I}})$
satisfies the normalised gauge condition (\ref{eq:NormalisedGaugeCondition})
(for more details on this transformation, see Lemma 3.2 in \cite{MoschidisVlasovWellPosedness}).
Let us also $\bar{f}_{/}^{(AdS)}:[0,\sqrt{-\frac{3}{\Lambda}}\pi)\times[0,+\infty)^{2}\rightarrow[0,+\infty)$
in terms of $\bar{f}_{/}^{\prime}$ by the expression 
\[
\bar{f}_{/}^{(AdS)}(v;\,p^{u},l)\doteq\bar{f}_{/}^{\prime}(\frac{\sqrt{-\frac{3}{\Lambda}}\pi}{v_{\mathcal{I}}}\cdot v;\,p^{u},l).
\]

We will define $f^{(AdS)}=f^{(AdS)}[\bar{f}_{/};v_{\mathcal{I}}]:T\mathcal{M}_{AdS}\rightarrow[0,+\infty)$
to be the unique solution of the massless Vlasov equation (\ref{eq:Vlasov})
on $(\mathcal{M}_{AdS},g_{AdS})$ with initial conditions corresponding
to $\bar{f}_{/}^{(AdS)}$, i.\,e. satisfying at $u=0$:
\begin{equation}
f^{(AdS)}(0,v;p^{u},\frac{l^{2}}{p^{u}\cdot\Omega_{AdS}^{2}r_{AdS}^{2}(0,v)},l)=\bar{f}_{/}^{(AdS)}(v;\,p^{u},l)\cdot\delta\Big(\Omega_{AdS}^{2}(0,v)p^{u}p^{v}-\frac{l^{2}}{r_{AdS}^{2}(0,v)}\Big),
\end{equation}
where $\Omega_{AdS}^{2}$, $r_{AdS}$ are the coefficients of $g_{AdS}$
given by (\ref{eq:AdSMetricValues-1}). For any $\bar{u}\ge0$ and
$\bar{v}\in(\bar{u},\bar{u}+\sqrt{-\frac{3}{\Lambda}}\pi)$, we will
also set
\begin{equation}
\Big[\frac{rT_{vv}}{\partial_{v}r}\Big]^{(AdS)}(\bar{u},\bar{v})\doteq\frac{r_{AdS}T_{vv}[f^{(AdS)}]}{\partial_{v}r_{AdS}}(\bar{u},\bar{v})\label{eq:VlasovEnergyConcentration}
\end{equation}
(and similarly for $\Big[\frac{rT_{uv}}{-\partial_{u}r}\Big]^{(AdS)}$,
$\Big[\frac{rT_{uv}}{\partial_{v}r}\Big]^{(AdS)}$ and $\Big[\frac{rT_{uu}}{-\partial_{u}r}\Big]^{(AdS)}$),
where the energy momentum components $T_{\alpha\beta}[f^{(AdS)}]$
are defined using the relations (\ref{eq:ComponentsStressEnergy})
with $\Omega_{AdS}^{2}$, $r_{AdS}$ in place of $\Omega^{2}$, $r$.
\end{defn}
Using the mapping $(r_{/},\Omega_{/}^{2},\bar{f}_{/};v_{\mathcal{I}})\rightarrow f^{(AdS)}$
fixed in Definition \ref{def:ComparisonVlasovField}, we will define
the following positive definite functional on $\mathfrak{B}_{0}$
(see also Section 6.1 in \cite{MoschidisVlasovWellPosedness}):
\begin{defn}
\label{def:InitialDataNorm} For any $(r_{/},\Omega_{/}^{2},\bar{f}_{/};v_{\mathcal{I}})\in\mathfrak{B}_{0}$,
we will define the ``norm'' $||(r_{/},\Omega_{/}^{2},\bar{f}_{/};v_{\mathcal{I}})||$
of $(r_{/},\Omega_{/}^{2},\bar{f}_{/};v_{\mathcal{I}})$ in terms
of the free Vlasov field $f^{(AdS)}$ on $(\mathcal{M}_{AdS},g_{AdS})$
as follows:
\begin{align}
||(r_{/},\Omega_{/}^{2},\bar{f}_{/};v_{\mathcal{I}})||\doteq\sup_{U_{*}\ge0} & \int_{U_{*}}^{U_{*}+\sqrt{-\frac{3}{\Lambda}}\pi}\Big(\Big[\frac{rT_{vv}}{\partial_{v}r}\Big]^{(AdS)}(U_{*},v)+\Big[\frac{rT_{uv}}{-\partial_{u}r}\Big]^{(AdS)}(U_{*},v)\Big)\,dv+\label{eq:InitialDataNorm}\\
+ & \sup_{V_{*}\ge0}\int_{\max\{0,V_{*}-\sqrt{-\frac{3}{\Lambda}}\pi\}}^{V_{*}}\Big(\Big[\frac{rT_{uu}}{-\partial_{u}r}\Big]^{(AdS)}(u,V_{*})+\Big[\frac{rT_{uv}}{\partial_{v}r}\Big]^{(AdS)}(u,V_{*})\Big)\,du+\nonumber \\
 & +\sqrt{-\Lambda}\tilde{m}_{/}|_{v=v_{\mathcal{I}}}.\nonumber 
\end{align}
Given a smooth solution asymptotically AdS solution $(r,\Omega^{2};f)$
of (\ref{eq:RequationFinal})\textendash (\ref{NullShellFinal}) on
a domain $\mathcal{U}_{u_{1};v_{\mathcal{I}}}$ of the form (\ref{eq:GeneralDomain})
with axis $\{u=v\}$ and conformal infinity $\{u=v-v_{\mathcal{I}}\}$,
we will similarly define the ``norm'' of the initial data induced
by $(r,\Omega^{2};f)$ on slices of the form $\{u=u_{*}\}\cap\mathcal{U}_{u_{1};v_{\mathcal{I}}}$
for any $u_{*}\in(0,u_{1})$ as follows:
\begin{equation}
||(r,\Omega^{2};f)|_{u=u_{*}}||\doteq||(r_{/u_{*}},\Omega_{/u_{*}}^{2},\bar{f}_{/u_{*}};v_{\mathcal{I}})||\label{eq:InitialDataNormSlice}
\end{equation}
where 
\[
(r_{/u_{*}},\Omega_{/u_{*}}^{2})(\bar{v})\doteq(r,\Omega^{2})(u_{*},u_{*}+\bar{v})
\]
and 
\[
\bar{f}_{/u_{*}}(\bar{v};p,l)=\bar{f}(u_{*},u_{*}+\bar{v};p,\frac{l^{2}}{\Omega^{2}r^{2}|_{(u_{*},u_{*}+\bar{v})}p},l)
\]
(where the function $\bar{f}$ is related to the distribution $f$
by (\ref{eq:ExpressionRegularF})).
\end{defn}
\begin{rem*}
The fact that $||\cdot||$ takes finite values on $\mathfrak{B}_{0}$
follows readily from the condition (\ref{eq:BoundedSupportDefinition})
on the support of $\bar{f}_{/}$. Moreover, $||(r_{/},\Omega_{/}^{2},\bar{f}_{/})||=0$
if and only if $\bar{f}_{/}\equiv0$; in this case, $(r_{/},\Omega_{/}^{2},0)$
can be identified through a gauge transformation with the rescaled
trivial data 
\begin{align}
r_{AdS}^{(v_{\mathcal{I}})}(u,v) & =r_{AdS}\Big(\sqrt{-\frac{3}{\Lambda}}\pi\frac{u}{v_{\mathcal{I}}},\sqrt{-\frac{3}{\Lambda}}\pi\frac{v}{v_{\mathcal{I}}}\Big),\label{eq:AdSMetricValuesRescaled}\\
\big(\Omega_{AdS}^{(v_{\mathcal{I}})}\big)^{2}(u,v) & =-\frac{3}{\Lambda}\frac{\pi^{2}}{v_{\mathcal{I}}^{2}}\Omega_{AdS}^{2}\Big(\sqrt{-\frac{3}{\Lambda}}\pi\frac{u}{v_{\mathcal{I}}},\sqrt{-\frac{3}{\Lambda}}\pi\frac{v}{v_{\mathcal{I}}}\Big).\nonumber 
\end{align}
 For a detailed discussion on the special properties and the scale-invariant
character of $||\cdot||$, see Section 6.1 of \cite{MoschidisVlasovWellPosedness}. 
\end{rem*}
The following result provides a Cauchy stability statement for the
trivial solution $(\mathcal{M}_{AdS},g_{AdS})$ of (\ref{eq:RequationFinal})\textendash (\ref{NullShellFinal})
in the context of the initial data topology defined by (\ref{eq:InitialDataNorm})
on $\mathfrak{B}_{0}$:
\begin{prop}
\emph{(Theorem 6.1 in \cite{MoschidisVlasovWellPosedness}).} \label{prop:CauchyStabilityAdS}
For any $v_{\mathcal{I}}>0$, any $U>0$ and any $C_{0}>0$, there
exist $\varepsilon_{0}>0$ and $C_{1}>0$ such that the following
statement holds: For any $0\le\varepsilon<\varepsilon_{0}$ and any
smooth initial data set $(r_{/},\Omega_{/}^{2},\bar{f}_{/};v_{\mathcal{I}})\in\mathfrak{B}_{0}$
satisfying 
\begin{equation}
||(r_{/},\Omega_{/}^{2},\bar{f}_{/};v_{\mathcal{I}})||<\varepsilon\label{eq:SmallnessInitialNorm}
\end{equation}
(where $||\cdot||$ is defined by (\ref{eq:InitialDataNorm})) and
the bound (\ref{eq:BoundedSupportDefinition}) with $C_{0}$ in place
of $C$, the maximal future development $(\mathcal{U}_{max};f,\Omega^{2},f)$
of $(r_{/},\Omega_{/}^{2},\bar{f}_{/};v_{\mathcal{I}})$ under the
reflecting boundary condition on $\mathcal{I}$ (see Proposition \ref{cor:MaximalDevelopment})
satisfies 
\[
\mathcal{U}_{U;v_{\mathcal{I}}}\subset\mathcal{U}_{max}
\]
(where the domain $\mathcal{U}_{U;v_{\mathcal{I}}}\subset\mathbb{R}^{2}$
is defined in terms of $U$, $v_{\mathcal{I}}$ by (\ref{eq:GeneralDomain})).
Furthermore, $(\mathcal{U}_{U;v_{\mathcal{I}}};r,\Omega^{2},f)$ satisfies
the following bounds:
\begin{equation}
\sup_{u_{*}\in(0,U)}||(r,\Omega^{2};f)|_{u=u_{*}}||\le C_{1}\varepsilon,\label{eq:SmallnessInNormCauchy}
\end{equation}
\begin{equation}
\sup_{(u,v)\in\mathcal{U}_{U;v_{\mathcal{I}}}}\Big(\sup_{p^{u},p^{v}\in supp(f(u,v;\cdot,\cdot,\cdot))}\Big(\Omega^{2}(u,v)\big(p^{u}+p^{v}\big)\Big)\Big)\le(1+C_{1}\varepsilon)C_{0},\label{eq:SmallChangeBoundedSupportCauchy}
\end{equation}
\begin{equation}
\sup_{u\in(0,U)}\int_{u}^{u+v_{\mathcal{I}}}r\Big(\frac{T_{vv}[f]}{\partial_{v}r}+\frac{T_{uv}[f]}{-\partial_{u}r}\Big)(u,v)\,dv+\sup_{v\in(0,U+v_{\mathcal{I}})}\int_{\max\{0,v-v_{\mathcal{I}}\}}^{\min\{v,U\}}r\Big(\frac{T_{uv}[f]}{\partial_{v}r}+\frac{T_{uu}[f]}{-\partial_{u}r}\Big)(u,v)\,du\le C_{1}\varepsilon,\label{eq:SmallnessRightHandSideConstraintsCauchy}
\end{equation}
and
\begin{equation}
\sup_{\mathcal{U}_{U;v_{\mathcal{I}}}}\frac{2\tilde{m}}{r}<C_{1}\varepsilon.\label{eq:SmallnessTrappingCauchy}
\end{equation}
\end{prop}
For the proof of Proposition \ref{prop:CauchyStabilityAdS}, see Section
6.2 of \cite{MoschidisVlasovWellPosedness}.

\section{\label{sec:The-main-result} Statement of the main result}

In this section, we will present a detailed formulation of Theorem
\ref{thm:TheTheoremIntro} on the instability of $(\mathcal{M}_{AdS},g_{AdS})$
as a solution of the system (\ref{eq:RequationFinal})\textendash (\ref{NullShellFinal}).
In particular, the main result of this paper can be stated as follows:

\begin{customthm}{1}[final version]\label{thm:TheTheorem} 

There exists a 1-parameter family $\mathcal{S}^{(\varepsilon)}=(r_{/}^{(\varepsilon)},(\Omega_{/}^{(\varepsilon)})^{2},\bar{f}_{/}^{(\varepsilon)};\sqrt{-\frac{3}{\Lambda}}\pi)$,
$\varepsilon\in(0,1]$, of smoothly compatible, asymptotically AdS
initial data sets\emph{ }for (\ref{eq:RequationFinal})\textendash (\ref{NullShellFinal})
with bounded support in phase space (see Definition \ref{def:SmoothlyCompatibleInitialData}),
satisfying the following conditions:

\begin{enumerate}

\item $\mathcal{S}^{(\varepsilon)}$ converge to the trivial initial
data as $\varepsilon\rightarrow0$ with respect to the topology defined
by (\ref{eq:InitialDataNorm}), i.\,e.
\begin{equation}
||(r_{/}^{(\varepsilon)},(\Omega_{/}^{(\varepsilon)})^{2},\bar{f}_{/}^{(\varepsilon)};\sqrt{-\frac{3}{\Lambda}}\pi)||\xrightarrow{\varepsilon\rightarrow0}0.\label{eq:ConvergenceToTrivialTheorem}
\end{equation}

\item For any $\varepsilon\in(0,1]$, the corresponding maximal future
development $(\mathcal{U}_{max}^{(\varepsilon)};r_{\varepsilon},\Omega_{\varepsilon}^{2},f_{\varepsilon})$
of $\mathcal{S}^{(\varepsilon)}$ with reflecting boundary conditions
on $\mathcal{I}$ contains a point $(u_{\dagger}^{(\varepsilon)},v_{\dagger}^{(\varepsilon)})$
such that
\begin{equation}
\frac{2m_{\varepsilon}}{r_{\varepsilon}}(u_{\dagger}^{(\varepsilon)},v_{\dagger}^{(\varepsilon)})>1.\label{eq:TrappedSphere}
\end{equation}

\end{enumerate}

\end{customthm}
\begin{rem*}
As a consequence of (\ref{eq:NoTrappingPastInfinity}), the relation
(\ref{eq:TrappedSphere}) implies that $(\mathcal{U}_{max}^{(\varepsilon)};r_{\varepsilon},\Omega_{\varepsilon}^{2},f_{\varepsilon})$
contains a non-trivial black hole region for any $\varepsilon\in(0,1]$.
Furthermore, Lemma \ref{lem:CompletenessOfI} implies that, in this
case, $(\mathcal{U}_{max}^{(\varepsilon)};r_{\varepsilon},\Omega_{\varepsilon}^{2},f_{\varepsilon})$
possesses a complete null infinity $\mathcal{I}_{\varepsilon}$ for
any $\varepsilon\in(0,1]$.

Using the estimates established in the proof of Theorem \ref{thm:TheTheorem},
it can be actually shown that there exists an advanced time $v_{*}^{(\varepsilon)}>0$
satisfying 
\[
v_{\dagger}^{(\varepsilon)}<v_{*}^{(\varepsilon)}<u_{\mathcal{I}_{\varepsilon}}+\sqrt{-\frac{3}{\Lambda}}\pi,
\]
 such that, in the region $\mathcal{V}_{\infty}\doteq\{v\ge v_{*}^{(\varepsilon)}\}\cap\mathcal{U}_{max}^{(\varepsilon)}$,
the Vlasov field $f^{(\varepsilon)}$ vanishes identically and the
solution is locally isometric to a member of the Schwarzschild\textendash AdS
family. This fact implies, in particular, that $u_{\mathcal{I}}<+\infty$
and that the future boundary $\zeta$ of $\mathcal{U}_{max}^{(\varepsilon)}$
is strictly spacelike in a neighborhood of ``future timelike infinity''
$(u_{\mathcal{I}_{\varepsilon}},u_{\mathcal{I}_{\varepsilon}}+\sqrt{-\frac{3}{\Lambda}}\pi)$.
However, we are not able to rule out the possibility of $\zeta$ containing
a null segment emanating from $r=0$ and corresponding to a Cauchy
horizon for the maximal future development $(\mathcal{U}_{max}^{(\varepsilon)};r_{\varepsilon},\Omega_{\varepsilon}^{2},f_{\varepsilon})$
(note that the extension principle along $r=0$ provided by Theorem
5.1 in \cite{MoschidisVlasovWellPosedness} only applies under the
condition that $\frac{2\tilde{m}}{r}\ll1$). We will not pursue this
issue any further in this paper.
\end{rem*}
The proof of Theorem \ref{thm:TheTheorem} will occupy Sections \ref{sec:Auxiliary-constructions}\textendash \ref{sec:The-final-stage}.
In particular, the construction of the initial data family $\mathcal{S}^{(\varepsilon)}$
will be presented in Section \ref{sec:Auxiliary-constructions}, with
(\ref{eq:ConvergenceToTrivialTheorem}) established in Section (\ref{subsec:The-initial-data-family})
(see Lemma \ref{lem:SmallnessFamilyInitialData}). The fact that the
corresponding maximal developments $(\mathcal{U}_{max}^{(\varepsilon)};r_{\varepsilon},\Omega_{\varepsilon}^{2},f_{\varepsilon})$
contain points where (\ref{eq:TrappedSphere}) holds will finally
be established in Section \ref{sec:The-final-stage}, using the technical
machinery developed in Section \ref{sec:The-technical-core} and the
fact that the specific choice of the initial data family leads to
the formation of an intermediate profile with certain properties (see
Section \ref{sec:ThefirstStage}). For a sketch of the proof, see
also Section \ref{subsec:Discussion-on-the} of the introduction.

Before proceeding to the proof of Theorem \ref{thm:TheTheorem} and
its related constructions, we will need to establish a number of fundamental
estimates that will allow us to control the geodesic flow on solutions
of (\ref{eq:RequationFinal})\textendash (\ref{NullShellFinal}) under
minimal assumptions on their geometry; this will be achieved in Section
\ref{sec:GeodesicPathsAndDifferenceEstimates}.

\section{\label{sec:GeodesicPathsAndDifferenceEstimates} Auxiliary estimates
for the null geodesic flow in the case $2\tilde{m}/r\ll1$ }

In this section, we will establish a number of estimates related to
the paths of null geodesics on asymptotically AdS solutions $(r,\Omega^{2},f)$
of (\ref{eq:RequationFinal})\textendash (\ref{NullShellFinal}),
assuming, in addition, that the spacetimes under consideration satisfy
the smallness condition
\[
\frac{2\tilde{m}}{r}\le\delta_{0}\ll1.
\]
The results of this section will be crucial for the proof of Theorem
\ref{thm:TheTheorem}, since they will allow us to estimate the paths
traced out by narrow Vlasov beams with minimal control on the spacetime
geometry.

\subsection{\label{subsec:GeodesicPaths} Geodesic paths under rough assumptions
on the spacetime geometry}

For any $U>0$ and $v_{\mathcal{I}}>0$, let $\mathcal{U}_{U;v_{\mathcal{I}}}$
be the domain in the $(u,v)$-plane defined by (\ref{eq:GeneralDomain}).
Let also $(r,\Omega^{2},f)$ be a smooth solution of (\ref{eq:RequationFinal})\textendash (\ref{NullShellFinal})
on $\mathcal{U}_{U;v_{\mathcal{I}}}$, with smooth axis $\{u=v\}$
and smooth conformal infinity $\{u=v-v_{\mathcal{I}}\}$, in accordance
with to Definitions 3.1\textendash 3.3 of \cite{MoschidisVlasovWellPosedness}. 

The following result provides quantitative bounds for the paths of
null geodesics in $\mathcal{U}_{U;v_{\mathcal{I}}}$:
\begin{lem}
\label{lem:GeodesicPaths} Let $0<\delta_{0}\ll1$ be a sufficiently
small absolute constant, and let $\mathcal{U}_{U;v_{\mathcal{I}}}$
and $(r,\Omega^{2},f)$ be as above. Assume that the following bounds
are satisfied for some $C_{0}>100$:
\begin{equation}
\sup_{\mathcal{U}_{U;v_{\mathcal{I}}}}\Bigg(\Big|\log\Big(\frac{\partial_{v}r}{1-\frac{1}{3}\Lambda r^{2}}\Big)\Big|+\Big|\log\Big(\frac{-\partial_{u}r}{1-\frac{1}{3}\Lambda r^{2}}\Big)\Big|\Bigg)\le C_{0}.\label{eq:BoundGeometryC0}
\end{equation}
and 
\begin{equation}
\sup_{\mathcal{U}_{U;v_{\mathcal{I}}}}\Big(\frac{2\tilde{m}}{r}+\sqrt{-\Lambda}\tilde{m}\Big)\le\delta_{0}.\label{eq:SmallnessTrapping}
\end{equation}

Let $\gamma:[0,a)\rightarrow\mathcal{U}_{U;v_{\mathcal{I}}}$ (with
$a\in(0,+\infty]$) be a future inextendible, future directed, affinely
parametrised null geodesic of $(r,\Omega^{2})$ satisfying the following
conditions:

\begin{itemize}

\item $\gamma$ is initially ingoing, i.\,e.
\begin{equation}
\dot{\gamma}^{u}(0)>\dot{\gamma}^{v}(0),\label{eq:initiallyIngoing}
\end{equation}

\item $\gamma$ has angular momentum $l$ satisfying the bound
\begin{equation}
0<\frac{l}{E_{0}}\sqrt{-\Lambda}\le e^{-50C_{0}},\label{eq:SmallnessAngularMomentum}
\end{equation}
where
\begin{equation}
E_{0}\doteq\frac{1}{2}\Big(\Omega^{2}\dot{\gamma}^{u}+\Omega^{2}\dot{\gamma}^{v}\Big)(0),\label{eq:InitialEnergy}
\end{equation}

\item $\gamma(0)$ satisfies 
\begin{equation}
r(\gamma(0))\ge e^{50C_{0}}\frac{l}{E_{0}}.\label{eq:InitialRGamma2}
\end{equation}

\end{itemize}

Then, the following statements hold for $\gamma$:

\begin{enumerate}

\item Setting 
\[
(u_{0},v_{0})\doteq\gamma(0)
\]
and 
\begin{align}
\mathcal{V}_{\nwarrow} & =[u_{0},v_{0}+e^{150C_{0}}\frac{l}{E_{0}}]\times[v_{0}-e^{150C_{0}}\frac{l}{E_{0}},v_{0}+e^{150C_{0}}\frac{l}{E_{0}}],\label{eq:IngoingUGeneral}\\
\mathcal{V}_{\nearrow} & =[v_{0}-e^{150C_{0}}\frac{l}{E_{0}},v_{0}+e^{150C_{0}}\frac{l}{E_{0}}]\times[v_{0}-e^{150C_{0}}\frac{l}{E_{0}},v_{0}+v_{\mathcal{I}}+e^{150C_{0}}\frac{l}{E_{0}}],\label{eq:OutgoingUGeneral}
\end{align}
 the curve $\gamma$ is contained in the following region: 
\begin{equation}
\gamma\subset\Big\{ r\ge e^{-6C_{0}}\frac{l}{E_{0}}\Big\}\cap\Big(\mathcal{V}_{\nwarrow}\cup\mathcal{V}_{\nearrow}\Big)\cap\mathcal{U}_{U;v_{\mathcal{I}}}\label{eq:DomainForGamma}
\end{equation}
(see Figure \ref{fig:GeodesicPath}).

\item For any $s\in[0,a)$, we can estimate
\begin{equation}
e^{-100C_{0}}E_{0}\le\frac{1}{2}\Big(\Omega^{2}\dot{\gamma}^{u}(s)+\Omega^{2}\dot{\gamma}^{v}(s)\Big)\le e^{e^{200C_{0}}}E_{0}.\label{eq:EnergyEstimateGeneralDomain}
\end{equation}

\item Let $s_{c}\in(0,a]$ be defined as 
\begin{equation}
s_{c}=\sup\big\{ s\in(0,a):\text{ }u(\gamma(s))+v(\gamma(s))\le u_{0}+v_{0}+v_{\mathcal{I}}\big\}.\label{eq:CriticalS}
\end{equation}
Then, for any $s\in[0,s_{c})$, we can bound 
\begin{equation}
\frac{\dot{\gamma}^{v}}{\dot{\gamma}^{u}}\le e^{e^{200C_{0}}}\frac{l^{2}}{E_{0}^{2}}\frac{1-\frac{1}{3}\Lambda r^{2}}{r^{2}}\Bigg|_{\gamma(s)}\label{eq:IngoingRegion}
\end{equation}
while, for any $s\in(s_{c},a)$, we have:\footnote{This is a non-trivial case only when $s_{c}<a$.}
\begin{equation}
\frac{\dot{\gamma}^{u}}{\dot{\gamma}^{v}}(s)\le e^{e^{200C_{0}}}\frac{l^{2}}{E_{0}^{2}}\frac{1-\frac{1}{3}\Lambda r^{2}}{r^{2}}\Bigg|_{\gamma(s)}.\label{eq:OutgoingRegion}
\end{equation}

\end{enumerate}
\end{lem}
\begin{figure}[h] 
\centering 
\scriptsize
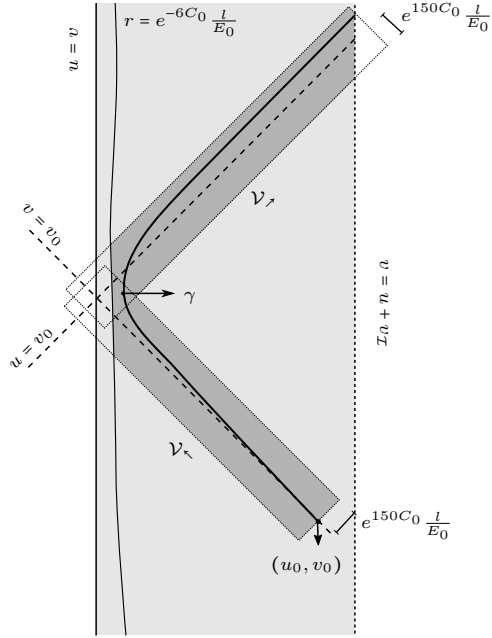 
\caption{Schematic depiction of the rectangular domains $\mathcal{V}_{\nwarrow}$ and $\mathcal{V}_{\nearrow}$ in the statement of Lemma \ref{lem:GeodesicPaths}. \label{fig:GeodesicPath}}
\end{figure}
\begin{rem*}
In the case when one considers future inextendible geodesics $\gamma$
in $\mathcal{U}_{U;v_{\mathcal{I}}}$ with future endpoints on conformal
infinity, the statement of Lemma \ref{lem:GeodesicPaths} can be readily
generalised to the extension of such geodesics through their reflection
off $\{u=v-v_{\mathcal{I}}\}$; see Corollary \ref{cor:GeodesicPathsLongTimes}.
Notice also that the condition (\ref{eq:BoundGeometryC0}) implies
that $\frac{1}{10}e^{-C_{0}}\le\sqrt{-\Lambda}v_{\mathcal{I}}\le10e^{C_{0}}$,
i.\,e.~that, in the class of spacetimes satisfying (\ref{eq:BoundGeometryC0}),
$v_{\mathcal{I}}$ and $(-\Lambda)^{-\frac{1}{2}}$ can be used almost
interchangeably as units of length with merely $O(e^{C_{0}})$ errors
occuring in the transition. Let us finally remark that Lemma \ref{lem:GeodesicPaths}
is also valid in the case when the initial point of $\gamma$ lies
on $\mathcal{I}$, i.\,e.~$\gamma(0)$ in the statement of the Lemma
is replaced by $\gamma(-\infty)$, with $\gamma(-\infty)\in\mathcal{I}$,
and (\ref{eq:initiallyIngoing}) is replaced by $\frac{\Omega^{2}\dot{\gamma}^{u}}{\Omega^{2}\dot{\gamma}^{v}}(-\infty)>1$.
\end{rem*}
\begin{proof}
We will adopt the following convention regarding the parametrization
of $\gamma$: We will denote with $s$ the affine parametrization
of $\gamma$ (and the corresponding derivative by $\dot{}$), while
$\tau$ will denote the parameter corresponding to $u+v$. Let also
$[\tau_{0},\tau_{1})$ be the parameter interval for $\gamma$ associated
to the parameter $\tau$, i.\,e.~$\tau_{0}=u_{0}+v_{0}$ and $\tau_{1}=\lim_{s\rightarrow a}\big(u(\gamma(s))+v(\gamma(s))\big)$.

In view of the formula (\ref{eq:DefinitionHereHawkingMass}), the
bounds (\ref{eq:BoundGeometryC0}) and (\ref{eq:SmallnessTrapping})
imply that 
\begin{equation}
\sup_{\mathcal{U}_{U;v_{\mathcal{I}}}}\Big|\log\Big(\frac{\Omega^{2}}{1-\frac{1}{3}\Lambda r^{2}}\Big)\Big|\le2C_{0}.\label{eq:BoundOmegaC0}
\end{equation}
Using (\ref{eq:XrhsimhGewdaisiakh-U}) with 
\begin{equation}
u_{1}(v)=\begin{cases}
u_{0}, & v\le u_{0}+v_{\mathcal{I}},\\
v-v_{\mathcal{I}}, & v\ge u_{0}+v_{\mathcal{I}},
\end{cases}\label{eq:U1(v)FromInitial}
\end{equation}
 and then using the relation (\ref{eq:BoundaryConditionOmegaInfinity})
for $\frac{\Omega^{2}}{1-\frac{1}{3}\Lambda r^{2}}$ along $\{u=v-v_{\mathcal{I}}\}$,
we readily infer that, for any $\tau\in(\tau_{0},\tau_{1})$:
\begin{align}
\log\big(\Omega^{2}\dot{\gamma}^{u}\big)(\tau)- & \log\big(\Omega^{2}\dot{\gamma}^{u}\big)(\tau_{0})=\int_{v_{0}}^{v(\gamma(\tau))}\int_{u_{1}(v)}^{u(\gamma(\tau_{v}))}\Big(\frac{1}{2}\frac{\frac{6\tilde{m}}{r}-1}{r^{2}}\Omega^{2}-24\pi T_{uv}\Big)\,dudv+\label{eq:UsefulRelationForGeodesicWithMu-2}\\
 & +\Big(\log\Big(\frac{\Omega^{2}}{1-\frac{1}{3}\Lambda r^{2}}\frac{1-\frac{1}{3}\Lambda r^{2}}{-\Lambda r^{2}}\Big)\big(u_{1}(v(\gamma(\tau))),v(\gamma(\tau))\big)-\log\Big(\frac{\Omega^{2}}{1-\frac{1}{3}\Lambda r^{2}}\frac{1-\frac{1}{3}\Lambda r^{2}}{-\Lambda r^{2}}\Big)(u_{0},v_{0})\Big),\nonumber 
\end{align}
where $\tau_{v}$ is defined by 
\begin{equation}
v(\gamma(\tau_{\bar{v}}))=\bar{v}.
\end{equation}
In view of the fact that the first term in the right hand side of
(\ref{eq:UsefulRelationForGeodesicWithMu-2}) is non-positive (as
a consequence of (\ref{eq:SmallnessTrapping})), using (\ref{eq:BoundOmegaC0}),
(\ref{eq:initiallyIngoing}), (\ref{eq:InitialEnergy}) and the fact
that 
\[
r(u_{1}(v(\gamma(\tau))),v(\gamma(\tau)))\ge r(u_{0},v_{0})
\]
(since $\gamma$ has a timelike projection on the $(u,v)$-plane and
thus $v(\gamma(\tau))\ge v_{0}$, while $r(u_{1}(v),v)=\infty$ if
$u_{1}(v)\neq u_{0}$), we infer from (\ref{eq:UsefulRelationForGeodesicWithMu-2})
that, for any $\tau\in(\tau_{0},\tau_{1})$:
\begin{equation}
\Omega^{2}\dot{\gamma}^{u}(\tau)\le e^{3C_{0}}\Omega^{2}\dot{\gamma}^{u}(\tau_{0})\le2e^{3C_{0}}E_{0}.\label{eq:BoundForMaximumOfPu}
\end{equation}

For any $\tau\in(\tau_{0},\tau_{1})$ such that $\frac{d}{d\tau}r(\gamma(\tau))\le0$,
i.\,e.
\begin{equation}
\partial_{v}r|_{\gamma(\tau)}\dot{\gamma}^{v}(\tau)\le-\partial_{u}r|_{\gamma(\tau)}\dot{\gamma}^{u}(\tau),\label{eq:BoundOneSidedDecreasingR}
\end{equation}
we can estimate from (\ref{eq:NullShellNewAngularMomentum}) using
(\ref{eq:BoundGeometryC0}), (\ref{eq:BoundForMaximumOfPu}) and (\ref{eq:BoundOneSidedDecreasingR}):
\begin{align}
\frac{\Omega^{2}l^{2}}{r^{2}}\Big|_{\gamma(\tau)}= & (\Omega^{2}|_{\gamma(\tau)}\dot{\gamma}^{u}(\tau))\cdot(\Omega^{2}|_{\gamma(\tau)}\dot{\gamma}^{v}(\tau))\le\label{eq:AlmostMinimumR}\\
 & \le\sup_{\mathcal{U}_{U;v_{\mathcal{I}}}}\big(\frac{\partial_{v}r}{-\partial_{u}r}\big)(\Omega^{2}|_{\gamma(\tau)}\dot{\gamma}^{u}(\tau))^{2}\le\nonumber \\
 & \le4e^{8C_{0}}E_{0}^{2}.\nonumber 
\end{align}
Since (\ref{eq:AlmostMinimumR}) holds whenever $\frac{d}{d\tau}r(\gamma(\tau))\le0$,
using (\ref{eq:BoundOmegaC0}) we deduce that 
\begin{equation}
\inf_{\gamma}r\ge\frac{1}{2}e^{-5C_{0}}\frac{l}{E_{0}}.\label{eq:LowerBoundRGeneralCase}
\end{equation}

The identity (\ref{eq:UsefulRelationForGeodesicWithMu-2}) implies,
using the bound (\ref{eq:SmallnessTrapping}) for $\tilde{m}/r$,
combined with the bound 
\[
T_{uv}\le e^{C_{0}}\frac{-\partial_{u}\tilde{m}}{r^{2}}
\]
(following readily from (\ref{eq:TildeUMaza}) and (\ref{eq:BoundGeometryC0})),
as well as the bounds (\ref{eq:BoundGeometryC0}) and (\ref{eq:BoundOmegaC0}),
that the following estimate holds for any $\tau\in(\tau_{0},\tau_{1})$
\begin{align}
\Big|\log\big(\Omega^{2}\dot{\gamma}^{u}\big)(\tau)- & \log\big(\Omega^{2}\dot{\gamma}^{u}\big)(\tau_{0})\Big|\le\label{eq:EstimateForSmallnessOfGu}\\
\le & e^{4C_{0}}\int_{u_{0}+v_{\mathcal{I}}}^{v(\gamma(\tau))}\int_{u_{1}(v)}^{u(\gamma(\tau_{v}))}\frac{(-\partial_{u}r)}{r^{2}}\,dudv+e^{2C_{0}}\int_{u_{1}(v)}^{v(\gamma(\tau))}\int_{u_{1}(v)}^{u(\gamma(\tau_{v}))}(-\partial_{u}\tilde{m})\frac{1-\frac{1}{3}\Lambda r^{2}}{r^{2}}\,dudv+2C_{0}\le\nonumber \\
\le & e^{4C_{0}}\Big(\sup_{v\in[v_{0},v(\gamma(\tau))]}\int_{u_{1}(v)}^{u(\gamma(\tau_{v}))}\frac{(-\partial_{u}r)}{r^{2}}\,du\Big)\big(v(\gamma(\tau))-v(\gamma(\tau_{0}))\big)+\nonumber \\
 & \hphantom{e^{4C_{0}}\Big(}+e^{2C_{0}}\int_{u_{1}(v)}^{v(\gamma(\tau))}\int_{u_{1}(v)}^{u(\gamma(\tau_{v}))}4\frac{\tilde{m}}{r}\frac{1-\frac{1}{3}\Lambda r^{2}}{r^{2}}(-\partial_{u}r)\,dudv+\\
 & \hphantom{e^{4C_{0}}\Big(}+e^{2C_{0}}\int_{u_{1}(v)}^{v(\gamma(\tau))}\tilde{m}\frac{1-\frac{1}{3}\Lambda r^{2}}{r^{2}}\Big|_{u=u_{1}(v)}\,dv-e^{2C_{0}}\int_{u_{1}(v)}^{v(\gamma(\tau))}\tilde{m}\frac{1-\frac{1}{3}\Lambda r^{2}}{r^{2}}\Big|_{u=u(\gamma(\tau_{v}))}\,dv\le\nonumber \\
\le & e^{4C_{0}}\sup_{\bar{\tau}\in(\tau_{0},\tau)}\frac{1}{r(\gamma(\bar{\tau}))}\cdot\big(v(\gamma(\tau))-v(\gamma(\tau_{0}))\big)+2C_{0},\nonumber 
\end{align}
where, in passing from the first to the second line in (\ref{eq:EstimateForSmallnessOfGu}),
we integrated in $u$ for the $\partial_{u}\tilde{m}$ term. On the
other hand, the identity (\ref{eq:UsefulRelationForGeodesicWithMu-2})
can also be used to similarly obtain an one sided bound for $\log\big(\Omega^{2}\dot{\gamma}^{u}\big)(\tau)-\log\big(\Omega^{2}\dot{\gamma}^{u}\big)(\tau_{0})$,
using the fact that $T_{uv}\ge0$: 
\begin{align}
\log\big(\Omega^{2}\dot{\gamma}^{u}\big)(\tau)-\log\big(\Omega^{2}\dot{\gamma}^{u}\big)(\tau_{0})\le & \int_{u_{0}+v_{\mathcal{I}}}^{v(\gamma(\tau))}\int_{u_{1}(v)}^{u(\gamma(\tau_{v}))}\Big(\frac{1}{2}\frac{\frac{6\tilde{m}}{r}-1}{r^{2}}\Omega^{2}\Big)\,dudv+\label{eq:EstimateForSmallnessOfGu-1}\\
 & +\Big(\log\Big(\frac{\Omega^{2}}{1-\frac{1}{3}\Lambda r^{2}}\Big)\big(u_{1}(v(\gamma(\tau))),v(\gamma(\tau))\big)-\log\Big(\frac{\Omega^{2}}{1-\frac{1}{3}\Lambda r^{2}}\Big)(u_{0},v_{0})\Big)\le\nonumber \\
\le & -e^{-4C_{0}}\mathcal{B}[\tau]+2C_{0},\nonumber 
\end{align}
where 
\begin{equation}
\mathcal{B}[\tau]\doteq\int_{u_{1}(v)}^{v(\gamma(\tau))}\int_{v-v_{\mathcal{I}}}^{u(\gamma(\tau_{v}))}\frac{(-\partial_{u}r)}{r^{2}}\,dudv.\label{eq:B=00005Bt=00005D}
\end{equation}
Therefore, in view of (\ref{eq:InitialEnergy}), (\ref{eq:initiallyIngoing}):
\begin{equation}
-e^{4C_{0}}\sup_{\bar{\tau}\in(\tau_{0},\tau)}\frac{1}{r(\gamma(\bar{\tau}))}\cdot\big(v(\gamma(\tau))-v(\gamma(\tau_{0}))\big)-2C_{0}\le\log\Big(\frac{\Omega^{2}\dot{\gamma}^{u}(\tau)}{E_{0}}\Big)\le-e^{-4C_{0}}\mathcal{B}[\tau]+2C_{0}.\label{eq:EstimateDecreaseGammaU}
\end{equation}

Let us define $\tau_{in}\in[\tau_{0},\tau_{1}]$ as follows: 
\begin{equation}
\tau_{in}\doteq\sup\Big\{\tau\in(\tau_{0},\tau_{1}):\text{ }r(\gamma(\bar{\tau}))\ge e^{30C_{0}}\frac{l}{E_{0}}\text{ for all }\bar{\tau}\le\tau\Big\}.\label{eq:DefinitionFirstPointNearRegion}
\end{equation}
Our analysis of the path traced by $\gamma$ will be separated into
two regimes: The \emph{ingoing} interval $\tau\in[\tau_{0},\tau_{in})$
and the \emph{outgoing} interval $\tau\in[\tau_{in},\tau_{1})$. Note
that the latter interval will be trivial if $\tau_{in}=\tau_{1}$.
However, in view of (\ref{eq:InitialRGamma2}), it is necessary that
the ingoing inteval is non-trivial, i.\,e.:
\begin{equation}
\tau_{in}>\tau_{0}.
\end{equation}
 \medskip{}

\noindent \emph{The ingoing regime $\tau\in[\tau_{0},\tau_{in})$.}
Let us define $\tau_{*}\in(\tau_{0},\tau_{in}]$ by the relation:
\begin{equation}
\tau_{*}=\sup\Big\{\tau\in(\tau_{0},\tau_{c}):\text{ }v(\gamma(\bar{\tau}))-v(\gamma(\tau_{0}))\le e^{20C_{0}}\frac{l}{E_{0}}\text{ for all }\bar{\tau}\le\tau\Big\}.\label{eq:DefinitionFirstPointNearRegion-1}
\end{equation}

The estimate (\ref{eq:EstimateDecreaseGammaU}) implies, in view of
the fact that $\text{ }r(\gamma(\bar{\tau}))\ge e^{30C_{0}}\frac{l}{E_{0}}$
for $\tau<\tau_{in}$ (folowing from the definition (\ref{eq:DefinitionFirstPointNearRegion})
of $\tau_{in}$) and the definition (\ref{eq:DefinitionFirstPointNearRegion-1})
of $\tau_{*}$, that, for all $\tau\in[\tau_{0},\tau_{*})$:
\begin{equation}
\big(\Omega^{2}\dot{\gamma}^{u}\big)(\tau)\ge e^{-4C_{0}}E_{0}.\label{eq:LowerBoundGammaUIngoing}
\end{equation}
Thus, in view of (\ref{eq:NullShellNewAngularMomentum}) and (\ref{eq:BoundOmegaC0}),
we can bound for all $\tau\in[\tau_{0},\tau_{*})$:
\begin{equation}
(\Omega^{2}\dot{\gamma}^{v})(\tau)\le e^{6C_{0}}\frac{l^{2}(1-\frac{1}{3}\Lambda r^{2})}{E_{0}r^{2}}\Big|_{\gamma(\tau)}.\label{eq:UpperBoundGammaVIngoing}
\end{equation}
In view of (\ref{eq:BoundForMaximumOfPu}), (\ref{eq:LowerBoundRGeneralCase}),
(\ref{eq:LowerBoundGammaUIngoing}), and (\ref{eq:UpperBoundGammaVIngoing}),
we therefore infer that, for all $\tau\in[\tau_{0},\tau_{*})$: 
\begin{equation}
e^{-4C_{0}}E_{0}\le\big(\Omega^{2}\dot{\gamma}^{u}\big)(\tau)+\big(\Omega^{2}\dot{\gamma}^{v}\big)(\tau)\le e^{10C_{0}}E_{0}\label{eq:UpperboundEnergyIngoing}
\end{equation}
and
\begin{equation}
\frac{(\Omega^{2}\dot{\gamma}^{v})(\tau)}{\big(\Omega^{2}\dot{\gamma}^{u}\big)(\tau)+(\Omega^{2}\dot{\gamma}^{v})(\tau)}\le e^{10C_{0}}\frac{l^{2}(1-\frac{1}{3}\Lambda r^{2})}{E_{0}^{2}r^{2}}\Big|_{\gamma(\tau)}.\label{eq:UpperBoundGammaVParametrization}
\end{equation}
Moreover, as a result of (\ref{eq:BoundGeometryC0}), (\ref{eq:UpperboundEnergyIngoing}),
(\ref{eq:UpperBoundGammaVParametrization}) and t he definition (\ref{eq:DefinitionFirstPointNearRegion})
of $\tau_{in}$, as well as using assumption (\ref{eq:SmallnessAngularMomentum})
for $\gamma$, we deduce that, for all $\tau\in[\tau_{0},\tau_{*})$:
\begin{equation}
\frac{d}{d\tau}\Big(\tan^{-1}\big(\sqrt{-\Lambda}r(\gamma(\tau))\big)\Big)\le-e^{-4C_{0}}\Big(1-e^{10C_{0}}\frac{l^{2}(1-\frac{1}{3}\Lambda r^{2})}{E_{0}^{2}r^{2}}\Big|_{\gamma(\tau)}\Big)\sqrt{-\Lambda}\le-\frac{1}{2}e^{-4C_{0}}\sqrt{-\Lambda}.\label{eq:LowerBoundRChangeWhenL<<1}
\end{equation}

Integrating (\ref{eq:UpperBoundGammaVParametrization}) over $\tau\in[\tau_{0},\tau_{*})$
and using (\ref{eq:LowerBoundRChangeWhenL<<1}) and (\ref{eq:LowerBoundRGeneralCase}),
we infer that, for all $\tau\in[\tau_{0},\tau_{*})$:
\begin{equation}
v(\gamma(\tau))-v_{0}\le\frac{1}{2}e^{20C_{0}}\frac{l}{E_{0}}.\label{eq:SmallVDeviation}
\end{equation}
In view of (\ref{eq:SmallVDeviation}), the definition (\ref{eq:DefinitionFirstPointNearRegion-1})
of $\tau_{*}$ implies (through a standard continuity argument) that
\begin{equation}
\tau_{*}=\tau_{in}.\label{eq:T*-Tin}
\end{equation}
Therefore, the estimates (\ref{eq:LowerBoundGammaUIngoing})\textendash (\ref{eq:SmallVDeviation})hold
for all $\tau\in[\tau_{0},\tau_{in})$. Moreover, the bound (\ref{eq:SmallVDeviation})
and the fact that $u\le v$ on $\mathcal{U}_{U;v_{\mathcal{I}}}$
implies that 
\begin{equation}
\tau_{in}=\lim_{\tau\rightarrow\tau_{in}^{-}}\big(u(\gamma(\tau))+v(\gamma(\tau))\big)\le2v_{0}+e^{20C_{0}}\frac{l}{E_{0}}.\label{eq:UpperBoundTin}
\end{equation}

\medskip{}

\noindent \emph{The outgoing regime $\tau\in[\tau_{in},\tau_{1})$.}
In the case when $\tau_{in}<\tau_{1}$ (which is necessarily the case,
for instance, when $\tau_{1}>2v_{0}+e^{20C_{0}}\frac{l}{E_{0}}$,
as a consequence of (\ref{eq:UpperBoundTin})), the definition (\ref{eq:DefinitionFirstPointNearRegion})
of $\tau_{in}$ implies that 
\begin{equation}
r(\gamma(\tau_{in}))=e^{30C_{0}}\frac{l}{E_{0}}\label{eq:CloseToAxisTin}
\end{equation}
from which we obtain, in view of the bound (\ref{eq:BoundGeometryC0})
on $\partial_{v}r$ and the boundary condition $r|_{\{u=v\}}=0$,
that 
\begin{equation}
v(\gamma(\tau_{in}))-u(\gamma(\tau_{in}))\le e^{32C_{0}}\frac{l}{E_{0}}.\label{eq:CloseToAxisUVTin}
\end{equation}
The bounds (\ref{eq:SmallVDeviation}) and (\ref{eq:CloseToAxisUVTin})
therefore yield: 
\begin{equation}
u(\gamma(\tau_{in}))-v_{0}\le2e^{32C_{0}}\frac{l}{E_{0}}.\label{eq:SmallUDeviation}
\end{equation}

For any $\tau\in[\tau_{in},\tau_{1})$, the analogue of the identity
(\ref{eq:UsefulRelationForGeodesicWithMu-2}) after replacing $u_{1}(v)$
with
\[
u_{2}(v)=\begin{cases}
u(\gamma(\tau_{in})), & v\le u(\gamma(\tau_{in}))+v_{\mathcal{I}},\\
v-v_{\mathcal{I}}, & v\ge u(\gamma(\tau_{in}))+v_{\mathcal{I}},
\end{cases}
\]
is
\begin{align}
\log\big(\Omega^{2}\dot{\gamma}^{u}\big)(\tau)-\log\big(\Omega^{2}\dot{\gamma}^{u}\big)(\tau_{in})= & \int_{v(\gamma(\tau_{in}))}^{v(\gamma(\tau))}\int_{u_{2}(v)}^{u(\gamma(\tau_{v}))}\Big(\frac{1}{2}\frac{\frac{6\tilde{m}}{r}-1}{r^{2}}\Omega^{2}-24\pi T_{uv}\Big)\,dudv+\label{eq:UsefulRelationForGeodesicWithMuForOutgoing}\\
 & +\Big(\log\Big(\frac{\Omega^{2}}{-\Lambda r^{2}}\Big)\big(u_{2}(v(\gamma(\tau))),v(\gamma(\tau))\big)-\log\Big(\frac{\Omega^{2}}{-\Lambda r^{2}}\Big)(u(\gamma(\tau_{in})),v(\gamma(\tau_{in})))\Big).\nonumber 
\end{align}
Using the fact that the fisrt term in the right hand side of (\ref{eq:UsefulRelationForGeodesicWithMuForOutgoing})
is non-positive (in view of (\ref{eq:SmallnessTrapping})), the bounds
(\ref{eq:BoundOmegaC0}) for $\Omega^{2}/(1-\frac{1}{3}\Lambda r^{2})$
and (\ref{eq:BoundForMaximumOfPu}) for $\dot{\gamma}^{u}(\tau_{in})$
imply that, for any $\tau\in[\tau_{in},\tau_{1})$: 
\begin{align}
\Omega^{2}\dot{\gamma}^{u}(\tau) & \le\Omega^{2}\dot{\gamma}^{u}(\tau_{in})\times\label{eq:BoundGammaUBeforeOutgoing}\\
 & \hphantom{\le}\times\exp\Bigg\{\log\big(\frac{1-\frac{1}{3}\Lambda r^{2}}{-\Lambda r^{2}}\big)\Big|_{(u_{2}(v(\gamma(\tau))),v(\gamma(\tau)))}-\log\big(\frac{1-\frac{1}{3}\Lambda r^{2}}{-\Lambda r^{2}}\big)\Big|_{(u(\gamma(\tau_{in})),v(\gamma(\tau_{in})))}+4C_{0}\Bigg\}\le\nonumber \\
 & \le2e^{7C_{0}}\frac{r^{2}(u(\gamma(\tau_{in})),v(\gamma(\tau_{in})))}{r^{2}(u_{2}(v(\gamma(\tau))),v(\gamma(\tau)))}\cdot\frac{\big(1-\frac{1}{3}\Lambda r^{2}\big)(u_{2}(v(\gamma(\tau))),v(\gamma(\tau)))}{\big(1-\frac{1}{3}\Lambda r^{2}\big)(u(\gamma(\tau_{in})),v(\gamma(\tau_{in})))}E_{0}.\nonumber 
\end{align}
Using (\ref{eq:SmallnessAngularMomentum}) and (\ref{eq:CloseToAxisTin}),
the bound (\ref{eq:BoundGammaUBeforeOutgoing}) yields for any $\tau\in[\tau_{in},\tau_{1})$:
\begin{equation}
\Omega^{2}\dot{\gamma}^{u}(\tau)\le e^{70C_{0}}\frac{1-\frac{1}{3}\Lambda r^{2}}{r^{2}}\Big|_{\big(u_{2}(v(\gamma(\tau))),v(\gamma(\tau))\big)}\frac{l^{2}}{E_{0}}.\label{eq:UpperBoundGammaUDecayInR}
\end{equation}
The relation (\ref{eq:NullShellNewAngularMomentum}) implies, in view
of (\ref{eq:BoundOmegaC0}), (\ref{eq:UpperBoundGammaUDecayInR})
and the fact that $\partial_{u}r<0$, that, for any $\tau\in[\tau_{in},\tau_{1})$:
\begin{equation}
\Omega^{2}\dot{\gamma}^{v}(\tau)\ge e^{-72C_{0}}\frac{\frac{1-\frac{1}{3}\Lambda r^{2}}{r^{2}}\Big|_{\big(u(\gamma(\tau))),v(\gamma(\tau))\big)}}{\frac{1-\frac{1}{3}\Lambda r^{2}}{r^{2}}\Big|_{\big(u_{2}(v(\gamma(\tau))),v(\gamma(\tau))\big)}}E_{0}\ge e^{-72C_{0}}E_{0}.\label{eq:LowerBoundGammaVOutgoing}
\end{equation}

The estimates (\ref{eq:UpperBoundGammaUDecayInR}) and (\ref{eq:LowerBoundGammaVOutgoing}),
combined with (\ref{eq:BoundGeometryC0}), (\ref{eq:SmallnessAngularMomentum})
and (\ref{eq:CloseToAxisTin}), imply that, for any $\tau\in[\tau_{in},\tau_{1})$:
\begin{align}
|u(\gamma(\tau))-u(\gamma(\tau_{in}))| & =\int_{\tau_{in}}^{\tau}\dot{\gamma}^{u}(\bar{\tau})\,d\bar{\tau}=\int_{v(\gamma(\tau_{in}))}^{v(\gamma(\tau))}\frac{\Omega^{2}\dot{\gamma}^{u}}{\Omega^{2}\dot{\gamma}^{v}}(\bar{\tau})\,dv(\gamma(\bar{\tau}))\le\label{eq:SmallUDeviationOutgoing}\\
 & \le e^{150C_{0}}\frac{l^{2}}{E_{0}^{2}}\int_{v(\gamma(\tau_{in}))}^{v(\gamma(\tau))}\frac{1-\frac{1}{3}\Lambda r^{2}}{r^{2}}(u_{2}(v),v)\,dv\le\nonumber \\
 & \le e^{160C_{0}}\frac{l^{2}}{E_{0}^{2}}\frac{1}{r(\gamma(\tau_{in}))}\le\nonumber \\
 & \le e^{130C_{0}}\frac{l}{E_{0}}.\nonumber 
\end{align}
Using (\ref{eq:CloseToAxisTin}) and (\ref{eq:SmallUDeviationOutgoing})
(as well as (\ref{eq:BoundGeometryC0}), (\ref{eq:SmallnessAngularMomentum})
and (\ref{eq:BoundOmegaC0})), we can estimate 
\begin{align}
\int_{v(\gamma(\tau_{in}))}^{v(\gamma(\tau))}\int_{u_{2}(v)}^{u(\gamma(\tau_{v}))}\frac{1}{r^{2}}\Omega^{2}\,dudv & \le\int_{\{r\ge r(\gamma(\tau_{in}))\}\cap\{|u-u(\gamma(\tau_{in}))|\le e^{130C_{0}}\frac{l}{E_{0}}\}}\frac{1}{r^{2}}\Omega^{2}\,dudv\le\label{eq:BoundForDecreaseOutgoing}\\
 & \le e^{150C_{0}}\frac{l}{E_{0}}\int_{r(\gamma(\tau_{in}))}^{\infty}\frac{1}{r^{2}}\,dr\le\nonumber \\
 & \le e^{120C_{0}}.\nonumber 
\end{align}
Similarly, in view of (\ref{eq:TildeVMaza}) and (\ref{eq:SmallnessTrapping}):
\begin{align}
\int_{v(\gamma(\tau_{in}))}^{v(\gamma(\tau))}\int_{u_{2}(v)}^{u(\gamma(\tau_{v}))} & T_{uv}\,dudv=\label{eq:BoundForDecreaseTuvutgoing}\\
 & =\int_{v(\gamma(\tau_{in}))}^{v(\gamma(\tau))}\int_{u_{2}(v)}^{u(\gamma(\tau_{v}))}\frac{\partial_{v}\tilde{m}}{8\pi r^{2}}\frac{\Omega^{2}}{-\partial_{u}r}\,dudv\le\nonumber \\
 & \le e^{4C_{0}}\int_{v(\gamma(\tau_{in}))}^{v(\gamma(\tau))}\int_{u_{2}(v)}^{u(\gamma(\tau_{v}))}\frac{\partial_{v}\tilde{m}}{r^{2}}\,dudv\le\nonumber \\
 & \le e^{4C_{0}}\Big(\int_{v(\gamma(\tau_{in}))}^{v(\gamma(\tau))}\int_{u_{2}(v)}^{u(\gamma(\tau_{v}))}2\frac{\tilde{m}}{r}\frac{\partial_{v}r}{r^{2}}\,dudv+\int_{u_{2}(v(\gamma(\tau)))}^{u(v(\gamma(\tau)))}\frac{\tilde{m}}{r^{2}}(u,v(\gamma(\tau)))\,du\Big)\le\nonumber \\
 & \le e^{10C_{0}}\delta_{0}\Big(\int_{v(\gamma(\tau_{in}))}^{v(\gamma(\tau))}\int_{u_{2}(v)}^{u(\gamma(\tau_{v}))}\frac{\Omega^{2}}{r^{2}}\,dudv+\int_{u_{2}(v(\gamma(\tau)))}^{u(v(\gamma(\tau)))}\frac{1}{r(\gamma(\tau_{1}))}\,du\Big)\le\nonumber \\
 & \le e^{150C_{0}}\delta_{0}.\nonumber 
\end{align}

Using (\ref{eq:SmallnessTrapping}), (\ref{eq:BoundForDecreaseOutgoing})
and (\ref{eq:BoundForDecreaseTuvutgoing}), we infer from (\ref{eq:UsefulRelationForGeodesicWithMuForOutgoing})
that, for any $\tau\in[\tau_{in},\tau_{1})$: 
\begin{align}
\log\big(\Omega^{2}\dot{\gamma}^{u}\big)(\tau) & -\log\big(\Omega^{2}\dot{\gamma}^{u}\big)(\tau_{in})\ge\label{eq:UsefulRelationForGeodesicWithMuForOutgoing-1}\\
 & \ge-2e^{120C_{0}}+\Big(\log\Big(\frac{\Omega^{2}}{-\Lambda r^{2}}\Big)\big(u_{2}(v(\gamma(\tau))),v(\gamma(\tau))\big)-\log\Big(\frac{\Omega^{2}}{-\Lambda r^{2}}\Big)(u(\gamma(\tau_{in})),v(\gamma(\tau_{in})))\Big)
\end{align}
and, therefore, in view of the bounds (\ref{eq:BoundOmegaC0}) for
$\Omega^{2}/(1-\frac{1}{3}\Lambda r^{2})$ and (\ref{eq:LowerBoundGammaUIngoing})
for $\dot{\gamma}^{u}(\tau_{in})$: 
\begin{align}
\Omega^{2}\dot{\gamma}^{u}(\tau) & \ge\Omega^{2}\dot{\gamma}^{u}(\tau_{in})\times\label{eq:BoundGammaUBeforeOutgoing-1}\\
 & \hphantom{\ge}\times\exp\Bigg\{-e^{130C_{0}}+\log\big(\frac{1-\frac{1}{3}\Lambda r^{2}}{-\Lambda r^{2}}\big)\Big|_{(u_{2}(v(\gamma(\tau))),v(\gamma(\tau)))}-\log\big(\frac{1-\frac{1}{3}\Lambda r^{2}}{-\Lambda r^{2}}\big)\Big|_{(u(\gamma(\tau_{in})),v(\gamma(\tau_{in})))}+4C_{0}\Bigg\}\le\nonumber \\
 & \ge e^{-e^{140C_{0}}}\frac{r^{2}(u(\gamma(\tau_{in})),v(\gamma(\tau_{in})))}{r^{2}(u_{2}(v(\gamma(\tau))),v(\gamma(\tau)))}\cdot\frac{\big(1-\frac{1}{3}\Lambda r^{2}\big)(u_{2}(v(\gamma(\tau))),v(\gamma(\tau)))}{\big(1-\frac{1}{3}\Lambda r^{2}\big)(u(\gamma(\tau_{in})),v(\gamma(\tau_{in})))}E_{0}.\nonumber 
\end{align}
The estimate (\ref{eq:BoundGammaUBeforeOutgoing-1}) implies, in view
of (\ref{eq:NullShellNewAngularMomentum}), (\ref{eq:SmallnessAngularMomentum}),
(\ref{eq:BoundOmegaC0}), (\ref{eq:LowerBoundRGeneralCase}) and the
fact that $\partial_{v}r>0$, that, for any $\tau\in[\tau_{in},\tau_{1})$:
\begin{align}
\Omega^{2}\dot{\gamma}^{v}(\tau) & \le e^{2C_{0}}l^{2}\frac{1-\frac{1}{3}\Lambda r^{2}}{r^{2}}\Bigg|_{\gamma(\tau)}\frac{1}{\Omega^{2}\dot{\gamma}^{u}(\tau)}\le\label{eq:UpperBoundGammaVOutgoing}\\
 & \le e^{e^{145C_{0}}}\frac{l^{2}}{E_{0}}\frac{1-\frac{1}{3}\Lambda r^{2}}{r^{2}}\Bigg|_{\gamma(\tau)}\cdot\frac{r^{2}}{\big(1-\frac{1}{3}\Lambda r^{2}\big)}\Bigg|_{(u_{2}(v(\gamma(\tau))),v(\gamma(\tau)))}\cdot\frac{\big(1-\frac{1}{3}\Lambda r^{2}\big)}{r^{2}}\Bigg|_{\gamma(\tau_{in})}\le\nonumber \\
 & \le e^{e^{145C_{0}}}\frac{l^{2}}{E_{0}}\frac{1-\frac{1}{3}\Lambda(\inf_{\gamma}r)^{2}}{(\inf_{\gamma}r)^{2}}\le\nonumber \\
 & \le e^{e^{150C_{0}}}E_{0}.\nonumber 
\end{align}
The bounds (\ref{eq:BoundForMaximumOfPu}), (\ref{eq:LowerBoundGammaVOutgoing})
and (\ref{eq:UpperBoundGammaVOutgoing}) therefore imply that, for
any $\tau\in[\tau_{in},\tau_{1})$: 
\begin{equation}
e^{-72C_{0}}E_{0}\le\Omega^{2}\dot{\gamma}^{u}(\tau)+\Omega^{2}\dot{\gamma}^{v}(\tau)\le e^{e^{200C_{0}}}E_{0}.\label{eq:EnergyEstimateOutgoing}
\end{equation}

\medskip{}

\noindent The estimates we have established so far are sufficient
to complete the proof of Lemma \ref{lem:GeodesicPaths}. In particular:

\begin{itemize}

\item The bound (\ref{eq:DomainForGamma}) follows readily from the
bound (\ref{eq:LowerBoundRGeneralCase}) on $\inf_{\gamma}r$, the
bounds (\ref{eq:SmallVDeviation}) and (\ref{eq:SmallUDeviationOutgoing})
on the total change of $u,v$ along the intervals $[\tau_{0},\tau_{in})$,
$[\tau_{in},\tau_{1})$, respectively (in view also of (\ref{eq:T*-Tin})),
and the bounds (\ref{eq:CloseToAxisUVTin}) and (\ref{eq:SmallUDeviation})
on $\gamma(\tau_{in})$.

\item The energy bound (\ref{eq:EnergyEstimateGeneralDomain}) follows
immediately from (\ref{eq:UpperboundEnergyIngoing}) and (\ref{eq:EnergyEstimateOutgoing}).

\item The estimates (\ref{eq:IngoingRegion}) and (\ref{eq:OutgoingRegion})
follow readily from the bounds (\ref{eq:LowerBoundGammaUIngoing}),
(\ref{eq:UpperBoundGammaVIngoing}) for $\tau\in[\tau_{0},\tau_{in})$
and the bounds (\ref{eq:UpperBoundGammaUDecayInR}), (\ref{eq:LowerBoundGammaVOutgoing})
for $\tau\in[\tau_{in},\tau_{1})$, as well as the fact that, for
any $\tau\in[\tau_{0},\tau_{1})$ such that $r(\gamma(\tau))\le e^{50C_{0}}\frac{l}{E_{0}}$,
we can estimate as a consequence of (\ref{eq:NullShellNewAngularMomentum}),
(\ref{eq:BoundOmegaC0}), (\ref{eq:BoundForMaximumOfPu}), (\ref{eq:UpperboundEnergyIngoing})
and (\ref{eq:EnergyEstimateOutgoing}): 
\begin{equation}
\frac{\dot{\gamma}^{v}}{\dot{\gamma}^{u}}(\tau)=l^{2}\frac{\Omega^{2}}{r^{2}}\Big|_{\gamma(\tau)}\frac{1}{(\Omega^{2}\dot{\gamma}^{u})^{2}(\tau)}\ge e^{-60C_{0}}
\end{equation}
and 
\begin{equation}
\frac{\dot{\gamma}^{v}}{\dot{\gamma}^{u}}(\tau)=l^{-2}\frac{r^{2}}{\Omega^{2}}\Big|_{\gamma(\tau)}(\Omega^{2}\dot{\gamma}^{v})^{2}(\tau)\le e^{e^{200C_{0}}}.
\end{equation}

\end{itemize}
\end{proof}
By applying Lemma \ref{lem:GeodesicPaths} successively between the
points of reflection off $\mathcal{I}$ of a maximally extended null
geodesic $\gamma$,
obtain the following useful generalisation of Lemma \ref{lem:GeodesicPaths}:
\begin{cor}
\label{cor:GeodesicPathsLongTimes} Let $0<\delta_{0}\ll1$ be a sufficiently
small absolute constant, and let $\mathcal{U}_{U;v_{\mathcal{I}}}$
and $(r,\Omega^{2},f)$ be as in Lemma \ref{lem:GeodesicPaths}, satisfying
(\ref{eq:BoundGeometryC0}) and (\ref{eq:SmallnessTrapping}) for
some $C_{0}>100$. Let also $\gamma_{n}:(a_{n},b_{n})\rightarrow\mathcal{U}_{U;v_{\mathcal{I}}}$,
$0\le n<N+1$ (for some $N\in\mathbb{N}\cup\{\infty\}$ and $-\infty\le a_{n}<b_{n}\le+\infty$
), be a collection of future directed, affinely parametrized null
geodesics in $(\mathcal{U}_{U;v_{\mathcal{I}}};r,\Omega^{2})$ with
$a_{0}=0$ and $\gamma_{0}(0)\in\{u=0\}$, such that $\gamma_{n}$
is the reflection of $\gamma_{n-1}$ off $\mathcal{I}$; thus, $\gamma=\cup_{n=0}^{N}\gamma_{n}$
constitutes an affinely parametrised, maximally extended geodesic
through reflections off $\mathcal{I}$, in accordance with Definition
2.3 in \cite{MoschidisVlasovWellPosedness}. 

Assume that $\gamma_{0}$ satisfies initially the conditions (\ref{eq:initiallyIngoing}),
(\ref{eq:InitialRGamma2}) and 
\begin{equation}
0<\frac{l}{E_{0}}\sqrt{-\Lambda}\le e^{-400(1+\lceil v_{\mathcal{I}}^{-1}U\rceil)C_{0}},\label{eq:SmallAngularMomentumMaximallyExtended}
\end{equation}
with $E_{0}$ defined by (\ref{eq:InitialEnergy}). Then, the following
statements hold for the maximally extended geodesic $\gamma$:

\begin{itemize}

\item The curve $\gamma=\cup_{n=0}^{N}\gamma_{n}$ is contained in
the region
\begin{equation}
\gamma\subset\Big\{ r\ge e^{-e^{400C_{0}}(1+v_{\mathcal{I}}^{-1}U)}\frac{l}{E_{0}}\Big\}\cap\bigcup_{k=0}^{\lceil v_{\mathcal{I}}^{-1}U\rceil}\Big(\mathcal{V}_{\nwarrow}^{(k)}\cup\mathcal{V}_{\nearrow}^{(k)}\Big)\cap\mathcal{U}_{U;v_{\mathcal{I}}},\label{eq:DomainForGammaMaximallyExtended}
\end{equation}
where, setting 
\[
v_{0}\doteq v\big(\gamma_{0}(0)\big),
\]
 the domains $\mathcal{V}_{\nwarrow}^{(k)}$, $\mathcal{V}_{\nearrow}^{(k)}$
are defined for any $k\in\mathbb{N}$ by
\begin{align}
\mathcal{V}_{\nwarrow}^{(k)} & =[kv_{\mathcal{I}},v_{0}+kv_{\mathcal{I}}+e^{300(k+1)C_{0}}\frac{l}{E_{0}}]\times[v_{0}+kv_{\mathcal{I}}-e^{300(k+1)C_{0}}\frac{l}{E_{0}},v_{0}+kv_{\mathcal{I}}+e^{300(k+1)C_{0}}\frac{l}{E_{0}}],\label{eq:IngoingUGeneralMaximallyExtended}\\
\mathcal{V}_{\nearrow}^{(k)} & =[v_{0}+kv_{\mathcal{I}}-e^{300(k+1)C_{0}}\frac{l}{E_{0}},v_{0}+kv_{\mathcal{I}}+e^{300(k+1)C_{0}}\frac{l}{E_{0}}]\times\label{eq:OutgoingUGeneralMaximallyExtended}\\
 & \hphantom{[v_{0}+kv_{\mathcal{I}}-e^{300(k+1)C_{0}}\frac{l}{E_{0}}=}\times[v_{0}+kv_{\mathcal{I}}-e^{300(k+1)C_{0}}\frac{l}{E_{0}},v_{0}+(k+1)v_{\mathcal{I}}+e^{300(k+1)C_{0}}\frac{l}{E_{0}}].\nonumber 
\end{align}

\item Denoting by $\dot{\gamma}$ the derivative of $\gamma$ with
respect to the affine parametrisation of the $\gamma_{n}$'s, we can
estimate
\begin{equation}
e^{-300C_{0}(1+v_{\mathcal{I}}^{-1}U)}E_{0}\le\Omega^{2}(\dot{\gamma}^{u}+\dot{\gamma}^{v})\le e^{e^{300C_{0}}(1+v_{\mathcal{I}}^{-1}U)}E_{0}.\label{eq:EnergyGrowthGammaMaximallyExtended}
\end{equation}

\item For any $0\le\bar{u}<U$, defining $n[\bar{u}]$ by the condition
that $\gamma\cap\{u=\bar{u}\}\in\gamma_{n[\bar{u}]}$, we can bound
at the point $\gamma\cap\{u=\bar{u}\}$:
\begin{align}
\frac{\dot{\gamma}^{v}}{\dot{\gamma}^{u}}\Bigg|_{\gamma\cap\{u=\bar{u}\}} & \le e^{e^{300C_{0}}(1+v_{\mathcal{I}}^{-1}U)}\frac{l^{2}}{E_{0}^{2}}\frac{1-\frac{1}{3}\Lambda r^{2}}{r^{2}}\Bigg|_{\gamma\cap\{u=\bar{u}\}}\text{ if }\gamma\cap\{u=\bar{u}\}\in\cup_{k\in\mathbb{N}}\mathcal{V}_{\nwarrow}^{(k)},\label{eq:IngoingRegionMaximallyExtended}\\
\frac{\dot{\gamma}^{u}}{\dot{\gamma}^{v}}\Bigg|_{\gamma\cap\{u=\bar{u}\}} & \le e^{e^{300C_{0}}(1+v_{\mathcal{I}}^{-1}U)}\frac{l^{2}}{E_{0}^{2}}\frac{1-\frac{1}{3}\Lambda r^{2}}{r^{2}}\Bigg|_{\gamma\cap\{u=\bar{u}\}}\text{ if }\gamma\cap\{u=\bar{u}\}\in\cup_{k\in\mathbb{N}}\mathcal{V}_{\nearrow}^{(k)}.\label{eq:OutgoingRegionMaximallyExtended}
\end{align}

\end{itemize}
\end{cor}
\begin{proof}
The proof of Corollary (\ref{cor:GeodesicPathsLongTimes}) follows
by applying Lemma (\ref{lem:GeodesicPaths}) successively on the curves
$\gamma_{n}$, treating the cases $n\ge1$ by considering the limit
where the initial point of $\gamma$ in the statement of Lemma (\ref{lem:GeodesicPaths})
is sent to $\mathcal{I}$ and establishing (as a consequence of (\ref{eq:EnergyEstimateGeneralDomain}))
the inductive bound
\begin{equation}
e^{-100C_{0}}E_{n-1}\le E_{n}\le e^{e^{200C_{0}}}E_{n-1}\label{eq:InductionEnergy}
\end{equation}
for the energy $E_{n}=\frac{1}{2}\Big(\Omega^{2}\dot{\gamma}_{n}^{u}+\Omega^{2}\dot{\gamma}_{n}^{v}\Big)\big|_{s=-\infty}$
of $\gamma_{n}$ at its initial point on $\mathcal{I}$. Following
this procedure, (\ref{eq:DomainForGammaMaximallyExtended}) is inferred
from (\ref{eq:DomainForGamma}), (\ref{eq:EnergyGrowthGammaMaximallyExtended})
is inferred from (\ref{eq:InductionEnergy}) and (\ref{eq:IngoingRegionMaximallyExtended})\textendash (\ref{eq:OutgoingRegionMaximallyExtended})
are inferred from (\ref{eq:IngoingRegion})\textendash (\ref{eq:OutgoingRegion}),
using also that (\ref{eq:EnergyGrowthGammaMaximallyExtended}) and
(\ref{eq:NullShellNewAngularMomentum}) imply that 
\[
e^{-e^{300C_{0}}(1+v_{\mathcal{I}}^{-1}U)}\frac{l^{2}}{E_{0}^{2}}\frac{1-\frac{1}{3}\Lambda r^{2}}{r^{2}}\le\frac{\dot{\gamma}^{v}}{\dot{\gamma}^{u}}\le e^{e^{300C_{0}}(1+v_{\mathcal{I}}^{-1}U)}\frac{l^{2}}{E_{0}^{2}}\frac{1-\frac{1}{3}\Lambda r^{2}}{r^{2}}
\]
in the region $\cup_{k=0}^{\lceil v_{\mathcal{I}}^{-1}U\rceil}(\mathcal{V}_{\nwarrow}^{(k)}\cap\mathcal{V}_{\nearrow}^{(k)})$.
We will omit the trivial details. 
\end{proof}

\section{\label{sec:Auxiliary-constructions}Construction of the initial data
and notation}

In this section, we will construct the family of initial data $(r_{/}^{(\varepsilon)},(\Omega_{/}^{(\varepsilon)})^{2},\bar{f}_{/}^{(\varepsilon)};\sqrt{-\frac{3}{\Lambda}}\pi)\in\mathfrak{B}_{0}$
appearing in the statement of Theorem \ref{thm:TheTheorem}. To this
end, we will first introduce a hierarchy of parameters depending on
$\varepsilon$, the precise choice of which will be crucial for the
proof of Theorem \ref{thm:TheTheorem}. We will also introduce some
shorthand notation associated to a few fundamental constructions on
the maximal future development $(\mathcal{U}_{max}^{(\varepsilon)};r,\Omega^{2},f_{\varepsilon})$
of $(r_{/}^{(\varepsilon)},(\Omega_{/}^{(\varepsilon)})^{2},\bar{f}_{/}^{(\varepsilon)};\sqrt{-\frac{3}{\Lambda}}\pi)$.

\subsection{\label{subsec:The-hierarchy-of-parameters}The hierarchy of parameters}

In this section, we will introduce a set of parameters that will be
used in the construction of the family of initial data $(r_{/}^{(\varepsilon)},(\Omega_{/}^{(\varepsilon)})^{2},\bar{f}_{/}^{(\varepsilon)};\sqrt{-\frac{3}{\Lambda}}\pi)$.

We will first introduce the following hierarchy of parameters:
\begin{defn}
\label{def:HierarchyParameters} Let $0<\varepsilon_{1}\ll1$ be a
sufficiently small absolute constant. For any $\varepsilon\in(0,\varepsilon_{1})$,
we will define the parameters $\delta_{\varepsilon},\rho_{\varepsilon},\sigma_{\varepsilon}$
through the following hierarchy of relations: 
\begin{align}
\varepsilon & =\exp\Big(-\exp(\rho_{\varepsilon}^{-10})\Big),\label{eq:HierarchyOfParameters}\\
\rho_{\varepsilon} & =\exp\Big(-\exp\big(\exp(\exp(\exp(\delta_{\varepsilon}^{-10})))\big)\Big),\nonumber \\
\delta_{\varepsilon} & =\exp\Big(-\exp\big(\exp(\sigma_{\varepsilon}^{-10})\big)\Big).\nonumber 
\end{align}
For $\varepsilon\in[\varepsilon_{1},1]$, we will define the parameters
$\delta_{\varepsilon},\rho_{\varepsilon},\sigma_{\varepsilon}$ to
be equal to $\delta_{\varepsilon_{1}},\rho_{\varepsilon_{1}},\sigma_{\varepsilon_{1}}$,
respectively. 
\end{defn}
We will also set 
\begin{equation}
N_{\varepsilon}\doteq\lceil\rho_{\varepsilon}^{-1}\exp\big(\exp(\delta_{\varepsilon}^{-15})\big)\rceil.\label{eq:NumberOfBeams}
\end{equation}
Notice that 
\begin{equation}
\lim_{\varepsilon\rightarrow0}\delta_{\varepsilon}=\lim_{\varepsilon\rightarrow0}\rho_{\varepsilon}=\lim_{\varepsilon\rightarrow0}\sigma_{\varepsilon}=\lim_{\varepsilon\rightarrow0}\frac{1}{N_{\varepsilon}}=0\label{eq:ParametersGoingToZero}
\end{equation}
and, as $\varepsilon\rightarrow0$:

\begin{equation}
\varepsilon\ll\rho_{\varepsilon}\ll\delta_{\varepsilon}\ll\sigma_{\varepsilon}\ll1.\label{eq:HeuristicRelationParameters}
\end{equation}

Finally, for any $\varepsilon\in(0,1]$ and any $0\le i\le N_{\varepsilon}$,
we will define the parameter $\varepsilon^{(i)}$ by the recursive
relation 
\begin{equation}
\begin{cases}
\varepsilon^{(i+1)}=\exp\big(-\exp\big((\varepsilon^{(i)})^{-2}\big)\big),\\
\varepsilon^{(0)}=\varepsilon.
\end{cases}\label{eq:RecursiveEi}
\end{equation}
Note that, as $\varepsilon\rightarrow0$: 
\begin{equation}
1\gg\varepsilon^{(0)}\gg\varepsilon^{(1)}\gg\ldots\gg\varepsilon^{(N_{\varepsilon})}.
\end{equation}
\begin{rem*}
In the rest of the paper, we will frequently use the relation (\ref{eq:HierarchyOfParameters})
in order to bound an expression involving $\sigma_{\varepsilon},\delta_{\varepsilon},\rho_{\varepsilon}$
(appearing usually as an error term in some estimate) by a simpler
one; for instance, (\ref{eq:HierarchyOfParameters}) allows us to
bound 
\[
\exp\big(\exp(e^{\delta_{\varepsilon}^{-6}})\big)\le\rho_{\varepsilon}^{\frac{1}{20}}.
\]
We will \emph{not} always explicitly refer to (\ref{eq:HierarchyOfParameters})
when using such bounds while passing from one line to the next in
a complicated estimate. 
\end{rem*}

\subsection{\label{subsec:The-initial-data-family}The initial data family }

In this section, we will define the initial data family $(r_{/}^{(\varepsilon)},(\Omega_{/}^{(\varepsilon)})^{2},\bar{f}_{/}^{(\varepsilon)};\sqrt{-\frac{3}{\Lambda}}\pi)$
appearing in the statement of Theorem \ref{thm:TheTheorem} in terms
of the parameters introduced in the previous section. The construction
of the initial data family will proceed in two steps: We will first
obtain a gauge normalised expression for $(r_{/}^{(\varepsilon)},(\Omega_{/}^{(\varepsilon)})^{2},\bar{f}_{/}^{(\varepsilon)};\sqrt{-\frac{3}{\Lambda}}\pi)$
(in accordance with Definition \ref{def:GaugeNormalisation}) by suitably
prescribing the value of $\bar{f}_{/}^{(\varepsilon)}$ and using
Lemma \ref{lem:FreeFVlasov}, and we will then obtain a smoothly compatible
initial data set through the gauge transformation provided by Lemma
\ref{lem:NormalisationToSmoothCompatibility}.

Let us fix a smooth cut-off function $\chi:\mathbb{R}\rightarrow[0,1]$
such that $\chi|_{[-1.1]}=1$ and $\chi_{\mathbb{R}\backslash(-2,2)}=0$.
The following functions will later be used to define the initial Vlasov
field $\bar{f}_{/}^{(\varepsilon)}$:
\begin{defn}
\label{def:SeedFunctions} For any $\varepsilon\in(0,\varepsilon_{1}]$,
where $0<\varepsilon_{1}\ll1$ is the constant appearing in Definition
\ref{def:HierarchyParameters}, we will define the following sequence
of smooth functions $F_{i}^{(\varepsilon)}:[0,\sqrt{-\frac{3}{\Lambda}}\pi]\times[0,+\infty)^{2}\rightarrow[0,+\infty)$
for any $0\le i\le N_{\varepsilon}$: 
\begin{equation}
F_{i}^{(\varepsilon)}(v;p,l)\doteq\frac{1}{(\varepsilon^{(i)})^{2}}\chi\Big(\frac{\sqrt{-\Lambda}(v-v_{\varepsilon,i})}{\varepsilon^{(i)}}\Big)\cdot\chi\big(p-3\big)\cdot\chi\Big(\frac{\sqrt{-\Lambda}l}{\varepsilon^{(i)}}-4\Big),\label{eq:InitialVlasovSeeds}
\end{equation}
where
\begin{equation}
v_{\varepsilon,i}\doteq\sqrt{-\frac{3}{\Lambda}}\frac{\pi}{2}+\rho_{\varepsilon}^{-1}\sum_{j=0}^{i-1}\varepsilon^{(j)}(-\Lambda)^{-1/2}\label{eq:InitialCenters}
\end{equation}
and $\varepsilon^{(i)}$, $\rho_{\varepsilon}$ are defined in terms
of $\varepsilon$ by (\ref{eq:HierarchyOfParameters}) and (\ref{eq:RecursiveEi}).
\end{defn}
\begin{rem*}
Note that, for any $\varepsilon\ll1$ and any $0\le i\le N_{\varepsilon}$,
the functions $F_{i}^{(\varepsilon)}$ satisfy 
\begin{equation}
\int_{0}^{\sqrt{-\frac{3}{\Lambda}}\pi}(1-\frac{1}{3}\Lambda r_{AdS/}^{2})\Big(\frac{r_{AdS/}T_{vv}^{(AdS)}[F_{i}^{(\varepsilon)}]}{\partial_{v}r_{AdS/}}+\frac{r_{AdS/}T_{uv}^{(AdS)}[F_{i}^{(\varepsilon)}]}{-(\partial_{u}r_{AdS})_{/}}\Big)(v)\,dv\le C\varepsilon^{(i)},\label{eq:SmallnessInitialBeams}
\end{equation}
for some absolute constant $C>0$, where $T_{\mu\nu}^{(AdS)}[F_{i}^{(\varepsilon)}]$
is defined by
\begin{align*}
T_{vv}^{(AdS)}[F_{i}^{(\varepsilon)}] & \doteq\frac{\pi}{2}r_{AdS/}^{-2}(v)\int_{0}^{+\infty}\int_{0}^{+\infty}\Big(\frac{\Omega_{AdS/}^{2}}{\partial_{v}r_{AdS/}}(v)p\Big)^{2}\cdot F_{i}^{(\varepsilon)}(v;p,l)\,\frac{dp}{p}ldl,\\
T_{uv}^{(AdS)}[F_{i}^{(\varepsilon)}] & \doteq\frac{\pi}{2}r_{AdS/}^{-2}(v)\int_{0}^{+\infty}\int_{0}^{+\infty}\frac{\Omega_{AdS/}^{2}(v)l^{2}}{r_{AdS/}^{2}(v)}\cdot F_{i}^{(\varepsilon)}(v;p,l)\,\frac{dp}{p}ldl
\end{align*}
and $r_{AdS/},\text{ }\Omega_{AdS/}^{2}$ are given by (\ref{eq:AdSMetricValues}).
\end{rem*}
We will define the initial data family $(r_{/}^{(\varepsilon)},(\Omega_{/}^{(\varepsilon)})^{2},\bar{f}_{/}^{(\varepsilon)};\sqrt{-\frac{3}{\Lambda}}\pi)$
in terms of $F_{i}^{(\varepsilon)}$ as follows:
\begin{defn}
\label{def:InitialDataFamily} Let $0<\varepsilon_{1}\ll1$ be the
constant appearing in Definition \ref{def:HierarchyParameters}. For
any $\varepsilon\in(0,\varepsilon_{1}]$ and any finite sequence $\{a_{\varepsilon i}\}_{i=0}^{N_{\varepsilon}}\in(0,\sigma_{\varepsilon})$
satisfying the smallness condition
\begin{equation}
\sum_{i=0}^{N_{\varepsilon}}a_{\varepsilon i}\le\rho_{\varepsilon}^{-1}\sigma_{\varepsilon},\label{eq:UpperBoundSumWeightsBeam}
\end{equation}
we will define 
\begin{equation}
F^{(\varepsilon)}(v;p,l)=\sum_{i=0}^{N_{\varepsilon}}a_{\varepsilon i}F_{i}^{(\varepsilon)}(v;p,l),\label{eq:InitialVlasovSeedsTotal}
\end{equation}
where $\bar{f}_{i}^{(\varepsilon)}$ are given by (\ref{eq:InitialVlasovSeeds}).
Let also $(r_{/}^{\prime(\varepsilon)},\Omega_{/}^{\prime(\varepsilon)},\bar{f}_{/}^{\prime(\varepsilon)};\sqrt{-\frac{3}{\Lambda}}\pi)$
be the gauge normalised, asymptotically AdS initial data set provided
by Lemma \ref{lem:FreeFVlasov} for $F=F^{(\varepsilon)}$ and $v_{\mathcal{I}}=\sqrt{-\frac{3}{\Lambda}}\pi$;
recall that, according to Lemma \ref{lem:FreeFVlasov}, $\bar{f}_{/}^{\prime(\varepsilon)}$
is related to $F^{(\varepsilon)}$ by 
\begin{equation}
\bar{f}_{/}^{\prime(\varepsilon)}(v;p^{u},l)=F^{(\varepsilon)}\big(v;\text{ }\partial_{v}r_{/}^{\prime(\varepsilon)}(v)p^{u},\text{ }l\big).\label{eq:AlmostVlasovFieldInvariant Form-1}
\end{equation}

For any $\varepsilon\in(0,\varepsilon_{1}]$, we will define $(r_{/}^{(\varepsilon)},\Omega_{/}^{(\varepsilon)},\bar{f}_{/}^{(\varepsilon)};\sqrt{-\frac{3}{\Lambda}}\pi)$
as the (unique) smoothly compatible, asymptotically AdS initial data
set which is obtained from the gauge normalised initial data set $(r_{/}^{\prime(\varepsilon)},\Omega_{/}^{\prime(\varepsilon)},\bar{f}_{/}^{\prime(\varepsilon)};\sqrt{-\frac{3}{\Lambda}}\pi)$
through the gauge transformation of Lemma \ref{lem:NormalisationToSmoothCompatibility}
(note that the notation for $(r_{/}^{(\varepsilon)},\Omega_{/}^{(\varepsilon)},\bar{f}_{/}^{(\varepsilon)};\sqrt{-\frac{3}{\Lambda}}\pi)$
and $(r_{/}^{\prime(\varepsilon)},\Omega_{/}^{\prime(\varepsilon)},\bar{f}_{/}^{\prime(\varepsilon)};\sqrt{-\frac{3}{\Lambda}}\pi)$
is inverted in Lemma \ref{lem:NormalisationToSmoothCompatibility}),
with $\varepsilon$ in place of $\varepsilon_{0}$ in (\ref{eq:FirstDerivativeTransformation})\textendash (\ref{eq:SecondDerivativeTransformation}).

For $\varepsilon\in(\varepsilon_{1},1]$, we will set 
\begin{equation}
(r_{/}^{(\varepsilon)},(\Omega_{/}^{(\varepsilon)})^{2},\bar{f}_{/}^{(\varepsilon)})\doteq(r_{/}^{(\varepsilon_{1})},(\Omega_{/}^{(\varepsilon_{1})})^{2},\bar{f}_{/}^{(\varepsilon_{1})}).\label{eq:DefinitionFamilyLargeEpsilon}
\end{equation}
\end{defn}
\begin{rem*}
The fact that $F^{(\varepsilon)}$ is compactly supported in $(0,\sqrt{-\frac{3}{\Lambda}}\pi)\times(0,+\infty)^{2}$,
satisfying in particular 
\begin{equation}
F^{(\varepsilon)}(v;p,l)=0\text{ for }v\in[0,v_{\varepsilon,0}-2\frac{\varepsilon^{(0)}}{\sqrt{-\Lambda}}]\cup[v_{\varepsilon,N_{\varepsilon}}+2\frac{\varepsilon^{(N_{\varepsilon})}}{\sqrt{-\Lambda}},+\infty)\text{ or }l\in[0,2\frac{\varepsilon^{(N_{\varepsilon})}}{\sqrt{-\Lambda}}]\label{eq:SupportFepsilon}
\end{equation}
(see (\ref{eq:InitialVlasovSeeds}) and (\ref{eq:InitialVlasovSeedsTotal})),
allows us to apply Lemma \ref{lem:NormalisationToSmoothCompatibility}
in the statement of Definition \ref{def:InitialDataFamily}. In view
of the fact that, as a consequence of (\ref{eq:ConditionGaugeOnlAway}),
the gauge transformation provided by Lemma \ref{lem:NormalisationToSmoothCompatibility}
is the identity when restricted to $v\in[0,v_{\varepsilon,N_{\varepsilon}}+2\frac{\varepsilon^{(N_{\varepsilon})}}{\sqrt{-\Lambda}}]$
(which includes the support of $F^{(\varepsilon)}$ in the $v$-variable),
the following relations hold between $(r_{/}^{(\varepsilon)},\Omega_{/}^{(\varepsilon)},\bar{f}_{/}^{(\varepsilon)};\sqrt{-\frac{3}{\Lambda}}\pi)$
and $(r_{/}^{\prime(\varepsilon)},\Omega_{/}^{\prime(\varepsilon)},\bar{f}_{/}^{\prime(\varepsilon)};\sqrt{-\frac{3}{\Lambda}}\pi)$:
\begin{equation}
\bar{f}_{/}^{(\varepsilon)}=\bar{f}_{/}^{\prime(\varepsilon)}\text{ on }(0,\sqrt{-\frac{3}{\Lambda}}\pi)\times[0,+\infty)^{2}\label{eq:EqualVlasovFieldsTransformation}
\end{equation}
and 
\begin{equation}
(r_{/}^{(\varepsilon)},(\Omega_{/}^{(\varepsilon)})^{2})(v)=(r_{/}^{\prime(\varepsilon)},(\Omega_{/}^{\prime(\varepsilon)})^{2})(v)\text{ for }v\in[0,v_{\varepsilon,N_{\varepsilon}}+2\frac{\varepsilon^{(N_{\varepsilon})}}{\sqrt{-\Lambda}}].\label{eq:EqualComponentsTransformation}
\end{equation}
In particular, (\ref{eq:InitialVlasovSeedsTotal}) and (\ref{eq:AlmostVlasovFieldInvariant Form-1})
imply that:

\begin{equation}
\bar{f}_{/}^{(\varepsilon)}(v;p^{u},l)=\sum_{i=0}^{N_{\varepsilon}}a_{\varepsilon i}F_{i}^{(\varepsilon)}\big(v;\text{ }\partial_{v}r_{/}^{\prime(\varepsilon)}(v)p^{u},\text{ }l\big).\label{eq:InitialVlasovTotal}
\end{equation}
\end{rem*}
\begin{figure}[h] 
\centering 
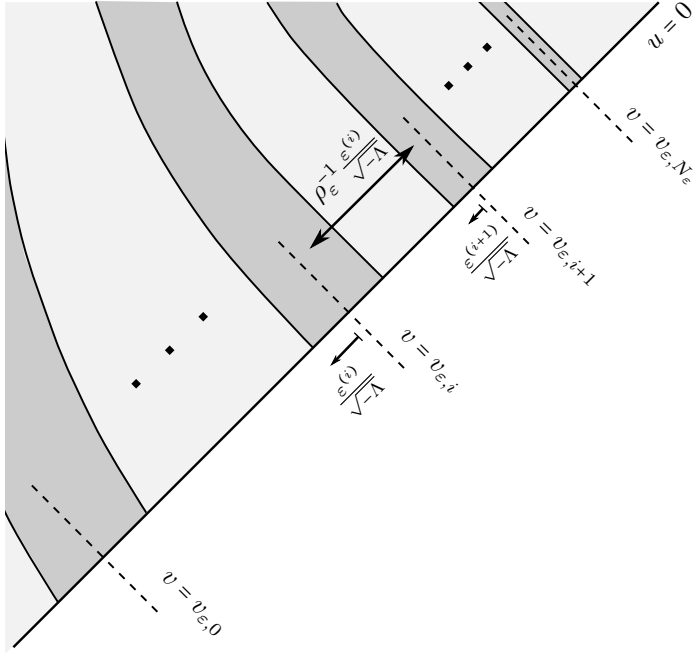 
\caption{Schematic depiction of the support in the $(u,v)$-plane of the Vlasov beams emanating from the initial data at $u=0$. The $i$-th beam in the picture corresponds to $\bar{f}_{i}^{(\varepsilon)}$, $0\le i\le N_{\varepsilon}$, and contains null geodesics with angular momenta $\sim \varepsilon^{(i)}(-\Lambda)^{-1/2}$, i.\,e.~proportional to the width of the beam.}
\end{figure}

The following estimates for the initial data family $(r_{/}^{(\varepsilon)},(\Omega_{/}^{(\varepsilon)})^{2},\bar{f}_{/}^{(\varepsilon)})$
are an immediate consequence of the expression (\ref{eq:InitialVlasovSeeds})
for $F_{i}^{(\varepsilon)}$ and the quantitative bounds provided
by Lemmas \ref{lem:FreeFVlasov} and \ref{lem:NormalisationToSmoothCompatibility}:
\begin{lem}
\label{lem:FirstSmallnessBoundInitialData} There exists some $C>0$,
such that, for all $\varepsilon\in(0,1)$, the initial data set $(r_{/}^{(\varepsilon)},(\Omega_{/}^{(\varepsilon)})^{2},\bar{f}_{/}^{(\varepsilon)};\sqrt{-\frac{3}{\Lambda}}\pi)$
satisfies: 
\begin{equation}
\sup_{v\in(0,\sqrt{-\frac{3}{\Lambda}}\pi)}\Bigg(\Big|\frac{\partial_{v}r_{/}^{(\varepsilon)}}{1-\frac{1}{3}\Lambda(r_{/}^{(\varepsilon)})^{2}}-\frac{\partial_{v}r_{AdS/}}{1-\frac{1}{3}\Lambda r_{AdS/}^{2}}\Big|(v)+\Big|\frac{(\Omega_{/}^{(\varepsilon)})^{2}}{1-\frac{1}{3}\Lambda(r_{/}^{(\varepsilon)})^{2}}-\frac{\Omega_{AdS/}^{2}}{1-\frac{1}{3}\Lambda r_{AdS/}^{2}}\Big|(v)\Bigg)\le C\varepsilon.\label{eq:ComparableInitialDataWithAdSEpsilonn}
\end{equation}
Furthermore, for any $i=0,\ldots,N_{\varepsilon}$, we can estimate
on the support of $F_{i}^{(\varepsilon)}$:
\begin{equation}
\sup_{v\in(0,\sqrt{-\frac{3}{\Lambda}}\pi)}\Bigg[\Big|\partial_{v}\Big(\frac{\partial_{v}r_{/}^{(\varepsilon)}}{1-\frac{1}{3}\Lambda(r_{/}^{(\varepsilon)})^{2}}\Big)\Big|(v)+\Big|\partial_{v}\Big(\frac{(\Omega_{/}^{(\varepsilon)})^{2}}{1-\frac{1}{3}\Lambda(r_{/}^{(\varepsilon)})^{2}}\Big)\Big|(v)\Bigg]\le C\sqrt{-\Lambda}.\label{eq:BoundDvdvRAndDvOmegaInitiallyInBeams}
\end{equation}
Finally, the following estimate holds: 
\begin{equation}
\sup_{v\in(0,\sqrt{-\frac{3}{\Lambda}}\pi)}\frac{2\tilde{m}_{/}^{(\varepsilon)}}{r_{/}^{(\varepsilon)}}<C\varepsilon,\label{eq:SmallRationInitially}
\end{equation}
where $\tilde{m}_{/}^{(\varepsilon)}$ is defined in terms of $r_{/}^{(\varepsilon)},\text{ }(\Omega_{/}^{(\varepsilon)})^{2},\text{ }\bar{f}_{/}^{(\varepsilon)}$
by (\ref{eq:TildeVMaza}), i.\,e.~by the implicit relation 
\begin{equation}
\tilde{m}_{/}^{(\varepsilon)}(v)=2\pi\int_{0}^{v}\Bigg(\big(1-\frac{2\tilde{m}_{/}^{(\varepsilon)}}{r_{/}^{(\varepsilon)}}-\frac{1}{3}\Lambda r_{/}^{(\varepsilon)}\big)\frac{(r_{/}^{(\varepsilon)})^{2}(T_{/}^{(\varepsilon)})_{vv}}{\partial_{v}r_{/}^{(\varepsilon)}}+4\frac{\partial_{v}r_{/}^{(\varepsilon)}}{(\Omega_{/}^{(\varepsilon)})^{2}}(r_{/}^{(\varepsilon)})^{2}(T_{/}^{(\varepsilon)})_{uv}\Bigg)(\bar{v})\,d\bar{v}.\label{eq:RenormalisedMassInitially}
\end{equation}
\end{lem}
\begin{proof}
In view of (\ref{eq:DefinitionFamilyLargeEpsilon}), it suffices to
establish (\ref{eq:ComparableInitialDataWithAdSEpsilonn})\textendash (\ref{eq:SmallRationInitially})
for $\varepsilon\in(0,\varepsilon_{1}]$. Furthermore, it suffices
to establish the bounds (\ref{eq:ComparableInitialDataWithAdSEpsilonn})\textendash (\ref{eq:SmallRationInitially})
for the intermediate, gauge normalised initial data set $(r_{/}^{\prime(\varepsilon)},\Omega_{/}^{\prime(\varepsilon)},\bar{f}_{/}^{\prime(\varepsilon)};\sqrt{-\frac{3}{\Lambda}}\pi)$
constructed in Definition \ref{def:InitialDataFamily}. Since $(r_{/}^{\prime(\varepsilon)},\Omega_{/}^{\prime(\varepsilon)},\bar{f}_{/}^{\prime(\varepsilon)};\sqrt{-\frac{3}{\Lambda}}\pi)$
is related to $(r_{/}^{(\varepsilon)},\Omega_{/}^{(\varepsilon)},\bar{f}_{/}^{(\varepsilon)};\sqrt{-\frac{3}{\Lambda}}\pi)$
by the gauge transformation of Lemma \ref{lem:NormalisationToSmoothCompatibility}
with $\varepsilon$ in place of $\varepsilon_{0}$,\footnote{Note that the $\prime$-notation for $(r_{/}^{(\varepsilon)},\Omega_{/}^{(\varepsilon)},\bar{f}_{/}^{(\varepsilon)};\sqrt{-\frac{3}{\Lambda}}\pi)$
and $(r_{/}^{\prime(\varepsilon)},\Omega_{/}^{\prime(\varepsilon)},\bar{f}_{/}^{\prime(\varepsilon)};\sqrt{-\frac{3}{\Lambda}}\pi)$
is inverted in Lemma \ref{lem:NormalisationToSmoothCompatibility}.} the bounds (\ref{eq:FirstDerivativeTransformation}), (\ref{eq:SecondDerivativeTransformation})
and (\ref{eq:SmallnessInitialBeams}) imply (in view of the relation
\begin{align}
r_{/}^{\prime(\varepsilon)}(V(v)) & =r_{/}^{(\varepsilon)}(v),\label{eq:InitialDataGaugeTransformation-1}\\
(\Omega_{/}^{\prime(\varepsilon)})^{2}(V(v)) & =\frac{1}{\frac{dV}{du}(0)\cdot\frac{dV}{dv}(v)}\Omega_{/}^{(\varepsilon)}(v),\nonumber 
\end{align}
 between $(r_{/}^{\prime(\varepsilon)},\Omega_{/}^{\prime(\varepsilon)})$
and $(r_{/}^{(\varepsilon)},\Omega_{/}^{(\varepsilon)})$) that if
(\ref{eq:ComparableInitialDataWithAdSEpsilonn})\textendash (\ref{eq:BoundDvdvRAndDvOmegaInitiallyInBeams})
hold for $(r_{/}^{\prime(\varepsilon)},\Omega_{/}^{\prime(\varepsilon)},\bar{f}_{/}^{\prime(\varepsilon)};\sqrt{-\frac{3}{\Lambda}}\pi)$,
then they also hold for $(r_{/}^{(\varepsilon)},\Omega_{/}^{(\varepsilon)},\bar{f}_{/}^{(\varepsilon)};\sqrt{-\frac{3}{\Lambda}}\pi)$
with $2C$ in place of $C$. The bound (\ref{eq:SmallRationInitially}),
on the other hand, is gauge invariant.

The bound (\ref{eq:ComparableInitialDataWithAdSEpsilonn}) for $r_{/}^{\prime(\varepsilon)}$
is a corollary of the gauge condition (\ref{eq:NormalisedGaugeCondition})
relating $\Omega_{/}^{\prime(\varepsilon)}$ to $\partial_{v}r_{/}^{\prime(\varepsilon)}$
and Lemma \ref{lem:FreeFVlasov} relating $(r_{/}^{\prime(\varepsilon)},\Omega_{/}^{\prime(\varepsilon)},\bar{f}_{/}^{\prime(\varepsilon)};\sqrt{-\frac{3}{\Lambda}}\pi)$
to $F^{(\varepsilon)}$ (in particular, the estimate (\ref{eq:InitialEstimateDifferenceFromAdS})),
noting that the bound (\ref{eq:SmallnessConditionForConstruction})
for $F^{(\varepsilon)}$ is a direct consequence of (\ref{eq:SmallnessInitialBeams})
(assuming that $\varepsilon_{1}$ has been fixed smaller than $c_{0}$
in (\ref{eq:SmallnessConditionForConstruction})).

We will now proceed to establish that 
\begin{equation}
\sup_{v\in(0,v_{\mathcal{I}})}\Big|\partial_{v}\Big(\frac{\partial_{v}r_{/}^{\prime(\varepsilon)}}{1-\frac{1}{3}\Lambda(r_{/}^{\prime(\varepsilon)})^{2}}\Big)(v)\Big|\le C\sqrt{-\Lambda}.\label{eq:BoundDvdvRAndDvOmegaInitiallyInBeams-1}
\end{equation}
Note that, since $\Omega_{/}^{\prime(\varepsilon)}$ and $\partial_{v}r_{/}^{\prime(\varepsilon)}$
are related by the gauge normalising condition (\ref{eq:NormalisedGaugeCondition}),
the bounds (\ref{eq:BoundDvdvRAndDvOmegaInitiallyInBeams})\textendash (\ref{eq:BoundDvdvRAndDvOmegaInitiallyInBeams})
for $(r_{/}^{\prime(\varepsilon)},\Omega_{/}^{\prime(\varepsilon)})$
follow immediately from (\ref{eq:BoundDvdvRAndDvOmegaInitiallyInBeams-1})
and (\ref{eq:ComparableInitialDataWithAdSEpsilonn}).

The alternative form (\ref{eq:FormulaForDvR/2}) of the gauge condition
(\ref{eq:NormalisedGaugeCondition}) for $(r_{/}^{\prime(\varepsilon)},\Omega_{/}^{\prime(\varepsilon)})$
yields, after differentiating in $v$:
\begin{equation}
\partial_{v}\Big(\frac{\partial_{v}r_{/}^{\prime(\varepsilon)}}{1-\frac{1}{3}\Lambda(r_{/}^{\prime(\varepsilon)})^{2}}\Big)(v)=\Big(4\pi\frac{r_{/}^{\prime(\varepsilon)}(T_{/}^{\prime(\varepsilon)})_{vv}}{(\partial_{v}r_{/}^{\prime(\varepsilon)})}(v)\Big)\cdot\frac{\partial_{v}r_{/}^{\prime(\varepsilon)}}{1-\frac{1}{3}\Lambda(r_{/}^{\prime(\varepsilon)})^{2}},\label{eq:FormulaForDvR/2-1}
\end{equation}
where, in view of the relation (\ref{eq:AlmostVlasovFieldInvariant Form-1})
between $\bar{f}_{/}^{\prime(\varepsilon)}$ and $F^{(\varepsilon)}$
and the gauge condition (\ref{eq:NormalisedGaugeCondition}) for $(r_{/}^{\prime(\varepsilon)},\Omega_{/}^{\prime(\varepsilon)})$:
\begin{equation}
(T_{/}^{\prime(\varepsilon)})_{vv}(v)\doteq8\pi\frac{(\partial_{v}r_{/}^{\prime(\varepsilon)})^{2}}{\big(r_{/}^{\prime(\varepsilon)}-\frac{1}{3}\Lambda(r_{/}^{\prime(\varepsilon)})^{3}\big)^{2}}(v)\int_{0}^{+\infty}\int_{0}^{+\infty}p^{2}F^{(\varepsilon)}(v;p,l)\,\frac{dp}{p}ldl.\label{eq:ExpressionEnergyMomentumPrime}
\end{equation}
In view of (\ref{eq:FormulaForDvR/2-1}) and (\ref{eq:ExpressionEnergyMomentumPrime}),
the bound (\ref{eq:BoundDvdvRAndDvOmegaInitiallyInBeams-1}) follows
from the expression (\ref{eq:InitialVlasovSeeds}) for $F_{i}^{(\varepsilon)}$
and the fact that $F^{(\varepsilon)}=\sum_{i=0}^{N_{\varepsilon}}F_{i}^{(\varepsilon)}$.

The bound (\ref{eq:SmallRationInitially}) follows readily by applying
Gronwall's inequality on the integral relation (\ref{eq:RenormalisedMassInitially})
for $\frac{2\tilde{m}_{/}^{(\varepsilon)}}{r_{/}^{(\varepsilon)}}$,
using the estimates (\ref{eq:SmallnessInitialBeams}) and (\ref{eq:ComparableInitialDataWithAdSEpsilonn}).
\end{proof}
The initial data family $(r_{/}^{(\varepsilon)},(\Omega_{/}^{(\varepsilon)})^{2},\bar{f}_{/}^{(\varepsilon)})$
satisfies the following smallness condition with respect to the initial
data norm $||\cdot||$ introduced by Definition \ref{def:InitialDataNorm}:
\begin{lem}
\label{lem:SmallnessFamilyInitialData} There exists some $C>0$,
such that, for all $\varepsilon\in(0,1)$, the initial data set $(r_{/}^{(\varepsilon)},(\Omega_{/}^{(\varepsilon)})^{2},\bar{f}_{/}^{(\varepsilon)};\sqrt{-\frac{3}{\Lambda}}\pi)$
satisfies: 
\begin{equation}
||(r_{/}^{(\varepsilon)},(\Omega_{/}^{(\varepsilon)})^{2},\bar{f}_{/}^{(\varepsilon)};\sqrt{-\frac{3}{\Lambda}}\pi)||\le C\sigma_{\varepsilon},\label{eq:SizeInitialData}
\end{equation}
where $||\cdot||$ is defined by (\ref{eq:InitialDataNorm}). In particular,
as a consequence of (\ref{eq:ParametersGoingToZero}):
\begin{equation}
\lim_{\varepsilon\rightarrow0}||(r_{/}^{(\varepsilon)},(\Omega_{/}^{(\varepsilon)})^{2},\bar{f}_{/}^{(\varepsilon)};\sqrt{-\frac{3}{\Lambda}}\pi)||=0.\label{eq:InitialDataGoingTo0}
\end{equation}
\end{lem}
\begin{proof}
Since the bound (\ref{eq:SizeInitialData}) is non-trivial only in
the limit $\varepsilon\rightarrow0$, it suffices to establish it
for $\varepsilon\in(0,\varepsilon_{1}]$. 

Let $(r_{/}^{\prime(\varepsilon)},(\Omega_{/}^{\prime(\varepsilon)})^{2},\bar{f}_{/}^{\prime(\varepsilon)};\sqrt{-\frac{3}{\Lambda}}\pi)$
be the gauge-normalised expression of $(r_{/}^{(\varepsilon)},(\Omega_{/}^{(\varepsilon)})^{2},\bar{f}_{/}^{(\varepsilon)};\sqrt{-\frac{3}{\Lambda}}\pi)$,
constructed in Definition \ref{def:InitialDataFamily}. Recall that
$(r_{/}^{\prime(\varepsilon)},(\Omega_{/}^{\prime(\varepsilon)})^{2},\bar{f}_{/}^{\prime(\varepsilon)};\sqrt{-\frac{3}{\Lambda}}\pi)$
and $(r_{/}^{(\varepsilon)},(\Omega_{/}^{(\varepsilon)})^{2},\bar{f}_{/}^{(\varepsilon)};\sqrt{-\frac{3}{\Lambda}}\pi)$
satisfy (\ref{eq:EqualVlasovFieldsTransformation}) and (\ref{eq:EqualComponentsTransformation}),
i.\,e.~coincide in the support of $\bar{f}_{/}^{(\varepsilon)}=\bar{f}_{/}^{\prime(\varepsilon)}$.

For any $\varepsilon\in(0,\varepsilon_{1}]$ and $0\le i\le N_{\varepsilon}$,
let us define $f_{\varepsilon,i}^{(AdS)}$ and $\Big[\frac{rT_{\alpha\beta}}{\partial_{v}r}\Big]_{\varepsilon,i}^{(AdS)}$
in terms of $v_{\mathcal{I}}=\sqrt{-\frac{3}{\Lambda}}\pi$ and 
\begin{equation}
\bar{f}_{i}^{(\varepsilon)}(v;p^{u},l)\doteq F_{i}^{(\varepsilon)}\big(v;\text{ }\partial_{v}r_{/}^{\prime(\varepsilon)}(v)p^{u},\text{ }l\big)\label{eq:FiEpsilonDefinition-1}
\end{equation}
as in Definition \ref{def:ComparisonVlasovField}; we will similarly
define $f_{\varepsilon}^{(AdS)}$ and $\Big[\frac{rT_{\alpha\beta}}{\partial_{v}r}\Big]_{\varepsilon}^{(AdS)}$
in terms of $\bar{f}_{/}^{\prime(\varepsilon)}=\bar{f}_{/}^{(\varepsilon)}$
(see (\ref{eq:EqualVlasovFieldsTransformation})). Note that the relations
(\ref{eq:AlmostVlasovFieldInvariant Form-1}), (\ref{eq:EqualVlasovFieldsTransformation})
and (\ref{eq:InitialVlasovTotal}) imply that 
\begin{equation}
f_{\varepsilon}^{(AdS)}\doteq\sum_{i=0}^{N_{\varepsilon}}a_{\varepsilon i}f_{\varepsilon,i}^{(AdS)}\label{eq:SumVlasovBeams}
\end{equation}
and 
\begin{equation}
\Big[\frac{rT_{\alpha\beta}}{\partial_{v}r}\Big]_{\varepsilon}^{(AdS)}=\sum_{i=0}^{N_{\varepsilon}}a_{\varepsilon i}\Big[\frac{rT_{\alpha\beta}}{\partial_{v}r}\Big]_{\varepsilon,i}^{(AdS)}.\label{eq:SumEnergyMomenta}
\end{equation}
Note that, in view of (\ref{eq:EqualComponentsTransformation}) and
(\ref{eq:SupportFepsilon}), (\ref{eq:FiEpsilonDefinition-1}) implies
that
\begin{equation}
\bar{f}_{i}^{(\varepsilon)}(v;p^{u},l)\doteq F_{i}^{(\varepsilon)}\big(v;\text{ }\partial_{v}r_{/}^{(\varepsilon)}(v)p^{u},\text{ }l\big)\label{eq:FiEpsilonDefinition}
\end{equation}

From the expression (\ref{eq:InitialVlasovSeeds}) for $F_{i}^{(\varepsilon)}$,
the relation (\ref{eq:FiEpsilonDefinition}) between $\bar{f}_{i}^{(\varepsilon)}$
and $F_{i}^{(\varepsilon)}$, the bound (\ref{eq:ComparableInitialDataWithAdSEpsilonn})
for $r_{/}^{(\varepsilon)}$ and properties of the geodesic flow on
AdS when $\frac{l}{E}\ll1$ (see the relation A.10 in the Appendix
of \cite{MoschidisVlasovWellPosedness}), it readily follows that
the support of $f_{\varepsilon,i}^{(AdS)}$ satisfies (for some absolute
constant $C>0$) 
\begin{equation}
supp\big(f_{\varepsilon,i}^{(AdS)}\big)\subset\Bigg(\Big\{(u,v)\in\mathcal{V}_{\varepsilon,i}^{(AdS)}\Big\}\cap\Big\{\Omega_{AdS}^{2}(u,v)\Big(p^{u}+p^{v}\Big)\le C\Big\}\cap\Big\{2\le\frac{\sqrt{-\Lambda}l}{\varepsilon^{i+1}}\le6\Big\}\Bigg),\label{eq:SupportOfVlasovOnAdS}
\end{equation}
where 
\begin{align}
\mathcal{V}_{\varepsilon,i}^{(AdS)}\doteq\bigcup_{k\in\mathbb{Z}}\Big(\Big\{\Big|v & -v_{\varepsilon,i}-k\sqrt{-\frac{3}{\Lambda}}\pi\Big|\le C\frac{\varepsilon^{(i)}}{\sqrt{-\Lambda}}\Big\}\cup\Big\{\Big|u-v_{\varepsilon,i}-k\sqrt{-\frac{3}{\Lambda}}\pi\Big|\le C\frac{\varepsilon^{(i)}}{\sqrt{-\Lambda}}\Big\}\Big)\cap\label{eq:DomainForBeamsOnAdS}\\
 & \cap\Big\{ r_{AdS}(u,v)\ge C^{-1}\frac{\varepsilon^{(i)}}{\sqrt{-\Lambda}}\Big\}.\nonumber 
\end{align}

\begin{figure}[h] 
\centering 
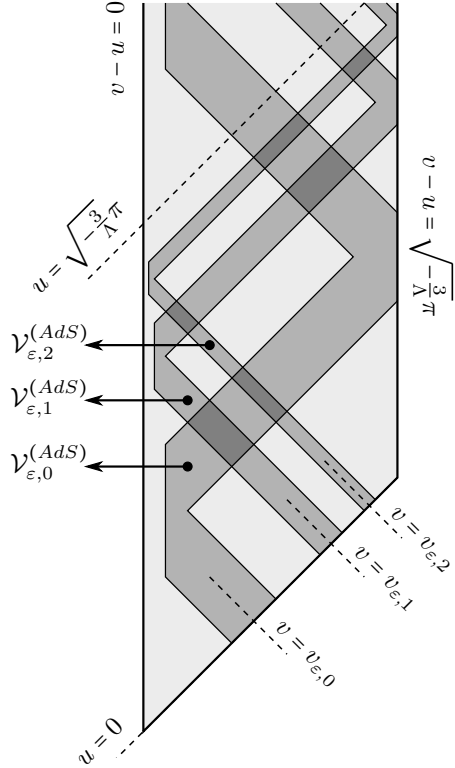 
\caption{Schematic depiction of the regions $\mathcal{V}_{\varepsilon,i}^{(AdS)}$ bounding the trajectories of the geodesics in the support of $f_{\epsilon,i}^{(AdS)}$, for $0\le i\le 2$. The region $\mathcal{V}_{\varepsilon,i}^{(AdS)}$ has width proportional to $\frac{\varepsilon^{(i)}}{\sqrt{-\Lambda}}$. The minimum value of $r$ along $\mathcal{V}_{\varepsilon,i}^{(AdS)}$ is also proportional to $\frac{\varepsilon^{(i)}}{\sqrt{-\Lambda}}$. On the other hand, the separation of $\mathcal{V}_{\varepsilon,i}^{(AdS)}$ from the rest of these regions when $u=0$ is proportional to $\rho_{\varepsilon}^{-1} \frac{\varepsilon^{(i)}}{\sqrt{-\Lambda}}$. \label{fig:Domains_AdS}}
\end{figure}

In view of (\ref{eq:SupportOfVlasovOnAdS}), using (\ref{eq:VDerivativeMassfromTotalParticleCurrent})
(with $r_{AdS}$ in place of $r$ and $m_{AdS}=\frac{1}{6}\Lambda r_{AdS}^{3}$
in place of $m$), as well as the conservation (\ref{eq:ConservedParticle})
of the particle current $N_{\mu}[f_{\varepsilon,i}^{(AdS)}]$ of $f_{\varepsilon,i}^{(AdS)}$,
we can readily calculate for any $\varepsilon\in(0,\varepsilon_{1}]$,
any $U_{*}\ge0$ and any $0\le i\le N_{\varepsilon}$:
\begin{align}
\int_{U_{*}}^{U_{*}+\sqrt{-\frac{3}{\Lambda}}\pi}\Big(\Big[\frac{rT_{vv}}{\partial_{v}r}\Big]_{\varepsilon,i}^{(AdS)}(U_{*},v)+\Big[\frac{rT_{uv}}{-\partial_{u}r}\Big]_{\varepsilon,i}^{(AdS)}(U_{*},v)\Big)\,dv & \le C\int_{U_{*}}^{U_{*}+\sqrt{-\frac{3}{\Lambda}}\pi}r(U_{*},v)\cdot N_{v}[f_{\varepsilon,i}^{(AdS)}](U_{*},v)\,dv\le\label{eq:FirstBoundForInitialDataNorm}\\
 & \le C\Big(\sup_{\mathcal{V}_{\varepsilon,i}^{(AdS)}\cap\{u=U_{*}\}}\frac{1}{r_{AdS}}\Big)\cdot\int_{u=0}r_{AdS}^{2}N_{v}[f_{\varepsilon,i}^{(AdS)}]\,dv\nonumber 
\end{align}
for some absolute constant $C>0$. Using the expression (\ref{eq:ComponentsParticle}),
the explicit formula (\ref{eq:InitialVlasovSeeds}) for $F_{i}^{(\varepsilon)}$
and the bound (\ref{eq:ComparableInitialDataWithAdSEpsilonn}) (the
latter used in order to estimate the term $\partial_{v}r_{/}^{(\varepsilon)}$
in the relation (\ref{eq:FiEpsilonDefinition}) between $F_{i}^{(\varepsilon)}$
and $f_{\varepsilon,i}^{(AdS)}|_{u=0}$), we can readily estimate
\begin{equation}
\int_{u=0}r_{AdS}^{2}N_{v}[f_{\varepsilon,i}^{(AdS)}]\,dv\le C\frac{\varepsilon^{(i)}}{\sqrt{-\Lambda}},\label{eq:InitialCurrentEstimate}
\end{equation}
thus obtaining from (\ref{eq:FirstBoundForInitialDataNorm}) that,
for any $U_{*}\ge0$ and any $0\le i\le N_{\varepsilon}$: 
\begin{equation}
\int_{U_{*}}^{U_{*}+\sqrt{-\frac{3}{\Lambda}}\pi}\Big(\Big[\frac{rT_{vv}}{\partial_{v}r}\Big]_{\varepsilon,i}^{(AdS)}(U_{*},v)+\Big[\frac{rT_{uv}}{-\partial_{u}r}\Big]_{\varepsilon,i}^{(AdS)}(U_{*},v)\Big)\,dv\le C\frac{\varepsilon^{(i)}}{\sqrt{-\Lambda}}\sup_{\mathcal{V}_{\varepsilon,i}^{(AdS)}\cap\{u=U_{*}\}}\frac{1}{r_{AdS}}.\label{eq:SecondBoudForInitialDataNorm}
\end{equation}

For any $\varepsilon\in(0,\varepsilon_{1}]$ and any $U_{*}\ge0$,
let us define $i_{\varepsilon}[U_{*}]$ to be the unique integer $i$
in $[0,N_{\varepsilon}]$ for which 
\begin{equation}
v_{\varepsilon,i-1}+k\sqrt{-\frac{3}{\Lambda}}\pi+\rho_{\varepsilon}^{-\frac{1}{2}}\frac{\varepsilon^{(i-1)}}{\sqrt{-\Lambda}}<U_{*}\le v_{\varepsilon,i}+k\sqrt{-\frac{3}{\Lambda}}\pi+\rho_{\varepsilon}^{-\frac{1}{2}}\frac{\varepsilon^{(i)}}{\sqrt{-\Lambda}}
\end{equation}
for some $k\in\mathbb{Z}$, with the convention that 
\begin{equation}
v_{\varepsilon,-1}\doteq v_{\varepsilon,N_{\varepsilon}}-\sqrt{-\frac{3}{\Lambda}}\pi
\end{equation}
and 
\[
\varepsilon^{(-1)}\doteq\varepsilon^{(N_{\varepsilon})}.
\]
Then, in view of the bound (\ref{eq:SupportOfVlasovOnAdS}) for the
support of $f_{\varepsilon,i}^{(AdS)}$, the definition (\ref{eq:DomainForBeamsOnAdS})
of the domains $\mathcal{V}_{\varepsilon,i}^{(AdS)}$ (see also Figure
\ref{fig:Domains_AdS}), the relation (\ref{eq:InitialCenters}) defining
$v_{\varepsilon,i}$ and the relations (\ref{eq:HierarchyOfParameters})
and (\ref{eq:RecursiveEi}) between $\varepsilon$, $\rho_{\varepsilon}$
and $\varepsilon^{(i)}$, we infer that, given any $U_{*}\ge0$, we
can bound 
\begin{equation}
\begin{cases}
\inf_{\mathcal{V}_{\varepsilon,i}^{(AdS)}\cap\{u=U_{*}\}}r_{AdS}\ge c\varepsilon^{(i)}(-\Lambda)^{-\frac{1}{2}}, & \text{ for }i=i_{\varepsilon}[U_{*}],\\
\inf_{\mathcal{V}_{\varepsilon,i}^{(AdS)}\cap\{u=U_{*}\}}r_{AdS}\ge\rho_{\varepsilon}^{-1}\varepsilon^{(i)}(-\Lambda)^{-\frac{1}{2}}, & \text{ for }i_{\varepsilon}[U_{*}]<i\le N_{\varepsilon},\\
\inf_{\mathcal{V}_{\varepsilon,i}^{(AdS)}\cap\{u=U_{*}\}}r_{AdS}\ge c(-\Lambda)^{-\frac{1}{2}}, & \text{ for }0\le i<i_{\varepsilon}[U_{*}],
\end{cases}\label{eq:MinimumDistanceForBeams}
\end{equation}
for some absolute constant $c>0$. Multiplying (\ref{eq:SecondBoudForInitialDataNorm})
with $a_{\varepsilon i}$ and adding the resulting bounds for $0\le i\le N_{\varepsilon}$,
we therefore infer using (\ref{eq:MinimumDistanceForBeams}) (in view
also of (\ref{eq:SumEnergyMomenta})) and (\ref{eq:RecursiveEi})
that, for any $\varepsilon\in(0,\varepsilon_{1}]$ and any $U_{*}\ge0$:
\begin{equation}
\int_{U_{*}}^{U_{*}+\sqrt{-\frac{3}{\Lambda}}\pi}\Big(\Big[\frac{rT_{vv}}{\partial_{v}r}\Big]_{\varepsilon}^{(AdS)}(U_{*},v)+\Big[\frac{rT_{uv}}{-\partial_{u}r}\Big]_{\varepsilon}^{(AdS)}(U_{*},v)\Big)\,dv\le C\Big(\varepsilon+\max_{0\le i\le N_{\varepsilon}}a_{\varepsilon i}+\sum_{i=0}^{N_{\varepsilon}}\rho_{\varepsilon}a_{\varepsilon i}\Big).\label{eq:ComeOn1}
\end{equation}
Similarly, we can estimate for any $V_{*}\ge0$:
\begin{equation}
\int_{\max\{0,V_{*}-\sqrt{-\frac{3}{\Lambda}}\pi\}}^{V_{*}}\Big(\Big[\frac{rT_{uu}}{-\partial_{u}r}\Big]_{\varepsilon}^{(AdS)}(u,V_{*})+\Big[\frac{rT_{uv}}{\partial_{v}r}\Big]_{\varepsilon}^{(AdS)}(u,V_{*})\Big)\,du\le C\Big(\varepsilon+\max_{0\le i\le N_{\varepsilon}}a_{\varepsilon i}+\sum_{i=0}^{N_{\varepsilon}}\rho_{\varepsilon}a_{\varepsilon i}\Big).\label{eq:ComeOn2}
\end{equation}

Using the expression (\ref{eq:ComponentsStressEnergy}), the explicit
formula (\ref{eq:InitialVlasovSeeds}) for $F_{i}^{(\varepsilon)}$
and the relation (\ref{eq:FiEpsilonDefinition}) between $F_{i}^{(\varepsilon)}$
for $f_{\varepsilon,i}^{(AdS)}|_{u=0}$ (arguing similarly as for
obtaining (\ref{eq:InitialCurrentEstimate})), we can bound 
\begin{equation}
\sqrt{-\Lambda}\tilde{m}_{\varepsilon}|_{\mathcal{I}}\le C\big(\max_{0\le i\le N_{\varepsilon}}a_{\varepsilon i}\big)\varepsilon.\label{eq:ComeOn3}
\end{equation}
 Therefore, for any $\varepsilon\in(0,\varepsilon_{1}]$, the following
estimate for the size of $(r_{/}^{(\varepsilon)},(\Omega_{/}^{(\varepsilon)})^{2},\bar{f}_{/}^{(\varepsilon)};\sqrt{-\frac{3}{\Lambda}}\pi)$
with respect to the norm (\ref{eq:InitialDataNorm}) follows readily
by adding (\ref{eq:ComeOn1}), (\ref{eq:ComeOn2}) and (\ref{eq:ComeOn3})
and using (\ref{eq:HierarchyOfParameters}), (\ref{eq:UpperBoundSumWeightsBeam})
and the assumption that $a_{\varepsilon i}\in(0,\sigma_{\varepsilon})$:
\begin{equation}
||(r_{/}^{(\varepsilon)},(\Omega_{/}^{(\varepsilon)})^{2},\bar{f}_{/}^{(\varepsilon)};\sqrt{-\frac{3}{\Lambda}}\pi)||\le C\Big(\sigma_{\varepsilon}+\sum_{i=0}^{N_{\varepsilon}}\rho_{\varepsilon}a_{\varepsilon i}\Big)\le C\sigma_{\varepsilon}.
\end{equation}
In particular, (\ref{eq:SizeInitialData}) holds.
\end{proof}

\subsection{\label{subsec:Notational-conventions-and}Notational conventions
for domains and fundamental computations}

For the rest of this paper, we will denote with $(\mathcal{U}_{max}^{(\varepsilon)};r_{\varepsilon},\Omega_{\varepsilon}^{2},f_{\varepsilon})$
the maximal future development of the initial data set $(r_{/}^{(\varepsilon)},(\Omega_{/}^{(\varepsilon)})^{2},\bar{f}_{/}^{(\varepsilon)};\sqrt{-\frac{3}{\Lambda}}\pi)$
(for the defintion of the notion of a maximal future development for
(\ref{eq:RequationFinal})\textendash (\ref{NullShellFinal}), see
Proposition \ref{cor:MaximalDevelopment}). We will also denote with
$\mathcal{I}_{\varepsilon}$ and $\gamma_{\mathcal{Z}\varepsilon}$
the conformal infinity and axis, respectively, of $(\mathcal{U}_{max}^{(\varepsilon)};r_{\varepsilon},\Omega_{\varepsilon}^{2},f_{\varepsilon})$,
with corresponding endpoint parameters $u_{\mathcal{I}_{\varepsilon}}\in(0,+\infty]$
and $u_{\gamma_{\mathcal{Z}\varepsilon}}\in(0,+\infty]$, defined
in accordance with Definition \ref{def:DevelopmentSets}. Note that
the proof of Theorem \ref{thm:TheTheorem} will consist of showing
that $u_{\mathcal{I}_{\varepsilon}}<+\infty$.
\begin{rem*}
In order to simplify the heavy notation associated with all the parameters
that will be introduced in the proof of Theorem \ref{thm:TheTheorem},
we will frequently \emph{drop }the subscript $\varepsilon$ in $r_{\varepsilon}$,
$\Omega_{\varepsilon}^{2}$, $m_{\varepsilon}$ and $\tilde{m}_{\varepsilon}$,
but not in $f_{\varepsilon}$. Therefore, from now on we will almost
always denote $(\mathcal{U}_{max}^{(\varepsilon)};r_{\varepsilon},\Omega_{\varepsilon}^{2},f_{\varepsilon})$
as $(\mathcal{U}_{max}^{(\varepsilon)};r,\Omega^{2},f_{\varepsilon})$. 
\end{rem*}
For any $\delta\in(0,1]$ and any $\varepsilon\in(0,1]$, we will
define $u_{\delta;\varepsilon}^{+}\in[0,u_{\mathcal{I}_{\varepsilon}}]$
by
\begin{equation}
u_{\delta;\varepsilon}^{+}\doteq\sup\Big\{\bar{u}\in[0,u_{\mathcal{I}_{\varepsilon}}):\text{ }\frac{2\tilde{m}}{r}(u,v)<\delta\text{ for all }(u,v)\in\mathcal{U}_{max}^{(\varepsilon)}\text{ with }0<u<\bar{u}\Big\}.\label{eq:Definitionudelta}
\end{equation}
Similarly, we will define for any $K>0$ and any $\varepsilon\in(0,1]$:
\begin{align}
u_{K;\varepsilon}^{\sharp}\doteq\sup\Big\{\bar{u}\in[0,u_{\mathcal{I}_{\varepsilon}}): & \text{ }\sup_{u\in(0,\bar{u})}\int_{u}^{u+\sqrt{-\frac{3}{\Lambda}}\pi}r\Big(\frac{T_{vv}[f_{\varepsilon}]}{\partial_{v}r}+\frac{T_{uv}[f_{\varepsilon}]}{-\partial_{u}r}\Big)(u,v)\,dv<K\label{eq:DefinitionUSharp}\\
 & \text{ and}\sup_{v\in(0,\bar{u}+\sqrt{-\frac{3}{\Lambda}}\pi)}\int_{\max\{0,v-\sqrt{-\frac{3}{\Lambda}}\pi\}}^{\min\{v,\bar{u}\}}r\Big(\frac{T_{uv}[f_{\varepsilon}]}{\partial_{v}r}+\frac{T_{uu}[f_{\varepsilon}]}{-\partial_{u}r}\Big)(u,v)\,du<K\Big\}.\nonumber 
\end{align}

We will define $\mathcal{U}_{K,\delta}^{(\varepsilon)}\subset\mathcal{U}_{max}^{(\varepsilon)}$
to be the open subset 
\begin{equation}
\mathcal{U}_{K,\delta}^{(\varepsilon)}\doteq\big\{0<u<\min\{u_{\delta;\varepsilon}^{+},\text{ }u_{K;\varepsilon}^{\sharp}\big\}.\label{eq:DefinitionUKappadelta}
\end{equation}
Note that $\mathcal{U}_{K,\delta}^{(\varepsilon)}$ is non-empty if
and only if $\delta$ and $K$ satisfy in terms of the initial data:
\begin{equation}
\sup_{v\in(0,\sqrt{-\frac{3}{\Lambda}}\pi)}\frac{2\tilde{m}_{/}^{(\varepsilon)}}{r_{/}^{(\varepsilon)}}(v)<\delta\label{eq:DeltaForInitialData}
\end{equation}
and 
\begin{equation}
\int_{0}^{\sqrt{-\frac{3}{\Lambda}}\pi}r_{/}^{(\varepsilon)}\Big(\frac{(T_{/}^{(\varepsilon)})_{vv}}{\partial_{v}r_{/}^{(\varepsilon)}}+\frac{(T_{/}^{(\varepsilon)})_{uv}}{-\partial_{u}r_{/}^{(\varepsilon)}}\Big)(v)\,dv<K.\label{eq:LambdaForInitialData}
\end{equation}
As a consequence of (\ref{eq:InitialDataGoingTo0}) and Lemma \ref{lem:FirstSmallnessBoundInitialData},
given any $\delta\in(0,1)$ and $K>0$, there always exists an $\varepsilon_{0}\in(0,\varepsilon_{1}]$
such that (\ref{eq:DeltaForInitialData}) and (\ref{eq:LambdaForInitialData})
are satified for all $\varepsilon\in(0,\varepsilon_{0})$ (in particular,
it suffices to choose any $\varepsilon_{0}$ for which $\sigma_{\varepsilon_{0}}\lesssim\min\{K,\delta\}$).
\begin{rem*}
In the case when $\mathcal{U}_{K,\delta}^{(\varepsilon)}$ is non-empty,
$(\mathcal{U}_{K,\delta}^{(\varepsilon)};r,\Omega^{2},f_{\varepsilon})$
is a future development of $(r_{/}^{(\varepsilon)},(\Omega_{/}^{(\varepsilon)})^{2},\bar{f}_{/}^{(\varepsilon)};\sqrt{-\frac{3}{\Lambda}}\pi)$
for (\ref{eq:RequationFinal})\textendash (\ref{NullShellFinal})
with reflecting boundary conditions on conformal infinity, in accordance
with Definition \ref{def:Development}. Note also that 
\[
u_{1;\varepsilon}^{+}=u_{\mathcal{I}_{\varepsilon}}
\]
 as a trivial consequence of (\ref{eq:NoTrappingPastInfinity}). 
\end{rem*}
Let 
\[
0<\eta_{0}\ll1
\]
 be a small absolute constant. For the rest of the paper, we will
assume that $\eta_{0}$ has been fixed small enough in terms of the
absolute constant $\delta_{0}$ appearing in the statements of Lemma
\ref{lem:GeodesicPaths} and Proposition  \ref{cor:GeneralContinuationCriterion}. 

The following domains in $\mathcal{U}_{max}^{(\varepsilon)}$ will
be play a central role in the proof of Theorem \ref{thm:TheTheorem}: 
\begin{defn}
\label{def:Domains}We will define the domains $\mathcal{U}_{\varepsilon}^{+},\mathcal{T}_{\varepsilon}^{+}\subset\mathcal{U}_{max}^{(\varepsilon)}$
by 
\begin{equation}
\mathcal{U}_{\varepsilon}^{+}\doteq\mathcal{U}_{\sigma_{\varepsilon}^{-1};\eta_{0}}^{(\varepsilon)}\cap\big\{ u<\frac{\sigma_{\varepsilon}^{-2}}{\sqrt{-\Lambda}}\big\}\label{eq:BootstrapDomain}
\end{equation}
and 
\begin{equation}
\mathcal{T}_{\varepsilon}^{+}\doteq\mathcal{U}_{\delta_{\varepsilon}^{-1};\eta_{0}}^{(\varepsilon)}\cap\big\{ u<\frac{\sigma_{\varepsilon}^{-2}}{\sqrt{-\Lambda}}\big\}\label{eq:LargerBootstrapDomain}
\end{equation}
where $\mathcal{U}_{\sigma_{\varepsilon}^{-1};\eta_{0}}^{(\varepsilon)}$,
$\mathcal{U}_{\delta_{\varepsilon}^{-1};\eta_{0}}^{(\varepsilon)}$
are defined by (\ref{eq:DefinitionUKappadelta}). 
\end{defn}
\begin{rem*}
Note that, since $\delta_{\varepsilon}\ll\sigma_{\varepsilon}$, 
\[
\mathcal{U}_{\varepsilon}^{+}\subseteq\mathcal{T}_{\varepsilon}^{+}.
\]
\end{rem*}
Let us define 
\[
u[\mathcal{U}_{\varepsilon}^{+}]\doteq\min\big\{ u_{\eta_{0};\varepsilon}^{+},\text{ }u_{\sigma_{\varepsilon}^{-1};\varepsilon}^{\sharp},\frac{\sigma_{\varepsilon}^{-2}}{\sqrt{-\Lambda}}\big\}
\]
and 
\[
u[\mathcal{T}_{\varepsilon}^{+}]\doteq\min\big\{ u_{\eta_{0};\varepsilon}^{+},\text{ }u_{\delta_{\varepsilon}^{-1};\varepsilon}^{\sharp},\frac{\sigma_{\varepsilon}^{-2}}{\sqrt{-\Lambda}}\big\},
\]
so that 
\begin{align}
\mathcal{U}_{\varepsilon}^{+} & =\big\{0<u<u[\mathcal{U}_{\varepsilon}^{+}]\big\}\cap\big\{ u<v<u+\sqrt{-\frac{3}{\Lambda}}\pi\big\},\label{eq:DefinitionBootstrapDomainWithU}\\
\mathcal{T}_{\varepsilon}^{+} & =\big\{0<u<u[\mathcal{T}_{\varepsilon}^{+}]\big\}\cap\big\{ u<v<u+\sqrt{-\frac{3}{\Lambda}}\pi\big\}.\nonumber 
\end{align}
As an immediate consequence of the extension principle of Proposition
\ref{cor:GeneralContinuationCriterion}, we infer the following condition
for the boundary of $\mathcal{U}_{\varepsilon}^{+}$ and $\mathcal{T}_{\varepsilon}^{+}$
in $\mathcal{U}_{max}^{(\varepsilon)}$ :
\begin{lem}
\label{lem:ExtensionPrincipleSpecialDomains} For any $\varepsilon\in(0,\varepsilon_{1}]$,
\begin{equation}
u[\mathcal{U}_{\varepsilon}^{+}]<u[\mathcal{T}_{\varepsilon}^{+}]<u_{\mathcal{I}_{\varepsilon}}.\label{eq:Blablabla}
\end{equation}
In particular, there exists a $u_{0}>0$ such that 
\begin{equation}
\big\{0<u<u[\mathcal{T}_{\varepsilon}^{+}]+u_{0}\big\}\subset\mathcal{U}_{max}^{(\varepsilon)}.\text{ }\label{eq:Ablabla}
\end{equation}
Furthermore, in the case when 
\begin{equation}
u[\mathcal{U}_{\varepsilon}^{+}]<\frac{\sigma_{\varepsilon}^{-2}}{\sqrt{-\Lambda}},\label{eq:UpperBoundUExtensionCondition}
\end{equation}
one of the following three conditions hold for the future boundary
$\{u=u[\mathcal{U}_{\varepsilon}^{+}]\}$ of $\mathcal{U}_{\varepsilon}^{+}$:
\begin{equation}
\limsup_{p\rightarrow\{u=u[\mathcal{U}_{\varepsilon}^{+}]\}}\frac{2\tilde{m}}{r}(p)=\eta_{0},\label{eq:ExtensionConditionMuTilde}
\end{equation}
\begin{equation}
\text{ }\lim\sup_{u\rightarrow u[\mathcal{U}_{\varepsilon}^{+}]}\int_{u}^{u+\sqrt{-\frac{3}{\Lambda}}\pi}r\Big(\frac{T_{vv}[f_{\varepsilon}]}{\partial_{v}r}+\frac{T_{uv}[f_{\varepsilon}]}{-\partial_{u}r}\Big)(u,v)\,dv=\sigma_{\varepsilon}^{-1},\label{eq:ExtensionConditionNormUCons}
\end{equation}
or 
\begin{equation}
\sup_{v\in(0,u[\mathcal{U}_{\varepsilon}^{+}]+\sqrt{-\frac{3}{\Lambda}}\pi)}\int_{\max\{0,v-\sqrt{-\frac{3}{\Lambda}}\pi\}}^{\min\{v,u[\mathcal{U}_{\varepsilon}^{+}]\}}r\Big(\frac{T_{uv}[f_{\varepsilon}]}{\partial_{v}r}+\frac{T_{uu}[f_{\varepsilon}]}{-\partial_{u}r}\Big)(u,v)\,du=\sigma_{\varepsilon}^{-1}.\label{eq:ExtensionConditionNormVCons}
\end{equation}
Similarly, the same condition holds on $\{u=u[\mathcal{T}_{\varepsilon}^{+}]\}$
for $\mathcal{T}_{\varepsilon}^{+}$, with $\delta_{\varepsilon}$
in place of $\sigma_{\varepsilon}$ in (\ref{eq:ExtensionConditionNormUCons})
and (\ref{eq:ExtensionConditionNormVCons}).
\end{lem}
\begin{proof}
The proof of (\ref{eq:Blablabla}) and (\ref{eq:Ablabla}) follows
immediately by applying Proposition \ref{cor:GeneralContinuationCriterion}
to $\mathcal{T}_{\varepsilon}^{+}$, using the fact that 
\[
\sup_{\mathcal{T}_{\varepsilon}^{+}}\frac{2\tilde{m}}{r}\le\eta_{0}
\]
 (following from (\ref{eq:Definitionudelta}) and the definition of
$\mathcal{T}_{\varepsilon}^{+}$), as well as the fact that the initial
Vlasov field $\bar{f}_{/}^{(\varepsilon)}$ (introduced in Definition
\ref{def:InitialDataFamily}) is compactly supported in phase space.

In order to establish that at least one of the relations \ref{eq:ExtensionConditionMuTilde}\textendash \ref{eq:ExtensionConditionNormVCons}
hold for $\mathcal{U}_{\varepsilon}^{+}$, we will assume, for the
sake of contradiction, that there exists a (possibly small)$\delta>0$
\begin{equation}
\limsup_{p\rightarrow\{u=u[\mathcal{U}_{\varepsilon}^{+}]\}}\frac{2\tilde{m}}{r}(p)<\eta_{0}-\delta,\label{eq:ExtensionConditionMuTilde-1}
\end{equation}
\begin{equation}
\text{ }\lim\sup_{u\rightarrow u[\mathcal{U}_{\varepsilon}^{+}]}\int_{u}^{u+\sqrt{-\frac{3}{\Lambda}}\pi}r\Big(\frac{T_{vv}[f_{\varepsilon}]}{\partial_{v}r}+\frac{T_{uv}[f_{\varepsilon}]}{-\partial_{u}r}\Big)(u,v)\,dv<\sigma_{\varepsilon}^{-1}-\delta\label{eq:ExtensionConditionNormUCons-1}
\end{equation}
and
\begin{equation}
\sup_{v\in(0,u[\mathcal{U}_{\varepsilon}^{+}]+\sqrt{-\frac{3}{\Lambda}}\pi)}\int_{\max\{0,v-\sqrt{-\frac{3}{\Lambda}}\pi\}}^{\min\{v,u[\mathcal{U}_{\varepsilon}^{+}]\}}r\Big(\frac{T_{uv}[f_{\varepsilon}]}{\partial_{v}r}+\frac{T_{uu}[f_{\varepsilon}]}{-\partial_{u}r}\Big)(u,v)\,du<\sigma_{\varepsilon}^{-1}-\delta.\label{eq:ExtensionConditionNormVCons-1}
\end{equation}
Then, we readily infer by continuity (using (\ref{eq:Ablabla}) and
the fact that $r^{2}T_{\mu\nu}$ and $\tilde{m}$ extend continuously
on $\mathcal{I}_{\varepsilon}$) that there exists some $0<u_{0}^{\prime}<\frac{\sigma_{\varepsilon}^{-2}}{\sqrt{-\Lambda}}-u[\mathcal{U}_{\varepsilon}^{+}]$
such that 
\begin{equation}
\sup_{\{0<u<u[\mathcal{U}_{\varepsilon}^{+}]+u_{0}^{\prime}\}}\frac{2\tilde{m}}{r}(p)<\eta_{0}-\frac{\delta}{2},\label{eq:ExtensionConditionMuTilde-1-1}
\end{equation}
\begin{equation}
\text{ }\sup_{\{0<u<u[\mathcal{U}_{\varepsilon}^{+}]+u_{0}^{\prime}\}}\int_{u}^{u+\sqrt{-\frac{3}{\Lambda}}\pi}r\Big(\frac{T_{vv}[f_{\varepsilon}]}{\partial_{v}r}+\frac{T_{uv}[f_{\varepsilon}]}{-\partial_{u}r}\Big)(u,v)\,dv<\sigma_{\varepsilon}^{-1}-\frac{\delta}{2}\label{eq:ExtensionConditionNormUCons-1-1}
\end{equation}
and
\begin{equation}
\sup_{v\in(0,u[\mathcal{U}_{\varepsilon}^{+}]+u_{0}^{\prime}+\sqrt{-\frac{3}{\Lambda}}\pi)}\int_{\max\{0,v-\sqrt{-\frac{3}{\Lambda}}\pi\}}^{\min\{v,u[\mathcal{U}_{\varepsilon}^{+}]+u_{0}^{\prime}\}}r\Big(\frac{T_{uv}[f_{\varepsilon}]}{\partial_{v}r}+\frac{T_{uu}[f_{\varepsilon}]}{-\partial_{u}r}\Big)(u,v)\,du<\sigma_{\varepsilon}^{-1}-\frac{\delta}{2}.\label{eq:ExtensionConditionNormVCons-1-1}
\end{equation}
Then, in view of the definition (\ref{eq:DefinitionUKappadelta})
and (\ref{eq:BootstrapDomain}) of $\mathcal{U}_{\sigma_{\varepsilon}^{-1},\eta_{0}}^{(\varepsilon)}$
and $\mathcal{U}_{\varepsilon}^{+}$, as well as assumption (\ref{eq:UpperBoundUExtensionCondition}),
the bounds (\ref{eq:ExtensionConditionMuTilde-1-1})\textendash (\ref{eq:ExtensionConditionNormVCons-1-1})
imply that 
\[
\{0<u<u[\mathcal{U}_{\varepsilon}^{+}]+u_{0}^{\prime}\}\subset u[\mathcal{U}_{\varepsilon}^{+}],
\]
which is a contradiction, in view of (\ref{eq:DefinitionBootstrapDomainWithU}).
Therefore, \ref{eq:ExtensionConditionMuTilde}\textendash \ref{eq:ExtensionConditionNormVCons}
hold for $\mathcal{U}_{\varepsilon}^{+}$. The proof of the analogous
relations for $\mathcal{T}_{\varepsilon}^{+}$ follows in exactly
the same way.
\end{proof}
For any $\varepsilon\in(0,1]$ and any $0\le i\le N_{\varepsilon}$,
we will define $f_{\varepsilon i}$ to be the solution of the Vlasov
equation (\ref{eq:VlasovFinal}) on $(\mathcal{U}_{max}^{(\varepsilon)};r,\Omega^{2})$
satisfying at $u=0$: 
\begin{equation}
f_{\varepsilon i}(0,v;p^{u},p^{v},l)=F_{i}^{(\varepsilon)}(v;\text{ }\partial_{v}r_{/}^{(\varepsilon)}p^{u},\text{ }l)\cdot\delta\Big(\Omega^{2}(0,v)p^{u}p^{v}-\frac{l^{2}}{r^{2}(0,v)}\Big),\label{eq:Fepsiloni}
\end{equation}
where $F_{i}^{(\varepsilon)}$ is given by (\ref{eq:InitialVlasovSeeds})
for $\varepsilon\in(0,\varepsilon_{1}]$, and $F_{i}^{(\varepsilon)}=F_{i}^{(\varepsilon_{1})}$
for $\varepsilon\in(\varepsilon_{1},1]$. Note that, as a consequence
of (\ref{eq:InitialVlasovTotal}), the total Vlasov field $f_{\varepsilon}$
is expressed as 
\begin{equation}
f_{\varepsilon}=\sum_{i=0}^{N_{\varepsilon}}a_{\varepsilon i}\cdot f_{\varepsilon i}.\label{eq:LinearCombinationVlasovFields}
\end{equation}

We will also define the functions $\bar{f}_{\varepsilon}$, $\bar{f}_{\varepsilon i}$
on the phase space over $\mathcal{U}_{max}^{(\varepsilon)}$, associated
to the Vlasov distributions $f_{\varepsilon}$, $f_{\varepsilon i}$,
respectively, as in Section \ref{subsec:VlasovEquations}, i.\,e.~by
the relation 
\begin{equation}
f_{\varepsilon i}(u,v;p^{u},p^{v},l)\doteq\bar{f}_{\varepsilon i}(u,v;p^{u},p^{v},l)\cdot\delta\Big(\Omega^{2}p^{u}p^{v}-\frac{l^{2}}{r^{2}}\Big).\label{eq:RelationFBarEpsilon}
\end{equation}
and similarly for $\bar{f}_{\varepsilon}$, $f_{\varepsilon}$ in
place of $\bar{f}_{\varepsilon i}$, $f_{\varepsilon i}$. Note that
the relation \ref{eq:RelationFBarEpsilon} uniquely determines $\bar{f}_{\varepsilon i}$
only on the shell $\big\{\Omega^{2}p^{u}p^{v}=\frac{l^{2}}{r^{2}}\big\}$;
see Section \ref{subsec:VlasovEquations}. As a consequence of the
Vlasov equation for (\ref{eq:VlasovFinal}), the functions $\bar{f}_{\varepsilon i}$
and $\bar{f}_{\varepsilon}$ are conserved along the integral curves
of (\ref{eq:NewNullGeodesicsSphericalSymmetry}); since $l$ is a
constant of motion for (\ref{eq:NewNullGeodesicsSphericalSymmetry}),
we can estimate using the explicit formula (\ref{eq:InitialVlasovSeeds})
for $F_{i}^{(\varepsilon)}$ and the bound (\ref{eq:ComparableInitialDataWithAdSEpsilonn})
for $\partial_{v}r_{/}^{(\varepsilon)}$:
\begin{align}
\sup_{(u,v)\in\mathcal{U}_{max}^{(\varepsilon)},\text{ }p^{u},p^{v}\in(0,+\infty)} & \int_{0}^{+\infty}\bar{f}_{\varepsilon i}(u,v;p^{u},p^{v},l)|_{\Omega^{2}p^{u}p^{v}=\frac{l^{2}}{r^{2}}}\,ldl=\label{eq:UpperBoundFbarEpsilon}\\
 & =\sup_{v\in(0,\sqrt{-\frac{3}{\Lambda}}\pi),\text{ }p^{u}\in(0,+\infty)}\int_{0}^{+\infty}F_{i}^{(\varepsilon)}(v;\text{ }\partial_{v}r_{/}^{(\varepsilon)}p^{u},\text{ }l)\,ldl\le\nonumber \\
 & \le16.\nonumber 
\end{align}

In view of the formula (\ref{eq:InitialVlasovSeeds}) for $F_{i}^{(\varepsilon)}$,
the bound (\ref{eq:ComparableInitialDataWithAdSEpsilonn}) for $\partial_{v}r_{/}^{(\varepsilon)}$
and the form (\ref{eq:InitialVlasovTotal}) of the initial Vlasov
distribution $\bar{f}_{/}^{(\varepsilon)}$, we infer that the total
renormalised Hawking mass $\tilde{m}|_{\mathcal{I}_{\varepsilon}}$
of $(\mathcal{U}_{max}^{(\varepsilon)};r,\Omega^{2},f_{\varepsilon})$
at $\mathcal{I}_{\varepsilon}$ (which is conserved, in view of the
reflecting boundary condition on $\mathcal{I}$; see (\ref{eq:ConservationHawkingMassAtInfinity}))
satisfies 
\begin{equation}
\sqrt{-\Lambda}\tilde{m}|_{\mathcal{I}_{\varepsilon}}\sim\sum_{i=0}^{N_{\varepsilon}}a_{\varepsilon i}\varepsilon^{(i)},\label{eq:TotalHawkingMassE}
\end{equation}
where the constants implicit in the $\sim$ notation in (\ref{eq:TotalHawkingMassE})
are independent of $\varepsilon$ and $\{a_{\varepsilon i}\}_{i=0}^{N_{\varepsilon}}$.
In particular, in view of (\ref{eq:RecursiveEi}) and the assumption
$a_{\varepsilon i}\in[0,\sigma_{\varepsilon})$: 
\begin{equation}
\sqrt{-\Lambda}\tilde{m}|_{\mathcal{I}_{\varepsilon}}\le\sigma_{\varepsilon}\varepsilon.\label{eq:TrivialUpperBoundTotalMass}
\end{equation}

As an immediate consequence of the definition (\ref{eq:Definitionudelta})
and (\ref{eq:DefinitionUSharp}) of $u_{\eta_{0};\varepsilon}^{+}$
and $u_{\sigma_{\varepsilon}^{-1};\varepsilon}^{\sharp}$, respectively,
and the definition (\ref{eq:BootstrapDomain}) of $\mathcal{U}_{\varepsilon}^{+}$,
we can bound for any $\varepsilon\in(0,\varepsilon_{1}]$:
\begin{equation}
\sup_{\mathcal{U}_{\varepsilon}^{+}}\frac{2\tilde{m}}{r}\le\eta_{0},\label{eq:UpperBoundMuBootstrapDomain}
\end{equation}
\begin{align}
\text{ }\sup_{U\ge0}\int_{\{u=U\}\cap\mathcal{U}_{\varepsilon}^{+}}r\Big(\frac{T_{vv}[f_{\varepsilon}]}{\partial_{v}r}+\frac{T_{uv}[f_{\varepsilon}]}{-\partial_{u}r}\Big)(U,v)\,dv+\label{eq:UpperBoundNormBootstrapDomain}\\
\sup_{V\ge0}\int_{\{v=V\}\cap\mathcal{U}_{\varepsilon}^{+}}r\Big(\frac{T_{uv}[f_{\varepsilon}]}{\partial_{v}r}+\frac{T_{uu}[f_{\varepsilon}]}{-\partial_{u}r}\Big)(u,V)\,du & \le2\sigma_{\varepsilon}^{-1}\nonumber 
\end{align}
and 
\begin{equation}
\sup_{\mathcal{U}_{\varepsilon}^{+}}(u+v)\le\frac{2}{\sqrt{-\Lambda}}\sigma_{\varepsilon}^{-2}+\sqrt{-\frac{3}{\Lambda}}\pi.\label{eq:UpperBoundU+VBootstrapDomain}
\end{equation}
In particular, integrating the constraint equations (\ref{eq:EquationROutside})
and (\ref{eq:EquationROutside}) along lines of the form $v=const$
and $u=const$, respectively, using also the boundary conditions (\ref{eq:BoundaryConditionDrAxis})
and (\ref{eq:BoundaryConditionDrInfinity}) for $\partial_{v}r$,
$\partial_{u}r$ on $\mathcal{I}_{\varepsilon}$ and $\gamma_{\mathcal{Z}_{\varepsilon}}$,
we can estimate for any $\varepsilon\in(0,\varepsilon_{1}]$
\begin{align}
\sup_{\mathcal{U}_{\varepsilon}^{+}}\Bigg(\Big|\log\Big(\frac{\partial_{v}r}{1-\frac{2m}{r}}\Big)\Big|+\Big|\log\Big(\frac{-\partial_{u}r}{1-\frac{2m}{r}}\Big)\Big|\Bigg)\le & \big(\sup_{\mathcal{U}_{\varepsilon}^{+}}\sqrt{-\Lambda}(u+v)\big)\cdot\Bigg(\sup_{\bar{v}\ge0}\int_{\{v=\bar{v}\}\cap\mathcal{U}_{\varepsilon}^{+}}r\frac{T_{uu}[f_{\varepsilon}]}{-\partial_{u}r}(u,\bar{v})\,du+\label{eq:EstimateGeometryLikeC0}\\
 & \hphantom{\big(\sup_{\mathcal{U}_{\varepsilon}^{+}}\sqrt{-\Lambda}(}+\sup_{\bar{u}\ge0}\int_{\{u=\bar{u}\}\cap\mathcal{U}_{\varepsilon}^{+}}r\frac{T_{vv}[f_{\varepsilon}]}{\partial_{v}r}(\bar{u},v)\,dv\Bigg)+\nonumber \\
 & \text{ }+\sup_{v\in[0,\sqrt{-\frac{3}{\Lambda}}\pi)}\Big|\log\Big(\frac{\partial_{v}r_{/}^{(\varepsilon)}}{1-\frac{1}{3}\Lambda(r_{/}^{(\varepsilon)})^{2}}\Big)\Big|\le\nonumber \\
\le & 5\sigma_{\varepsilon}^{-3}\nonumber 
\end{align}
 where, in passing from the second to the third line in (\ref{eq:EstimateGeometryLikeC0}),
we made use of (\ref{eq:UpperBoundNormBootstrapDomain}), (\ref{eq:UpperBoundU+VBootstrapDomain})
and Lemma \ref{lem:FirstSmallnessBoundInitialData}.

Similarly, the (\ref{eq:LargerBootstrapDomain}) of $\mathcal{T}_{\varepsilon}^{+}$
yields for any $\varepsilon\in(0,\varepsilon_{1}]$:
\begin{equation}
\sup_{\mathcal{T}_{\varepsilon}^{+}}\frac{2\tilde{m}}{r}\le\eta_{0},\label{eq:UpperBoundMuLargerDomain}
\end{equation}
\begin{align}
\text{ }\sup_{U\ge0}\int_{\{u=U\}\cap\mathcal{T}_{\varepsilon}^{+}}r\Big(\frac{T_{vv}[f_{\varepsilon}]}{\partial_{v}r}+\frac{T_{uv}[f_{\varepsilon}]}{-\partial_{u}r}\Big)(U,v)\,dv+\label{eq:UpperBoundNormLargerDomain}\\
\sup_{V\ge0}\int_{\{v=V\}\cap\mathcal{T}_{\varepsilon}^{+}}r\Big(\frac{T_{uv}[f_{\varepsilon}]}{\partial_{v}r}+\frac{T_{uu}[f_{\varepsilon}]}{-\partial_{u}r}\Big)(u,V)\,du & \le2\delta_{\varepsilon}^{-1},\nonumber 
\end{align}
\begin{equation}
\sup_{\mathcal{T}_{\varepsilon}^{+}}(u+v)\le\frac{2}{\sqrt{-\Lambda}}\sigma_{\varepsilon}^{-2}+\sqrt{-\frac{3}{\Lambda}}\pi\label{eq:UpperBoundU+VLargerDomain}
\end{equation}
and, in analogy to (\ref{eq:EstimateGeometryLikeC0}):
\begin{equation}
\sup_{\mathcal{T}_{\varepsilon}^{+}}\Bigg(\Big|\log\Big(\frac{\partial_{v}r}{1-\frac{2m}{r}}\Big)\Big|+\Big|\log\Big(\frac{-\partial_{u}r}{1-\frac{2m}{r}}\Big)\Big|\Bigg)\le5\sigma_{\varepsilon}^{-2}\delta_{\varepsilon}^{-1}\le\delta_{\varepsilon}^{-2}.\label{eq:EstimateGeometryLikeC0LargerDomain}
\end{equation}

\subsection{\label{subsec:Some-basic-geometric-constructions} Notational conventions
for the beams and their intersection regions}

In this Section we will introduce some shorthand notation for regions
of the $(u,v)$-plane which, when intersected with the domain $\mathcal{U}_{max}^{(\varepsilon)}$
of the maximal future development of $(r_{/}^{(\varepsilon)},(\Omega_{/}^{(\varepsilon)})^{2},\bar{f}_{/}^{(\varepsilon)};\sqrt{-\frac{3}{\Lambda}}\pi)$,
will contain the support of the Vlasov beams emanating from the initial
data on $u=0$. 

\begin{figure}[h] 
\centering 
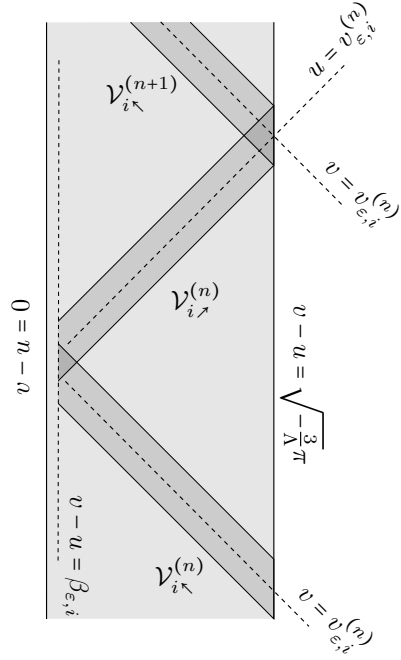 
\caption{Schematic depiction of the domains $\mathcal{V}_{i\nwarrow}^{(n)}$ and $\mathcal{V}_{i\nearrow}^{(n)}$ for some $\varepsilon>0$ and some $0\le i \le N_{\varepsilon}$. For the definition of $v_{\varepsilon,i}^{(n)}$,  $h_{\varepsilon,i}$ and $\alpha_{\varepsilon,i}$, see the relation (\ref{eq:ShorthandNotationCornerPoints}).}
\end{figure}

Using the shorthand notation 
\begin{align}
v_{\varepsilon,i}^{(n)} & \doteq v_{\varepsilon,i}+n\sqrt{-\frac{3}{\Lambda}}\pi,\label{eq:ShorthandNotationCornerPoints}\\
h_{\varepsilon,i} & \doteq e^{\sigma_{\varepsilon}^{-6}}\frac{\varepsilon^{(i)}}{\sqrt{-\Lambda}},\nonumber \\
\beta_{\varepsilon,i} & \doteq\exp\big(-\exp(\sigma_{\varepsilon}^{-4})\big)\frac{\varepsilon^{(i)}}{\sqrt{-\Lambda}},\nonumber 
\end{align}
we will define the following ``narrow'' sets for any $n\in\mathbb{N}$,
any $\varepsilon\in(0,\varepsilon_{1}]$ and any integer $0\le i\le N_{\varepsilon}$: 

\begin{equation}
\mathcal{V}_{i}^{(n)}\doteq\Big(\mathcal{V}_{i\nwarrow}^{(n)}\cup\mathcal{V}_{i\nearrow}^{(n)}\Big),\label{eq:BeamDomain}
\end{equation}
where 
\begin{align}
\mathcal{V}_{i\nwarrow}^{(n)} & \doteq\Big\{\Big|v-v_{\varepsilon,i}^{(n)}\Big|\le h_{\varepsilon,i}\Big\}\cap\Big\{\beta_{\varepsilon,i}\le v-u\le\sqrt{-\frac{3}{\Lambda}}\pi\Big\},\label{eq:IngoingOutgoingDomains}\\
\mathcal{V}_{i\nearrow}^{(n)} & \doteq\Big\{\Big|u-v_{\varepsilon,i}^{(n)}\Big|\le h_{\varepsilon,i}\Big\}\cap\Big\{\beta_{\varepsilon,i}\le v-u\le\sqrt{-\frac{3}{\Lambda}}\pi\Big\}.\nonumber 
\end{align}
We will also set 
\[
\mathcal{V}_{i}\doteq\cup_{n\in\mathbb{N}}\mathcal{V}_{i}^{(n)}.
\]
\begin{rem*}
As a consequence of the formula (\ref{eq:InitialVlasovSeeds}) for
$F_{i}^{(\varepsilon)}$, the definition (\ref{eq:Fepsiloni}) of
$f_{\varepsilon i}$ in terms of $F_{i}^{(\varepsilon)}$, the properties
of the geodesic flow on AdS spacetime (see the relation A.10 in the
Appendix of \cite{MoschidisVlasovWellPosedness}) and the Cauchy stability
statement \ref{prop:CauchyStabilityAdS} for (\ref{eq:RequationFinal})\textendash (\ref{NullShellFinal}),
we readily infer that there exists some $C_{\varepsilon}>0$ with
$C_{\varepsilon}\xrightarrow{\varepsilon\rightarrow0}+\infty$ such
that the Vlasov beam corresponding to $f_{\varepsilon i}$ is supported
in $\mathcal{V}_{i}$ in the retarded time interval $0\le u\le C_{\varepsilon}$,
i.\,e.
\begin{equation}
supp(f_{\varepsilon i})\cap\Big\{0\le u\le C_{\varepsilon}\Big\}\subset\Big\{(u,v)\in\bigcup_{n\in\mathbb{N}}\mathcal{V}_{i}^{(n)}\cap\mathcal{U}_{max}^{(\varepsilon)}\Big\}.\label{eq:VlasovBeamRestrictionQualitative}
\end{equation}
In Section \ref{subsec:Control-of-the-Vlasov-beams}, we will establish
a quantitative refinement of (\ref{eq:VlasovBeamRestrictionQualitative})
(see Lemma \ref{lem:QuantitativeBeamEstimate}).
\end{rem*}
We will also define the intersection regions $\mathcal{R}_{i;j}^{(n)}$
for any $n\in\mathbb{N}$ and any integers $0\le i\neq j\le N_{\varepsilon}$
as follows: 
\begin{equation}
\mathcal{R}_{i;j}^{(n)}\doteq\begin{cases}
\mathcal{V}_{i\nwarrow}^{(n)}\cap\mathcal{V}_{j\nearrow}^{(n)}, & \text{ if }i>j,\\
\mathcal{V}_{i\nwarrow}^{(n+1)}\cap\mathcal{V}_{j\nearrow}^{(n)}, & \text{ if }i<j.
\end{cases}\label{eq:IntersectionDomain}
\end{equation}
The self-intersection regions $\mathcal{R}_{i;\gamma_{\mathcal{Z}}}^{(n)}$
and $\mathcal{R}_{i;\mathcal{I}}^{(n)}$ will be defined for any $n\in\mathbb{N}$
and any $0\le i\le N_{\varepsilon}$ as 
\begin{equation}
\mathcal{R}_{i;\gamma_{\mathcal{Z}}}^{(n)}\doteq\mathcal{V}_{i\nwarrow}^{(n)}\cap\mathcal{V}_{i\nearrow}^{(n)}\Big\{ v-u\ge\beta_{\varepsilon,i}\Big\}\label{eq:IntersectionDomainNearAxis}
\end{equation}
 and 
\begin{equation}
\mathcal{R}_{i;\mathcal{I}}^{(n)}\doteq\mathcal{V}_{i\nwarrow}^{(n+1)}\cap\mathcal{V}_{i\nearrow}^{(n)}\Big\{ v-u\le\sqrt{-\frac{3}{\Lambda}}\pi\Big\}.\label{eq:IntersectionDomainNearInfinity}
\end{equation}

\begin{figure}[h] 
\centering 
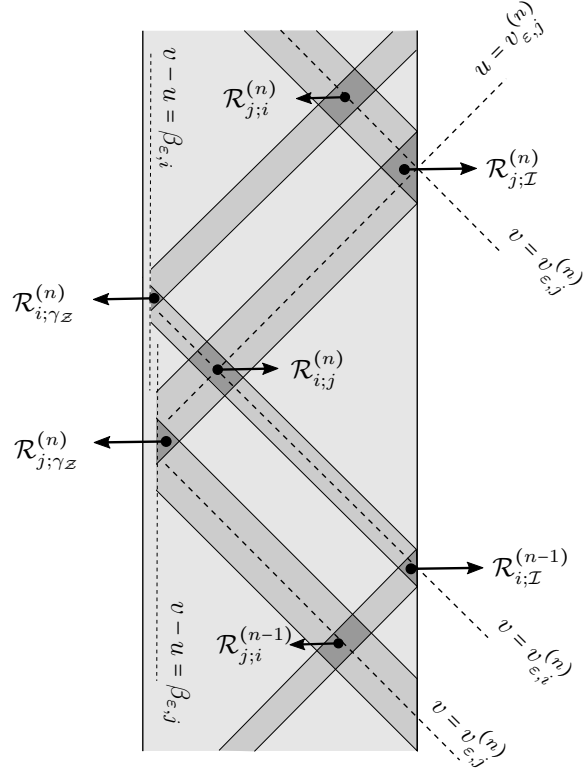 
\caption{Schematic depiction of the intersection domains $\mathcal{R}_{i;j}^{(n)}$, $\mathcal{R}_{i;\gamma_{\mathcal{Z}}}^{(n)}$ and $\mathcal{R}_{i;\mathcal{I}}^{(n)}$ for $i>j$.}
\end{figure}
\begin{rem*}
Note that the domains $\mathcal{R}_{i;j}^{(n)}$, $\mathcal{R}_{i;\gamma_{\mathcal{Z}}}^{(n)}$
and $\mathcal{R}_{i;\mathcal{I}}^{(n)}$ can be also expressed as
\begin{equation}
\mathcal{R}_{i;j}^{(n)}=\begin{cases}
[v_{\varepsilon,j}^{(n)}-h_{\varepsilon,j},v_{\varepsilon,j}^{(n)}+h_{\varepsilon,j}]\times[v_{\varepsilon,i}^{(n)}-h_{\varepsilon,i},v_{\varepsilon,i}^{(n)}+h_{\varepsilon,i}], & \text{ if }i>j,\\{}
[v_{\varepsilon,j}^{(n)}-h_{\varepsilon,j},v_{\varepsilon,j}^{(n)}+h_{\varepsilon,j}]\times[v_{\varepsilon,i}^{(n+1)}-h_{\varepsilon,i},v_{\varepsilon,i}^{(n+1)}+h_{\varepsilon,i}], & \text{ if }i<j,
\end{cases}\label{eq:ExplicitIntersectionDomain}
\end{equation}
\begin{equation}
\mathcal{R}_{i;\gamma_{\mathcal{Z}}}^{(n)}=[v_{\varepsilon,i}^{(n)}-h_{\varepsilon,i},v_{\varepsilon,i}^{(n)}+h_{\varepsilon,i}]\times[v_{\varepsilon,i}^{(n)}-h_{\varepsilon,i},v_{\varepsilon,i}^{(n)}+h_{\varepsilon,i}]\cap\big\{ v-u\ge\beta_{\varepsilon,i}\big\}\label{eq:ExplicitSpecialDomainAxis}
\end{equation}
and 
\begin{equation}
\mathcal{R}_{i;\mathcal{I}}^{(n)}=[v_{\varepsilon,i}^{(n)}-h_{\varepsilon,i},v_{\varepsilon,i}^{(n)}+h_{\varepsilon,i}]\times[v_{\varepsilon,i}^{(n+1)}-h_{\varepsilon,i},v_{\varepsilon,i}^{(n+1)}+h_{\varepsilon,i}]\cap\big\{ v-u\le\sqrt{-\frac{3}{\Lambda}}\pi\big\}.\label{eq:ExplicitSpecialDomainInfinity}
\end{equation}
\end{rem*}
For any point $p=(\bar{u},\bar{v})\in\mathcal{U}_{max}^{(\varepsilon)}$,
we will define the pair of crooked lines $\zeta_{\nwarrow}[p]$ and
$\zeta_{\nearrow}[p]$ as follows: 
\begin{equation}
\zeta_{\nwarrow}[p]\doteq\bigcup_{k\in\mathbb{N}}\Big(\big\{ v=\bar{v}-k\sqrt{-\frac{3}{\Lambda}}\pi\big\}\cup\big\{ u=\bar{v}-(k+1)\sqrt{-\frac{3}{\Lambda}}\pi\big\}\Big)\cap\mathcal{U}_{max}^{(\varepsilon)}\label{eq:CrookedLineIngoing}
\end{equation}
and 
\begin{equation}
\zeta_{\nearrow}[p]\doteq\bigcup_{k\in\mathbb{N}}\Big(\big\{ u=\bar{u}-k\sqrt{-\frac{3}{\Lambda}}\pi\big\}\cup\big\{ v=\bar{u}-k\sqrt{-\frac{3}{\Lambda}}\pi\big\}\Big)\cap\mathcal{U}_{max}^{(\varepsilon)}.\label{eq:CrookedLineOutgoing}
\end{equation}
We will also define the following functions on $\mathcal{U}_{max}^{(\varepsilon)}$:
\begin{equation}
dist_{\nwarrow}[p]\doteq\inf_{\zeta_{\nwarrow}[p]\cap\big(\cup_{n\in\mathbb{N}}\cup_{i=0}^{N_{\varepsilon}}\mathcal{V}_{i}^{(n)}\big)}(v-u)\label{eq:IngoingDistanceFunction}
\end{equation}
and 
\begin{equation}
dist_{\nearrow}[p]\doteq\inf_{\zeta_{\nearrow}[p]\cap\big(\cup_{n\in\mathbb{N}}\cup_{i=0}^{N_{\varepsilon}}\mathcal{V}_{i}^{(n)}\big)}(v-u).\label{eq:OutgoingDistanceFunction}
\end{equation}

\begin{figure}[h] 
\centering 
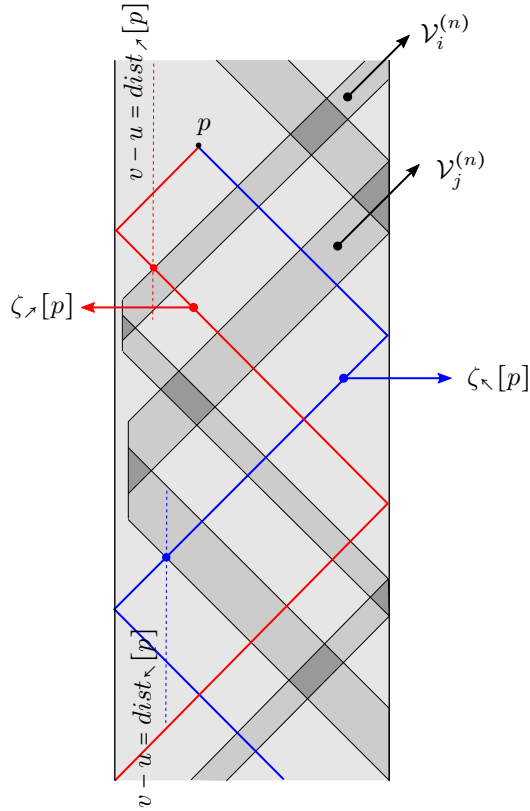 
\caption{Schematic depiction of the crooked lines $\zeta_{\nwarrow}[p]$ and $\zeta_{\nearrow}[p]$ emanating from a point $p\in \mathcal{U}^{(\epsilon )}_{max}$. The distance functions $dist_{\nwarrow}[p]$ and $dist_{\nearrow}[p]$ are defined as the infimum of the distance from the axis (measured by $v-u$) of the points on the intersection of the union of the beams $\mathcal{V}_{i}^{(n)}$ with $\zeta_{\nwarrow}[p]$ and $\zeta_{\nearrow}[p]$, respectively.}
\end{figure}
\begin{rem*}
Notice that, if $p\in\cup_{n\in\mathbb{N}}\mathcal{V}_{i\nwarrow}^{(n)}$,
then 
\begin{equation}
dist_{\nwarrow}[p]=\beta_{\varepsilon,i}.\label{eq:SimpleEqualityDistance1}
\end{equation}
Similarly, if $p\in\cup_{n\in\mathbb{N}}\mathcal{V}_{i\nearrow}^{(n)}$:
\begin{equation}
dist_{\nearrow}[p]=\beta_{\varepsilon,i}.\label{SimpleEqualityDistance2}
\end{equation}
\end{rem*}
For any $n\in\mathbb{N}$, any $\varepsilon\in(0,\varepsilon_{1}]$
and any integers $0\le i,j\le N_{\varepsilon}$, $i\neq j$, such
that $\mathcal{R}_{i;j}^{(n)}\subset\mathcal{U}_{max}^{(\varepsilon)}$,
we will introduce the following quantities related to the energy content
of $\mathcal{V}_{i\nwarrow}^{(n)}$, $\mathcal{V}_{j\nearrow}^{(n)}$
before and after their intersection with the region $\mathcal{R}_{i;j}^{(n)}$:
\begin{align}
\mathcal{E}_{\nwarrow}^{(-)}[n;i,j] & \doteq\begin{cases}
\tilde{m}_{\varepsilon}\Big(v_{\varepsilon,j}^{(n)}-h_{\varepsilon,j},\text{ }v_{\varepsilon,i}^{(n)}+h_{\varepsilon,i}\Big)-\tilde{m}_{\varepsilon}\Big(v_{\varepsilon,j}^{(n)}-h_{\varepsilon,j},\text{ }v_{\varepsilon,i}^{(n)}-h_{\varepsilon,i}\Big), & \text{ if }i>j,\\
\tilde{m}_{\varepsilon}\Big(v_{\varepsilon,j}^{(n)}-h_{\varepsilon,j},\text{ }v_{\varepsilon,i}^{(n+1)}+h_{\varepsilon,i}\Big)-\tilde{m}_{\varepsilon}\Big(v_{\varepsilon,j}^{(n)}-h_{\varepsilon,j},\text{ }v_{\varepsilon,i}^{(n+1)}-h_{\varepsilon,i}\Big), & \text{ if }i<j,
\end{cases}\label{eq:IngoingEnergyBefore}\\
\nonumber \\
\mathcal{E}_{\nwarrow}^{(+)}[n;i,j] & \doteq\begin{cases}
\tilde{m}_{\varepsilon}\Big(v_{\varepsilon,j}^{(n)}+h_{\varepsilon,j},\text{ }v_{\varepsilon,i}^{(n)}+h_{\varepsilon,i}\Big)-\tilde{m}_{\varepsilon}\Big(v_{\varepsilon,j}^{(n)}+h_{\varepsilon,j},\text{ }v_{\varepsilon,i}^{(n)}-h_{\varepsilon,i}\Big), & \text{ if }i>j,\\
\tilde{m}_{\varepsilon}\Big(v_{\varepsilon,j}^{(n)}+h_{\varepsilon,j},\text{ }v_{\varepsilon,i}^{(n+1)}+h_{\varepsilon,i}\Big)-\tilde{m}_{\varepsilon}\Big(v_{\varepsilon,j}^{(n)}+h_{\varepsilon,j},\text{ }v_{\varepsilon,i}^{(n+1)}-h_{\varepsilon,i}\Big), & \text{ if }i<j
\end{cases}\label{eq:IngoingEnergyAfter}
\end{align}
and 
\begin{align}
\mathcal{E}_{\nearrow}^{(-)}[n;i,j] & \doteq\begin{cases}
\tilde{m}_{\varepsilon}\Big(v_{\varepsilon,j}^{(n)}-h_{\varepsilon,j},\text{ }v_{\varepsilon,i}^{(n)}-h_{\varepsilon,i}\Big)-\tilde{m}_{\varepsilon}\Big(v_{\varepsilon,j}^{(n)}+h_{\varepsilon,j},\text{ }v_{\varepsilon,i}^{(n)}-h_{\varepsilon,i}\Big), & \text{ if }i>j,\\
\tilde{m}_{\varepsilon}\Big(v_{\varepsilon,j}^{(n)}-h_{\varepsilon,j},\text{ }v_{\varepsilon,i}^{(n+1)}-h_{\varepsilon,i}\Big)-\tilde{m}_{\varepsilon}\Big(v_{\varepsilon,j}^{(n)}+h_{\varepsilon,j},\text{ }v_{\varepsilon,i}^{(n+1)}-h_{\varepsilon,i}\Big), & \text{ if }i<j,
\end{cases}\label{eq:OutgoingEnergyBefore}\\
\nonumber \\
\mathcal{E}_{\nearrow}^{(+)}[n;i,j] & \doteq\begin{cases}
\tilde{m}_{\varepsilon}\Big(v_{\varepsilon,j}^{(n)}-h_{\varepsilon,j},\text{ }v_{\varepsilon,i}^{(n)}+h_{\varepsilon,i}\Big)-\tilde{m}_{\varepsilon}\Big(v_{\varepsilon,j}^{(n)}+h_{\varepsilon,j},\text{ }v_{\varepsilon,i}^{(n)}+h_{\varepsilon,i}\Big), & \text{ if }i>j,\\
\tilde{m}_{\varepsilon}\Big(v_{\varepsilon,j}^{(n)}-h_{\varepsilon,j},\text{ }v_{\varepsilon,i}^{(n+1)}+h_{\varepsilon,i}\Big)-\tilde{m}_{\varepsilon}\Big(v_{\varepsilon,j}^{(n)}+h_{\varepsilon,j},\text{ }v_{\varepsilon,i}^{(n+1)}+h_{\varepsilon,i}\Big), & \text{ if }i<j,
\end{cases}\label{eq:OutgoingEnergyAfter}
\end{align}
For any $n\in\mathbb{N}$, any $\varepsilon\in(0,\varepsilon_{1}]$
and any integer $0\le i\le N_{\varepsilon}$ such that $\mathcal{R}_{i;\gamma_{\mathcal{Z}}}^{(n)}\subset\mathcal{U}_{max}^{(\varepsilon)}$
and $\mathcal{R}_{i;\mathcal{I}}^{(n)}\subset\mathcal{U}_{max}^{(\varepsilon)}$,
we will define, respectively,
\begin{equation}
\mathcal{E}_{\gamma_{\mathcal{Z}}}[n;i]\doteq\tilde{m}\Big(v_{\varepsilon,i}^{(n)}-h_{\varepsilon,i},\text{ }v_{\varepsilon,i}^{(n)}+h_{\varepsilon,i}\Big),\label{eq:IngoingEnergyAxis}
\end{equation}
and 
\begin{equation}
\mathcal{E}_{\mathcal{I}}[n;i]\doteq\tilde{m}|_{\mathcal{I}_{\varepsilon}}-\tilde{m}\Big(v_{\varepsilon,i}^{(n)}+h_{\varepsilon,i},\text{ }v_{\varepsilon,i}^{(n+1)}-h_{\varepsilon,i}\Big).\label{eq:IngoingEnergyInfinity}
\end{equation}
\begin{rem*}
When $i>j$, the quantity $\mathcal{E}_{\nwarrow}^{(-)}[n;i,j]$ measures
the energy content of the ingoing beam $\mathcal{V}_{i\nwarrow}^{(n)}$
right \emph{before} entering the region $\mathcal{R}_{i;j}^{(n)}$,
while $\mathcal{E}_{\nwarrow}^{(+)}[n;i,j]$ measures the energy content
of $\mathcal{V}_{i\nwarrow}^{(n)}$ right \emph{after }leaving $\mathcal{R}_{i;j}^{(n)}$
(when $i<j$, the same holds after replacing $\mathcal{V}_{i\nwarrow}^{(n)}$
with $\mathcal{V}_{i\nwarrow}^{(n+1)}$). Similarly, $\mathcal{E}_{\nearrow}^{(-)}[n;i,j]$
and $\mathcal{E}_{\nearrow}^{(+)}[n;i,j]$ measure the energy content
of the outgoing beam $\mathcal{V}_{j\nearrow}^{(n)}$ right before
and right after, respectively, $\mathcal{R}_{j;i}^{(n)}$. Finally,
$\mathcal{E}_{\gamma_{\mathcal{Z}}}[n;i]$ measures the energy content
of $\mathcal{V}_{i}^{(n)}$ measured at the region $\mathcal{R}_{i;\gamma_{\mathcal{Z}}}^{(n)}$,
while $\mathcal{E}_{\mathcal{I}}[n;i]$ measures the energy content
of $\mathcal{V}_{i}^{(n)}$ at the region $\mathcal{R}_{i;\mathcal{I}}^{(n)}$.
For a schematic depiction of the definition of the above quantities,
see Figure \ref{fig:FigureEnergy}.
\end{rem*}
\begin{figure}[h] 
\centering 
\begingroup%
  \makeatletter%
  \providecommand\color[2][]{%
    \errmessage{(Inkscape) Color is used for the text in Inkscape, but the package 'color.sty' is not loaded}%
    \renewcommand\color[2][]{}%
  }%
  \providecommand\transparent[1]{%
    \errmessage{(Inkscape) Transparency is used (non-zero) for the text in Inkscape, but the package 'transparent.sty' is not loaded}%
    \renewcommand\transparent[1]{}%
  }%
  \providecommand\rotatebox[2]{#2}%
  \newcommand*\fsize{\dimexpr\f@size pt\relax}%
  \newcommand*\lineheight[1]{\fontsize{\fsize}{#1\fsize}\selectfont}%
  \ifx\svgwidth\undefined%
    \setlength{\unitlength}{225bp}%
    \ifx\svgscale\undefined%
      \relax%
    \else%
      \setlength{\unitlength}{\unitlength * \real{\svgscale}}%
    \fi%
  \else%
    \setlength{\unitlength}{\svgwidth}%
  \fi%
  \global\let\svgwidth\undefined%
  \global\let\svgscale\undefined%
  \makeatother%
  \begin{picture}(1,1)%
    \lineheight{1}%
    \setlength\tabcolsep{0pt}%
    \put(0,0){\includegraphics[width=\unitlength,page=1]{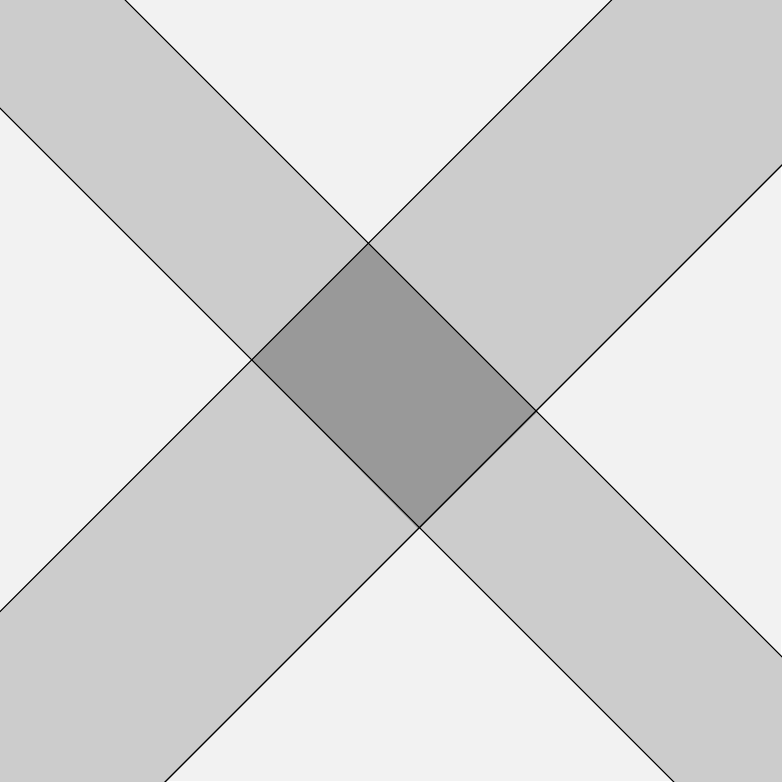}}%
    \put(0.4521307,0.51765848){\color[rgb]{0,0,0}\makebox(0,0)[lt]{\lineheight{1.25}\smash{\begin{tabular}[t]{l}$\mathcal{R}_{i;j}^{(n)}$\end{tabular}}}}%
    \put(0,0){\includegraphics[width=\unitlength,page=2]{Energy_Definition.pdf}}%
    \put(0.6559406,0.2376238){\color[rgb]{0,0,0}\makebox(0,0)[lt]{\lineheight{1.25}\smash{\begin{tabular}[t]{l}$\mathcal{E}_{\nwarrow}^{(-)}[n;i,j]$\end{tabular}}}}%
    \put(0.14727723,0.75495048){\color[rgb]{0,0,0}\makebox(0,0)[lt]{\lineheight{1.25}\smash{\begin{tabular}[t]{l}$\mathcal{E}_{\nwarrow}^{(+)}[n;i,j]$\end{tabular}}}}%
    \put(0.15717824,0.23886142){\color[rgb]{0,0,0}\makebox(0,0)[lt]{\lineheight{1.25}\smash{\begin{tabular}[t]{l}$\mathcal{E}_{\nearrow}^{(-)}[n;i,j]$\end{tabular}}}}%
    \put(0.65099007,0.75247524){\color[rgb]{0,0,0}\makebox(0,0)[lt]{\lineheight{1.25}\smash{\begin{tabular}[t]{l}$\mathcal{E}_{\nearrow}^{(+)}[n;i,j]$\end{tabular}}}}%
  \end{picture}%
\endgroup%
 
\caption{The quantities $\mathcal{E}_{\nwarrow}^{(\pm)}[n;i,j]$ and $\mathcal{E}_{\nearrow}^{(\pm)}[n;i,j]$ measure the energy content of the beams $\mathcal{V}_{i\nwarrow}^{(n)}$ (with $n+1$ in place of $n$ when $i<j$) and $\mathcal{V}_{j\nearrow}^{(n)}$ right before and right after intersecting the region $\mathcal{R}_{i;j}^{(n)}$.\label{fig:FigureEnergy}}
\end{figure}

For any $n\in\mathbb{N}$, any $\varepsilon\in(0,\varepsilon_{1}]$
and any integers $0\le i,j\le N_{\varepsilon}$ such that $i>0$ or
$j>0$, we will introduce the following quantities measuring the separation
of two successive beams of matter (defined when the corresponding
regions of integration lie in the domain $\mathcal{U}_{max}^{(\varepsilon)}\cap\big\{\frac{2m}{r}<1\big\}$):
\begin{align}
\mathfrak{D}r_{\nwarrow}^{(\pm)}[n;i,j] & \doteq\begin{cases}
\int_{\text{ }v_{\varepsilon,i-1}^{(n)}+(\rho_{\varepsilon}^{-\frac{7}{8}}+1)h_{\varepsilon,i-1}}^{\text{ }v_{\varepsilon,i}^{(n)}-(\rho_{\varepsilon}^{-\frac{7}{8}}+1)h_{\varepsilon,i-1}}\frac{\partial_{v}r}{1-\frac{2m}{r}}(v_{\varepsilon,j}^{(n)}\pm h_{\varepsilon,j},v)\,dv, & \text{ if }i>j,\\
\int_{\text{ }v_{\varepsilon,i-1}^{(n+1)}+(\rho_{\varepsilon}^{-\frac{7}{8}}+1)h_{\varepsilon,i-1}}^{\text{ }v_{\varepsilon,i}^{(n+1)}-(\rho_{\varepsilon}^{-\frac{7}{8}}+1)h_{\varepsilon,i-1}}\frac{\partial_{v}r}{1-\frac{2m}{r}}(v_{\varepsilon,j}^{(n)}\pm h_{\varepsilon,j},v)\,dv, & \text{ if }i\le j,
\end{cases}\label{eq:IngoingRSeparationBeforeAfter}
\end{align}
(well-defined when $i>0$) and 
\begin{align}
\mathfrak{D}r_{\nearrow}^{(\pm)}[n;i,j] & \doteq\begin{cases}
\int_{\text{ }v_{\varepsilon,j-1}^{(n)}+(\rho_{\varepsilon}^{-\frac{7}{8}}+1)h_{\varepsilon,j-1}}^{\text{ }v_{\varepsilon,j}^{(n)}-(\rho_{\varepsilon}^{-\frac{7}{8}}+1)h_{\varepsilon,j-1}}\frac{-\partial_{u}r}{1-\frac{2m}{r}}(u,v_{\varepsilon,i}^{(n)}\pm h_{\varepsilon,i})\,du, & \text{ if }i\ge j,\\
\int_{\text{ }v_{\varepsilon,j-1}^{(n)}+(\rho_{\varepsilon}^{-\frac{7}{8}}+1)h_{\varepsilon,j-1}}^{\text{ }v_{\varepsilon,j}^{(n)}-(\rho_{\varepsilon}^{-\frac{7}{8}}+1)h_{\varepsilon,j-1}}\frac{-\partial_{u}r}{1-\frac{2m}{r}}(u,v_{\varepsilon,i}^{(n+1)}\pm h_{\varepsilon,i})\,du, & \text{ if }i<j.
\end{cases}\label{eq:OutgoingRSeparationBeforeAfter}
\end{align}
(well-defined when $j>0$). 
\begin{rem*}
Notice that, when $\frac{2\tilde{m}}{r}\ll1$ and $\frac{\partial_{v}r}{1-\frac{1}{3}\Lambda r^{2}},\frac{\partial_{u}r}{1-\frac{1}{3}\Lambda r^{2}}\ll\rho_{\varepsilon}^{-\delta}$,
in the case when $i=j$ we have 
\begin{equation}
\mathfrak{D}r_{\nwarrow}^{(\pm)}[n;i,i]\sim\max_{\mathcal{R}_{i-1;i}^{(n)}}\frac{1}{-\frac{1}{3}\Lambda r}
\end{equation}
and 
\begin{equation}
\mathfrak{D}r_{\nearrow}^{(\pm)}[n;i,j]\sim\min_{\mathcal{R}_{i;i-1}^{(n)}}r.
\end{equation}
\end{rem*}
\begin{figure}[h] 
\centering 
\scriptsize
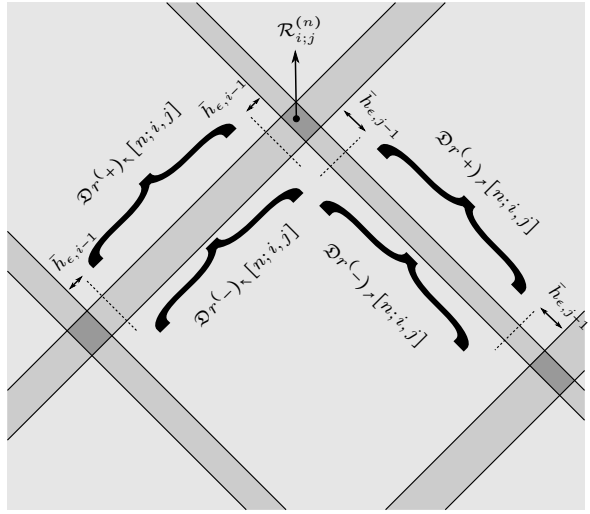 
\caption{In the figure above, we present a schematic depiction of the configuration of neighboring beams when $i\neq j$, using the shorthand notation $\bar{h}_{\epsilon,k}=\rho_{\epsilon}^{-\frac{7}{8}}h_{\epsilon,k}$. The quantities $\mathfrak{D}r_{\nearrow}^{(-)}[n;i,j]$ and $\mathfrak{D}r_{\nearrow}^{(+)}[n;i,j]$ measure the geometric separation of the beams $\mathcal{V}_{j\nearrow}^{(n)}$ and $\mathcal{V}_{j-1\nearrow}^{(n)}$ right before and right after their intersection with $\mathcal{V}_{i\nwarrow}^{(n)}$ (or $\mathcal{V}_{i\nwarrow}^{(n+1)}$, when $i<j$), respectively Similarly, $\mathfrak{D}r_{\nwarrow}^{(-)}[n;i,j]$ and $\mathfrak{D}r_{\nwarrow}^{(+)}[n;i,j]$ measure the separation of $\mathcal{V}_{i\nwarrow}^{(n)}$ and $\mathcal{V}_{i-1\nwarrow}^{(n)}$ (or $\mathcal{V}_{i\nwarrow}^{(n+1)}$ and $\mathcal{V}_{i-1\nwarrow}^{(n+1)}$, when $i<j$) right before and right after their intersection with $\mathcal{V}_{j\nearrow}^{(n)}$, respectively. \label{fig:FigureRSeparation}}
\end{figure}

Finally, setting 
\begin{align}
\tilde{h}_{\varepsilon,i} & \doteq e^{\delta_{\varepsilon}^{-6}}\frac{\varepsilon^{(i)}}{\sqrt{-\Lambda}},\label{eq:ShorthandNotationTilde}\\
\tilde{\beta}_{\varepsilon,i} & \doteq\exp\big(-\exp(\delta_{\varepsilon}^{-4})\big)\frac{\varepsilon^{(i)}}{\sqrt{-\Lambda}},\nonumber 
\end{align}
(noting that $\tilde{h}_{\varepsilon,i}$, $\tilde{\beta}_{\varepsilon,i}$
are defined like $h_{\varepsilon,i}$, $\beta_{\varepsilon,i}$, albeit
with $\delta_{\varepsilon}$ in place of $\sigma_{\varepsilon}$),
we will define $\widetilde{\mathcal{V}}_{i}^{(n)}$, $\widetilde{\mathcal{V}}_{i\nwarrow}^{(n)}$
and $\widetilde{\mathcal{V}}_{i\nearrow}^{(n)}$ by (\ref{eq:BeamDomain})\textendash (\ref{eq:IngoingOutgoingDomains})
with $\tilde{h}_{\varepsilon,i}$, $\tilde{\beta}_{\varepsilon,i}$
in place of $h_{\varepsilon,i}$, $\beta_{\varepsilon,i}$, i.\,e.:
\begin{equation}
\widetilde{\mathcal{V}}_{i}^{(n)}\doteq\Big(\widetilde{\mathcal{V}}_{i\nwarrow}^{(n)}\cup\widetilde{\mathcal{V}}_{i\nearrow}^{(n)}\Big),\label{eq:BeamDomainTilde}
\end{equation}
\begin{align}
\widetilde{\mathcal{V}}_{i\nwarrow}^{(n)} & \doteq\Big\{\Big|v-v_{\varepsilon,i}^{(n)}\Big|\le\tilde{h}_{\varepsilon,i}\Big\}\cap\Big\{\tilde{\beta}_{\varepsilon,i}\le v-u\le\sqrt{-\frac{3}{\Lambda}}\pi\Big\},\label{eq:IngoingOutgoingDomainsTilde}\\
\widetilde{\mathcal{V}}_{i\nearrow}^{(n)} & \doteq\Big\{\Big|u-v_{\varepsilon,i}^{(n)}\Big|\le\tilde{h}_{\varepsilon,i}\Big\}\cap\Big\{\tilde{\beta}_{\varepsilon,i}\le v-u\le\sqrt{-\frac{3}{\Lambda}}\pi\Big\}.\nonumber 
\end{align}
Similarly, we will define $\widetilde{\mathcal{R}}_{i;j}^{(n)}$,
$\widetilde{\mathcal{R}}_{i;\gamma_{\mathcal{Z}}}^{(n)}$, $\widetilde{\mathcal{R}}_{i;\mathcal{I}}^{(n)}$,
$\widetilde{\mathcal{E}}_{\nwarrow}^{(\pm)}[n;i,j]$, $\widetilde{\mathcal{E}}_{\nearrow}^{(\pm)}[n;i,j]$,
$\widetilde{\mathfrak{D}}r_{\nwarrow}^{(\pm)}[n;i,j]$ and $\widetilde{\mathfrak{D}}r_{\nearrow}^{(\pm)}[n;i,j]$
by (\ref{eq:ExplicitIntersectionDomain}), (\ref{eq:ExplicitSpecialDomainAxis}),
(\ref{eq:ExplicitSpecialDomainInfinity}), (\ref{eq:IngoingEnergyBefore})\textendash (\ref{eq:IngoingEnergyAfter}),
(\ref{eq:OutgoingEnergyBefore})\textendash (\ref{eq:OutgoingEnergyAfter}),
(\ref{eq:IngoingRSeparationBeforeAfter}) and (\ref{eq:OutgoingRSeparationBeforeAfter}),
respectively (i.\,e.~using the same definitions as for $\mathcal{V}_{i}^{(n)}$,
$\mathcal{R}_{i;j}^{(n)}$, $\mathcal{E}_{\nwarrow}^{(\pm)}[n;i,j]$,
$\mathcal{E}_{\nearrow}^{(\pm)}[n;i,j]$, $\mathfrak{D}r_{\nwarrow}^{(\pm)}[n;i,j]$
and $\mathfrak{D}r_{\nearrow}^{(\pm)}[n;i,j]$), with $\tilde{h}_{\varepsilon,i}$,
$\tilde{\beta}_{\varepsilon,i}$ in place of $h_{\varepsilon,i}$,
$\beta_{\varepsilon,i}$. 
\begin{rem*}
Note that $\mathcal{V}_{i}^{(n)}\subset\widetilde{\mathcal{V}}_{i}^{(n)}$,
and similarly for $\mathcal{V}_{i\nwarrow}^{(n)}$, $\mathcal{V}_{i\nearrow}^{(n)}$
and $\widetilde{\mathcal{V}}_{i\nwarrow}^{(n)}$, $\widetilde{\mathcal{V}}_{i\nearrow}^{(n)}$.
\end{rem*}

\section{\label{sec:The-technical-core} First steps for the proof of Theorem
\ref{thm:TheTheorem}: Beam interactions and energy concentration}

This section will constitute the technical core of the proof of Theorem
\ref{thm:TheTheorem}. First, in Section \ref{subsec:Control-of-the-Vlasov-beams},
we will obtain estimates controlling the geodesics in the support
of the components $f_{\varepsilon i}$ constituting the total Vlasov
field $f_{\varepsilon}$ (see the relation (\ref{eq:LinearCombinationVlasovFields}))
in the regions $\mathcal{U}_{\varepsilon}^{+},\mathcal{T}_{\varepsilon}^{+}\subset\mathcal{U}_{max}^{(\varepsilon)}$,
showing that the supports of the $f_{\varepsilon i}$'s form a configuration
of intersecting beams in physical space. Then, in Sections \ref{subsec:BeamInteractions}\textendash \ref{subsec:Control-of-the-evolution_ERM},
we will proceed to establish refined estimates for the exchange of
energy occuring at the intersection of any two of those beams, as
well as for the change in the geometric separation of the beams over
time; these bounds will be used in Sections \ref{sec:ThefirstStage}\textendash \ref{sec:The-final-stage}
to show that, provided the initial data parameters $a_{\varepsilon i}$
in (\ref{eq:InitialVlasovTotal}) are chosen appropriately, the total
energy of $f_{\varepsilon}$ is eventually concentrated in regions
of sufficiently small scale in phase space, resulting in the formation
of a trapped sphere. 

\subsection{\label{subsec:Control-of-the-Vlasov-beams} Control of the Vlasov
beams and the spacetime geometry away from the trapped region}

The following lemma will allow us to control the support of the Vlasov
beams $f_{\varepsilon i}$ in the regions $\mathcal{U}_{\varepsilon}^{+}$
and $\mathcal{T}_{\varepsilon}^{+}$ introduced in Section \ref{subsec:Notational-conventions-and}. 
\begin{lem}
\label{lem:QuantitativeBeamEstimate} For any $\varepsilon\in(0,\varepsilon_{1}]$
and any $0\le i\le N_{\varepsilon}$, the support of the Vlasov field
$f_{\varepsilon i}=f_{\varepsilon i}(u,v;p^{u},p^{v},l)$ on $\mathcal{U}_{\varepsilon}^{+}$
satisfies 
\begin{equation}
supp(f_{\varepsilon i})\cap\big\{(u,v)\in\mathcal{U}_{\varepsilon}^{+}\big\}\subset\Big\{(u,v)\in\mathcal{V}_{i}\Big\}\cap\Big\{\exp(-\sigma_{\varepsilon}^{-6})\le\Omega^{2}(p^{u}+p^{v})\le\exp\big(\exp(\sigma_{\varepsilon}^{-4})\big)\Big\},\label{eq:BoundSupportFepsiloni}
\end{equation}
where the regions $\mathcal{V}_{i}=\cup_{n}\mathcal{V}_{i}^{(n)}$
in the $(u,v)$-plane are defined by (\ref{eq:BeamDomain}). Furthermore,
if $\gamma\subset\mathcal{U}_{\varepsilon}^{+}$ is a future directed,
affinely parametrised null geodesic in the support of $f_{\varepsilon i}$
which is maximally extended through reflections off $\mathcal{I}$
(see Definition 2.3 in \cite{MoschidisVlasovWellPosedness}) and $p$
is any point on $\gamma$, then:
\begin{equation}
\frac{\dot{\gamma}^{v}}{\dot{\gamma}^{u}}\Big|_{p}\le\exp\big(\exp(\sigma_{\varepsilon}^{-4})\big)l^{2}\frac{1-\frac{1}{3}\Lambda r^{2}}{r^{2}}\Big|_{p}\text{ if }p\in\cup_{n\in\mathbb{N}}\mathcal{V}_{i\nwarrow}^{(n)}\label{eq:IngoingBoundBootstrapDomain}
\end{equation}
and 
\begin{equation}
\frac{\dot{\gamma}^{u}}{\dot{\gamma}^{v}}\Big|_{p}\le\exp\big(\exp(\sigma_{\varepsilon}^{-4})\big)l^{2}\frac{1-\frac{1}{3}\Lambda r^{2}}{-\Lambda r^{2}}\Big|_{p}\text{ if }p\in\cup_{n\in\mathbb{N}}\mathcal{V}_{i\nearrow}^{(n)},\label{eq:OutgoingBoundBootstrapDomain}
\end{equation}
where $l$ is the angular momentum of $\gamma$.

Similarly, on $\mathcal{T}_{\varepsilon}^{+}$, the support of $f_{\varepsilon i}$
satisfies 
\begin{equation}
supp(f_{\varepsilon i})\cap\big\{(u,v)\in\mathcal{T}_{\varepsilon}^{+}\big\}\subset\Big\{(u,v)\in\widetilde{\mathcal{V}}_{i}\Big\}\cap\Big\{\exp(-\delta_{\varepsilon}^{-6})\le\Omega^{2}(p^{u}+p^{v})\le\exp\big(\exp(\delta_{\varepsilon}^{-4})\big)\Big\},\label{eq:BoundSupportFepsiloni-1}
\end{equation}
Furthermore, if $\gamma\subset\mathcal{T}_{\varepsilon}^{+}$ is a
future directed, affinely parametrised null geodesic in the support
of $f_{\varepsilon i}$ which is maximally extended through reflections
and $p$ is any point on $\gamma$, then (\ref{eq:IngoingBoundBootstrapDomain})\textendash (\ref{eq:OutgoingBoundBootstrapDomain})
hold with $\delta_{\varepsilon}$ in place of $\sigma_{\varepsilon}$.
\end{lem}
\begin{rem*}
The bound (\ref{eq:BoundSupportFepsiloni}) implies that the domains
$\mathcal{V}_{i}$ (or $\widetilde{\mathcal{V}}_{i}$, in the case
of (\ref{eq:BoundSupportFepsiloni-1})) strictly contain the Vlasov
beams $\zeta_{i}$ appearing in the discussion of Section \ref{subsec:Discussion-on-the}.
\end{rem*}
\begin{proof}
The proof of Lemma (\ref{lem:QuantitativeBeamEstimate}) will be a
simple consequence of Corollary \ref{cor:GeodesicPathsLongTimes}.
In particular, let $\gamma\subset\mathcal{U}_{\varepsilon}^{+}$ is
a future directed, affinely parametrised null geodesic in the support
of $f_{\varepsilon i}$ which is maximally extended through reflections.
Setting 
\begin{equation}
C_{0}\doteq5\sigma_{\varepsilon}^{-3},\label{eq:ConstantCZero}
\end{equation}
\begin{equation}
v_{\mathcal{I}}\doteq\sqrt{-\frac{3}{\Lambda}}\pi
\end{equation}
 and 
\begin{equation}
U\doteq\sup_{\mathcal{U}_{\varepsilon}^{+}}u,\label{eq:DefiU}
\end{equation}
we readily observe the following:

\begin{itemize}

\item Using the definition (\ref{eq:BootstrapDomain}) of $\mathcal{U}_{\varepsilon}^{+}$
and (\ref{eq:DefiU}), we can readily bound 
\begin{equation}
U\le\frac{\sigma_{\varepsilon}^{-2}}{\sqrt{-\Lambda}}.\label{eq:UpperU}
\end{equation}

\item The bound (\ref{eq:SmallnessTrapping}) holds on $\mathcal{U}_{\varepsilon}^{+}\doteq\mathcal{U}_{U;v_{\mathcal{I}}}$,
in view of (\ref{eq:UpperBoundMuBootstrapDomain}) (using also the
assumption that $\eta_{0}<\delta_{0}$). Moreover, the bound (\ref{eq:BoundGeometryC0})
follows from (\ref{eq:UpperU}) and the estimate (\ref{eq:EstimateGeometryLikeC0}),
assuming that $\varepsilon_{1}$ has been fixed small enough.

\item As a consequence of the expression (\ref{eq:InitialVlasovSeeds})
for $F_{i}^{(\varepsilon)}$ and the relation (\ref{eq:Fepsiloni})
between $F_{i}^{(\varepsilon)}$ and $f_{\varepsilon i}$ (using also
the bound (\ref{eq:ComparableInitialDataWithAdSEpsilonn}) for $\partial_{v}r_{/}^{(\varepsilon)}$
in (\ref{eq:Fepsiloni})), we can estimate for the angular momentum
$l$ and the initial energy $E_{0}$ of $\gamma$ (defined by (\ref{eq:InitialEnergy}))
that 
\[
\frac{1}{10}\varepsilon^{(i)}\le\sqrt{-\Lambda}l\le10\varepsilon^{(i)}
\]
and
\[
\frac{1}{10}\le E_{0}\le10,
\]
as well as 
\[
\frac{\dot{\gamma}^{v}}{\dot{\gamma}^{u}}\Big|_{u=0}<1.
\]
Therefore, $\gamma$ satisfies the conditions (\ref{eq:initiallyIngoing}),
(\ref{eq:InitialRGamma2}) and (\ref{eq:SmallAngularMomentumMaximallyExtended}).

\end{itemize}

Hence, the conditions of Corollary \ref{cor:GeodesicPathsLongTimes}
are satisfied for $\gamma$, provided $\varepsilon_{1}$ is chosen
smaller than some absolute constant. As a result, (\ref{eq:BoundSupportFepsiloni})
follows readily from (\ref{eq:DomainForGammaMaximallyExtended}) and
(\ref{eq:EnergyGrowthGammaMaximallyExtended}), while (\ref{eq:IngoingBoundBootstrapDomain})
and (\ref{eq:OutgoingBoundBootstrapDomain}) follow from \ref{eq:IngoingRegionMaximallyExtended}\textendash \ref{eq:OutgoingRegionMaximallyExtended}.

The corresponding statements for $\gamma\subset\mathcal{T}_{\varepsilon}^{+}$
follow by exactly the same arguments, after replacing $\sigma_{\varepsilon}$
with $\delta_{\varepsilon}$ in (\ref{eq:ConstantCZero}) and using
(\ref{eq:EstimateGeometryLikeC0LargerDomain}) in place of (\ref{eq:EstimateGeometryLikeC0}).
\end{proof}
The following Lemma will allow us to control various quantities related
to the geometry of $(\mathcal{U}_{\varepsilon}^{+};r,\Omega^{2})$
and $(\mathcal{T}_{\varepsilon}^{+};r,\Omega^{2})$, some which are
of higher regularity than that controlled by the norm \ref{eq:InitialDataNormSlice}.
Effective control on such quantities will be obtained through the
quantitative estimates provided by Lemma \ref{lem:QuantitativeBeamEstimate}
on the support of the Vlasov fields $f_{\varepsilon i}$.
\begin{lem}
\label{lem:ControlHigherOrderDerivatives} For any $\varepsilon\in(0,\varepsilon_{1}]$
and any $0\le i\le N_{\varepsilon}$, the following estimate holds
on $\mathcal{U}_{\varepsilon}^{+}$: 
\begin{equation}
\Big|\log\Big(\frac{\Omega^{2}}{1-\frac{1}{3}\Lambda r^{2}}\Big)\Big|\le10\sigma_{\varepsilon}^{-3},\label{eq:UpperBoundOmegaLikeC0}
\end{equation}
while the following estimates hold in the regions $\cup_{k\in\mathbb{N}}\mathcal{V}_{i}^{(k)}\cap\mathcal{U}_{\varepsilon}^{+}$:
\begin{equation}
\inf_{\mathcal{V}_{i}\cap\mathcal{U}_{\varepsilon}^{+}}r\ge\exp\big(-2\exp(\sigma_{\varepsilon}^{-4})\big)\frac{\varepsilon^{(i)}}{\sqrt{-\Lambda}},\label{eq:LowerBoundROnDomains}
\end{equation}
\begin{align}
\begin{cases}
r^{2}T_{vv}[f_{\varepsilon i}](u,v)\le\exp\big(\exp(\sigma_{\varepsilon}^{-5})\big) & \text{and}\\
r^{2}T_{uu}[f_{\varepsilon i}](u,v)\le\exp\big(\exp(2\sigma_{\varepsilon}^{-5})\big)\frac{(\varepsilon^{(i)})^{4}}{r^{4}(u,v)}(-\Lambda)^{-2}, & \text{ if }(u,v)\in\cup_{n\in\mathbb{N}}\mathcal{V}_{i\nwarrow}^{(n)},
\end{cases}\label{eq:UpperBoundTuuTvvFi}\\
\begin{cases}
r^{2}T_{uu}[f_{\varepsilon i}](u,v)\le\exp\big(\exp(\sigma_{\varepsilon}^{-5})\big) & \text{and}\\
r^{2}T_{vv}[f_{\varepsilon i}](u,v)\le\exp\big(\exp(2\sigma_{\varepsilon}^{-5})\big)\frac{(\varepsilon^{(i)})^{4}}{r^{4}(u,v)}(-\Lambda)^{-2}, & \text{ if }(u,v)\in\cup_{n\in\mathbb{N}}\mathcal{V}_{i\nearrow}^{(n)},
\end{cases}\nonumber 
\end{align}
and 
\begin{equation}
r^{2}T_{uv}[f_{\varepsilon i}](u,v)\le\exp\big(\exp(\sigma_{\varepsilon}^{-5})\big)\cdot\frac{(\varepsilon^{(i)})^{2}}{r^{2}(u,v)}(-\Lambda)^{-1}.\label{eq:UpperBoundTuvFi}
\end{equation}
Furthermore, we can estimate 
\begin{align}
\sup_{(u,v)\in\mathcal{U}_{\varepsilon}^{+}}\Big|dist_{\nwarrow}[(u,v)]\cdot\partial_{v}\log\Big(\frac{\Omega^{2}}{1-\frac{1}{3}\Lambda r^{2}}\Big)(u,v)\Big|+\label{eq:BoundDOmegaEverywhere}\\
+\sup_{(u,v)\in\mathcal{U}_{\varepsilon}^{+}}\Big|dist_{\nearrow}[(u,v)]\cdot & \partial_{u}\log\Big(\frac{\Omega^{2}}{1-\frac{1}{3}\Lambda r^{2}}\Big)(u,v)\Big|\le\exp\big(\exp(\sigma_{\varepsilon}^{-5})\big)\nonumber 
\end{align}
and 
\begin{align}
\sup_{(u,v)\in\mathcal{U}_{\varepsilon}^{+}}\Big|dist_{\nwarrow}[(u,v)]\cdot\partial_{v}\Big(\frac{\partial_{v}r}{1-\frac{1}{3}\Lambda r^{2}}\Big)(u,v)\Big|+\label{eq:BoundDDrEverywhere}\\
+\sup_{(u,v)\in\mathcal{U}_{\varepsilon}^{+}}\Big|dist_{\nearrow}[(u,v)]\cdot & \partial_{u}\Big(\frac{\partial_{u}r}{1-\frac{1}{3}\Lambda r^{2}}\Big)(u,v)\Big|\le\exp\big(\exp(\sigma_{\varepsilon}^{-5})\big),\nonumber 
\end{align}
where the functions $dist_{\nwarrow}[\cdot]$ and $dist_{\nearrow}[\cdot]$
are defined by (\ref{eq:IngoingDistanceFunction}) and (\ref{eq:OutgoingDistanceFunction}).
Moreover, for any $n\in\mathbb{N}$ and $0\le i,j\le N_{\varepsilon}$,
$i\neq j$, such that $\mathcal{R}_{i;j}^{(n)}\subset\mathcal{U}_{+}^{(\varepsilon)}$,
the following estimates hold on $\mathcal{R}_{i;j}^{(n)}$, depending
on whether $i>j$ or $i<j$:

\begin{itemize}

\item In the case $i>j$, 
\begin{equation}
\exp(-\sigma_{\varepsilon}^{-7})\rho_{\varepsilon}^{-1}\le\frac{\sqrt{-\Lambda}r|_{\mathcal{R}_{i;j}^{(n)}}}{\varepsilon^{(j)}}\le\exp(\sigma_{\varepsilon}^{-7})\rho_{\varepsilon}^{-1}\label{eq:BoundsRIntersectionRegioni>j}
\end{equation}
and 
\begin{equation}
\sup_{\mathcal{R}_{i;j}^{(n)}}r-\inf_{\mathcal{R}_{i;j}^{(n)}}r\le\frac{\exp(\sigma_{\varepsilon}^{-7})}{\sqrt{-\Lambda}}\varepsilon^{(j)}.\label{eq:ComparisonRIntersectionRegioni>j}
\end{equation}

\item In the case $i<j$, 
\begin{equation}
\exp(-\sigma_{\varepsilon}^{-4})\rho_{\varepsilon}\frac{1}{\varepsilon^{(i)}}\le\sqrt{-\Lambda}r|_{\mathcal{R}_{i;j}^{(n)}}\le\exp(\sigma_{\varepsilon}^{-4})\rho_{\varepsilon}\frac{1}{\varepsilon^{(i)}}\label{eq:BoundsRIntersectionRegioni<j}
\end{equation}
and 
\begin{equation}
\sup_{\mathcal{R}_{i;j}^{(n)}}\frac{1}{r}-\inf_{\mathcal{R}_{i;j}^{(n)}}\frac{1}{r}\le\exp(\sigma_{\varepsilon}^{-7})\sqrt{-\Lambda}\varepsilon^{(i)}.\label{eq:ComparisonRIntersectionRegioni<j}
\end{equation}

\end{itemize}

Replacing $\mathcal{U}_{\varepsilon}^{+}$ with $\mathcal{T}_{\varepsilon}^{+}$,
the estimates (\ref{eq:UpperBoundOmegaLikeC0})\textendash (\ref{eq:ComparisonRIntersectionRegioni<j})
still hold with $\delta_{\varepsilon}$ in place of $\sigma_{\varepsilon}$,
$\widetilde{\mathcal{V}}_{i}^{(n)}$ in place of $\mathcal{V}_{i}^{(n)}$
and $\widetilde{\mathcal{R}}_{i;j}^{(n)}$ in place of $\mathcal{R}_{i;j}^{(n)}$.
\end{lem}
\begin{rem*}
Note that, in view of the relations (\ref{eq:SimpleEqualityDistance1})
and (\ref{SimpleEqualityDistance2}), the estimates (\ref{eq:BoundDOmegaEverywhere})
and (\ref{eq:BoundDDrEverywhere}) yield, as a special case, that,
for any $0\le i\le N_{\varepsilon}$:
\begin{equation}
\sup_{\cup_{k\in\mathbb{N}}\mathcal{V}_{i\nwarrow}^{(k)}\cap\mathcal{U}_{\varepsilon}^{+}}\Big|\partial_{v}\log\Big(\frac{\Omega^{2}}{1-\frac{1}{3}\Lambda r^{2}}\Big)\Big|+\sup_{\cup_{k\in\mathbb{N}}\mathcal{V}_{i\nearrow}^{(k)}\cap\mathcal{U}_{\varepsilon}^{+}}\Big|\partial_{u}\log\Big(\frac{\Omega^{2}}{1-\frac{1}{3}\Lambda r^{2}}\Big)\Big|\le\exp\big(2\exp(\sigma_{\varepsilon}^{-5})\big)\frac{\sqrt{-\Lambda}}{\varepsilon^{(i)}}\label{eq:BoundDOmegaOnDomainsSpecialCase}
\end{equation}
and 
\begin{equation}
\sup_{\cup_{k\in\mathbb{N}}\mathcal{V}_{i\nwarrow}^{(k)}\cap\mathcal{U}_{\varepsilon}^{+}}\Big|\partial_{v}\Big(\frac{\partial_{v}r}{1-\frac{1}{3}\Lambda r^{2}}\Big)\Big|+\sup_{\cup_{k\in\mathbb{N}}\mathcal{V}_{i\nearrow}^{(k)}\cap\mathcal{U}_{\varepsilon}^{+}}\Big|\partial_{u}\Big(\frac{\partial_{u}r}{1-\frac{1}{3}\Lambda r^{2}}\Big)\Big|\le\exp\big(2\exp(\sigma_{\varepsilon}^{-5})\big)\frac{\sqrt{-\Lambda}}{\varepsilon^{(i)}}\label{eq:BoundDdrOnDomainsSpecialCase}
\end{equation}
Notice that the left hand sides of (\ref{eq:BoundDOmegaOnDomainsSpecialCase})
and (\ref{eq:BoundDdrOnDomainsSpecialCase}) are not be estimated
by the low regularity norm (\ref{eq:InitialDataNormSlice}). This
loss of regularity is reflected in the fact that the right hand sides
of (\ref{eq:BoundDOmegaOnDomainsSpecialCase}) and (\ref{eq:BoundDdrOnDomainsSpecialCase})
can not be bounded merely by terms of the form $(\exp(\exp(\sigma_{\varepsilon}^{-C}))$,
but have additional $(\varepsilon^{(i)})^{-1}$ terms which, in these
cases, are optional. 
\end{rem*}
\begin{proof}
\noindent Let $\varepsilon\in(0,\varepsilon_{1}]$ and let $0\le i\le N_{\varepsilon}$.
In view of the formula (\ref{eq:DefinitionHereHawkingMass}) and the
bound (\ref{eq:UpperBoundMuBootstrapDomain}) for $2\tilde{m}/r$,
the estimate (\ref{eq:EstimateGeometryLikeC0}) readily implies (\ref{eq:UpperBoundOmegaLikeC0}).

Using the fact that 
\[
\inf_{\mathcal{V}_{i}}(v-u)=\exp\big(-\exp(\sigma_{\varepsilon}^{-4})\big)\frac{\varepsilon^{(i)}}{\sqrt{-\Lambda}}
\]
(following from the definition of (\ref{eq:IngoingOutgoingDomains})$\mathcal{V}_{i\nwarrow}^{(n)},\mathcal{V}_{i\nearrow}^{(n)}$),
the lower bound (\ref{eq:LowerBoundROnDomains}) for $r$ follows
readily after integrating $(\partial_{v}-\partial_{u})r$ in $\partial_{v}-\partial_{u}$
starting from $\gamma_{\mathcal{Z}_{\varepsilon}}$ and using the
bound (\ref{eq:EstimateGeometryLikeC0}).

The estimates (\ref{eq:UpperBoundTuuTvvFi}) and (\ref{eq:UpperBoundTuvFi})
for $T_{\mu\nu}[f_{\varepsilon i}]$ follow readily from the explicit
formulas (\ref{eq:ComponentsStressEnergy}) for $T_{\mu\nu}[f_{\varepsilon i}]$,
the estimates (\ref{eq:BoundSupportFepsiloni}) and (\ref{eq:IngoingBoundBootstrapDomain})\textendash (\ref{eq:OutgoingBoundBootstrapDomain})
for the support of $f_{\varepsilon i}$ in $p^{u}$, $p^{v}$, the
fact that $f_{\varepsilon i}$ is supported on $\{2\le\frac{l\sqrt{-\Lambda}}{\varepsilon^{(i)}}\le6\}$
(in view of (\ref{eq:InitialVlasovSeeds}) and (\ref{eq:Fepsiloni})),
the bound (\ref{eq:UpperBoundFbarEpsilon}), and the bounds (\ref{eq:EstimateGeometryLikeC0}),
(\ref{eq:UpperBoundOmegaLikeC0}). 

For any $(\bar{u},\bar{v})\in\mathcal{U}_{\varepsilon}^{+}$, integrating
the renormalised equation (\ref{eq:RenormalisedEquations}) for $\Omega$
in $u$ along $\zeta_{\nwarrow}[(\bar{u},\bar{v})]$ and in $v$ along
$\zeta_{\nearrow}[(\bar{u},\bar{v})]$ (see (\ref{eq:CrookedLineIngoing})
and (\ref{eq:CrookedLineOutgoing}) for the definition of $\zeta_{\nwarrow}[\cdot]$
and $\zeta_{\nearrow}[\cdot]$), making use of the boundary conditions
(\ref{eq:BoundaryConditionOmegaAxis})\textendash (\ref{eq:BoundaryConditionOmegaInfinity})
for $\Omega^{2}$ at $\gamma_{\mathcal{Z}}$, $\mathcal{I}_{\varepsilon}$,
we infer that:
\begin{align}
\Big|\partial_{v} & \log\Big(\frac{\Omega^{2}}{1-\frac{1}{3}\Lambda r^{2}}\Big)(\bar{u},\bar{v})\Big|\le\label{eq:EstimateDvOmegaDomains}\\
\le & \int_{\zeta_{\nwarrow}[(\bar{u},\bar{v})]}\Big(\frac{\tilde{m}}{r}\Big(\frac{1}{r^{2}}+\frac{1}{3}\Lambda\frac{\Lambda r^{2}-1}{1-\frac{1}{3}\Lambda r^{2}}\Big)\cdot\frac{\Omega^{2}}{1-\frac{1}{3}\Lambda r^{2}}-16\pi\frac{1-\frac{1}{2}\Lambda r^{2}}{1-\frac{1}{3}\Lambda r^{2}}\frac{1}{r^{2}}(r^{2}T_{uv}[f_{\varepsilon}])\Big)\,d(u+v)+\nonumber \\
 & +\Big|\partial_{v}\log\Big(\frac{(\Omega_{/}^{(\varepsilon)})^{2}}{1-\frac{1}{3}\Lambda(r_{/}^{(\varepsilon)})^{2}}\Big)\Big|\Bigg|_{\{u=0\}\cap\zeta_{\nwarrow}[(\bar{u},\bar{v})]}\nonumber 
\end{align}
and 
\begin{align}
\Big|\partial_{u} & \log\Big(\frac{\Omega^{2}}{1-\frac{1}{3}\Lambda r^{2}}\Big)(\bar{u},\bar{v})\Big|\le\label{eq:EstimateDuOmegaDomains}\\
\le & \int_{\zeta_{\nearrow}[(\bar{u},\bar{v})]}\Big(\frac{\tilde{m}}{r}\Big(\frac{1}{r^{2}}+\frac{1}{3}\Lambda\frac{\Lambda r^{2}-1}{1-\frac{1}{3}\Lambda r^{2}}\Big)\cdot\frac{\Omega^{2}}{1-\frac{1}{3}\Lambda r^{2}}-16\pi\frac{1-\frac{1}{2}\Lambda r^{2}}{1-\frac{1}{3}\Lambda r^{2}}\frac{1}{r^{2}}(r^{2}T_{uv}[f_{\varepsilon}])\Big)\,d(u+v)+\nonumber \\
 & +\Big|\partial_{v}\log\Big(\frac{(\Omega_{/}^{(\varepsilon)})^{2}}{1-\frac{1}{3}\Lambda(r_{/}^{(\varepsilon)})^{2}}\Big)\Big|\Bigg|_{\{u=0\}\cap\zeta_{\nearrow}[(\bar{u},\bar{v})]}\nonumber 
\end{align}
Making use of the following:

\begin{itemize}

\item The bounds (\ref{eq:UpperBoundMuBootstrapDomain}) and (\ref{eq:UpperBoundU+VBootstrapDomain})
for $2\tilde{m}/r$ and $\sqrt{-\Lambda}(u+v)$ (the latter estimating
the number of straight segments contained in $\zeta_{\nwarrow}[(\bar{u},\bar{v})]$
and $\zeta_{\nearrow}[(\bar{u},\bar{v})]$),

\item The bounds (\ref{eq:EstimateGeometryLikeC0}) and (\ref{eq:UpperBoundOmegaLikeC0})
for $\frac{\partial_{v}r}{1-\frac{1}{3}\Lambda r^{2}}$, $\frac{\partial_{u}r}{1-\frac{1}{3}\Lambda r^{2}}$
and $\frac{\Omega^{2}}{1-\frac{1}{3}\Lambda r^{2}}$

\item The bound .
\begin{equation}
r^{2}T_{uv}[f_{\varepsilon}]\le\frac{1}{2\pi}\frac{-\partial_{u}r}{1-\frac{2m}{r}}\partial_{v}\tilde{m}\text{ and }r^{2}T_{uv}[f_{\varepsilon}]\le\frac{1}{2\pi}\frac{\partial_{v}r}{1-\frac{2m}{r}}(-\partial_{u}\tilde{m})
\end{equation}
(following from (\ref{eq:TildeUMaza})\textendash (\ref{eq:TildeVMaza}))

\item The fact that the support of $\tilde{m}|_{\zeta_{\nwarrow}[(\bar{u},\bar{v})]}$,
$T_{\mu\nu}|_{\zeta_{\nwarrow}[(\bar{u},\bar{v})]}$ is contained
in $\{v-u\ge dist_{\nwarrow}[(\bar{u},\bar{v})]\}$ (and similarly
for $\zeta_{\nearrow}[(\bar{u},\bar{v})]$),

\item The trivial estimates
\begin{align}
\sup_{0\le\hat{u}\le\bar{u}}\int_{\{u=\hat{u}\}\cap\{v-u\ge dist_{\nwarrow}[(\bar{u},\bar{v})]\}}\frac{1}{r^{2}}(\partial_{v}r)\,dv+\\
+\sup_{0\le\hat{v}\le\bar{v}}\int_{\{v=\hat{v}\}\cap\{v-u\ge dist_{\nwarrow}[(\bar{u},\bar{v})]\}} & \frac{1}{r^{2}}(-\partial_{u}r)\,du\le e^{\sigma_{\varepsilon}^{-4}}\frac{1}{dist_{\nwarrow}[(\bar{u},\bar{v})]}\nonumber 
\end{align}
and 
\begin{align}
\sup_{0\le\hat{u}\le\bar{u}}\int_{\{u=\hat{u}\}\cap\{v-u\ge dist_{\nearrow}[(\bar{u},\bar{v})]\}}\frac{1}{r^{2}}(\partial_{v}r)\,dv+\\
+\sup_{0\le\hat{v}\le\bar{v}}\int_{\{v=\hat{v}\}\cap\{v-u\ge dist_{\nearrow}[(\bar{u},\bar{v})]\}} & \frac{1}{r^{2}}(-\partial_{u}r)\,du\le e^{\sigma_{\varepsilon}^{-4}}\frac{1}{dist_{\nearrow}[(\bar{u},\bar{v})]}\nonumber 
\end{align}
(following from (\ref{eq:EstimateGeometryLikeC0})),

\item The initial data estimates (\ref{eq:ComparableInitialDataWithAdSEpsilonn})
and (\ref{eq:BoundDvdvRAndDvOmegaInitiallyInBeams}) for $\partial_{v}\log\Big(\frac{(\Omega_{/}^{(\varepsilon)})^{2}}{1-\frac{1}{3}\Lambda(r_{/}^{(\varepsilon)})^{2}}\Big)$,

\end{itemize}

we infer from (\ref{eq:EstimateDvOmegaDomains})\textendash (\ref{eq:EstimateDuOmegaDomains})
that 
\begin{equation}
\Big|\partial_{v}\log\Big(\frac{\Omega^{2}}{1-\frac{1}{3}\Lambda r^{2}}\Big)(\bar{u},\bar{v})\Big|\le e^{2\sigma_{\varepsilon}^{-4}}\frac{1}{dist_{\nwarrow}[(\bar{u},\bar{v})]}+C\sqrt{-\Lambda}\label{eq:EstimateDvOmegaDomains-1}
\end{equation}
and 
\begin{equation}
\Big|\partial_{u}\log\Big(\frac{\Omega^{2}}{1-\frac{1}{3}\Lambda r^{2}}\Big)(\bar{u},\bar{v})\Big|\le e^{2\sigma_{\varepsilon}^{-4}}\frac{1}{dist_{\nearrow}[(\bar{u},\bar{v})]}+C\sqrt{-\Lambda}.\label{eq:EstimateDvOmegaDomains-1-1}
\end{equation}
The bound (\ref{eq:BoundDOmegaEverywhere}) follows readily from (\ref{eq:EstimateDvOmegaDomains-1})\textendash (\ref{eq:EstimateDvOmegaDomains-1-1}).

The constraint equations (\ref{eq:ConstraintVFinal}) and (\ref{eq:ConstraintUFinal})
imply that 
\begin{equation}
\partial_{v}\Big(\frac{\partial_{v}r}{1-\frac{1}{3}\Lambda r^{2}}\Big)=\partial_{v}\log\Big(\frac{\Omega^{2}}{1-\frac{1}{3}\Lambda r^{2}}\Big)\cdot\frac{\partial_{v}r}{1-\frac{1}{3}\Lambda r^{2}}-\frac{4\pi}{1-\frac{1}{3}\Lambda r^{2}}\frac{1}{r}\cdot r^{2}T_{vv}[f_{\varepsilon}]\label{eq:ConstraintVRenormalised}
\end{equation}
and 
\begin{equation}
\partial_{u}\Big(\frac{\partial_{u}r}{1-\frac{1}{3}\Lambda r^{2}}\Big)=\partial_{u}\log\Big(\frac{\Omega^{2}}{1-\frac{1}{3}\Lambda r^{2}}\Big)\cdot\frac{\partial_{u}r}{1-\frac{1}{3}\Lambda r^{2}}-\frac{4\pi}{1-\frac{1}{3}\Lambda r^{2}}\frac{1}{r}\cdot r^{2}T_{uu}[f_{\varepsilon}].\label{eq:ConstraintURenormalised}
\end{equation}
The estimate (\ref{eq:BoundDDrEverywhere}) is obtained readily from
the relations (\ref{eq:ConstraintVRenormalised})\textendash (\ref{eq:ConstraintURenormalised}),
the bound (\ref{eq:BoundDOmegaEverywhere}) for $\frac{\partial\Omega^{2}}{1-\frac{1}{3}\Lambda r^{2}}$,
the bound (\ref{eq:EstimateGeometryLikeC0}) for $\frac{\partial r}{1-\frac{1}{3}\Lambda r^{2}}$,
the bounds (\ref{eq:UpperBoundTuuTvvFi}) for $f_{\varepsilon i}$,
the fact that $T_{\mu\nu}[f_{\varepsilon}]$ is supported only on
$\cup_{i=0}^{N_{\varepsilon}}\mathcal{V}_{i}$ and the trivial estimate
\begin{align*}
\sup_{(u,v)\in\cup_{n\in\mathbb{N}}\cup_{i=0}^{N_{\varepsilon}}\mathcal{V}_{i}^{(n)}}\Big( & \frac{\max\{dist_{\nwarrow}[(u,v)],\text{ }dist_{\nearrow}[(u,v)]\}}{r(u,v)}\Big)\le\\
 & \le e^{\sigma_{\varepsilon}^{-4}}\sup_{(u,v)\in\cup_{n\in\mathbb{N}}\cup_{i=0}^{N_{\varepsilon}}\mathcal{V}_{i}^{(n)}}\Big(\frac{\max\{dist_{\nwarrow}[(u,v)],\text{ }dist_{\nearrow}[(u,v)]\}}{v-u}\Big)\le e^{\sigma_{\varepsilon}^{-4}}
\end{align*}
(following from the bound (\ref{eq:EstimateGeometryLikeC0}) and the
definition (\ref{eq:IngoingDistanceFunction}), (\ref{eq:OutgoingDistanceFunction})
of $dist_{\nwarrow}[\cdot]$, $dist_{\nearrow}[\cdot]$).

Finally, let $n\in\mathbb{N}$ and $0\le i,j\le N_{\varepsilon}$,
$i\neq j$, be such that $\mathcal{R}_{i;j}^{(n)}\subset\mathcal{U}_{+}^{(\varepsilon)}$.
In view of the form (\ref{eq:ExplicitIntersectionDomain}) of $\mathcal{R}_{i;j}^{(n)}$,
we infer the following: 

\begin{itemize}

\item In the case $i>j$, integrating $(\partial_{v}-\partial_{u})r$
in $\partial_{v}-\partial_{u}$ from $(\frac{u+v}{2},\frac{u+v}{2})\in\gamma_{\mathcal{Z}_{\varepsilon}}$
up to $(u,v)\in\mathcal{\mathcal{R}}_{i;j}^{(n)}$, using the bound
(\ref{eq:EstimateGeometryLikeC0}) and the formulas (\ref{eq:InitialCenters})
and (\ref{eq:ShorthandNotationCornerPoints}) for $v_{i,\varepsilon}$
and $v_{i,\varepsilon}^{(n)}$, we obtain (\ref{eq:BoundsRIntersectionRegioni>j})
and (\ref{eq:ComparisonRIntersectionRegioni>j}).

\item In the case $i<j$, arguing similarly but integrating $(\partial_{v}-\partial_{u})\frac{1}{r}$
in $\partial_{v}-\partial_{u}$ from $(\frac{u+v}{2}-\frac{1}{2}\sqrt{-\frac{3}{\Lambda}}\pi,\frac{u+v}{2}+\frac{1}{2}\sqrt{-\frac{3}{\Lambda}}\pi)\in\mathcal{I}_{\varepsilon}$
up to $(u,v)\in\mathcal{\mathcal{R}}_{i;j}^{(n)}$, we obtain (\ref{eq:BoundsRIntersectionRegioni<j})
and (\ref{eq:ComparisonRIntersectionRegioni<j}). 

\end{itemize}

The proof of the analogous estimates for $\mathcal{T}_{\varepsilon}^{+}$
in place of $\mathcal{U}_{\varepsilon}^{+}$ (with $\delta_{\varepsilon}$,
$\widetilde{\mathcal{V}}_{i}^{(n)}$, $\widetilde{\mathcal{R}}_{i;j}^{(n)}$
replacing $\sigma_{\varepsilon}$, $\mathcal{V}_{i}^{(n)}$, $\mathcal{R}_{i;j}^{(n)}$)
follows in exactly the same way, using (\ref{eq:EstimateGeometryLikeC0LargerDomain})
in place of (\ref{eq:EstimateGeometryLikeC0}). 
\end{proof}

\subsection{\label{subsec:BeamInteractions} Interaction of the Vlasov beams:
Energy exchange and concentration}

In this section, we will establish a number of results providing quantitative
bounds on the change of the energy content (as measured by (\ref{eq:IngoingEnergyBefore})\textendash (\ref{eq:OutgoingEnergyAfter}))
and the geometric separation (as measured by (\ref{eq:IngoingRSeparationBeforeAfter})\textendash (\ref{eq:OutgoingRSeparationBeforeAfter}))
of the beams $\mathcal{V}_{i}^{(n)}$, before and after their pairwise
intersections. As a corollary of these technical bounds, we will be
able to estimate the total change of the energy content and the separation
of the beams after each successive reflection off $\mathcal{I}_{\varepsilon}$
in the next section (see Proposition \ref{prop:TotalEnergyChange}).

The next result provides an estimate for the change of the energy
content of the beams $\mathcal{V}_{j\nearrow}^{(n)}$ and $\mathcal{V}_{i\nwarrow}^{(n)}$
(or $\mathcal{V}_{i\nwarrow}^{(n+1)}$, if $i<j$) before and after
their intersection.
\begin{prop}
\label{prop:EnergyChangeInteraction} Let $\varepsilon\in(0,\varepsilon_{1}]$
and let $n\in\mathbb{N}$ and $0\le i,j\le N_{\varepsilon}$, $i\neq j$,
be such that 
\[
\mathcal{R}_{i;j}^{(n)}\subset\mathcal{U}_{\varepsilon}^{+}.
\]
Let us also define 
\begin{equation}
r_{n;i.j}\doteq\inf_{\mathcal{R}_{i;j}^{(n)}}r.\label{eq:DefinitionRnij}
\end{equation}
Then the following relations hold for the change of the energy of
the two Vlasov beams entering and leaving the intersection region
$\mathcal{R}_{i;j}^{(n)}$:

\begin{itemize}

\item If $i>j$, then 
\begin{align}
\mathcal{E}_{\nwarrow}^{(+)}[n;i,j] & =\mathcal{E}_{\nwarrow}^{(-)}[n;i,j]\cdot\exp\Big(\frac{2\mathcal{E}_{\nearrow}^{(-)}[n;i,j]}{r_{n;i.j}}+O(\rho_{\varepsilon}^{\frac{3}{2}})\Big)+O\Big(\rho_{\varepsilon}^{\frac{3}{2}}\frac{\varepsilon^{(i)}}{\sqrt{-\Lambda}}\Big),\label{eq:IncreaseIngoingEnergyCloseToAxis}\\
\mathcal{E}_{\nearrow}^{(+)}[n;i,j] & =\mathcal{E}_{\nearrow}^{(-)}[n;i,j]\cdot\big(1+O(\varepsilon)\big)+O\Big(\rho_{\varepsilon}^{\frac{3}{2}}\frac{\varepsilon^{(j)}}{\sqrt{-\Lambda}}\Big)\label{eq:DecreaseOutgoingEnergyCloseToAxis}
\end{align}

\item If $i<j$, then 
\begin{align}
\mathcal{E}_{\nwarrow}^{(+)}[n;i,j] & =\mathcal{E}_{\nwarrow}^{(-)}[n;i,j]\cdot\big(1+O(\varepsilon)\big)+O\Big(\varepsilon\frac{\varepsilon^{(i)}}{\sqrt{-\Lambda}}\Big),\label{eq:IngoingEnergyCloseToInfinity}\\
\mathcal{E}_{\nearrow}^{(+)}[n;i,j] & =\mathcal{E}_{\nearrow}^{(-)}[n;i,j]\cdot\big(1+O(\varepsilon)\big)+O\Big(\varepsilon\frac{\varepsilon^{(j)}}{\sqrt{-\Lambda}}\Big).\label{eq:OutgoingEnergyCloseToInfinity}
\end{align}


\end{itemize}

In the case when $\widetilde{\mathcal{R}}_{i;j}^{(n)}\subset\mathcal{T}_{\varepsilon}^{+}$,
the relations (\ref{eq:IncreaseIngoingEnergyCloseToAxis})\textendash (\ref{eq:OutgoingEnergyCloseToInfinity})
also hold for $\widetilde{\mathcal{E}}_{\nwarrow}^{(\pm)}[n;i,j]$,
$\widetilde{\mathcal{E}}_{\nearrow}^{(\pm)}[n;i,j]$ in place of $\mathcal{E}_{\nwarrow}^{(\pm)}[n;i,j]$,
$\mathcal{E}_{\nearrow}^{(\pm)}[n;i,j]$.
\end{prop}
\begin{proof}
For the purpose of simplifying the expressions appearing in the proof
of Proposition \ref{prop:EnergyChangeInteraction}, let us introduce
the shorthand notation 
\begin{align}
v_{n;i,j}^{(\pm)} & =\begin{cases}
v_{\varepsilon,i}^{(n)}\pm h_{\varepsilon,i}, & \text{ if }i>j,\\
v_{\varepsilon,i}^{(n+1)}\pm h_{\varepsilon,i} & \text{ if }i<j,
\end{cases}\label{eq:v+-}\\
u_{n;i,j}^{(\pm)} & =v_{\varepsilon,j}^{(n)}\pm h_{\varepsilon,j}.\label{eq:u+-}
\end{align}
Note that, with this notation, 
\[
\mathcal{R}_{i;j}^{(n)}=[u_{n;i,j}^{(-)},u_{n;i,j}^{(+)}]\times[v_{n;i,j}^{(-)},v_{n;i,j}^{(+)}].
\]

Let us introduce the following energy densities: On $\{v_{n;i,j}^{(-)}\le v\le v_{n;i,j}^{(+)}\}$,
we will set 
\begin{equation}
E_{\nwarrow}[n;i,j]\doteq2\pi\frac{1-\frac{2m}{r}}{\partial_{v}r}r^{2}\cdot a_{\varepsilon i}T_{vv}[f_{\varepsilon i}],\label{eq:IngoingMainEnergyDensity}
\end{equation}
while on $\{u_{n;i,j}^{(-)}\le u\le u_{n;i,j}^{(+)}\}$ we will set
\begin{equation}
E_{\nearrow}[n;i,j]\doteq2\pi\frac{1-\frac{2m}{r}}{-\partial_{u}r}r^{2}\cdot a_{\varepsilon j}T_{uu}[f_{\varepsilon j}].\label{eq:OutgoingMainEnergyDensity}
\end{equation}
We will also define the following energy-related quantities by integrating
$E_{\nwarrow}[n;i,j]$ and $E_{\nearrow}[n;i,j]$ in directions transverse
to the corresponding index arrow: On $\{v_{n;i,j}^{(-)}\le v\le v_{n;i,j}^{(+)}\}$,
we will define
\begin{equation}
\mathcal{E}_{\nwarrow}[n;i,j](u,v)\doteq\int_{v_{n;i,j}^{(-)}}^{v}E_{\nwarrow}[n;i.j](u,\bar{v})\,d\bar{v},\label{eq:IngoingEnergyForBootstrap}
\end{equation}
while on $\{u_{n;i,j}^{(-)}\le u\le u_{n;i,j}^{(+)}\}$ we will define
\begin{equation}
\mathcal{E}_{\nearrow}[n;i,j](u,v)\doteq\int_{u_{n;i,j}^{(-)}}^{u}E_{\nearrow}[n;i.j](\bar{u},v)\,d\bar{u}.\label{eq:OutgoingEnergyForBootstrap}
\end{equation}

In view of the fact that, among all the $f_{\varepsilon k}$'s, only
$f_{\varepsilon i}$ and $f_{\varepsilon j}$ are supported on $\mathcal{R}_{i;j}^{(n)}$,
the expression (\ref{eq:LinearCombinationVlasovFields}) for $f_{\varepsilon}$
implies that 
\begin{equation}
T_{\mu\nu}[f_{\varepsilon}]|_{\mathcal{R}_{i;j}^{(n)}}=a_{\varepsilon i}T_{\mu\nu}[f_{\varepsilon i}]|_{\mathcal{R}_{i;j}^{(n)}}+a_{\varepsilon j}T_{\mu\nu}[f_{\varepsilon j}]|_{\mathcal{R}_{i;j}^{(n)}}.\label{eq:OnlyTwoComponents}
\end{equation}

Notice that, in view of the fact that $T_{\mu\nu}[f_{\varepsilon i}]$
and $T_{\mu\nu}[f_{\varepsilon j}]$ are supported on $\mathcal{V}_{i}$
and $\mathcal{V}_{j}$, respectively (and hence vanish to infinite
order on $v=v_{n;i,j}^{(\pm)}$ and $u=u_{n;i,j}^{(\pm)}$, respectively),
the relation (\ref{eq:OnlyTwoComponents}) implies (in view of the
relations (\ref{eq:TildeUMaza})\textendash (\ref{eq:TildeVMaza})
for $\tilde{m}$, the definition (\ref{eq:IngoingMainEnergyDensity})\textendash (\ref{eq:OutgoingMainEnergyDensity})
of $E_{\nwarrow}$, $E_{\nearrow}$ and the bounds (\ref{eq:EstimateGeometryLikeC0})
and (\ref{eq:UpperBoundTuvFi}) on $\partial r$, $T_{uv}$) that:
\begin{align}
\tilde{m}\big(u,v_{n;i,j}^{(+)}\big)-\tilde{m}\big(u,v_{n;i,j}^{(-)}\big) & =\mathcal{E}_{\nwarrow}[n;i,j](u,\text{ }v_{n;i,j}^{(+)})+\mathfrak{Err}_{i,j\nwarrow}^{(n)}(u),\text{ for any }u\in[u_{n;i,j}^{(-)},u_{n;i,j}^{(+)}],\label{eq:FormulaForFinalEnergiesIntermediateFromDensities-1}\\
\tilde{m}(u_{n;i,j}^{(-)},v)-\tilde{m}(u_{n;i,j}^{(+)},v) & =\mathcal{E}_{\nearrow}[n;i,j](u_{n;i,j}^{(+)},\text{ }v)+\mathfrak{Err}_{i,j\nearrow}^{(n)}(v),\text{ for any }v\in[v_{n;i,j}^{(-)},v_{n;i,j}^{(+)}],\nonumber 
\end{align}
where 
\begin{align}
|\mathfrak{Err}_{i,j\nwarrow}^{(n)}(u)| & \le\exp\big(\exp(2\sigma_{\varepsilon}^{-5})\big)\cdot\frac{(\varepsilon^{(i)})^{2}}{(-\Lambda)(\inf_{\mathcal{R}_{i;j}^{(n)}}r)^{2}}\cdot\frac{\varepsilon^{(i)}}{\sqrt{-\Lambda}},\label{eq:ErrorTermForEnergyFinal}\\
|\mathfrak{Err}_{i,j\nearrow}^{(n)}(v)| & \le\exp\big(\exp(2\sigma_{\varepsilon}^{-5})\big)\cdot\frac{(\varepsilon^{(j)})^{2}}{(-\Lambda)(\inf_{\mathcal{R}_{i;j}^{(n)}}r)^{2}}\cdot\frac{\varepsilon^{(j)}}{\sqrt{-\Lambda}}.\nonumber 
\end{align}
In particular, in view of the definition (\ref{eq:IngoingEnergyBefore})\textendash (\ref{eq:OutgoingEnergyAfter})
of $\mathcal{E}_{\nwarrow}^{(\pm)}$, $\mathcal{E}_{\nearrow}^{(\pm)}$:
\begin{align}
\mathcal{E}_{\nwarrow}^{(\pm)}[n;i,j] & =\mathcal{E}_{\nwarrow}[n;i,j](u_{n;i,j}^{(\pm)},\text{ }v_{n;i,j}^{(+)})+\mathfrak{Err}_{i,j\nwarrow}^{(n)}(u_{n;i,j}^{(\pm)}),\label{eq:FormulaForFinalEnergiesFromDensities}\\
\mathcal{E}_{\nearrow}^{(\pm)}[n;i,j] & =\mathcal{E}_{\nearrow}[n;i,j](u_{n;i,j}^{(+)},\text{ }v_{n;i,j}^{(\pm)})+\mathfrak{Err}_{i,j\nearrow}^{(n)}(v_{n;i,j}^{(\pm)}).\nonumber 
\end{align}

The conservation of energy relation (\ref{eq:ConservedEnergy}) for
the Vlasov field $f_{\varepsilon i}$ reads (in view of the relation
$T_{uv}[f_{\varepsilon i}]=\frac{1}{4}\Omega^{2}g^{AB}T_{AB}[f_{\varepsilon i}]$
holding for all massless Vlasov fields):
\begin{align}
\partial_{u}(r^{2}T_{vv}[f_{\varepsilon i}]) & =-\partial_{v}(r^{2}T_{uv}[f_{\varepsilon i}])+\Big(\partial_{v}\log(\Omega^{2})-2\frac{\partial_{v}r}{r}\Big)(r^{2}T_{uv}[f_{\varepsilon i}]),\label{eq:ConservationOfEnergyDu}\\
\partial_{v}(r^{2}T_{uu}[f_{\varepsilon i}]) & =-\partial_{u}(r^{2}T_{uv}[f_{\varepsilon i}])+\Big(\partial_{u}\log(\Omega^{2})-2\frac{\partial_{u}r}{r}\Big)(r^{2}T_{uv}[f_{\varepsilon i}])\label{eq:ConservationOfEnergyDv}
\end{align}
 and similarly for $f_{\varepsilon j}$. Furthermore, equations (\ref{eq:EquationROutside})
and (\ref{eq:EquationROutside}) readily yield: 
\begin{align}
\partial_{u}\Big(\frac{1-\frac{2m}{r}}{\partial_{v}r}\Big) & =\Big(4\pi\frac{rT_{uu}[f_{\varepsilon}]}{-\partial_{u}r}\Big)\cdot\frac{1-\frac{2m}{r}}{\partial_{v}r},\label{eq:KappaExponentialDu}\\
\partial_{v}\Big(\frac{1-\frac{2m}{r}}{-\partial_{u}r}\Big) & =\Big(-4\pi\frac{rT_{vv}[f_{\varepsilon}]}{\partial_{v}r}\Big)\cdot\frac{1-\frac{2m}{r}}{-\partial_{u}r}.\label{eq:KappaBarExponentialDv}
\end{align}
Differentiating (\ref{eq:IngoingMainEnergyDensity}) with resepect
to $\partial_{u}$ and using (\ref{eq:ConservationOfEnergyDu}) and
(\ref{eq:KappaExponentialDu}), we obtain: 
\begin{equation}
\partial_{u}E_{\nwarrow}[n;i.j]=\Big(4\pi\frac{rT_{uu}[f_{\varepsilon}]}{-\partial_{u}r}\Big)\cdot E_{\nwarrow}[n;i.j]-a_{\varepsilon i}\Big\{2\pi\frac{1-\frac{2m}{r}}{\partial_{v}r}\partial_{v}(r^{2}T_{uv}[f_{\varepsilon i}])+2\pi\frac{1-\frac{2m}{r}}{\partial_{v}r}\Big(\partial_{v}\log(\Omega^{2})-2\frac{\partial_{v}r}{r}\Big)(r^{2}T_{uv}[f_{\varepsilon i}])\Big\}.\label{eq:ODEforEnergyFlux}
\end{equation}

\medskip{}

\noindent \emph{Remark.} Notice that the coefficient of $E_{\nwarrow}[n;i.j]$
in the right hand side of (\ref{eq:ODEforEnergyFlux}) is strictly
positive. It is the sign of this coefficient that will lead to the
increase of the energy quantity $\mathcal{E}_{\nwarrow}[n;i,j]$ as
quantified by (\ref{eq:IncreaseIngoingEnergyCloseToAxis}).

\medskip{}

\noindent From (\ref{eq:ODEforEnergyFlux}), we obtain the following
explicit formula for $E_{\nwarrow}[n;i.j](u,v)$ for $(u,v)\in\mathcal{R}_{i;j}^{(n)}$:
\begin{equation}
E_{\nwarrow}[n;i.j](u,v)=\exp\Big(\int_{u_{n;i,j}^{(-)}}^{u}4\pi\frac{rT_{uu}[f_{\varepsilon}]}{-\partial_{u}r}(\bar{u},v)\,d\bar{u}\Big)\cdot E_{\nwarrow}[n;i.j](u_{n;i,j}^{(-)},v)-a_{\varepsilon i}\mathfrak{Err}_{\nwarrow}[n;i.j](u,v),\label{eq:EnergyDensityIncreaseIngoing}
\end{equation}
where
\begin{align}
\mathfrak{Err}_{\nwarrow}[n;i.j](u,v)\doteq & -\int_{u_{n;i,j}^{(-)}}^{u}\exp\Big(\int_{\bar{u}}^{u}4\pi\frac{rT_{uu}[f_{\varepsilon}]}{-\partial_{u}r}(\hat{u},v)\,d\hat{u}\Big)\times\label{eq:ErrorTermForEnergyChangeIngoing}\\
 & \hphantom{-e^{\int_{v_{\varepsilon,j}^{(n)}-h_{\varepsilon,j}}^{u}4\pi}}\times\Bigg\{2\pi\frac{1-\frac{2m}{r}}{\partial_{v}r}\partial_{v}(r^{2}T_{uv}[f_{\varepsilon i}])+2\pi\frac{1-\frac{2m}{r}}{\partial_{v}r}\Big(\partial_{v}\log(\Omega^{2})-2\frac{\partial_{v}r}{r}\Big)(r^{2}T_{uv}[f_{\varepsilon i}])\Bigg\}(\bar{u},v)\,d\bar{u}.\nonumber 
\end{align}
Similarly, differentiating (\ref{eq:OutgoingMainEnergyDensity}) with
respect to $\partial_{v}$, we infer the following formula for $E_{\nearrow}[n;i.j](u,v)$
on $\mathcal{R}_{i;j}^{(n)}$: 
\begin{equation}
E_{\nearrow}[n;i.j](u,v)=\exp\Big(-\int_{v_{n;i,j}^{(-)}}^{v}4\pi\frac{rT_{vv}[f_{\varepsilon}]}{\partial_{v}r}(u,\bar{v})\,d\bar{v}\Big)\cdot E_{\nearrow}[n;i.j](u,v_{n;i,j}^{(-)})-a_{\varepsilon j}\mathfrak{Err}_{\nearrow}[n;i.j](u,v),\label{eq:EnergyDensityIncreaseOutgoing}
\end{equation}
where
\begin{align}
\mathfrak{Err}_{\nearrow}[n;i.j](u,v)\doteq & -\int_{v_{n;i,j}^{(-)}}^{v}\exp\Big(-\int_{\bar{v}}^{v}4\pi\frac{rT_{vv}[f_{\varepsilon}]}{\partial_{v}r}(u,\hat{v})\,d\hat{v}\Big)\times\label{eq:ErrorTermForEnergyChangeOutgoing}\\
 & \hphantom{-e^{\int_{v_{\varepsilon,j}^{(n)}-h_{\varepsilon,j}}^{u}4\pi}}\times\Bigg\{2\pi\frac{1-\frac{2m}{r}}{-\partial_{u}r}\partial_{u}(r^{2}T_{uv}[f_{\varepsilon j}])+2\pi\frac{1-\frac{2m}{r}}{-\partial_{u}r}\Big(\partial_{u}\log(\Omega^{2})-2\frac{\partial_{u}r}{r}\Big)(r^{2}T_{uv}[f_{\varepsilon j}])\Bigg\}(u,\bar{v})\,d\bar{v}.\nonumber 
\end{align}

In view of the relation (\ref{eq:OnlyTwoComponents}) for $f_{\varepsilon}$,
$f_{\varepsilon i}$ and $f_{\varepsilon j}$ on $\mathcal{R}_{i;j}^{(n)}$
and the definition (\ref{eq:IngoingMainEnergyDensity})\textendash (\ref{eq:OutgoingMainEnergyDensity})
of $E_{\nwarrow}$, $E_{\nearrow}$, we have for any $(u,v)\in[u_{n;i,j}^{(-)},u_{n;i,j}^{(+)}]\times[v_{n;i,j}^{(-)},v_{n;i,j}^{(+)}]$:
\begin{equation}
\int_{u_{n;i,j}^{(-)}}^{u}4\pi\frac{rT_{uu}[f_{\varepsilon}]}{-\partial_{u}r}(\bar{u},v)\,d\bar{u}=\int_{u_{n;i,j}^{(-)}}^{u}\Bigg\{\frac{2E_{\nearrow}[n;i,j]}{r}(\bar{u},v)\cdot\Big(1-\frac{2m}{r}(\bar{u},v)\Big)^{-1}+\frac{4\pi a_{\varepsilon i}}{r}\cdot\frac{r^{2}T_{uu}[f_{\varepsilon i}]}{-\partial_{u}r}(\bar{u},v)\Bigg\}\,d\bar{u}\label{eq:CoefficientExpressionForIngoing}
\end{equation}
and 
\begin{equation}
\int_{v_{n;i,j}^{(-)}}^{v}4\pi\frac{rT_{vv}[f_{\varepsilon}]}{\partial_{v}r}(u,\bar{v})\,d\bar{v}=\int_{v_{n;i,j}^{(-)}}^{v}\Bigg\{\frac{2E_{\nwarrow}[n;i,j]}{r}(u,\bar{v})\cdot\Big(1-\frac{2m}{r}(u,\bar{v})\Big)^{-1}+\frac{4\pi a_{\varepsilon j}}{r}\cdot\frac{r^{2}T_{vv}[f_{\varepsilon j}]}{\partial_{v}r}(u,\bar{v})\Bigg\}\,d\bar{v}.\label{eq:CoefficientExpressionForOutgoing}
\end{equation}
Using (\ref{eq:UpperBoundTuuTvvFi}) for $T_{uu}[f_{\varepsilon i}]$,
$T_{vv}[f_{\varepsilon j}]$, recalling that $\mathcal{R}_{i;j}^{(n)}=\mathcal{V}_{i\nwarrow}^{(n)}\cap\mathcal{V}_{j\nearrow}^{(n)}$,
as well as the bound (\ref{eq:EstimateGeometryLikeC0}) and the assumption
$a_{\varepsilon k}\in(0,\sigma_{\varepsilon}]$, the relations (\ref{eq:CoefficientExpressionForIngoing})
and (\ref{eq:CoefficientExpressionForOutgoing}) yield: 
\begin{equation}
\int_{u_{n;i,j}^{(-)}}^{u}4\pi\frac{rT_{uu}[f_{\varepsilon}]}{-\partial_{u}r}(\bar{u},v)\,d\bar{u}=\int_{u_{n;i,j}^{(-)}}^{u}\Bigg\{\frac{2E_{\nearrow}[n;i,j]}{r\cdot\big(1-\frac{1}{3}\Lambda r^{2}-O(\mu_{i;j})\big)}(\bar{u},v)+O\Big((-\Lambda)^{-2}\exp\big(\exp(\sigma_{\varepsilon}^{-6})\big)\frac{(\varepsilon^{(i)})^{4}}{r^{5}(\bar{u},v)}\Big)\Bigg\}\,d\bar{u}\label{eq:CoefficientExpressionForIngoing-1}
\end{equation}
and 
\begin{equation}
\int_{v_{n;i,j}^{(-)}}^{v}4\pi\frac{rT_{vv}[f_{\varepsilon}]}{\partial_{v}r}(u,\bar{v})\,d\bar{v}=\int_{v_{n;i,j}^{(-)}}^{v}\Bigg\{\frac{2E_{\nwarrow}[n;i,j]}{r\cdot\big(1-\frac{1}{3}\Lambda r^{2}-O(\mu_{i;j})\big)}(u,\bar{v})+O\Big((-\Lambda)^{-2}\exp\big(\exp(\sigma_{\varepsilon}^{-6})\big)\frac{(\varepsilon^{(j)})^{4}}{r^{5}(u,\bar{v})}\Big)\Bigg\}\,d\bar{v},\label{eq:CoefficientExpressionForOutgoing-1}
\end{equation}
where 
\begin{equation}
\mu_{i;j}\doteq\sup_{\mathcal{R}_{i;j}^{(n)}}\frac{2\tilde{m}}{r}.
\end{equation}
Note that 
\begin{equation}
\mu_{i;j}=O(\eta_{0}),\label{eq:UpperBoundMuIntersectionTrivial}
\end{equation}
as a trivial consequence of (\ref{eq:UpperBoundMuBootstrapDomain}).

For the rest of the proof of (\ref{eq:IncreaseIngoingEnergyCloseToAxis})\textendash (\ref{eq:OutgoingEnergyCloseToInfinity}),
we will consider the cases $i>j$ and $i<j$ separately.

\medskip{}

\noindent \emph{The case $i>j$: Proof of (\ref{eq:IncreaseIngoingEnergyCloseToAxis})
and (\ref{eq:DecreaseOutgoingEnergyCloseToAxis}).} Integrating (\ref{eq:TildeVMaza})
in $v$ from the axis $\gamma_{\mathcal{Z}_{\varepsilon}}$ up to
$\mathcal{R}_{i;j}^{(n)}$, using the fact that, among all the $f_{\varepsilon k}$'s,
only $f_{\varepsilon\bar{j}}$, $j\le\bar{j}\le i$ are supported
on $\{u_{n;i,j}^{(-)}\le u\le u_{n;i,j}^{(+)}\}\cap\{v\le v_{n;i,j}^{(+)}\}$,
we can readily estimate (using (\ref{eq:EstimateGeometryLikeC0})
and (\ref{eq:UpperBoundTuuTvvFi}), (\ref{eq:UpperBoundTuvFi})):
\begin{equation}
\sup_{\mathcal{R}_{i;j}^{(n)}}\tilde{m}\le\exp(\exp(\sigma_{\varepsilon}^{-6}))\frac{\varepsilon^{(j)}}{\sqrt{-\Lambda}}.\label{eq:UpperBoundHawkingMassRnij}
\end{equation}
From (\ref{eq:UpperBoundHawkingMassRnij}) and the bound (\ref{eq:BoundsRIntersectionRegioni>j})
for $r$ on $\mathcal{R}_{i;j}^{(n)}$, we immediately infer that:
\begin{equation}
\mu_{i;j}=\sup_{\mathcal{R}_{i;j}^{(n)}}\frac{2\tilde{m}}{r}\le\exp(\exp(\sigma_{\varepsilon}^{-7}))\rho_{\varepsilon}\le\rho_{\varepsilon}^{\frac{3}{4}}.\label{eq:NonTrivialUpperBoundMuij}
\end{equation}

Using the bounds (\ref{eq:NonTrivialUpperBoundMuij}), (\ref{eq:BoundsRIntersectionRegioni>j})
and (\ref{eq:ComparisonRIntersectionRegioni>j}) on $\mathcal{R}_{i;j}^{(n)}$
as well as (\ref{eq:HierarchyOfParameters}) and (\ref{eq:RecursiveEi}),
the relation (\ref{eq:CoefficientExpressionForIngoing-1}) yields:
\begin{align}
\int_{u_{n;i,j}^{(-)}}^{u}4\pi\frac{rT_{uu}[f_{\varepsilon}]}{-\partial_{u}r}(\bar{u},v)\,d\bar{u} & =\int_{v_{\varepsilon,j}^{(n)}-h_{\varepsilon,j}}^{u}\Bigg\{\frac{2E_{\nearrow}[n;i,j](\bar{u},v)}{r_{n;i,j}}\cdot\big(1+O(\rho_{\varepsilon}^{\frac{3}{4}})\big)+O\Big(\sqrt{-\Lambda}\rho_{\varepsilon}\exp(e^{\sigma_{\varepsilon}^{-8}})\frac{(\varepsilon^{(i)})^{4}}{(\varepsilon^{(j)})^{5}}\Big)\Bigg\}\,d\bar{u}=\label{eq:CoefficientExpressionForIngoing-1-1}\\
 & =\frac{2\mathcal{E}_{\nearrow}[n;i,j](u,v)}{r_{n;i,j}}\cdot\big(1+O(\rho_{\varepsilon}^{\frac{3}{4}})\big)+O\Big(\sqrt{-\Lambda}\rho_{\varepsilon}\exp(e^{\sigma_{\varepsilon}^{-8}})\frac{(\varepsilon^{(i)})^{4}}{(\varepsilon^{(j)})^{5}}h_{\varepsilon,j}\Big)=\nonumber \\
 & =\frac{2\mathcal{E}_{\nearrow}[n;i,j](u,v)}{r_{n;i,j}}\cdot\big(1+O(\rho_{\varepsilon}^{\frac{3}{4}})\big)+O\Big(\sqrt{-\Lambda}\rho_{\varepsilon}^{2}\exp(e^{\sigma_{\varepsilon}^{-8}})\frac{(\varepsilon^{(i)})^{4}}{(\varepsilon^{(j)})^{4}}\Big)=\nonumber \\
 & =\frac{2\mathcal{E}_{\nearrow}[n;i,j](u,v)}{r_{n;i,j}}\cdot\big(1+O(\rho_{\varepsilon}^{\frac{3}{4}})\big)+O(\varepsilon^{4})\nonumber 
\end{align}
(where we have used the bound $\frac{\varepsilon^{(i)}}{\varepsilon^{(j)}}\le\varepsilon$
when $i>j$). Similarly, (\ref{eq:CoefficientExpressionForOutgoing-1})
yields:
\begin{align}
\int_{v_{n;i,j}^{(-)}}^{v}4\pi\frac{rT_{vv}[f_{\varepsilon}]}{\partial_{v}r}(u,\bar{v})\,d\bar{v} & =\frac{2\mathcal{E}_{\nwarrow}[n;i,j](u,v)}{r_{n;i,j}}\cdot\big(1+O(\rho_{\varepsilon}^{\frac{3}{4}})\big)+O(\sqrt{-\Lambda}\rho_{\varepsilon}\exp(e^{\sigma_{\varepsilon}^{-8}})\frac{(\varepsilon^{(j)})^{4}}{(\varepsilon^{(j)})^{5}}h_{\varepsilon,i})=,\label{eq:CoefficientExpressionForOutgoing-1-1}\\
 & =\frac{2\mathcal{E}_{\nwarrow}[n;i,j](u,v)}{r_{n;i,j}}\cdot\big(1+O(\rho_{\varepsilon}^{\frac{3}{4}})\big)+O(\varepsilon).\nonumber 
\end{align}
where $r_{n;i,j}$ is defined by (\ref{eq:DefinitionRnij}).

Substituting (\ref{eq:CoefficientExpressionForIngoing-1-1}) in (\ref{eq:EnergyDensityIncreaseIngoing})
and integrating in $v$ over $[v_{n;i,j}^{(-)},v]$, we obtain:
\begin{equation}
\mathcal{E}_{\nwarrow}[n;i.j](u,v)=\int_{v_{n;i,j}^{(-)}}^{v}\Bigg\{\exp\Big(\frac{2\mathcal{E}_{\nearrow}[n;i,j](u,\bar{v})}{r_{n;i,j}}\cdot\big(1+O(\rho_{\varepsilon}^{\frac{3}{4}})\big)+O(\varepsilon)\Big)\cdot E_{\nwarrow}[n;i.j](u_{n;i,j}^{(-)},\bar{v})\Bigg\}\,d\bar{v}-a_{\varepsilon i}\int_{v_{n;i,j}^{(-)}}^{v}\mathfrak{Err}_{\nwarrow}[n;i.j](u,\bar{v})\,d\bar{v}.\label{eq:EnergyIntegralIncreaseIngoing}
\end{equation}
Integrating by parts in $\partial_{v}$ for the term $\partial_{v}(r^{2}T_{uv})$
in $\int_{v_{n;i,j}^{(-)}}^{v}\mathfrak{Err}_{\nwarrow}[n;i.j](u,\bar{v})\,d\bar{v}$
(see the expression (\ref{eq:ErrorTermForEnergyChangeIngoing})),
and using the fact that $T_{uv}[f_{\varepsilon i}]$ is supported
on $\mathcal{V}_{i}$ (and hence vanishes to infinite order on $v=v_{n;i,j}^{(\pm)}$),
we calculate
\begin{align}
\Big|\int_{v_{n;i,j}^{(-)}}^{v} & \mathfrak{Err}_{\nwarrow}[n;i.j](u,\bar{v})\,d\bar{v}\Big|=\label{eq:ErrorFinalRelation}\\
 & =\Bigg|\int_{v_{\varepsilon,i}^{(n)}-h_{\varepsilon,i}}^{v}\int_{v_{\varepsilon,j}^{(n)}-h_{\varepsilon,j}}^{u}\Bigg[\exp\Big(\int_{\bar{u}}^{u}4\pi\frac{rT_{uu}[f_{\varepsilon}]}{-\partial_{u}r}(\hat{u},\bar{v})\,d\hat{u}\Big)\times\nonumber \\
 & \hphantom{-e^{\int_{v_{\varepsilon,j}^{(n)}-h_{\varepsilon,j}}^{u}4\pi}}\times\Bigg\{2\pi\frac{1-\frac{2m}{r}}{\partial_{v}r}\partial_{v}(r^{2}T_{uv}[f_{\varepsilon i}])+2\pi\frac{1-\frac{2m}{r}}{\partial_{v}r}\Big(\partial_{v}\log(\Omega^{2})-2\frac{\partial_{v}r}{r}\Big)(r^{2}T_{uv}[f_{\varepsilon i}])\Bigg\}\Bigg](\bar{u},\bar{v})\,d\bar{u}d\bar{v}\Bigg|=\nonumber \\
 & =\Bigg|\int_{v_{\varepsilon,i}^{(n)}-h_{\varepsilon,i}}^{v}\int_{v_{\varepsilon,j}^{(n)}-h_{\varepsilon,j}}^{u}\Bigg[\exp\Big(\int_{\bar{u}}^{u}4\pi\frac{rT_{uu}[f_{\varepsilon}]}{-\partial_{u}r}(\hat{u},\bar{v})\,d\hat{u}\Big)\cdot2\pi\frac{1-\frac{2m}{r}}{\partial_{v}r}r^{2}T_{uv}[f_{\varepsilon i}]\times\nonumber \\
 & \hphantom{-e^{\int_{v_{\varepsilon,j}^{(n)}-h_{\varepsilon,j}}^{u}4\pi}}\times\Bigg\{-\int_{\bar{u}}^{u}\partial_{v}\Big(\frac{4\pi}{r(-\partial_{u}r)}\Big)\cdot r^{2}T_{uu}[f_{\varepsilon}](\hat{u},\bar{v})\,d\hat{u}-\int_{\bar{u}}^{u}\frac{4\pi}{r}\frac{1}{-\partial_{u}r}\partial_{v}(r^{2}T_{uu}[f_{\varepsilon}])(\hat{u},\bar{v})\,d\hat{u}-\nonumber \\
 & \hphantom{-e^{\int_{v_{\varepsilon,j}^{(n)}-h_{\varepsilon,j}}^{u}4\pi}\times\Bigg\{}-\partial_{v}\log\Big(\frac{1-\frac{2m}{r}}{\partial_{v}r}\Big)+\Big(\partial_{v}\log(\Omega^{2})-2\frac{\partial_{v}r}{r}\Big)\Bigg\}\Bigg](\bar{u},\bar{v})\,d\bar{u}d\bar{v}\Bigg|.\nonumber 
\end{align}
 We will estimate the right hand side of (\ref{eq:ErrorFinalRelation})
in a number of steps:

\begin{itemize}

\item Using the bounds (\ref{eq:UpperBoundMuBootstrapDomain}), (\ref{eq:EstimateGeometryLikeC0}),
(\ref{eq:UpperBoundTuuTvvFi}), (\ref{eq:OnlyTwoComponents}), (\ref{eq:BoundsRIntersectionRegioni>j})
and (\ref{eq:ComparisonRIntersectionRegioni>j}), we can estimate:
\begin{align}
\sup_{(u,\bar{v})\in\mathcal{R}_{i;j}^{(n)},\text{ }\bar{u}\in[v_{\varepsilon,j}^{(n)}-h_{\varepsilon,j},u]}\Big|\int_{\bar{u}}^{u}4\pi\frac{rT_{uu}[f_{\varepsilon}]}{-\partial_{u}r}(\hat{u},\bar{v})\,d\hat{u}\Big| & \le\int_{v_{\varepsilon,j}^{(n)}-h_{\varepsilon,j}}^{v_{\varepsilon,j}^{(n)}+h_{\varepsilon,j}}\exp\big(\exp(4\sigma_{\varepsilon}^{-5})\big)\cdot\frac{1}{r_{n;i,j}}\,du\le\label{eq:BoundExponentialFactorSummand}\\
 & \le\exp\big(\exp(\sigma_{\varepsilon}^{-6})\big)\cdot\frac{h_{\varepsilon,j}}{r_{n;i,j}}\le\nonumber \\
 & \le1.\nonumber 
\end{align}

\item Using the bounds (\ref{eq:EstimateGeometryLikeC0}), (\ref{eq:UpperBoundTuvFi}),
(\ref{eq:BoundsRIntersectionRegioni>j}) and (\ref{eq:ComparisonRIntersectionRegioni>j}),
we can estimate: 
\begin{equation}
2\pi\frac{1-\frac{2m}{r}}{\partial_{v}r}r^{2}T_{uv}[f_{\varepsilon i}]\le\exp\big(\exp(\sigma_{\varepsilon}^{-6})\big)\frac{(\varepsilon^{(i)})^{2}}{r_{n;i,j}^{2}}(-\Lambda)^{-1}.
\end{equation}

\item Using equation (\ref{eq:EquationROutside}) and the bounds
(\ref{eq:UpperBoundMuBootstrapDomain}), (\ref{eq:EstimateGeometryLikeC0}),
(\ref{eq:UpperBoundTuuTvvFi}), (\ref{eq:UpperBoundTuvFi}), (\ref{eq:BoundsRIntersectionRegioni>j})
and (\ref{eq:ComparisonRIntersectionRegioni>j}) (as well as the relation
(\ref{eq:OnlyTwoComponents}) between $f_{\varepsilon}$, $f_{\varepsilon i}$
and $f_{\varepsilon j}$ on $\mathcal{R}_{i;j}^{(n)}$ and the bound
$a_{\varepsilon k}\le1$), we can estimate 
\begin{align}
\sup_{(u,\bar{v})\in\mathcal{R}_{i;j}^{(n)},\text{ }\bar{u}\in[v_{\varepsilon,j}^{(n)}-h_{\varepsilon,j},u]}\Big|\int_{\bar{u}}^{u} & \partial_{v}\Big(\frac{4\pi}{r(-\partial_{u}r)}\Big)\cdot r^{2}T_{uu}[f_{\varepsilon}](\hat{u},\bar{v})\,d\hat{u}\Big|=\label{eq:BoundFirstSummand}\\
=\sup_{(u,\bar{v})\in\mathcal{R}_{i;j}^{(n)},\text{ }\bar{u}\in[v_{\varepsilon,j}^{(n)}-h_{\varepsilon,j},u]} & \Big|\int_{\bar{u}}^{u}4\pi\Big(-\frac{\partial_{v}r}{r^{2}(-\partial_{u}r)}+\frac{\partial_{u}\partial_{v}r}{r(-\partial_{u}r)^{2}}\Big)r^{2}T_{uu}[f_{\varepsilon}](\hat{u},\bar{v})\,d\hat{u}\Big|=\nonumber \\
=\sup_{(u,\bar{v})\in\mathcal{R}_{i;j}^{(n)},\text{ }\bar{u}\in[v_{\varepsilon,j}^{(n)}-h_{\varepsilon,j},u]} & \Big|\int_{\bar{u}}^{u}4\pi\Big(-\frac{\partial_{v}r}{r^{2}(-\partial_{u}r)}+\frac{\frac{2\tilde{m}}{r}-\frac{2}{3}\Lambda r^{2}}{r^{2}(1-\frac{2m}{r})}\frac{\partial_{v}r}{-\partial_{u}r}+\frac{4\pi r^{2}T_{uv}}{r^{2}(-\partial_{u}r)^{2}}\Big)r^{2}T_{uu}[f_{\varepsilon}](\hat{u},\bar{v})\,d\hat{u}\Big|\le\nonumber \\
 & \le\int_{v_{\varepsilon,j}^{(n)}-h_{\varepsilon,j}}^{v_{\varepsilon,j}^{(n)}+h_{\varepsilon,j}}\exp\big(\exp(\sigma_{\varepsilon}^{-9})\big)\frac{1}{r_{n;i,j}^{2}}\,du\le\nonumber \\
 & \le2\exp\big(\exp(\sigma_{\varepsilon}^{-9})\big)\frac{h_{\varepsilon,j}}{r_{n;i,j}^{2}}\le\nonumber \\
 & \le\exp\big(\exp(2\sigma_{\varepsilon}^{-9})\big)\rho_{\varepsilon}\frac{1}{r_{n;i,j}}\le\nonumber \\
 & \le\rho_{\varepsilon}^{\frac{1}{2}}\frac{1}{r_{n;i,j}}.\nonumber 
\end{align}

\item Using (\ref{eq:ConservationOfEnergyDv}) for $f_{\varepsilon}$
in place of $f_{\varepsilon i}$ to express $\partial_{v}(r^{2}T_{uu}[f_{\varepsilon}])$
in terms of $\partial_{u}(r^{2}T_{uv}[f_{\varepsilon}])$ and integrating
by parts in $\partial_{u}$, we calculate (in view of the bounds (\ref{eq:UpperBoundMuBootstrapDomain}),
(\ref{eq:EstimateGeometryLikeC0}), (\ref{eq:UpperBoundTuvFi}), (\ref{eq:BoundDOmegaOnDomainsSpecialCase})
for $\partial_{u}\Omega^{2}$ with $j$ in place of $i$, (\ref{eq:BoundDdrOnDomainsSpecialCase})
for $\partial_{u}^{2}r$ with $j$ in place of $i$, (\ref{eq:BoundsRIntersectionRegioni>j})
and (\ref{eq:ComparisonRIntersectionRegioni>j})): 
\begin{align}
 & \sup_{(u,\bar{v})\in\mathcal{R}_{i;j}^{(n)},\text{ }\bar{u}\in[v_{\varepsilon,j}^{(n)}-h_{\varepsilon,j},u]}\Big|\int_{\bar{u}}^{u}\frac{4\pi}{r}\frac{1}{-\partial_{u}r}\partial_{v}(r^{2}T_{uu}[f_{\varepsilon}])(\hat{u},\bar{v})\,d\hat{u}\Big|=\label{eq:BoundSecondSummand}\\
 & \hphantom{ll}=\sup_{(u,\bar{v})\in\mathcal{R}_{i;j}^{(n)},\text{ }\bar{u}\in[v_{\varepsilon,j}^{(n)}-h_{\varepsilon,j},u]}\Bigg|\int_{\bar{u}}^{u}\frac{4\pi}{r}\frac{1}{-\partial_{u}r}\Big(-\partial_{u}(r^{2}T_{uv}[f_{\varepsilon}])+\Big(\partial_{u}\log(\Omega^{2})-2\frac{\partial_{u}r}{r}\Big)(r^{2}T_{uv}[f_{\varepsilon}])\Big)(\hat{u},\bar{v})\,d\hat{u}\Bigg|=\nonumber \\
 & \hphantom{ll}=\sup_{(u,\bar{v})\in\mathcal{R}_{i;j}^{(n)},\text{ }\bar{u}\in[v_{\varepsilon,j}^{(n)}-h_{\varepsilon,j},u]}\Bigg|\int_{\bar{u}}^{u}4\pi\Big(\partial_{u}\big(\frac{1}{r(-\partial_{u}r)}\big)\cdot r^{2}T_{uv}[f_{\varepsilon}]+\nonumber \\
 & \hphantom{ll=\sup_{(u,\bar{v})\in\mathcal{R}_{i;j}^{(n)},\text{ }\bar{u}\in[v_{\varepsilon,j}^{(n)}-h_{\varepsilon,j},u]}\Big|\int_{\bar{u}}^{u}4\pi\Big(}+\frac{1}{r(-\partial_{u}r)}\Big(\partial_{u}\log(\Omega^{2})-2\frac{\partial_{u}r}{r}\Big)(r^{2}T_{uv}[f_{\varepsilon}])\Big)(\hat{u},\bar{v})\,d\hat{u}+\nonumber \\
 & \hphantom{ll=\sup_{(u,\bar{v})\in\mathcal{R}_{i;j}^{(n)},\text{ }\bar{u}\in[v_{\varepsilon,j}^{(n)}-h_{\varepsilon,j},u]}}+\frac{4\pi}{r}\frac{1}{-\partial_{u}r}\cdot r^{2}T_{uv}[f_{\varepsilon}](\bar{u},\bar{v})-\frac{4\pi}{r}\frac{1}{-\partial_{u}r}\cdot r^{2}T_{uv}[f_{\varepsilon}](u,\bar{v})\Bigg|\le\nonumber \\
 & \hphantom{ll}\le\int_{v_{\varepsilon,j}^{(n)}-h_{\varepsilon,j}}^{v_{\varepsilon,j}^{(n)}+h_{\varepsilon,j}}\exp\big(\exp(\sigma_{\varepsilon}^{-9})\big)\frac{1}{r_{n;i,j}}\big(\frac{\sqrt{-\Lambda}}{\varepsilon^{(j)}}+\frac{1}{r_{n;i,j}}\big)\frac{(\varepsilon^{(j)})^{2}}{(-\Lambda)r_{n;i,j}^{2}}\,du+\exp\big(\exp(\sigma_{\varepsilon}^{-9})\big)\frac{1}{r_{n;i,j}}\frac{(\varepsilon^{(j)})^{2}}{(-\Lambda)r_{n;i,j}^{2}}\le\nonumber \\
 & \hphantom{ll}\le\exp\big(\exp(2\sigma_{\varepsilon}^{-9})\big)\rho_{\varepsilon}^{2}\frac{1}{r_{n;i,j}}\le\nonumber \\
 & \hphantom{ll}\le\rho_{\varepsilon}\frac{1}{r_{n;i,j}}.\nonumber 
\end{align}

\item Using the the relation (\ref{eq:TildeVMaza}) for $\partial_{v}\tilde{m}$,
the estimate (\ref{eq:BoundDdrOnDomainsSpecialCase}) for $\partial_{v}^{2}r$,
as well as the bounds (\ref{eq:EstimateGeometryLikeC0}), (\ref{eq:UpperBoundTuuTvvFi}),
(\ref{eq:UpperBoundTuvFi}), (\ref{eq:UpperBoundMuBootstrapDomain}),
(\ref{eq:BoundsRIntersectionRegioni>j}) and (\ref{eq:ComparisonRIntersectionRegioni>j}),
we can estimate:
\begin{equation}
\sup_{\mathcal{R}_{i;j}^{(n)}\text{ }}\Big|\partial_{v}\log\Big(\frac{1-\frac{2m}{r}}{\partial_{v}r}\Big)\Big|\le\exp\big(\exp(\sigma_{\varepsilon}^{-6})\big)\Big(\frac{\sqrt{-\Lambda}}{\varepsilon^{(i)}}+\frac{1}{r_{n;i,j}}\Big).\label{eq:BoundThirdSummand}
\end{equation}

\item Using the estimate (\ref{eq:BoundDOmegaOnDomainsSpecialCase})
for $\partial_{v}\Omega^{2}$, as well as the bounds (\ref{eq:EstimateGeometryLikeC0}),
(\ref{eq:UpperBoundTuvFi}), (\ref{eq:BoundsRIntersectionRegioni>j})
and (\ref{eq:ComparisonRIntersectionRegioni>j}), we can estimate:
\begin{equation}
\sup_{\mathcal{R}_{i;j}^{(n)}\text{ }}\Big|\partial_{v}\log(\Omega^{2})-2\frac{\partial_{v}r}{r}\Big|\le\exp\big(\exp(\sigma_{\varepsilon}^{-6})\big)\Big(\frac{\sqrt{-\Lambda}}{\varepsilon^{(i)}}+\frac{1}{r_{n;i,j}}\Big).\label{eq:BoundFourthSummand}
\end{equation}

\end{itemize}

Using the estimates (\ref{eq:BoundExponentialFactorSummand})\textendash (\ref{eq:BoundFourthSummand})
(together with the relation of the parameters $\varepsilon$, $\rho_{\varepsilon}$,
$\delta_{\varepsilon}$ and $\sigma_{\varepsilon}$) to bound the
right hand side of (\ref{eq:ErrorFinalRelation}), we therefore obtain:
\begin{align}
\sup_{(u,v)\in\mathcal{R}_{i;j}^{(n)}}\Big|\int_{v_{n;i,j}^{(-)}}^{v}\mathfrak{Err}_{\nwarrow}[n;i.j](u,\bar{v})\,d\bar{v}\Big| & \le\int_{v_{\varepsilon,i}^{(n)}-h_{\varepsilon,i}}^{v_{\varepsilon,i}^{(n)}+h_{\varepsilon,i}}\int_{v_{\varepsilon,j}^{(n)}-h_{\varepsilon,j}}^{v_{\varepsilon,j}^{(n)}+h_{\varepsilon,j}}\frac{(\varepsilon^{(i)})^{2}}{r_{n;i,j}^{2}}\exp\big(\exp(\sigma_{\varepsilon}^{-6})\big)(-\Lambda)^{-1}\times\label{eq:FinalIngoingErrorEstimate}\\
 & \hphantom{\le\int_{v_{\varepsilon,i}^{(n)}-h_{\varepsilon,i}}^{v_{\varepsilon,i}^{(n)}+h_{\varepsilon,i}}}\times\Big(\exp\big(\exp(\sigma_{\varepsilon}^{-7})\big)\Big(\frac{\sqrt{-\Lambda}}{\varepsilon^{(i)}}+\frac{1}{r_{n;i,j}}\Big)+2\rho_{\varepsilon}^{\frac{1}{2}}\frac{1}{r_{n;i,j}}\Big)\,d\bar{u}d\bar{v}\le\nonumber \\
 & \le\exp\big(\exp(2\sigma_{\varepsilon}^{-7})\big)\rho_{\varepsilon}^{2}\frac{(\varepsilon^{(i)})^{2}}{(\varepsilon^{(j)})^{2}}\frac{\sqrt{-\Lambda}}{\varepsilon^{(i)}}h_{\varepsilon,i}h_{\varepsilon,j}\le\nonumber \\
 & \le\exp\big(\exp(\sigma_{\varepsilon}^{-8})\big)\rho_{\varepsilon}^{2}\frac{(\varepsilon^{(i)})^{2}}{(\varepsilon^{(j)})^{2}}\frac{\sqrt{-\Lambda}}{\varepsilon^{(i)}}\varepsilon^{(i)}\varepsilon^{(j)}(-\Lambda)^{-1}\le\nonumber \\
 & \le\rho_{\varepsilon}^{\frac{3}{2}}\varepsilon^{(i)}(-\Lambda)^{-\frac{1}{2}}\nonumber 
\end{align}
where, in the last step in (\ref{eq:FinalIngoingErrorEstimate}),
we have used the bound $\varepsilon^{(i)}<\varepsilon^{(j)}$ (holding
when $i>j$). Returning to (\ref{eq:EnergyIntegralIncreaseIngoing})
and using (\ref{eq:FinalIngoingErrorEstimate}) to estimate the last
term in the right hand side (using also the bound $a_{\varepsilon i}<\delta_{\varepsilon}\ll\exp\big(-\exp(\sigma_{\varepsilon}^{-8})\big)$),
we infer that, for any $(u,v)\in\mathcal{R}_{i;j}^{(n)}$: 
\begin{equation}
\mathcal{E}_{\nwarrow}[n;i.j](u,v)=\int_{v_{n;i,j}^{(-)}}^{v}\Bigg\{\exp\Big(\frac{2\mathcal{E}_{\nearrow}[n;i,j](u,\bar{v})}{r_{n;i,j}}\cdot\big(1+O(\rho_{\varepsilon}^{\frac{3}{4}})\big)+O(\varepsilon)\Big)\cdot E_{\nwarrow}[n;i.j](u_{n;i,j}^{(-)},\bar{v})\Bigg\}\,d\bar{v}+O\Big(\rho_{\varepsilon}^{\frac{3}{2}}\frac{\varepsilon^{(i)}}{\sqrt{-\Lambda}}\Big).\label{eq:AlmostThereIngoingEnergyChange}
\end{equation}
Similarly, substituting (\ref{eq:CoefficientExpressionForOutgoing-1-1})
in (\ref{eq:EnergyDensityIncreaseOutgoing}) and integrating in $u$
over $[u_{n;i,j}^{(-)},u]$, estimating $\int_{u_{n;i,j}^{(-)}}^{u}\mathfrak{Err}_{\nearrow}[n;i,j](\bar{u},v)\,d\bar{u}$
in the same way as we did for (\ref{eq:FinalIngoingErrorEstimate}),
we obtain for any $(u,v)\in\mathcal{R}_{i;j}^{(n)}$: 
\begin{equation}
\mathcal{E}_{\nearrow}[n;i.j](u,v)=\int_{u_{n;i,j}^{(-)}}^{u}\Bigg\{\exp\Big(-\frac{2\mathcal{E}_{\nwarrow}[n;i,j](\bar{u},v)}{r_{n;i,j}}\cdot\big(1+O(\rho_{\varepsilon}^{\frac{3}{4}})\big)+O(\varepsilon)\Big)\cdot E_{\nearrow}[n;i.j](\bar{u},v_{n;i,j}^{(-)})\Bigg\}\,d\bar{u}+O\Big(\rho_{\varepsilon}^{\frac{3}{2}}\frac{\varepsilon^{(j)}}{\sqrt{-\Lambda}}\Big).\label{eq:AlmostThereOutgoingEnergyChange}
\end{equation}

Using the bounds (\ref{eq:EstimateGeometryLikeC0}), (\ref{eq:UpperBoundTuuTvvFi}),
(\ref{eq:BoundsRIntersectionRegioni>j}) and (\ref{eq:ComparisonRIntersectionRegioni>j}),
we can trivially estimate 
\begin{equation}
\sup_{(u,v)\in\mathcal{R}_{i;j}^{(n)}}\frac{\mathcal{E}_{\nearrow}[n;i.j](u,v)}{r_{n;i,j}}\le\frac{\int_{v_{\varepsilon,j}^{(n)}-h_{\varepsilon,j}}^{v_{\varepsilon,j}^{(n)}+h_{\varepsilon,j}}\exp\big(\exp(\sigma_{\varepsilon}^{-6})\big)\,du}{r_{n;i,j}}\le2\exp\big(\exp(\sigma_{\varepsilon}^{-6})\big)\frac{h_{\varepsilon,j}}{r_{n;i,j}}\le\rho_{\varepsilon}^{\frac{3}{4}}.\label{eq:TrivialUpperBoundOutgoing}
\end{equation}
Dividing (\ref{eq:AlmostThereIngoingEnergyChange}) with $r_{n;i,j}$
and using (\ref{eq:TrivialUpperBoundOutgoing}) to estimate its right
hand side, making also use of (\ref{eq:EstimateGeometryLikeC0}) and
(\ref{eq:UpperBoundTuuTvvFi}) to estimate $E_{\nwarrow}[n;i.j](u_{n;i,j}^{(-)},\bar{v})$,
we obtain: 
\begin{align}
\sup_{(u,v)\in\mathcal{R}_{i;j}^{(n)}}\frac{\mathcal{E}_{\nwarrow}[n;i.j](u,v)}{r_{n;i,j}} & \le\frac{\int_{v_{\varepsilon,i}^{(n)}-h_{\varepsilon,i}}^{v_{\varepsilon,i}^{(n)}+h_{\varepsilon,i}}\Big\{\exp(O(\rho_{\varepsilon}^{\frac{3}{4}}))E_{\nwarrow}[n;i.j](u_{n;i,j}^{(-)},\bar{v})+O\Big(\rho_{\varepsilon}^{\frac{3}{2}}\frac{\varepsilon^{(i)}}{\sqrt{-\Lambda}}\Big)\Big\}\,d\bar{v}+}{r_{n;i,j}}\le\label{eq:SecondTrivialUpperBoundEnergy}\\
 & \le2\int_{v_{\varepsilon,i}^{(n)}-h_{\varepsilon,i}}^{v_{\varepsilon,i}^{(n)}+h_{\varepsilon,i}}\frac{\exp\big(\exp(\sigma_{\varepsilon}^{-6})\big)}{r_{n;i,j}}\,d\bar{v}\le\nonumber \\
 & \le4\exp\big(\exp(\sigma_{\varepsilon}^{-6})\big)\frac{h_{\varepsilon,i}}{r_{n;i,j}}\le\nonumber \\
 & \le\varepsilon,\nonumber 
\end{align}
where, in passing from the third to the fourth line in (\ref{eq:SecondTrivialUpperBoundEnergy}),
we made use of the bound $\frac{\varepsilon^{(i)}}{\varepsilon^{(j)}}\le\varepsilon$.

Returning to (\ref{eq:AlmostThereOutgoingEnergyChange}) and using
the bound (\ref{eq:SecondTrivialUpperBoundEnergy}), we infer: 
\begin{align}
\mathcal{E}_{\nearrow}[n;i.j](u,v) & =\int_{u_{n;i,j}^{(-)}}^{u}\Bigg\{\exp\Big(O(\varepsilon)\Big)\cdot E_{\nearrow}[n;i.j](\bar{u},v_{n;i,j}^{(-)})\Bigg\}\,d\bar{u}+O\Big(\rho_{\varepsilon}^{\frac{3}{2}}\frac{\varepsilon^{(j)}}{\sqrt{-\Lambda}}\Big)=\label{eq:FinallyThereOutgoingEnergyChange}\\
 & =(1+O(\varepsilon))\int_{u_{n;i,j}^{(-)}}^{u}E_{\nearrow}[n;i.j](\bar{u},v_{n;i,j}^{(-)})\,d\bar{u}+O\Big(\rho_{\varepsilon}^{\frac{3}{2}}\frac{\varepsilon^{(j)}}{\sqrt{-\Lambda}}\Big)=\nonumber \\
 & =(1+O(\varepsilon))\mathcal{E}_{\nearrow}[n;i.j](u,v_{n;i,j}^{(-)})+O\Big(\rho_{\varepsilon}^{\frac{3}{2}}\frac{\varepsilon^{(j)}}{\sqrt{-\Lambda}}\Big),\nonumber 
\end{align}
from which (\ref{eq:DecreaseOutgoingEnergyCloseToAxis}) follows by
setting $u=u_{n;i,j}^{(+)}$ and $v=v_{n;i,j}^{(+)}$ and using (\ref{eq:FormulaForFinalEnergiesFromDensities}).
Returning, now, to (\ref{eq:AlmostThereIngoingEnergyChange}) and
using (\ref{eq:FinallyThereOutgoingEnergyChange}) to estimate the
exponential in the right hand side, as well as the bound (\ref{eq:TrivialUpperBoundOutgoing})
to estimate the $\frac{2\mathcal{E}_{\nearrow}[n;i.j]}{r_{n;i,j}}\cdot O(\rho_{\varepsilon}^{\frac{3}{4}})$
error term, we obtain: 
\begin{align}
\mathcal{E}_{\nwarrow}[n;i.j](u,v) & =\int_{v_{n;i,j}^{(-)}}^{v}\Bigg\{\exp\Big(\frac{2\mathcal{E}_{\nearrow}[n;i.j](u,v_{n;i,j}^{(-)})}{r_{n;i,j}}\cdot\big(1+O(\rho_{\varepsilon}^{\frac{3}{4}})\big)+O(\rho_{\varepsilon}^{2})\Big)\cdot E_{\nwarrow}[n;i.j](u_{n;i,j}^{(-)},\bar{v})\Bigg\}\,d\bar{v}+O\Big(\rho_{\varepsilon}^{\frac{3}{2}}\frac{\varepsilon^{(i)}}{\sqrt{-\Lambda}}\Big)=\label{eq:FinallyThereIngoingEnergyChange}\\
 & =\exp\Big(\frac{2\mathcal{E}_{\nearrow}[n;i.j](u,v_{n;i,j}^{(-)})}{r_{n;i,j}}+O(\rho_{\varepsilon}^{\frac{3}{2}})\Big)\cdot\int_{v_{n;i,j}^{(-)}}^{v}E_{\nwarrow}[n;i.j](u_{n;i,j}^{(-)},\bar{v})\,d\bar{v}+O\Big(\rho_{\varepsilon}^{\frac{3}{2}}\frac{\varepsilon^{(i)}}{\sqrt{-\Lambda}}\Big)=\nonumber \\
 & =\exp\Big(\frac{2\mathcal{E}_{\nearrow}[n;i.j](u,v_{n;i,j}^{(-)})}{r_{n;i,j}}+O(\rho_{\varepsilon}^{\frac{3}{2}})\Big)\cdot\mathcal{E}_{\nwarrow}[n;i.j](u_{n;i,j}^{(-)},v)+O\Big(\rho_{\varepsilon}^{\frac{3}{2}}\frac{\varepsilon^{(i)}}{\sqrt{-\Lambda}}\Big),\nonumber 
\end{align}
from which (\ref{eq:IncreaseIngoingEnergyCloseToAxis}) follows by
setting $u=u_{n;i,j}^{(+)}$ and $v=v_{n;i,j}^{(+)}$ and using (\ref{eq:FormulaForFinalEnergiesFromDensities}). 

\medskip{}

\noindent \emph{Remark.} More generally, setting $u=u_{n;i,j}^{(+)}$
in (\ref{eq:FinallyThereOutgoingEnergyChange}), $v=v_{n;i,j}^{(+)}$
in (\ref{eq:FinallyThereIngoingEnergyChange}) and using (\ref{eq:FormulaForFinalEnergiesIntermediateFromDensities-1}),
we obtain
\begin{equation}
\tilde{m}(u_{n;i,j}^{(-)},v)-\tilde{m}(u_{n;i,j}^{(+)},v)=\mathcal{E}_{\nearrow}^{(-)}[n;i,j]\cdot\big(1+O(\varepsilon)\big)+O\Big(\rho_{\varepsilon}^{\frac{3}{2}}\frac{\varepsilon^{(j)}}{\sqrt{-\Lambda}}\Big)\text{ for all }v\in[v_{n;i,j}^{(-)},v_{n;i,j}^{(+)}]\label{eq:GeneralBoundMassDifferenceOutfoinfi>j}
\end{equation}
and 
\begin{equation}
\tilde{m}\big(u,v_{n;i,j}^{(+)}\big)-\tilde{m}\big(u,v_{n;i,j}^{(-)}\big)\le\mathcal{E}_{\nwarrow}^{(-)}[n;i,j]\cdot\exp\Big(\frac{2\mathcal{E}_{\nearrow}^{(-)}[n;i,j]}{r_{n;i.j}}+O(\rho_{\varepsilon}^{\frac{3}{2}})\Big)+O\Big(\rho_{\varepsilon}^{\frac{3}{2}}\frac{\varepsilon^{(i)}}{\sqrt{-\Lambda}}\Big)\text{ for all }u\in[u_{n;i,j}^{(-)},u_{n;i,j}^{(+)}].\label{eq:GeneralBoundMassDifferenceIngoingi>j}
\end{equation}

\medskip{}

\noindent \emph{The case $i<j$: Proof of (\ref{eq:IngoingEnergyCloseToInfinity})
and (\ref{eq:OutgoingEnergyCloseToInfinity}). }The proof of (\ref{eq:IngoingEnergyCloseToInfinity})
and (\ref{eq:OutgoingEnergyCloseToInfinity}) will follow by the same
arguments as the proof of (\ref{eq:DecreaseOutgoingEnergyCloseToAxis})
and (\ref{eq:IncreaseIngoingEnergyCloseToAxis}), the main difference
being that in this case, we will use (\ref{eq:BoundsRIntersectionRegioni<j})
and (\ref{eq:ComparisonRIntersectionRegioni<j}) in place of (\ref{eq:BoundsRIntersectionRegioni>j})
and (\ref{eq:ComparisonRIntersectionRegioni>j}), respectively. In
particular, in this case, the fact that $r_{n;i,j}\gtrsim\frac{1}{\varepsilon}$
will actually render all the error terms appearing in the relevant
computations of order $O(\varepsilon)$ or smaller, simplyfying the
whole procedure substantially.

Using the bounds (\ref{eq:EstimateGeometryLikeC0}), (\ref{eq:UpperBoundTuuTvvFi}),
(\ref{eq:BoundsRIntersectionRegioni<j}) and (\ref{eq:ComparisonRIntersectionRegioni<j})
on $\mathcal{R}_{i;j}^{(n)}$ as well as (\ref{eq:HierarchyOfParameters})
and (\ref{eq:RecursiveEi}), we can estimate
\begin{align}
\int_{u_{n;i,j}^{(-)}}^{u}4\pi\frac{rT_{uu}[f_{\varepsilon}]}{-\partial_{u}r}(\bar{u},v)\,d\bar{u} & =\int_{u_{n;i,j}^{(-)}}^{u}4\pi\frac{1}{r(1-\frac{1}{3}\Lambda r^{2})}\frac{1-\frac{1}{3}\Lambda r^{2}}{-\partial_{u}r}r^{2}T_{uu}[f_{\varepsilon}](\bar{u},v)\,d\bar{u}\le\label{eq:CoefficientExpressionForIngoing-1-1Infinity}\\
 & \le\exp(\exp(\sigma_{\varepsilon}^{-7}))\frac{1}{r_{n;i,j}^{3}}h_{\varepsilon,j}(-\Lambda)^{-1}\le\varepsilon\nonumber 
\end{align}
and, similarly
\begin{equation}
\int_{v_{n;i,j}^{(-)}}^{v}4\pi\frac{rT_{vv}[f_{\varepsilon}]}{\partial_{v}r}(u,\bar{v})\,d\bar{v}\le\varepsilon.\label{eq:CoefficientExpressionForOutgoing-1-1INfinity}
\end{equation}

Substituting (\ref{eq:CoefficientExpressionForIngoing-1-1Infinity})
in (\ref{eq:EnergyDensityIncreaseIngoing}) and integrating in $v$
over $[v_{n;i,j}^{(-)},v]$, we obtain the following analogue of (\ref{eq:EnergyIntegralIncreaseIngoing}):
\begin{align}
\mathcal{E}_{\nwarrow}[n;i.j](u,v) & =\int_{v_{n;i,j}^{(-)}}^{v}\Bigg\{\exp\Big(O(\varepsilon)\Big)\cdot E_{\nwarrow}[n;i.j](u_{n;i,j}^{(-)},\bar{v})\Bigg\}\,d\bar{v}-a_{\varepsilon i}\int_{v_{n;i,j}^{(-)}}^{v}\mathfrak{Err}_{\nwarrow}[n;i.j](u,\bar{v})\,d\bar{v}=\label{eq:IngoingIntegralInfinity}\\
 & =(1+O(\varepsilon))\mathcal{E}_{\nwarrow}[n;i.j](u_{n;i,j}^{(-)},v)-a_{\varepsilon i}\int_{v_{n;i,j}^{(-)}}^{v}\mathfrak{Err}_{\nwarrow}[n;i.j](u,\bar{v})\,d\bar{v}.\nonumber 
\end{align}
 Repeating the same procedure as for the proof of (\ref{eq:FinalIngoingErrorEstimate}),
but using (\ref{eq:BoundsRIntersectionRegioni<j}) and (\ref{eq:ComparisonRIntersectionRegioni<j})
in place of (\ref{eq:BoundsRIntersectionRegioni>j}) and (\ref{eq:ComparisonRIntersectionRegioni>j}),
we can estimate:
\begin{equation}
\sup_{(u,v)\in\mathcal{R}_{i;j}^{(n)}}\Big|\int_{v_{n;i,j}^{(-)}}^{v}\mathfrak{Err}_{\nwarrow}[n;i.j](u,\bar{v})\,d\bar{v}\Big|\le\varepsilon\cdot\frac{\varepsilon^{(i)}}{\sqrt{-\Lambda}}\label{eq:ErrorIngoingInfinity}
\end{equation}
Therefore, from (\ref{eq:IngoingIntegralInfinity}) we infer that
\begin{equation}
\mathcal{E}_{\nwarrow}[n;i.j](u,v)=(1+O(\varepsilon))\mathcal{E}_{\nwarrow}[n;i.j](u_{n;i,j}^{(-)},v)+O\Big(\varepsilon\cdot\frac{\varepsilon^{(i)}}{\sqrt{-\Lambda}}\Big),\label{eq:AlmostThereIngoingEnergyChange-1}
\end{equation}
from which (\ref{eq:IngoingEnergyCloseToInfinity}) follows by setting
$u=u_{n;i,j}^{(+)}$ and $v=v_{n;i,j}^{(+)}$ and using (\ref{eq:FormulaForFinalEnergiesFromDensities}).
Similarly, 
\begin{equation}
\mathcal{E}_{\nearrow}[n;i.j](u,v)=(1+O(\varepsilon))\mathcal{E}_{\nearrow}[n;i.j](u,v_{n;i,j}^{(-)})+O\Big(\varepsilon\cdot\frac{\varepsilon^{(j)}}{\sqrt{-\Lambda}}\Big),\label{eq:AlmostThereIngoingEnergyChange-1-1}
\end{equation}
from which the estimate (\ref{eq:OutgoingEnergyCloseToInfinity})
follows by setting $u=u_{n;i,j}^{(+)}$ and $v=v_{n;i,j}^{(+)}$.

\medskip{}

\noindent \emph{Remark.} Similarly as in the case $i>j$, from (\ref{eq:AlmostThereIngoingEnergyChange-1})
and (\ref{eq:AlmostThereIngoingEnergyChange-1}) (\ref{eq:AlmostThereIngoingEnergyChange-1-1}),
using (\ref{eq:FormulaForFinalEnergiesIntermediateFromDensities-1}),
we obtain 
\begin{equation}
\tilde{m}\big(u,v_{n;i,j}^{(+)}\big)-\tilde{m}\big(u,v_{n;i,j}^{(-)}\big)=\big(1+O(\varepsilon)\big)\mathcal{E}_{\nwarrow}^{(-)}[n;i,j]+O\Big(\varepsilon\frac{\varepsilon^{(i)}}{\sqrt{-\Lambda}}\Big)\text{ for all }u\in[u_{n;i,j}^{(-)},u_{n;i,j}^{(+)}]\label{eq:GeneralBoundMassDifferenceIngoingi<j}
\end{equation}
 and 
\begin{equation}
\tilde{m}(u_{n;i,j}^{(-)},v)-\tilde{m}(u_{n;i,j}^{(+)},v)=\big(1+O(\varepsilon)\big)\mathcal{E}_{\nearrow}^{(-)}[n;i,j]+O\Big(\varepsilon\frac{\varepsilon^{(j)}}{\sqrt{-\Lambda}}\Big)\text{ for all }v\in[v_{n;i,j}^{(-)},v_{n;i,j}^{(+)}].\label{eq:GeneralBoundMassDifferenceOutfoinfi<j}
\end{equation}

\medskip{}

The relations (\ref{eq:IncreaseIngoingEnergyCloseToAxis})\textendash (\ref{eq:OutgoingEnergyCloseToInfinity})
for $\widetilde{\mathcal{E}}_{\nwarrow}^{(\pm)}[n;i,j]$, $\widetilde{\mathcal{E}}_{\nearrow}^{(\pm)}[n;i,j]$
in place of $\mathcal{E}_{\nwarrow}^{(\pm)}[n;i,j]$, $\mathcal{E}_{\nearrow}^{(\pm)}[n;i,j]$
follow in exactly the same way, after replacing $\sigma_{\varepsilon}$,
$h_{\varepsilon,i}$, $\mathcal{V}_{i}^{(n)}$, $\mathcal{R}_{i;j}^{(n)}$
with $\delta_{\varepsilon}$, $\tilde{h}_{\varepsilon,i}$, $\widetilde{\mathcal{V}}_{i}^{(n)}$,
$\mathcal{R}_{i;j}^{(n)}$ respectively, in all the expressions above
and using (\ref{eq:EstimateGeometryLikeC0LargerDomain}) in place
of (\ref{eq:EstimateGeometryLikeC0}).
\end{proof}
The next result provides an estimate for the change of the geometric
separation of the beams $\mathcal{V}_{i\nwarrow}^{(n)}$ and $\mathcal{V}_{i-1\nwarrow}^{(n)}$
before and after their intersection with $\mathcal{V}_{j\nearrow}^{(n)}$,
as well as the change of the separation of $\mathcal{V}_{j\nearrow}^{(n)}$
and $\mathcal{V}_{j-1\nearrow}^{(n)}$before and after their intersection
with $\mathcal{V}_{i\nwarrow}^{(n)}$ (with $\mathcal{V}_{i\nwarrow}^{(n+1)}$
and $\mathcal{V}_{i-1\nwarrow}^{(n+1)}$ in place of $\mathcal{V}_{i\nwarrow}^{(n)}$
and $\mathcal{V}_{i-1\nwarrow}^{(n)}$, if $i<j$).
\begin{prop}
\label{prop:RSeparationChangeInteraction} Let $\varepsilon\in(0,\varepsilon_{1}]$
and let $n\in\mathbb{N}$ and $0\le i,j\le N_{\varepsilon}$, $i\neq j$,
be such that 
\[
\mathcal{R}_{i;j}^{(n)}\subset\mathcal{U}_{\varepsilon}^{+}.
\]
Let also $r_{n;i,j}$ be defined by (\ref{eq:DefinitionRnij}). Then,
the following relations hold regarding the change of the separation-measuring
quantities $\mathfrak{D}r_{\nwarrow}^{(\pm)}[n;i,j]$, $\mathfrak{D}r_{\nearrow}^{(\pm)}[n;i,j]$:

\begin{itemize}

\item In the case $i>j$, the quantities $\mathfrak{D}r_{\nearrow}^{(\pm)}[n;i,j]$
(defined for $j>0$) satisfy 
\begin{equation}
\mathfrak{D}r_{\nearrow}^{(+)}[n;i,j]=\mathfrak{D}r_{\nearrow}^{(-)}[n;i,j]\cdot\big(1+O(\varepsilon)\big),\label{eq:IncreaseOutgoingRSeparationCloseToAxis}
\end{equation}
while for the quantities $\mathfrak{D}r_{\nwarrow}^{(\pm)}[n;i,j]$
the following hold:

\begin{itemize}

\item If $i=j+1$, 
\begin{equation}
\mathfrak{D}r_{\nwarrow}^{(+)}[n;i,j]=\mathfrak{D}r_{\nwarrow}^{(-)}[n;i,j]\cdot\big(1+O(\rho_{\varepsilon}^{\frac{3}{4}})\big).\label{eq:DecreaseIngoingRSeparationCloseToAxisi=00003Dj+1}
\end{equation}

\item If $i>j+1$, 
\begin{equation}
\mathfrak{D}r_{\nwarrow}^{(+)}[n;i,j]=\mathfrak{D}r_{\nwarrow}^{(-)}[n;i,j]\cdot\exp\Big(-\frac{2\mathcal{E}_{\nearrow}^{(-)}[n;i,j]}{r_{n;i.j}}+O(\rho_{\varepsilon}^{\frac{3}{2}})\Big).\label{eq:DecreaseIngoingRSeparationCloseToAxisi>j+1}
\end{equation}

\end{itemize} 

\item In the case $i<j$, the quantities $\mathfrak{D}r_{\nearrow}^{(\pm)}[n;i,j]$
satisfy
\begin{equation}
\mathfrak{D}r_{\nearrow}^{(+)}[n;i,j]=\mathfrak{D}r_{\nearrow}^{(-)}[n;i,j]\cdot\big(1+O(\varepsilon)\big),\label{eq:OutgoingRSeparationCloseToInfinity}
\end{equation}
while the quantities $\mathfrak{D}r_{\nwarrow}^{(\pm)}[n;i,j]$ (defined
for $i>0$) satisfy:
\begin{align}
\mathfrak{D}r_{\nwarrow}^{(+)}[n;i,j] & =\mathfrak{D}r_{\nwarrow}^{(-)}[n;i,j]\cdot\big(1+O(\varepsilon)\big).\label{eq:IngoingRSeparationCloseToInfinity}
\end{align}

\end{itemize}

In the case when 
\[
\mathcal{R}_{i;\gamma_{\mathcal{Z}}}^{(n)},\mathcal{R}_{i;\mathcal{I}}^{(n)}\subset\mathcal{U}_{\varepsilon}^{+}
\]
the following relations hold for $\mathfrak{D}r_{\nwarrow}^{(\pm)}[n;i,i]$,
$\mathfrak{D}r_{\nearrow}^{(\pm)}[n;i,i]$: 
\begin{align}
\mathfrak{D}r_{\nwarrow}^{(+)}[n;i,i] & =\mathfrak{D}r_{\nwarrow}^{(-)}[n;i,i]\cdot\big(1+O(\varepsilon)\big),\label{eq:SpecialRSeparationCloseToInfinity}\\
\mathfrak{D}r_{\nearrow}^{(+)}[n;i,i] & =\mathfrak{D}r_{\nearrow}^{(-)}[n;i,i]\cdot\big(1+O(\rho_{\varepsilon}^{\frac{3}{4}})\big).\label{eq:SpecialRSeparationCloseToAxis}
\end{align}

Replacing $\mathcal{U}_{\varepsilon}^{+}$ with $\mathcal{T}_{\varepsilon}^{+}$
and $\mathcal{V}_{i}^{(n)}$ with $\widetilde{\mathcal{V}}_{i}^{(n)}$,
the relations (\ref{eq:IncreaseOutgoingRSeparationCloseToAxis})\textendash (\ref{eq:SpecialRSeparationCloseToAxis})
also hold with $\widetilde{\mathfrak{D}}r_{\nwarrow}^{(\pm)}[n;i,j]$,
$\widetilde{\mathfrak{D}}r_{\nearrow}^{(\pm)}[n;i,j]$ in place of
$\mathfrak{D}r_{\nwarrow}^{(\pm)}[n;i,j]$, $\mathfrak{D}r_{\nearrow}^{(\pm)}[n;i,j]$.
\end{prop}
\begin{proof}
\noindent In order to establish (\ref{eq:IncreaseOutgoingRSeparationCloseToAxis})\textendash (\ref{eq:SpecialRSeparationCloseToAxis}),
we will assume without loss of generality that $i>0$ and $j>0$,
so that both $\mathfrak{D}r_{\nwarrow}^{(\pm)}[n;i,j]$ and $\mathfrak{D}r_{\nearrow}^{(\pm)}[n;i,j]$
are well defined. In the case when $i=0$ (when $\mathfrak{D}r_{\nwarrow}^{(\pm)}[n;i,j]$
is not defined), the proof of (\ref{eq:OutgoingRSeparationCloseToInfinity})
follows exactly as in the case $i>0$, and similarly for the proof
of (\ref{eq:DecreaseIngoingRSeparationCloseToAxisi=00003Dj+1})\textendash (\ref{eq:DecreaseIngoingRSeparationCloseToAxisi>j+1})
in the case $j=0$. 

As we did in the proof of Proposition \ref{prop:EnergyChangeInteraction},
we will use the shorthand notation $v_{n;i,j}^{(\pm)}$, $u_{n;i,j}^{(\pm)}$
for the expressions (\ref{eq:v+-}), (\ref{eq:u+-}), respectively.
We will also define the energy quantities $E_{\nwarrow}[n;i,j](u,v)$,
$E_{\nearrow}[n;i,j](u,v)$, $\mathcal{E}_{\nwarrow}[n;i,j](u,v)$
and $\mathcal{E}_{\nearrow}[n;i,j](u,v)$ by (\ref{eq:IngoingMainEnergyDensity}),
(\ref{eq:OutgoingMainEnergyDensity}), (\ref{eq:IngoingEnergyForBootstrap})
and (\ref{eq:OutgoingEnergyForBootstrap}), respectively. 

Let us define the domains
\begin{align}
\mathcal{W}_{\nwarrow}[n;i,j]\doteq & [u_{n;i,j-1}^{(+)}+\rho_{\varepsilon}^{-\frac{7}{8}}h_{\varepsilon,j-1},u_{n;i,j}^{(-)}-\rho_{\varepsilon}^{-\frac{7}{8}}h_{\varepsilon,j-1}]\times[v_{n;i,j}^{(-)},v_{n;i,j}^{(+)}],\label{eq:IngoingDomainForRSeparation}\\
\mathcal{W}_{\nearrow}[n;i,j]\doteq & [u_{n;i,j}^{(-)},u_{n;i,j}^{(+)}]\times[v_{n;i-1,j}^{(+)}+\rho_{\varepsilon}^{-\frac{7}{8}}h_{\varepsilon,i-1},v_{n;i,j}^{(-)}-\rho_{\varepsilon}^{-\frac{7}{8}}h_{\varepsilon,i-1}],\label{eq:OutgoingDomainForRSeparation}
\end{align}
with the following convention for $v_{i-1,i-1}^{(+)}$ (recall that
(\ref{eq:v+-}) defined $v_{n;i,j}^{(\pm)}$ only for $i\neq j$):
\begin{align}
v_{n;i-1,i-1}^{(+)} & =v_{\varepsilon,i-1}^{(n)}+h_{\varepsilon,i-1}.\label{eq:ConventionV+-}
\end{align}
Notice that the quantities $\mathfrak{D}r_{\nwarrow}^{(\pm)}[n;i,j]$
and $\mathfrak{D}r_{\nearrow}^{(\pm)}[n;i,j]$ (given by (\ref{eq:IngoingRSeparationBeforeAfter})
and (\ref{eq:OutgoingRSeparationBeforeAfter})) are defined through
integration on the $u=u_{n;i,j}^{(\pm)}$ and $v=v_{n;i,j}^{(\pm)}$
parts of the boundary of $\mathcal{W}_{\nearrow}[n;i,j]$ and $\mathcal{W}_{\nwarrow}[n;i,j]$,
respectively. Note also that 
\[
\mathcal{W}_{\nwarrow}[n;i,j]\subset\mathcal{V}_{i\nwarrow}^{(n)}
\]
 (with $\mathcal{V}_{i\nwarrow}^{(n+1)}$ in place of $\mathcal{V}_{i\nwarrow}^{(n)}$
if $i<j$) and 
\[
\mathcal{W}_{\nearrow}[n;i,j]\subset\mathcal{V}_{j\nearrow}^{(n)},
\]
 as well as 
\[
\mathcal{W}_{\nwarrow}[n;i,j]\cap\mathcal{R}_{i;j}^{(n)}=\mathcal{W}_{\nearrow}[n;i,j]\cap\mathcal{R}_{i;j}^{(n)}=\emptyset.
\]

\begin{figure}[h] 
\centering 
\scriptsize
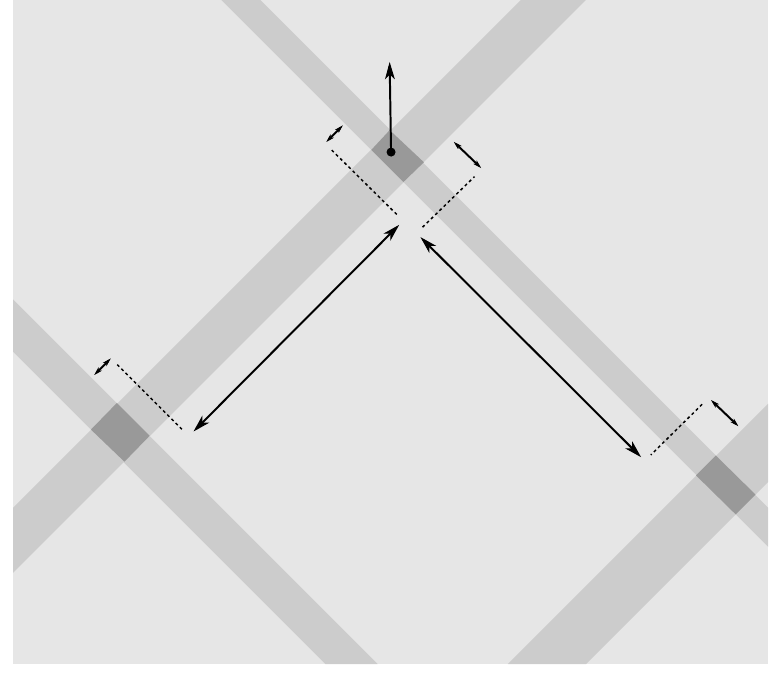 
\caption{Schematic depiction of the domains $\mathcal{W}_{\nwarrow}[n;i,j]$ and $\mathcal{W}_{\nearrow}[n;i,j]$. \label{fig:WDomains}}
\end{figure}

In view of the definition (\ref{eq:IngoingDistanceFunction}) and
(\ref{eq:OutgoingDistanceFunction}) of $dist_{\nwarrow}[\cdot]$
and $dist_{\nearrow}[\cdot]$, we can readily calculate that:

\begin{itemize}

\item For all $(u,v)\in\mathcal{W}_{\nwarrow}[n;i,j]$:
\begin{equation}
dist_{\nwarrow}[(u,v)]=\exp\big(-\exp(\sigma_{\varepsilon}^{-4})\big)\frac{\varepsilon^{(i)}}{\sqrt{-\Lambda}}\label{eq:DistanceFunctionsForWIngoing}
\end{equation}
and 
\begin{equation}
dist_{\nearrow}[(u,v)]\ge e^{-\sigma_{\varepsilon}^{-6}}\rho_{\varepsilon}^{-\frac{7}{8}}\frac{\varepsilon^{(j-1)}}{\sqrt{-\Lambda}}.\label{eq:LowerBoundTransversalDistanceForWIngoing}
\end{equation}

\item For all $(u,v)\in\mathcal{W}_{\nearrow}[n;i,j]$:
\begin{equation}
dist_{\nearrow}[(u,v)]=\exp\big(-\exp(\sigma_{\varepsilon}^{-4})\big)\frac{\varepsilon^{(j)}}{\sqrt{-\Lambda}}\label{eq:DistanceFunctionsForWOutgoing}
\end{equation}
and
\begin{equation}
dist_{\nwarrow}[(u,v)]\ge e^{-\sigma_{\varepsilon}^{-6}}\rho_{\varepsilon}^{-\frac{7}{8}}\frac{\varepsilon^{(i-1)}}{\sqrt{-\Lambda}}.\label{eq:LowerBoundTransversalDistanceForWOutgoing}
\end{equation}

\end{itemize}

In view of the bound (\ref{eq:BoundSupportFepsiloni}) on the support
of $f_{\varepsilon i}$, we know that, among all the $f_{\varepsilon k}$'s,
only $f_{\varepsilon i}$ is supported on $\mathcal{W}_{\nwarrow}[n;i,j]$,
and only $f_{\varepsilon j}$ is supported on $\mathcal{W}_{\nearrow}[n;i,j]$.
As a result,
\begin{align}
T_{\mu\nu}[f_{\varepsilon}]\big|_{\mathcal{W}_{\nwarrow}[n;i,j]} & =a_{\varepsilon i}T_{\mu\nu}[f_{\varepsilon i}]\big|_{\mathcal{W}_{\nwarrow}[n;i,j]},\label{eq:OnlyOneComponentTube}\\
T_{\mu\nu}[f_{\varepsilon}]\big|_{\mathcal{W}_{\nearrow}[n;i,j]} & =a_{\varepsilon j}T_{\mu\nu}[f_{\varepsilon j}]\big|_{\mathcal{W}_{\nearrow}[n;i,j]}.\nonumber 
\end{align}
 Furthermore, (\ref{eq:EstimateGeometryLikeC0}) and the definition
(\ref{eq:IngoingDomainForRSeparation}), (\ref{eq:OutgoingDomainForRSeparation})
of $\mathcal{W}_{\nwarrow}[n;i,j]$, $\mathcal{W}_{\nearrow}[n;i,j]$,
readily yield the following lower bounds: 

\begin{itemize}

\item When $i>j$:
\begin{equation}
\inf_{\mathcal{W}_{\nwarrow}[n;i,j]}r\ge e^{-\sigma_{\varepsilon}^{-6}}\rho_{\varepsilon}^{-1}\frac{\varepsilon^{(j)}}{\sqrt{-\Lambda}}\text{ and }\inf_{\mathcal{W}_{\nearrow}[n;i,j]}r\ge e^{-\sigma_{\varepsilon}^{-6}}\rho_{\varepsilon}^{-\frac{7}{8}}\frac{\varepsilon^{(j)}}{\sqrt{-\Lambda}}.\label{eq:LowerBoundRW}
\end{equation}

\item When $i<j$: 
\begin{equation}
\inf_{\mathcal{W}_{\nwarrow}[n;i,j]}r\ge e^{-\sigma_{\varepsilon}^{-6}}\rho_{\varepsilon}\frac{1}{\varepsilon^{(i)}\sqrt{-\Lambda}}\text{ and }\inf_{\mathcal{W}_{\nearrow}[n;i,j]}r\ge e^{-\sigma_{\varepsilon}^{-6}}\rho_{\varepsilon}\frac{1}{\varepsilon^{(i)}\sqrt{-\Lambda}}.\label{eq:LowerBoundRWNearInfinity}
\end{equation}

\end{itemize}

Integrating (\ref{eq:EquationROutside}) from $u=u_{n;i,j}^{(-)}$
up to $u=u_{n;i,j}^{(+)}$, exponentiating the resulting expression
and then integrating in $v\in[v_{n;i-1,j}^{(+)}+\rho_{\varepsilon}^{-\frac{7}{8}}h_{\varepsilon,i-1},v_{n;i,j}^{(-)}-\rho_{\varepsilon}^{-\frac{7}{8}}h_{\varepsilon,i-1}]$,
we readily obtain using the definition (\ref{eq:IngoingRSeparationBeforeAfter})
of $\mathfrak{D}r_{\nwarrow}^{(+)}[n;i,j]$: 
\begin{align}
\mathfrak{D}r_{\nwarrow}^{(+)}[n;i,j] & =\int_{\text{ }v_{n;i-1,j}^{(+)}+\rho_{\varepsilon}^{-\frac{7}{8}}h_{\varepsilon,i-1}}^{\text{ }v_{n;i,j}^{(-)}-\rho_{\varepsilon}^{-\frac{7}{8}}h_{\varepsilon,i-1}}\exp\Big(-\int_{u_{n;i,j}^{(-)}}^{u_{n;i,j}^{(+)}}4\pi\frac{rT_{uu}[f_{\varepsilon}]}{-\partial_{u}r}(\bar{u},\bar{v})\,d\bar{u}\Big)\frac{\partial_{v}r}{1-\frac{2m}{r}}(u_{n;i,j}^{(-)},\bar{v})\,d\bar{v}.\label{eq:IngoingRSeparationBeforeAfter-1}
\end{align}
Similarly, after integrating (\ref{eq:EquationROutside}): 
\begin{align}
\mathfrak{D}r_{\nearrow}^{(+)}[n;i,j] & =\int_{\text{ }u_{n;i,j-1}^{(+)}+\rho_{\varepsilon}^{-\frac{7}{8}}h_{\varepsilon,j-1}}^{\text{ }u_{n;i,j}^{(-)}-\rho_{\varepsilon}^{-\frac{7}{8}}h_{\varepsilon,j-1}}\exp\Big(\int_{v_{n;i,j}^{(-)}}^{v_{n;i,j}^{(+)}}4\pi\frac{rT_{vv}[f_{\varepsilon}]}{\partial_{v}r}(\bar{u},\bar{v})\,d\bar{v}\Big)\frac{-\partial_{u}r}{1-\frac{2m}{r}}(\bar{u},v_{n;i,j}^{(-)})\,d\bar{u}.\label{eq:OutgoingRSeparationBeforeAfter-1}
\end{align}

In view of (\ref{eq:OnlyOneComponentTube}), we readily infer (arguing
exactly as in the proof of (\ref{eq:CoefficientExpressionForIngoing-1})\textendash (\ref{eq:CoefficientExpressionForOutgoing-1}))
that, for all $(u,v)\in\mathcal{W}_{\nearrow}[n;i,j]$,
\begin{equation}
\int_{u_{n;i,j}^{(-)}}^{u}4\pi\frac{rT_{uu}[f_{\varepsilon}]}{-\partial_{u}r}(\bar{u},v)\,d\bar{u}=\int_{u_{n;i,j}^{(-)}}^{u}\frac{2E_{\nearrow}[n;i,j]}{r\cdot\big(1-\frac{1}{3}\Lambda r^{2}-O(\mu_{\nearrow ij})\big)}(\bar{u},v)\,d\bar{u}\label{eq:CoefficientExpressionForIngoingForDr}
\end{equation}
and, for all $(u,v)\in\mathcal{W}_{\nwarrow}[n;i,j]$, 
\begin{equation}
\int_{v_{n;i,j}^{(-)}}^{v}4\pi\frac{rT_{vv}[f_{\varepsilon}]}{\partial_{v}r}(u,\bar{v})\,d\bar{v}=\int_{v_{n;i,j}^{(-)}}^{v}\frac{2E_{\nwarrow}[n;i,j]}{r\cdot\big(1-\frac{1}{3}\Lambda r^{2}-O(\mu_{\nwarrow ij})\big)}(u,\bar{v})\,d\bar{v},\label{eq:CoefficientExpressionForOutgoingForDr}
\end{equation}
where 
\begin{align}
\mu_{\nearrow ij} & \doteq\sup_{\mathcal{W}_{\nearrow}[n;i,j]}\frac{2\tilde{m}}{r},\\
\mu_{\nwarrow ij} & \doteq\sup_{\mathcal{W}_{\nwarrow}[n;i,j]}\frac{2\tilde{m}}{r}.\nonumber 
\end{align}
Note that 
\begin{equation}
\mu_{\nearrow ij},\text{ }\mu_{\nwarrow ij}=O(\eta_{0}),\label{eq:UpperBoundMuIntersectionTrivial-1}
\end{equation}
as a trivial consequence of (\ref{eq:UpperBoundMuBootstrapDomain}).

Substituting (\ref{eq:CoefficientExpressionForIngoingForDr}) and
(\ref{eq:CoefficientExpressionForOutgoingForDr}) in (\ref{eq:IngoingRSeparationBeforeAfter-1})
and (\ref{eq:OutgoingRSeparationBeforeAfter-1}), respectively, we
obtain 
\begin{align}
\mathfrak{D}r_{\nwarrow}^{(+)}[n;i,j] & =\int_{\text{ }v_{n;i-1,j}^{(+)}+\rho_{\varepsilon}^{-\frac{7}{8}}h_{\varepsilon,i-1}}^{\text{ }v_{n;i,j}^{(-)}-\rho_{\varepsilon}^{-\frac{7}{8}}h_{\varepsilon,i-1}}\exp\Big(-\int_{u_{n;i,j}^{(-)}}^{u_{n;i,j}^{(+)}}\frac{2E_{\nearrow}[n;i,j]}{r\cdot\big(1-\frac{1}{3}\Lambda r^{2}-O(\mu_{\nearrow ij})\big)}(\bar{u},\bar{v})\,d\bar{u}\Big)\cdot\frac{\partial_{v}r}{1-\frac{2m}{r}}(u_{n;i,j}^{(-)},\bar{v})\,d\bar{v}\label{eq:IngoingRSeparationBeforeAfterFinalGeneral}
\end{align}
and 
\begin{align}
\mathfrak{D}r_{\nearrow}^{(+)}[n;i,j] & =\int_{\text{ }u_{n;i,j-1}^{(+)}+\rho_{\varepsilon}^{-\frac{7}{8}}h_{\varepsilon,j-1}}^{\text{ }u_{n;i,j}^{(-)}-\rho_{\varepsilon}^{-\frac{7}{8}}h_{\varepsilon,j-1}}\exp\Big(\int_{v_{n;i,j}^{(-)}}^{v_{n;i,j}^{(+)}}\frac{2E_{\nwarrow}[n;i,j]}{r\cdot\big(1-\frac{1}{3}\Lambda r^{2}-O(\mu_{\nwarrow ij})\big)}(\bar{u},\bar{v})\,d\bar{v}\Big)\cdot\frac{-\partial_{u}r}{1-\frac{2m}{r}}(\bar{u},v_{n;i,j}^{(-)})\,d\bar{u}.\label{eq:OutgoingRSeparationBeforeAfterFinalGeneral}
\end{align}

From (\ref{eq:ODEforEnergyFlux}) (and the analogous equation for
$\partial_{v}E_{\nearrow}[n;i,j](u,v)$), we obtain the following
formulas (in analogy with (\ref{eq:EnergyDensityIncreaseIngoing})
and (\ref{eq:EnergyDensityIncreaseOutgoing})): 

\begin{itemize}

\item For $(u,v)\in\mathcal{W}_{\nwarrow}[n;i,j]$, 
\begin{equation}
E_{\nwarrow}[n;i.j](u,v)=\exp\Big(-\int_{u}^{u_{n;i,j}^{(-)}}4\pi\frac{rT_{uu}[f_{\varepsilon}]}{-\partial_{u}r}(\bar{u},v)\,d\bar{u}\Big)\cdot E_{\nwarrow}[n;i.j](u_{n;i,j}^{(-)},v)-a_{\varepsilon i}\mathfrak{Err}_{\nwarrow}^{(0)}[n;i.j](u,v),\label{eq:EnergyDensityIncreaseIngoingDr}
\end{equation}
where
\begin{align}
\mathfrak{Err}_{\nwarrow}^{(0)}[n;i.j](u,v)\doteq & \int_{u}^{u_{n;i,j}^{(-)}}\exp\Big(-\int_{u}^{\bar{u}}4\pi\frac{rT_{uu}[f_{\varepsilon}]}{-\partial_{u}r}(\hat{u},v)\,d\hat{u}\Big)\times\label{eq:ErrorTermForEnergyChangeIngoingDr}\\
 & \hphantom{-}\times\Bigg\{2\pi\frac{1-\frac{2m}{r}}{\partial_{v}r}\partial_{v}(r^{2}T_{uv}[f_{\varepsilon i}])+2\pi\frac{1-\frac{2m}{r}}{\partial_{v}r}\Big(\partial_{v}\log(\Omega^{2})-2\frac{\partial_{v}r}{r}\Big)(r^{2}T_{uv}[f_{\varepsilon i}])\Bigg\}(\bar{u},v)\,d\bar{u}.\nonumber 
\end{align}

\item For $(u,v)\in\mathcal{W}_{\nearrow}[n;i,j]$, 
\begin{equation}
E_{\nearrow}[n;i.j](u,v)=\exp\Big(\int_{v}^{v_{n;i,j}^{(-)}}4\pi\frac{rT_{vv}[f_{\varepsilon}]}{\partial_{v}r}(u,\bar{v})\,d\bar{v}\Big)\cdot E_{\nearrow}[n;i.j](u,v_{n;i,j}^{(-)})-a_{\varepsilon j}\mathfrak{Err}_{\nearrow}^{(0)}[n;i.j](u,v),\label{eq:EnergyDensityIncreaseOutgoingDr}
\end{equation}
where
\begin{align}
\mathfrak{Err}_{\nearrow}^{(0)}[n;i.j](u,v)\doteq & \int_{v}^{v_{n;i,j}^{(-)}}\exp\Big(-\int_{v}^{\bar{v}}4\pi\frac{rT_{vv}[f_{\varepsilon}]}{\partial_{v}r}(u,\hat{v})\,d\hat{v}\Big)\times\label{eq:ErrorTermForEnergyChangeOutgoingDr}\\
 & \hphantom{-}\times\Bigg\{2\pi\frac{1-\frac{2m}{r}}{-\partial_{u}r}\partial_{u}(r^{2}T_{uv}[f_{\varepsilon j}])+2\pi\frac{1-\frac{2m}{r}}{-\partial_{u}r}\Big(\partial_{u}\log(\Omega^{2})-2\frac{\partial_{u}r}{r}\Big)(r^{2}T_{uv}[f_{\varepsilon j}])\Bigg\}(u,\bar{v})\,d\bar{v}.\nonumber 
\end{align}

\end{itemize}

We will now procced to treat the cases $i>j$ and $i<j$ separately.

\medskip{}

\noindent \emph{The case $i>j$: Proof of (\ref{eq:IncreaseOutgoingRSeparationCloseToAxis})\textendash (\ref{eq:DecreaseIngoingRSeparationCloseToAxisi>j+1}).}
Integrating (\ref{eq:TildeVMaza}) in $v$ from the axis $\gamma_{\mathcal{Z}_{\varepsilon}}$
up to $\mathcal{W}_{\nearrow}[n;i,j]$, $\mathcal{W}_{\nwarrow}[n;i,j]$,
using the fact that, among all the $f_{\varepsilon k}$'s, only $f_{\varepsilon\bar{j}}$,
$j\le\bar{j}\le i$ are supported on the domain 
\[
\{\inf_{\mathcal{W}_{\nearrow}[n;i,j]}u\le u\le\sup_{\mathcal{W}_{\nearrow}[n;i,j]}u\}\cap\{v\le\sup_{\mathcal{W}_{\nearrow}[n;i,j]}v\}
\]
 (and the same for $\mathcal{W}_{\nwarrow}[n;i,j]$), we can readily
estimate (using (\ref{eq:EstimateGeometryLikeC0}) and (\ref{eq:UpperBoundTuuTvvFi}),
(\ref{eq:UpperBoundTuvFi})): 
\begin{equation}
\sup_{\mathcal{W}_{\nearrow}[n;i,j]\cup\mathcal{W}_{\nwarrow}[n;i,j]}\tilde{m}\le\exp(\exp(\sigma_{\varepsilon}^{-6}))\frac{\varepsilon^{(j)}}{\sqrt{-\Lambda}}.\label{eq:UpperBoundHawkingMassRnij-1}
\end{equation}
From (\ref{eq:UpperBoundHawkingMassRnij-1}) and the bound (\ref{eq:LowerBoundRW})
for $r$ on $\mathcal{W}_{\nearrow}[n;i,j]$, $\mathcal{W}_{\nwarrow}[n;i,j]$,
we immediately infer that: 
\begin{equation}
\mu_{\nearrow ij}+\mu_{\nwarrow ij}=\sup_{\mathcal{W}_{\nearrow}[n;i,j]}\frac{2\tilde{m}}{r}+\sup_{\mathcal{W}_{\nwarrow}[n;i,j]}\frac{2\tilde{m}}{r}\le\exp(\exp(\sigma_{\varepsilon}^{-7}))\rho_{\varepsilon}^{\frac{7}{8}}\le\rho_{\varepsilon}^{\frac{3}{4}}.\label{eq:NonTrivialUpperBoundMuij-1}
\end{equation}

Using the fact that $T_{\mu\nu}[f_{\varepsilon}]=a_{\varepsilon j}T_{\mu\nu}[f_{\varepsilon j}]$
on $\mathcal{W}_{\nearrow}[n;i,j]$, the bound (\ref{eq:UpperBoundTuuTvvFi})
for $T_{vv}[f_{\varepsilon j}]$ on $\mathcal{W}_{\nearrow}[n;i,j]\subset\mathcal{V}_{j\nearrow}^{(n)}$,
combined with (\ref{eq:EstimateGeometryLikeC0}), implies that
\begin{align}
\sup_{(u,v)\in\mathcal{W}_{\nearrow}[n;i,j]} & \int_{v}^{v_{n;i,j}^{(-)}}4\pi\frac{rT_{vv}[f_{\varepsilon}]}{\partial_{v}r}(u,\bar{v})\,d\bar{v}\label{eq:TrivialBoundChangeOfEnergyFlux}\\
 & \le a_{\varepsilon j}\exp\big(\exp(\sigma_{\varepsilon}^{-6})\big)(-\Lambda)^{-2}\int_{r(u_{n;i,j}^{(+)},v_{n;i-1,j}^{(+)}+\rho_{\varepsilon}^{-\frac{7}{8}}h_{\varepsilon,i-1})}^{r(u_{n;i,j}^{(-)},v_{n;i,j}^{(-)})}\frac{(\varepsilon^{(j)})^{4}}{\bar{r}^{5}}\,d\bar{r}\le\nonumber \\
 & \le\exp\big(\exp(\sigma_{\varepsilon}^{-6})\big)(-\Lambda)^{-2}(\varepsilon^{(j)})^{4}\frac{r(u_{n;i,j}^{(-)},v_{n;i,j}^{(-)})-r(u_{n;i,j}^{(+)},v_{n;i-1,j}^{(+)}+\rho_{\varepsilon}^{-\frac{7}{8}}h_{\varepsilon,i-1})}{\big(r(u_{n;i,j}^{(+)},v_{n;i-1,j}^{(+)}+\rho_{\varepsilon}^{-\frac{7}{8}}h_{\varepsilon,i-1})\big)^{5}}\le\nonumber \\
 & \le\exp\big(\exp(2\sigma_{\varepsilon}^{-6})\big)(-\Lambda)^{-2}(\varepsilon^{(j)})^{4}\frac{(u_{n;i,j}^{(+)}-u_{n;i,j}^{(-)})+(v_{n;i,j}^{(-)}-v_{n;i-1,j}^{(+)}-\rho_{\varepsilon}^{-\frac{7}{8}}h_{\varepsilon,i-1})}{(\rho_{\varepsilon}^{-\frac{7}{8}}h_{\varepsilon,i-1}+v_{n;i-1,j}^{(+)}-u_{n;i,j}^{(+)})^{5}}\le\nonumber \\
 & \le\exp\big(\exp(\sigma_{\varepsilon}^{-7})\big)(\varepsilon^{(j)})^{4}\frac{\varepsilon^{(j)}+\rho_{\varepsilon}^{-1}\varepsilon^{(i-1)}}{(\rho_{\varepsilon}^{-\frac{7}{8}}\varepsilon^{(i-1)}+\text{sgn}(i-j-1)\cdot\rho_{\varepsilon}^{-1}\varepsilon^{(j)})^{5}}\le\nonumber \\
 & \le\rho_{\varepsilon},\nonumber 
\end{align}
where 
\[
\text{sgn}(i-j-1)=\begin{cases}
0, & i=j+1,\\
+1, & i>j+1
\end{cases}
\]
and, in passing from the second to the third line in (\ref{eq:TrivialBoundChangeOfEnergyFlux}),
we used the fact that $\partial_{u}r<0$ and $\partial_{v}r>0$ on
$\mathcal{U}_{\varepsilon}^{+}$. Substituting the bound (\ref{eq:TrivialBoundChangeOfEnergyFlux})
in (\ref{eq:EnergyDensityIncreaseOutgoingDr}), we infer that, for
any $(u,v)\in\mathcal{W}_{\nearrow}[n;i,j]$:
\begin{equation}
E_{\nearrow}[n;i.j](u,v)=\big(1+O(\rho_{\varepsilon})\big)\cdot E_{\nearrow}[n;i.j](u,v_{n;i,j}^{(-)})-a_{\varepsilon j}\mathfrak{Err}_{\nearrow}^{(0)}[n;i.j](u,v).\label{eq:EnergyDensityIncreaseOutgoingDr-1}
\end{equation}
Dividing (\ref{eq:EnergyDensityIncreaseOutgoingDr-1}) with $r$ and
integrating in $u$, we infer that, for any $(u,v)\in\mathcal{W}_{\nearrow}[n;i,j]$:
\begin{equation}
\int_{u_{n;i,j}^{(-)}}^{u}\frac{E_{\nearrow}[n;i.j]}{r}(\bar{u},v)\,d\bar{u}=\big(1+O(\rho_{\varepsilon})\big)\cdot\int_{u_{n;i,j}^{(-)}}^{u}\frac{E_{\nearrow}[n;i.j](\bar{u},v_{n;i,j}^{(-)})}{r(\bar{u},v)}\,d\bar{u}-a_{\varepsilon j}\int_{u_{n;i,j}^{(-)}}^{u}\frac{\mathfrak{Err}_{\nearrow}^{(0)}[n;i.j]}{r}(\bar{u},v)\,d\bar{u}.\label{eq:Almost1/renergyformula}
\end{equation}
Using the bound
\begin{equation}
\sup_{v\ge0}\Bigg(\sup_{(u_{1},v),(u_{2},v)\in\mathcal{W}_{\nearrow}[n;i,j]}\Big|\frac{r(u_{1},v)}{r(u_{2},v)}-1\Big|\Bigg)\le e^{\sigma_{\varepsilon}^{-6}}\rho_{\varepsilon}^{\frac{7}{8}}\le\rho_{\varepsilon}^{\frac{1}{2}}\label{eq:DifferenceRW}
\end{equation}
(following readily from (\ref{eq:EstimateGeometryLikeC0}) and the
definition (\ref{eq:OutgoingDomainForRSeparation}) of $\mathcal{W}_{\nearrow}[n;i,j]$),
we readily infer from (\ref{eq:Almost1/renergyformula}) that: 
\begin{align}
\int_{u_{n;i,j}^{(-)}}^{u}\frac{E_{\nearrow}[n;i.j]}{r}(\bar{u},v)\,d\bar{u} & =\big(1+O(\rho_{\varepsilon}^{\frac{1}{2}})\big)\cdot\frac{1}{r(u,v)}\int_{u_{n;i,j}^{(-)}}^{u}E_{\nearrow}[n;i.j](\bar{u},v_{n;i,j}^{(-)})\,d\bar{u}-\label{eq:1/rEnergyUsefulFormulaOutgoing}\\
 & \hphantom{=\big(1+O(\rho_{\varepsilon}^{\frac{1}{2}})\big)\cdot}-a_{\varepsilon j}\int_{u_{n;i,j}^{(-)}}^{u}\frac{\mathfrak{Err}_{\nearrow}^{(0)}[n;i.j]}{r}(\bar{u},v)\,d\bar{u}=\nonumber \\
 & =\big(1+O(\rho_{\varepsilon}^{\frac{1}{2}})\big)\cdot\frac{\mathcal{E}_{\nearrow}[n;i.j](u,v_{n;i,j}^{(-)})}{r(u,v)}-a_{\varepsilon j}\int_{u_{n;i,j}^{(-)}}^{u}\frac{\mathfrak{Err}_{\nearrow}^{(0)}[n;i.j]}{r}(\bar{u},v)\,d\bar{u}.\nonumber 
\end{align}
 Similarly, for any $(u,v)\in\mathcal{W}_{\nwarrow}[n;i,j]$: 
\begin{equation}
\int_{v_{n;i,j}^{(-)}}^{v}\frac{E_{\nwarrow}[n;i.j]}{r}(u,\bar{v})\,d\bar{v}=\big(1+O(\rho_{\varepsilon}^{\frac{1}{2}})\big)\cdot\frac{\mathcal{E}_{\nwarrow}[n;i.j](u_{n;i,j}^{(-)},v)}{r(u,v)}-a_{\varepsilon i}\int_{v_{n;i,j}^{(-)}}^{v}\frac{\mathfrak{Err}_{\nwarrow}^{(0)}[n;i.j]}{r}(u,\bar{v})\,d\bar{v}.\label{eq:1/rEnergyUsefulFormulaIngoing}
\end{equation}

Arguing similarly as for the derivation of (\ref{eq:ErrorFinalRelation}),
integrating by parts in $\partial_{u}$ for the term\\ $\partial_{u}(r^{2}T_{uv})$
in $\int_{u_{n;i,j}^{(-)}}^{u}\frac{\mathfrak{Err}_{\nearrow}^{(0)}[n;i.j]}{r}(\bar{u},v)\,d\bar{u}$
(see the expression (\ref{eq:ErrorTermForEnergyChangeOutgoingDr})),
we calculate: 
\begin{align}
\Big|\int_{u_{n;i,j}^{(-)}}^{u} & \frac{\mathfrak{Err}_{\nearrow}^{(0)}[n;i.j]}{r}(\bar{u},v)\,d\bar{u}\Big|=\label{eq:ErrorFinalRelation-1}\\
 & =\Bigg|\int_{v_{\varepsilon,j}^{(n)}-h_{\varepsilon,j}}^{u}\int_{v}^{v_{\varepsilon,i}^{(n)}-h_{\varepsilon,i}}\frac{1}{r(\bar{u},v)}\Bigg[\exp\Big(-\int_{v}^{\bar{v}}4\pi\frac{rT_{vv}[f_{\varepsilon}]}{\partial_{v}r}(\bar{u},\hat{v})\,d\hat{v}\Big)\times\nonumber \\
 & \hphantom{-e^{\int^{u}}}\times\Bigg\{2\pi\frac{1-\frac{2m}{r}}{-\partial_{u}r}\partial_{u}(r^{2}T_{uv}[f_{\varepsilon j}])+2\pi\frac{1-\frac{2m}{r}}{-\partial_{u}r}\Big(\partial_{u}\log(\Omega^{2})-2\frac{\partial_{u}r}{r}\Big)(r^{2}T_{uv}[f_{\varepsilon j}])\Bigg\}(\bar{u},\bar{v})\Bigg]\,d\bar{v}d\bar{u}\Bigg|=\nonumber \\
 & =\Bigg|\int_{v_{\varepsilon,j}^{(n)}-h_{\varepsilon,j}}^{u}\int_{v}^{v_{\varepsilon,i}^{(n)}-h_{\varepsilon,i}}\Bigg(\frac{1}{r(\bar{u},v)}\Bigg[\exp\Big(-\int_{v}^{\bar{v}}4\pi\frac{rT_{vv}[f_{\varepsilon}]}{\partial_{v}r}(\bar{u},\hat{v})\,d\hat{v}\Big)\cdot2\pi\frac{1-\frac{2m}{r}}{-\partial_{u}r}r^{2}T_{uv}[f_{\varepsilon j}](\bar{u},\bar{v})\times\nonumber \\
 & \hphantom{-e^{\int^{u}}}\times\Bigg\{\int_{v}^{\bar{v}}\partial_{u}\Big(\frac{4\pi}{r\partial_{v}r}\Big)r^{2}T_{vv}[f_{\varepsilon}](\bar{u},\hat{v})\,d\hat{v}+\int_{v}^{\bar{v}}\frac{4\pi}{r\partial_{v}r}\partial_{u}(r^{2}T_{vv}[f_{\varepsilon}])(\bar{u},\hat{v})\,d\hat{v}-\nonumber \\
 & \hphantom{-e^{\int_{v_{\varepsilon,j}^{(n)}-h_{\varepsilon,j}}^{u}}\times\Bigg\{}-\partial_{u}\log\Big(\frac{1-\frac{2m}{r}}{-\partial_{u}r}\Big)(\bar{u},\bar{v})+\Big(\partial_{u}\log(\Omega^{2})-2\frac{\partial_{u}r}{r}\Big)(\bar{u},\bar{v})\Bigg\}\Bigg]+\nonumber \\
 & \hphantom{-e^{\int_{v_{\varepsilon,j}^{(n)}-h_{\varepsilon,j}}^{u}}\times\Bigg\{}+\frac{\partial_{u}r}{r^{2}}(\bar{u},v)\exp\Big(-\int_{v}^{\bar{v}}4\pi\frac{rT_{vv}[f_{\varepsilon}]}{\partial_{v}r}(\bar{u},\hat{v})\,d\hat{v}\Big)\cdot2\pi\frac{1-\frac{2m}{r}}{-\partial_{u}r}r^{2}T_{uv}[f_{\varepsilon i}](\bar{u},\bar{v})\Bigg)\,d\bar{v}d\bar{u}\Bigg|.\nonumber 
\end{align}
We will estimate the right hand side of (\ref{eq:ErrorFinalRelation-1})
similarly as we did for (\ref{eq:ErrorFinalRelation}): 

\begin{itemize}

\item Using the bounds (\ref{eq:UpperBoundMuBootstrapDomain}), (\ref{eq:EstimateGeometryLikeC0}),
(\ref{eq:UpperBoundTuuTvvFi}), (\ref{eq:OnlyOneComponentTube}) (and
the fact that $i\ge j+1$), we can estimate:
\begin{align}
\sup_{(u,v)\in\mathcal{W}_{\nearrow}[n;i,j],\text{ }(\bar{u},\bar{v})\in[u_{n;i-1,j}^{(-)},u]\times[v,v_{n;i,j}^{(-)}]} & \Big|\int_{v}^{\bar{v}}4\pi\frac{rT_{vv}[f_{\varepsilon}]}{\partial_{v}r}(\bar{u},\hat{v})\,d\hat{v}\Big|\le\label{eq:BoundExponentialFactorSummand-1}\\
 & \le\int_{v_{n;i-1,j}^{(-)}+(\rho_{\varepsilon}^{-\frac{7}{8}}+1)h_{\varepsilon,i-1}}^{v_{n;i,j}^{(-)}}\exp\big(\exp(4\sigma_{\varepsilon}^{-5})\big)\cdot\frac{(\varepsilon^{(j)})^{4}}{r^{5}(u_{n;i,j}^{(+)},v)}(-\Lambda)^{-2}\,dv\le\nonumber \\
 & \le\exp\big(\exp(\sigma_{\varepsilon}^{-6})\big)\cdot\frac{(\varepsilon^{(j)})^{4}}{r^{4}(u_{n;i,j}^{(+)},v_{n;i-1,j}^{(-)}+(\rho_{\varepsilon}^{-\frac{7}{8}}+1)h_{\varepsilon,i-1})}(-\Lambda)^{-2}\le\nonumber \\
 & \le\exp\big(\exp(2\sigma_{\varepsilon}^{-6})\big)\cdot\rho_{\varepsilon}^{\frac{7}{2}}\le\nonumber \\
 & \le1.\nonumber 
\end{align}

\item Using the bounds (\ref{eq:EstimateGeometryLikeC0}) and (\ref{eq:UpperBoundTuvFi}),
we can estimate: 
\begin{align}
\sup_{\mathcal{W}_{\nearrow}[n;i,j]}\Big(2\pi\frac{1-\frac{2m}{r}}{\partial_{v}r}r^{2}T_{uv}[f_{\varepsilon j}]\Big) & \le\exp\big(\exp(\sigma_{\varepsilon}^{-6})\big)\sup_{\mathcal{W}_{\nearrow}[n;i,j]}\frac{(\varepsilon^{(j)})^{2}}{r^{2}}(-\Lambda)^{-1}\le\\
 & \le\exp\big(\exp(\sigma_{\varepsilon}^{-6})\big)\frac{(\varepsilon^{(j)})^{2}}{r^{2}(u_{n;i,j}^{(+)},v_{n;i-1,j}^{(-)}+(\rho_{\varepsilon}^{-\frac{7}{8}}+1)h_{\varepsilon,i-1})}(-\Lambda)^{-1}\le\nonumber \\
 & \le\exp\big(\exp(2\sigma_{\varepsilon}^{-6})\big)\rho_{\varepsilon}^{\frac{7}{4}}\le\nonumber \\
 & \le\rho_{\varepsilon}^{\frac{3}{2}}.\nonumber 
\end{align}

\item Using equation (\ref{eq:EquationROutside}) and the bounds
(\ref{eq:UpperBoundMuBootstrapDomain}), (\ref{eq:EstimateGeometryLikeC0}),
(\ref{eq:OnlyOneComponentTube}), (\ref{eq:UpperBoundTuuTvvFi}),
(\ref{eq:UpperBoundTuvFi}), and (\ref{eq:LowerBoundRW}) (as well
as the trivial bound $a_{\varepsilon k}\le1$), we can estimate for
all $(u,v)\in\mathcal{W}_{\nearrow}[n;i,j]$:
\begin{align}
\sup_{(\bar{u},\bar{v})\in[u_{n;i,j}^{(-)},u]\times[v,v_{n;i,j}^{(-)}-\rho_{\varepsilon}^{-\frac{7}{8}}h_{\varepsilon,i-1}]}\Big|\int_{v}^{\bar{v}} & \partial_{u}\Big(\frac{4\pi}{r\partial_{v}r}\Big)r^{2}T_{vv}[f_{\varepsilon}](\bar{u},\hat{v})\,d\hat{v}\Big|=\label{eq:BoundFirstSummand-1}\\
=\sup_{(\bar{u},\bar{v})\in[u_{n;i-,j}^{(-)},u]\times[v,v_{n;i,j}^{(-)}-\rho_{\varepsilon}^{-\frac{7}{8}}h_{\varepsilon,i-1}]} & \Big|\int_{v}^{\bar{v}}4\pi\Big(-\frac{\partial_{u}r}{r^{2}\partial_{v}r}-\frac{\partial_{u}\partial_{v}r}{r(\partial_{v}r)^{2}}\Big)r^{2}T_{vv}[f_{\varepsilon}](\bar{u},\hat{v})\,d\hat{v}\Big|=\nonumber \\
=\sup_{(\bar{u},\bar{v})\in[u_{n;i,j}^{(-)},u]\times[v,v_{n;i,j}^{(-)}-\rho_{\varepsilon}^{-\frac{7}{8}}h_{\varepsilon,i-1}]} & \Big|\int_{v}^{\bar{v}}4\pi\Big(-\frac{\partial_{u}r}{r^{2}\partial_{v}r}+\frac{\frac{2\tilde{m}}{r}-\frac{2}{3}\Lambda r^{2}}{r^{2}(1-\frac{2m}{r})}\frac{-\partial_{u}r}{\partial_{v}r}+\frac{4\pi r^{2}T_{uv}}{r^{2}(\partial_{v}r)^{2}}\Big)r^{2}T_{vv}[f_{\varepsilon}](\bar{u},\hat{v})\,d\hat{v}\Big|\le\nonumber \\
 & \le\int_{v}^{v_{n;i,j}^{(-)}}\exp\big(\exp(\sigma_{\varepsilon}^{-9})\big)\frac{1}{r^{2}(u,\hat{v})}\frac{(\varepsilon^{(j)})^{4}}{r^{4}(u,\hat{v})}(-\Lambda)^{-2}\,d\hat{v}\le\nonumber \\
 & \le\exp\big(\exp(\sigma_{\varepsilon}^{-9})\big)\frac{(\varepsilon^{(j)})^{4}}{r^{5}(u,v)}(-\Lambda)^{-2}\le\nonumber \\
 & \le\exp\big(\exp(2\sigma_{\varepsilon}^{-9})\big)\rho_{\varepsilon}^{\frac{7}{2}}\frac{1}{r(u,v)}.\nonumber 
\end{align}

\item Using (\ref{eq:ConservationOfEnergyDu}) to express $\partial_{u}(r^{2}T_{vv}[f_{\varepsilon}])$
in terms of $\partial_{v}(r^{2}T_{uv}[f_{\varepsilon}])$ and integrating
by parts in $\partial_{v}$, we calculate for any $(u,v)\in\mathcal{W}_{\nearrow}[n;i,j]$
(in view of (\ref{eq:OnlyOneComponentTube}) and the bounds (\ref{eq:UpperBoundMuBootstrapDomain}),
(\ref{eq:EstimateGeometryLikeC0}), (\ref{eq:OnlyOneComponentTube}),
(\ref{eq:UpperBoundTuvFi}), (\ref{eq:LowerBoundRW}), (\ref{eq:BoundDOmegaEverywhere})
for $\partial_{v}\Omega^{2}$, (\ref{eq:BoundDDrEverywhere}) for
$\partial_{v}^{2}r$ and (\ref{eq:LowerBoundTransversalDistanceForWOutgoing})
for $dist_{\nwarrow}[\cdot]$ on $\mathcal{W}_{\nearrow}[n;i,j]$):
\begin{align}
 & \sup_{(\bar{u},\bar{v})\in[u_{n;i,j}^{(-)},u]\times[v,v_{n;i,j}^{(-)}-\rho_{\varepsilon}^{-\frac{7}{8}}h_{\varepsilon,i-1}]}\Big|\int_{v}^{\bar{v}}\frac{4\pi}{r\partial_{v}r}\partial_{u}(r^{2}T_{vv}[f_{\varepsilon}])(\bar{u},\hat{v})\,d\hat{v}\Big|=\label{eq:BoundSecondSummand-1}\\
 & \hphantom{blablabla}=\sup_{(\bar{u},\bar{v})\in[u_{n;i,j}^{(-)},u]\times[v,v_{n;i,j}^{(-)}-\rho_{\varepsilon}^{-\frac{7}{8}}h_{\varepsilon,i-1}]}\Big|\int_{v}^{\bar{v}}\frac{4\pi}{r}\frac{1}{\partial_{v}r}\Big(-\partial_{v}(r^{2}T_{uv}[f_{\varepsilon}])+\nonumber \\
 & \hphantom{blablabla}\hphantom{\sup_{(\bar{u},\bar{v})\in[u_{n;i,j}^{(-)},u]\times[v,v_{n;i,j}^{(-)}-\rho_{\varepsilon}^{-\frac{7}{8}}h_{\varepsilon,i-1}]}\Big|\int_{v}^{\bar{v}}}+\Big(\partial_{v}\log(\Omega^{2})-2\frac{\partial_{v}r}{r}\Big)(r^{2}T_{uv}[f_{\varepsilon}])\Big)(\bar{u},\hat{v})\,d\hat{v}\Big|=\nonumber\\
 & \hphantom{blablabla}=\sup_{(\bar{u},\bar{v})\in[u_{n;i,j}^{(-)},u]\times[v,v_{n;i,j}^{(-)}-\rho_{\varepsilon}^{-\frac{7}{8}}h_{\varepsilon,i-1}]}\Big|\int_{v}^{\bar{v}}4\pi\Big(\partial_{v}\big(\frac{1}{r\partial_{v}r}\big)\cdot r^{2}T_{uv}[f_{\varepsilon}]+\nonumber \\
 & \hphantom{blablabla}\hphantom{\sup_{(\bar{u},\bar{v})\in[u_{n;i,j}^{(-)},u]\times[v,v_{n;i,j}^{(-)}-\rho_{\varepsilon}^{-\frac{7}{8}}h_{\varepsilon,i-1}]}}+\frac{1}{r\partial_{v}r}\Big(\partial_{v}\log(\Omega^{2})-2\frac{\partial_{v}r}{r}\Big)(r^{2}T_{uv}[f_{\varepsilon}])\Big)(\bar{u},\hat{v})\,d\hat{v}+\nonumber\\
 & \hphantom{blablabla}\hphantom{\sup_{(\bar{u},\bar{v})\in[u_{n;i,j}^{(-)},u]\times[v,v_{n;i,j}^{(-)}-}}+\frac{4\pi}{r}\frac{1}{\partial_{v}r}\cdot r^{2}T_{uv}[f_{\varepsilon}](\bar{u},\bar{v})-\frac{4\pi}{r}\frac{1}{\partial_{v}r}\cdot r^{2}T_{uv}[f_{\varepsilon}](\bar{u},v)\Big|\le\nonumber \\
 & \hphantom{blablabla}\le\int_{v}^{v_{n;i,j}^{(-)}-\rho_{\varepsilon}^{-\frac{7}{8}}h_{\varepsilon,i-1}}\exp\big(\exp(\sigma_{\varepsilon}^{-9})\big)\frac{1}{r(u,\hat{v})}\big(\frac{1}{dist_{\nwarrow}[(u,\hat{v})]}+\frac{1}{r(u,\hat{v})}\big)\frac{(\varepsilon^{(j)})^{2}}{(-\Lambda)r^{2}(u,\hat{v})}\,d\hat{v}+\nonumber \\
 & \hphantom{blablabla}\hphantom{\sup_{(\bar{u},\bar{v})\in[u_{n;i,j}^{(-)},u]\times[v,v_{n;i,j}^{(-)}}}+\exp\big(\exp(\sigma_{\varepsilon}^{-9})\big)\frac{1}{r(u,v)}\frac{(\varepsilon^{(j)})^{2}}{(-\Lambda)r^{2}(u,v)}\le\nonumber\\
 & \hphantom{blablabla}\le\exp\big(2\exp(\sigma_{\varepsilon}^{-9})\big)\frac{1}{r(u,v)}\Big(1+\int_{v_{n;i-1,j}^{(+)}+\rho_{\varepsilon}^{-\frac{7}{8}}h_{\varepsilon,i-1}}^{v_{n;i,j}^{(-)}-\rho_{\varepsilon}^{-\frac{7}{8}}h_{\varepsilon,i-1}}\frac{1}{dist_{\nwarrow}[(u,\hat{v})]}\,d\hat{v}\Big)\frac{(\varepsilon^{(j)})^{2}}{(-\Lambda)r^{2}(u,\hat{v})}+\nonumber \\
 & \hphantom{blablabla}\hphantom{\sup_{(\bar{u},\bar{v})\in[u_{n;i,j}^{(-)},u]\times[v,v_{n;i,j}^{(-)}}}+\exp\big(\exp(\sigma_{\varepsilon}^{-9})\big)\frac{1}{r(u,v)}\frac{(\varepsilon^{(j)})^{2}}{(-\Lambda)r^{2}(u,v)}\le\nonumber\\
 & \hphantom{blablabla}\le\exp\big(4\exp(\sigma_{\varepsilon}^{-9})\big)\frac{1}{r(u,v)}\Big(1+\exp(\sigma_{\varepsilon}^{-7})\rho_{\varepsilon}^{-\frac{1}{8}}\Big)\frac{(\varepsilon^{(j)})^{2}}{(-\Lambda)r^{2}(u,\hat{v})}+\nonumber \\
 & \hphantom{blablabla}\hphantom{\sup_{(\bar{u},\bar{v})\in[u_{n;i,j}^{(-)},u]\times[v,v_{n;i,j}^{(-)}}}+\exp\big(\exp(\sigma_{\varepsilon}^{-9})\big)\frac{1}{r(u,v)}\frac{(\varepsilon^{(j)})^{2}}{(-\Lambda)r^{2}(u,v)}\le\nonumber\\
 & \hphantom{blablabla}\le\exp\big(\exp(2\sigma_{\varepsilon}^{-9})\big)\rho_{\varepsilon}^{\frac{13}{8}}\frac{1}{r(u,v)}.\nonumber 
\end{align}

\item Using the the relation (\ref{eq:TildeUMaza}) for $\partial_{u}\tilde{m}$,
the estimate (\ref{eq:BoundDDrEverywhere}) for $\partial_{u}^{2}r$,
as well as the bounds (\ref{eq:EstimateGeometryLikeC0}), (\ref{eq:UpperBoundTuuTvvFi}),
(\ref{eq:UpperBoundTuvFi}) and (\ref{eq:UpperBoundMuBootstrapDomain}),
we can estimate for any $(\bar{u},\bar{v})\in\mathcal{W}_{\nearrow}[n;i,j]$:
\begin{equation}
\Big|\partial_{u}\log\Big(\frac{1-\frac{2m}{r}}{-\partial_{u}r}\Big)(\bar{u},\bar{v})\Big|\le\exp\big(\exp(\sigma_{\varepsilon}^{-6})\big)\Big(\frac{1}{dist_{\nwarrow}[(\bar{u},\bar{v})]}+\frac{1}{r(\bar{u},\bar{v})}\Big).\label{eq:BoundThirdSummand-1}
\end{equation}

\item Using the estimate (\ref{eq:BoundDOmegaEverywhere}) for $\partial_{u}\Omega^{2}$,
as well as the bound (\ref{eq:EstimateGeometryLikeC0}),, we can estimate
for any $(\bar{u},\bar{v})\in\mathcal{W}_{\nearrow}[n;i,j]$:
\begin{equation}
\Big|\partial_{v}\log(\Omega^{2})-2\frac{\partial_{v}r}{r}\Big|(\bar{u},\bar{v})\le\exp\big(\exp(\sigma_{\varepsilon}^{-6})\big)\Big(\frac{1}{dist_{\nwarrow}[(\bar{u},\bar{v})]}+\frac{1}{r(\bar{u},\bar{v})}\Big).\label{eq:BoundFourthSummand-1}
\end{equation}

\end{itemize}

Using the estimates (\ref{eq:BoundExponentialFactorSummand-1})\textendash (\ref{eq:BoundFourthSummand-1})
(together with (\ref{eq:EstimateGeometryLikeC0}), (\ref{eq:UpperBoundTuvFi}),
(\ref{eq:OnlyOneComponentTube}), (\ref{eq:LowerBoundTransversalDistanceForWOutgoing}),
(\ref{eq:LowerBoundRW}) and the relation of the parameters $\varepsilon$,
$\rho_{\varepsilon}$, $\delta_{\varepsilon}$ and $\sigma_{\varepsilon}$)
to bound the right hand side of (\ref{eq:ErrorFinalRelation-1}),
we therefore obtain for any $(u,v)\in\mathcal{W}_{\nearrow}[n;i,j]$:
\begin{align}
\Big|\int_{u_{n;i,j}^{(-)}}^{u} & \frac{\mathfrak{Err}_{\nearrow}^{(0)}[n;i.j]}{r}(\bar{u},v)\,d\bar{u}\Big|\le\label{eq:FinalIngoingErrorEstimate-1}\\
 & \le\int_{v_{\varepsilon,j}^{(n)}-h_{\varepsilon,j}}^{u}\int_{v}^{v_{\varepsilon,i}^{(n)}-h_{\varepsilon,i-1}}\Bigg(\frac{1}{r(\bar{u},v)}\rho_{\varepsilon}^{\frac{3}{2}}\Bigg\{\exp\big(\exp(2\sigma_{\varepsilon}^{-9})\big)\rho_{\varepsilon}^{\frac{7}{2}}\frac{1}{r(u,v)}+\nonumber\\
 & \hphantom{\le\int_{v_{\varepsilon,j}^{(n)}-h_{\varepsilon,j}}^{u}\int_{v}^{v_{\varepsilon,i}^{(n)}-h_{\varepsilon,i-1}}\Bigg(}+\exp\big(\exp(2\sigma_{\varepsilon}^{-9})\big)\rho_{\varepsilon}^{\frac{13}{8}}\frac{1}{r(u,v)}+\exp\big(\exp(\sigma_{\varepsilon}^{-6})\big)\Big(\frac{1}{dist_{\nwarrow}[(\bar{u},\bar{v})]}+\frac{1}{r(\bar{u},\bar{v})}\Big)\Bigg\}+\nonumber \\
 & \hphantom{\le\int_{v_{\varepsilon,j}^{(n)}-h_{\varepsilon,j}}^{u}\int_{v}^{v_{\varepsilon,i}^{(n)}-h_{\varepsilon,i-1}}\Bigg(}+e^{\sigma_{\varepsilon}^{-6}}\rho_{\varepsilon}^{\frac{3}{2}}\frac{1}{r^{2}(\bar{u},v)}\Bigg)\,d\bar{v}d\bar{u}\le\nonumber\\
 & \le\exp\big(\exp(3\sigma_{\varepsilon}^{-9})\big)\rho_{\varepsilon}^{\frac{3}{2}}h_{\varepsilon,j}\cdot\rho_{\varepsilon}^{-1}h_{\varepsilon,i-1}\cdot\Big(\rho_{\varepsilon}^{\frac{7}{8}}\frac{\sqrt{-\Lambda}}{\varepsilon^{(i-1)}}+\frac{1}{r(u,v)}\Big)\frac{1}{r(u,v)}\le\nonumber \\
 & \le\exp\big(\exp(4\sigma_{\varepsilon}^{-9})\big)\rho_{\varepsilon}^{\frac{1}{2}}\varepsilon^{(j)}\varepsilon^{(i-1)}\cdot\Big(\rho_{\varepsilon}^{\frac{7}{8}}\frac{1}{\varepsilon^{(i-1)}}+\rho_{\varepsilon}^{\frac{7}{8}}\frac{1}{\varepsilon^{(j)}}\Big)(-\Lambda)^{-\frac{1}{2}}\frac{1}{r(u,v)}\le\nonumber \\
 & \le\exp\big(\exp(4\sigma_{\varepsilon}^{-9})\big)\rho_{\varepsilon}^{\frac{11}{8}}\max\{\varepsilon^{(i-1)},\varepsilon^{(j)}\}(-\Lambda)^{-\frac{1}{2}}\frac{1}{r(u,v)}\le\nonumber \\
 & \le\rho_{\varepsilon}^{\frac{5}{4}}\frac{\varepsilon^{(j)}}{\sqrt{-\Lambda}}\frac{1}{r(u,v)}.\nonumber 
\end{align}

Returning to (\ref{eq:1/rEnergyUsefulFormulaOutgoing}) and using
(\ref{eq:FinalIngoingErrorEstimate-1}) to estimate the last term
in the right hand side, we infer that, for any $(u,v)\in\mathcal{W}_{\nearrow}[n;i,j]$:
\begin{equation}
\int_{u_{n;i,j}^{(-)}}^{u}\frac{E_{\nearrow}[n;i.j]}{r}(\bar{u},v)\,d\bar{u}=\big(1+O(\rho_{\varepsilon}^{\frac{1}{2}})\big)\cdot\frac{\mathcal{E}_{\nearrow}[n;i.j](u,v_{n;i,j}^{(-)})+O\Big(\rho_{\varepsilon}^{\frac{5}{4}}\frac{\varepsilon^{(j)}}{\sqrt{-\Lambda}}\Big)}{r(u,v)}.\label{eq:Final1/rEnergyOutgoing}
\end{equation}
Similarly, estimating $\int_{v_{n;i,j}^{(-)}}^{v}\frac{\mathfrak{Err}_{\nwarrow}^{(0)}[n;i.j]}{r}(u,\bar{v})\,d\bar{v}$
similarly as we did for $\int_{u_{n;i,j}^{(-)}}^{u}\frac{\mathfrak{Err}_{\nearrow}^{(0)}[n;i.j]}{r}(\bar{u},v)\,d\bar{u}$,
we infer from (\ref{eq:1/rEnergyUsefulFormulaIngoing}) that, for
all $(u,v)\in\mathcal{W}_{\nwarrow}[n;i,j]$: 
\begin{equation}
\int_{v_{n;i,j}^{(-)}}^{v}\frac{E_{\nwarrow}[n;i.j]}{r}(u,\bar{v})\,d\bar{v}=\big(1+O(\rho_{\varepsilon}^{\frac{1}{2}})\big)\cdot\frac{\mathcal{E}_{\nwarrow}[n;i.j](u_{n;i,j}^{(-)},v)+O\Big(\rho_{\varepsilon}^{\frac{5}{4}}\frac{\varepsilon^{(i)}}{\sqrt{-\Lambda}}\Big)}{r(u,v)}.\label{eq:Final1/rEnergyIngoing}
\end{equation}

Substituting (\ref{eq:Final1/rEnergyOutgoing}) in (\ref{eq:IngoingRSeparationBeforeAfterFinalGeneral})
and using (\ref{eq:NonTrivialUpperBoundMuij-1}), (\ref{eq:FormulaForFinalEnergiesFromDensities})
and the fact that $\Lambda r^{2}=O(\varepsilon)$ on $\mathcal{W}_{\nearrow}[n;i,j]$
when $i>j$, we infer that
\begin{equation}
\mathfrak{D}r_{\nwarrow}^{(+)}[n;i,j]=\int_{\text{ }v_{n;i-1,j}^{(+)}+\rho_{\varepsilon}^{-\frac{7}{8}}h_{\varepsilon,i-1}}^{\text{ }v_{n;i,j}^{(-)}-\rho_{\varepsilon}^{-\frac{7}{8}}h_{\varepsilon,i-1}}\exp\Bigg(-\big(1+O(\rho_{\varepsilon}^{\frac{3}{4}})\big)\cdot\frac{2\mathcal{E}_{\nearrow}^{(-)}[n;i,j]+O\Big(\rho_{\varepsilon}^{\frac{5}{4}}\frac{\varepsilon^{(j)}}{\sqrt{-\Lambda}}\Big)}{r(u_{n;i,j}^{(+)},\bar{v})}\Bigg)\cdot\frac{\partial_{v}r}{1-\frac{2m}{r}}(u_{n;i,j}^{(-)},\bar{v})\,d\bar{v}.\label{eq:IngoingRSeparationBeforeAfterFinalGeneral-1}
\end{equation}
Similarly, from (\ref{eq:OutgoingRSeparationBeforeAfterFinalGeneral})
we infer that 
\begin{equation}
\mathfrak{D}r_{\nearrow}^{(+)}[n;i,j]=\int_{\text{ }u_{n;i,j-1}^{(+)}+\rho_{\varepsilon}^{-\frac{7}{8}}h_{\varepsilon,j-1}}^{\text{ }u_{n;i,j}^{(-)}-\rho_{\varepsilon}^{-\frac{7}{8}}h_{\varepsilon,j-1}}\exp\Big(\big(1+O(\rho_{\varepsilon}^{\frac{3}{4}})\big)\cdot\frac{2\mathcal{E}_{\nwarrow}^{(-)}[n;i,j]+O\Big(\rho_{\varepsilon}^{\frac{5}{4}}\frac{\varepsilon^{(i)}}{\sqrt{-\Lambda}}\Big)}{r(\bar{u},v_{n;i,j}^{(-)})}\Big)\cdot\frac{-\partial_{u}r}{1-\frac{2m}{r}}(\bar{u},v_{n;i,j}^{(-)})\,d\bar{u}.\label{eq:OutgoingRSeparationBeforeAfterFinalGeneral-1}
\end{equation}

\begin{itemize}

\item In the case when $i=j+1$, using the bounds (\ref{eq:EstimateGeometryLikeC0}),
(\ref{eq:UpperBoundTuuTvvFi}) and (\ref{eq:LowerBoundRW}), we can
trivially estimate 
\begin{align}
\sup_{\bar{v}\in[v_{n;i-1,j}^{(+)}+\rho_{\varepsilon}^{-\frac{7}{8}}h_{\varepsilon,i-1},v_{n;i,j}^{(-)}-\rho_{\varepsilon}^{-\frac{7}{8}}h_{\varepsilon,i-1}]} & \frac{\mathcal{E}_{\nearrow}^{(-)}[n;i,j]+O\Big(\rho_{\varepsilon}^{\frac{5}{4}}\frac{\varepsilon^{(j)}}{\sqrt{-\Lambda}}\Big)}{r(u_{n;i,j}^{(+)},\bar{v})}\label{eq:TrivialUpperBoundOutgoingExponent}\\
 & \le\frac{\int_{u_{n;i,j}^{(-)}}^{u_{n;i,j}^{(+)}}\exp\big(\exp(\sigma_{\varepsilon}^{-6})\big)\,du+O\Big(\rho_{\varepsilon}^{\frac{5}{4}}\frac{\varepsilon^{(j)}}{\sqrt{-\Lambda}}\Big)}{e^{-\sigma_{\varepsilon}^{-6}}\rho_{\varepsilon}^{-\frac{7}{8}}\frac{\varepsilon^{(j)}}{\sqrt{-\Lambda}}}\nonumber\\
 & \le\exp\big(\exp(2\sigma_{\varepsilon}^{-6})\big)\rho_{\varepsilon}^{\frac{7}{8}}\nonumber\\
 & \le\rho_{\varepsilon}^{\frac{3}{4}}.\nonumber
\end{align}
From (\ref{eq:IngoingRSeparationBeforeAfterFinalGeneral-1}), we therefore
infer that 
\begin{align}
\mathfrak{D}r_{\nwarrow}^{(+)}[n;i,j] & =\int_{\text{ }v_{n;i-1,j}^{(+)}+\rho_{\varepsilon}^{-\frac{7}{8}}h_{\varepsilon,i-1}}^{\text{ }v_{n;i,j}^{(-)}-\rho_{\varepsilon}^{-\frac{7}{8}}h_{\varepsilon,i-1}}\exp\big(O(\rho_{\varepsilon}^{\frac{3}{4}})\big)\cdot\frac{\partial_{v}r}{1-\frac{2m}{r}}(u_{n;i,j}^{(-)},\bar{v})\,d\bar{v}=\label{eq:FinalChangeInIngoingRDifferenei=00003Dj+1}\\
 & =\int_{\text{ }v_{n;i-1,j}^{(+)}+\rho_{\varepsilon}^{-\frac{7}{8}}h_{\varepsilon,i-1}}^{\text{ }v_{n;i,j}^{(-)}-\rho_{\varepsilon}^{-\frac{7}{8}}h_{\varepsilon,i-1}}\big(1+O(\rho_{\varepsilon}^{\frac{3}{4}})\big)\frac{\partial_{v}r}{1-\frac{2m}{r}}(u_{n;i,j}^{(-)},\bar{v})\,d\bar{v}=\nonumber \\
 & =\big(1+O(\rho_{\varepsilon}^{\frac{3}{4}})\big)\mathfrak{D}r_{\nwarrow}^{(-)}[n;i,j].\nonumber 
\end{align}
As a result, the relation (\ref{eq:DecreaseIngoingRSeparationCloseToAxisi=00003Dj+1})
follows readily for $i=j+1$.

\item In the case when $i>j+1$, we can trivially estimate (using
(\ref{eq:BoundsRIntersectionRegioni>j}) and the relation (\ref{eq:RecursiveEi})
between the $\varepsilon^{(k)}$'s) that, for any $\bar{v}\in[v_{n;i-1,j}^{(+)}+\rho_{\varepsilon}^{-\frac{7}{8}}h_{\varepsilon,i-1},v_{n;i,j}^{(-)}-\rho_{\varepsilon}^{-\frac{7}{8}}h_{\varepsilon,i-1}]$:
\begin{equation}
1-\frac{r(u_{n;i,j}^{(+)},\bar{v})}{r_{n;i,j}}=\frac{r(u_{n;i,j}^{(+)},v_{n;i,j}^{(-)})-r(u_{n;i,j}^{(+)},\bar{v})}{r(u_{n;i,j}^{(+)},v_{n;i,j}^{(-)})}=O(\varepsilon).
\end{equation}
Thus, from (\ref{eq:IngoingRSeparationBeforeAfterFinalGeneral-1}),
we infer that 
\begin{align}
\mathfrak{D}r_{\nwarrow}^{(+)}[n;i,j] & =\int_{\text{ }v_{n;i-1,j}^{(+)}+\rho_{\varepsilon}^{-\frac{7}{8}}h_{\varepsilon,i-1}}^{\text{ }v_{n;i,j}^{(-)}-\rho_{\varepsilon}^{-\frac{7}{8}}h_{\varepsilon,i-1}}\exp\Bigg(-\big(1+O(\rho_{\varepsilon}^{\frac{3}{4}})+O(\varepsilon)\big)\cdot\frac{2\mathcal{E}_{\nearrow}^{(-)}[n;i,j]+O\Big(\rho_{\varepsilon}^{\frac{5}{4}}\frac{\varepsilon^{(j)}}{\sqrt{-\Lambda}}\Big)}{r_{n;i,j}}\Bigg)\times\label{eq:FinalChangeIngoingInRDifferencei>j+1}\\
 & \hphantom{\int_{\text{ }v_{n;i-1,j}^{(+)}+\rho_{\varepsilon}^{-\frac{7}{8}}h_{\varepsilon,i-1}}^{\text{ }v_{n;i,j}^{(-)}-\rho_{\varepsilon}^{-\frac{7}{8}}h_{\varepsilon,i-1}}\exp}\times\frac{\partial_{v}r}{1-\frac{2m}{r}}(u_{n;i,j}^{(-)},\bar{v})\,d\bar{v}=\nonumber\\
 & =\exp\Bigg(-\big(1+O(\rho_{\varepsilon}^{\frac{3}{4}})\big)\cdot\frac{2\mathcal{E}_{\nearrow}^{(-)}[n;i,j]+O\Big(\rho_{\varepsilon}^{\frac{5}{4}}\frac{\varepsilon^{(j)}}{\sqrt{-\Lambda}}\Big)}{r_{n;i,j}}\Bigg)\cdot\mathfrak{D}r_{\nwarrow}^{(-)}[n;i,j].\nonumber 
\end{align}
From (\ref{eq:FinalChangeIngoingInRDifferencei>j+1}) and (\ref{eq:BoundsRIntersectionRegioni>j})
(as well as the upper bound (\ref{eq:TrivialUpperBoundOutgoing})
for $\frac{2\mathcal{E}_{\nearrow}^{(-)}[n;i,j]}{r_{n;i,j}}$), the
relation (\ref{eq:DecreaseIngoingRSeparationCloseToAxisi>j+1}) follows
readily for $i>j+1$.

\item In all the cases when $i>j$, we can readily estimate using
(\ref{eq:EstimateGeometryLikeC0}) and (\ref{eq:UpperBoundTuuTvvFi}):
\begin{equation}
\mathcal{E}_{\nearrow}^{(-)}[n;i,j]\le\exp\big(\exp(2\sigma_{\varepsilon}^{-6})\big)\varepsilon^{(i)}(-\Lambda)^{-\frac{1}{2}}
\end{equation}
and 
\begin{equation}
\inf_{\bar{u}\in[\text{ }u_{n;i,j-1}^{(+)}+\rho_{\varepsilon}^{-\frac{7}{8}}h_{\varepsilon,j-1},\text{ }u_{n;i,j}^{(-)}+\rho_{\varepsilon}^{-\frac{7}{8}}h_{\varepsilon,j-1}]}r(\bar{u},v_{n;i,j}^{(-)})\ge e^{-\sigma_{\varepsilon}^{-6}}\rho_{\varepsilon}^{-\frac{7}{8}}\varepsilon^{(j-1)}(-\Lambda)^{-\frac{1}{2}}.
\end{equation}
Thus, from (\ref{eq:OutgoingRSeparationBeforeAfterFinalGeneral-1})
(and the relation (\ref{eq:RecursiveEi}) between the $\varepsilon^{(k)}$'s)
we infer that 
\begin{align}
\mathfrak{D}r_{\nearrow}^{(+)}[n;i,j] & =\int_{\text{ }u_{n;i,j-1}^{(+)}+\rho_{\varepsilon}^{-\frac{7}{8}}h_{\varepsilon,j-1}}^{\text{ }u_{n;i,j}^{(-)}-\rho_{\varepsilon}^{-\frac{7}{8}}h_{\varepsilon,j-1}}\exp\Big(O\Big(\frac{\exp\big(\exp(2\sigma_{\varepsilon}^{-6})\big)\varepsilon^{(i)}}{e^{-\sigma_{\varepsilon}^{-6}}\rho_{\varepsilon}^{-\frac{7}{8}}\varepsilon^{(j-1)}}\Big)\Big)\cdot\frac{-\partial_{u}r}{1-\frac{2m}{r}}(\bar{u},v_{n;i,j}^{(-)})\,d\bar{u}=\label{eq:FinalChangeInOutgoingRDifferencei>j}\\
 & =\int_{\text{ }u_{n;i,j-1}^{(+)}+\rho_{\varepsilon}^{-\frac{7}{8}}h_{\varepsilon,j-1}}^{\text{ }u_{n;i,j}^{(-)}-\rho_{\varepsilon}^{-\frac{7}{8}}h_{\varepsilon,j-1}}\exp\Big(O(\varepsilon)\Big)\cdot\frac{-\partial_{u}r}{1-\frac{2m}{r}}(\bar{u},v_{n;i,j}^{(-)})\,d\bar{u}=\nonumber \\
 & =(1+O(\varepsilon))\cdot\mathfrak{D}r_{\nearrow}^{(-)}[n;i,j].\nonumber 
\end{align}
In particular, (\ref{eq:IncreaseOutgoingRSeparationCloseToAxis})
follows from (\ref{eq:FinalChangeInOutgoingRDifferencei>j}).

\end{itemize}

\medskip{}

\noindent \emph{The case $i<j$: Proof of (\ref{eq:OutgoingRSeparationCloseToInfinity})\textendash (\ref{eq:IngoingRSeparationCloseToInfinity}).}
In the case when $i<j$, the proof of (\ref{eq:OutgoingRSeparationCloseToInfinity})\textendash (\ref{eq:IngoingRSeparationCloseToInfinity})
follows by repaeting exactly the same steps as for the proof of (\ref{eq:IncreaseOutgoingRSeparationCloseToAxis})\textendash (\ref{eq:DecreaseIngoingRSeparationCloseToAxisi>j+1}),
but using the bound (\ref{eq:LowerBoundRWNearInfinity}) in place
of (\ref{eq:LowerBoundRW}). This results in several simplifications
and improvements in the bounds of the various error terms (compare
with the proof of (\ref{eq:IngoingEnergyCloseToInfinity})\textendash (\ref{eq:OutgoingEnergyCloseToInfinity})
in relation to (\ref{eq:IncreaseIngoingEnergyCloseToAxis})\textendash (\ref{eq:DecreaseOutgoingEnergyCloseToAxis})):
using (\ref{eq:LowerBoundRWNearInfinity}), it readily follows that
the first term in the right hand side of (\ref{eq:1/rEnergyUsefulFormulaOutgoing})\textendash (\ref{eq:1/rEnergyUsefulFormulaIngoing})
is of order $O(\varepsilon^{\frac{1}{2}})$, while the right hand
side of (\ref{eq:ErrorFinalRelation-1}) is of order $O(\varepsilon)$.
As a result, using once more the bound (\ref{eq:LowerBoundRWNearInfinity})
for the $\Lambda r^{2}$ terms, one infers that the arguments of the
exponentials in (\ref{eq:IngoingRSeparationBeforeAfterFinalGeneral})\textendash (\ref{eq:OutgoingRSeparationBeforeAfterFinalGeneral})
are of order $O(\varepsilon)$, therefore obtaining (\ref{eq:OutgoingRSeparationCloseToInfinity})\textendash (\ref{eq:IngoingRSeparationCloseToInfinity}).
We will omit the tedious details.

\medskip{}

\noindent \emph{Proof of (\ref{eq:SpecialRSeparationCloseToInfinity})\textendash (\ref{eq:SpecialRSeparationCloseToAxis}).}
The proof of (\ref{eq:SpecialRSeparationCloseToInfinity}) follows
by repeating exactly the same steps as for the proof of (\ref{eq:DecreaseIngoingRSeparationCloseToAxisi=00003Dj+1}),
while the proof of (\ref{eq:SpecialRSeparationCloseToAxis}) follows
exactly in the same way as the proof of (\ref{eq:IngoingRSeparationCloseToInfinity}).
We will omit the relevant details.

\medskip{}

The relations (\ref{eq:IncreaseOutgoingRSeparationCloseToAxis})\textendash (\ref{eq:SpecialRSeparationCloseToAxis})
for $\widetilde{\mathfrak{D}}r_{\nwarrow}^{(\pm)}[n;i,j]$, $\widetilde{\mathfrak{D}}r_{\nearrow}^{(\pm)}[n;i,j]$
in place of $\mathfrak{D}r_{\nwarrow}^{(\pm)}[n;i,j]$, $\mathfrak{D}r_{\nearrow}^{(\pm)}[n;i,j]$
follow by repeating exactly the same steps, after replacing $\sigma_{\varepsilon}$,
$h_{\varepsilon,i}$, $\mathcal{V}_{i}^{(n)}$, $\mathcal{R}_{i;j}^{(n)}$
with $\delta_{\varepsilon}$, $\tilde{h}_{\varepsilon,i}$, $\widetilde{\mathcal{V}}_{i}^{(n)}$,
$\widetilde{\mathcal{R}}_{i;j}^{(n)}$ respectively, in all the expressions
above and using (\ref{eq:EstimateGeometryLikeC0LargerDomain}) in
place of (\ref{eq:EstimateGeometryLikeC0}).
\end{proof}

\subsection{\label{subsec:The-instability-mechanism}The instability mechanism:
Energy growth for the Vlasov beams}

In this section, we will use Propositions \ref{prop:EnergyChangeInteraction}
and \ref{prop:RSeparationChangeInteraction} in order to obtain quantitative
control on the total change in the energy content and the geometric
separation of the beams $\mathcal{V}_{i}^{(n)}$ between two successive
reflections off $\mathcal{I}_{\varepsilon}$. To this end, we will
first introduce the quantities $\mu_{i}[n]$, $\mathcal{E}_{i}[n]$
and $R_{i}[n]$, which are determined by a recursive system of relations
and will be later shown to approximate sufficiently $\frac{2\mathcal{E}_{\nwarrow}^{(-)}[n;i,0]}{\mathfrak{D}r_{\nwarrow}^{(-)}[n;i,0]}$,
$\mathcal{E}_{\nwarrow}^{(-)}[n;i,0]$ and $\mathfrak{D}r_{\nwarrow}^{(-)}[n;i,0]$,
respectively:
\begin{defn}
\label{def:RecursiveSystemMuER} For any $\varepsilon\in(0,\varepsilon_{1}]$,
let us define the sequences $\mu_{i}:\mathbb{N}\rightarrow(0,+\infty)$,
$0\le i\le N_{\varepsilon}-1$, by the recursive relations
\begin{equation}
\mu_{i}[n+1]=\mu_{i}[n]\cdot\exp\Big(2\sum_{j=0}^{i-1}\mu_{j}[n+1]\Big),\label{eq:RecursiveFormulaMu}
\end{equation}
with initial conditions 
\begin{equation}
\mu_{i}[0]=\frac{2\mathcal{E}_{\nwarrow}^{(-)}[0;i,0]}{\mathfrak{D}r_{\nwarrow}^{(-)}[0;i+1,0]}.\label{eq:InitialMuI}
\end{equation}

We will also define$\mathcal{E}_{i}:\mathbb{N}\rightarrow(0,+\infty)$
(for $0\le i\le N_{\varepsilon}$) and $R_{i}:\mathbb{N}\rightarrow(0,+\infty)$
(for $1\le i\le N_{\varepsilon}$) by the following recursive system
of relations:
\begin{align}
\mathcal{E}_{i}[n+1] & =\mathcal{E}_{i}[n]\cdot\exp\Big(\sum_{j=0}^{i-1}\mu_{j}[n+1]\Big),\label{eq:FormulaRecursiveEnergy}\\
R_{i}[n+1] & =R_{i}[n]\cdot\exp\Big(-\sum_{j=0}^{i-2}\mu_{j}[n+1]\Big),\label{eq:FormulaRecursiveDr}
\end{align}
with initial conditions 
\begin{align}
\mathcal{E}_{i}[0] & =\mathcal{E}_{\nwarrow}^{(-)}[0;i,0],\label{eq:InitialConditionsImplicitRelations}\\
R_{i}[0] & =\mathfrak{D}r_{\nwarrow}^{(-)}[0;i,0].\nonumber 
\end{align}
Notice that the quantities $\mu_{i}[n]$, $\mathcal{E}_{i}[n]$ and
$R_{i+1}[n]$ satisfy for all $0\le i\le N_{\varepsilon}-1$: 
\begin{equation}
\frac{2\mathcal{E}_{i}[n]}{R_{i+1}[n]}=\mu_{i}[n].\label{eq:QuotientERMu}
\end{equation}
\end{defn}
\begin{rem*}
The relations (\ref{eq:RecursiveFormulaMu}) and (\ref{eq:InitialMuI})
uniquely determine $\mu_{i}[n]$ for all $0\le i\le N_{\varepsilon}-1$,
$n\in\mathbb{N}$, as can be seen by arguing inductively on $i$:
For $i=0$, (\ref{eq:RecursiveFormulaMu}) yields that $\mu_{i}[n]=\mu_{i}[0]$
for all $n\in\mathbb{N}$. Provided $\mu_{\bar{i}}:\mathbb{N}\rightarrow(0,+\infty)$
has been determined for $0\le\bar{i}\le i-1$, the relation (\ref{eq:RecursiveFormulaMu})
yields 
\begin{equation}
\mu_{i}[n]=\mu_{i}[0]\cdot\exp\Big(2\sum_{\bar{n}=1}^{n}\sum_{j=0}^{i-1}\mu_{j}[\bar{n}]\Big)\label{eq:EplicitRelationMu}
\end{equation}
for all $n\in\mathbb{N}$. In particular, note that, as a consequence
of (\ref{eq:EplicitRelationMu}), for any $i>0$, 
\[
\mu_{i}[n]\xrightarrow{n\rightarrow\infty}+\infty.
\]
 
\end{rem*}
The following proposition is the main result of this section. It will
provide us with useful approximate formulas for the total change of
the energy content and the geometric separation of the beams between
two successive reflections off $\mathcal{I}_{\varepsilon}$, expressed
in terms of the quantities $\mathcal{E}_{i}[n]$ and $R_{i}[n]$.
In particular, it will be readily inferred from these formulas that,
for any $i>0$, the energy content of each beam $\mathcal{V}_{i}^{(n)}$
increases in $n$, while the geometric separation of the beams remains
under control.\footnote{Both these statements hold modulo error terms that will be shown to
be negligible after a careful choice of the initial weights $a_{\varepsilon i}$
in the next Section. }
\begin{prop}
\label{prop:TotalEnergyChange} Let $n\in\mathbb{N}$ be such that
\begin{equation}
\{0\le u\le v_{\varepsilon,0}^{(n)}-h_{\varepsilon,0}\}\cap\big\{ u<v<u+\sqrt{-\frac{3}{\Lambda}}\pi\big\}\subset\mathcal{U}_{\varepsilon}^{+}.\label{eq:NoonTrivialN}
\end{equation}
Then the following relations hold: 
\begin{align}
\mathcal{E}_{\nwarrow}^{(-)}[n;i,0] & =\mathcal{E}_{i}[n]+O\Big(\rho_{\varepsilon}^{\frac{1}{16}}\frac{\varepsilon^{(i)}}{\sqrt{-\Lambda}}\Big)\text{ for all }0\le i\le N_{\varepsilon},\label{eq:EqualEnergyFlux}\\
\mathfrak{D}r_{\nwarrow}^{(-)}[n;i,0] & =R_{i}[n]\cdot\big(1+O(\rho_{\varepsilon}^{\frac{1}{16}})\big)\text{ for all }1\le i\le N_{\varepsilon},\label{eq:EqualRSeparation}
\end{align}
where the sequences $\mathcal{E}_{i}$ and $R_{i}$ were introduced
in Definition \ref{def:RecursiveSystemMuER}. 

In addition, for any $0\le j\le N_{\varepsilon}-1$ such that 
\begin{equation}
\widetilde{\mathcal{R}}_{N_{\varepsilon};j}^{(n)}\subset\mathcal{T}_{\varepsilon}^{+},\label{eq:IntersectionDomainInTau}
\end{equation}
we have 
\begin{equation}
\widetilde{\mathcal{E}}_{\nearrow}^{(-)}[n;N_{\varepsilon},j]=\mathcal{E}_{j}[n+1]+O\Big(\rho_{\varepsilon}^{\frac{1}{17}}\frac{\varepsilon^{(j)}}{\sqrt{-\Lambda}}\Big)\label{eq:EqualEnergyFluxOutgoingTau}
\end{equation}
and, if $j\ge1$:
\begin{equation}
\widetilde{\mathfrak{D}}r_{\nearrow}^{(-)}[n;N_{\varepsilon},j]=R_{j}[n+1]\cdot\big(1+O(\rho_{\varepsilon}^{\frac{1}{17}})\big).\label{eq:EqualRSeparationOutgoingTau}
\end{equation}
Finally, if 
\[
\widetilde{\mathcal{R}}_{N_{\varepsilon};N_{\varepsilon}-1}^{(n)}\subset\mathcal{T}_{\varepsilon}^{+},
\]
we have 
\begin{equation}
\widetilde{\mathcal{E}}_{\nwarrow}^{(-)}[n;N_{\varepsilon},N_{\varepsilon}-1]=\mathcal{E}_{N_{\varepsilon}}[n+1]+O\Big(\rho_{\varepsilon}^{\frac{1}{17}}\frac{\varepsilon^{(N_{\varepsilon})}}{\sqrt{-\Lambda}}\Big).\label{eq:EqualEnergyFluxFinalIngoingTau}
\end{equation}
\end{prop}
\begin{figure}[h] 
\centering 
\scriptsize
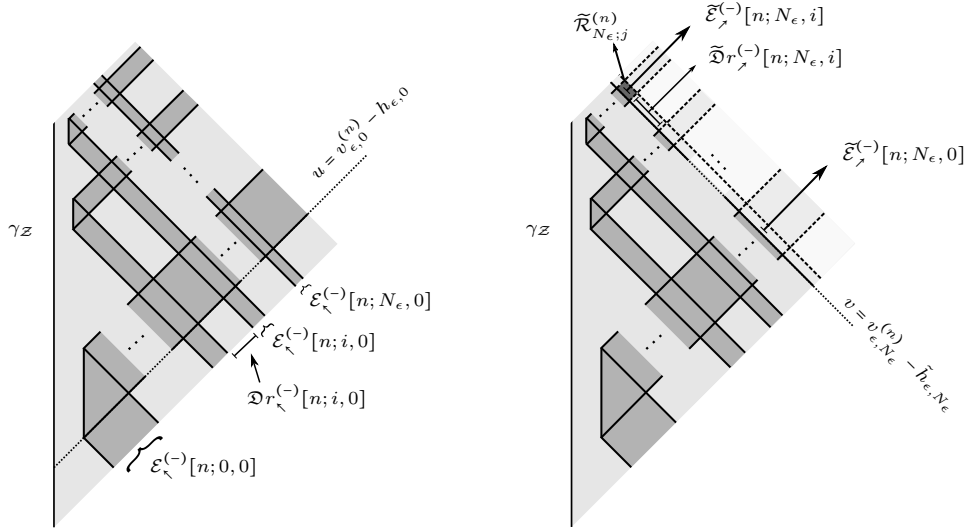 
\caption{The relations (\ref{eq:EqualEnergyFlux})--(\ref{eq:EqualRSeparation}) provide approximate formulas for the incoming energy $\mathcal{E}_{\nwarrow}^{(-)}[n;i,0]$ of the $i$-th beam, as well as its geometric separation $\mathfrak{D}r_{\nwarrow}^{(-)}[n;i,0]$ from the $(i-1)$-th beam, as measured at $u=v_{\epsilon,0}^{(n)}-h_{\epsilon,0}$ (schematic depiction on the left). Similarly, the relations (\ref{eq:EqualEnergyFluxOutgoingTau})--(\ref{eq:EqualRSeparationOutgoingTau}) provide approximate formulas for the outgoing energy of the $i$-th beam and its geometric separation from the $(i-1)$-th beam, as measured before its intersection with the $N_{\epsilon}$-th beam, i.\,e.~at $v=v_{\epsilon,N_{\epsilon}}^{(n)}-\tilde{h}_{\epsilon,N_{\epsilon}}$ (schematic depiction on the right). The reason for using the $\tilde{\cdot}$ quantities (as well as the slightly larger beams $\widetilde{\mathcal{V}}_i$) in the latter case  is that the relations (\ref{eq:EqualEnergyFluxOutgoingTau})--(\ref{eq:EqualRSeparationOutgoingTau}) will later be used in a region of $\mathcal{U}_{max}^{(\epsilon)}$ which is a subset of $\mathcal{T}^+_{\epsilon}$, but not a subset of $\mathcal{U}^+_{\epsilon}$; however, the analogous relations also hold (with exactly the same proof) for $\mathcal{E}_{\nearrow}^{(-)}[n;N_{\epsilon},j]$ and $\mathfrak{D}r_{\nearrow}^{(-)}[n;N_{\epsilon},j]$ in the region $\mathcal{U}^+_{\epsilon}$. }
\end{figure}
\begin{proof}
The proof of Proposition \ref{prop:TotalEnergyChange} will be separated
in a number of steps. We will first establish a number of auxiliary
relations and estimates, before proceeding with the proof of (\ref{eq:EqualEnergyFlux})\textendash (\ref{eq:EqualRSeparation})
and (\ref{eq:EqualEnergyFluxOutgoingTau})\textendash (\ref{eq:EqualEnergyFluxFinalIngoingTau}).

\medskip{}

\noindent \emph{Auxiliary bounds and relations.} In view of the fact
that 
\[
supp(T_{\mu\nu}[f_{\varepsilon}])\cap\mathcal{U}_{\varepsilon}^{+}\subset\bigcup_{k=0}^{N_{\varepsilon}}\mathcal{V}_{k}
\]
(following readily from the bound (\ref{eq:BoundSupportFepsiloni})
on the support of the $f_{\varepsilon k}$'s and the relation (\ref{eq:LinearCombinationVlasovFields})
between $f_{\varepsilon}$ and the $f_{\varepsilon k}$'s), we infer
from equations (\ref{eq:TildeUMaza}) and (\ref{eq:TildeVMaza}) for
$\tilde{m}$ that the function $\tilde{m}(u,v)$ is constant in every
connected component of $\mathcal{U}_{\varepsilon}^{+}\backslash\bigcup_{n\in\mathbb{N}}\bigcup_{k=0}^{N_{\varepsilon}}\mathcal{V}_{k}^{(n)}$,
i.\,e.~in the regions between the beams $\mathcal{V}_{k}^{(n)}$.
This fact immediately implies, in view of the definition (\ref{eq:IngoingEnergyBefore})\textendash (\ref{eq:IngoingEnergyInfinity})
of the quantities $\mathcal{E}_{\nwarrow}^{(\pm)}$, $\mathcal{E}_{\nearrow}^{(\pm)}$,
$\mathcal{E}_{\gamma_{\mathcal{Z}}}$ and $\mathcal{E}_{\mathcal{I}}$,
that, for all $n\in\mathbb{N}$ and all $0\le i,j\le N_{\varepsilon}$
with $i\neq j$: 
\begin{equation}
\mathcal{E}_{\nwarrow}^{(-)}[n;i,j]=\mathcal{E}_{\nwarrow}^{(+)}[n;i,j-1]\label{eq:EqualitySuccessiveEnergiesIngoing}
\end{equation}
 and 
\begin{equation}
\mathcal{E}_{\nearrow}^{(-)}[n;i,j]=\mathcal{E}_{\nearrow}^{(+)}[n;i-1,j],\label{eq:EqualitySuccessiveEnergiesOutgoing}
\end{equation}
where we have used the following index convention for (\ref{eq:EqualitySuccessiveEnergiesIngoing})
and (\ref{eq:EqualitySuccessiveEnergiesOutgoing}): 

\begin{itemize}

\item When $i=j-1$ or $j=i-1$: 
\begin{equation}
\mathcal{E}_{\nwarrow}^{(-)}[n;i,i]\doteq\mathcal{E}_{\gamma_{\mathcal{Z}}}[n;i],\label{eq:ConventionFirstEnergy}
\end{equation}
\[
\mathcal{E}_{\nearrow}^{(-)}[n;i,i]\doteq\mathcal{E}_{\mathcal{I}}[n;i],
\]
\[
\mathcal{E}_{\nwarrow}^{(+)}[n;i,i]\doteq\mathcal{E}_{\mathcal{I}}[n;i]
\]
 and 
\begin{equation}
\mathcal{E}_{\nearrow}^{(+)}[n;i,i]\doteq\mathcal{E}_{\gamma_{\mathcal{Z}}}[n;i].\label{eq:ConventionFinalEnergy}
\end{equation}

\item When $j=-1$ or $i=-1$: 
\begin{equation}
\mathcal{E}_{\nwarrow}^{(+)}[n;i,-1]\doteq\mathcal{E}_{\nwarrow}^{(+)}[n-1;i,N_{\varepsilon}]
\end{equation}
and 
\begin{equation}
\mathcal{E}_{\nearrow}^{(+)}[n;-1,j]\doteq\mathcal{E}_{\nearrow}^{(+)}[n;N_{\varepsilon},j]\label{eq:LastConventionEnergy}
\end{equation}

\end{itemize}

The relations (\ref{eq:EqualitySuccessiveEnergiesIngoing}) and (\ref{eq:EqualitySuccessiveEnergiesOutgoing})
also hold with $\widetilde{\mathcal{E}}_{\nwarrow}^{(\pm)}$, $\widetilde{\mathcal{E}}_{\nearrow}^{(\pm)}$
in place of $\mathcal{E}_{\nwarrow}^{(\pm)}$, $\mathcal{E}_{\nwarrow}^{(\pm)}$.

By the same reasoning, the right hand side of the contraint equations
(\ref{eq:EquationROutside}) and (\ref{eq:EquationROutside}) vanishes
in every connected component of $\mathcal{U}_{\varepsilon}^{+}\backslash\bigcup_{k=0}^{N_{\varepsilon}}\mathcal{V}_{k}$;
hence, it readily follows (by the definition (\ref{eq:IngoingRSeparationCloseToInfinity})\textendash (\ref{eq:OutgoingRSeparationBeforeAfter})
of $\mathfrak{D}r_{\nwarrow}^{(\pm)}$, $\mathfrak{D}r_{\nearrow}^{(\pm)}$)
that, for all $n\in\mathbb{N}$ and all $0\le i,j\le N_{\varepsilon}$,
$i\neq j$: 
\begin{equation}
\mathfrak{D}r_{\nwarrow}^{(-)}[n;i,j]=\mathfrak{D}r_{\nwarrow}^{(+)}[n;i,j-1]\label{eq:EqualitySuccessiveRSeparationsIngoing}
\end{equation}
 and 
\begin{equation}
\mathfrak{D}r_{\nearrow}^{(-)}[n;i,j]=\mathfrak{D}r_{\nearrow}^{(+)}[n;i-1,j],\label{eq:EqualitySuccessiveRSeparationsOutgoing}
\end{equation}
where we have used the index convention 
\begin{equation}
\mathfrak{D}r_{\nwarrow}^{(+)}[n;i,-1]\doteq\mathfrak{D}r_{\nwarrow}^{(+)}[n-1;i,N_{\varepsilon}]\label{eq:FirstConventionRSeparation}
\end{equation}
and 
\begin{equation}
\mathfrak{D}r_{\nearrow}^{(+)}[n;-1,j]\doteq\mathfrak{D}r_{\nearrow}^{(+)}[n;N_{\varepsilon},j].\label{eq:LastConventionRSeparation}
\end{equation}
Similarly for $\widetilde{\mathfrak{D}}r_{\nwarrow}^{(\pm)}$, $\widetilde{\mathfrak{D}}r_{\nearrow}^{(\pm)}$
in place of $\mathfrak{D}r_{\nwarrow}^{(\pm)}$, $\mathfrak{D}r_{\nearrow}^{(\pm)}$.

The following bounds will be useful for estimating the error terms
appearing after repeated applications of the formulas (\ref{eq:IncreaseIngoingEnergyCloseToAxis})\textendash (\ref{eq:OutgoingEnergyCloseToInfinity}):
In view of the bounds (\ref{eq:UpperBoundTuuTvvFi}) and (\ref{eq:UpperBoundTuvFi})
for the components of the energy momentum tensor, the bounds (\ref{eq:EstimateGeometryLikeC0})
and (\ref{eq:BoundsRIntersectionRegioni>j})\textendash (\ref{eq:ComparisonRIntersectionRegioni<j})
for $r$ on $\mathcal{R}_{i;j}^{(n)}$ and the relations (\ref{eq:TildeUMaza})\textendash (\ref{eq:TildeVMaza})
for $\tilde{m}$, we can readily bound for any $n\in\mathbb{N}$ and
$0\le i,j\le N_{\varepsilon}$, $i\neq j$, such that $\mathcal{R}_{i;j}^{(n)}\subset\mathcal{U}_{\varepsilon}^{+}$:
\begin{equation}
\frac{\mathcal{E}_{\nwarrow}^{(\pm)}[n;i,j]}{r_{n;i,j}}+\frac{\mathcal{E}_{\nearrow}^{(\pm)}[n;i,j]}{r_{n;i,j}}\le\exp\big(\exp(2\sigma_{\varepsilon}^{-5})\big)\rho_{\varepsilon}.\label{eq:TrivialBoundEnergiesConcentration}
\end{equation}
Moreover, in view of the definition (\ref{eq:IngoingRSeparationCloseToInfinity})\textendash (\ref{eq:OutgoingRSeparationBeforeAfter})
of $\mathfrak{D}r_{\nwarrow}^{(\pm)}$, $\mathfrak{D}r_{\nearrow}^{(\pm)}$,
the definition (\ref{eq:InitialCenters}) of $v_{\varepsilon,i}$
and the estimate (\ref{eq:EstimateGeometryLikeC0}), we infer that,
for any $n\in\mathbb{N}$ and any $0\le i,j\le N_{\varepsilon}$,
such that $\mathcal{R}_{i;j}^{(n)}\subset\mathcal{U}_{\varepsilon}^{+}$
if $i\neq j$ or $\mathcal{R}_{i;\gamma_{\mathcal{Z}}}^{(n)},\mathcal{R}_{i;\mathcal{I}}^{(n)}\subset\mathcal{U}_{\varepsilon}^{+}$
if $i=j$: 
\begin{align}
e^{-\sigma_{\varepsilon}^{-4}}\rho_{\varepsilon}^{-1}\frac{\varepsilon^{(i-1)}}{\sqrt{-\Lambda}} & \le\mathfrak{D}r_{\nwarrow}^{(\pm)}[n;i,j]\le e^{\sigma_{\varepsilon}^{-4}}\rho_{\varepsilon}^{-1}\frac{\varepsilon^{(i-1)}}{\sqrt{-\Lambda}},\label{eq:TrivialBoundsRSeparation}\\
e^{-\sigma_{\varepsilon}^{-4}}\rho_{\varepsilon}^{-1}\frac{\varepsilon^{(j-1)}}{\sqrt{-\Lambda}} & \le\mathfrak{D}r_{\nearrow}^{(\pm)}[n;i,j]\le e^{\sigma_{\varepsilon}^{-4}}\rho_{\varepsilon}^{-1}\frac{\varepsilon^{(j-1)}}{\sqrt{-\Lambda}}.\nonumber 
\end{align}

For any $0\le i\le N_{\varepsilon}-1$ such that $\mathcal{R}_{i+1;i}^{(n)}\subset\mathcal{U}_{\varepsilon}^{+}$,
we can express $r_{n;i+1,i}$ by integrating $\partial_{v}r$ in $v$
from $(u_{n;i+1,i}^{(+)},u_{n;i+1,i}^{(+)})\in\gamma_{\mathcal{Z}_{\varepsilon}}$
up to $(u_{n;i+1,i}^{(+)},v_{n;i+1,i}^{(-)})$ as follows (using the
notational conventions (\ref{eq:v+-})\textendash (\ref{eq:u+-})
and (\ref{eq:ConventionV+-}), as well as the bounds (\ref{eq:BoundsRIntersectionRegioni>j})\textendash (\ref{eq:ComparisonRIntersectionRegioni>j}),
the fact that $\tilde{m}=0$ on $\{u_{n;i+1,i}^{(+)}\}\times[u_{n;i+1,i}^{(+)},v_{n;i+1,i}^{(-)}]$
and the bounds (\ref{eq:EstimateGeometryLikeC0}) and (\ref{eq:TrivialBoundsRSeparation})):
\begin{align}
r_{n;i+1,i} & =r(u_{n;i+1,i}^{(+)},v_{n;i+1,i}^{(-)})=\label{eq:ExpressionRnii-1inV}\\
 & =\int_{u_{n;i+1,i}^{(+)}}^{v_{n;i+1,i}^{(-)}}\partial_{v}r(u_{n;i+1,i}^{(+)},v)\,dv=\nonumber \\
 & =\int_{u_{n;i+1,i}^{(+)}}^{v_{n;i+1,i}^{(-)}}\frac{\partial_{v}r}{1-\frac{2m}{r}}(u_{n;i+1,i}^{(+)},v)\cdot(1+O(\varepsilon))\,dv=\nonumber \\
 & =\int_{v_{n;i,i}^{(+)}+\rho_{\varepsilon}^{-\frac{7}{8}}h_{\varepsilon,i-1}}^{v_{n;i+1,i}^{(-)}-\rho_{\varepsilon}^{-\frac{7}{8}}h_{\varepsilon,i-1}}\frac{\partial_{v}r}{1-\frac{2m}{r}}(u_{n;i+1,i}^{(+)},v)\cdot(1+O(\varepsilon))\,dv+O\Big(\exp\big(\exp(2\sigma_{\varepsilon}^{-5})\big)\rho_{\varepsilon}^{-\frac{7}{8}}\frac{\varepsilon^{(i-1)}}{\sqrt{-\Lambda}}\Big)=\nonumber \\
 & =(1+O(\varepsilon))\mathfrak{D}r_{\nwarrow}^{(+)}[n;i+1,i]+O\Big(\exp\big(\exp(2\sigma_{\varepsilon}^{-5})\big)\rho_{\varepsilon}^{-\frac{7}{8}}\frac{\varepsilon^{(i-1)}}{\sqrt{-\Lambda}}\Big)=\nonumber \\
 & =\mathfrak{D}r_{\nwarrow}^{(+)}[n;i+1,i]\cdot\big(1+O(\rho_{\varepsilon}^{\frac{1}{10}})\big).\nonumber 
\end{align}
However, we can similarly express $r_{n;i+1,i}$ as an integral of
$-\partial_{u}r$ in $u$ (provided $(u_{n;i+1,i}^{(+)},v_{n;i+1,i}^{(-)})\times\{v_{n;i+1,i}^{(-)}\}\subset\mathcal{U}_{\varepsilon}^{+}$,
which is necessarily true if $\mathcal{R}_{i+1;\gamma_{\mathcal{Z}}}^{(n)}\subset\mathcal{U}_{\varepsilon}^{+}$:
\begin{align}
r_{n;i+1,i} & =\int_{u_{n;i+1,i}^{(+)}}^{v_{n;i+1,i}^{(-)}}(-\partial_{u}r)(u,v_{n;i+1,i}^{(-)})\,du=\label{eq:ExpressionRnii-1inU}\\
 & =\mathfrak{D}r_{\nearrow}^{(-)}[n;i+1,i+1]\cdot\big(1+O(\rho_{\varepsilon}^{\frac{1}{10}})\big).\nonumber 
\end{align}
Arguing similarly (replacing $\sigma_{\varepsilon}$ with $\delta_{\varepsilon}$
and using (\ref{eq:EstimateGeometryLikeC0LargerDomain}) in place
of (\ref{eq:EstimateGeometryLikeC0})), we infer that, for any $0\le i\le N_{\varepsilon}-1$
such that $\widetilde{\mathcal{R}}_{i+1;\gamma_{\mathcal{Z}}}^{(n)}\subset\mathcal{T}_{\varepsilon}^{+}$,
the relations (\ref{eq:ExpressionRnii-1inV}) and (\ref{eq:ExpressionRnii-1inU})
also hold with 
\begin{equation}
\tilde{r}_{n;i+1,i}\doteq\inf_{\widetilde{\mathcal{R}}_{i+1;i}^{(n)}}r
\end{equation}
in place of $r_{n;i+1,i}$ and $\widetilde{\mathfrak{D}}r_{\nwarrow}^{(+)}[n;i+1,i]$,
$\widetilde{\mathfrak{D}}r_{\nearrow}^{(-)}[n;i+1,i+1]$ in place
of $\mathfrak{D}r_{\nwarrow}^{(+)}[n;i+1,i]$, $\mathfrak{D}r_{\nearrow}^{(-)}[n;i+1,i+1]$,
respectively.

From (\ref{eq:ExpressionRnii-1inV}) and (\ref{eq:ExpressionRnii-1inU}),
we immediately infer that, for all $0\le i\le N_{\varepsilon}-1$
such that $\mathcal{R}_{i+1;\gamma_{\mathcal{Z}}}^{(n)}\subset\mathcal{U}_{\varepsilon}^{+}$:
\begin{equation}
\mathfrak{D}r_{\nearrow}^{(-)}[n;i+1,i+1]=\mathfrak{D}r_{\nwarrow}^{(+)}[n;i+1,i]\cdot\big(1+O(\rho_{\varepsilon}^{\frac{1}{10}})\big).\label{eq:EqualityRSeparationsNearAxis}
\end{equation}
Similarly, expressing $\frac{1}{r_{n;i,i+1}}$ as an integral of $\partial_{v}(\frac{1}{r})$
and $\partial_{u}(\frac{1}{r})$ from $\mathcal{I}_{\varepsilon}$
up to $(u_{n;i,i+1}^{(+)},v_{n;i,i+1}^{(-)})$, we infer that, for
all $0\le i\le N_{\varepsilon}-1$ such that $\mathcal{R}_{i+1;\mathcal{I}}^{(n)}\subset\mathcal{U}_{\varepsilon}^{+}$
\begin{equation}
\mathfrak{D}r_{\nwarrow}^{(-)}[n;i+1,i+1]=\mathfrak{D}r_{\nearrow}^{(+)}[n;i,i+1]\cdot\big(1+O(\rho_{\varepsilon}^{\frac{1}{10}})\big).\label{eq:EqualityRSeparationsNearInfinity}
\end{equation}

\medskip{}

\noindent \emph{Proof of (\ref{eq:EqualEnergyFlux})\textendash (\ref{eq:EqualRSeparation}).}
In order to establish (\ref{eq:EqualEnergyFlux}) and (\ref{eq:EqualRSeparation}),
we will first show that, for all $n\in\mathbb{N}$ such that 
\begin{equation}
\{0\le u\le v_{\varepsilon,0}^{(n+1)}-h_{\varepsilon,0}\}\cap\big\{ u<v<u+\sqrt{-\frac{3}{\Lambda}}\pi\big\}\subset\mathcal{U}_{\varepsilon}^{+}\label{eq:AlternativeNonTrivialU}
\end{equation}
and any $0\le i\le N_{\varepsilon}$:
\begin{equation}
\mathcal{E}_{\nwarrow}^{(-)}[n+1;i,0]=\mathcal{E}_{\nwarrow}^{(-)}[n;i,0]\cdot\exp\Big(\sum_{j=0}^{i-1}\frac{2\mathcal{E}_{\nwarrow}^{(-)}[n+1;j,0]}{\mathfrak{D}r_{\nwarrow}^{(-)}[n+1;j+1,0]}+O(\rho_{\varepsilon}^{\frac{1}{12}})\Big)+O\Big(\rho_{\varepsilon}^{\frac{1}{4}}\frac{\varepsilon^{(i)}}{\sqrt{-\Lambda}}\Big)\label{eq:ChangeOfEnergyFinalAgain}
\end{equation}
and, for all $1\le i\le N_{\varepsilon}$:
\begin{equation}
\mathfrak{D}r_{\nwarrow}^{(-)}[n+1;i,0]=\mathfrak{D}r_{\nwarrow}^{(-)}[n;i,0]\cdot\exp\Big(-\sum_{j=0}^{i-2}\frac{2\mathcal{E}_{\nearrow}^{(-)}[n;i,j]}{\mathfrak{D}r_{\nwarrow}^{(-)}[n+1;j+1,0]}+O(\rho_{\varepsilon}^{\frac{1}{12}})\Big).\label{eq:ChangeOfRSeparationFinalAgain}
\end{equation}
Note that the recursive system (\ref{eq:ChangeOfEnergyFinalAgain})\textendash (\ref{eq:ChangeOfRSeparationFinalAgain}),
modulo the $O(\cdot)$ error terms, is in fact the same as the recursive
system (\ref{eq:FormulaRecursiveEnergy})\textendash (\ref{eq:FormulaRecursiveDr})
for $\mathcal{E}_{i}[n]$, $R_{i}[n]$, with the quantities $\mathcal{E}_{\nwarrow}^{(-)}[n;i,0]$,
$\mathfrak{D}r_{\nwarrow}^{(-)}[n;i,0]$ and $\mathcal{E}_{i}[n]$,
$R_{i}[n]$, respectively, satisfying the same initial conditions
at $n=0$.

Let us first assume that (\ref{eq:ChangeOfEnergyFinalAgain})\textendash (\ref{eq:ChangeOfRSeparationFinalAgain})
have been established. The relations (\ref{eq:EqualEnergyFlux}) and
(\ref{eq:EqualRSeparation}) will then follow by showing that the
quantities 
\begin{align*}
e_{i}[n] & \doteq\frac{\mathcal{E}_{\nwarrow}^{(-)}[n;i,0]-\mathcal{E}_{i}[n]}{\varepsilon^{(i)}}\sqrt{-\Lambda}\text{ for }0\le i\le N_{\varepsilon},\\
r_{i}[n] & \doteq\frac{\mathfrak{D}r_{\nwarrow}^{(-)}[n;i,0]}{R_{i}[n]}\text{ for }1\le i\le N_{\varepsilon},\\
\bar{\mu}_{i}[n] & \doteq\rho_{\varepsilon}^{-1}\Big(\frac{2\mathcal{E}_{\nwarrow}^{(-)}[n;i,0]}{\mathfrak{D}r_{\nwarrow}^{(-)}[n;i+1,0]}-\mu_{i}[n]\Big)\text{ for }0\le i\le N_{\varepsilon}-1
\end{align*}
 satisfy 
\begin{equation}
|e_{i}[n]|,\text{ }|r_{i}[n]-1|,\text{ }|\bar{\mu}_{i}[n]|\le\rho_{\varepsilon}^{\frac{1}{16}}.\label{eq:BoundToShowSequence}
\end{equation}

To this end, let $n^{*}$ be the maximum number in $\{0,1,\ldots,n\}$
such that, for all $0\le\bar{n}\le n^{*}$:
\begin{equation}
|\bar{\mu}_{i}[\bar{n}]|\le\rho_{\varepsilon}^{\frac{1}{15}}\text{ for all }0\le i\le N_{\varepsilon}-1\label{eq:UpperBoundMuInductionContradiction}
\end{equation}
(note that (\ref{eq:UpperBoundMuInductionContradiction}) is trivially
true for $\bar{n}=0$, since $\bar{\mu}_{i}[0]=0$ by (\ref{eq:InitialConditionsImplicitRelations})).
Assuming, for the sake of contradiction, that $n^{*}<n$, we will
show that (\ref{eq:UpperBoundMuInductionContradiction}) also holds
for $n^{*}+1$, hence contradicting the maximality of $n^{*}$. Note
that, in view of the definition (\ref{eq:BootstrapDomain}), it is
necessary that 
\begin{equation}
n^{*}<n\lesssim\sigma_{\varepsilon}^{-2}\label{eq:UpperBoundNuDomain}
\end{equation}
(otherwise, (\ref{eq:AlternativeNonTrivialU}) cannot hold). We will
argue inductively on $i$, assuming that, for all $0\le\bar{i}\le i-1$,
\begin{equation}
|\bar{\mu}_{\bar{i}}[\bar{n}]|\le\rho_{\varepsilon}^{\frac{1}{15}}\text{ for all }0\le\bar{n}\le n^{*}+1\label{eq:InductionOnIMu}
\end{equation}
and then showing that (\ref{eq:InductionOnIMu}) also holds for $\bar{i}=i$.
Note that (\ref{eq:InductionOnIMu}) holds trivially for $\bar{i}=0$,
since, in this case, (\ref{eq:RecursiveFormulaMu}), (\ref{eq:ChangeOfEnergyFinalAgain})\textendash (\ref{eq:ChangeOfRSeparationFinalAgain})
and (\ref{eq:UpperBoundNuDomain}) imply that 
\begin{equation}
\mu_{0}[\bar{n}]=\mu_{0}[0]
\end{equation}
 and 
\begin{equation}
\frac{2\mathcal{E}_{\nwarrow}^{(-)}[\bar{n};0,0]}{\mathfrak{D}r_{\nwarrow}^{(-)}[\bar{n};1,0]}=(1+O(\rho_{\varepsilon}^{\frac{1}{14}}))\frac{2\mathcal{E}_{\nwarrow}^{(-)}[0;0,0]}{\mathfrak{D}r_{\nwarrow}^{(-)}[0;1,0]}+O(\rho_{\varepsilon}^{1+\frac{1}{6}}),
\end{equation}
which yield 
\begin{equation}
|\bar{\mu}_{0}[\bar{n}]|\le\rho_{\varepsilon}^{-1}\Big|O(\rho_{\varepsilon}^{\frac{1}{14}}))\frac{2\mathcal{E}_{\nwarrow}^{(-)}[0;0,0]}{\mathfrak{D}r_{\nwarrow}^{(-)}[0;1,0]}+O(\rho_{\varepsilon}^{1+\frac{1}{6}})\Big|\le\exp(\exp(\sigma_{\varepsilon}^{-9}))\rho_{\varepsilon}^{\frac{1}{14}}.\label{eq:TrivialBoundMu0}
\end{equation}

In view of the relation (\ref{eq:RecursiveFormulaMu}) for $\mu_{i}[n]$,
the relations (\ref{eq:ChangeOfEnergyFinalAgain})\textendash (\ref{eq:ChangeOfRSeparationFinalAgain})
for $\mathcal{E}_{\nwarrow}^{(-)}[n;i,0]$, $\mathfrak{D}r_{\nwarrow}^{(-)}[n;i+1,0]$,
the bounds (\ref{eq:TrivialBoundEnergiesConcentration})\textendash (\ref{eq:TrivialBoundsRSeparation}),
the bound 
\begin{equation}
\sum_{j=0}^{N_{\varepsilon}-1}\frac{2\mathcal{E}_{\nwarrow}^{(-)}[k;j,0]}{\mathfrak{D}r_{\nwarrow}^{(-)}[k;j+1,0]}\le\exp(\exp(2\delta_{\varepsilon}^{-15}))\text{ for all }0\le k\le n\label{eq:BoundSumEnergiesByDelta}
\end{equation}
(following from (\ref{eq:TrivialBoundEnergiesConcentration})\textendash (\ref{eq:TrivialBoundsRSeparation})
and the fact that $N_{\varepsilon}=\rho_{\varepsilon}^{-1}\exp(e^{\delta_{\varepsilon}^{-15}})$),
the bound 
\begin{equation}
\sum_{j=0}^{i-1}\bar{\mu}_{j}[k]\le\rho_{\varepsilon}^{-1+\frac{1}{15}}\exp(\exp(\delta_{\varepsilon}^{-15}))\text{ for all }0\le k\le n^{*}+1
\end{equation}
 (following from the inductive assumption (\ref{eq:InductionOnIMu}))
and the bound (\ref{eq:UpperBoundMuInductionContradiction}) for $0\le\bar{n}\le n^{*}$,
we infer that, for any $0\le\bar{n}\le n^{*}$:
\begin{align}
\big|\bar{\mu}_{i} & [\bar{n}+1]\big|=\rho_{\varepsilon}^{-1}\Bigg|\frac{\mathcal{E}_{\nwarrow}^{(-)}[\bar{n};i,0]}{\mathfrak{D}r_{\nwarrow}^{(-)}[\bar{n};i+1,0]}\exp\Big(2\sum_{j=0}^{i-1}\frac{2\mathcal{E}_{\nwarrow}^{(-)}[\bar{n}+1;j,0]}{\mathfrak{D}r_{\nwarrow}^{(-)}[\bar{n}+1;j+1,0]}+O(\rho_{\varepsilon}^{\frac{1}{12}})\Big)+O(\rho_{\varepsilon}^{1+\frac{1}{6}})-\mu_{i}[\bar{n}]\exp\Big(2\sum_{j=0}^{i-1}\mu_{j}[\bar{n}+1]\Big)\Bigg|=\label{BoundMuBarInitial}\\
 & =\exp\Big(2\sum_{j=0}^{i-1}\frac{2\mathcal{E}_{\nwarrow}^{(-)}[\bar{n}+1;j,0]}{\mathfrak{D}r_{\nwarrow}^{(-)}[\bar{n}+1;j+1,0]}\Big)\rho_{\varepsilon}^{-1}\Bigg|\frac{\mathcal{E}_{\nwarrow}^{(-)}[\bar{n};i,0]}{\mathfrak{D}r_{\nwarrow}^{(-)}[\bar{n};i+1,0]}-\mu_{i}[\bar{n}]\exp\Big(-2\rho_{\varepsilon}\sum_{j=0}^{i-1}\bar{\mu}_{j}[\bar{n}+1]\Big)+O(\rho_{\varepsilon}^{1+\frac{1}{13}})\Bigg|=\nonumber \\
 & =\exp\Big(O(\exp(\exp(2\delta_{\varepsilon}^{-15})))\Big)\rho_{\varepsilon}^{-1}\Bigg|\frac{\mathcal{E}_{\nwarrow}^{(-)}[\bar{n};i,0]}{\mathfrak{D}r_{\nwarrow}^{(-)}[\bar{n};i+1,0]}-\mu_{i}[\bar{n}]\Big(1-2\rho_{\varepsilon}\sum_{j=0}^{i-1}\bar{\mu}_{j}[\bar{n}+1]+O(\rho_{\varepsilon}^{\frac{2}{15}}\delta_{\varepsilon}^{-4})\Big)+O(\rho_{\varepsilon}^{1+\frac{1}{13}})\Bigg|\le\nonumber \\
 & \le\exp\Big(\exp\big(e^{\delta_{\varepsilon}^{-16}}\big)\Big)\cdot|\bar{\mu}_{i}[\bar{n}]|+\exp\Big(\exp\big(e^{\delta_{\varepsilon}^{-16}}\big)\Big)\rho_{\varepsilon}\sum_{j=0}^{i-1}|\bar{\mu}_{j}[\bar{n}+1]|+\rho_{\varepsilon}^{\frac{1}{14}}.\nonumber 
\end{align}
Applying (\ref{BoundMuBarInitial}) successively for $\bar{n}=0,\ldots n^{*}$,
using also the bound (\ref{eq:UpperBoundNuDomain}) for $n^{*}<n$,
we obtain 
\begin{align}
\max_{0\le\bar{n}\le n^{*}}\big|\bar{\mu}_{i}[\bar{n}+1]\big| & \le\exp\Big(n^{*}\exp\big(e^{\delta_{\varepsilon}^{-16}}\big)\Big)\cdot\sum_{\bar{n}=0}^{n^{*}}\sum_{j=0}^{i-1}|\bar{\mu}_{j}[\bar{n}+1]|+\exp\Big(n^{*}\exp\big(e^{\delta_{\varepsilon}^{-16}}\big)\Big)\cdot\rho_{\varepsilon}^{\frac{1}{14}}\le\label{eq:BoundMuBarFinal}\\
 & \le\bar{\delta}_{\varepsilon}^{-1}\rho_{\varepsilon}\sum_{j=0}^{i-1}\big(\max_{0\le\bar{n}\le n^{*}}\bar{\mu}_{j}[\bar{n}+1]\big)+\bar{\delta}_{\varepsilon}^{-1}\rho_{\varepsilon}^{\frac{1}{14}},\nonumber 
\end{align}
where 
\begin{equation}
\bar{\delta}_{\varepsilon}\doteq\exp\Big(-\exp\big(2e^{\delta_{\varepsilon}^{-16}}\big)\Big).\label{eq:BarDelta}
\end{equation}
Since (\ref{eq:BoundMuBarFinal}) is similarly valid for any $\bar{i}$
with $0\le\bar{i}\le i$ in place of $i$, we infer from (\ref{eq:BoundMuBarFinal})
after applying a discrete Gronwal-type argument in the $i$ variable
(using also the bound (\ref{eq:TrivialBoundMu0}) for $\bar{\mu}_{0}$
and the fact that $i\le N_{\varepsilon}$): 
\begin{equation}
\max_{0\le\bar{n}\le n^{*}}\big|\bar{\mu}_{i}[\bar{n}+1]\big|\le\exp(2\bar{\delta}_{\varepsilon}^{-1}\rho_{\varepsilon}N_{\varepsilon})\rho_{\varepsilon}^{\frac{1}{14}}\le\exp(\bar{\delta}_{\varepsilon}^{-2})\rho_{\varepsilon}^{\frac{1}{14}},
\end{equation}
from which (\ref{eq:InductionOnIMu}) for $\bar{i}=i$ follows, in
view of the relation (\ref{eq:HierarchyOfParameters}) between $\rho_{\varepsilon}$
and $\delta_{\varepsilon}$. 

As a result, we have established inductively that (\ref{eq:InductionOnIMu})
holds for any $0\le i\le N_{\varepsilon}-1$, and hence: 
\begin{equation}
\max_{0\le i\le N_{\varepsilon}-1}\max_{0\le\bar{n}\le n}|\bar{\mu}_{i}[\bar{n}]|\le\rho_{\varepsilon}^{\frac{1}{15}}.\label{eq:MuBarBoundFinal}
\end{equation}

From the relations (\ref{eq:ChangeOfEnergyFinalAgain})\textendash (\ref{eq:ChangeOfRSeparationFinalAgain})
for $\mathcal{E}_{\nwarrow}^{(-)}[n;i,0]$, $\mathfrak{D}r_{\nwarrow}^{(-)}[n;i+1,0]$
and (\ref{eq:FormulaRecursiveEnergy})\textendash (\ref{eq:FormulaRecursiveDr})
for $\mathcal{E}_{i}[n]$, $R_{i}[n]$, the bounds (\ref{eq:TrivialBoundEnergiesConcentration})\textendash (\ref{eq:TrivialBoundsRSeparation}),
the bound (\ref{eq:BoundSumEnergiesByDelta}) and the bound (\ref{eq:MuBarBoundFinal}),
we obtain for any $0\le\bar{n}\le n-1$ and any $0\le i\le N_{\varepsilon}$:
\begin{align}
|e_{i}[\bar{n}+1]| & =\Bigg|\frac{\sqrt{-\Lambda}}{\varepsilon^{(i)}}\mathcal{E}_{\nwarrow}^{(-)}[\bar{n};i,0]\cdot\exp\Big(\sum_{j=0}^{i-1}\frac{2\mathcal{E}_{\nwarrow}^{(-)}[\bar{n}+1;j,0]}{\mathfrak{D}r_{\nwarrow}^{(-)}[\bar{n}+1;j+1,0]}+O(\rho_{\varepsilon}^{\frac{1}{12}})\Big)+O(\rho_{\varepsilon}^{\frac{1}{4}})-\frac{\sqrt{-\Lambda}}{\varepsilon^{(i)}}\mathcal{E}_{i}[\bar{n}]\cdot\exp\Big(\sum_{j=0}^{i-1}\mu_{j}[\bar{n}]\Big)\Bigg|=\label{eq:AlmostGronwale}\\
 & =\exp\Big(\sum_{j=0}^{i-1}\frac{2\mathcal{E}_{\nwarrow}^{(-)}[n+1;j,0]}{\mathfrak{D}r_{\nwarrow}^{(-)}[n+1;j+1,0]}\Big)\Bigg|(1+O(\rho_{\varepsilon}^{\frac{1}{12}}))\frac{\sqrt{-\Lambda}}{\varepsilon^{(i)}}\mathcal{E}_{\nwarrow}^{(-)}[\bar{n};i,0]-\nonumber \\
 & \hphantom{=\exp\Big(\sum_{j=0}^{i-1}=\exp\Big(\sum_{j=0}^{i-1}}-\frac{\sqrt{-\Lambda}}{\varepsilon^{(i)}}\mathcal{E}_{i}[\bar{n}]\cdot\exp\Big(-\rho_{\varepsilon}\sum_{j=0}^{i-1}\bar{\mu}_{j}[\bar{n}+1]\Big)+O(\rho_{\varepsilon}^{\frac{1}{4}})\Bigg|\le\nonumber\\
 & \le\bar{\delta}_{\varepsilon}^{-\frac{1}{2}}\Bigg|e_{i}[\bar{n}]+O(\rho_{\varepsilon}^{\frac{1}{12}})\frac{\sqrt{-\Lambda}}{\varepsilon^{(i)}}\mathcal{E}_{\nwarrow}^{(-)}[\bar{n};i,0]+O(\exp(\exp(\delta_{\varepsilon}^{-15}))\rho_{\varepsilon}^{\frac{1}{15}})\frac{\sqrt{-\Lambda}}{\varepsilon^{(i)}}\mathcal{E}_{i}[\bar{n}]+O(\rho_{\varepsilon}^{\frac{1}{4}})\Bigg|\le\nonumber \\
 & \le\bar{\delta}_{\varepsilon}^{-1}\Big(|e_{i}[\bar{n}]|+\rho_{\varepsilon}^{\frac{1}{15}}\Big)\nonumber 
\end{align}
and, for any $1\le i\le N_{\varepsilon}$:
\begin{align}
|r_{i}[\bar{n}+1]-1| & =\Bigg|\frac{\mathfrak{D}r_{\nwarrow}^{(-)}[\bar{n};i,0]\cdot\exp\Big(-\sum_{j=0}^{i-2}\frac{2\mathcal{E}_{\nwarrow}^{(-)}[\bar{n}+1;j,0]}{\mathfrak{D}r_{\nwarrow}^{(-)}[\bar{n}+1;j+1,0]}+O(\rho_{\varepsilon}^{\frac{1}{12}})\Big)}{R_{i}[\bar{n}]\cdot\exp\Big(-\sum_{j=0}^{i-2}\bar{\mu}_{j}[\bar{n}+1]\Big)}-1\Bigg|=\label{eq:AlmostGronwalR}\\
 & =\Bigg|r_{i}[\bar{n}]\cdot\exp\Big(-\rho_{\varepsilon}\sum_{j=0}^{i-2}\bar{\mu}_{j}[\bar{n}+1]+O(\rho_{\varepsilon}^{\frac{1}{12}})\Big)-1\Bigg|=\nonumber \\
 & =\Bigg|r_{i}[\bar{n}]\cdot\exp\Big(O(\exp(\exp(\delta_{\varepsilon}^{-15}))\rho_{\varepsilon}^{\frac{1}{15}})\Big)-1\Bigg|\le\nonumber \\
 & \le\big(1+\exp(\exp(\delta_{\varepsilon}^{-16}))\rho_{\varepsilon}^{\frac{1}{15}}\big)\cdot|r_{i}[\bar{n}+1]-1|+\exp(\exp(\delta_{\varepsilon}^{-15}))\rho_{\varepsilon}^{\frac{1}{15}}.\nonumber 
\end{align}
From (\ref{eq:AlmostGronwale}) and (\ref{eq:AlmostGronwalR}), using
also the initial conditions $e_{i}[0]=0$ and $r_{i}[0]=1$, we obtain
(using also (\ref{eq:UpperBoundNuDomain}))
\begin{equation}
\max_{0\le i\le N_{\varepsilon}}\max_{0\le\bar{n}\le n}|e_{i}[\bar{n}+1]|\le\exp(n\bar{\delta}_{\varepsilon}^{-1})\bar{\delta}_{\varepsilon}^{-1}\rho_{\varepsilon}^{\frac{1}{15}}\le\rho_{\varepsilon}^{\frac{1}{16}}\label{eq:finale}
\end{equation}
and 
\begin{equation}
\max_{1\le i\le N_{\varepsilon}}\max_{0\le\bar{n}\le n}|r_{i}[\bar{n}]-1|\le\bar{\delta}_{\varepsilon}^{-1}n\rho^{\frac{1}{15}}\le\rho_{\varepsilon}^{\frac{1}{16}}.\label{eq:Finalr}
\end{equation}
From (\ref{eq:MuBarBoundFinal}), (\ref{eq:finale}) and (\ref{eq:Finalr}),
we therefore infer (\ref{eq:BoundToShowSequence}), thus obtaining
(\ref{eq:EqualEnergyFlux}) and (\ref{eq:EqualRSeparation}) (assuming
that (\ref{eq:ChangeOfEnergyFinalAgain}) and (\ref{eq:ChangeOfRSeparationFinalAgain})
have been proven).

\medskip{}

\noindent \emph{Proof of (\ref{eq:ChangeOfEnergyFinalAgain})\textendash (\ref{eq:ChangeOfRSeparationFinalAgain}).}
We will now proceed with the proof of (\ref{eq:ChangeOfEnergyFinalAgain})\textendash (\ref{eq:ChangeOfRSeparationFinalAgain}).
Let $n\in\mathbb{N}$ be an integer satisfying (\ref{eq:AlternativeNonTrivialU})
and let $0\le i\le N_{\varepsilon}$. 

\begin{figure}[h] 
\centering 
\scriptsize
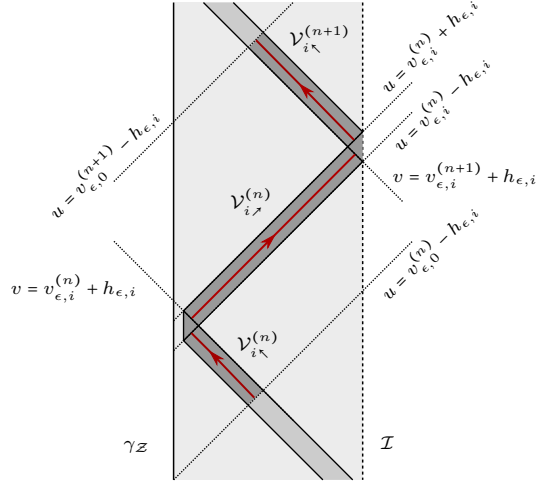 
\caption{For the proof of (\ref{eq:ChangeOfEnergyFinalAgain}), we move along $\mathcal{V}_i$ in the direction of the red arrows in three steps: First,  from $u=v_{\epsilon,0}^{(n)}-h_{\epsilon,i}$ up to $u=v_{\epsilon,i}^{(n)}-h_{\epsilon,i}$ (ingoing regime), then from $v=v_{\epsilon,i}^{(n)}+h_{\epsilon,i}$ up to $v=v_{\epsilon,i}^{(n+1)}+h_{\epsilon,i}$ (outgoing regime) and finally again from $u=v_{\epsilon,i}^{(n)}+h_{\epsilon,i}$ up to $u=v_{\epsilon,0}^{(n+1)}-h_{\epsilon,i}$ (ingoing regime). Along the way, we use the formulas (\ref{eq:IncreaseIngoingEnergyCloseToAxis})--(\ref{eq:OutgoingEnergyCloseToInfinity}) to calculate the change in the energy content of $\mathcal{V}_i$ after each intersection with one of the beams $\mathcal{V}_j$, $j\neq i$. \label{fig:Integration_Direction}}
\end{figure}

\begin{enumerate}

\item First, moving along the beam $\mathcal{V}_{i\nwarrow}^{(n)}$
from $u=v_{\varepsilon,0}^{(n)}-h_{\varepsilon,i}$ up to $u=v_{\varepsilon,i}^{(n)}-h_{\varepsilon,i}$
(see Figure \ref{fig:Integration_Direction}) and calculating the
differene between $\mathcal{E}_{\nwarrow}^{(-)}[n;i,j]$ and $\mathcal{E}_{\nwarrow}^{(+)}[n;i,j]$
using (\ref{eq:IncreaseIngoingEnergyCloseToAxis}) for all $0\le j<i$
(measuring the change in the energy content of $\mathcal{V}_{i\nwarrow}^{(n)}$
as it crosses each of the outgoing beams $\mathcal{V}_{j\nearrow}^{(n)}$,
$j<i$), making also use of the equality (\ref{eq:EqualitySuccessiveEnergiesIngoing})
(expressing the conservation of the energy content of $\mathcal{V}_{i\nwarrow}^{(n)}$
in the region between two successive intersections with the outgoing
beams) and the bound (\ref{eq:TrivialBoundEnergiesConcentration}),
we infer 
\begin{align}
\mathcal{E}_{\gamma_{\mathcal{Z}}}[n;i] & =\mathcal{E}_{\nwarrow}^{(-)}[n;i,0]\cdot\exp\Big(\sum_{j=0}^{i-1}\frac{2\mathcal{E}_{\nearrow}^{(-)}[n;i,j]}{r_{n;i.j}}+\sum_{j=0}^{i-1}O(\rho_{\varepsilon}^{\frac{3}{2}})\Big)+\label{eq:ChangeOfEnergyIngoingFirst}\\
 & \hphantom{=\mathcal{E}_{\nwarrow}^{(-)}[n;i,0]\cdot}+\sum_{j=0}^{i-1}\Big\{\exp\Big(\sum_{k=0}^{j-1}\frac{2\mathcal{E}_{\nearrow}^{(-)}[n;i,k]}{r_{n;i.k}}+\sum_{k=0}^{j-1}O(\rho_{\varepsilon}^{\frac{3}{2}})\Big)\cdot O\Big(\rho_{\varepsilon}^{\frac{3}{2}}\frac{\varepsilon^{(i)}}{\sqrt{-\Lambda}}\Big)\Big\}=\nonumber\\
 & =\mathcal{E}_{\nwarrow}^{(-)}[n;i,0]\cdot\exp\Big(\sum_{j=0}^{i-1}\frac{2\mathcal{E}_{\nearrow}^{(-)}[n;i,j]}{r_{n;i.j}}+O(N_{\varepsilon}\rho_{\varepsilon}^{\frac{3}{2}})\Big)+O\Big(\exp\big(N_{\varepsilon}\exp\big(\exp(2\sigma_{\varepsilon}^{-5})\big)\rho_{\varepsilon}\big)\cdot N_{\varepsilon}\rho_{\varepsilon}^{\frac{3}{2}}\frac{\varepsilon^{(i)}}{\sqrt{-\Lambda}}\Big)=\nonumber \\
 & =\mathcal{E}_{\nwarrow}^{(-)}[n;i,0]\cdot\exp\Big(\sum_{j=0}^{i-1}\frac{2\mathcal{E}_{\nearrow}^{(-)}[n;i,j]}{r_{n;i.j}}+O(\rho_{\varepsilon}^{\frac{1}{4}})\Big)+O\Big(\rho_{\varepsilon}^{\frac{1}{4}}\frac{\varepsilon^{(i)}}{\sqrt{-\Lambda}}\Big),\nonumber 
\end{align}
where, in passing from the second to the third line in (\ref{eq:ChangeOfEnergyIngoingFirst}),
we have made use of the definition (\ref{eq:NumberOfBeams}) of $N_{\varepsilon}$
and the relations (\ref{eq:HierarchyOfParameters}) between $\rho_{\varepsilon}$,
$\delta_{\varepsilon}$ and $\sigma_{\varepsilon}$. 

\item Moving along $\mathcal{V}_{i\nearrow}^{(n)}$ from $v=v_{\varepsilon,i}^{(n)}+h_{\varepsilon,i}$
up to $v=v_{\varepsilon,i}^{(n+1)}-h_{\varepsilon,i}$ (see Figure
\ref{fig:Integration_Direction}), calculating the differene between
$\mathcal{E}_{\nearrow}^{(-)}[n;j,i]$ and $\mathcal{E}_{\nearrow}^{(+)}[n;j,i]$
using (\ref{eq:DecreaseOutgoingEnergyCloseToAxis}) for $i<j\le N_{\varepsilon}$
and (\ref{eq:OutgoingEnergyCloseToInfinity}) for $0\le j<i$ (measuring
the change in the energy content of $\mathcal{V}_{i\nearrow}^{(n)}$
as it crosses each of the ingoing beams $\mathcal{V}_{j\nwarrow}^{(n)}$,
$1\le j\le N_{\varepsilon}$), making also use of the equality (\ref{eq:EqualitySuccessiveEnergiesOutgoing}),
we similarly infer that 
\begin{align}
\mathcal{E}_{\mathcal{I}}[n;i] & =\mathcal{E}_{\gamma_{\mathcal{Z}}}[n;i]\cdot\big(1+\sum_{j=0}^{N_{\varepsilon}}O(\varepsilon)\big)+\sum_{j=0}^{i-1}O\Big(\rho_{\varepsilon}^{\frac{3}{2}}\frac{\varepsilon^{(j)}}{\sqrt{-\Lambda}}\Big)+\sum_{j=i+1}^{N_{\varepsilon}}O\Big(\varepsilon\frac{\varepsilon^{(j)}}{\sqrt{-\Lambda}}\Big)=\label{eq:ChangeOfEnergyOutgoing}\\
 & =\mathcal{E}_{\gamma_{\mathcal{Z}}}[n;i]\cdot\big(1+O(N_{\varepsilon}\varepsilon)\big)+O\Big(N_{\varepsilon}(\rho_{\varepsilon}^{\frac{3}{2}}+\varepsilon)\frac{\varepsilon^{(j)}}{\sqrt{-\Lambda}}\Big)=\nonumber \\
 & =\mathcal{E}_{\gamma_{\mathcal{Z}}}[n;i]\cdot\big(1+O(\varepsilon^{\frac{1}{2}})\big)+O\Big(\rho_{\varepsilon}^{\frac{1}{4}}\frac{\varepsilon^{(j)}}{\sqrt{-\Lambda}}\Big).\nonumber 
\end{align}

\item Finally, moving along $\mathcal{V}_{i\nwarrow}^{(n+1)}$ from
$u=v_{\varepsilon,i}^{(n)}+h_{\varepsilon,i}$ up to $u=v_{\varepsilon,0}^{(n+1)}-h_{\varepsilon,i}$
(see Figure \ref{fig:Integration_Direction}), calculating the difference
between $\mathcal{E}_{\nwarrow}^{(-)}[n;i,j]$ and $\mathcal{E}_{\nwarrow}^{(+)}[n;i,j]$
using (\ref{eq:IngoingEnergyCloseToInfinity}) for all $j>i$ (making
use of the equality (\ref{eq:EqualitySuccessiveEnergiesIngoing})),
we obtain: 
\begin{align}
\mathcal{E}_{\nwarrow}^{(-)}[n+1;i,0] & =\mathcal{E}_{\mathcal{I}}[n;i]\cdot\big(1+\sum_{j=i+1}^{N_{\varepsilon}}O(\varepsilon)\big)+\sum_{j=i+1}^{N_{\varepsilon}}O\Big(\varepsilon\frac{\varepsilon^{(i)}}{\sqrt{-\Lambda}}\Big)=\label{eq:ChangeOfEnergyIngoingSecond}\\
 & =\mathcal{E}_{\mathcal{I}}[n;i]\cdot\big(1+O(\varepsilon^{\frac{1}{2}})\big)+O\Big(\varepsilon^{\frac{1}{2}}\frac{\varepsilon^{(i)}}{\sqrt{-\Lambda}}\Big).\nonumber 
\end{align}

\end{enumerate}

Combining (\ref{eq:ChangeOfEnergyIngoingFirst}), (\ref{eq:ChangeOfEnergyOutgoing})
and (\ref{eq:ChangeOfEnergyIngoingSecond}) and using the relations
(\ref{eq:HierarchyOfParameters}) between $\varepsilon$ and $\rho_{\varepsilon}$,
we therefore infer that: 
\begin{equation}
\mathcal{E}_{\nwarrow}^{(-)}[n+1;i,0]=\mathcal{E}_{\nwarrow}^{(-)}[n;i,0]\cdot\exp\Big(\sum_{j=0}^{i-1}\frac{2\mathcal{E}_{\nearrow}^{(-)}[n;i,j]}{r_{n;i.j}}+O(\rho_{\varepsilon}^{\frac{1}{4}})\Big)+O\Big(\rho_{\varepsilon}^{\frac{1}{4}}\frac{\varepsilon^{(i)}}{\sqrt{-\Lambda}}\Big).\label{eq:ChangeOfEnergyAlmostFinal}
\end{equation}
Using the estimates (\ref{eq:BoundsRIntersectionRegioni>j}), (\ref{eq:EstimateGeometryLikeC0}),
(\ref{eq:UpperBoundMuBootstrapDomain}) and the fact that $\Lambda R^{2}=O(\varepsilon)$
on $\{u_{n;i,j}^{(+)}\}\times[v_{n;j+1,j}^{(-)},v_{n;i,j}^{(-)}]$,
we can readily estimate:
\begin{align}
\frac{r_{n;i,j}}{r_{n;j+1,j}}-1 & =\frac{\int_{v_{n;j+1,j}^{(-)}}^{v_{n;i,j}^{(-)}}\partial_{v}r(u_{n;i,j}^{(+)},v)\,dv}{r_{n;j+1,j}}\le\label{eq:ComparableRnij}\\
 & \le\frac{1}{r_{n;j+1,j}}\sup_{\{u_{n;i,j}^{(+)}\}\times[v_{n;j+1,j}^{(-)},v_{n;i,j}^{(-)}]}(\partial_{v}r)\sum_{k=j+1}^{i-1}\rho_{\varepsilon}^{-1}\frac{\varepsilon^{(k)}}{\sqrt{-\Lambda}}\le\nonumber \\
 & \le\exp(\exp(\sigma_{\varepsilon}^{-6}))\rho_{\varepsilon}^{-1}\frac{\sum_{k=j+1}^{i-1}\varepsilon^{(k)}}{\varepsilon^{(j)}}\le\nonumber \\
 & \le\varepsilon^{\frac{1}{2}}.\nonumber 
\end{align}
Using the relation (\ref{eq:ExpressionRnii-1inU}) and the bounds
(\ref{eq:TrivialBoundEnergiesConcentration}) and (\ref{eq:ComparableRnij}),
from (\ref{eq:ChangeOfEnergyAlmostFinal}) we then infer that
\begin{equation}
\mathcal{E}_{\nwarrow}^{(-)}[n+1;i,0]=\mathcal{E}_{\nwarrow}^{(-)}[n;i,0]\cdot\exp\Big(\sum_{j=0}^{i-1}\frac{2\mathcal{E}_{\nearrow}^{(-)}[n;i,j]}{\mathfrak{D}r_{\nearrow}^{(-)}[n;j+1,j+1]}+O(\rho_{\varepsilon}^{\frac{1}{12}})\Big)+O\Big(\rho_{\varepsilon}^{\frac{1}{4}}\frac{\varepsilon^{(i)}}{\sqrt{-\Lambda}}\Big).\label{eq:ChangeOfEnergyFinal}
\end{equation}

\textgreek{A}ssuming that $1\le i\le N_{\varepsilon}$, we will now
repeat the same procedure for the geometric separation $\mathfrak{D}r$
in place of the energy content $\mathcal{E}$ of the beams:

\begin{enumerate}

\item First, moving in the $u$ direction along the strip 
\[
\mathcal{S}_{i\nwarrow}^{(n)}\doteq\{v_{\varepsilon,i-1}^{(n)}+(\rho_{\varepsilon}^{-1}+1)h_{\varepsilon,i-1}\le v\le v_{\varepsilon,i}^{(n)}-(\rho_{\varepsilon}^{-1}+1)h_{\varepsilon,i-1}\},
\]
from $u=v_{\varepsilon,0}^{(n)}$ up to the axis $\gamma_{\mathcal{Z}_{\varepsilon}}$,
calculating the differene between $\mathfrak{D}r_{\nwarrow}^{(-)}[n;i,j]$
and $\mathfrak{D}r_{\nwarrow}^{(+)}[n;i,j]$ using (\ref{eq:DecreaseIngoingRSeparationCloseToAxisi>j+1})
for all $0\le j\le i-1$ and (\ref{eq:DecreaseIngoingRSeparationCloseToAxisi=00003Dj+1})
for $j=i-1$, making also use of the equality (\ref{eq:EqualitySuccessiveRSeparationsIngoing})
between $\mathfrak{D}r_{\nwarrow}^{(-)}[n;i,j]$ and $\mathfrak{D}r_{\nwarrow}^{(+)}[n;i,j-1]$,
we infer 
\begin{align}
\mathfrak{D}r_{\nwarrow}^{(+)}[n;i,i-1] & =\mathfrak{D}r_{\nwarrow}^{(-)}[n;i,0]\cdot\exp\Big(-\sum_{j=0}^{i-2}\frac{2\mathcal{E}_{\nearrow}^{(-)}[n;i,j]}{r_{n;i.j}}+\sum_{j=0}^{i-2}O(\rho_{\varepsilon}^{\frac{3}{2}})+O(\rho_{\varepsilon}^{\frac{3}{4}})\Big)=\label{eq:ChangeOfRSeparationIngoingFirst}\\
 & =\mathfrak{D}r_{\nwarrow}^{(-)}[n;i,0]\cdot\exp\Big(-\sum_{j=0}^{i-2}\frac{2\mathcal{E}_{\nearrow}^{(-)}[n;i,j]}{r_{n;i.j}}+O(N_{\varepsilon}\rho_{\varepsilon}^{\frac{3}{2}})+O(\rho_{\varepsilon}^{\frac{3}{4}})\Big)=\nonumber \\
 & =\mathfrak{D}r_{\nwarrow}^{(-)}[n;i,0]\cdot\exp\Big(-\sum_{j=0}^{i-2}\frac{2\mathcal{E}_{\nearrow}^{(-)}[n;i,j]}{r_{n;i.j}}+O(\rho_{\varepsilon}^{\frac{1}{4}})\Big).\nonumber 
\end{align}

\item Moving in the $v$ direction along along the strip 
\[
\mathcal{S}_{i\nearrow}^{(n)}\doteq\{v_{\varepsilon,i-1}^{(n)}+(\rho_{\varepsilon}^{-1}+1)h_{\varepsilon,i-1}\le u\le v_{\varepsilon,i}^{(n)}-(\rho_{\varepsilon}^{-1}+1)h_{\varepsilon,i-1}\},
\]
 from the axis $\gamma_{\mathcal{Z}_{\varepsilon}}$ up to conformal
infinity $\mathcal{I}_{\varepsilon}$, calculating the differene between
$\mathfrak{D}r_{\nearrow}^{(-)}[n;j,i]$ and $\mathfrak{D}r_{\nearrow}^{(+)}[n;j,i]$
using (\ref{eq:SpecialRSeparationCloseToAxis}) for $j=i$, (\ref{eq:IncreaseOutgoingRSeparationCloseToAxis})
for $i<j\le N_{\varepsilon}$ and (\ref{eq:OutgoingRSeparationCloseToInfinity})
for $0\le j<i$, making also use of the equality (\ref{eq:EqualitySuccessiveRSeparationsOutgoing})
as well as the approximate equality (\ref{eq:EqualityRSeparationsNearAxis})
between $\mathfrak{D}r_{\nearrow}^{(-)}[n;i,i]$ and $\mathfrak{D}r_{\nwarrow}^{(+)}[n;i,i-1]$,
we similarly infer that 
\begin{align}
\mathfrak{D}r_{\nearrow}^{(+)}[n;N_{\varepsilon},i] & =\mathfrak{D}r_{\nearrow}^{(-)}[n;i,i]\cdot\big(1+\sum_{j=0,j\neq i}^{N_{\varepsilon}}O(\varepsilon)+O(\rho_{\varepsilon}^{\frac{3}{4}})\big)=\label{eq:ChangeOfRSeparationOutgoing}\\
 & =\mathfrak{D}r_{\nwarrow}^{(+)}[n;i,i-1]\cdot\big(1+O((N_{\varepsilon}-1)\varepsilon)+O(\rho_{\varepsilon}^{\frac{3}{4}})+O(\rho_{\varepsilon}^{\frac{1}{10}})\big)=\nonumber \\
 & =\mathfrak{D}r_{\nwarrow}^{(+)}[n;i,i-1]\cdot\big(1+O(\rho_{\varepsilon}^{\frac{1}{10}})\big).\nonumber 
\end{align}

\item Finally, moving in the $u$ direction along $\mathcal{S}_{i\nwarrow}^{(n+1)}$
from $\mathcal{I}_{\varepsilon}$ up to $u=v_{\varepsilon,0}^{(n+1)}-h_{\varepsilon,i}$,
calculating the difference between $\mathfrak{D}r_{\nwarrow}^{(-)}[n;j,i]$
and $\mathfrak{D}r_{\nwarrow}^{(+)}[n;j,i]$ using (\ref{eq:SpecialRSeparationCloseToInfinity})
for $j=i$ and (\ref{eq:IngoingRSeparationCloseToInfinity}) for all
$j>i$, making use of the equality (\ref{eq:EqualitySuccessiveRSeparationsIngoing})
and the approximate equality (\ref{eq:EqualityRSeparationsNearInfinity})
between $\mathfrak{D}r_{\nwarrow}^{(-)}[n;i,i]$ and $\mathfrak{D}r_{\nearrow}^{(+)}[n;i-1,i]$,
we obtain: 
\begin{align}
\mathfrak{D}r_{\nwarrow}^{(-)}[n+1;i,0] & =\mathfrak{D}r_{\nwarrow}^{(-)}[n;i,i]\cdot\big(1+\sum_{j=i}^{N_{\varepsilon}}O(\varepsilon)\big)=\label{eq:ChangeOfRSeparationIngoingSecond}\\
 & =\mathfrak{D}r_{\nearrow}^{(+)}[n;i-1,i]\cdot\big(1+O(\rho_{\varepsilon}^{\frac{1}{10}})\big).\nonumber 
\end{align}

\end{enumerate}

Combining (\ref{eq:ChangeOfRSeparationIngoingFirst}), (\ref{eq:ChangeOfRSeparationOutgoing})
and (\ref{eq:ChangeOfRSeparationIngoingSecond}), we obtain that 
\begin{equation}
\mathfrak{D}r_{\nwarrow}^{(-)}[n+1;i,0]=\mathfrak{D}r_{\nwarrow}^{(-)}[n;i,0]\cdot\exp\Big(-\sum_{j=0}^{i-2}\frac{2\mathcal{E}_{\nearrow}^{(-)}[n;i,j]}{r_{n;i.j}}+O(\rho_{\varepsilon}^{\frac{1}{10}})\Big).\label{eq:AlmostFinalRSeparation}
\end{equation}
Using the relation (\ref{eq:ExpressionRnii-1inU}) and the bounds
(\ref{eq:TrivialBoundEnergiesConcentration}) and (\ref{eq:ComparableRnij}),
from (\ref{eq:AlmostFinalRSeparation}) we then infer that
\begin{equation}
\mathfrak{D}r_{\nwarrow}^{(-)}[n+1;i,0]=\mathfrak{D}r_{\nwarrow}^{(-)}[n;i,0]\cdot\exp\Big(-\sum_{j=0}^{i-2}\frac{2\mathcal{E}_{\nearrow}^{(-)}[n;i,j]}{\mathfrak{D}r_{\nearrow}^{(-)}[n;j+1,j+1]}+O(\rho_{\varepsilon}^{\frac{1}{12}})\Big).\label{eq:ChangeOfRSeparationFinal}
\end{equation}

For any $n\in\mathbb{N}$ such that (\ref{eq:AlternativeNonTrivialU})
is satisfied, arguing in exactly the same way as for the proof of
(\ref{eq:ChangeOfEnergyOutgoing}), but moving along $\mathcal{V}_{j\nearrow}^{(n)}$
starting from $v=v_{\varepsilon,i}^{(n)}-h_{\varepsilon,i}$ (instead
of $v=v_{\varepsilon,j}^{(n)}+h_{\varepsilon,j}$) up to $v=v_{\varepsilon,j}^{(n+1)}-h_{\varepsilon,j}$,
we infer that, for any $0\le j\le i-1$: 
\begin{equation}
\mathcal{E}_{\mathcal{I}}[n;j]=\mathcal{E}_{\nearrow}^{(-)}[n;i,j]\cdot\big(1+O(\varepsilon^{\frac{1}{2}})\big)+O\Big(\rho_{\varepsilon}^{\frac{1}{4}}\frac{\varepsilon^{(j)}}{\sqrt{-\Lambda}}\Big).\label{eq:OneMoreAuxiliary}
\end{equation}
Using (\ref{eq:ChangeOfEnergyIngoingSecond}) (for $j$ in place of
$i$) and (\ref{eq:OneMoreAuxiliary}), we therefore infer that 
\begin{equation}
\mathcal{E}_{\nearrow}^{(-)}[n;i,j]=\mathcal{E}_{\nwarrow}^{(-)}[n+1;j,0]\cdot\big(1+O(\varepsilon^{\frac{1}{2}})\big)+O\Big(\rho_{\varepsilon}^{\frac{1}{4}}\frac{\varepsilon^{(j)}}{\sqrt{-\Lambda}}\Big).\label{eq:OneMore...}
\end{equation}
Similarly, using the first line of (\ref{eq:ChangeOfRSeparationOutgoing})
and (\ref{eq:ChangeOfRSeparationIngoingSecond}), we obtain: 
\begin{equation}
\mathfrak{D}r_{\nwarrow}^{(+)}[n;j+1,j+1]=\mathfrak{D}r_{\nwarrow}^{(-)}[n+1;j+1,0]\cdot\big(1+O(\rho_{\varepsilon}^{\frac{1}{10}})\big).\label{eq:TwoMore...}
\end{equation}
Substituting $\mathcal{E}_{\nearrow}^{(-)}[n;i,j]$ and $\mathfrak{D}r_{\nearrow}^{(-)}[n;j+1,j+1]$
in the right hand side of (\ref{eq:ChangeOfEnergyFinal}) with (\ref{eq:OneMore...})
and (\ref{eq:TwoMore...}), respectively, we therefore obtain (\ref{eq:ChangeOfEnergyFinalAgain}).
Similarly, from (\ref{eq:ChangeOfRSeparationFinal}) we infer (\ref{eq:ChangeOfRSeparationFinalAgain}).

\medskip{}

\noindent \emph{Remark. }For any $0\le j_{1}\le i\le N_{\varepsilon}$
and $0\le j_{0}\le j_{1}$, it readily follows from the proof of (\ref{eq:ChangeOfEnergyFinalAgain})
and (\ref{eq:ChangeOfRSeparationFinalAgain}) (after restricting ourselves
to the interactions of the beams taking place only between $u=v_{j_{0},\varepsilon}^{(n)}-h_{\varepsilon,j_{0}}$
and $v=v_{i,\varepsilon}^{(n)}-h_{\varepsilon,i}$) that:
\begin{equation}
\mathcal{E}_{\nwarrow}^{(-)}[n;i,j_{1}]=\mathcal{E}_{\nwarrow}^{(-)}[n;i,j_{0}]\cdot\exp\Big(\sum_{j=j_{0}}^{j_{1}-1}\frac{2\mathcal{E}_{\nearrow}^{(-)}[n;i,j]}{\mathfrak{D}r_{\nearrow}^{(-)}[n;i,j+1]}+O(\rho_{\varepsilon}^{\frac{1}{12}})\Big)+O\Big(\rho_{\varepsilon}^{\frac{1}{4}}\frac{\varepsilon^{(i)}}{\sqrt{-\Lambda}}\Big),\label{eq:ChangeOfEnergyFinalAgainGeneralIngoing}
\end{equation}
\begin{equation}
\mathcal{E}_{\nearrow}^{(-)}[n;i,j_{1}]=\mathcal{E}_{\nwarrow}^{(-)}[n;j_{1},j_{0}]\cdot\exp\Big(\sum_{j=j_{0}}^{j_{1}-1}\frac{2\mathcal{E}_{\nearrow}^{(-)}[n;i,j]}{\mathfrak{D}r_{\nearrow}^{(-)}[n;i,j+1]}+O(\rho_{\varepsilon}^{\frac{1}{12}})\Big)+O\Big(\rho_{\varepsilon}^{\frac{1}{4}}\frac{\varepsilon^{(j_{1})}}{\sqrt{-\Lambda}}\Big)\label{eq:ChangeOfEnergyFinalAgainGeneralOutgoing}
\end{equation}
(with the convention that, when $i=j_{1}$, $\mathcal{E}_{\nwarrow}^{(-)}[n;i,i]=\mathcal{E}_{\nwarrow}^{(-)}[n;i,i]=\mathcal{E}_{\gamma_{\mathcal{Z}}}[n;i]$)
and 
\begin{equation}
\mathfrak{D}r_{\nwarrow}^{(-)}[n;i,j_{1}]=\mathfrak{D}r_{\nwarrow}^{(-)}[n;i,j_{0}]\cdot\exp\Big(-\sum_{j=j_{0}}^{j_{1}-2}\frac{2\mathcal{E}_{\nearrow}^{(-)}[n;i,j]}{\mathfrak{D}r_{\nearrow}^{(-)}[n;i,j+1]}+O(\rho_{\varepsilon}^{\frac{1}{12}})\Big),\text{ provided }i>0,\label{eq:ChangeOfRSeparationFinalAgainGeneralIngoing}
\end{equation}
\begin{equation}
\mathfrak{D}r_{\nearrow}^{(-)}[n;i,j_{1}]=\mathfrak{D}r_{\nwarrow}^{(-)}[n;j_{1},j_{0}]\cdot\exp\Big(-\sum_{j=j_{0}}^{j_{1}-2}\frac{2\mathcal{E}_{\nearrow}^{(-)}[n;i,j]}{\mathfrak{D}r_{\nearrow}^{(-)}[n;i,j+1]}+O(\rho_{\varepsilon}^{\frac{1}{12}})\Big),\text{ provided }j_{1}>0\label{eq:ChangeOfRSeparationFinalAgainGeneralOutgoing}
\end{equation}
(with the convention that, when $i=j_{1}$, $\mathfrak{D}r_{\nwarrow}^{(-)}[n;i,i]=\mathfrak{D}r_{\nearrow}^{(-)}[n;i,j_{1}]$). 

Similarly, for $0\le i\le j_{1}\le N_{\varepsilon}$ and $0\le j_{0}\le i$,
\begin{equation}
\mathcal{E}_{\nwarrow}^{(-)}[n;i,j_{1}]=\mathcal{E}_{\nwarrow}^{(-)}[n;i,j_{0}]\cdot\exp\Big(\sum_{j=j_{0}}^{i-1}\frac{2\mathcal{E}_{\nearrow}^{(-)}[n;i,j]}{\mathfrak{D}r_{\nearrow}^{(-)}[n;i,j+1]}+O(\rho_{\varepsilon}^{\frac{1}{12}})\Big)+O\Big(\rho_{\varepsilon}^{\frac{1}{4}}\frac{\varepsilon^{(i)}}{\sqrt{-\Lambda}}\Big),\label{eq:ChangeOfEnergyFinalAgainGeneralIngoing-1}
\end{equation}
\begin{equation}
\mathcal{E}_{\nearrow}^{(-)}[n;i,j_{1}]=\mathcal{E}_{\nwarrow}^{(-)}[n;j_{1},j_{0}]\cdot\exp\Big(\sum_{j=j_{0}}^{j_{1}-1}\frac{2\mathcal{E}_{\nearrow}^{(-)}[n;i,j]}{\mathfrak{D}r_{\nearrow}^{(-)}[n;i,j+1]}+O(\rho_{\varepsilon}^{\frac{1}{12}})\Big)+O\Big(\rho_{\varepsilon}^{\frac{1}{4}}\frac{\varepsilon^{(j_{1})}}{\sqrt{-\Lambda}}\Big)\label{eq:ChangeOfEnergyFinalAgainGeneralOutgoing-1}
\end{equation}
(with the convention that, when $i=j_{1}$, $\mathcal{E}_{\nwarrow}^{(-)}[n;i,i]=\mathcal{E}_{\nwarrow}^{(-)}[n;i,i]=\mathcal{E}_{\mathcal{I}}[n;i]$)
and
\begin{equation}
\mathfrak{D}r_{\nwarrow}^{(-)}[n;i,j_{1}]=\mathfrak{D}r_{\nwarrow}^{(-)}[n;i,j_{0}]\cdot\exp\Big(-\sum_{j=j_{0}}^{i-2}\frac{2\mathcal{E}_{\nearrow}^{(-)}[n;i,j]}{\mathfrak{D}r_{\nearrow}^{(-)}[n;i,j+1]}+O(\rho_{\varepsilon}^{\frac{1}{12}})\Big),\text{ provided }i>0,\label{eq:ChangeOfRSeparationFinalAgainGeneralIngoing-1}
\end{equation}
\begin{equation}
\mathfrak{D}r_{\nearrow}^{(-)}[n;i,j_{1}]=\mathfrak{D}r_{\nwarrow}^{(-)}[n;j_{1},j_{0}]\cdot\exp\Big(-\sum_{j=j_{0}}^{j_{1}-2}\frac{2\mathcal{E}_{\nearrow}^{(-)}[n;i,j]}{\mathfrak{D}r_{\nearrow}^{(-)}[n;i,j+1]}+O(\rho_{\varepsilon}^{\frac{1}{12}})\Big),\text{ provided }j_{1}>0\label{eq:ChangeOfRSeparationFinalAgainGeneralOutgoing-1}
\end{equation}
(with the convention that, when $i=j_{1}$, $\mathfrak{D}r_{\nearrow}^{(-)}[n;i,j_{1}]=\mathfrak{D}r_{\nwarrow}^{(-)}[n;i,i]$). 

The same relations also hold for $\widetilde{\mathcal{E}}^{(\pm)}$,
$\widetilde{\mathfrak{D}}r^{(\pm)}$ in place of $\mathcal{E}^{(\pm)}$,
$\mathfrak{D}r^{(\pm)}$.

\medskip{}

\noindent \emph{Proof of (\ref{eq:EqualEnergyFluxOutgoingTau})\textendash }(\ref{eq:EqualEnergyFluxFinalIngoingTau})\emph{.}
In order to establish (\ref{eq:EqualEnergyFluxOutgoingTau}) and (\ref{eq:EqualRSeparationOutgoingTau}),
we will use the fact that, for any $n\in\mathbb{N}$ and $0\le j\le N_{\varepsilon}-1$
such that (\ref{eq:NoonTrivialN}) and (\ref{eq:IntersectionDomainInTau})
hold, we have for all $0\le\bar{j}\le j$
\begin{equation}
\widetilde{\mathcal{E}}_{\nearrow}^{(-)}[n;N_{\varepsilon},\bar{j}]=\mathcal{E}_{\nwarrow}^{(-)}[n;\bar{j},0]\cdot\exp\Big(\sum_{k=0}^{\bar{j}-1}\frac{2\widetilde{\mathcal{E}}_{\nearrow}^{(-)}[n;N_{\varepsilon},k]}{\widetilde{\mathfrak{D}}r_{\nearrow}^{(-)}[n;N_{\varepsilon},k+1]}+O(\rho_{\varepsilon}^{\frac{1}{12}})\Big)+O\Big(\rho_{\varepsilon}^{\frac{1}{4}}\frac{\varepsilon^{(\bar{j})}}{\sqrt{-\Lambda}}\Big)\label{eq:ChangeOfEnergyFinalAgainTau}
\end{equation}
and, for all $1\le\bar{j}\le j$
\begin{equation}
\widetilde{\mathfrak{D}}r_{\nearrow}^{(-)}[n;N_{\varepsilon},\bar{j}]=\mathfrak{D}r_{\nwarrow}^{(-)}[n;\bar{j},0]\cdot\exp\Big(-\sum_{k=0}^{\bar{j}-2}\frac{2\widetilde{\mathcal{E}}_{\nearrow}^{(-)}[n;N_{\varepsilon},k]}{\widetilde{\mathfrak{D}}r_{\nearrow}^{(-)}[n;N_{\varepsilon},k+1]}+O(\rho_{\varepsilon}^{\frac{1}{12}})\Big).\label{eq:ChangeOfRSeparationFinalAgainTau}
\end{equation}
Note that, if (\ref{eq:NoonTrivialN}) holds, then it is also necessarily
true that 
\begin{equation}
\widetilde{\mathcal{R}}_{N_{\varepsilon};\bar{j}}^{(n)}\subset\mathcal{T}_{\varepsilon}^{+}\text{ for all }0\le\bar{j}\le j,
\end{equation}
and hence all the terms in the relations (\ref{eq:ChangeOfEnergyFinalAgainTau})\textendash (\ref{eq:ChangeOfRSeparationFinalAgainTau})
are well defined. The relations (\ref{eq:ChangeOfEnergyFinalAgainTau})
and (\ref{eq:ChangeOfRSeparationFinalAgainTau}) are immediate corollaries
of (\ref{eq:ChangeOfEnergyFinalAgainGeneralOutgoing}) and (\ref{eq:ChangeOfRSeparationFinalAgainGeneralOutgoing})
(for $\widetilde{\mathcal{E}}^{(\pm)}$, $\widetilde{\mathfrak{D}}r^{(\pm)}$
in place of $\mathcal{E}^{(\pm)}$, $\mathfrak{D}r^{(\pm)}$) with
$i=N_{\varepsilon}$, $j_{1}=\bar{j}$ and $j_{0}=0$, using also
the fact that, since (\ref{eq:NoonTrivialN}) holds, 
\begin{align}
\widetilde{\mathcal{E}}_{\nwarrow}^{(-)}[n;\bar{j},0] & =\mathcal{E}_{\nwarrow}^{(-)}[n;\bar{j},0]\text{ for all }0\le\bar{j}\le j\text{ and }\label{eq:EqualTildeNonTilde}\\
\widetilde{\mathfrak{D}}r_{\nwarrow}^{(-)}[n;\bar{j},0] & =\mathfrak{D}r_{\nwarrow}^{(-)}[n;\bar{j},0]\Big(1+O(\rho_{\varepsilon}^{\frac{1}{9}})\Big)\text{ for all }1\le\bar{j}\le j\nonumber 
\end{align}
(which is inferred from the definition of $\mathcal{E}_{\nwarrow}^{(-)}$,
$\mathfrak{D}r_{\nwarrow}^{(-)}$, $\widetilde{\mathcal{E}}_{\nwarrow}^{(-)}$and
$\widetilde{\mathfrak{D}}r_{\nwarrow}^{(-)}$ in Section \ref{subsec:Some-basic-geometric-constructions},
as well as the fact that
\[
\tilde{m}(v_{\varepsilon,0}^{(n)}-h_{\varepsilon,0},v_{\varepsilon,\bar{j}}^{(n)}\pm h_{\varepsilon,\bar{j}})=\tilde{m}(v_{\varepsilon,0}^{(n)}-h_{\varepsilon,0},v_{\varepsilon,\bar{j}}^{(n)}\pm\tilde{h}_{\varepsilon,\bar{j}}),
\]
since the support of $T_{\mu\nu}[f]$ in $\{0\le u\le v_{\varepsilon,0}^{(n)}-h_{\varepsilon,0}\}\subset\mathcal{U}_{\varepsilon}^{+}$
is contained in $\cup_{k\in\mathbb{N}}\cup_{i=0}^{N_{\varepsilon}}\mathcal{V}_{i}^{(k)}$).

In order to infer (\ref{eq:EqualEnergyFluxOutgoingTau})\textendash (\ref{eq:EqualRSeparationOutgoingTau})
from (\ref{eq:ChangeOfEnergyFinalAgainTau})\textendash (\ref{eq:ChangeOfRSeparationFinalAgainTau}),
we will argue similarly as in the case of (\ref{eq:EqualEnergyFlux})\textendash (\ref{eq:EqualRSeparation}):
Defining the quantities 
\begin{align*}
\tilde{e}_{\bar{j}}[n+1] & \doteq\frac{\widetilde{\mathcal{E}}_{\nearrow}^{(-)}[n;N_{\varepsilon},\bar{j}]-\mathcal{E}_{\bar{j}}[n+1]}{\varepsilon^{(\bar{j})}}\sqrt{-\Lambda}\text{ for }0\le\bar{j}\le j,\\
\tilde{r}_{\bar{j}}[n+1] & \doteq\frac{\widetilde{\mathfrak{D}}r_{\nearrow}^{(-)}[n;N_{\varepsilon},\bar{j}]}{R_{\bar{j}}[n+1]}\text{ for }1\le\bar{j}\le j,\\
\tilde{\mu}_{\bar{j}}[n+1] & \doteq\rho_{\varepsilon}^{-1}\Big(\frac{2\widetilde{\mathcal{E}}_{\nearrow}^{(-)}[n;N_{\varepsilon},\bar{j}]}{\widetilde{\mathfrak{D}}r_{\nearrow}^{(-)}[n;N_{\varepsilon},\bar{j}+1]}-\mu_{\bar{j}}[n+1]\Big)\text{ for }0\le\bar{j}\le j,
\end{align*}
the relations (\ref{eq:EqualEnergyFluxOutgoingTau}) and (\ref{eq:EqualRSeparationOutgoingTau})
will follow by showing that 
\begin{equation}
\max_{0\le\bar{j}\le j}\big|\tilde{e}_{\bar{j}}[n+1]\big|,\text{ }\max_{1\le\bar{j}\le j}\big|\tilde{r}_{\bar{j}}[n+1]-1\big|,\text{ }\max_{0\le\bar{j}\le j}\big|\tilde{\mu}_{\bar{j}}[n+1]\big|\le\rho_{\varepsilon}^{\frac{1}{17}}.\label{eq:TildeQuantitiesBound}
\end{equation}

We will argue by induction on $\bar{j}$: For any $0\le\bar{j}\le j$,
we will show that, if 
\begin{equation}
\max_{0\le k\le\bar{j}-1}\big|\tilde{\mu}_{k}[n+1]\big|\le\rho_{\varepsilon}^{\frac{1}{16}},\label{eq:InductiveMuKTilde}
\end{equation}
then 
\begin{equation}
\big|\tilde{\mu}_{\bar{j}}[n+1]\big|\le\rho_{\varepsilon}^{\frac{1}{16}}.\label{eq:BoundToShowInductionMuTilde}
\end{equation}
In view of the relations (\ref{eq:ChangeOfEnergyFinalAgainTau})\textendash (\ref{eq:ChangeOfRSeparationFinalAgainTau})
for $\widetilde{\mathcal{E}}_{\nearrow}^{(-)}[n;N_{\varepsilon},\bar{j}]$,
$\widetilde{\mathfrak{D}}r_{\nearrow}^{(-)}[n;N_{\varepsilon},\bar{j}]$
and the relation (\ref{eq:RecursiveFormulaMu}) for $\mu_{\bar{j}}[n]$,
we infer using the bounds (\ref{eq:TrivialBoundEnergiesConcentration})\textendash (\ref{eq:TrivialBoundsRSeparation})
(for $\widetilde{\mathcal{E}}^{(\pm)}$, $\mathfrak{\widetilde{D}}r^{(\pm)}$
and $\delta_{\varepsilon}$ in place of $\mathcal{E}^{(\pm)}$, $\mathfrak{D}r^{(\pm)}$
and $\sigma_{\varepsilon}$, respectively), the bound 
\begin{equation}
\sum_{k=0}^{j-1}\frac{2\widetilde{\mathcal{E}}_{\nearrow}^{(-)}[n;N_{\varepsilon},k]}{\widetilde{\mathfrak{D}}r_{\nearrow}^{(-)}[n;N_{\varepsilon},k+1]}\le\exp(\exp(\delta_{\varepsilon}^{-6}))\text{ for all }0\le k\le n\label{eq:BoundSumEnergiesByDelta-1}
\end{equation}
(following from (\ref{eq:TrivialBoundEnergiesConcentration})\textendash (\ref{eq:TrivialBoundsRSeparation})
for $\widetilde{\mathcal{E}}^{(\pm)}$, $\mathfrak{\widetilde{D}}r^{(\pm)}$
and the fact that $N_{\varepsilon}=\rho_{\varepsilon}^{-1}\exp\big(e^{\delta_{\varepsilon}^{-15}}\big)$),
the bound 
\begin{equation}
\sum_{k=0}^{\bar{j}-1}\tilde{\mu}_{k}[n+1]\le\rho_{\varepsilon}^{-1+\frac{1}{16}}\exp\big(e^{\delta_{\varepsilon}^{-15}}\big)
\end{equation}
 (following from the inductive assumption (\ref{eq:InductiveMuKTilde}))
and the estimate (\ref{eq:MuBarBoundFinal}) for $\bar{\mu}_{\bar{j}}[n]$
established previously that:
\begin{align}
\big|\tilde{\mu}_{\bar{j}}[n+ & 1]\big|=\label{eq:BoundMuTildeForGronwal}\\
 & =\rho_{\varepsilon}^{-1}\Bigg|\frac{2\mathcal{E}_{\nwarrow}^{(-)}[n;\bar{j},0]}{\mathfrak{D}r_{\nwarrow}^{(-)}[n;\bar{j}+1,0]}\exp\Big(2\sum_{k=0}^{\bar{j}-1}\frac{2\widetilde{\mathcal{E}}_{\nearrow}^{(-)}[n;N_{\varepsilon},k]}{\widetilde{\mathfrak{D}}r_{\nearrow}^{(-)}[n;N_{\varepsilon},k+1]}+O(\rho_{\varepsilon}^{\frac{1}{12}})\Big)+O(\rho_{\varepsilon}^{1+\frac{1}{6}})-\nonumber\\
 & \hphantom{=\rho_{\varepsilon}^{-1}\Bigg|\frac{2\mathcal{E}_{\nwarrow}^{(-)}[n;\bar{j},0]}{\mathfrak{D}r_{\nwarrow}^{(-)}[n;\bar{j}+1,0]}\exp\Big(2\sum_{k=0}^{\bar{j}-1}\frac{2\widetilde{\mathcal{E}}_{\nearrow}^{(-)}[n;N_{\varepsilon},k]}{\widetilde{\mathfrak{D}}r_{\nearrow}^{(-)}[n;N_{\varepsilon},k+1]}}-\mu_{\bar{j}}[n]\exp\Big(2\sum_{k=0}^{\bar{j}-1}\mu_{k}[n+1]\Big)\Bigg|=\nonumber\\
 & =\exp\Big(2\sum_{k=0}^{\bar{j}-1}\frac{2\widetilde{\mathcal{E}}_{\nearrow}^{(-)}[n;N_{\varepsilon},k]}{\widetilde{\mathfrak{D}}r_{\nearrow}^{(-)}[n;N_{\varepsilon},k+1]}\Big)\rho_{\varepsilon}^{-1}\Bigg|\frac{2\mathcal{E}_{\nwarrow}^{(-)}[n;\bar{j},0]}{\mathfrak{D}r_{\nwarrow}^{(-)}[n;\bar{j}+1,0]}(1+O(\rho_{\varepsilon}^{\frac{1}{12}}))-\nonumber \\
 & \hphantom{=\exp\Big(2\sum_{k=0}^{\bar{j}-1}\frac{2\widetilde{\mathcal{E}}_{\nearrow}^{(-)}[n;N_{\varepsilon},k]}{\widetilde{\mathfrak{D}}r_{\nearrow}^{(-)}[n;N_{\varepsilon},k+1]}\Big)\rho_{\varepsilon}^{-1}\Bigg|}-\mu_{\bar{j}}[n]\exp\Big(-2\rho_{\varepsilon}\sum_{k=0}^{\bar{j}-1}\tilde{\mu}_{k}[n+1]\Big)+O(\rho_{\varepsilon}^{1+\frac{1}{6}})\Bigg|=\nonumber\\
 & =\exp\big(O(\exp\big(2e^{\delta_{\varepsilon}^{-15}}\big))\big)\Bigg|\rho_{\varepsilon}^{-1}\Big(\frac{2\mathcal{E}_{\nwarrow}^{(-)}[n;\bar{j},0]}{\mathfrak{D}r_{\nwarrow}^{(-)}[n;\bar{j}+1,0]}-\mu_{\bar{j}}[n](1-2\rho_{\varepsilon}\sum_{k=0}^{\bar{j}-1}\tilde{\mu}_{k}[n+1]\Big)+O(\rho_{\varepsilon}^{\frac{1}{13}})\Bigg|\le\nonumber \\
 & \le\bar{\delta}_{\varepsilon}^{-\frac{1}{2}}\Big\{|\bar{\mu}_{\bar{j}}[n]|+\rho_{\varepsilon}\sum_{k=0}^{\bar{j}-1}|\tilde{\mu}_{k}[n+1]|+O(\rho_{\varepsilon}^{\frac{1}{13}})\Big\}\le\nonumber \\
 & \le\bar{\delta}_{\varepsilon}^{-1}\Big\{\rho_{\varepsilon}\sum_{k=0}^{\bar{j}-1}|\tilde{\mu}_{k}[n+1]|+\rho_{\varepsilon}^{\frac{1}{15}}\Big\},\nonumber 
\end{align}
where $\bar{\delta}_{\varepsilon}$ was defined in terms of $\delta_{\varepsilon}$
by (\ref{eq:HierarchyOfParameters}). Note that, for $\bar{j}=0$,
from (\ref{eq:BoundMuTildeForGronwal}) we infer that 
\begin{equation}
\big|\tilde{\mu}_{0}[n+1]\big|\le\bar{\delta}_{\varepsilon}^{-1}\rho_{\varepsilon}^{\frac{1}{15}}.\label{eq:InitialBoundMuTilde}
\end{equation}
In general, for $0\le\bar{j}\le j$, applying a Gronwal-type inequality
in the $\bar{j}$ variable, from (\ref{eq:BoundMuTildeForGronwal})
and (\ref{eq:InitialBoundMuTilde}) we infer that:
\begin{equation}
\big|\tilde{\mu}_{\bar{j}}[n+1]\big|\le\exp\Big(\bar{\delta}_{\varepsilon}^{-1}N_{\varepsilon}\rho_{\varepsilon}\Big)\rho_{\varepsilon}^{\frac{1}{15}}\le\rho_{\varepsilon}^{\frac{1}{16}},
\end{equation}
thus establishing (\ref{eq:BoundToShowInductionMuTilde}). As a result,
\begin{equation}
\max_{0\le\bar{j}\le j}\big|\tilde{\mu}_{\bar{j}}[n+1]\big|\le\rho_{\varepsilon}^{\frac{1}{16}}.\label{eq:BoundMuTildeImproved}
\end{equation}

From the relations (\ref{eq:ChangeOfEnergyFinalAgainTau})\textendash (\ref{eq:ChangeOfRSeparationFinalAgainTau})
for $\widetilde{\mathcal{E}}_{\nearrow}^{(-)}[n;N_{\varepsilon},\bar{j}]$,
$\widetilde{\mathfrak{D}}r_{\nearrow}^{(-)}[n;N_{\varepsilon},\bar{j}]$
and (\ref{eq:FormulaRecursiveEnergy})\textendash (\ref{eq:FormulaRecursiveDr})
for $\mathcal{E}_{\bar{j}}[n+1]$, $R_{\bar{j}}[n+1]$, in view of
the bounds (\ref{eq:TrivialBoundEnergiesConcentration})\textendash (\ref{eq:TrivialBoundsRSeparation})
(for $\widetilde{\mathcal{E}}^{(\pm)}$, $\widetilde{\mathfrak{D}}r^{(\pm)}$),
the bound (\ref{eq:BoundSumEnergiesByDelta-1}), the bound (\ref{eq:BoundMuTildeImproved}),
as well as the approximate equalities (\ref{eq:EqualEnergyFlux})
and (\ref{eq:EqualRSeparation}) between $\mathcal{E}_{\nwarrow}^{(-)}[n;\bar{j},0]$,
$\mathfrak{D}r_{\nwarrow}^{(-)}[n;\bar{j},0]$ and $\mathcal{E}_{\bar{j}}[n]$,
$R_{\bar{j}}[n]$, respectively, we can estimate for any $0\le\bar{j}\le j$
(arguing similarly as for the derivation of (\ref{eq:AlmostGronwale})
and (\ref{eq:AlmostGronwalR})):
\begin{align}
|\tilde{e}_{\bar{j}} & [n+1]|=\label{eq:AlmostGronwaleTau}\\
 & =\Bigg|\frac{\sqrt{-\Lambda}}{\varepsilon^{(\bar{j})}}\mathcal{E}_{\nwarrow}^{(-)}[n;\bar{j},0]\cdot\exp\Big(\sum_{k=0}^{\bar{j}-1}\frac{2\widetilde{\mathcal{E}}_{\nearrow}^{(-)}[n;N_{\varepsilon},k]}{\widetilde{\mathfrak{D}}r_{\nearrow}^{(-)}[n;N_{\varepsilon},k+1]}+O(\rho_{\varepsilon}^{\frac{1}{12}})\Big)+O(\rho_{\varepsilon}^{\frac{1}{4}})-\frac{\sqrt{-\Lambda}}{\varepsilon^{(\bar{j})}}\mathcal{E}_{i}[n]\cdot\exp\Big(\sum_{k=0}^{\bar{j}-1}\mu_{k}[n+1]\Big)\Bigg|=\nonumber\\
 & =\exp\Big(\sum_{k=0}^{\bar{j}-1}\frac{2\widetilde{\mathcal{E}}_{\nearrow}^{(-)}[n;N_{\varepsilon},k]}{\widetilde{\mathfrak{D}}r_{\nearrow}^{(-)}[n;N_{\varepsilon},k+1]}\Big)\Bigg|(1+O(\rho_{\varepsilon}^{\frac{1}{12}}))\frac{\sqrt{-\Lambda}}{\varepsilon^{(\bar{j})}}\mathcal{E}_{\nwarrow}^{(-)}[n;\bar{j},0]-\nonumber \\
 & \hphantom{=\exp\Big(\sum_{k=0}^{\bar{j}-1}\frac{2\widetilde{\mathcal{E}}_{\nearrow}^{(-)}[n;N_{\varepsilon},k]}{\widetilde{\mathfrak{D}}r_{\nearrow}^{(-)}[n;N_{\varepsilon},k+1]}\Big)\Bigg|}-\frac{\sqrt{-\Lambda}}{\varepsilon^{(\bar{j})}}\mathcal{E}_{\bar{j}}[n]\cdot\exp\Big(-\rho_{\varepsilon}\sum_{k=0}^{\bar{j}-1}\tilde{\mu}_{\bar{j}}[\bar{n}+1]\Big)+O(\rho_{\varepsilon}^{\frac{1}{4}})\Bigg|\le\nonumber\\
 & \le\exp\big(\exp(\exp(2\delta_{\varepsilon}^{-15}))\big)\Bigg|\frac{\sqrt{-\Lambda}}{\varepsilon^{(\bar{j})}}\Big(\mathcal{E}_{\nwarrow}^{(-)}[n;\bar{j},0]-\mathcal{E}_{\bar{j}}[n]\Big)+\nonumber \\
 & \hphantom{\le\exp\big(\exp(\exp(2\delta_{\varepsilon}^{-15}))\big)\Bigg|}+O(\rho_{\varepsilon}^{\frac{1}{12}})\frac{\sqrt{-\Lambda}}{\varepsilon^{(i)}}\mathcal{E}_{\nwarrow}^{(-)}[n;\bar{j},0]+O(\rho_{\varepsilon}^{\frac{1}{15}})\frac{\sqrt{-\Lambda}}{\varepsilon^{(i)}}\mathcal{E}_{i}[n]+O(\rho_{\varepsilon}^{\frac{1}{4}})\Bigg|\le\nonumber\\
 & \le\exp\big(\exp(\exp(4\delta_{\varepsilon}^{-15}))\big)\Big(\rho_{\varepsilon}^{\frac{1}{16}}+\exp(\exp(\sigma_{\varepsilon}^{-6}))\rho_{\varepsilon}^{\frac{1}{15}}\Big)\le\nonumber \\
 & \le\bar{\delta}_{\varepsilon}^{-1}\rho_{\varepsilon}^{\frac{1}{16}}\nonumber 
\end{align}
and, for any $1\le\bar{j}\le j$:
\begin{align}
|\tilde{r}_{i}[n+1]-1| & =\Bigg|\frac{\mathfrak{D}r_{\nwarrow}^{(-)}[n;\bar{i},0]\cdot\exp\Big(-\sum_{k=0}^{\bar{j}-2}\frac{2\widetilde{\mathcal{E}}_{\nearrow}^{(-)}[n;N_{\varepsilon},k]}{\widetilde{\mathfrak{D}}r_{\nearrow}^{(-)}[n;N_{\varepsilon},k+1]}+O(\rho_{\varepsilon}^{\frac{1}{12}})\Big)}{R_{\bar{j}}[n]\cdot\exp\Big(-\sum_{k=0}^{\bar{j}-1}\mu_{k}[n+1]\Big)}-1\Bigg|=\label{eq:AlmostGronwalRTau}\\
 & =\Bigg|\frac{\mathfrak{D}r_{\nwarrow}^{(-)}[n;\bar{i},0]}{R_{\bar{j}}[n]}\cdot\exp\Big(-\rho_{\varepsilon}\sum_{k=0}^{\bar{j}-1}\tilde{\mu}_{k}[n+1]+O(\rho_{\varepsilon}^{\frac{1}{12}})\Big)-1\Bigg|=\nonumber \\
 & =\Bigg|\Big(1+O(\rho_{\varepsilon}^{\frac{1}{16}})\Big)\cdot\exp\Big(O\big(\exp(e^{\delta_{\varepsilon}^{-15}})\rho_{\varepsilon}^{\frac{1}{15}}\big)\Big)-1\Bigg|\le\nonumber \\
 & \le\bar{\delta}_{\varepsilon}^{-1}\rho_{\varepsilon}^{\frac{1}{16}}.\nonumber 
\end{align}
From (\ref{eq:AlmostGronwaleTau}), (\ref{eq:AlmostGronwalRTau})
and (\ref{eq:BoundMuTildeImproved}), in view of the relation (\ref{eq:HierarchyOfParameters})
between $\rho_{\varepsilon}$, $\delta_{\varepsilon}$ and $\sigma_{\varepsilon}$,
we readily obtain (\ref{eq:TildeQuantitiesBound}). Thus, we infer
(\ref{eq:EqualEnergyFluxOutgoingTau}) and (\ref{eq:EqualRSeparationOutgoingTau}).

Using the relation (\ref{eq:ChangeOfEnergyFinalAgainGeneralIngoing})
for $i=N_{\varepsilon}$, $j_{0}=0$, $j_{1}=N_{\varepsilon}-1$ with
$\widetilde{\mathcal{E}}$, $\widetilde{\mathfrak{D}}r$ in place
of $\mathcal{E}$, $\mathfrak{D}r$, the relation (\ref{eq:IncreaseIngoingEnergyCloseToAxis})
for $i=N_{\varepsilon}$, $j=N_{\varepsilon}-1$ with $\widetilde{\mathcal{E}}$
in place of $\mathcal{E}$, as well as the relation (\ref{eq:ExpressionRnii-1inU})
for $i=N_{\varepsilon}-1$ with $\tilde{r}_{n;i+1,i}$, $\widetilde{\mathfrak{D}}r$
in place of $r_{n;i+1,i}$, $\mathfrak{D}r$, we readily infer that:
\begin{equation}
\widetilde{\mathcal{E}}_{\nwarrow}^{(+)}[n;N_{\varepsilon},N_{\varepsilon}-1]=\widetilde{\mathcal{E}}_{\nwarrow}^{(-)}[n;N_{\varepsilon},0]\cdot\exp\Big(\sum_{j=0}^{N_{\varepsilon}-1}\frac{2\widetilde{\mathcal{E}}_{\nearrow}^{(-)}[n;N_{\varepsilon},j]}{\widetilde{\mathfrak{D}}r_{\nearrow}^{(-)}[n;N_{\varepsilon},j+1]}+O(\rho_{\varepsilon}^{\frac{1}{12}})\Big)+O\Big(\rho_{\varepsilon}^{\frac{1}{4}}\frac{\varepsilon^{(N_{\varepsilon})}}{\sqrt{-\Lambda}}\Big).\label{eq:ChangeOfEnergyFinalAgainGeneralIngoing-2-1}
\end{equation}
The relation (\ref{eq:EqualEnergyFluxFinalIngoingTau}) now readily
follows from (\ref{eq:ChangeOfEnergyFinalAgainGeneralIngoing-2-1})
using (\ref{eq:EqualEnergyFluxOutgoingTau}), (\ref{eq:EqualRSeparationOutgoingTau})
and (\ref{eq:EqualTildeNonTilde}) for the right hand side of (\ref{eq:ChangeOfEnergyFinalAgainGeneralIngoing-2-1}),
as well as the relation (\ref{eq:FormulaRecursiveEnergy}) for $\mathcal{E}_{N_{\varepsilon}}[n+1]$
and the fact that (\ref{eq:EqualEnergyFlux}) holds for $\mathcal{E}_{\nwarrow}^{(-)}[n;N_{\varepsilon},0]$.
\end{proof}

\subsection{\label{subsec:Control-of-the-evolution_ERM} Control of the evolution
in terms of $\mathcal{E}_{i}[n]$, $R_{i}[n]$, $\mu_{i}[n]$}

In this section, we will establish some additional bounds on various
quantities related to the geometry of $(\mathcal{U}_{max}^{(\varepsilon)};r,\Omega^{2},f_{\varepsilon})$
in terms of the quantities $\mathcal{E}_{i}[n]$, $R_{i}[n]$ and
$\mu_{i}[n]$. These bounds will enable us to obtain a priori control
of the evolution of $(r_{/}^{(\varepsilon)},(\Omega_{/}^{(\varepsilon)})^{2},\bar{f}_{/}^{(\varepsilon)})$
by estimating the growth rate of solutions to the recursive systems
\ref{eq:RecursiveFormulaMu} and \ref{eq:FormulaRecursiveEnergy}\textendash \ref{eq:FormulaRecursiveDr}.

The following result can be viewed as a supplement to Proposition
\ref{prop:TotalEnergyChange}, providing us with additional bounds
on the energy content and the geometric separation of the beams on
the regions $\mathcal{R}_{i;j}^{(n)}$ (not necessarily with $j=0$
or $i=N_{\varepsilon}$):
\begin{lem}
\label{lem:ComparisonEstimatesEnergiesAndRSeparations} For any $n\in\mathbb{N}$
such that 
\begin{equation}
\{0\le u\le v_{\varepsilon,0}^{(n)}-h_{\varepsilon,0}\}\cap\big\{ u<v<u+\sqrt{-\frac{3}{\Lambda}}\pi\big\}\subset\mathcal{U}_{\varepsilon}^{+}\label{eq:NoonTrivialN-1}
\end{equation}
and any $0\le i,j\le N_{\varepsilon}$, such that $\mathcal{R}_{i;j}^{(n)}\subset\mathcal{U}_{\varepsilon}^{+}$,
if $i\neq j$, and $\mathcal{R}_{i;\gamma_{\mathcal{Z}}}^{(n)},\mathcal{R}_{i;\mathcal{I}}^{(n)}\subset\mathcal{U}_{\varepsilon}^{+}$,
if $i=j$,\footnote{Note that $\mathcal{R}_{i;j}^{(n)}$, $\mathcal{R}_{i;\gamma_{\mathcal{Z}}}^{(n)}$
and $\mathcal{R}_{i;\mathcal{I}}^{(n)}$ are contained in $\{u\ge v_{\varepsilon,0}^{(n)}-h_{\varepsilon,0}\}$.} we can estimate: 
\begin{align}
\mathcal{E}_{\nwarrow}^{(\pm)}[n;i,j] & \le\mathcal{E}_{i}[n+1]+\rho_{\varepsilon}^{\frac{1}{18}}\frac{\varepsilon^{(i)}}{\sqrt{-\Lambda}},\label{eq:UpperBoundEnergiesFromE}\\
\mathcal{E}_{\nearrow}^{(\pm)}[n;i,j] & \le\mathcal{E}_{j}[n+1]+\rho_{\varepsilon}^{\frac{1}{18}}\frac{\varepsilon^{(j)}}{\sqrt{-\Lambda}}\nonumber 
\end{align}
(if $i\neq j$), 
\begin{align}
\mathcal{E}_{\gamma_{\mathcal{Z}}}[n;i] & \le\mathcal{E}_{i}[n+1]+\rho_{\varepsilon}^{\frac{1}{18}}\frac{\varepsilon^{(i)}}{\sqrt{-\Lambda}},\label{eq:UpperBoundEnergiesFromESpecial}\\
\mathcal{E}_{\mathcal{I}}[n;i] & \le\mathcal{E}_{j}[n+1]+\rho_{\varepsilon}^{\frac{1}{18}}\frac{\varepsilon^{(j)}}{\sqrt{-\Lambda}}\nonumber 
\end{align}
(if $i=j)$ and 
\begin{align}
\mathfrak{D}r_{\nwarrow}^{(\pm)}[n;i,j] & \ge R_{i}[n+1]\cdot\big(1-\rho_{\varepsilon}^{\frac{1}{18}}\big),\text{ if }i>0,\label{eq:LowerBoundDRFromR}\\
\mathfrak{D}r_{\nearrow}^{(\pm)}[n;i,j] & \ge R_{j}[n+1]\cdot\big(1-\rho_{\varepsilon}^{\frac{1}{18}}\big),\text{ if }j>0.\nonumber 
\end{align}
 Similarly, for any $0\le i,j\le N_{\varepsilon}$ such that $\widetilde{\mathcal{R}}_{i;j}^{(n)}\subset\mathcal{T}_{\varepsilon}^{+}$
if $i\neq j$, or $\widetilde{\mathcal{R}}_{i;\gamma_{\mathcal{Z}}}^{(n)},\widetilde{\mathcal{R}}_{i;\mathcal{I}}^{(n)}\subset\mathcal{T}_{\varepsilon}^{+}$,
if $i=j$, the bounds (\ref{eq:UpperBoundEnergiesFromE})\textendash (\ref{eq:LowerBoundDRFromR})
also hold with $\widetilde{\mathcal{E}}_{\nwarrow}^{(\pm)}$, $\widetilde{\mathcal{E}}_{\nearrow}^{(\pm)}$,
$\widetilde{\mathcal{E}}_{\gamma_{\mathcal{Z}}}$, $\widetilde{\mathcal{E}}_{\mathcal{I}}$,
$\widetilde{\mathfrak{D}}r_{\nwarrow}^{(\pm)}$ and $\widetilde{\mathfrak{D}}r_{\nearrow}^{(\pm)}$
in place of $\mathcal{E}_{\nwarrow}^{(\pm)}$, $\mathcal{E}_{\nearrow}^{(\pm)}$,
$\mathcal{E}_{\gamma_{\mathcal{Z}}}$, $\mathcal{E}_{\mathcal{I}}$,
$\mathfrak{D}r_{\nwarrow}^{(\pm)}$ and $\mathfrak{D}r_{\nearrow}^{(\pm)}$,
respectively.
\end{lem}
\begin{proof}
Let $n\in\mathbb{N}$ be such that (\ref{eq:NoonTrivialN-1}) is satisfied,
and let $0\le i,j\le N_{\varepsilon}$, $i>j$, be such that 
\begin{equation}
\mathcal{R}_{i;j}^{(n)}\subset\mathcal{U}_{\varepsilon}^{+}.\label{eq:RnijInUe}
\end{equation}
 Notice that (\ref{eq:RnijInUe}) implies that 
\begin{equation}
\mathcal{R}_{\bar{i};\bar{j}}^{(n)}\subset\mathcal{U}_{\varepsilon}^{+}\text{ for all }0\le\bar{i}\le i,\text{ }0\le\bar{j}\le j,\textnormal{ }\bar{i}\neq\bar{j}
\end{equation}
 and 
\begin{equation}
\mathcal{R}_{\bar{i};\gamma_{\mathcal{Z}}}^{(n)},\mathcal{R}_{\bar{i};\mathcal{I}}^{(n)}\subset\mathcal{U}_{\varepsilon}^{+}\text{ for all }0\le\bar{i}\le j.
\end{equation}

Using (\ref{eq:ChangeOfEnergyFinalAgainGeneralOutgoing}) and (\ref{eq:ChangeOfRSeparationFinalAgainGeneralOutgoing})
for $j_{0}=0$ and $j_{1}=\bar{j}$, we obtain for any $0\le\bar{j}\le j$:
\begin{equation}
\mathcal{E}_{\nearrow}^{(-)}[n;i,\bar{j}]=\mathcal{E}_{\nwarrow}^{(-)}[n;\bar{j},0]\cdot\exp\Big(\sum_{k=0}^{\bar{j}-1}\frac{2\mathcal{E}_{\nearrow}^{(-)}[n;i,k]}{\mathfrak{D}r_{\nearrow}^{(-)}[n;i,k+1]}+O(\rho_{\varepsilon}^{\frac{1}{12}})\Big)+O\Big(\rho_{\varepsilon}^{\frac{1}{4}}\frac{\varepsilon^{(\bar{j})}}{\sqrt{-\Lambda}}\Big)\label{eq:ChangeOfEnergyFinalAgainForEstimatesOutgoing}
\end{equation}
and, provided $\bar{j}\ge1$:
\begin{equation}
\mathfrak{D}r_{\nearrow}^{(-)}[n;i,\bar{j}]=\mathfrak{D}r_{\nwarrow}^{(-)}[n;\bar{j},0]\cdot\exp\Big(-\sum_{k=0}^{\bar{j}-2}\frac{2\mathcal{E}_{\nearrow}^{(-)}[n;i,k]}{\mathfrak{D}r_{\nearrow}^{(-)}[n;i,k+1]}+O(\rho_{\varepsilon}^{\frac{1}{12}})\Big).\label{eq:ChangeOfRSeparationFinalAgainForEstimatesOutgoing}
\end{equation}
Arguing exactly as in the proof of (\ref{eq:EqualEnergyFluxOutgoingTau})\textendash (\ref{eq:EqualRSeparationOutgoingTau}),
by comparing the system (\ref{eq:ChangeOfEnergyFinalAgainForEstimatesOutgoing})\textendash (\ref{eq:ChangeOfRSeparationFinalAgainForEstimatesOutgoing})
for $\mathcal{E}_{\nearrow}^{(-)}[n;i,\bar{j}]$, $\mathfrak{D}r_{\nearrow}^{(-)}[n;i,\bar{j}]$
with the system (\ref{eq:FormulaRecursiveEnergy})\textendash (\ref{eq:FormulaRecursiveDr})
for $\mathcal{E}_{\bar{j}}[n+1]$, $R_{\bar{j}}[n+1]$, using also
the approximate equalities (\ref{eq:EqualEnergyFlux})\textendash (\ref{eq:EqualRSeparation})
for $\mathcal{E}_{\nwarrow}^{(-)}[n;\bar{j},0]$, $\mathfrak{D}r_{\nwarrow}^{(-)}[n;\bar{j},0]$
and $\mathcal{E}_{\bar{j}}[n]$, $R_{\bar{j}}[n]$, respectively,
we infer that 
\begin{align}
\mathcal{E}_{\nearrow}^{(-)}[n;i,\bar{j}] & =\mathcal{E}_{\bar{j}}[n+1]+O\Big(\rho_{\varepsilon}^{\frac{1}{16}}\frac{\varepsilon^{(\bar{j})}}{\sqrt{-\Lambda}}\Big)\text{ for all }0\le\bar{j}\le j,\label{eq:EqualEnergyFluxGeneralCase}\\
\mathfrak{D}r_{\nearrow}^{(-)}[n;i,\bar{j}] & =R_{\bar{j}}[n+1]\cdot\big(1+O(\rho_{\varepsilon}^{\frac{1}{16}})\big)\text{ for all }1\le\bar{j}\le j.\label{eq:EqualRSeparationGeneralCase}
\end{align}

Using (\ref{eq:ChangeOfEnergyFinalAgainGeneralIngoing}) and (\ref{eq:ChangeOfRSeparationFinalAgainGeneralIngoing})
for $j_{0}=0$ and $j_{1}=\bar{j}$, we obtain for any $0\le\bar{j}\le j$:
\begin{equation}
\mathcal{E}_{\nwarrow}^{(-)}[n;i,\bar{j}]=\mathcal{E}_{\nwarrow}^{(-)}[n;i,0]\cdot\exp\Big(\sum_{k=0}^{\bar{j}-1}\frac{2\mathcal{E}_{\nearrow}^{(-)}[n;i,k]}{\mathfrak{D}r_{\nearrow}^{(-)}[n;i,k+1]}+O(\rho_{\varepsilon}^{\frac{1}{12}})\Big)+O\Big(\rho_{\varepsilon}^{\frac{1}{4}}\frac{\varepsilon^{(i)}}{\sqrt{-\Lambda}}\Big)\label{eq:ChangeOfEnergyFinalAgainForEstimatesIngoing}
\end{equation}
and, provided $\bar{j}\ge1$:
\begin{equation}
\mathfrak{D}r_{\nwarrow}^{(-)}[n;i,\bar{j}]=\mathfrak{D}r_{\nwarrow}^{(-)}[n;i,0]\cdot\exp\Big(-\sum_{k=0}^{\bar{j}-2}\frac{2\mathcal{E}_{\nearrow}^{(-)}[n;i,j]}{\mathfrak{D}r_{\nearrow}^{(-)}[n;i,j+1]}+O(\rho_{\varepsilon}^{\frac{1}{12}})\Big).\label{eq:ChangeOfRSeparationFinalAgainForEstimatesIngoing}
\end{equation}
Using the approximate equalities (\ref{eq:EqualEnergyFlux})\textendash (\ref{eq:EqualRSeparation})
for $\mathcal{E}_{\nwarrow}^{(-)}[n;i,0]$, $\mathfrak{D}r_{\nwarrow}^{(-)}[n;i,0]$
and $\mathcal{E}_{i}[n]$, $R_{i}[n]$, respectively, as well as the
approximate equalities (\ref{eq:EqualEnergyFluxGeneralCase})\textendash (\ref{eq:EqualRSeparationGeneralCase})
for $\mathcal{E}_{\nearrow}^{(-)}[n;i,\bar{j}]$, $\mathfrak{D}r_{\nearrow}^{(-)}[n;i,\bar{j}]$
and $\mathcal{E}_{\bar{j}}[n+1]$, $R_{\bar{j}}[n+1]$, respectively,
and the bounds (\ref{eq:TrivialBoundEnergiesConcentration}) and (\ref{eq:TrivialBoundsRSeparation}),
we obtain from (\ref{eq:ChangeOfEnergyFinalAgainForEstimatesIngoing})\textendash (\ref{eq:ChangeOfRSeparationFinalAgainForEstimatesIngoing})
that, for any $0\le\bar{j}\le j$: 
\begin{align}
\mathcal{E}_{\nwarrow}^{(-)}[n;i,\bar{j}] & =\mathcal{E}_{i}[n]\cdot\exp\Big(\sum_{k=0}^{\bar{j}-1}\frac{2\mathcal{E}_{k}[n+1]}{R_{k+1}[n+1]}+O(\rho_{\varepsilon}^{\frac{1}{17}})\Big)+O\Big(\rho_{\varepsilon}^{\frac{1}{4}}\frac{\varepsilon^{(i)}}{\sqrt{-\Lambda}}\Big)\le\label{eq:ChangeOfEnergyFinalAgainForEstimatesIngoing-1}\\
 & \le\mathcal{E}_{i}[n]\cdot\exp\Big(\sum_{k=0}^{i-1}\frac{2\mathcal{E}_{k}[n+1]}{R_{k+1}[n+1]}+O(\rho_{\varepsilon}^{\frac{1}{17}})\Big)+O\Big(\rho_{\varepsilon}^{\frac{1}{4}}\frac{\varepsilon^{(i)}}{\sqrt{-\Lambda}}\Big)\nonumber 
\end{align}
and, provided $i\ge1$:
\begin{align}
\mathfrak{D}r_{\nwarrow}^{(-)}[n;i,\bar{j}] & =R_{i}[n]\cdot\exp\Big(-\sum_{k=0}^{\bar{j}-2}\frac{2\mathcal{E}_{k}[n+1]}{R_{k+1}[n+1]}+O(\rho_{\varepsilon}^{\frac{1}{17}})\Big)\ge\label{eq:ChangeOfRSeparationFinalAgainForEstimatesIngoing-1}\\
 & \ge R_{i}[n]\cdot\exp\Big(-\sum_{k=0}^{i-2}\frac{2\mathcal{E}_{k}[n+1]}{R_{k+1}[n+1]}+O(\rho_{\varepsilon}^{\frac{1}{17}})\Big)\nonumber 
\end{align}
(where we have used the fact that $i\ge\bar{j}$). On the other hand,
from (\ref{eq:FormulaRecursiveEnergy})\textendash (\ref{eq:FormulaRecursiveDr})
we obtain: 
\begin{equation}
\mathcal{E}_{i}[n+1]=\mathcal{E}_{i}[n]\cdot\exp\Big(\sum_{k=0}^{i-1}\frac{2\mathcal{E}_{k}[n+1]}{R_{k+1}[n+1]}\Big)\label{eq:RecursiveEnergiesBarJ}
\end{equation}
and, for $i\ge1$:
\begin{equation}
R_{i}[n+1]=R_{i}[n]\cdot\exp\Big(-\sum_{k=0}^{i-2}\frac{2\mathcal{E}_{k}[n+1]}{R_{k+1}[n+1]}\Big).\label{eq:RecursiveRSeparationBarJ}
\end{equation}
Comparing (\ref{eq:ChangeOfEnergyFinalAgainForEstimatesIngoing-1})\textendash (\ref{eq:ChangeOfRSeparationFinalAgainForEstimatesIngoing-1})
and (\ref{eq:RecursiveEnergiesBarJ})\textendash (\ref{eq:RecursiveRSeparationBarJ}),
using also the bounds (\ref{eq:TrivialBoundEnergiesConcentration})
and (\ref{eq:TrivialBoundsRSeparation}), we obtain
\begin{align}
\mathcal{E}_{\nwarrow}^{(-)}[n;i,\bar{j}] & \le\mathcal{E}_{i}[n+1]+\rho_{\varepsilon}^{\frac{1}{18}}\text{ for all }0\le\bar{j}\le j,\label{eq:UpperBoundEnergyFluxGeneralCase}\\
\mathfrak{D}r_{\nwarrow}^{(-)}[n;i,\bar{j}] & \ge R_{i}[n+1]\Big(1-\rho_{\varepsilon}^{\frac{1}{18}}\Big)\text{ for all }0\le\bar{j}\le j,\text{ if }i\ge1.\label{eq:UpperBoundRSeparationFluxGeneralCase}
\end{align}

From (\ref{eq:EqualEnergyFluxGeneralCase})\textendash (\ref{eq:EqualRSeparationGeneralCase})
and (\ref{eq:UpperBoundEnergyFluxGeneralCase})\textendash (\ref{eq:UpperBoundRSeparationFluxGeneralCase}),
we infer (\ref{eq:UpperBoundEnergiesFromE}) and (\ref{eq:LowerBoundDRFromR})
in the case $i>j$. The proof of (\ref{eq:UpperBoundEnergiesFromE})
and (\ref{eq:LowerBoundDRFromR}), when $i<j$, or (\ref{eq:UpperBoundEnergiesFromESpecial})
and (\ref{eq:LowerBoundDRFromR}), when $i=j$, follows in exactly
the same way (using (\ref{eq:ChangeOfEnergyFinalAgainGeneralIngoing-1})\textendash (\ref{eq:ChangeOfRSeparationFinalAgainGeneralOutgoing-1})
in place of (\ref{eq:ChangeOfEnergyFinalAgainGeneralIngoing})\textendash (\ref{eq:ChangeOfRSeparationFinalAgainGeneralOutgoing})),
and hence the details will be omitted.
\end{proof}
The following result will be useful in obtaining a priori control
on the concentration of the energy of $f_{\varepsilon}$ on $\mathcal{U}_{\varepsilon}^{+}$
and $\mathcal{T}_{\varepsilon}^{+}$:
\begin{lem}
\label{lem:ControlForExtension} For any $n\in\mathbb{N}$ such that
\begin{equation}
\{0\le u\le v_{\varepsilon,0}^{(n)}-h_{\varepsilon,0}\}\cap\big\{ u<v<u+\sqrt{-\frac{3}{\Lambda}}\pi\big\}\subset\mathcal{U}_{\varepsilon}^{+},\label{eq:NoonTrivialN-1-1}
\end{equation}
 we can estimate on 
\begin{equation}
\mathcal{U}_{\varepsilon;n}^{+}\doteq\mathcal{U}_{\varepsilon}^{+}\cap\{v_{\varepsilon,0}^{(n)}-h_{\varepsilon,0}\le u\le v_{\varepsilon,0}^{(n+1)}-h_{\varepsilon,0}\}\label{eq:U+n}
\end{equation}
that: 
\begin{align}
\text{ }\sup_{V\ge0}\int_{\{v=V\}\cap\mathcal{U}_{\varepsilon;n}^{+}}r\Big(\frac{T_{uu}[f_{\varepsilon}]}{-\partial_{u}r}+\frac{T_{uv}[f_{\varepsilon}]}{\partial_{v}r}\Big) & (u,V)\,du+\label{eq:UpperBoundNormBootstrapDomainFromSequence}\\
\sup_{U\ge0}\int_{\{u=U\}\cap\mathcal{U}_{\varepsilon;n}^{+}}r\Big(\frac{T_{vv}[f_{\varepsilon}]}{\partial_{v}r}+ & \frac{T_{uv}[f_{\varepsilon}]}{-\partial_{u}r}\Big)(U,v)\,dv\le\nonumber \\
\le & 8\sum_{i=0}^{N_{\varepsilon}-1}\mu_{i}[n+1]+\max_{0\le i\le N_{\varepsilon}}\big\{\big(\exp(e^{\sigma_{\varepsilon}^{-7}}\big)a_{\varepsilon i}\big\}+\rho_{\varepsilon}^{\frac{1}{19}}\nonumber 
\end{align}
(where $\mu_{i}[n]$ were introduced in Definition \ref{def:RecursiveSystemMuER}).
Furthermore, 
\begin{equation}
\sup_{\mathcal{U}_{\varepsilon;n}^{+}}\frac{2\tilde{m}}{r}\le\max_{0\le i\le N_{\varepsilon}}\big\{\big(\exp(e^{\sigma_{\varepsilon}^{-8}}\big)a_{\varepsilon i}\big\}+\varepsilon^{\frac{1}{2}}\label{eq:UpperBoundBootstrapDomainMuTilde/R}
\end{equation}
and, for any $0\le j\le N_{\varepsilon}$: 
\begin{equation}
\sup_{\mathcal{U}_{\varepsilon;n}^{+}\cap\{u\le v_{\varepsilon,j}^{(n)}+h_{\varepsilon,j}\}}\frac{2\tilde{m}}{r}\le\max_{0\le i\le j}\big\{\big(\exp(e^{\sigma_{\varepsilon}^{-8}}\big)a_{\varepsilon i}\big\}+\varepsilon^{\frac{1}{2}}.\label{eq:UpperBoundBootstrapDomainRestrictedMu}
\end{equation}
Similarly, the estimates (\ref{eq:UpperBoundNormBootstrapDomainFromSequence}),
(\ref{eq:UpperBoundBootstrapDomainMuTilde/R}) and (\ref{eq:UpperBoundBootstrapDomainRestrictedMu})
also hold on 
\begin{equation}
\mathcal{T}_{\varepsilon;n}^{+}\doteq\mathcal{T}_{\varepsilon}^{+}\cap\{v_{\varepsilon,0}^{(n)}-\tilde{h}_{\varepsilon,0}\le u\le v_{\varepsilon,0}^{(n+1)}-\tilde{h}_{\varepsilon,0}\},\label{eq:T_n_e}
\end{equation}
with $\delta_{\varepsilon}$ and $\tilde{h}_{\varepsilon,j}$ in place
of $\sigma_{\varepsilon}$ and $h_{\varepsilon,j}$, respectively. 
\end{lem}
\begin{rem*}
Notice that $\mathcal{U}_{\varepsilon;n}^{+}$ can be alternatively
expressed as 
\[
\mathcal{U}_{\varepsilon;n}^{+}=\big\{ v_{\varepsilon,0}^{(n)}-h_{\varepsilon,0}\le u\le u_{\varepsilon,n}^{+}\big\}\cap\big\{ u<v<u+\sqrt{-\frac{3}{\Lambda}}\pi\big\},
\]
where 
\begin{equation}
u_{\varepsilon,n}^{+}\doteq\min\{v_{\varepsilon,0}^{(n+1)}-h_{\varepsilon,0},u[\mathcal{U}_{\varepsilon}^{+}]\}\label{u+epsilonn}
\end{equation}
(see the relation (\ref{eq:DefinitionBootstrapDomainWithU}) for $\mathcal{U}_{\varepsilon}^{+}$).
\end{rem*}
\begin{proof}
In order to show (\ref{eq:UpperBoundNormBootstrapDomainFromSequence}),
we will first show that, for any $V\ge0$, 
\begin{equation}
\text{ }\int_{\{v=V\}\cap\mathcal{U}_{\varepsilon;n}^{+}}r\Big(\frac{T_{uu}[f_{\varepsilon}]}{-\partial_{u}r}+\frac{T_{uv}[f_{\varepsilon}]}{\partial_{v}r}\Big)(u,V)\,du\le\frac{1}{2}\max_{0\le i\le N_{\varepsilon}}\big\{\big(\exp(e^{\sigma_{\varepsilon}^{-7}}\big)a_{\varepsilon i}\big\}+4\sum_{i=0}^{N_{\varepsilon}-1}\mu_{i}[n+1]+\frac{1}{2}\rho_{\varepsilon}^{\frac{1}{19}}\label{eq:UpperBoundInUTvvTuv}
\end{equation}
Note that, in view of the definition (\ref{eq:U+n}) of $\mathcal{U}_{\varepsilon;n}^{+}$,
the inequality (\ref{eq:UpperBoundInUTvvTuv}) is non trivial only
when 
\[
v_{\varepsilon,0}^{(n)}-h_{\varepsilon,0}<V<v_{\varepsilon,0}^{(n+1)}-h_{\varepsilon,0}+\sqrt{-\frac{3}{\Lambda}}\pi.
\]
In view of the relation (\ref{eq:TildeUMaza}) for $\partial_{u}\tilde{m}$,
the linear relation (\ref{eq:LinearCombinationVlasovFields}) between
$f_{\varepsilon}$ and $f_{\varepsilon i}$, as well as the bound
(\ref{eq:BoundSupportFepsiloni}) on the support of $f_{\varepsilon i}$,
we have 
\begin{equation}
\int_{\{v=V\}\cap\mathcal{U}_{\varepsilon;n}^{+}}r\Big(\frac{T_{uu}[f_{\varepsilon}]}{-\partial_{u}r}+\frac{T_{uv}[f_{\varepsilon}]}{\partial_{v}r}\Big)(u,V)\,du=\frac{1}{4\pi}\sum_{i=0}^{N_{\varepsilon}}\int_{\{v=V\}\cap(\mathcal{V}_{i}^{(n)}\cup\mathcal{V}_{i}^{(n+1)})\cap\mathcal{U}_{\varepsilon;n}^{+}}\big(1-\frac{2\tilde{m}}{r}-\frac{1}{3}\Lambda r^{2}\big)^{-1}\frac{-\partial_{u}\tilde{m}}{r}(u,V)\,du.\label{eq:LinearCombinationTvvTuv}
\end{equation}

We will proceed to establish (\ref{eq:UpperBoundInUTvvTuv}) by considering
the cases when $V^{(-)}\le V\le V^{(+)}$, $V<V^{(-)}$ and $V>V^{(+)}$
separately, where we have set 
\begin{align}
V^{(-)} & \doteq v_{\varepsilon,N_{\varepsilon}}^{(n)}+h_{\varepsilon,N_{\varepsilon}},\label{eq:V-}\\
V^{(+)} & \doteq v_{\varepsilon,N_{\varepsilon}}^{(n)}+\rho_{\varepsilon}^{-2}\frac{\varepsilon}{\sqrt{-\Lambda}}\label{eq:V+}
\end{align}
(see Figure \ref{fig:Concentration_Estimates}).

\begin{figure}[h] 
\centering 
\scriptsize
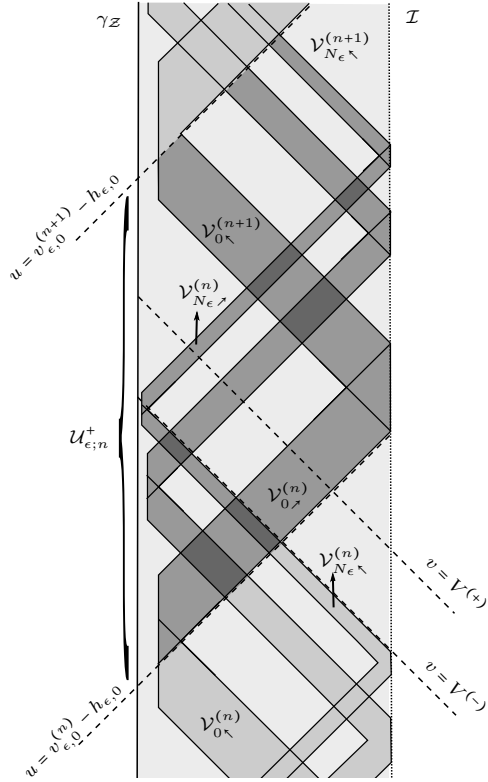 
\caption{Schematic depiction of the domain $\mathcal{U}_{\epsilon;n}^{+}$ in the case when $v_{\epsilon,0}^{(n+1)}-h_{\epsilon,0}<u[\mathcal{U}_{\epsilon}^{+}]$. For simplicity, we have only depicted three of the beam domains $\mathcal{V}_i$. \label{fig:Concentration_Estimates}}
\end{figure}

\medskip{}

\noindent \emph{Case I: $V\in[V^{(-)},V^{(+)}]$. }In this case, the
bound (\ref{eq:EstimateGeometryLikeC0}) and the definition (\ref{eq:V+})
of $V^{(+)}$ imply that 
\begin{equation}
r|_{\{v=V\}\cap(\cup_{i=0}^{N_{\varepsilon}}(\mathcal{V}_{i}^{(n)}\cup\mathcal{V}_{i}^{(n+1)}))\cap\mathcal{U}_{\varepsilon}^{+}}\le e^{\sigma_{\varepsilon}^{-4}}\rho_{\varepsilon}^{-2}\frac{\varepsilon}{\sqrt{-\Lambda}}.\label{eq:UpperBoundRFirstCase}
\end{equation}
Furthermore, for any $0\le i\le N_{\varepsilon}$, we have 
\begin{equation}
\{v=V\}\cap(\mathcal{V}_{i}^{(n)}\cup\mathcal{V}_{i}^{(n+1)})\cap\mathcal{U}_{\varepsilon;n}^{+}=\{v=V\}\cap\mathcal{V}_{i\nearrow}^{(n)}\cap\mathcal{U}_{\varepsilon;n}^{+}.\label{eq:IntersectionOnlywithOutgoingRegions}
\end{equation}
Let us define $i_{+}$ as the maximum number in $\{0,1,\ldots,N_{\varepsilon}\}$
such that 
\begin{equation}
\mathcal{V}_{i\nearrow}^{(n)}\subset\mathcal{U}_{\varepsilon;n}^{+}\text{ for all }i<i_{+}.\label{eq:i+definition}
\end{equation}
Note that, if $i_{+}<N_{\varepsilon}$, then it is necessary that
$\mathcal{V}_{j\nearrow}^{(n)}\cap\mathcal{U}_{\varepsilon}^{+}=\emptyset$
for all $\iota_{+}<j\le N_{\varepsilon}$. Moreover, for any $i<i_{+}$:
\begin{equation}
\inf_{\{v=V\}\cap\mathcal{V}_{i\nearrow}^{(n)}\cap\mathcal{U}_{\varepsilon;n}^{+}}(v-u)\ge v_{n;N_{\varepsilon},i_{+}+1}^{(+)}-u_{n;N_{\varepsilon},i_{+}+1}^{(+)}\ge\rho_{\varepsilon}^{-1}\frac{\varepsilon^{(i_{+}+1)}}{\sqrt{-\Lambda}}.\label{eq:LowerBoundv-u}
\end{equation}
Note that the above definition of $i_{+}$ implies that, in the extreme
case when $\mathcal{V}_{j\nearrow}^{(n)}\cap\mathcal{U}_{\varepsilon}^{+}=\emptyset$
for all $j$, we have $i_{+}=0$.

\medskip{}

\noindent \emph{Remark.} In the case when $v_{\varepsilon,0}^{(n+1)}-h_{\varepsilon,0}=u[\mathcal{U}_{\varepsilon}^{+}]$
(and hence $u_{\varepsilon,n}^{+}=v_{\varepsilon,0}^{(n+1)}-h_{\varepsilon,0}$
in (\ref{u+epsilonn})), such as the case depicted in Figure \ref{fig:Concentration_Estimates},
the parameter $i_{+}$ is equal to $N_{\varepsilon}$; similarly,
in this case, for the parameters $i_{+}^{(1)},i_{+}^{(2)}$ defined
by (\ref{eq:DefinitionI1+})\textendash (\ref{eq:DefinitionI2+}),
we have $i_{+}^{(1)}=i_{+}^{(2)}=N_{\varepsilon}$. The parameters
$i_{+}$, $i_{+}^{(1)}$ and $i_{+}^{(2)}$ are only introduced to
treat the case when $u_{\varepsilon,n}^{+}<v_{\varepsilon,0}^{(n+1)}-h_{\varepsilon,0}$
and $u=u_{\varepsilon,n}^{+}$ intersects one of the outgoing beam
components $\mathcal{V}_{i\nearrow}^{(n)}$, $0\le i\le N_{\varepsilon}$. 

\medskip{}

For any $i<i_{+}$, let us denote for simplicity 
\begin{align}
r_{min}^{(i)}(V) & =\inf_{\{v=V\}\cap\mathcal{V}_{i\nearrow}^{(n)}\cap\mathcal{U}_{\varepsilon;n}^{+}}r=r\Big(v_{\varepsilon,i}^{(n)}+h_{\varepsilon,i},V\Big),\label{eq:RiMinMax}\\
r_{max}^{(i)}(V) & =\sup_{\{v=V\}\cap\mathcal{V}_{i\nearrow}^{(n)}\cap\mathcal{U}_{\varepsilon;n}^{+}}r=r\Big(v_{\varepsilon,i}^{(n)}-h_{\varepsilon,i},V\Big).\nonumber 
\end{align}
 In view of the definition (\ref{eq:BeamDomain}) of $\mathcal{V}_{i\nearrow}^{(n)}$,
the bound (\ref{eq:BoundsRIntersectionRegioni>j}) for $r_{n;N_{\varepsilon},i}$,
the bound (\ref{eq:EstimateGeometryLikeC0}) on $\partial_{u}r$,
the bounds (\ref{eq:UpperBoundMuBootstrapDomain}), (\ref{eq:UpperBoundRFirstCase})
on $\frac{2\tilde{m}}{r}$, $r$, the fact that $\partial_{v}r>0$
on $\mathcal{U}_{\varepsilon}^{+}$ and the lower bound (\ref{eq:LowerBoundv-u}),
we have for any $i<i_{+}$:
\begin{equation}
\frac{r_{max}^{(i)}(V)-r_{min}^{(i)}(V)}{r_{min}^{(i)}(V)}\le\frac{\int_{([v_{\varepsilon,i}^{(n)}-h_{\varepsilon,i},v_{\varepsilon,i}^{(n)}+h_{\varepsilon,i}]\times\{V\})\cap\mathcal{U}_{\varepsilon;n}^{+}}\big(-\partial_{u}r\big)\,du}{r_{min}^{(i)}(V^{(-)})}\le\frac{\exp(2\sigma_{\varepsilon}^{-6})\varepsilon^{(i)}}{r_{n;N_{\varepsilon},i}}\le\rho_{\varepsilon}^{\frac{3}{4}}.\label{eq:SmallChangeInR}
\end{equation}

Using (\ref{eq:UpperBoundMuBootstrapDomain}), (\ref{eq:UpperBoundRFirstCase}),
(\ref{eq:IntersectionOnlywithOutgoingRegions}) and (\ref{eq:SmallChangeInR}),
together with the fact that $\partial_{u}\tilde{m}\le0$ on $\mathcal{U}_{\varepsilon}^{+}$,
we infer that, for any $i<i_{+}$:
\begin{align}
\int_{\{v=V\}\cap(\mathcal{V}_{i}^{(n)}\cup\mathcal{V}_{i}^{(n+1)})\cap\mathcal{U}_{\varepsilon;n}^{+}}\big(1-\frac{2\tilde{m}}{r}- & \frac{1}{3}\Lambda r^{2}\big)^{-1}\frac{-\partial_{u}\tilde{m}}{r}(u,V)\,du=\label{eq:FirstUpperBoundDM}\\
 & =\int_{\{v=V\}\cap\mathcal{V}_{i\nearrow}^{(n)}\cap\mathcal{U}_{\varepsilon;n}^{+}}(1+O(\eta_{0})+O(\varepsilon))\frac{-\partial_{u}\tilde{m}}{r_{min}^{(i)}(V)\big(1+O(\rho_{\varepsilon}^{\frac{3}{4}})\big)}(u,V)\,du\le\nonumber \\
 & \le\frac{2}{r_{min}^{(i)}(V)}\int_{\{v=V\}\cap\mathcal{V}_{i\nearrow}^{(n)}\cap\mathcal{U}_{\varepsilon;n}^{+}}(-\partial_{u}\tilde{m})(u,V)\,du.\nonumber 
\end{align}
Using the definition (\ref{eq:OutgoingEnergyAfter}) of $\mathcal{E}_{\nearrow}^{(+)}[n;N_{\varepsilon},i]$,
with the convention that 
\[
\mathcal{E}_{\nearrow}^{(+)}[n;N_{\varepsilon},N_{\varepsilon}]=\mathcal{E}_{\gamma_{\mathcal{Z}}}[n;N_{\varepsilon}],
\]
as well as the fact that $\tilde{m}$ is constant on each connected
component of $\mathcal{U}_{\varepsilon}^{+}\backslash\cup_{j=0}^{N_{\varepsilon}}\mathcal{V}_{j}$
and that $\partial_{v}r>0$, from (\ref{eq:FirstUpperBoundDM}) we
obtain that, for any $i<i_{+}$: 
\begin{align}
\int_{\{v=V\}\cap(\mathcal{V}_{i}^{(n)}\cup\mathcal{V}_{i}^{(n+1)})\cap\mathcal{U}_{\varepsilon;n}^{+}}\big(1-\frac{2\tilde{m}}{r}- & \frac{1}{3}\Lambda r^{2}\big)^{-1}\frac{-\partial_{u}\tilde{m}}{r}(u,V)\,du\le\label{eq:SecondUpperBoundDM}\\
 & \le\frac{2}{r_{min}^{(i)}(V^{(-)})}\mathcal{E}_{\nearrow}^{(+)}[n;N_{\varepsilon},i]\le\nonumber \\
 & \le\frac{2}{r_{n;N_{\varepsilon},i}}\mathcal{E}_{\nearrow}^{(+)}[n;N_{\varepsilon},i].\nonumber 
\end{align}
Using the relations (\ref{eq:ExpressionRnii-1inU}) and (\ref{eq:ComparableRnij})
(implying that $r_{n;N_{\varepsilon},i}=\mathfrak{D}r_{\nearrow}^{(-)}[n;i+1,i+1]\big(1+O(\rho_{\varepsilon}^{\frac{1}{10}})\big)$),
the bound (\ref{eq:TrivialBoundEnergiesConcentration}) for $\mathcal{E}_{\nearrow}^{(\pm)}$
and the estimates (\ref{eq:UpperBoundEnergiesFromE})\textendash (\ref{eq:LowerBoundDRFromR})
for $\mathcal{E}_{\nearrow}^{(\pm)}$, $\mathfrak{D}r_{\nearrow}^{(-)}$,
we infer from (\ref{eq:SecondUpperBoundDM}) that, for any $i<i_{+}$:
\begin{align}
\int_{\{v=V\}\cap(\mathcal{V}_{i}^{(n)}\cup\mathcal{V}_{i}^{(n+1)})\cap\mathcal{U}_{\varepsilon;n}^{+}}\big(1-\frac{2\tilde{m}}{r}-\frac{1}{3}\Lambda r^{2}\big)^{-1}\frac{-\partial_{u}\tilde{m}}{r}(u,V)\,du & \le4\frac{\mathcal{E}_{i}[n+1]}{R_{i}[n+1]}+O(\rho_{\varepsilon}^{1+\frac{1}{18}})\le\label{eq:FinalUpperBoundDM}\\
 & \le2\mu_{i}[n+1]+O(\rho_{\varepsilon}^{1+\frac{1}{18}}).\nonumber 
\end{align}

On the other hand, for $i=i_{+}$, using directly the bounds (\ref{eq:EstimateGeometryLikeC0}),
(\ref{eq:LowerBoundROnDomains}), (\ref{eq:UpperBoundTuuTvvFi}),
(\ref{eq:UpperBoundTuvFi}) and the fact that 
\begin{equation}
f_{\varepsilon}|_{\mathcal{V}_{j\nearrow}^{(n)}\cap\{V^{(-)}\le v\le V^{(+)}\}}=a_{\varepsilon j}f_{\varepsilon j}|_{\mathcal{V}_{j\nearrow}^{(n)}\cap\{V^{(-)}\le v\le V^{(+)}\}}
\end{equation}
(as a consequence of (\ref{eq:LinearCombinationVlasovFields}), (\ref{eq:BoundSupportFepsiloni})
and the definition of $V^{(-)}$, $V^{(+)}$), we infer that 
\begin{align}
\int_{\{v=V\}\cap(\mathcal{V}_{i_{+}}^{(n)}\cup\mathcal{V}_{i_{+}}^{(n+1)})\cap\mathcal{U}_{\varepsilon;n}^{+}} & \big(1-\frac{2\tilde{m}}{r}-\frac{1}{3}\Lambda r^{2}\big)^{-1}\frac{-\partial_{u}\tilde{m}}{r}(u,V)\,du=\label{eq:EstimateExceptionalBeamForNorm}\\
 & =\int_{\{v=V\}\cap\mathcal{V}_{i_{+}\nearrow}^{(n)}\cap\mathcal{U}_{\varepsilon;n}^{+}}r\Big(\frac{T_{uu}[f_{\varepsilon}]}{-\partial_{u}r}+\frac{T_{uv}[f_{\varepsilon}]}{\partial_{v}r}\Big)(u,V)\,du=\nonumber \\
 & =a_{\varepsilon i_{+}}\int_{\{v=V\}\cap\mathcal{V}_{i_{+}\nearrow}^{(n)}\cap\mathcal{U}_{\varepsilon;n}^{+}}r\Big(\frac{T_{uu}[f_{\varepsilon i_{+}}]}{-\partial_{u}r}+\frac{T_{uv}[f_{\varepsilon i_{+}}]}{\partial_{v}r}\Big)(u,V)\,du\le\nonumber \\
 & \le a_{\varepsilon i_{+}}\exp\big(\exp(3\sigma_{\varepsilon}^{-5})\big)\int_{\{v=V\}\cap\mathcal{V}_{i_{+}\nearrow}^{(n)}\cap\mathcal{U}_{\varepsilon;n}^{+}}\Big(\frac{(\varepsilon^{(i_{+})})^{4}}{r^{5}(u,v)}(-\Lambda)^{-2}+\frac{(\varepsilon^{(i_{+})})^{2}}{r^{3}(u,v)}(-\Lambda)^{-1}\Big)(u,V)\,du\le\nonumber \\
 & \le a_{\varepsilon i_{+}}\exp\big(\exp(3\sigma_{\varepsilon}^{-5})\big)\Big(\frac{(\varepsilon^{(i_{+})})^{4}}{(\min_{\mathcal{V}_{i_{+}\nearrow}^{(n)}}r)^{4}}(-\Lambda)^{-2}+\frac{(\varepsilon^{(i_{+})})^{2}}{(\min_{\mathcal{V}_{i_{+}\nearrow}^{(n)}}r)^{2}}(-\Lambda)^{-1}\Big)\le\nonumber \\
 & \le a_{\varepsilon i_{+}}\exp\big(\exp(\sigma_{\varepsilon}^{-6})\big).\nonumber 
\end{align}

From (\ref{eq:LinearCombinationTvvTuv}), (\ref{eq:IntersectionOnlywithOutgoingRegions}),
(\ref{eq:FinalUpperBoundDM}), (\ref{eq:EstimateExceptionalBeamForNorm})
and the fact that 
\[
\sum_{i=0}^{N_{\varepsilon}}O(\rho_{\varepsilon}^{1+\frac{1}{18}})=O(N_{\varepsilon}\rho_{\varepsilon}^{1+\frac{1}{18}})=O(\rho_{\varepsilon}^{\frac{1}{19}})
\]
 and 
\[
\mathcal{V}_{j\nearrow}^{(n)}\cap\mathcal{U}_{\varepsilon}^{+}=\emptyset\text{ for any }j>i_{+},
\]
we immediately infer (\ref{eq:UpperBoundInUTvvTuv}) in the case $V\in[V^{(-)},V^{(+)}]$.

\medskip{}

\noindent \emph{Case II: $V\in[v_{\varepsilon,0}^{(n)}-h_{\varepsilon,0},V^{(-)})$.
}In this case, the upper bound (\ref{eq:UpperBoundRFirstCase}) for
$r$ still holds. However, for any $0\le i\le N_{\varepsilon}$, we
now have 
\begin{equation}
\{v=V\}\cap(\mathcal{V}_{i}^{(n)}\cup\mathcal{V}_{i+1}^{(n)})\cap\mathcal{U}_{\varepsilon;n}^{+}=\{v=V\}\cap\mathcal{V}_{i}^{(n)}\cap\mathcal{U}_{\varepsilon;n}^{+}.\label{eq:IntersectionOnlyWithTwoRegionsSecondcase}
\end{equation}

Let us define $i_{+}^{(1)}$ as the maximum number in the set $\{0,1,\ldots,N_{\varepsilon}\}$
such that 
\begin{equation}
\{v=V\}\cap\mathcal{V}_{i}^{(n)}\cap\mathcal{U}_{\varepsilon;n}^{+}\neq\emptyset\text{ for all }i\le i_{+}^{(1)}\label{eq:DefinitionI1+}
\end{equation}
and let $i_{+}^{(2)}$ be the maximum number in $\{0,1,\ldots,i_{+}^{(1)}\}$
such that 
\begin{equation}
v_{\varepsilon,i}^{(n)}+h_{\varepsilon,i}<u_{\varepsilon,n}^{+}\text{ for all }i\le i_{+}^{(2)}\label{eq:DefinitionI2+}
\end{equation}
(where $\{u=u_{n,\varepsilon}^{+}\}$ is the future boundary of $\mathcal{U}_{\varepsilon;n}^{+}$).
Note that the definition of $i_{+}^{(1)}$ implies that 
\begin{equation}
\{v=V\}\cap\mathcal{V}_{i}^{(n)}\cap\mathcal{U}_{\varepsilon;n}^{+}=\emptyset\text{ for all }i>i_{+}^{(1)}\label{eq:TrivialIntersection}
\end{equation}
(which is a non-trivial statement only if $i_{+}^{(1)}<N_{\varepsilon}$)
and that 
\begin{equation}
V\ge v_{\varepsilon,i_{+}^{(1)}}^{(n)}-h_{\varepsilon,i_{+}^{(1)}}.\label{eq:LowerBoundVi+}
\end{equation}
Let us also remark that, trivially, in view of the form (\ref{eq:BeamDomain})
of $\mathcal{V}_{i}^{(n)}=\mathcal{V}_{i\nwarrow}^{(n)}\cup\mathcal{V}_{i\nearrow}^{(n)}$,
\begin{equation}
i_{+}^{(2)}=\begin{cases}
i_{+}^{(1)}-2, & \text{if }u_{\varepsilon,n}^{+}\in[v_{\varepsilon,i_{+}^{(1)}-1}^{(n)}-h_{\varepsilon,i_{+}^{(1)}-1},v_{\varepsilon,i_{+}^{(1)}-1}^{(n)}+h_{\varepsilon,i_{+}^{(1)}-1}]\text{ and }V\le v_{\varepsilon,i_{+}^{(1)}}^{(n)}+h_{\varepsilon,i_{+}^{(1)}},\\
i_{+}^{(1)}-1, & \text{ if }u_{\varepsilon,n}^{+}\in[v_{\varepsilon,i_{+}^{(1)}}^{(n)}-h_{\varepsilon,i_{+}^{(1)}},v_{\varepsilon,i_{+}^{(1)}}^{(n)}+h_{\varepsilon,i_{+}^{(1)}}],\\
i_{+}^{(1)}, & \text{ for all other values of }u_{\varepsilon,n}^{+}\le v_{\varepsilon,0}^{(n)}-h_{\varepsilon,0}.
\end{cases}\label{eq:I+2}
\end{equation}

It can be readily inferred from (\ref{eq:LowerBoundVi+}) and the
form (\ref{eq:BeamDomain}) of $\mathcal{V}_{i}^{(n)}=\mathcal{V}_{i\nwarrow}^{(n)}\cup\mathcal{V}_{i\nearrow}^{(n)}$
(and in particular, the fact that $\mathcal{V}_{i\nwarrow}^{(n)}\subset\{v\le v_{\varepsilon,i_{+}^{(1)}}^{(n)}-h_{\varepsilon,i_{+}^{(1)}}\}$
when $i<i_{+}^{(1)}$) that 
\begin{equation}
\{v=V\}\cap\mathcal{V}_{i}^{(n)}\cap\mathcal{U}_{\varepsilon;n}^{+}=\{v=V\}\cap\mathcal{V}_{i\nearrow}^{(n)}\cap\mathcal{U}_{\varepsilon;n}^{+}\text{ for all }i<i_{+}^{(1)}\label{eq:IntersectionONlyOutgoingSecondCase}
\end{equation}
and 
\begin{equation}
\inf_{\{v=V\}\cap\mathcal{V}_{i}^{(n)}\cap\mathcal{U}_{\varepsilon}^{+}}(v-u)\ge\rho_{\varepsilon}^{-1}\frac{\varepsilon^{(i)}}{\sqrt{-\Lambda}}\text{ for all }i<i_{+}^{(1)}.\label{eq:LowerBoundV-USecondCase}
\end{equation}

From (\ref{eq:UpperBoundRFirstCase}), (\ref{eq:LowerBoundV-USecondCase}),
(\ref{eq:EstimateGeometryLikeC0}) and the fact that $\partial_{v}r>0$,
we infer that, for any $i<i_{+}^{(2)}$, analogously to (\ref{eq:SmallChangeInR}):
\begin{equation}
\frac{r_{max}^{(i)}(V)-r_{min}^{(i)}(V)}{r_{min}^{(i)}(V)}\le\frac{\int_{([v_{\varepsilon,i}^{(n)}-h_{\varepsilon,i},v_{\varepsilon,i}^{(n)}+h_{\varepsilon,i}]\times\{V\})\cap\mathcal{U}_{\varepsilon;n}^{+}}\big(-\partial_{u}r\big)\,du}{r_{min}^{(i)}(v_{\varepsilon,i_{+}^{(1)}}^{(n)}-h_{\varepsilon,i_{+}^{(1)}})}\le\frac{\exp(2\sigma_{\varepsilon}^{-6})\varepsilon^{(i)}}{r_{n;i_{+}^{(1)},i}}\le\rho_{\varepsilon}^{\frac{3}{4}},\label{eq:SmallChangeInR-1}
\end{equation}
where $r_{max}^{(i)}(V),$ $r_{min}^{(i)}(V)$ are defined by (\ref{eq:RiMinMax}).
Therefore, using (\ref{eq:IntersectionOnlyWithTwoRegionsSecondcase})
and (\ref{eq:SmallChangeInR-1}) and arguing as in the proof of (\ref{eq:FirstUpperBoundDM})\textendash (\ref{eq:FinalUpperBoundDM}),
using in addition the estimate
\[
\int_{\{v=V\}\cap\mathcal{V}_{i\nearrow}^{(n)}\cap\mathcal{U}_{\varepsilon;n}^{+}}(-\partial_{u}\tilde{m})\,du\le\mathcal{E}_{\nearrow}^{(-)}[n;i_{+}^{(1)},i]\cdot\big(1+C\varepsilon\big)+C\rho_{\varepsilon}^{\frac{3}{2}}\frac{\varepsilon^{(j)}}{\sqrt{-\Lambda}}
\]
 (following from (\ref{eq:GeneralBoundMassDifferenceOutfoinfi>j}))
in the case when $V\in[v_{\varepsilon,i_{+}^{(1)}}^{(n)}-h_{\varepsilon,i_{+}^{(1)}},v_{\varepsilon,i_{+}^{(1)}}^{(n)}+h_{\varepsilon,i_{+}^{(1)}}]$,
we obtain for any $i<i_{+}^{(2)}$: 
\begin{equation}
\int_{\{v=V\}\cap(\mathcal{V}_{i}^{(n)}\cup\mathcal{V}_{i}^{(n+1)})\cap\mathcal{U}_{\varepsilon;n}^{+}}\big(1-\frac{2\tilde{m}}{r}-\frac{1}{3}\Lambda r^{2}\big)^{-1}\frac{-\partial_{u}\tilde{m}}{r}(u,V)\,du\le2\mu_{i}[n+1]+O(\rho_{\varepsilon}^{1+\frac{1}{18}})\label{eq:FinalUpperBoundDMSecondCase}
\end{equation}

On the other hand, for $i_{+}^{(2)}\le i\le i_{+}^{(1)}$, using the
relation 
\begin{align*}
\int_{\{v=V\}\cap\mathcal{V}_{i}^{(n)}\cap\mathcal{U}_{\varepsilon;n}^{+}\backslash\cup_{j<i_{+}^{(2)}}\mathcal{V}_{j}^{(n)}}\big(1-\frac{2\tilde{m}}{r}- & \frac{1}{3}\Lambda r^{2}\big)^{-1}\frac{-\partial_{u}\tilde{m}}{r}(u,V)\,du=\\
 & =a_{\varepsilon i}\int_{\{v=V\}\cap\mathcal{V}_{i\nearrow}^{(n)}\cap\mathcal{U}_{\varepsilon;n}^{+}\backslash\cup_{j<i_{+}^{(2)}}\mathcal{V}_{j}^{(n)}}r\Big(\frac{T_{uu}[f_{\varepsilon i}]}{-\partial_{u}r}+\frac{T_{uv}[f_{\varepsilon i}]}{\partial_{v}r}\Big)(u,V)\,du
\end{align*}
(following from (\ref{eq:TildeUMaza}), (\ref{eq:LinearCombinationVlasovFields})
and (\ref{eq:BoundSupportFepsiloni})), we infer by arguing exactly
as in the proof of (\ref{eq:EstimateExceptionalBeamForNorm}) that
\begin{equation}
\sum_{i=i_{+}^{(2)}}^{i_{+}^{(1)}}\int_{\{v=V\}\cap\mathcal{V}_{i}^{(n)}\cap\mathcal{U}_{\varepsilon;n}^{+}\backslash\cup_{j<i_{+}^{(2)}}\mathcal{V}_{j}^{(n)}}\big(1-\frac{2\tilde{m}}{r}-\frac{1}{3}\Lambda r^{2}\big)^{-1}\frac{-\partial_{u}\tilde{m}}{r}(u,V)\,du\le\sum_{i=i_{+}^{(2)}}^{i_{+}^{(1)}}a_{\varepsilon i_{+}}\exp\big(\exp(\sigma_{\varepsilon}^{-6})\big).\label{eq:EstimateExceptionalBeamForNormSecondCase}
\end{equation}

From (\ref{eq:LinearCombinationTvvTuv}), (\ref{eq:IntersectionOnlyWithTwoRegionsSecondcase}),
(\ref{eq:TrivialIntersection})(\ref{eq:FinalUpperBoundDMSecondCase}),
(\ref{eq:EstimateExceptionalBeamForNormSecondCase}) and the fact
that $|i_{+}^{(1)}-i_{+}^{(2)}|\le2$, we readily infer (\ref{eq:UpperBoundInUTvvTuv})
in the case $V<V^{(-)}$.

\medskip{}

\noindent \emph{Case III: $V\in(V^{(+)},v_{\varepsilon,0}^{(n+1)}-h_{\varepsilon,0}+\sqrt{-\frac{3}{\Lambda}}\pi)$.
}In this case, we will split the left hand side of (\ref{eq:LinearCombinationTvvTuv})
as 
\begin{align}
\int_{\{v=V\}\cap\mathcal{U}_{\varepsilon;n}^{+}}r\Big( & \frac{T_{uu}[f_{\varepsilon}]}{-\partial_{u}r}+\frac{T_{uv}[f_{\varepsilon}]}{\partial_{v}r}\Big)(u,V)\,du\label{eq:SplittingOfTuuTuv}\\
= & \int_{\{v=V\}\cap\{u\le U^{(+)}\}\cap\mathcal{U}_{\varepsilon;n}^{+}}r\Big(\frac{T_{uu}[f_{\varepsilon}]}{-\partial_{u}r}+\frac{T_{uv}[f_{\varepsilon}]}{\partial_{v}r}\Big)(u,V)\,du+\nonumber \\
 & +\int_{\{v=V\}\cap\{u\ge U^{(+)}\}\cap\mathcal{U}_{\varepsilon;n}^{+}}r\Big(\frac{T_{uu}[f_{\varepsilon}]}{-\partial_{u}r}+\frac{T_{uv}[f_{\varepsilon}]}{\partial_{v}r}\Big)(u,V)\,du,\nonumber 
\end{align}
where 
\begin{equation}
U^{(+)}\doteq v_{\varepsilon,0}^{(n+1)}-\rho_{\varepsilon}^{-2}\frac{\varepsilon}{\sqrt{-\Lambda}},\label{eq:DefinitionU+}
\end{equation}
and we will estimate each term in the right hand side of (\ref{eq:SplittingOfTuuTuv})
separately.

From the form (\ref{eq:BeamDomain}) of $\mathcal{V}_{i}^{(n)}$ and
the definitions (\ref{eq:V+}), (\ref{eq:DefinitionU+}) of $V^{(+)}$,
$U^{(+)}$, respectively, we infer that 
\begin{equation}
\inf_{\{v\ge V^{(+)}\}\cap\{u\le U^{(+)}\}\cap(\cup_{i=0}^{N_{\varepsilon}}\mathcal{V}_{i}^{(n)})\cap\mathcal{U}_{\varepsilon;n}^{+}}(v-u)\ge\frac{1}{2}\rho_{\varepsilon}^{-2}\frac{\varepsilon}{\sqrt{-\Lambda}}.
\end{equation}
Thus, using (\ref{eq:EstimateGeometryLikeC0}), we infer that 
\begin{equation}
\inf_{\{v\ge V^{(+)}\}\cap\{u\le U^{(+)}\}\cap(\cup_{i=0}^{N_{\varepsilon}}\mathcal{V}_{i}^{(n)})\cap\mathcal{U}_{\varepsilon;n}^{+}}r\ge e^{-\sigma_{\varepsilon}^{-4}}\rho_{\varepsilon}^{-2}\frac{\varepsilon}{\sqrt{-\Lambda}}.\label{eq:LowerBoundRForSplittingAway}
\end{equation}

Using the relation (\ref{eq:LinearCombinationVlasovFields}) and the
bounds (\ref{eq:EstimateGeometryLikeC0}), (\ref{eq:UpperBoundTuuTvvFi})\textendash (\ref{eq:UpperBoundTuvFi}),
from (\ref{eq:LowerBoundRForSplittingAway}) we infer that 
\begin{align}
\int_{\{v=V\}\cap\{u\le U^{(+)}\}\cap\mathcal{U}_{\varepsilon;n}^{+}} & r\Big(\frac{T_{uu}[f_{\varepsilon}]}{-\partial_{u}r}+\frac{T_{uv}[f_{\varepsilon}]}{\partial_{v}r}\Big)(u,V)\,du\label{eq:BoundNormAwayRegionThirdCase}\\
 & =\sum_{i=0}^{N_{\varepsilon}}a_{\varepsilon i}\int_{\{v=V\}\cap\{u\le U^{(+)}\}\cap\mathcal{U}_{\varepsilon;n}^{+}}r\Big(\frac{T_{uu}[f_{\varepsilon i}]}{-\partial_{u}r}+\frac{T_{uv}[f_{\varepsilon i}]}{\partial_{v}r}\Big)(u,V)\,du\nonumber \\
 & \le\exp(\exp(\sigma_{\varepsilon}^{-6}))\sum_{i=0}^{N_{\varepsilon}}a_{\varepsilon i}\int_{\{v=V\}\cap\{u\le U^{(+)}\}\cap\mathcal{U}_{\varepsilon;n}^{+}}\frac{1}{1-\frac{1}{3}\Lambda r^{2}}\Big(\frac{(\varepsilon^{(i)})^{4}}{r^{5}}(-\Lambda)^{-1}\nonumber \\
 & \hphantom{\le\exp(\exp(\sigma_{\varepsilon}^{-6}))\sum_{i=0}^{N_{\varepsilon}}a_{\varepsilon i}\int_{\{v=V\}\cap\{u\le U^{(+)}\}\cap\mathcal{U}_{\varepsilon;n}^{+}}}+\frac{(\varepsilon^{(i)})^{2}}{r^{3}}(-\Lambda)^{-1}\Big)(u,V)\,du\nonumber\\
 & \le\exp(\exp(\sigma_{\varepsilon}^{-6}))\sum_{i=0}^{N_{\varepsilon}}a_{\varepsilon i}\Bigg\{\frac{1}{1-\frac{1}{3}\Lambda r^{2}}\Big(\frac{(\varepsilon^{(i)})^{4}}{r^{4}}(-\Lambda)^{-1}+\frac{(\varepsilon^{(i)})^{2}}{r^{2}}(-\Lambda)^{-1}\Bigg\}_{r=e^{-\sigma_{\varepsilon}^{-4}}\rho_{\varepsilon}^{-2}\frac{\varepsilon}{\sqrt{-\Lambda}}}\nonumber \\
 & \le\rho_{\varepsilon}^{3}.\nonumber 
\end{align}

On the other hand, in the case when 
\[
\{v=V\}\cap\{u\ge U^{(+)}\}\cap\mathcal{U}_{\varepsilon;n}^{+}\neq\emptyset,
\]
from the form (\ref{eq:BeamDomain}) of $\mathcal{V}_{i}^{(n)}$ and
the definitions (\ref{eq:V+}), (\ref{eq:DefinitionU+}) of $V^{(+)}$,
$U^{(+)}$, respectively, we infer that, depending on whether $V$
belongs to $\cup_{i=0}^{N_{\varepsilon}}[v_{\varepsilon,i}^{(n+1)}-h_{\varepsilon,i},v_{\varepsilon,i}^{(n+1)}+h_{\varepsilon,i}]$
or not:

\begin{itemize}

\item Either 
\[
\{v=V\}\cap\{u\ge U^{(+)}\}\cap\big\{\cup_{i=0}^{N_{\varepsilon}}(\mathcal{V}_{i}^{(n)}\cup\mathcal{V}_{i}^{(n+1)}\big\}\cap\mathcal{U}_{\varepsilon;n}^{+}=\emptyset,
\]
in which case 
\begin{equation}
\int_{\{v=V\}\cap\{u\ge U^{(+)}\}\cap\mathcal{U}_{\varepsilon;n}^{+}}r\Big(\frac{T_{uu}[f_{\varepsilon}]}{-\partial_{u}r}+\frac{T_{uv}[f_{\varepsilon}]}{\partial_{v}r}\Big)(u,V)\,du=0,\label{eq:TrivialBoundNormThirdCase}
\end{equation}

\item or 
\begin{align*}
\{v=V\}\cap\{u\ge U^{(+)}\}\cap\big\{\cup_{i=0}^{N_{\varepsilon}}( & \mathcal{V}_{i}^{(n)}\cup\mathcal{V}_{i}^{(n+1)}\big\}\cap\mathcal{U}_{\varepsilon;n}^{+}=\\
= & \{v=V\}\cap\{u\ge U^{(+)}\}\cap\{\mathcal{V}_{i_{0}\nwarrow}^{(n+1)}\big\}\cap\mathcal{U}_{\varepsilon;n}^{+}\text{ for some }0\le i_{0}\le N_{\varepsilon},
\end{align*}
in which case, using the bounds (\ref{eq:EstimateGeometryLikeC0}),
(\ref{eq:BoundSupportFepsiloni}), (\ref{eq:LowerBoundROnDomains}),
(\ref{eq:UpperBoundTuuTvvFi}) and the fact that the regions 
\[
\mathcal{V}_{i\nwarrow}^{(n+1)}\cap\{U^{(+)}\le u\le u_{\varepsilon,n}^{+}\}
\]
 are disjoint, we can estimate
\begin{align}
\int_{\{v=V\}\cap\{u\ge U^{(+)}\}\cap\mathcal{U}_{\varepsilon;n}^{+}}r & \Big(\frac{T_{uu}[f_{\varepsilon}]}{-\partial_{u}r}+\frac{T_{uv}[f_{\varepsilon}]}{\partial_{v}r}\Big)(u,V)\,du=\label{eq:NonTrivialBoundNormThrdCase}\\
 & =a_{\varepsilon i_{0}}\int_{\{v=V\}\cap\{u\ge U^{(+)}\}\cap\mathcal{U}_{\varepsilon;n}^{+}}r\Big(\frac{T_{uu}[f_{\varepsilon i_{0}}]}{-\partial_{u}r}+\frac{T_{uv}[f_{\varepsilon i_{0}}]}{\partial_{v}r}\Big)(u,V)\,du\le\nonumber \\
 & \le\exp(\exp(\sigma_{\varepsilon}^{-6}))a_{\varepsilon i_{0}}\int_{\inf_{\mathcal{V}_{i_{0}\nwarrow}^{(n+1)}}r}^{+\infty}\frac{1}{1-\frac{1}{3}\Lambda r^{2}}\Big(\frac{(\varepsilon^{(i_{0})})^{4}}{r^{5}}(-\Lambda)^{-1}+\frac{(\varepsilon^{(i_{0})})^{2}}{r^{3}}(-\Lambda)^{-1}\Big)\,dr\le\nonumber \\
 & \le\frac{1}{2}\exp(\exp(\sigma_{\varepsilon}^{-7}))a_{\varepsilon i_{0}}.\nonumber 
\end{align}

\end{itemize}

From (\ref{eq:BoundNormAwayRegionThirdCase}), (\ref{eq:TrivialBoundNormThirdCase})
and (\ref{eq:NonTrivialBoundNormThrdCase}), we therefore infer (\ref{eq:UpperBoundInUTvvTuv})
in the case $V>V^{(+)}$. Thus, we have established (\ref{eq:UpperBoundInUTvvTuv})
for all values of $V$.

\medskip{}

Arguing as for the proof of(\ref{eq:UpperBoundInUTvvTuv}), we similarly
obtain that, for all $U\ge0$:
\begin{equation}
\text{ }\int_{\{u=U\}\cap\mathcal{U}_{\varepsilon;n}^{+}}r\Big(\frac{T_{vv}[f_{\varepsilon}]}{\partial_{v}r}+\frac{T_{uv}[f_{\varepsilon}]}{-\partial_{u}r}\Big)(U,v)\,dv\le\frac{1}{2}\max_{0\le i\le N_{\varepsilon}}\big\{\big(\exp(e^{\sigma_{\varepsilon}^{-7}}\big)a_{\varepsilon i}\big\}+4\sum_{i=0}^{N_{\varepsilon}-1}\mu_{i}[n+1]+\frac{1}{2}\rho_{\varepsilon}^{\frac{1}{19}}.\label{eq:UpperBoundInUTvvTuv-1}
\end{equation}
Thus, adding (\ref{eq:UpperBoundInUTvvTuv}) and (\ref{eq:UpperBoundInUTvvTuv-1}),
we infer (\ref{eq:UpperBoundNormBootstrapDomainFromSequence}).

We will now proceed to establish (\ref{eq:UpperBoundBootstrapDomainMuTilde/R}).
To this end, let us define the domains 
\begin{equation}
\mathcal{\mathcal{Q}}_{i}^{(n)}=\Big(\big\{\{u\le v_{\varepsilon,i}^{(n)}+h_{\varepsilon,i}\}\cap\{v\ge v_{\varepsilon,i}^{(n)}-h_{\varepsilon,i}\}\big\}\cup\big\{ v\ge v_{\varepsilon,i}^{(n+1)}-h_{\varepsilon,i}\big\}\Big)\cap\big\{ v-u\ge\beta_{\varepsilon,i}\big\}\cap\mathcal{U}_{\varepsilon,n}^{+},\label{eq:QDomain}
\end{equation}
where $\beta_{\varepsilon,i}$ are defined by (\ref{eq:ShorthandNotationCornerPoints}).
(see Figure \ref{fig:Q_Domains}).

\begin{figure}[h] 
\centering 
\scriptsize
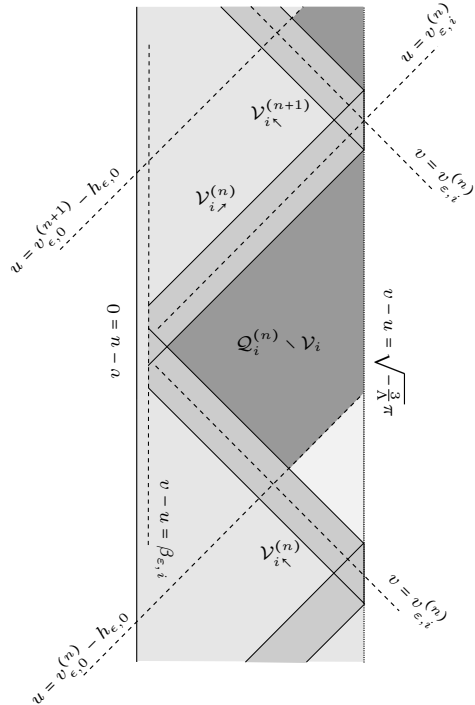 
\caption{The domain $\mathcal{Q}^{(n)}_i$ is equal to the union of the beam $\mathcal{V}_i \cap \mathcal{U}_{\epsilon,n}^{+}$ with the darker shaded region depicted above (where we assumed for simplicity that $v_{\epsilon,0}^{(n+1)}-h_{\epsilon,0}<u[\mathcal{U}_{\epsilon}^{+}]$). \label{fig:Q_Domains}}
\end{figure}

\medskip{}

\noindent \emph{Remark.} Notice that, for any $0\le i\le N_{\varepsilon}$,
the domain $\mathcal{\mathcal{Q}}_{i}^{(n)}$ consists of the region
``to the right'' of the beam $\mathcal{V}_{i}$ (including $\mathcal{V}_{i}$),
within the domain $\mathcal{U}_{\varepsilon,n}^{+}$. Moreover, in
view of (\ref{eq:BoundSupportFepsiloni}), we have 
\begin{equation}
\tilde{m}\equiv0\text{ on }\mathcal{U}_{\varepsilon;n}^{+}\backslash\cup_{i=0}^{N_{\varepsilon}}\mathcal{\mathcal{Q}}_{i}^{(n)}.\label{eq:VanishingMassToTheLeft}
\end{equation}
(since $\mathcal{U}_{\varepsilon;n}^{+}\backslash\cup_{i=0}^{N_{\varepsilon}}\mathcal{\mathcal{Q}}_{i}^{(n)}$
consists of the single connected component of $\mathcal{U}_{\varepsilon;n}^{+}\backslash\cup_{i=0}^{N_{\varepsilon}}\mathcal{V}_{i}$
containing $\gamma_{\mathcal{Z}}$). 

\medskip{}

\noindent As a corollary of the bound (\ref{eq:EstimateGeometryLikeC0})
for $\partial r$ and the definition (\ref{eq:ShorthandNotationCornerPoints})
of $\beta_{\varepsilon,i}$, we can bound for any $0\le i\le N_{\varepsilon}$:
\begin{equation}
\inf_{\mathcal{\mathcal{Q}}_{i}^{(n)}}r\ge\exp(-\exp(\sigma_{\varepsilon}^{-5}))\frac{\varepsilon^{(i)}}{\sqrt{-\Lambda}}.\label{eq:LowerBoundROnQ}
\end{equation}

In view of the relation (\ref{eq:LinearCombinationVlasovFields})
between $f_{\varepsilon}$ and the $f_{\varepsilon j}$'s, the fact
that $f_{\varepsilon j}$ is supported on $\cup_{k}\mathcal{V}_{j}^{(k)}$,
the bound (\ref{eq:EstimateGeometryLikeC0}) for $\partial r$ and
the bounds (\ref{eq:UpperBoundTuuTvvFi})\textendash (\ref{eq:UpperBoundTuvFi})
on $T_{\mu\nu}[f_{\varepsilon j}]$, we obtain from the relations
(\ref{eq:TildeUMaza})\textendash (\ref{eq:TildeVMaza}) for $\partial\tilde{m}$
(and the fact that $\tilde{m}|_{\gamma_{\mathcal{Z}}}=0$) that, for
any $0\le i\le N_{\varepsilon}$: 
\begin{equation}
\sup_{\mathcal{\mathcal{Q}}_{i}^{(n)}\backslash\cup_{j=0}^{i-1}\mathcal{\mathcal{Q}}_{j}^{(n)}}\tilde{m}\le\exp(\exp(\sigma_{\varepsilon}^{-7}))\sum_{j=i}^{N_{\varepsilon}}a_{\varepsilon j}\frac{\varepsilon^{(j)}}{\sqrt{-\Lambda}}\le\exp(\exp(\sigma_{\varepsilon}^{-7}))a_{\varepsilon i}\frac{\varepsilon^{(i)}}{\sqrt{-\Lambda}}+\exp(\exp(\sigma_{\varepsilon}^{-8}))\frac{\varepsilon^{(i+1)}}{\sqrt{-\Lambda}}.\label{eq:UpperBoundMuTildeQ}
\end{equation}
Combining (\ref{eq:LowerBoundROnQ}) and (\ref{eq:UpperBoundMuTildeQ}),
we infer that, for any $0\le i\le N_{\varepsilon}$, 
\begin{equation}
\sup_{\mathcal{\mathcal{Q}}_{i}^{(n)}\backslash\cup_{j=0}^{i-1}\mathcal{\mathcal{Q}}_{j}^{(n)}}\frac{2\tilde{m}}{r}\le\exp(\exp(\sigma_{\varepsilon}^{-8}))a_{\varepsilon i}+\varepsilon^{\frac{1}{2}}.\label{eq:UpperBoundMuOnQ}
\end{equation}
The upper bound (\ref{eq:UpperBoundBootstrapDomainMuTilde/R}) now
follows readily from (\ref{eq:UpperBoundMuOnQ}), (\ref{eq:VanishingMassToTheLeft})
and the fact that 
\begin{equation}
\mathcal{U}_{\varepsilon;n}^{+}=\bigcup_{i=0}^{N_{\varepsilon}}\Big(\mathcal{\mathcal{Q}}_{i}^{(n)}\backslash\cup_{j=0}^{i-1}\mathcal{\mathcal{Q}}_{j}^{(n)}\Big)\cup\Big(\mathcal{U}_{\varepsilon;n}^{+}\backslash\cup_{i=0}^{N_{\varepsilon}}\mathcal{\mathcal{Q}}_{i}^{(n)}\Big).\label{eq:LaLa}
\end{equation}
The upper bound (\ref{eq:UpperBoundBootstrapDomainRestrictedMu})
follows similarly from (\ref{eq:UpperBoundMuOnQ}) and (\ref{eq:VanishingMassToTheLeft}),
after noting that 
\[
(\mathcal{\mathcal{Q}}_{i}^{(n)}\backslash\cup_{k=0}^{i-1}\mathcal{\mathcal{Q}}_{k}^{(n)})\cap\{u\le v_{\varepsilon,j}^{(n)}+h_{\varepsilon,j}\}=\emptyset\text{ for all }j<i\le N_{\varepsilon}.
\]

The proof of the estimates (\ref{eq:UpperBoundNormBootstrapDomainFromSequence}),
(\ref{eq:UpperBoundBootstrapDomainMuTilde/R}) and (\ref{eq:UpperBoundBootstrapDomainRestrictedMu})
(with $\delta_{\varepsilon}$, $\tilde{h}_{\varepsilon,j}$ in place
of $\sigma_{\varepsilon}$, $h_{\varepsilon,j}$) on $\mathcal{T}_{\varepsilon;n}^{+}$
follows in exactly the same way, using (\ref{eq:EstimateGeometryLikeC0LargerDomain})
in place of (\ref{eq:EstimateGeometryLikeC0}) and replacing all the
statements about $\mathcal{V}_{i}^{(n)}$, $\mathcal{E}^{(\pm)}$,
$\mathfrak{D}r^{(\pm)}$ with the corresponding statements about $\widetilde{\mathcal{V}}_{i}^{(n)}$,
$\widetilde{\mathcal{E}}^{(\pm)}$, $\widetilde{\mathfrak{D}}r^{(\pm)}$,
respectively; we will omit the details. 
\end{proof}

\section{\label{sec:ThefirstStage}The first stage of the instability}

In this section, we will show that the parameters $\{a_{\varepsilon i}\}_{i=0}^{N_{\varepsilon}}$,
appearing in the definition of the initial data family $(r_{/}^{(\varepsilon)},(\Omega_{/}^{(\varepsilon)})^{2};\bar{f}_{/}^{(\varepsilon)})$,
can be carefully chosen (without violating the smallness condition
(\ref{eq:UpperBoundSumWeightsBeam})) so that, after $\sim\sigma_{\varepsilon}^{-\frac{3}{2}}$
reflections off $\mathcal{I}_{\varepsilon}$, the Vlasov beams form
a configuration of a particular form; this configuration will be shown
in the next section to guarantee the formation of a trapped sphere
in $O(1)$ retarded time. 

In particular, we will establish the following result (for the definition
of $v_{\varepsilon,i}^{(n)}$, $h_{\varepsilon,i}$ and the domain
$\mathcal{U}_{\varepsilon}^{+}\subset\mathcal{U}_{max}^{(\varepsilon)}$,
see Sections \ref{subsec:Notational-conventions-and}\textendash \ref{subsec:Some-basic-geometric-constructions};
for the definition of the sequences $\mu_{i}$, $\mathcal{E}_{i}$
amd $R_{i}$, see Section (\ref{subsec:The-instability-mechanism})):
\begin{prop}
\label{prop:FineTuningInitialData} For any $\varepsilon\in(0,\varepsilon_{1}]$,
there exists a finite sequence $\{a_{\varepsilon i}\}_{i=0}^{N_{\varepsilon}}\in(0,\sigma_{\varepsilon})$
satisfying (\ref{eq:UpperBoundSumWeightsBeam}),
\begin{equation}
\max_{0\le i\le N_{\varepsilon}-1}a_{\varepsilon i}<\exp(-\exp(\delta_{\varepsilon}^{-10}))\label{eq:UpperBoundAeAlmostAll}
\end{equation}
and 
\begin{equation}
a_{\varepsilon N_{\varepsilon}}<\exp(-\exp(\sigma_{\varepsilon}^{-9})),\label{eq:UpperBoundAeFinal}
\end{equation}
such that the following statements hold for the maximal future development
$(\mathcal{U}_{max}^{(\varepsilon)};r,\Omega^{2},f_{\varepsilon})$
of the initial data set $(r_{/}^{(\varepsilon)},(\Omega_{/}^{(\varepsilon)})^{2};\bar{f}_{/}^{(\varepsilon)})$
associated to $\{a_{\varepsilon i}\}_{i=0}^{N_{\varepsilon}}$ (see
Definition \ref{def:InitialDataFamily}):

\begin{enumerate}

\item Setting 
\begin{equation}
n_{+}\doteq\lceil\sigma_{\varepsilon}^{-\frac{3}{2}}\rceil,\label{eq:DefinitionN+}
\end{equation}
we have 
\begin{equation}
\big\{0<u\le v_{\varepsilon,0}^{(n_{+})}-h_{\varepsilon,0}\big\}\cap\big\{ u<v<u+\sqrt{-\frac{3}{\Lambda}}\pi\big\}\subset\mathcal{U}_{\varepsilon}^{+}.\label{eq:U+InLastStep}
\end{equation}

\item The quantities $\mu_{i}[n_{+}]$ (introduced in Definition
\ref{def:RecursiveSystemMuER}) satisfy, for all $0\le j\le N_{\varepsilon}-1$,
\begin{equation}
\mu_{j}[n_{+}]=\frac{\delta_{\varepsilon}^{-\frac{3}{4}}}{N_{\varepsilon}}e^{-2\frac{j}{N_{\varepsilon}}\delta_{\varepsilon}^{-\frac{3}{4}}},\label{eq:MuAllBeamsButOneN+}
\end{equation}
and, for $j=N_{\varepsilon}$: 
\begin{equation}
\mathcal{E}_{N_{\varepsilon}}[n_{+}]=\exp(-\exp(4\sigma_{\varepsilon}^{-9}))\frac{\varepsilon^{(N_{\varepsilon})}}{\sqrt{-\Lambda}}.\label{eq:EnergyTopBeamFinal}
\end{equation}

\end{enumerate}
\end{prop}
\begin{proof}
Let us set for convenience 
\begin{equation}
\mu_{N_{\varepsilon}}[n]\doteq2\rho_{\varepsilon}\mathcal{E}_{N_{\varepsilon}}[n]\frac{\sqrt{-\Lambda}}{\varepsilon^{(N_{\varepsilon})}}
\end{equation}
 for any $n\in\mathbb{N}$. In this way, the quantities $\mu_{i}[n]$
are defined for all $0\le i\le N_{\varepsilon}$ (note that Definition
\ref{def:RecursiveSystemMuER} only defined $\mu_{i}[n]$ for $0\le i\le N_{\varepsilon}-1$).
In particular, (\ref{eq:EnergyTopBeamFinal}) implies that 
\begin{equation}
\mu_{N_{\varepsilon}}[n_{+}]=2\rho_{\varepsilon}\exp(-\exp(4\sigma_{\varepsilon}^{-9})),\label{eq:MuTopBeam}
\end{equation}
while (\ref{eq:FormulaRecursiveEnergy}) implies that 
\begin{equation}
\mu_{N_{\varepsilon}}[n+1]=\mu_{N_{\varepsilon}}[n]\cdot\exp\Big(\sum_{j=0}^{N_{\varepsilon}-1}\mu_{j}[n+1]\Big)\label{eq:RecursiveMuTopBeam}
\end{equation}
for all $n\in\mathbb{N}$.

In view of the initial condition (\ref{eq:InitialMuI}) for $\mu_{i}$
for $0\le i\le N_{\varepsilon}-1$ (which, in particular, expresses
$\mu_{i}$ as a function of $T_{\mu\nu}$ and $r$ along $u=v_{\varepsilon,0}^{(0)}-h_{0,\varepsilon}$),
the form (\ref{eq:InitialVlasovTotal}) of $\bar{f}_{/}^{(\varepsilon)}$,
the bound (\ref{eq:SizeInitialData}) and the Cauchy stability statement
of Proposition \ref{prop:CauchyStabilityAdS} applied to $(r_{/}^{(\varepsilon)},(\Omega_{/}^{(\varepsilon)})^{2};\bar{f}_{/}^{(\varepsilon)})$
(implying, in particular, that $\frac{\partial_{v}r}{1-\frac{1}{3}\Lambda r^{2}}=\frac{1}{2}+O(\sigma_{\varepsilon})$
on $\{0\le u\le v_{\varepsilon,0}^{(0)}-h_{0,\varepsilon}\}$), we
readily infer that, for any $\varepsilon\in[0,\varepsilon_{1})$,
the quantities $\{\mu_{i}[0]\}_{i=0}^{N_{\varepsilon}}$ uniquely
determine $\{a_{\varepsilon i}\}_{i=0}^{N_{\varepsilon}}$ and vice-versa
and, moreover, 
\begin{equation}
C_{1}\rho_{\varepsilon}a_{\varepsilon i}\le\mu_{i}[0]\le C_{2}\rho_{\varepsilon}a_{\varepsilon i}\text{ for all }0\le i\le N_{\varepsilon},\label{eq:ComparableMu0Aei}
\end{equation}
 for some constants $C_{1},C_{2}>0$ independent of $i$, $\varepsilon$. 

By solving the recursive relation (\ref{eq:RecursiveFormulaMu}) for
$\mu_{i}[n]$ backwards in $n$, for $0\le i\le N_{\varepsilon}-1$,\footnote{Solving (\ref{eq:RecursiveFormulaMu}) backwards in $n$ can be performed
inductively in $i$: For $i=0$, $\mu_{i}[n]$ is constant in $n$,
while for any $i>0$, knowledge of $\{\mu_{\bar{i}}[\bar{n}]\}_{\bar{i}\le i-1}$
for all $0\le\bar{n}\le n$ completely determines $\mu_{i}[\bar{n}]$,
for all $0\le\bar{n}\le n$, in terms of $\mu_{i}[n]$.} and then solving (\ref{eq:RecursiveMuTopBeam}) backwards in $n$
for $i=N_{\varepsilon}$, we infer that the conditions (\ref{eq:MuAllBeamsButOneN+})
and (\ref{eq:MuTopBeam}) completely determine $\{\mu_{i}[0]\}_{i=0}^{N_{\varepsilon}}$
and, hence, $\{a_{\varepsilon i}\}_{i=0}^{N_{\varepsilon}}$. For
this reason, in order establish Proposition \ref{prop:FineTuningInitialData},
it suffices to show the following (in view of (\ref{eq:ComparableMu0Aei})):

\begin{itemize}

\item The finite sequence $\{\mu_{i}[0]\}_{i=0}^{N_{\varepsilon}}$,
fixed uniquely by the future conditions (\ref{eq:MuAllBeamsButOneN+})
and (\ref{eq:MuTopBeam}), satisfies 
\begin{equation}
\mu_{i}[0]<\rho_{\varepsilon}\exp(-\exp(2\delta_{\varepsilon}^{-10}))\text{ for all }0\le i\le N_{\varepsilon}-1,\label{eq:BoundMaxMui0}
\end{equation}
\begin{equation}
\mu_{N_{\varepsilon}}[0]<\rho_{\varepsilon}\exp(-\exp(2\sigma_{\varepsilon}^{-9}))\label{eq:BoundMuNe}
\end{equation}
 and 
\begin{equation}
\sum_{i=0}^{N_{\varepsilon}}\mu_{i}[0]\le C_{1}\sigma_{\varepsilon}.\label{eq:BoundSumMui0}
\end{equation}

\item The maximal future development $(\mathcal{U}_{max}^{(\varepsilon)};r,\Omega^{2},f_{\varepsilon})$
of the initial data set $(r_{/}^{(\varepsilon)},(\Omega_{/}^{(\varepsilon)})^{2};\bar{f}_{/}^{(\varepsilon)})$
associated to the finite sequence $\{a_{\varepsilon i}\}_{i=0}^{N_{\varepsilon}}$
(uniquely determined by $\{\mu_{i}[0]\}_{i=0}^{N_{\varepsilon}}$)
satisfies (\ref{eq:U+InLastStep}).

\end{itemize}

\medskip{}

\noindent \emph{Step 1: Proof of (\ref{eq:BoundMaxMui0})\textendash (\ref{eq:BoundSumMui0}).}
The relations (\ref{eq:MuAllBeamsButOneN+}) and (\ref{eq:MuTopBeam})
for $\mu_{i}[n_{+}]$ readily imply that 
\begin{equation}
\frac{1}{4}\le\sum_{i=0}^{N_{\varepsilon}}\mu_{i}[n_{+}]\le1.\label{eq:SumN+ComparableTo1}
\end{equation}

From (\ref{eq:RecursiveFormulaMu}) and (\ref{eq:RecursiveMuTopBeam})
we infer that, for any $0\le i\le N_{\varepsilon}$: 
\begin{equation}
\mu_{i}[n]\le\mu_{i}[n+1]\label{eq:IncreasingMuI}
\end{equation}
(with equality only when $i=0$). In particular,
\begin{equation}
\mu_{i}[0]\le\mu_{i}[n_{+}]\text{ for all }0\le i\le N_{\varepsilon}.\label{eq:EasyUpperBoundMuInitial}
\end{equation}
 From (\ref{eq:EasyUpperBoundMuInitial}), we infer, using (\ref{eq:MuAllBeamsButOneN+})
and (\ref{eq:MuTopBeam}), that 
\begin{align*}
\mu_{i}[0] & \le N_{\varepsilon}^{-1}\delta_{\varepsilon}^{-\frac{3}{4}},\\
\mu_{N_{\varepsilon}}[0] & \le2\rho_{\varepsilon}\exp(-\exp(4\sigma_{\varepsilon}^{-9})),
\end{align*}
from which (\ref{eq:BoundMaxMui0}) and (\ref{eq:BoundMuNe}) follow
readily, in view of the fact that $\delta_{\varepsilon}\ll\sigma_{\varepsilon}\ll1$
and $N_{\varepsilon}=\rho_{\varepsilon}^{-1}\exp(e^{\delta_{\varepsilon}^{-15}})$.

Let us define $i_{0}$ to be the minimum number in the set $\{0,1,\ldots,N_{\varepsilon}\}$
such that 
\begin{equation}
\sum_{i=0}^{i_{0}}\mu_{i}[0]\ge\frac{1}{2}C_{1}\sigma_{\varepsilon}.\label{eq:LowerBoundNormI0}
\end{equation}
Notice that the definition of $i_{0}$ implies that 
\begin{equation}
\sum_{i=0}^{i_{0}-1}\mu_{i}[0]<\frac{1}{2}C_{1}\sigma_{\varepsilon}.\label{eq:UpperBoundI0-1}
\end{equation}
Note also that the bound (\ref{eq:BoundMaxMui0}) (which we already
established) implies that, necessarily, 
\[
i_{0}\ge1.
\]

For the proof of (\ref{eq:BoundSumMui0}), we will consider two cases,
depending on the value of $i_{0}$: 

\begin{itemize}

\item In the case when $i_{0}=N_{\varepsilon}$, the bound (\ref{eq:LowerBoundNormI0})
trivially implies (\ref{eq:BoundSumMui0}).

\item In the case when $i_{0}\le N_{\varepsilon}-1$, for any $i_{0}+1\le j\le N_{\varepsilon}-1$
and any $n\ge0$, we can estimate using the recursive formula (\ref{eq:RecursiveFormulaMu}),
the monotonicity property (\ref{eq:IncreasingMuI}) and the lower
bound (\ref{eq:LowerBoundNormI0}): 
\begin{align}
\mu_{j}[n+1] & =\mu_{j}[n]\exp\Big(2\sum_{k=0}^{j-1}\mu_{k}[n+1]\Big)=\label{eq:MJAlmostThere}\\
 & =\mu_{j}[0]\exp\Big(2\sum_{\bar{n}=0}^{n}\sum_{k=0}^{j-1}\mu_{k}[\bar{n}+1]\Big)\ge\nonumber \\
 & \ge\mu_{j}[0]\exp\Big(2\sum_{\bar{n}=0}^{n}\sum_{k=0}^{i_{0}}\mu_{k}[\bar{n}+1]\Big)\ge\nonumber \\
 & \ge\mu_{j}[0]\exp\Big(2\sum_{\bar{n}=0}^{n}\sum_{k=0}^{i_{0}}\mu_{k}[0]\Big)\ge\nonumber \\
 & \ge\mu_{j}[0]\exp\Big(2\sum_{\bar{n}=0}^{n}(\frac{1}{2}C_{1}\sigma_{\varepsilon})\Big)=\nonumber \\
 & =\mu_{j}[0]\exp\Big(C_{1}\sigma_{\varepsilon}n\Big).\nonumber 
\end{align}
Similarly, for $j=N_{\varepsilon}$, we can estimate using (\ref{eq:RecursiveMuTopBeam})
in place of (\ref{eq:RecursiveFormulaMu}):
\begin{align}
\mu_{N_{\varepsilon}}[n+1] & =\mu_{N_{\varepsilon}}[n]\exp\Big(\sum_{k=0}^{N_{\varepsilon}-1}\mu_{k}[n+1]\Big)=\label{eq:MJAlmostThere-1}\\
 & =\mu_{N_{\varepsilon}}[0]\exp\Big(\sum_{\bar{n}=0}^{n}\sum_{k=0}^{N_{\varepsilon}-1}\mu_{k}[\bar{n}+1]\Big)\ge\nonumber \\
 & \ge\mu_{N_{\varepsilon}}[0]\exp\Big(\sum_{\bar{n}=0}^{n}\sum_{k=0}^{i_{0}}\mu_{k}[0]\Big)\ge\nonumber \\
 & =\mu_{N_{\varepsilon}}[0]\exp\Big(\frac{1}{2}C_{1}\sigma_{\varepsilon}n\Big).\nonumber 
\end{align}
From (\ref{eq:MJAlmostThere}) and (\ref{eq:MJAlmostThere-1}) for
$n=n_{+}-1$, using also the definition (\ref{eq:DefinitionN+}) of
$n_{+}$, the upper bound (\ref{eq:SumN+ComparableTo1}) and the fact
that $C_{1}$ is an absolute constant, we obtain that (provided $\varepsilon_{1}$
has been fixed small enough)
\begin{align}
\sum_{j=i_{0}+1}^{N_{\varepsilon}}\mu_{j}[0] & \le\exp\Big(-\frac{1}{2}C_{1}\sigma_{\varepsilon}(n_{+}-1)\Big)\sum_{j=i_{0}+1}^{N_{\varepsilon}}\mu_{j}[n_{+}]\le\label{eq:UpperBoundRestOfMu}\\
 & \le\exp\Big(-\frac{1}{2}C_{1}\sigma_{\varepsilon}(n_{+}-1)\Big)\le\nonumber \\
 & \le\exp\Big(-\frac{1}{2}C_{1}\sigma_{\varepsilon}^{-\frac{1}{2}}\Big)\le\nonumber \\
 & \le\exp\Big(-\sigma_{\varepsilon}^{-\frac{1}{4}}\Big).\nonumber 
\end{align}
From (\ref{eq:UpperBoundI0-1}), (\ref{eq:BoundMaxMui0}) and (\ref{eq:UpperBoundRestOfMu})
(using also the relation (\ref{eq:HierarchyOfParameters}) between
$\rho_{\varepsilon}$, $\delta_{\varepsilon}$, $\sigma_{\varepsilon}$)
, we therefore obtain that 
\begin{align*}
\sum_{j=0}^{N_{\varepsilon}}\mu_{j}[0] & =\sum_{j=0}^{i_{0}}\mu_{j}[0]+\mu_{i_{0}}[0]+\sum_{j=i_{0}+1}^{N_{\varepsilon}}\mu_{j}[0]\le\\
 & \le\frac{1}{2}C_{1}\sigma_{\varepsilon}+\rho_{\varepsilon}\exp(-\exp(2\delta_{\varepsilon}^{-10}))+\exp\Big(-\sigma_{\varepsilon}^{-\frac{1}{4}}\Big)\le\\
 & \le C_{1}\sigma_{\varepsilon},
\end{align*}
hence inferring (\ref{eq:BoundSumMui0}).

\end{itemize}

\medskip{}

\noindent \emph{Step 2: Proof of (\ref{eq:U+InLastStep}).} The inclusion
(\ref{eq:U+InLastStep}) is equivalent to the bound
\begin{equation}
v_{\varepsilon,0}^{(n_{+})}-h_{\varepsilon,0}<u[\mathcal{U}_{\varepsilon}^{+}],\label{eq:MaximumExtensionForUe}
\end{equation}
in view of the form (\ref{eq:DefinitionBootstrapDomainWithU}) of
$\mathcal{U}_{\varepsilon}^{+}$. For the sake of contradiction, let
us assume that 
\begin{equation}
u[\mathcal{U}_{\varepsilon}^{+}]\le v_{\varepsilon,0}^{(n_{+})}-h_{\varepsilon,0}.\label{eq:UeForContradiction}
\end{equation}
Notice that (\ref{eq:UeForContradiction}) implies (in view of (\ref{eq:DefinitionN+}))
that 
\begin{equation}
u[\mathcal{U}_{\varepsilon}^{+}]\lesssim\frac{\sigma_{\varepsilon}^{-\frac{3}{2}}}{\sqrt{-\Lambda}}\ll\frac{\sigma_{\varepsilon}^{-2}}{\sqrt{-\Lambda}}.\label{eq:Aaaaaaaaaaaah}
\end{equation}
Hence, Lemma \ref{lem:ExtensionPrincipleSpecialDomains} implies that
at least one of the relations (\ref{eq:ExtensionConditionMuTilde}),
(\ref{eq:ExtensionConditionNormUCons})and (\ref{eq:ExtensionConditionNormVCons})
holds.

In view of the bound (\ref{eq:UpperBoundNormBootstrapDomainFromSequence}),
we can estimate (using the hypothesis (\ref{eq:UeForContradiction})):
\begin{align}
\text{ }\sup_{V\ge0}\int_{\{v=V\}\cap\mathcal{U}_{\varepsilon}^{+}} & r\Big(\frac{T_{uu}[f_{\varepsilon}]}{-\partial_{u}r}+\frac{T_{uv}[f_{\varepsilon}]}{\partial_{v}r}\Big)(u,V)\,du+\label{eq:UpperBoundNormBootstrapDomainFromSequence-1}\\
+\sup_{U\ge0} & \int_{\{u=U\}\cap\mathcal{U}_{\varepsilon}^{+}}r\Big(\frac{T_{vv}[f_{\varepsilon}]}{\partial_{v}r}+\frac{T_{uv}[f_{\varepsilon}]}{-\partial_{u}r}\Big)(U,v)\,dv=\nonumber \\
= & \max_{n\le n_{+}-1}\Bigg\{\sup_{V\ge0}\int_{\{v=V\}\cap\mathcal{U}_{\varepsilon;n}^{+}}r\Big(\frac{T_{uu}[f_{\varepsilon}]}{-\partial_{u}r}+\frac{T_{uv}[f_{\varepsilon}]}{\partial_{v}r}\Big)(u,V)\,du+\nonumber \\
 & +\sup_{U\ge0}\int_{\{u=U\}\cap\mathcal{U}_{\varepsilon;n}^{+}}r\Big(\frac{T_{vv}[f_{\varepsilon}]}{\partial_{v}r}+\frac{T_{uv}[f_{\varepsilon}]}{-\partial_{u}r}\Big)(U,v)\,dv\Bigg\}\le\nonumber \\
\le & \max_{n\le n_{+}-1}\Bigg\{8\sum_{i=0}^{N_{\varepsilon}-1}\mu_{i}[n+1]+\max_{0\le i\le N_{\varepsilon}}\big\{\big(\exp(e^{\sigma_{\varepsilon}^{-7}})\big)a_{\varepsilon i}\big\}+\rho_{\varepsilon}^{\frac{1}{19}}\Bigg\}.\nonumber 
\end{align}
Using the bounds (\ref{eq:UpperBoundAeAlmostAll}) and (\ref{eq:UpperBoundAeFinal})
for $a_{\varepsilon i}$, the relation (\ref{eq:MuAllBeamsButOneN+})
for $\mu_{i}[n_{+}]$, as well as the fact that $\mu_{i}[n]$ is increasing
in $n$ for all $0\le i\le N_{\varepsilon}$, we infer from (\ref{eq:UpperBoundNormBootstrapDomainFromSequence-1})
that 
\begin{align}
\text{ }\sup_{V\ge0}\int_{\{v=V\}\cap\mathcal{U}_{\varepsilon}^{+}} & r\Big(\frac{T_{uu}[f_{\varepsilon}]}{-\partial_{u}r}+\frac{T_{uv}[f_{\varepsilon}]}{\partial_{v}r}\Big)(u,V)\,du+\label{eq:UpperBoundNormBootstrapDomainFromSequence-1-1}\\
+\sup_{U\ge0} & \int_{\{u=U\}\cap\mathcal{U}_{\varepsilon}^{+}}r\Big(\frac{T_{vv}[f_{\varepsilon}]}{\partial_{v}r}+\frac{T_{uv}[f_{\varepsilon}]}{-\partial_{u}r}\Big)(U,v)\,dv\le\nonumber \\
\le & 8\sum_{i=0}^{N_{\varepsilon}-1}\mu_{i}[n_{+}]+\exp(e^{\sigma_{\varepsilon}^{-7}}\big)\cdot\exp(e^{-\sigma_{\varepsilon}^{-9}}\big)+\rho_{\varepsilon}^{\frac{1}{19}}\Bigg\}\le\nonumber \\
\le & 8\sum_{i=0}^{N_{\varepsilon}-1}\frac{\delta_{\varepsilon}^{-\frac{3}{4}}}{N_{\varepsilon}}e^{-2\frac{j}{N_{\varepsilon}}\delta_{\varepsilon}^{-\frac{3}{4}}}+O(\sigma_{\varepsilon})\le\nonumber \\
\le & 20.\nonumber 
\end{align}
Furthermore, we infer from (\ref{eq:UpperBoundBootstrapDomainMuTilde/R})
(using the bounds (\ref{eq:UpperBoundAeAlmostAll}) and (\ref{eq:UpperBoundAeFinal})
for $a_{\varepsilon i}$) that 
\begin{align}
\sup_{\mathcal{U}_{\varepsilon}^{+}}\frac{2\tilde{m}}{r} & \le\max_{n\le n_{+}-1}\Big\{\sup_{\mathcal{U}_{\varepsilon;n}^{+}}\frac{2\tilde{m}}{r}\Big\}\le\label{eq:UpperBoundBootstrapDomainMuTilde/R-1}\\
 & \le\max_{0\le i\le N_{\varepsilon}}\big\{\big(\exp(e^{\sigma_{\varepsilon}^{-8}}\big)a_{\varepsilon i}\big\}+\varepsilon^{\frac{1}{2}}\le\nonumber \\
 & \le\exp(e^{\sigma_{\varepsilon}^{-8}}\big)\cdot\exp(e^{-\sigma_{\varepsilon}^{-9}}\big)+\varepsilon^{\frac{1}{2}}\le\nonumber \\
 & \le\frac{1}{2}\eta_{0}.\nonumber 
\end{align}
In view of the estimates \ref{eq:UpperBoundNormBootstrapDomainFromSequence-1-1}
and \ref{eq:UpperBoundBootstrapDomainMuTilde/R-1}, we therefore deduce
that none of the relations (\ref{eq:ExtensionConditionMuTilde}),
(\ref{eq:ExtensionConditionNormUCons}) and (\ref{eq:ExtensionConditionNormVCons})
can hold on $\mathcal{U}_{\varepsilon}^{+}$, which is a contradiction.
Hence, (\ref{eq:MaximumExtensionForUe}) holds.

Therefore, the proof of Proposition \ref{prop:FineTuningInitialData}
is complete.
\end{proof}

\section{\label{sec:The-final-stage}The final stage of the instability: Formation
of a black hole region }

In this section, we will show that, with the initial data parameters
$\{a_{\varepsilon i}\}_{i=0}^{N_{\varepsilon}}$ chosen as dictated
by Proposition \ref{prop:FineTuningInitialData}, the maximal future
development $(\mathcal{U}_{max}^{(\varepsilon)};r,\Omega^{2},f_{\varepsilon})$
of the associated initial data set $(r_{/}^{(\varepsilon)},(\Omega_{/}^{(\varepsilon)})^{2};\bar{f}_{/}^{(\varepsilon)})$
satisfies (\ref{eq:TrappedSphere}). Since (\ref{eq:ConvergenceToTrivialTheorem})
was already established in Lemma \ref{lem:SmallnessFamilyInitialData},
this section will complete the proof of Theorem \ref{thm:TheTheorem}.

\subsection{\label{subsec:Energy-growth-for-final-beam}Energy growth for the
final beam}

In order to complete the proof of Theorem \ref{thm:TheTheorem}, our
aim is to show that a trapped sphere is formed along the beam $\mathcal{V}_{N_{\varepsilon}\nwarrow}^{(n_{+})}$,
after its interaction with the beams $\mathcal{V}_{i\nearrow}^{(n_{+})}$,
$i\le N_{\varepsilon}-1$. To this end, in this section, we will first
establish the following result regarding the increase in the energy
content of the $N_{\varepsilon}$-th Vlasov beam occuring through
these interactions:
\begin{lem}
\label{lem:ProfileOfTheLastStep} For any $\varepsilon\in(0,\varepsilon_{1}]$,
let $\{a_{\varepsilon i}\}_{i=0}^{N_{\varepsilon}}$ and $n_{+}$
be as in Proposition \ref{prop:FineTuningInitialData}, and let $(\mathcal{U}_{max}^{(\varepsilon)};r,\Omega^{2},f_{\varepsilon})$
be the maximal future development of the initial data set $(r_{/}^{(\varepsilon)},(\Omega_{/}^{(\varepsilon)})^{2};\bar{f}_{/}^{(\varepsilon)})$
associated to $\{a_{\varepsilon i}\}_{i=0}^{N_{\varepsilon}}$. Then,
the following $(u,v)$-region is contained in the domain $\mathcal{T}_{\varepsilon}^{+}\subset\mathcal{U}_{max}^{(\varepsilon)}$
(see Section \ref{subsec:Notational-conventions-and} for the relevant
definitions):
\begin{equation}
\big\{0<u\le v_{\varepsilon,N_{\varepsilon}-1}^{(n_{+})}+\tilde{h}_{\varepsilon,N_{\varepsilon}-1}\big\}\cap\big\{ u<v<u+\sqrt{-\frac{3}{\Lambda}}\pi\big\}\subset\mathcal{T}_{\varepsilon}^{+}.\label{eq:LastStepTau}
\end{equation}
Furthermore, we have 
\begin{equation}
\widetilde{\mathcal{E}}_{\nwarrow}^{(+)}[n_{+};N_{\varepsilon},N_{\varepsilon}-1]\ge\delta_{\varepsilon}^{-\frac{1}{2}}\frac{\varepsilon^{(N_{\varepsilon})}}{\sqrt{-\Lambda}}.\label{eq:LowerBoundEnergyTopBeam}
\end{equation}
\end{lem}
\begin{proof}
Before establishing (\ref{eq:LastStepTau}) and (\ref{eq:LowerBoundEnergyTopBeam}),
we will first show that
\begin{equation}
\mu_{j}[n_{+}+1]=N_{\varepsilon}^{-1}\delta_{\varepsilon}^{-\frac{3}{4}}\text{ for any }0\le j\le N_{\varepsilon}-1.\label{eq:MuAllBeamsFinalStepOutgoing}
\end{equation}

The recursive formula (\ref{eq:RecursiveFormulaMu}) yields that,
for all $0\le j\le N_{\varepsilon}-1$: 
\begin{equation}
\mu_{j}[n_{+}+1]=\mu_{j}[n_{+}]\cdot\exp\Big(2\sum_{k=0}^{j-1}\mu_{k}[n_{+}+1]\Big),\label{eq:RecursiveFormulaMu-1}
\end{equation}
while, in view of the relation (\ref{eq:MuAllBeamsButOneN+}), we
have for all $0\le j\le N_{\varepsilon}-1$: 
\begin{equation}
\mu_{j}[n_{+}]=\frac{\delta_{\varepsilon}^{-\frac{3}{4}}}{N_{\varepsilon}}e^{-2\frac{j}{N_{\varepsilon}}\delta_{\varepsilon}^{-\frac{3}{4}}}.\label{eq:MuAllBeamsButOneN+-1}
\end{equation}
We will show (\ref{eq:MuAllBeamsFinalStepOutgoing}) by arguing inductively
in $j$:

\begin{itemize}

\item For $j=0$, (\ref{eq:RecursiveFormulaMu-1}) and (\ref{eq:MuAllBeamsButOneN+-1})
imply that 
\[
\mu_{0}[n_{+}+1]=\mu_{0}[n_{+}]=N_{\varepsilon}^{-1}\delta_{\varepsilon}^{-\frac{3}{4}}.
\]

\item Assuming that, for some $1\le j_{0}\le N_{\varepsilon}-1$,
the relation (\ref{eq:MuAllBeamsFinalStepOutgoing}) holds for all
$0\le j\le j_{0}-1$, we calculate from (\ref{eq:RecursiveFormulaMu-1})
and (\ref{eq:MuAllBeamsButOneN+-1}) for $j=j_{0}$ that 
\begin{align*}
\mu_{j_{0}}[n_{+}+1] & =\mu_{j_{0}}[n_{+}]\cdot\exp\Big(2\sum_{k=0}^{j_{0}-1}\mu_{k}[n_{+}+1]\Big)=\\
 & =\frac{\delta_{\varepsilon}^{-\frac{3}{4}}}{N_{\varepsilon}}e^{-2\frac{j_{0}}{N_{\varepsilon}}\delta_{\varepsilon}^{-\frac{3}{4}}}\cdot\exp\Big(2\sum_{k=0}^{j_{0}-1}(N_{\varepsilon}^{-1}\delta_{\varepsilon}^{-\frac{3}{4}})\Big)=\\
 & =\frac{\delta_{\varepsilon}^{-\frac{3}{4}}}{N_{\varepsilon}}e^{-2\frac{j_{0}}{N_{\varepsilon}}\delta_{\varepsilon}^{-\frac{3}{4}}}\cdot e^{+2\frac{j_{0}}{N_{\varepsilon}}\delta_{\varepsilon}^{-\frac{3}{4}}}=\\
 & =\frac{\delta_{\varepsilon}^{-\frac{3}{4}}}{N_{\varepsilon}},
\end{align*}
i.\,e.~(\ref{eq:MuAllBeamsFinalStepOutgoing}) also holds for $j=j_{0}$.
Therefore, (\ref{eq:MuAllBeamsFinalStepOutgoing}) holds for all $0\le j\le N_{\varepsilon}-1$.

\end{itemize}

We will now proceed to show the inclusion (\ref{eq:LastStepTau}).
In view of the form (\ref{eq:DefinitionBootstrapDomainWithU}) of
the domain $\mathcal{T}_{\varepsilon}^{+}$, (\ref{eq:LastStepTau})
is equivalent to the bound 
\begin{equation}
v_{\varepsilon,N_{\varepsilon}-1}^{(n_{+})}+\tilde{h}_{\varepsilon,N_{\varepsilon}-1}<u[\mathcal{T}_{\varepsilon}^{+}].\label{eq:MaximumExtensionForTe}
\end{equation}
In order to establish (\ref{eq:MaximumExtensionForTe}), we will assume,
for the sake of contradiction, that 
\begin{equation}
u[\mathcal{T}_{\varepsilon}^{+}]\le v_{\varepsilon,N_{\varepsilon}-1}^{(n_{+})}+\tilde{h}_{\varepsilon,N_{\varepsilon}-1}.\label{eq:ForContradictionTe}
\end{equation}
In view of the inclusion (\ref{eq:MaximumExtensionForUe}) for $\mathcal{U}_{\varepsilon}^{+}$
and the fact that $u[\mathcal{U}_{\varepsilon}^{+}]<u[\mathcal{T}_{\varepsilon}^{+}]$,
the bound (\ref{eq:MaximumExtensionForTe}) in fact implies that
\begin{equation}
v_{\varepsilon,0}^{(n_{+})}-h_{\varepsilon,0}<u[\mathcal{T}_{\varepsilon}^{+}]\le v_{\varepsilon,N_{\varepsilon}-1}^{(n_{+})}+\tilde{h}_{\varepsilon,N_{\varepsilon}-1}.\label{eq:ForContradictionTe-1}
\end{equation}

In view of (\ref{eq:ForContradictionTe}) and the definition (\ref{eq:DefinitionN+})
of $n_{+}$, we can bound
\begin{equation}
u[\mathcal{T}_{\varepsilon}^{+}]\ll\frac{\sigma_{\varepsilon}^{-2}}{\sqrt{-\Lambda}}.
\end{equation}
As a consequence of Lemma \ref{lem:ExtensionPrincipleSpecialDomains},
we therefore infer that one of the following conditions holds:
\begin{equation}
\limsup_{p\rightarrow\{u=u[\mathcal{T}_{\varepsilon}^{+}]\}}\frac{2\tilde{m}}{r}(p)=\eta_{0},\label{eq:ExtensionConditionMuTilde-2}
\end{equation}
\begin{equation}
\text{ }\lim\sup_{u\rightarrow u[\mathcal{T}_{\varepsilon}^{+}]}\int_{u}^{u+\sqrt{-\frac{3}{\Lambda}}\pi}r\Big(\frac{T_{vv}[f_{\varepsilon}]}{\partial_{v}r}+\frac{T_{uv}[f_{\varepsilon}]}{-\partial_{u}r}\Big)(u,v)\,dv=\delta_{\varepsilon}^{-1},\label{eq:ExtensionConditionNormUCons-2}
\end{equation}
or 
\begin{equation}
\sup_{v\in(0,u[\mathcal{T}_{\varepsilon}^{+}]+\sqrt{-\frac{3}{\Lambda}}\pi)}\int_{\max\{0,v-\sqrt{-\frac{3}{\Lambda}}\pi\}}^{\min\{v,u[\mathcal{T}_{\varepsilon}^{+}]\}}r\Big(\frac{T_{uv}[f_{\varepsilon}]}{\partial_{v}r}+\frac{T_{uu}[f_{\varepsilon}]}{-\partial_{u}r}\Big)(u,v)\,du=\delta_{\varepsilon}^{-1}.\label{eq:ExtensionConditionNormVCons-2}
\end{equation}

In view of the bounds (\ref{eq:ForContradictionTe-1}) and (\ref{eq:UpperBoundAeAlmostAll}),
Lemma \ref{lem:ControlForExtension} (and, in particular, the estimate
(\ref{eq:UpperBoundBootstrapDomainRestrictedMu}) for $\mathcal{T}_{\varepsilon;n}^{+}$,
$\delta_{\varepsilon}$ in place of $\mathcal{U}_{\varepsilon;n}^{+}$,
$\sigma_{\varepsilon}$) implies that 
\begin{equation}
\sup_{\mathcal{T}_{\varepsilon;n_{+}}^{+}\cap\{u\le v_{\varepsilon,N_{\varepsilon}-1}^{(n_{+})}+\tilde{h}_{\varepsilon,N_{\varepsilon}-1}\}}\frac{2\tilde{m}}{r}\le\max_{0\le i\le N_{\varepsilon}-1}\big\{\big(\exp(e^{\delta_{\varepsilon}^{-8}}\big)a_{\varepsilon i}\big\}+\varepsilon^{\frac{1}{2}}\le\exp(e^{-\delta_{\varepsilon}^{-9}}\big)\le\frac{1}{2}\eta_{0}.\label{eq:Ante}
\end{equation}
Furthermore, the bound (\ref{eq:UpperBoundNormBootstrapDomainFromSequence})
of Lemma \ref{lem:ControlForExtension} (with $\mathcal{T}_{\varepsilon;n}^{+}$,
$\delta_{\varepsilon}$ in place of $\mathcal{U}_{\varepsilon;n}^{+}$,
$\sigma_{\varepsilon}$), together with (\ref{eq:UpperBoundAeAlmostAll}),
(\ref{eq:MuAllBeamsFinalStepOutgoing}) and the fact that $\mu_{i}[n]$
is increasing in $n$ for all $0\le i\le N_{\varepsilon}-1$, imply
that 
\begin{align}
\text{ }\sup_{V\ge0}\int_{\{v=V\}\cap\mathcal{T}_{\varepsilon;n}^{+}}r\Big(\frac{T_{uu}[f_{\varepsilon}]}{-\partial_{u}r}+\frac{T_{uv}[f_{\varepsilon}]}{\partial_{v}r}\Big) & (u,V)\,du+\label{eq:UpperBoundNormBootstrapDomainFromSequence-2}\\
\sup_{U\ge0}\int_{\{u=U\}\cap\mathcal{T}_{\varepsilon;n}^{+}}r\Big(\frac{T_{vv}[f_{\varepsilon}]}{\partial_{v}r}+ & \frac{T_{uv}[f_{\varepsilon}]}{-\partial_{u}r}\Big)(U,v)\,dv\le\nonumber \\
\le & 8\sum_{i=0}^{N_{\varepsilon}-1}\mu_{i}[n_{+}+1]+\max_{0\le i\le N_{\varepsilon}}\big\{\big(\exp(e^{\delta_{\varepsilon}^{-7}}\big)a_{\varepsilon i}\big\}+\rho_{\varepsilon}^{\frac{1}{19}}\le\nonumber \\
\le & 8N_{\varepsilon}\cdot N_{\varepsilon}^{-1}\delta_{\varepsilon}^{-\frac{3}{4}}+\exp(e^{-\delta_{\varepsilon}^{-9}}\big)\le\nonumber \\
\le & \frac{1}{2}\delta_{\varepsilon}^{-1}.\nonumber 
\end{align}
The bounds (\ref{eq:Ante}) and (\ref{eq:UpperBoundNormBootstrapDomainFromSequence-2})
readily imply that none of the relations (\ref{eq:ExtensionConditionMuTilde-2})\textendash (\ref{eq:ExtensionConditionNormVCons-2})
can hold, which is contradiction. Hence, (\ref{eq:MaximumExtensionForTe})
(and, therefore, (\ref{eq:LastStepTau})) holds.

We are finally ready to establish (\ref{eq:LowerBoundEnergyTopBeam}):
From Proposition \ref{prop:TotalEnergyChange} (in particular, the
relation (\ref{eq:EqualEnergyFluxFinalIngoingTau})), the relations
(\ref{eq:EnergyTopBeamFinal}) for $\mathcal{E}_{N_{\varepsilon}}[n_{+}]$
and (\ref{eq:MuAllBeamsFinalStepOutgoing}) for $\mu_{j}[n_{+}+1]$,
as well as the recursive formula (\ref{eq:FormulaRecursiveEnergy})
for $\mathcal{E}_{N_{\varepsilon}}[n]$, we obtain that 
\begin{align}
\widetilde{\mathcal{E}}_{\nwarrow}^{(+)}[n_{+};N_{\varepsilon},N_{\varepsilon}-1] & =\mathcal{E}_{N_{\varepsilon}}[n_{+}+1]+O\Big(\rho_{\varepsilon}^{\frac{1}{17}}\frac{\varepsilon^{(N_{\varepsilon})}}{\sqrt{-\Lambda}}\Big)=\label{eq:AntaePia}\\
 & =\mathcal{E}_{N_{\varepsilon}}[n_{+}]\cdot\exp\Big(\sum_{j=0}^{N_{\varepsilon}-1}\mu_{j}[n_{+}+1]\Big)+O\Big(\rho_{\varepsilon}^{\frac{1}{17}}\frac{\varepsilon^{(N_{\varepsilon})}}{\sqrt{-\Lambda}}\Big)=\nonumber \\
 & =\exp(-\exp(4\sigma_{\varepsilon}^{-9}))\frac{\varepsilon^{(N_{\varepsilon})}}{\sqrt{-\Lambda}}\cdot\exp\Big(N_{\varepsilon}\cdot N_{\varepsilon}^{-1}\delta_{\varepsilon}^{-\frac{3}{4}}\Big)+O\Big(\rho_{\varepsilon}^{\frac{1}{17}}\frac{\varepsilon^{(N_{\varepsilon})}}{\sqrt{-\Lambda}}\Big)=\nonumber \\
 & =\Big(\exp\big(\delta_{\varepsilon}^{-\frac{3}{4}}-\exp(4\sigma_{\varepsilon}^{-9})\big)+O(\rho_{\varepsilon}^{\frac{1}{17}})\Big)\frac{\varepsilon^{(N_{\varepsilon})}}{\sqrt{-\Lambda}}.\nonumber 
\end{align}
The bound (\ref{eq:LowerBoundEnergyTopBeam}) now follows readily
from (\ref{eq:AntaePia}), in view of the relations (\ref{eq:HierarchyOfParameters})
between $\rho_{\varepsilon}$, $\sigma_{\varepsilon}$ and $\delta_{\varepsilon}$.
\end{proof}

\subsection{Trapped surface formation and completion of the proof of Theorem
\ref{thm:TheTheorem}}

In this section, we will show that the energy content of the $N_{\varepsilon}$-th
beam $\mathcal{V}_{N_{\varepsilon}\nwarrow}^{(n_{+})}$, after its
interaction with the rest of the beams at the final step of the evolution
(which was studied in the previous section), is sufficiently high
for a trapped sphere to form before $\mathcal{V}_{N_{\varepsilon}\nwarrow}^{(n_{+})}$
reaches its minimum distance from the axis $\gamma_{\mathcal{Z}}$.
This statement will thus conclude the proof of Theorem \ref{thm:TheTheorem}.

In particular, we will show the following:
\begin{prop}
\label{prop:FinalTrappedSphereFormation} For any $\varepsilon\in(0,\varepsilon_{1}]$,
let $\{a_{\varepsilon i}\}_{i=0}^{N_{\varepsilon}}$, $(\mathcal{U}_{max}^{(\varepsilon)};r,\Omega^{2},f_{\varepsilon})$
and $n_{+}$ be as in Proposition \ref{prop:FineTuningInitialData}
and Lemma \ref{lem:ProfileOfTheLastStep}. Then, setting 
\begin{align}
\mathcal{B}_{\varepsilon}\doteq\big\{ v_{\varepsilon,N_{\varepsilon}-1}^{(n_{+})} & +\tilde{h}_{\varepsilon,N_{\varepsilon}-1}\le u\le v_{\varepsilon,N_{\varepsilon}}^{(n_{+})}-\delta_{\varepsilon}^{-\frac{1}{4}}h_{\varepsilon,N_{\varepsilon}}\big\}\cap\\
 & \cap\big\{ v\le v_{\varepsilon,N_{\varepsilon}}^{(n_{+})}+\exp(e^{\sigma_{\varepsilon}^{-7}})h_{\varepsilon,N_{\varepsilon}}\big\}\cap\big\{ u<v\big\}\nonumber 
\end{align}
(where $h_{\varepsilon,N_{\varepsilon}}$ is defined by (\ref{eq:ShorthandNotationCornerPoints})
and $\tilde{h}_{\varepsilon,N_{\varepsilon}-1}$ is defined by (\ref{eq:ShorthandNotationTilde})),
there exists a point $(u_{\dagger},v_{\dagger})\in\mathcal{B}_{\varepsilon}\cap\mathcal{U}_{max}^{(\varepsilon)}$
such that
\begin{equation}
\frac{2m}{r}(u_{\dagger},v_{\dagger})>1.\label{eq:TrappedSphereAtDagger}
\end{equation}
 In particular, $(\mathcal{U}_{max}^{(\varepsilon)};r,\Omega^{2},f_{\varepsilon})$
contains a trapped sphere.
\end{prop}
\begin{proof}
In order to establish (\ref{eq:TrappedSphereAtDagger}) for some $(u_{\dagger},v_{\dagger})\in\mathcal{B}_{\varepsilon}\cap\mathcal{U}_{max}^{(\varepsilon)}$,
we will assume for the sake of contradiction that 
\begin{equation}
\frac{2m}{r}\le1\text{ everywhere on }\mathcal{B}_{\varepsilon}\cap\mathcal{U}_{max}^{(\varepsilon)}.\label{eq:ContradictionAssumptionTraooedSphere.}
\end{equation}

\begin{figure}[h] 
\centering 
\scriptsize
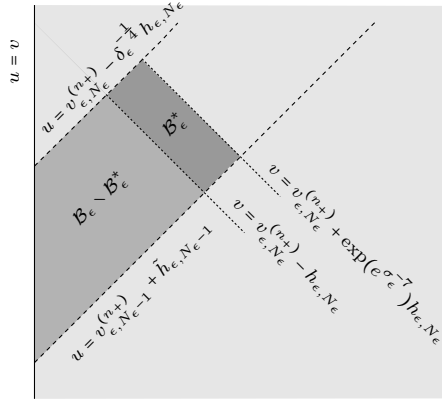 
\caption{In the figure above, the domain $\mathcal{B}_{\epsilon}$ consists of the two darker shaded regions $\mathcal{B}_{\epsilon} \setminus \mathcal{B}^{*}_{\epsilon}$ and $\mathcal{B}^{*}_{\epsilon}$. A fundamental step in the proof of Proposition \ref{prop:FinalTrappedSphereFormation} consists of showing that the physical-space support of the Vlasov field $f_{\epsilon N_{\epsilon}}$ in the region $\big\{ r\ge \delta_{\epsilon}^{-\frac{1}{4}}\frac{\epsilon^{(N_{\epsilon})}}{\sqrt{-\Lambda}}\big\}$ is contained in a domain $\mathcal{B}^{\sharp}_{\epsilon} \subset \mathcal{B}^{*}_{\epsilon}$. \label{fig:B_Domains}}
\end{figure}

Note that the bound (\ref{eq:ContradictionAssumptionTraooedSphere.})
implies, in view of the inequality 
\begin{equation}
\partial_{u}\Big(\frac{\Omega^{2}}{-\partial_{u}r}\Big)\le0\label{eq:NonIncreasingConstraint}
\end{equation}
(following readily from (\ref{eq:EquationROutside})), the relation
(\ref{eq:DefinitionHereHawkingMass}) and the fact that $\Omega^{2}$
is smooth on $\mathcal{U}_{max}^{(\varepsilon)}$, that 
\begin{equation}
\partial_{u}r<0\le\partial_{v}r\text{ on }\mathcal{B}_{\varepsilon}\cap\mathcal{U}_{max}^{(\varepsilon)}.\label{eq:InequalitiesRTrivialBepsilon}
\end{equation}
Furthermore, integrating the inequality (\ref{eq:NonIncreasingConstraint})
in $u$ from $u=v_{\varepsilon,0}^{(n_{+})}-h_{\varepsilon,0}$ and
using (\ref{eq:U+InLastStep}) and (\ref{eq:EstimateGeometryLikeC0})
at $u=v_{\varepsilon,0}^{(n_{+})}-h_{\varepsilon,0}$, we obtain the
following one-sided bound:
\begin{equation}
\sup_{\mathcal{B}_{\varepsilon}\cap\mathcal{U}_{max}^{(\varepsilon)}}\frac{\Omega^{2}}{-\partial_{u}r}\le\max_{u=v_{\varepsilon,0}^{(n_{+})}-h_{\varepsilon,0}}\frac{\Omega^{2}}{-\partial_{u}r}\le e^{\sigma_{\varepsilon}^{-4}}.\label{eq:OneSidedBoundDvr}
\end{equation}
Note also that, in view of (\ref{eq:InequalitiesRTrivialBepsilon})
and the bound (\ref{eq:EstimateGeometryLikeC0LargerDomain}) for $\{u=v_{\varepsilon,N_{\varepsilon}-1}^{(n_{+})}+\tilde{h}_{\varepsilon,N_{\varepsilon}-1}\}\cap\mathcal{B}_{\varepsilon}$
(which is contained in $\mathcal{T}_{\varepsilon}^{+}$, in view of
(\ref{eq:LastStepTau})), we can estimate 
\begin{equation}
\sup_{\mathcal{B}_{\varepsilon}}(-\Lambda r^{2})=-\Lambda r^{2}|_{(u,v)=\big(v_{\varepsilon,N_{\varepsilon}-1}^{(n_{+})}+\tilde{h}_{\varepsilon,N_{\varepsilon}-1},v_{\varepsilon,N_{\varepsilon}-1}^{(n_{+})}+\exp(e^{\sigma_{\varepsilon}^{-7}})h_{\varepsilon,N_{\varepsilon}}\big)}\le\varepsilon\label{eq:UpperBoundRho}
\end{equation}
and, hence, (\ref{eq:ContradictionAssumptionTraooedSphere.}) also
implies (in view of (\ref{eq:RenormalisedNewHawkingMass})) that 
\begin{equation}
\frac{2\tilde{m}}{r}\le1+\varepsilon\text{ everywhere on }\mathcal{B}_{\varepsilon}\cap\mathcal{U}_{max}^{(\varepsilon)}.\label{eq:ContradictionAssumptionTraooedSphere.-1}
\end{equation}

Among all components $f_{\varepsilon i}$ of the Vlasov field $f_{\varepsilon}$
(see the relation (\ref{eq:LinearCombinationVlasovFields})), only
$f_{\varepsilon N_{\varepsilon}}$ has non-trivial support on $\{u=v_{\varepsilon,N_{\varepsilon}-1}^{(n_{+})}+\tilde{h}_{\varepsilon,N_{\varepsilon}-1}\}\cap\mathcal{B}_{\varepsilon}\cap\mathcal{U}_{max}^{(\varepsilon)}$
(as a consequence of Lemma \ref{lem:QuantitativeBeamEstimate} on
the support of $f_{\varepsilon i}$ and (\ref{eq:LastStepTau})).
Hence, by the domain of dependence property, all the $f_{\varepsilon i}$'s
for $i\neq N_{\varepsilon}$ vanish on $\mathcal{B}_{\varepsilon}\cap\mathcal{U}_{max}^{(\varepsilon)}$,
i.\,e.:
\begin{equation}
f_{\varepsilon}|_{\mathcal{B}_{\varepsilon}\cap\mathcal{U}_{max}^{(\varepsilon)}}=a_{\varepsilon N_{\varepsilon}}f_{\varepsilon N_{\varepsilon}}|_{\mathcal{B}_{\varepsilon}\cap\mathcal{U}_{max}^{(\varepsilon)}}.\label{eq:OnlyOneVlasovFieldFepsilon}
\end{equation}

Let $\gamma\subset\mathcal{U}_{max}^{(\varepsilon)}$ be a future
directed null geodesic which is maximally extended through reflections
off $\mathcal{I}_{\varepsilon}$, in accordance with Definition 2.3
of \cite{MoschidisVlasovWellPosedness} (see also the statement of
Corollary \ref{cor:GeodesicPathsLongTimes}), such that $(\gamma,\dot{\gamma})$
lies in the support of $f_{\varepsilon N_{\varepsilon}}$ (where $\dot{\gamma}$
denotes the derivative with respect to the fixed affine parametrisation
of each maximal geodesic component $\gamma_{n}$ of $\gamma=\cup_{n}\gamma_{n}$).
In view of the bound (\ref{eq:BoundSupportFepsiloni}) for the support
of $f_{\varepsilon N_{\varepsilon}}$ and the inclusion (\ref{eq:U+InLastStep}),
we can trivially estimate that, at the point of $\gamma$ where $u=v_{\varepsilon,0}^{(n_{+})}-h_{\varepsilon,0}$:
\begin{equation}
v|_{\gamma\cap\{u=v_{\varepsilon,0}^{(n_{+})}-h_{\varepsilon,0}\}}\ge v_{\varepsilon,N_{\varepsilon}}^{(n_{+})}-h_{\varepsilon,N_{\varepsilon}}.\label{eq:vsupportGamma}
\end{equation}
Since $\gamma$ traces out a causal curve in $\mathcal{U}_{max}^{(\varepsilon)}$
(and, in particular, the coordinate function $v$ is non-decreasing
along $\gamma$), we infer from (\ref{eq:vsupportGamma}) that 
\begin{equation}
v|_{\gamma\cap\mathcal{B}_{\varepsilon}}\ge v_{\varepsilon,N_{\varepsilon}}^{(n_{+})}-h_{\varepsilon,N_{\varepsilon}}.
\end{equation}
Therefore: 
\begin{equation}
supp(f_{\varepsilon N_{\varepsilon}})\cap\mathcal{B}_{\varepsilon}\subset\{v\ge v_{\varepsilon,N_{\varepsilon}}^{(n_{+})}-h_{\varepsilon,N_{\varepsilon}}\}.\label{eq:LowerBoundVSupportFe}
\end{equation}
In view of(\ref{eq:OnlyOneVlasovFieldFepsilon}) and (\ref{eq:LowerBoundVSupportFe}),
we obtain that 
\begin{equation}
f_{\varepsilon}\equiv0\text{ on }\big(\mathcal{B}_{\varepsilon}\backslash\mathcal{B}_{\varepsilon}^{*}\big)\cap\mathcal{U}_{max}^{(\varepsilon)},
\end{equation}
where 
\[
\mathcal{B}_{\varepsilon}^{*}\doteq\{v\ge v_{\varepsilon,N_{\varepsilon}}^{(n_{+})}-h_{\varepsilon,N_{\varepsilon}}\}\cap\mathcal{B}_{\varepsilon}
\]
(see Figure \ref{fig:B_Domains}). By the domain of dependence property,
we therefore infer that the solution $(r,\Omega^{2},f_{\varepsilon})$
extends to the whole triangle 
\[
\mathcal{D}_{\varepsilon}\doteq\{v\le v_{\varepsilon,N_{\varepsilon}}^{(n_{+})}-h_{\varepsilon,N_{\varepsilon}}\}\cap\{u\ge v_{\varepsilon,N_{\varepsilon}-1}^{(n_{+})}+\tilde{h}_{\varepsilon,N_{\varepsilon}-1}\}\cap\{u<v\},
\]
i.\,e.~that 
\[
\mathcal{D}_{\varepsilon}\subset\mathcal{U}_{max}^{(\varepsilon)},
\]
 and that $\mathcal{D}_{\varepsilon}$ is a vacuum region, i.\,e.
\begin{equation}
f_{\varepsilon}|_{\mathcal{D}_{\varepsilon}}=0.\label{eq:VacuumTriangle}
\end{equation}

In view of (\ref{eq:VacuumTriangle}), we readily infer that $\partial_{v}r>0$
on $\mathcal{D}_{\varepsilon}$, and, therefore:
\[
\inf_{\mathcal{B}_{\varepsilon}^{*}\cap\{v=v_{\varepsilon,N_{\varepsilon}}^{(n_{+})}-h_{\varepsilon,N_{\varepsilon}}\}}r>0.
\]
Thus, using (\ref{eq:InequalitiesRTrivialBepsilon}), we infer that
\begin{align}
0<\inf_{\mathcal{B}_{\varepsilon}^{*}\cap\{v=v_{\varepsilon,N_{\varepsilon}}^{(n_{+})}-h_{\varepsilon,N_{\varepsilon}}\}} & r\le r|_{\mathcal{B}_{\varepsilon}^{*}\cap\mathcal{U}_{max}^{(\varepsilon)}}\label{eq:UpperLowerBoundr}\\
 & \le\max_{\mathcal{B}_{\varepsilon}^{*}\cap\{u=v_{\varepsilon,N_{\varepsilon}-1}^{(n_{+})}+\tilde{h}_{\varepsilon,N_{\varepsilon}-1}\}}r<+\infty.\nonumber 
\end{align}
The bounds (\ref{eq:ContradictionAssumptionTraooedSphere.}) and (\ref{eq:UpperLowerBoundr})
and the extension principle of Proposition \ref{prop:ExtensionPrincipleMihalis}
then readily imply that 
\begin{equation}
\mathcal{B}_{\varepsilon}^{*}\subset\mathcal{U}_{max}^{(\varepsilon)}.\label{eq:ExtensionOnAllOfB*}
\end{equation}

The following estimate for the support of the Vlasov field $f_{\varepsilon N_{\varepsilon}}$
will be crucial for the proof of Proposition \ref{prop:FinalTrappedSphereFormation}:
\begin{equation}
supp(f_{\varepsilon N_{\varepsilon}})\cap\mathcal{B}_{\varepsilon}\cap\big\{ r\ge\delta_{\varepsilon}^{-\frac{1}{4}}\frac{\varepsilon^{(N_{\varepsilon})}}{\sqrt{-\Lambda}}\big\}\subset\big\{(u,v)\in\mathcal{B}_{\varepsilon}^{\sharp}\big\},\label{eq:RefinedBoundSupportFepsilon}
\end{equation}
where 
\begin{align}
\mathcal{B}_{\varepsilon}^{\sharp}\doteq\big\{ v_{\varepsilon,N_{\varepsilon}}^{(n_{+})}-h_{\varepsilon,N_{\varepsilon}} & \le v\le v_{\varepsilon,N_{\varepsilon}}^{(n_{+})}+\frac{1}{2}\exp(e^{\sigma_{\varepsilon}^{-7}})h_{\varepsilon,N_{\varepsilon}}\big\}\cap\\
 & \cap\big\{ v_{\varepsilon,N_{\varepsilon}-1}^{(n_{+})}+\tilde{h}_{\varepsilon,N_{\varepsilon}-1}\le u\le v_{\varepsilon,N_{\varepsilon}}^{(n_{+})}-\delta_{\varepsilon}^{-\frac{1}{4}}h_{\varepsilon,N_{\varepsilon}}\big\}\subset\mathcal{B}_{\varepsilon}^{*}.\nonumber 
\end{align}

\medskip{}

\noindent \emph{Remark.} Note that the bound (\ref{eq:RefinedBoundSupportFepsilon})
does not follow from Lemma \ref{lem:QuantitativeBeamEstimate}, since
we do not expect the region $\mathcal{B}_{\varepsilon}^{\sharp}$
to lie within the domains $\mathcal{U}_{\varepsilon}^{+}$ or $\mathcal{T}_{\varepsilon}^{+}$,
where $\frac{2m}{r}\le\eta_{0}$ (in fact, our aim is to show that
$\frac{2m}{r}$ exceeds $1$ at some point on $\mathcal{B}_{\varepsilon}^{\sharp}$).
However, even when restricted to $\mathcal{B}_{\varepsilon}\cap\{u=v_{\varepsilon,N_{\varepsilon}-1}^{(n_{+})}+\tilde{h}_{\varepsilon,N_{\varepsilon}-1}\}$,
(\ref{eq:RefinedBoundSupportFepsilon}) is a partial improvement of
the bounds provided by Lemma \ref{lem:QuantitativeBeamEstimate},
since $\mathcal{B}_{\varepsilon}\cap\{u=v_{\varepsilon,N_{\varepsilon}-1}^{(n_{+})}+\tilde{h}_{\varepsilon,N_{\varepsilon}-1}\}$
is contained in $\mathcal{T}_{\varepsilon}^{+}$, but is not necessarily
contained in $\mathcal{U}_{\varepsilon}^{+}$, and hence Lemma \ref{lem:QuantitativeBeamEstimate}
can only guarantee bounds for $supp(f_{\varepsilon N_{\varepsilon}})$
in terms of $\tilde{h}_{\varepsilon N_{\varepsilon}}$.

\medskip{}

\noindent Let us assume, for a moment, that (\ref{eq:RefinedBoundSupportFepsilon})
has been established, and let us set 
\begin{equation}
(u_{\star},v_{\star})\doteq\text{future endpoint of }\big\{ v=v_{\varepsilon,N_{\varepsilon}-1}^{(n_{+})}+\frac{1}{2}\exp(e^{\sigma_{\varepsilon}^{-7}})h_{\varepsilon,N_{\varepsilon}}\big\}\cap\mathcal{B}_{\varepsilon}\cap\big\{ r\ge\delta_{\varepsilon}^{-\frac{1}{4}}\frac{\varepsilon^{(N_{\varepsilon})}}{\sqrt{-\Lambda}}\big\}.\label{eq:DefinitionU_*V_*}
\end{equation}

\begin{figure}[h] 
\centering 
\scriptsize
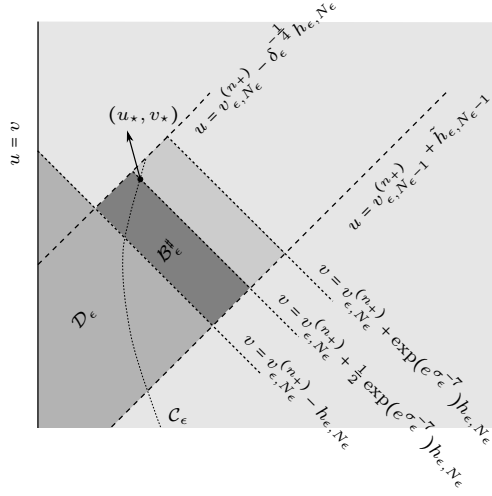 
\caption{Schematic depiction of the domains $\mathcal{D}_{\epsilon}$ and $\mathcal{B}^{\sharp}_{\epsilon}\subset \mathcal{B}^{*}_{\epsilon}$. The timelike curve $\mathcal{C}_{\epsilon} \doteq \{r=\delta_{\epsilon}^{-\frac{1}{4}}\frac{\epsilon^{(N_{\epsilon})}}{\sqrt{-\Lambda}}\}$ does not necessarily have to intersect the region $\mathcal{B}^{\sharp}_{\epsilon}$. In the case when the future boundary segment $ \{u=v_{\epsilon,N_{\epsilon}}^{(n_{+})}-\delta_{\epsilon}^{-\frac{1}{4}}h_{\epsilon,N_{\epsilon}}\} \cap \mathcal{B}^{\sharp}_{\epsilon}$ of $\mathcal{B}^{\sharp}_{\epsilon}$ lies to the left of the curve $\mathcal{C}_{\epsilon}$, the point $(u_{\star}, v_{\star})$ lies on $\mathcal{C}_{\epsilon}$. \label{fig:B_Domains_Second}}
\end{figure}

In view of (\ref{eq:OnlyOneVlasovFieldFepsilon}), the bound (\ref{eq:RefinedBoundSupportFepsilon})
on the support of $f_{\varepsilon N_{\varepsilon}}$ implies that
\[
f_{\varepsilon}\equiv0\text{ on }\big\{ v\ge v_{\varepsilon,N_{\varepsilon}-1}^{(n_{+})}+\frac{1}{2}\exp(e^{\sigma_{\varepsilon}^{-7}})h_{\varepsilon,N_{\varepsilon}}\big\}\cap\mathcal{B}_{\varepsilon}\cap\mathcal{U}_{max}^{(\varepsilon)},
\]
and hence $\tilde{m}$ is constant on $\big\{ v\ge v_{\varepsilon,N_{\varepsilon}-1}^{(n_{+})}+\frac{1}{2}\exp(e^{\sigma_{\varepsilon}^{-7}})h_{\varepsilon,N_{\varepsilon}}\big\}\cap\mathcal{B}_{\varepsilon}\cap\mathcal{U}_{max}^{(\varepsilon)}$.
This fact, combined with (\ref{eq:VacuumTriangle}) and the definition
(\ref{eq:IngoingEnergyAfter}) of $\widetilde{\mathcal{E}}_{\nwarrow}^{(+)}[n_{+};N_{\varepsilon},N_{\varepsilon}-1]$
implies that 
\begin{equation}
\tilde{m}(u_{\star},v_{\star})=\widetilde{\mathcal{E}}_{\nwarrow}^{(+)}[n_{+};N_{\varepsilon},N_{\varepsilon}-1].\label{eq:MtildeStar}
\end{equation}
The definition (\ref{eq:DefinitionU_*V_*}) of $(u_{\star},v_{\star})$
implies that:

\begin{itemize}

\item Either 
\begin{equation}
r(u_{\star},v_{\star})=\delta_{\varepsilon}^{-\frac{1}{4}}\frac{\varepsilon^{(N_{\varepsilon})}}{\sqrt{-\Lambda}},\label{eq:1111}
\end{equation}

\item Or 
\begin{equation}
(u_{\star},v_{\star})=\Big(v_{\varepsilon,N_{\varepsilon}-1}^{(n_{+})}-\delta_{\varepsilon}^{-\frac{1}{4}}h_{\varepsilon,N_{\varepsilon}},v_{\varepsilon,N_{\varepsilon}-1}^{(n_{+})}+\frac{1}{2}\exp(e^{\sigma_{\varepsilon}^{-7}})h_{\varepsilon,N_{\varepsilon}}\Big),
\end{equation}
in which case, by integrating the bound 
\[
\partial_{v}r=\frac{\partial_{v}r}{1-\frac{2m}{r}}\cdot\big(1-\frac{2\tilde{m}}{r}-\frac{1}{3}\Lambda r^{2}\big)\le2e^{\sigma_{\varepsilon}^{-4}}
\]
(following from (\ref{eq:OneSidedBoundDvr}) and (\ref{eq:UpperBoundRho}))
along $u=u_{\star}$ from $v=u_{\star}$ up to $v=v_{\star}$, we
can estimate 
\begin{equation}
r(u_{\star},v_{\star})\le\exp(e^{\sigma_{\varepsilon}^{-8}})\delta_{\varepsilon}^{-\frac{1}{4}}\frac{\varepsilon^{(N_{\varepsilon})}}{\sqrt{-\Lambda}}.\label{eq:2222}
\end{equation}

\end{itemize}

Since (\ref{eq:2222}) is weaker than (\ref{eq:1111}), we infer that,
in any case, (\ref{eq:2222}) always holds for $(u_{\star},v_{\star})$.
Combining (\ref{eq:MtildeStar}), (\ref{eq:2222}), (\ref{eq:UpperBoundRho})
and the lower bound (\ref{eq:LowerBoundEnergyTopBeam}) for $\widetilde{\mathcal{E}}_{\nwarrow}^{(+)}[n_{+};N_{\varepsilon},N_{\varepsilon}-1]$,
we therefore calculate (using the relation (\ref{eq:HierarchyOfParameters})
between $\sigma_{\varepsilon}$ and $\delta_{\varepsilon}$) that
\begin{align*}
\frac{2m}{r}(u_{\star},v_{\star}) & =\big(\frac{2\tilde{m}}{r}-\frac{1}{3}\Lambda r^{2}\big)(u_{\star},v_{\star})=\\
 & =\frac{2\widetilde{\mathcal{E}}_{\nwarrow}^{(+)}[n_{+};N_{\varepsilon},N_{\varepsilon}-1]}{r(u_{\star},v_{\star})}+O(\varepsilon)\ge\\
 & \ge\frac{\delta_{\varepsilon}^{-\frac{1}{2}}}{\exp(e^{\sigma_{\varepsilon}^{-8}})\delta_{\varepsilon}^{-\frac{1}{4}}}+O(\varepsilon)=\\
 & =\exp(-e^{\sigma_{\varepsilon}^{-8}})\delta_{\varepsilon}^{-\frac{1}{4}}+O(\varepsilon)\ge\\
 & \ge\delta_{\varepsilon}^{-\frac{1}{8}}>\\
 & >1,
\end{align*}
which is a contradiction, in view of our assumption (\ref{eq:ContradictionAssumptionTraooedSphere.}).
Thus, we infer that (\ref{eq:ContradictionAssumptionTraooedSphere.})
cannot hold, i.\,e.~that there exists some $(u_{\dagger},v_{\dagger})\in\mathcal{B}_{\varepsilon}\cap\mathcal{U}_{max}^{(\varepsilon)}$
for which (\ref{eq:TrappedSphereAtDagger}) holds.

In order to complete the proof of Proposition \ref{prop:FinalTrappedSphereFormation},
it therefore remains to establish (\ref{eq:RefinedBoundSupportFepsilon}).

\medskip{}

\noindent \emph{Proof of (\ref{eq:RefinedBoundSupportFepsilon}).}
Let $\gamma$ be any affinely parametrised, future directed, null
geodesic $\gamma$ (which is maximally extended through reflections
off $\mathcal{I}$) in the support of $f_{\varepsilon N_{\varepsilon}}$.
The definition (\ref{eq:InitialVlasovSeeds}) of $F_{N_{\varepsilon}}^{(\varepsilon)}$
and the relation (\ref{eq:Fepsiloni}) between $F_{N_{\varepsilon}}^{(\varepsilon)}$
and the initial data for $f_{\varepsilon N_{\varepsilon}}$ implies
that the angular momentum $l$ of $\gamma$ satisfies 
\begin{equation}
2\frac{\varepsilon^{(N_{\varepsilon})}}{\sqrt{-\Lambda}}\le l\le6\frac{\varepsilon^{(N_{\varepsilon})}}{\sqrt{-\Lambda}}.\label{eq:AngularMomentumComparable}
\end{equation}
In view of (\ref{eq:LowerBoundVSupportFe}), in order to show (\ref{eq:RefinedBoundSupportFepsilon}),
it suffices to show that 
\begin{equation}
\max\Big\{ v(p):\text{ }p\in\gamma\cap\mathcal{B}_{\varepsilon}\cap\big\{ r\ge\delta_{\varepsilon}^{-\frac{1}{4}}\frac{\varepsilon^{(N_{\varepsilon})}}{\sqrt{-\Lambda}}\big\}\Big\}\le v_{\varepsilon,N_{\varepsilon}-1}^{(n_{+})}+\frac{1}{2}\exp(e^{\sigma_{\varepsilon}^{-7}})h_{\varepsilon,N_{\varepsilon}}.\label{eq:UpperBoundVForGamma}
\end{equation}

\begin{figure}[h] 
\centering 
\scriptsize
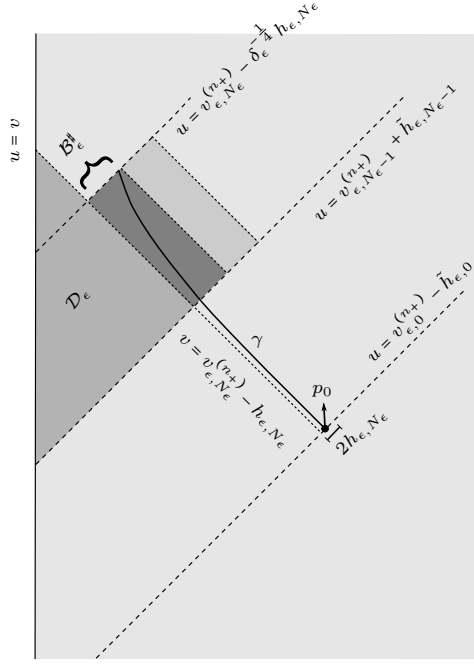 
\caption{Schematic depiction of a null geodesic $gamma$ in the support of $f_{\epsilon N_{\epsilon}}$. In order to establish the estimate (\ref{eq:UpperBoundVForGamma}), we will integrate the geodesic equation starting from the point $p_0$ (which lies in the region $\mathcal{U}^+_{\epsilon}$), i.\,e.~before the last interaction of $\gamma$ with the beams $\mathcal{V}^{(n_+)}_{i\nearrow}$, $0\le i \le N_{\epsilon}-1$.. \label{fig:B_Domains_Third}}
\end{figure}

Applying Lemma \ref{lem:QuantitativeBeamEstimate} for $i=N_{\varepsilon}$
and using (\ref{eq:U+InLastStep}) and (\ref{eq:EstimateGeometryLikeC0})
we infer that, at the point 
\[
p_{0}=(v_{\varepsilon,0}^{(n_{+})}-h_{\varepsilon,0},v_{0})\doteq\gamma\cap\{u=v_{\varepsilon,0}^{(n_{+})}-h_{\varepsilon,0}\},
\]
we can estimate 
\begin{equation}
\big|v_{\varepsilon,N_{\varepsilon}}^{(n_{+})}-v_{0}\big|\le h_{\varepsilon,N_{\varepsilon}},\label{eq:InitialVGamma}
\end{equation}
\begin{equation}
\exp(-\sigma_{\varepsilon}^{-6})\le\Omega^{2}(\dot{\gamma}^{u}+\dot{\gamma}^{u})\big|_{p_{0}}\le\exp\big(\exp(\sigma_{\varepsilon}^{-4})\big)\label{eq:InitialEnergyGamma}
\end{equation}
and
\begin{equation}
\frac{\dot{\gamma}^{v}}{\dot{\gamma}^{u}}\Big|_{p_{0}}\le\exp\big(\exp(\sigma_{\varepsilon}^{-4})\big)\frac{2l^{2}}{r^{2}}\Big|_{p_{0}}\le\varepsilon^{\frac{1}{2}}\label{eq:InitialIngoingGamma}
\end{equation}
(where, for \ref{eq:InitialIngoingGamma}, we made use of the fact
that $\varepsilon^{(N_{\varepsilon})}/r|_{p_{0}}\lesssim\varepsilon^{(N_{\varepsilon})}/\varepsilon^{(0)}<\varepsilon$).
Furthermore, using (\ref{eq:EstimateGeometryLikeC0}), (\ref{eq:BoundDOmegaEverywhere})
and the form (\ref{eq:BeamDomain}) of $\mathcal{V}_{N_{\varepsilon}\nwarrow}^{(n_{+})}$,
we can estimate for any $\bar{v}\in[v_{0},v_{0}+\exp(e^{\sigma_{\varepsilon}^{-7}})h_{\varepsilon,N_{\varepsilon}}]$:
\begin{equation}
\Big|\int_{v_{0}}^{\bar{v}}\big(\partial_{v}\log(\Omega^{2})-2\frac{\partial_{v}r}{r}\big)(v_{\varepsilon,0}^{(n_{+})}-h_{\varepsilon,0},v)\,dv\Big|\le\exp\big(\exp(\sigma_{\varepsilon}^{-6})\big)\sqrt{-\Lambda}\frac{|\bar{v}-v_{0}|}{\varepsilon^{(N_{\varepsilon})}}.\label{eq:BoundForInitialDataTermGeodesicEquation}
\end{equation}

We will establish (\ref{eq:UpperBoundVForGamma}) by continuity: We
will show that, for any $\bar{u}\in[v_{\varepsilon,0}^{(n_{+})}-h_{\varepsilon,0},v_{\varepsilon,N_{\varepsilon}-1}^{(n_{+})}-\delta_{\varepsilon}^{-\frac{1}{4}}h_{\varepsilon,N_{\varepsilon}}]$
such that 
\begin{equation}
v(p)\le v_{0}+\frac{1}{4}\exp(e^{\sigma_{\varepsilon}^{-7}})h_{\varepsilon,N_{\varepsilon}}\text{ for all }p\in\gamma\cap\big\{ v_{\varepsilon,0}^{(n_{+})}-h_{\varepsilon,0}\le u\le\bar{u}\big\}\cap\big\{ r\ge\delta_{\varepsilon}^{-\frac{1}{4}}\frac{\varepsilon^{(N_{\varepsilon})}}{\sqrt{-\Lambda}}\big\}\Big\}\label{eq:AssumptionContinuity}
\end{equation}
(note that (\ref{eq:AssumptionContinuity}) holds trivially when $\bar{u}=v_{\varepsilon,0}^{(n_{+})}-h_{\varepsilon,0}$),
the following stronger bound actually holds: 
\begin{equation}
v(p)\le v_{0}+\frac{1}{8}\exp(e^{\sigma_{\varepsilon}^{-7}})h_{\varepsilon,N_{\varepsilon}}\text{ for all }p\in\gamma\cap\big\{ v_{\varepsilon,0}^{(n_{+})}-h_{\varepsilon,0}\le u\le\bar{u}\big\}\cap\big\{ r\ge\delta_{\varepsilon}^{-\frac{1}{4}}\frac{\varepsilon^{(N_{\varepsilon})}}{\sqrt{-\Lambda}}\big\}\Big\}.\label{eq:ToShowContinuity}
\end{equation}

Let $\bar{u}\in[v_{\varepsilon,0}^{(n_{+})}-h_{\varepsilon,0},v_{\varepsilon,N_{\varepsilon}-1}^{(n_{+})}-\delta_{\varepsilon}^{-\frac{1}{4}}h_{\varepsilon,N_{\varepsilon}}]$
satisfy (\ref{eq:AssumptionContinuity}). Then, for any $u'\in[v_{\varepsilon,0}^{(n_{+})}-h_{\varepsilon,0},\bar{u}]$,
setting 
\begin{equation}
v'=\sup\Big\{ v(p):\text{ }p\in\gamma\cap\big\{ v_{\varepsilon,0}^{(n_{+})}-h_{\varepsilon,0}\le u\le u'\big\}\cap\big\{ r\ge\delta_{\varepsilon}^{-\frac{1}{4}}\frac{\varepsilon^{(N_{\varepsilon})}}{\sqrt{-\Lambda}}\big\}\Big\}\label{eq:v'}
\end{equation}
and applying (\ref{eq:XrhsimhGewdaisiakh-U}) for $u_{1}(v)=v_{\varepsilon,0}^{(n_{+})}-h_{\varepsilon,0}$,
we obtain:
\end{proof}
\begin{align}
\Big|\log\big(\Omega^{2}\dot{\gamma}^{u}\big)|_{v=v'}-\log\big(\Omega^{2}\dot{\gamma}^{u}\big)|_{v=v_{0}}\Big|\le & \Big|\int_{v_{0}}^{v'}\int_{v_{\varepsilon,0}^{(n_{+})}-h_{\varepsilon,0}}^{u(\gamma(s_{v}))}\Big(\frac{1}{2}\frac{\frac{6\tilde{m}}{r}-1}{r^{2}}\Omega^{2}-24\pi T_{uv}\Big)\,dudv\Big|+\label{eq:UsefulRelationForGeodesicWithMu-U-2}\\
 & +\Big|\int_{v_{0}}^{v'}\big(\partial_{v}\log(\Omega^{2})-2\frac{\partial_{v}r}{r}\big)(u_{1}(v),v)\,dv\Big|.\nonumber 
\end{align}

In view of the relation (\ref{eq:TildeUMaza}) for $\tilde{m}$, the
upper bound (\ref{eq:OneSidedBoundDvr}) and the fact that $\partial_{u}\tilde{m}\le0$
(which follows readily from (\ref{eq:TildeUMaza}) and (\ref{eq:InequalitiesRTrivialBepsilon})),
we can estimate on $\mathcal{B}_{\varepsilon}$:
\begin{align}
24\pi T_{uv} & \le3\frac{\Omega^{2}}{-\partial_{u}r}\frac{-\partial_{u}\tilde{m}}{r^{2}}\le\label{eq:TrivialBoundTuv}\\
 & \le3e^{\sigma_{\varepsilon}^{-4}}\frac{-\partial_{u}\tilde{m}}{r^{2}}=\nonumber \\
 & =3e^{\sigma_{\varepsilon}^{-4}}\Big(-\partial_{u}\big(\frac{\tilde{m}}{r^{2}}\big)+2\frac{\tilde{m}}{r^{3}}(-\partial_{u}r)\Big).\nonumber 
\end{align}
Furthermore, from the relation (\ref{eq:DefinitionHereHawkingMass})
and the upper bound (\ref{eq:OneSidedBoundDvr}), we can estimate
on $\mathcal{B}_{\varepsilon}$: 
\begin{equation}
\Omega^{2}\le4e^{\sigma_{\varepsilon}^{-4}}(-\partial_{u}r).\label{eq:UpperBoundOmega}
\end{equation}
Using the bounds (\ref{eq:ContradictionAssumptionTraooedSphere.-1}),
(\ref{eq:BoundForInitialDataTermGeodesicEquation}), (\ref{eq:TrivialBoundTuv})
and (\ref{eq:UpperBoundOmega}) to estimate the right hand side of
(\ref{eq:UsefulRelationForGeodesicWithMu-U-2}) (integrating, also,
in $u$ for the $\partial_{u}\big(\frac{\tilde{m}}{r^{2}}\big)$ term),
we obtain:
\begin{align}
\Big|\log\big(\Omega^{2}\dot{\gamma}^{u}\big)|_{v=v'}-\log\big(\Omega^{2}\dot{\gamma}^{u}\big)|_{v=v_{0}}\Big|\le & \int_{v_{0}}^{v'}\int_{v_{\varepsilon,0}^{(n_{+})}-h_{\varepsilon,0}}^{u(\gamma(s_{v}))}\Big(\frac{1}{2}\frac{\frac{6\tilde{m}}{r}+1}{r^{2}}\Omega^{2}\Big)\,dudv+24\pi\int_{v_{0}}^{v'}\int_{v_{\varepsilon,0}^{(n_{+})}-h_{\varepsilon,0}}^{u(\gamma(s_{v}))}T_{uv}\,dudv+\label{eq:UsefulRelationForGeodesicWithMu-U-2-1}\\
 & +\Big|\int_{v_{0}}^{v'}\big(\partial_{v}\log(\Omega^{2})-2\frac{\partial_{v}r}{r}\big)(u_{1}(v),v)\,dv\Big|\le\nonumber \\
\le & e^{\sigma_{\varepsilon}^{-5}}\int_{v_{0}}^{v'}\int_{v_{\varepsilon,0}^{(n_{+})}-h_{\varepsilon,0}}^{u(\gamma(s_{v}))}\frac{(-\partial_{u}r)}{r^{2}}\,dudv-3e^{\sigma_{\varepsilon}^{-4}}\int_{v_{0}}^{v'}\int_{v_{\varepsilon,0}^{(n_{+})}-h_{\varepsilon,0}}^{u(\gamma(s_{v}))}\partial_{u}\big(\frac{\tilde{m}}{r^{2}}\big)\,dudv+\nonumber \\
 & +\Big|\int_{v_{0}}^{v'}\big(\partial_{v}\log(\Omega^{2})-2\frac{\partial_{v}r}{r}\big)(u_{1}(v),v)\,dv\Big|\le\nonumber \\
\le & e^{\sigma_{\varepsilon}^{-5}}\frac{1}{r|_{\gamma\cap\{v=v'\}}}|v'-v_{0}|+\exp\big(\exp(\sigma_{\varepsilon}^{-6})\big)\sqrt{-\Lambda}\frac{|v'-v_{0}|}{\varepsilon^{(N_{\varepsilon})}}.\nonumber 
\end{align}
In view of the definition (\ref{eq:v'}) of $v'$ (and, in particular,
the fact that $r\ge\delta_{\varepsilon}^{-\frac{1}{4}}\frac{\varepsilon^{(N_{\varepsilon})}}{\sqrt{-\Lambda}}$
at $\gamma\cap\{v=v'\}$), from (\ref{eq:UsefulRelationForGeodesicWithMu-U-2-1})
we obtain that:
\begin{equation}
\Big|\log\big(\Omega^{2}\dot{\gamma}^{u}\big)|_{v=v'}-\log\big(\Omega^{2}\dot{\gamma}^{u}\big)|_{v=v_{0}}\Big|\le\exp\big(\exp(2\sigma_{\varepsilon}^{-6})\big)\sqrt{-\Lambda}\frac{|v'-v_{0}|}{\varepsilon^{(N_{\varepsilon})}}.\label{eq:AnteScedonTelos}
\end{equation}
In view of (\ref{eq:InitialEnergyGamma}) and assumption (\ref{eq:AssumptionContinuity}),
from (\ref{eq:AnteScedonTelos}) we infer that 
\begin{equation}
\Omega^{2}\dot{\gamma}^{u}|_{v=v'}\ge\exp\Big(-\exp\big(\exp(\sigma_{\varepsilon}^{-8})\Big).\label{eq:LowerBoundEnergyAnte}
\end{equation}
Furthermore, using the null-shell relation (\ref{eq:NullShellNewAngularMomentum})
for $\gamma$, the relation (\ref{eq:DefinitionHereHawkingMass})
for $\Omega^{2}$, the upper bounds (\ref{eq:ContradictionAssumptionTraooedSphere.})
and (\ref{eq:OneSidedBoundDvr}) and the lower bound (\ref{eq:LowerBoundEnergyAnte}),
we can readily estimate (using also the fact that $r\ge\delta_{\varepsilon}^{-\frac{1}{4}}\frac{\varepsilon^{(N_{\varepsilon})}}{\sqrt{-\Lambda}}$
at $\gamma\cap\{v=v'\}$, as a consequence of (\ref{eq:AssumptionContinuity})):
\begin{equation}
\frac{\partial_{v}r\cdot\dot{\gamma}^{v}}{-\partial_{u}r\cdot\dot{\gamma}^{u}}\big|_{v=v'}=\frac{1}{4}\big(1-\frac{2m}{r})\Big(\frac{\Omega^{2}}{-\partial_{u}r}\Big)^{2}\frac{l^{2}}{r^{2}}\frac{1}{(\Omega^{2}\dot{\gamma}^{u})^{2}}\Big|_{v=v'}\le\exp\Big(\exp\big(\exp(\sigma_{\varepsilon}^{-9})\Big)\delta_{\varepsilon}^{\frac{1}{2}}\le\delta_{\varepsilon}^{\frac{1}{4}}.\label{eq:ForEstimateInDr}
\end{equation}
In particular, (\ref{eq:ForEstimateInDr}) implies that the integration
forms $dr$ and $du$ along $\gamma$ satisfy at $v=v'$:
\begin{equation}
dr|_{\gamma\cap\{v=v'\}}=\big(\partial_{u}r+\frac{\dot{\gamma}^{v}}{\dot{\gamma}^{u}}\partial_{v}r\big)du|_{\gamma\cap\{v=v'\}}=\partial_{u}r(1+O(\delta_{\varepsilon}^{\frac{1}{4}}))du|_{\gamma\cap\{v=v'\}}.\label{eq:RelationForms}
\end{equation}

Using the relation (\ref{eq:NullShellNewAngularMomentum}) for $\gamma$
as well as the bounds (\ref{eq:OneSidedBoundDvr}), (\ref{eq:LowerBoundEnergyAnte}),
(\ref{eq:ForEstimateInDr}) and (\ref{eq:RelationForms}), we can
readily estimate that, for all $p\in\gamma\cap\big\{ v_{\varepsilon,0}^{(n_{+})}-h_{\varepsilon,0}\le u\le\bar{u}\big\}\cap\big\{ r\ge\delta_{\varepsilon}^{-\frac{1}{4}}\frac{\varepsilon^{(N_{\varepsilon})}}{\sqrt{-\Lambda}}\big\}\Big\}$:
\begin{align}
v(p)-v_{0} & \le\int_{v_{\varepsilon,0}^{(n_{+})}-h_{\varepsilon,0}}^{\bar{u}}\frac{\Omega^{2}\dot{\gamma}^{v}}{\Omega^{2}\dot{\gamma}^{u}}(s_{u})\,du=\label{eq:AmanPia!}\\
 & =\int_{v_{\varepsilon,0}^{(n_{+})}-h_{\varepsilon,0}}^{\bar{u}}\frac{\Omega^{2}l^{2}}{r^{2}}\frac{1}{(\Omega^{2}\dot{\gamma}^{u})^{2}}(s_{u})\,du\le\nonumber \\
 & \le\exp\Big(2\exp\big(\exp(\sigma_{\varepsilon}^{-8})\Big)\int_{v_{\varepsilon,0}^{(n_{+})}-h_{\varepsilon,0}}^{\bar{u}}\frac{\Omega^{2}l^{2}}{r^{2}}\Big|_{\gamma(s_{u})}\,du\le\nonumber \\
 & \le\exp\Big(4\exp\big(\exp(\sigma_{\varepsilon}^{-8})\Big)\int_{\gamma\cap\{v_{\varepsilon,0}^{(n_{+})}-h_{\varepsilon,0}\le u\le\bar{u}\}}\frac{l^{2}}{r^{2}}\frac{\Omega^{2}}{-\partial_{u}r}\,dr\le\nonumber \\
 & \le\exp\Big(\exp\big(\exp(\sigma_{\varepsilon}^{-9})\Big)\int_{\gamma\cap\{v_{\varepsilon,0}^{(n_{+})}-h_{\varepsilon,0}\le u\le\bar{u}\}}\frac{l^{2}}{r^{2}}\,dr\le\nonumber \\
 & \le\exp\Big(\exp\big(\exp(\sigma_{\varepsilon}^{-9})\Big)\frac{l^{2}}{\inf_{\gamma\cap\{v_{\varepsilon,0}^{(n_{+})}-h_{\varepsilon,0}\le u\le\bar{u}\}}r}\le\nonumber \\
 & \le\delta_{\varepsilon}^{\frac{1}{8}}\frac{\varepsilon^{(N_{\varepsilon})}}{\sqrt{-\Lambda}},\nonumber 
\end{align}
where, in passing to the last line of (\ref{eq:AmanPia!}), we made
use of (\ref{eq:AngularMomentumComparable}) and the bound $r\ge\delta_{\varepsilon}^{-\frac{1}{4}}\frac{\varepsilon^{(N_{\varepsilon})}}{\sqrt{-\Lambda}}$
on $\gamma\cap\{v_{\varepsilon,0}^{(n_{+})}-h_{\varepsilon,0}\le u\le\bar{u}\}$
(following from (\ref{eq:AssumptionContinuity})). From (\ref{eq:AmanPia!}),
the bound (\ref{eq:ToShowContinuity}) follows readily, in view of
the relation (\ref{eq:HierarchyOfParameters}) between $\sigma_{\varepsilon}$
and $\delta_{\varepsilon}$. Thus, by continuity, we have established
(\ref{eq:UpperBoundVForGamma}) and, therefore, (\ref{eq:RefinedBoundSupportFepsilon}).
Thus, the proof of Proposition \ref{prop:FinalTrappedSphereFormation}
has been completed.

\bibliographystyle{plain}
\bibliography{DatabaseExample}

\end{document}